\documentclass[12pt,reqno]{amsbook}

\usepackage[colorlinks=true,urlcolor=blue,
bookmarks=true,bookmarksopen=true,citecolor=blue,linkcolor=orange
]{hyperref}

\usepackage{color,soul}
\setstcolor{red}

\usepackage{pdfsync}
\usepackage[all, cmtip]{xy}
\usepackage{graphicx}

\usepackage{dsfont}
\usepackage{float}
\usepackage{bm}
\usepackage{mathtools}
\usepackage{amssymb}
\usepackage[T1]{fontenc}

\usepackage{epsfig}
\usepackage{amscd}
\usepackage[mathscr]{eucal}
\usepackage{amssymb}
\usepackage{bm}
\usepackage{amsxtra}
\usepackage{amsmath}
\usepackage{latexsym} 
\usepackage[all]{xy}
\usepackage{enumerate}

\usepackage{enumitem}

\usepackage{color}
\usepackage{xcolor, soul}
\sethlcolor{green}

\usepackage{bbm}

\theoremstyle{plain}

\newtheorem{mainthm}{Theorem}
\newtheorem{thm}{Theorem}[chapter]

\newtheorem{cor}[thm]{Corollary}

\newtheorem{lem}[thm]{Lemma}
\newtheorem{prop}[thm]{Proposition}

\theoremstyle{definition}
\newtheorem{dfn}[thm]{Definition}

\newtheorem*{claim-nonum}{Claim}

\theoremstyle{remark}
\newtheorem{remark}[thm]{Remark}
\newtheorem{rem}[thm]{Remark}
\newtheorem*{remnonum}{Remark}

\newtheorem{ex}[thm]{Example}

\theoremstyle{plain}

\newtheorem{lemma}[thm]{Lemma}

\def\ep{\epsilon}

\def\C{\mathcal{C}}
\def\A{\mathcal{A}}

\def\K{\mathcal K}
\def\P{\mathcal P}
\def\D{\mathcal{D}}

\def\k{{\bf k}}
\def\F{\mathcal F}
\def\T{\mathcal T}
\def\G{\mathcal G}
\def\U{\mathbb U}
\def\V{\mathbb V}
\def\W{\mathbb W}

\def\Mor{{\hom}}

\DeclarePairedDelimiter\ceil{\lceil}{\rceil}

\newcommand{\Qed}{\hfill \qedsymbol \medskip}

\newcommand{\id}{\mathds{1}}

\newcommand{\lag}{\mathcal{L}}

\newcommand{\R}{\mathbb{R}}
\newcommand{\Z}{\mathbb{Z}}
\newcommand{\Q}{\mathbb{Q}}
\newcommand{\N}{\mathbb{N}}

\newcommand{\fuk}{\mathcal{F}uk}

\newcommand{\rank}{\mathrm{rank\/}}

\newcommand{\mor}{{\hom}}

\def\Hom{{\hom}}

\renewcommand{\k}{\mathbf{k}}

\newcommand{\hhat}{roof diagram}
\newcommand{\hhats}{roof diagrams}

\numberwithin{section}{chapter}

\newcommand{\pbred}[1]{#1}

\newcommand{\pbnote}[1]{#1}

 \newcommand{\pbrev}[1]{#1}

 \newcommand{\pbrrvv}[1]{#1}

\newcommand{\ocnote}[1]{#1}
\newcommand{\jznote}[1]{#1}
\newcommand{\zjnote}[1]{#1}
\newcommand{\zhnote}[1]{#1}

\newcommand{\zjr}[1]{#1}

\newcommand{\Ob}{\text{Obj}}

\newcommand{\cl}{\Sigma}
\newcommand{\tr}{\text{tr}}
\newcommand{\srf}{S}
\newcommand{\extp}{\text{ep}}
\newcommand{\clus}{\text{Clus}}

\newcommand{\pbaddress}{biran@math.ethz.ch}
\newcommand{\ocaddress}{cornea@dms.umontreal.ca}

\begin{document}

\title{Triangulation, Persistence, and Fukaya categories}

\date{\today}


\author{Paul Biran, Octav Cornea\footnote{The second author was supported by an individual NSERC
  Discovery grant. } and Jun Zhang}
\address{Paul Biran, Department of Mathematics, ETH-Z\"{u}rich,
  R\"{a}mistrasse 101, 8092 Z\"{u}rich, Switzerland}
\email{\pbaddress}
 
\address{Octav Cornea, Department of Mathematics and Statistics,
  University of Montreal, C.P. 6128 Succ.  Centre-Ville Montreal, QC
  H3C 3J7, Canada}
 \email{\ocaddress}

\address{Jun Zhang, The Institute of Geometry and Physics, University of Science and Technology of China, 96 Jinzhai Road, Hefei, Anhui, 230026, China} \email{jzhang4518@ustc.edu.cn}

\keywords{Triangulated category, Persistence module, Symplectic manifold, Lagrangian submanifold, Floer homology, Fukaya category.}
\subjclass[2020]{55N31 53D12 (Primary); 53D37 (Secondary)}

\bibliographystyle{plain}
\bibliographystyle{alphanum}

%

%

\begin{abstract}
  This paper introduces a new algebraic notion - {\em triangulated
    persistence category} (TPC) - that refines that of triangulated
  category in the same sense that a \zhnote{persistence module} is a
  refinement of the notion of \pbred{a vector space}. \zhnote{The
    spaces of morphisms} of such a TPC are persistence modules and
  this category is endowed with a class of weighted distinguished
  triangles. Under favourable conditions we show that the derived
  Fukaya category admits a TPC refinement and this is applied to
  deduce a global rigidity result for spaces of compact, exact
  Lagrangians in certain Liouville manifolds: we construct a metric on
  this space with \pbred{intrinsic} symplectic properties.
\end{abstract}

\maketitle

%
%

\tableofcontents 

\chapter{Introduction} \label{sec-intro}

The last \zhnote{40} years have seen spectacular advances in
symplectic topology. Most of them, and particularly those exhibiting
aspects of symplectic rigidity, exploit algebraic structures that
encode the behavior of moduli spaces of solutions of Cauchy-Riemann
type equations associated to \zjr{(variants of)} the symplectic action
functional. Typical examples of such structures are Floer homology,
Gromov-Witten invariants, and the Fukaya category. Given that these
structures are, in essence, associated to a functional, they can be
expected to admit refinements endowed with a finer structure
reflective of an underlying filtration.  Making this statement precise
and incorporating this filtration in the respective algebraic
structures is sometimes \zhnote{straightforward from the algebraic
  viewpoint}. For instance, in favorable cases, the Floer complex -
just like its more down to earth precursor, the Morse complex - is
filtered, and its homology is a persistence module.

\

In other cases, such as that of the derived Fukaya category, which is
the one that interests us here, this is far from immediate.  In this
paper we set up a new algebraic structure called a {\em triangulated
  persistence category} (TPC) precisely to deal with this
situation. This structure puts together persistence and triangulation
and is a refinement of the notion of triangulated category. The
construction is abstract and applicable to a variety of contexts
unrelated to symplectic topology, \zhnote{as explained} in more detail
below, in \S\ref{subsec:intro1}.

The derived Fukaya category has a triangulated structure and we show
in the paper that, under certain constraints, it does admit a TPC
refinement that is unique up to equivalence. The construction of this
refinement and its uniqueness are delicate and require some novel
geometric and algebraic steps.  We describe in more \zhnote{detail}
the results and the constructions involved in
\S\ref{subsec:intro2}. These constructions too are of independent
interest.

\

A natural application of the construction of Fukaya type triangulated
persistence categories is a rigidity result for spaces of Lagrangian
\zhnote{submanifolds}.  To fix ideas, let $(X,\omega)$ be a symplectic
manifold. It is well-known since the pioneering work of Gromov and
Floer that closed Lagrangian submanifolds $L\subset X$ subject to
certain {\em purely topological} constraints - the one used in this
paper is exactness - exhibit strong, and often surprising, rigidity
properties that are intrinsically symplectic. Generally, this form of
rigidity reflects individual properties of each of the Lagrangians in
the fixed class. Two famous examples that have structured much of the
modern work in the subject are \zhnote{the} Arnold conjecture and the
nearby Lagrangian conjecture, also due to Arnold.

In this work, we show that, in the same setting, a global form of
rigidity is in effect.  More precisely, let $(X, \omega)$ be a
Liouville \zhnote{manifold that} satisfies an algebraic finiteness
condition that will be made explicit below.  The set of closed, exact
Lagrangians in $X$ is endowed with a class of metrics, called {\em
  symplectic fragmentation} metrics with some remarkable properties
(see Corollary \ref{cor:gl-metric} for a more precise version and
details):
\begin{itemize}
\item[-] Up to a multiplicative constant, these metrics are dominated
  by the spectral metric (that itself has as upper bound the Hofer
  metric), thus they carry symplectic content.
\item[-] The non-degeneracy of these metrics can be viewed as a form
  of Gromov's non squeezing theorem in the sense that the distance
  between two Lagrangians has a lower bound that can be expressed in
  terms of a purely geometric quantity, the supremum of radii of
  standard symplectic balls that \zjr{embed} in a certain position
  relative to the two Lagrangians.
\item[-] The metrics are finite and thus they allow meaningful
  comparison of Lagrangians that are very different as smooth
  submanifolds (non-isotopic, or of different homotopy types), when
  classical metrics, such as the Hofer distance, are infinite.
\item[-] At the same time they also satisfy a property of stability of
  intersections in the sense that, given two transverse Lagrangians
  $L$, and $N$, if a third Lagrangian $L'$ is sufficiently close to
  $L$ in one of these metrics, then the number of intersections of
  $L'$ with $N$ \zjr{cannot be smaller} than the number of
  intersections of $N$ and $L$.
\end{itemize}

The relation between this statement and the notion of TPC is that if
\zhnote{a triangulated} category admits a TPC refinement, then, by the
main algebraic result in this paper, its exact triangles are endowed
with a so-called {\em persistence} triangular weight. The set of
objects of a triangulated \zhnote{category, endowed} with such a
triangular \zhnote{weight, is} easily seen to carry a family of
natural pseudo-metrics called {\em fragmentation} pseudo-metrics. The
symplectic fragmentation metric mentioned above is deduced from the
fragmentation pseudo-metrics associated to the Fukaya TPC.

\begin{rem} Precursors of the metrics introduced here have appeared in
  \cite{Bi-Co-Sh:LagrSh}, based on Lagrangian cobordism machinery.
  However, the constructions in that paper lacked the proper algebraic
  setting, with the consequence that the finiteness of the distance
  between two Lagrangians depended on the existence of certain
  Lagrangian cobordisms. This issue was addressed, in part, in
  \cite{Bi-Co:LagPict} through considerations involving immersed
  Lagrangians, which allow the construction of an abundance of
  immersed cobordisms. However, the immersed cobordism approach is
  technically very delicate, \zhnote{and less} natural than the one
  proposed here with the consequence that it leads to family of
  metrics that are extremely hard to estimate.
\end{rem}

\section{Persistence and triangulation}\label{subsec:intro1}
Persistence theory, introduced in several pioneering works
\zjnote{\cite{ELZ-top-00, Co-Ra, ZC05, C-SEH-stab-07, Ghr-bar-08,
    Car-top-09, Wei-what-11, Les-mult-15}}, is an abstract framework
that emerged from investigations in parts of \zhnote{data science} as
well as \zhnote{in topology,} formalizing the structure and properties
of a class of phenomena that are most easily seen in the homology of
\zhnote{a chain complex $(C, d)$ endowed with an increasing filtration
  $(C^{\leq \alpha}, d) \subset (C,d)$ of subcomplexes parametrized by
  $\alpha \in \R$}. The \zhnote{homologies} of \zhnote{the
  subcomplexes form} a family $\{H(C^{\leq \alpha})\}_{\alpha\in \R}$,
\pbred{whose} members are related by maps
$i_{\alpha,\beta}: H(C^{\leq \alpha}) \to H(C^{\leq \beta})$,
$\alpha\leq \beta$ subject to obvious compatibilities. \zhnote{This is
  an example of a} persistence module.  Given two filtered complexes
$(C,d)$ and $(D,d)$ that are quasi-isomorphic, it is possible to
compare them by the \zhnote{so called} interleaving distance.  Its
definition is based on the fact that the \zhnote{space of} linear maps
$v: C\to D$ \zhnote{is itself} filtered by \zhnote{the} ``shift''
\zhnote{of a map}: $v$ is of shift $\leq r$ if
$v(C^{\leq \alpha})\subset D^{\leq \alpha +r}$, for all
$\alpha \in \R$.  Using this, given two chain maps $\phi: C\to D$,
$\psi : D\to C$ such that \zhnote{$\psi \circ \phi$ is chain homotopic
  to} \pbred{$\id_C$}, there is a natural measurement for how far the
composition $\psi \circ \phi$ is from the identity, namely the infimum
of the ``shifts'' of chain homotopies $h: C \to C$ such that
\pbred{$dh+hd=\psi\circ\phi - \id_C$}.  The machinery of persistence
modules is much more developed than the few elements mentioned
here. For a survey on this topic and its applications in various
mathematical branches, see research monographs and papers from
Edelsbrunner \cite{Ede14}, Oudot \cite{Oud-quiver-15}, Chazal-de
Silva-Glisse-Oudot \cite{CdeSGO16}, Polterovich-Shelukhin
\cite{PS-pers-16}, Polterovich-Rosen-Samvelyan-Zhang \cite{PRSZ20},
and Kislev-Shelukhin \cite{AS-spectral-21}. In particular, there is a
beautiful interpretation of the bottleneck distance in terms of
\zhnote{so called} barcodes (\cite{BL15, UZ16}).

The main question that we address in the algebraic part of this paper
is independent of symplectic considerations:
 
\
 
{\em How can one use a persistence type structure on the morphisms of
  \zhnote{a} category to compare not only (quasi)-isomorphic objects
  but rather define a pseudo-metric on the set of all
  \pbred{objects?}}

\

We provide here a solution to this question based on mixing
persistence with triangulation understood in the sense of triangulated
categories as introduced by Puppe \cite{Pup62} and Verdier
\cite{Ver77} in the early \zhnote{1960's}.  Given a triangulated
category $\mathcal{D}$ there is a simple notion of {\em triangular
  weight} $w$ on $\mathcal{D}$ that we introduce in \S
\ref{subsec:triang-gen}. This associates to each exact triangle
$\Delta$ in $\mathcal{D}$ a non-negative number $w(\Delta)$ satisfying
a couple of properties.  The most relevant of them is a weighted
variant of the octahedral axiom (we will give a more precise
definition later).  A basic example of a triangular weight is the flat
one: it associates to each exact triangle the value $1$.  The interest
of triangular weights is that they naturally lead to {\em
  fragmentation pseudo-metrics} on $\mathrm{Obj}(\mathcal{D})$ (we
assume here that $\mathcal{D}$ is small) defined roughly as follows
(see \S\ref{subsec:triang-gen} for details).  Such a pseudo-metric
depends on a family of objects $\mathcal{F}$ of $\mathcal{D}$.  With
$\mathcal{F}$ fixed, and up to a certain normalization, the
pseudo-distance $d^{\mathcal{F}}(X,Y)$ between $X$,
\zhnote{$Y\in \mathrm{Obj}(\mathcal{D})$} is (the symmetrization of)
the infimum of the total weight of exact triangles needed to construct
iteratively $X$ out of $Y$ by only attaching cones over morphisms with
domain in $\mathcal{F}$.  The weighted octahedral axiom implies that
this $d^{\mathcal{F}}$ satisfies the \pbnote{triangle} inequality.
Using such \jznote{pseudo-metrics} one can analyze rigidity properties
of various categories by exploring the induced topology on
$\mathrm{Obj}(\mathcal{D})$.
 
 \
  
 The main algebraic part of the paper is contained in Chapter
 \ref{chap:Alg}
and its aim is to use persistence machinery
to produce certain non-flat triangular weights. 
The main tool, as already mentioned above, is a  refinement of triangulated
categories, called {\em triangulated persistence categories} (TPC).  A
triangulated persistence category, $\mathcal{C}$, has two main
properties. First, it is a persistence category, a natural notion we
introduce in \S\ref{subsec:pers-cat}. This is a category $\mathcal{C}$
whose morphisms $\Mor_{\mathcal{C}}(A,B)$ are persistence modules and
the composition of morphisms is compatible with the persistence
structure. (\jznote{See \cite{PRSZ20} for a general introduction of
  the persistence module theory}). The second main
structural property of TPC's is that the objects of $\mathcal{C}$
together with the \zhnote{$0$-persistence} level morphisms
$\Mor^{0}_{\mathcal{C}}(A,B)$ have the structure of a triangulated
category $\C_{0}$.  The formal definition of TPC's is given
in~\S\ref{subsec:TPC}.  

It is natural to associate to a persistence category $\mathcal{C}$ a
limit category $\mathcal{C}_{\infty}$ that has the same objects as
$\C$ and has as morphisms the $\infty$-limits of the morphisms in
$\C$.  In a different direction a natural notion in a persistence
category is that of an \zhnote{$r$-acyclic object}: $K$ is called
$r$-acyclic if its identity morphism $\id_{K}\in \Mor^{0}_{\C}(K,K)$
is $0$ in $\Mor^{r}_{C}(K,K)$. The acyclic objects for all $r\geq 0$
form a full subcategory $\mathcal{A}\C$ of $\C$ that is also a
persistence category, \zhnote{and in} case $\C$ is a TPC, it is easy
to see that $\mathcal{A}\C$ is also a TPC. In particular,
$\mathcal{A}\C_{0}$ is triangulated.

These notions are tied together by the classical construction of
Verdier localization. Indeed, assuming as above that $\C$ is a TPC,
\zhnote{we will} see that $\C_{\infty}$ coincides with the Verdier
localization of $\C_{0}$ with respect to $\mathcal{A}\C_{0}$.  In
particular, the category $\C_{\infty}$ is also triangulated.

We now can state the main result of the algebraic part of the paper
(restated more precisely in Theorem \ref{thm:main-alg}).
\begin{mainthm} If $\C$ is a triangulated persistence category, and
  with the notation above, the Verdier localization $\C_{\infty}$
  admits a non-flat triangular weight induced from the persistence
  structure of $\mathcal{C}$.
\end{mainthm}
The construction of this triangular weight is based on a definition of
a class of weighted triangles in the category $\mathcal{C}$
itself. With this definition, the exact triangles in $\mathcal{C}_{0}$
have weight $0$, but there are also other triangles in $\C$ of
arbitrary positive weights.  While the category $\C$ together with the
class of finite weight triangles is not triangulated - even the formal
expression of these triangles in $\C$ does not fit the axioms of
triangulated categories - the properties of these triangles are
sufficient to induce a triangular weight on the exact triangles of
$\C_{\infty}$.

\

In summary, if a triangulated category $\mathcal{D}$ admits a TPC
refinement - that is a TPC, $\C$, such that $\C_{\infty}=\mathcal{D}$
(as triangulated categories), then $\mathcal{D}$ carries a non-flat
triangular weight induced from the persistence structure of $\C$.  As
a result, this construction provides a technique to build non-discrete
fragmentation pseudo-metrics on the objects of $\mathcal{D}$. 

\

Some classes of examples are discussed in
\S\ref{sec-example}. Triangulated persistence categories are expected
to be of use beyond the field of symplectic topology and Chapter
\ref{chap:Alg}, which is essentially self-contained, can be
\zhnote{read independently} of the symplectic considerations that
appear in Chapter \ref{s:tpfc}.

\begin{rem} \label{rem:genLS} Even the fragmentation pseudo-metrics associated to the
  flat weight are of interest.  Many qualitative questions concerned
  with numerical lower bounds for the complexity of certain geometric
  objects can be understood by means of inequalities involving such
  fragmentation pseudo-metrics. \zhnote{Classical examples are the
    Morse inequalities, the Lusternik-Schnirelmann inequality as well
    as, in symplectic topology, the inequalities predicted by the
    Arnold conjectures.} Remarkable results based on measurements
  using this flat weight and applied to the study of endofunctors have
  appeared recently in work of Orlov \cite{Or} as well as
  Dimitrov-Haiden-Katzarkov-Kontsevich \cite{Kon} and
  Fan-Filip~\cite{Fa-Si}.
\end{rem}

\section{TPC refinements of the Fukaya derived
  category.}\label{subsec:intro2}

\zhnote{Here is an overview of the geometric part of the paper
  (Chapter \ref{s:tpfc}).} The main step \zhnote{here} is to consider
a finite family $\mathcal{X}$ of closed, exact Lagrangians in
\zhnote{a symplectic manifold} $X$, assumed in \zhnote{general}
position, and construct a TPC refinement of the derived Fukaya
category \pbred{of} $\mathcal{X}$.

\zhnote{There are quite a few nuances here.} First, this requires the
construction of a filtered Fukaya type category with objects the
elements of $\mathcal{X}$, endowed with all possible primitives. A
{\em weakly} filtered such category has been constructed in
\cite{Bi-Co-Sh:LagrSh} but obtaining a \zhnote{genuinely} filtered
$A_{\infty}$-structure is more delicate. It requires careful control
of energy estimates (and the technique we use restricts us to
\zhnote{finite families} $\mathcal{X}$) but also the use of
``cluster'' type moduli spaces, that mix $J$-holomorphic polygons and
Morse trajectories. Fortunately, such moduli spaces have been studied
and used frequently since \cite{Co-La}, for instance in
\cite{Charest-thesis},\cite{Charest}.

The resulting filtered Fukaya category $\fuk(\mathcal{X})$ depends, of
course, on choices of \zhnote{auxiliary structures such as}
perturbation data that we omit from the notation here.  The next step
is to pursue the construction of the derived version. As in the
non-filtered version, this part is purely algebraic and applies to any
filtered $A_{\infty}$-category. Nonetheless, there are some
significant differences with respect to the non-filtered
case. Uniqueness up to equivalence is considerably more delicate to
achieve because several \zhnote{basic algebraic $A_{\infty}$-tools},
such as the Hochschild complex and related constructions, require
significant adjustment to adapt to the filtered setting. Moreover, at
a more conceptual level, the two natural constructions of the derived
category, one based on filtered twisted complexes and the other on the
Yoneda embedding and $A_{\infty}$ filtered modules, both lead to
useful natural notions, but not to equivalent ones. Denote by
$\C\fuk(\mathcal{X})$ the version based on filtered modules.  Let
$D\fuk(\mathcal{X})$ be the usual, \zhnote{unfiltered,} derived
\zhnote{Fukaya category} \pbred{of} $\mathcal{X}$ and assume that
$\mathcal{F}\subset \mathcal{X}$ is a family of triangular generators
for $D\fuk(\mathcal{X})$.  Fix also a second such family
$\mathcal{F}'$ with each element a being a small generic Hamiltonian
deformation of a corresponding elements in $\mathcal{F}$.

The main statement is the following - again in simplified form (the
full statement is in Theorem \ref{thm:appl-sympl}):
\begin{mainthm}\label{mthm2}
  The category $\C\fuk(\mathcal{X})$ is a TPC and it is independent of
  the defining data up to TPC equivalence. Moreover,
  $\C\fuk(\mathcal{X})_{\infty}$ is triangulated equivalent to
  $D\fuk(\mathcal{X})$. Finally, there exists a fragmentation metric
  on $\mathcal{X}$, that is independent of the choices used in the
  construction of $\C\fuk(\mathcal{X})$, and is defined by
$$D^{\mathcal{F},\mathcal{F}'}=\max \{D^{\mathcal{F}},D^{\mathcal{F}'}\}$$
where $D^{\mathcal{F}}$ are \pbred{the shift-invariant versions of
  the} fragmentation pseudo-metrics \pbred{$d^{\mathcal{F}}$}
constructed as outlined in \S\ref{subsec:intro1}.
\end{mainthm}
One delicate point worth emphasizing here is that while we expect
$\mathcal{C}\fuk(\mathcal{X})$ to be unique up to {\em canonical}
equivalence, the machinery in this paper does not produce fully
canonical equivalences (see Theorem \ref{t:fil-fuk}).

Of course, as the set $\mathcal{X}$ is finite, this metric
$D^{\mathcal{F},\mathcal{F}'}$ might appear to be uninteresting,
however the more precise result - Theorem \ref{thm:appl-sympl} - shows
that the pseudo-metrics $D^{\mathcal{F}}$, satisfy some remarkable
properties (see also Remark \ref{rem:symplcom}).  These properties are
then used to analyze how the pseudo-metrics change when the family
$\mathcal{X}$ increases. Ultimately, this leads to the definition of
the metric on the space of all \zhnote{closed exact} Lagrangians that
was claimed earlier in the \zhnote{introduction. This is} stated more
precisely in Corollary \ref{cor:gl-metric}.

The construction of TPC's is inspired by recent constructions in
symplectic topology and, in particular, by the shadow pseudo-metrics
introduced in \cite{Bi-Co-Sh:LagrSh} and \cite{Bi-Co:LagPict} in the
study of Lagrangian cobordism. This aspect is discussed in
\S\ref{sb:lcob}.  The construction of the filtered Fukaya category and
the associated TPC are expected to be of independent interest.

\

\noindent {\bf Acknowledgments}. The third author is grateful to Mike
Usher for useful discussions. We thank Leonid Polterovich for
mentioning to us the work of Fan-Filip~\cite{Fa-Si}. We thank the
referee of the first version of this paper for pointing out the
relation between our algebraic construction and \zhnote{the} Verdier
localization. \pbrev{We are also grateful to a second referee for a
  very careful and critical reading of a later version of the paper,
  for pointing out several inaccuracies, and making very useful
  suggestions that helped to improve the exposition.} We thank a third referee for pointing out the paper \cite{Sco20} (see also Remark \ref{rem:ear}).

Part of this work was completed while the third author held a CRM-ISM
Postdoctoral Research Fellowship at the {\em Centre de recherches
  math\'ematiques} in Montr\'eal. He thanks this Institute for its
warm hospitality. The third author is partially supported by National
Key R\&D Program of China No.~2023YFA1010500, NSFC No.~12301081, NSFC
No.~12361141812, and USTC Research Funds of the Double First-Class
Initiative.


\chapter{Triangulation Persistence Categories: Algebra 101}\label{chap:Alg}
This chapter contains the main algebraic machinery introduced in the
paper and it is self-contained, except for some basic elements of
homological algebra as can be found in \cite{Weibel}\footnote{A
  version of this chapter appeared earlier as an independent preprint
  \cite{Bi-Co-Zh:101}. The only changes compared to
  \cite{Bi-Co-Zh:101}, besides minor corrections of imprecisions,
  concern the relations to Verdier localization in
  \S\ref{subsubsec:localization}.}.

In \S\ref{subsec:triang-gen} we introduce briefly the notion of
triangular weight and discuss its application to measure the
complexity of cone-decompositions in triangulated categories. In
\S\ref{subsec:pers-cat} we introduce persistence categories which are,
in short, categories enriched by persistence modules. Triangulated
persistence categories are introduced in \S\ref{subsec:TPC}. In
\S\ref{subsec:trstr-weights} we prove the main algebraic result of the
\zhnote{chapter}, namely that the $\infty$-level of a TPC carries a
specific triangular weight induced from the persistence
structure. Finally, in \S\ref{sec-example} we discuss some classes of
natural TPC examples that are not symplectic in nature \pbred{(the
  symplectic examples are deferred to Chapter~\ref{s:tpfc})}.

\section{Triangular weights} \label{subsec:triang-gen} In this
subsection we introduce triangular weights associated to a
triangulated category $\mathcal{D}$.  Using such a triangular weight
$w$ on $\mathcal{D}$ we define a class of so-called fragmentation
pseudo-metrics $d^{\mathcal{F}}_{w}$ on
$\mathrm{Obj}(\mathcal{D})$. All categories used in this paper
($\mathcal{D}$ in particular) are assumed to be small unless otherwise
indicated.


\begin{dfn}\label{def:triang-cat-w} Let $\mathcal{D}$ be a 
  triangulated category and denote by $\mathcal{T}_{\mathcal{D}}$ its
  class of exact triangles. A {\em triangular weight} $w$ on
  $\mathcal{D}$ is a function
  $$w:\mathcal{T}_{\mathcal{D}}\to [0,\infty)$$
  that satisfies properties (i) and (ii) below:

  (i) [Weighted octahedral axiom] Assume that the triangles
  $\Delta_{1}: A\to B\to C\to TA$ and $\Delta_{2}: C\to D\to E\to TC$
  are both exact. There are exact triangles:
  $\Delta_{3}: B\to D\to F\to TB$ and
  $\Delta_{4}: TA\to F\to E\to T^{2}A$ making the diagram below
  commute, except for the right-most bottom square that anti-commutes,
  \[ \xymatrix{
      A \ar[r] \ar[d] & 0 \ar[r] \ar[d] & TA \ar[d]\ar[r]& TA\ar[d] \\
      B \ar[r] \ar[d] & D \ar[r] \ar[d] & F \ar[d]\ar[r]& TB \ar[d]\\
      C \ar[r] \ar[d] & D \ar[r]\ar[d] & E\ar[r]\ar[d] & TC\ar[d] \\
      TA\ar[r] & 0 \ar[r]  & T^{2}A \ar[r]& T^{2} A }
  \]
  and such that 
  \begin{equation}\label{eq:weight:ineq}
    w(\Delta_{3})+w(\Delta_{4})\leq w(\Delta_{1})+w(\Delta_{2})~.~
  \end{equation}

  (ii) [Normalization] There is some $w_{0}\in [0,\infty)$ such that
  \pbnote{$w(\Delta)\geq w_{0}$} for all
  \pbnote{$\Delta\in \mathcal{T}_{\mathcal{D}}$} and
  \pbnote{$w(\Delta')=w_{0}$} for all triangles \pbnote{$\Delta'$} of
  the form $0\to X\xrightarrow{\mathds{1}_{X}}X\to 0$,
  $X\in\mathrm{Obj}(\mathcal{D})$, and their rotations.  Moreover, in
  the diagram at (i) if $B=0$, we may take $\Delta_{3}$ to be
  \begin{equation}\label{eq:simpl-d3}
    \Delta_{3}: 0 \to D\to D\to 0~.~
  \end{equation}
  
\end{dfn}

\begin{remark}\label{rem:gen-weights} (a) Neglecting the weights
  constraints, given the triangles $\Delta_{1}$, $\Delta_{2}$,
  $\Delta_{3}$ as at point (i), the octahedral axiom is easily seen to
  imply the existence of $\Delta_{4}$ making the diagram commutative,
  as in the definition.

  (b) The condition at point (ii), above equation (\ref{eq:simpl-d3}),
  can be reformulated as a replacement property for exact triangles in
  the following sense: if $\Delta_{2}:C\to D\to E\to TC$ is exact and
  $C$ is isomorphic to $A'\ (=TA)$, then there is an exact triangle
  $A'\to D\to E\to TA'$ of weight at most
  $w(\Delta_{2})+w(\Delta_{1})-w_{0}$ where $\Delta_{1}$ is the exact
  triangle $T^{-1}A'\to 0 \to C\to A'$.
\end{remark}

Given an exact triangle $\Delta : A\to B\xrightarrow{f} C\to TA$ in
$\mathcal{D}$ and any $X\in \mathrm{Obj}(\mathcal{D})$ there is an
associated exact triangle
$ X\oplus\Delta :A\to X\oplus B\xrightarrow{\mathds{1}_{X}\oplus f}
X\oplus C\to TA$ and a similar one, $\Delta \oplus X$.  We say that a
triangular weight $w$ on $\mathcal{D}$ is {\em subadditive} if for any
exact triangle $\Delta \in \mathcal{T}_{\mathcal{D}}$ and any object
$X$ of $\mathcal{D}$ we have $$w(X\oplus\Delta ) \leq w(\Delta )$$ and
similarly for $\Delta \oplus X$.

\

The simplest example of a triangular weight on a triangulated category
$\mathcal{D}$ is the flat one, \pbnote{$w_{fl}(\Delta)=1$,} for all
triangles \pbnote{$\Delta\in \mathcal{T}_{\mathcal{D}}$.} This weight
is obviously sub-additive. A weight that is not proportional to the
flat one is called {\em non-flat}.

\

The interest of triangular weights comes from the next definition that
provides a measure for the complexity of cone-decompositions in
$\mathcal{D}$ and this leads in turn to the definition of
corresponding pseudo-metrics on the set $\mathrm{Obj}(\mathcal{D})$.

\begin{dfn}\label{def:iterated-coneD-tr} Fix a triangulated category
  $\mathcal{D}$ together with a triangular weight $w$ on
  $\mathcal{D}$.  Let $X$ be an object of $\mathcal{D}$. An {\em
    iterated cone decomposition} $D$ of $X$ with {\em linearization}
  $\ell(D) = (X_1,X_{2}, ..., X_n)$ consists of a family of exact
  triangles in $\mathcal{D}$:
\[ 
\left\{ 
\begin{array}{ll}
\Delta_{1}: \, \, & X_{1}\to 0\to Y_{1}\to  TX_{1}\\
\Delta_2: \,\, & X_2 \to Y_1 \to Y_2 \to  TX_2\\
\Delta_3: \,\, & X_3 \to Y_2 \to Y_3 \to  TX_3\\
&\,\,\,\,\,\,\,\,\,\,\,\,\,\,\,\,\,\,\,\,\vdots\\
\Delta_n: \,\, & X_n \to Y_{n-1} \to X \to  TX_n
\end{array} \right.\]
To accommodate the case $n=1$ we set $Y_0=0$.
The weight of such a cone decomposition is defined by:
\begin{equation}\label{eq:weight-cone} 
w(D)=\sum_{i=1}^{n}w(\Delta_{i})-w_{0}~.~
\end{equation}
\end{dfn}

This weight of cone-decompositions naturally leads to a class of
pseudo-metrics on the objects of $\mathcal{D}$, as follows.

\

Let $\mathcal F \subset {\rm Obj}(\mathcal{D})$.  For two objects
$X, X'$ of $\mathcal{D}$, define
\begin{equation} \label{frag-met-0} \delta^{\mathcal F} (X, X') =
  \inf\left\{ w(D) \, \Bigg| \,
    \begin{array}{ll} \mbox{$D$ is an
      iterated cone decomposition} \\
      \mbox{of $X$ with linearization}
      \ \mbox{$(F_1, ..., T^{-1}X', ..., F_k)$,}\\
      \mbox{where $F_i \in \mathcal{F}$, $k \geq 0$}
    \end{array}
  \right\}.
\end{equation}
\pbnote{Note that we allow here $k=0$, i.e.~the linearization of $D$
  is allowed to consist of only one element, $T^{-1}X'$, without using
  any elements $F_i$ from the family $\mathcal{F}$.} Fragmentation
pseudo-metrics are obtained by symmetrizing $\delta^{\mathcal{F}}$, as
below.
 
\begin{prop}\label{prop:tr-weights-gen}
  Let $\mathcal{D}$ be a triangulated category and let $w$ be a
  triangular weight on $\mathcal{D}$.  Fix $\F \subset {\rm Obj}(\C)$
  and define
  $$d^{\F}: {\rm Obj}(\C)\times {\rm Obj}(\C)\to [0,\infty)\cup \{ +\infty\}$$ 
  by:
  \[ d^{\F}(X, X') = \max\{\delta^{\F}(X, X'), \delta^{\F}(X',
    X)\}. \]
  \begin{itemize}
  \item[(i)] The \pbnote{map} $d^{\mathcal{F}}$ is a pseudo-metric
    called {\rm the fragmentation pseudo-metric} associated to $w$
    and $\mathcal{F}$.
  \item[(ii)] If $w$ is subadditive, then
    \begin{equation}\label{eq:subad}
      d^{\mathcal{F}}(A\oplus B, A'\oplus B')\leq d^{\mathcal{F}}(A,A') +
      d^{\mathcal{F}}(B,B') + w_0.
    \end{equation} 
    In particular, if $w_{0}=0$, then $\mathrm{Obj}(\mathcal{D})$ with
    the operation given by $\oplus$ and the topology induced by
    $d^{\mathcal{F}}$ is an H-space. (Recall that a topological space
    is called an {\it H-space} if there exists a continuous map
    $\mu: X \times X \to X$ with an identity element $e$ such that
    $\mu(e,x) = \mu(x,e) = x$ for any $x \in X$.)
  \end{itemize}
\end{prop}

The proof of Proposition \ref{prop:tr-weights-gen} is based on simple
manipulations with exact triangles.  We will prove a similar statement
in \S\ref{subsubsec:frag1} in a more complicated setting and we will
then briefly discuss in \S\ref{subsubsec:proof-prop} how the arguments
given in that case also imply \ref{prop:tr-weights-gen}.

\begin{rem}\label{rem:finite-metr}
  (a) In case $\mathcal{F}$ is invariant by translation in the sense
  that $T\mathcal{F}\subset \mathcal{F}$ and, moreover, $\mathcal{F}$
  is a family of triangular generators for $\mathcal{C}$, then the
  metric $d^{\mathcal{F}}$ admits finite values. This is not difficult
  to show by first proving that $\delta^{\mathcal{F}}(0, X)$ is finite
  for all $X\in \mathrm{Obj}(\C)$ (it is immediate that
  $\delta^{\mathcal{F}}(X,0)$ is finite).
 
  (b) It is sometimes useful to view an iterated cone-decomposition as
  in Definition \ref{def:iterated-coneD-tr} as a sequence of objects
  and maps forming the successive triangles $\Delta_{i}$ below
 
  \begin{equation}\label{eq:iterated-tr}\xymatrixcolsep{1pc}  \xymatrix{
      Y_{0} \ar[rr] &  &  Y_{1}\ar@{-->}[ldd]  \ar[r] &\ldots  \ar[r]& Y_{i} \ar[rr] &  &  Y_{i+1}\ar@{-->}[ldd]  \ar[r] &\ldots \ar[r]&Y_{n-1} \ar[rr] &   &Y_{n} \ar@{-->}[ldd]  &\\
      &         \Delta_{1}                  &  & & &  \Delta_{i+1}                          & &  &  &    \Delta_{n}             \\
      & X_{1}\ar[luu] &  & & &X_{i+1}\ar[luu] &  &  & &X_{n}\ar[luu] }
  \end{equation}
\end{rem}
where the dotted arrows represent maps $Y_{i} \to TX_{i}$ and
$Y_{0}=0$, $Y_{n}=X$.

(c) The definition of fragmentation pseudo-metrics is quite flexible
and there are a number of possible variants.  One of them will be
useful later. Instead of $\delta^{\mathcal{F}}$ as given in
(\ref{frag-met-0}) we may use:
\begin{equation}\label{eq:frag-simpl}
  \underline{\delta}^{\mathcal{F}}(X,X') =
  \inf\left\{ \sum_{i=1}^{n}w(\Delta_{i}) \, \Bigg| \,
    \begin{array}{ll}
      \mbox{$\Delta_{i}$ are successive exact triangles as in
      (\ref{eq:iterated-tr})} \\ \mbox{with \zjnote{$Y_{1}=X'$},
      $X=Y_{n}$  and  $X_i \in \mathcal{F}$,   $n \in \N$}
    \end{array} \right\}~.~
\end{equation}

For this to be coherent we need to assume here $0\in \mathcal{F}$.
Comparing with the definition of $\delta^{\mathcal{F}}$ in
(\ref{frag-met-0}), $\underline{\delta}^{\mathcal{F}}$ corresponds to
only taking into account cone decompositions with linearization
$(T^{-1}X', F_{1},\ldots , F_{n})$ and with the first triangle
\zjnote{$\Delta_{1}: T^{-1}X'\to 0 \to X' \to X'$}. There are two advantages of
this expression: the first is that it is trivial to see in this case
that $\underline{\delta}^{\mathcal{F}}$ satisfies the triangle
inequality, \zjr{which does} not even require the weighted octahedral
axiom. The other advantage is that one starts the sequence of
triangles from $X'$ and thus the negative translate $T^{-1}X'$ is not
needed to define $\underline{\delta}^{\mathcal{F}}$.  There is an
associated fragmentation pseudo-metric $\underline{d}^{\mathcal{F}}$
obtained by symmetrizing $\underline{\delta}^{\mathcal{F}}$ and this
satisfies a formula similar to (\ref{eq:subad}). Of course, the
disadvantage of this fragmentation pseudo-metric is that it is larger
than $d^{\mathcal{F}}$ and thus more often infinite.

\section{Persistence categories} \label{subsec:pers-cat} We introduce
in this section the notion of persistence category - a category whose
morphisms are persistence modules and such that composition respects
the persistence structure - and then pursue with a number of related
structures and immediate properties.

\subsection{Basic definitions}\label{subsubsec:pers-cat1}

View the real axis $\R$ as a category with
${\rm Obj}(\R) = \{x \,| \, x \in \R\}$ and for any
$x, y \in {\rm Obj}(\R)$, the hom-set
\[ \hom_{\R}(x,y) = \left\{ \begin{array}{lcl} i_{x,y} &
      \mbox{if} & x \leq y \\ \emptyset & \mbox{if} & x>y \end{array}
  \right..\] By definition, for any $x \leq y \leq z$ in $\R$,
$i_{y,z} \circ i_{x,y} = i_{x,z}$. We denote this category by
$(\R, \leq)$. It admits an additive structure. Explicitly, consider
the bifunctor $\oplus: (\R, \leq) \times (\R, \leq) \to (\R, \leq)$
defined by $\oplus(r,s) := r+s$, where $0 \in \R$ is the \pbnote{zero}
object and for any two pairs
$(r,s), (r',s') \in {\rm Obj}((\R, \leq) \times (\R, \leq))$,
\[ \hom_{(\R, \leq) \times (\R, \leq)}((r,s), (r',s')) =
  \left\{ \begin{array}{lcl} (i_{r,r'}, i_{s,s'}) & \mbox{if} & r \leq
      r' \,\,\mbox{and} \,\,s \leq s' \\ \emptyset & \mbox{if} &
      \mbox{otherwise} \end{array} \right.\] and further
$\oplus(i_{r,r'}, i_{s,s'}) : = i_{r+s, r'+s'} \in \zjnote{\hom}_{(\R,
  \leq)}(r+s, r'+s')$.  Fix a ground field $\k$ and denote by
${\rm Vect}_{\k}$ the category of $\k$-vector spaces.

\begin{dfn} \label{dfn-pc} A category $\C$ is called a {\em
    persistence category} if it is endowed with the following
  additional structure. For any $A, B \in {\rm Obj}(\C)$ we are given
  a functor $E_{A,B}: (\R, \leq) \to {\rm Vect}_{\k}$ such that the
  following two conditions are satisfied:
  \begin{itemize}
  \item[(i)] The hom-set in $\C$ is
    ${\hom}_{\C}(A,B) = \{(f,r) \,| \, f \in E_{A,B}(r)\}$. We
    denote ${\hom}_{\C}^r(A,B): = E_{A,B}(r)$, \jznote{or simply
      ${\hom}^r(A,B)$ when the ambient category $\C$ is not
      emphasized}.
  \item[(ii)] The composition
    $\circ: {\hom}_{\C}^r(A,B) \times {\hom}_{\C}^s(B,C) \to
    {\hom}_{\C}^{r+s}(A,C)$ in $\C$ is a natural transformation
    from $E_{A,B} \times E_{B,C}$ to $E_{A,C} \circ \oplus$ (with
    $\oplus$ the product $(\R, \leq) \times (\R, \leq)\to (\R,\leq)$). \zjnote{Explicitly, the following diagram  \begin{equation} \label{comp-nt} \xymatrix{ \Mor_{\C}^r(A,B) \times
      \Mor_{\C}^s(B,C) \ar[r]^-{\circ_{(r,s)}} \ar[d]_-{E_{A,B}(i_{r,r'})
        \times E_{B,C}(i_{s,s'})} & \Mor_{\C}^{r+s}(A,C)
      \ar[d]^-{E_{A,C}(i_{r+s, r'+s'})} \\
      \Mor_{\C}^{r'}(A,B) \times \Mor_{\C}^{s'}(B,C) \ar[r]^-{\circ_{(r',s')}} &
      \Mor_{\C}^{r'+s'}(A,C)}
  \end{equation}
  commutes.}
  \end{itemize}
\end{dfn}

\begin{remark} Item (i) means that each hom-set ${\hom}_{\C}(A,B)$
  is a persistence $\k$-module with persistence structure morphisms
  $E_{A,B}(i_{r,s})$ for any $r \leq s$ in $\R$. \jznote{Here, we use
    the weakest possible definition of a persistence $\k$-module in
    the sense that no regularities, such as the finiteness of the
    dimension of ${\hom}_{\C}^r(A,B)$ or the semi-continuity when
    changing the parameter $r$, are required (see subsection 1.1 in
    \cite{PRSZ20})}. \end{remark}

We will often denote an element in ${\hom}_{\C}(A,B)$, by a single
symbol $\bar{f}$ instead of a pair $(f,r)$. We will use the notation
$\ceil*{\bar{f}~} = r$ to denote the real \pbnote{number} $r$ and
refer to this number as the {\em shift} (or persistence level) of
$\bar{f}$. For each $A \in {\rm Obj}(\C)$ the identity
$\bar{\mathds{1}}_A := (\mathds{1}_A, 0) \in \Mor_{\C}^0(A,A)$ is of
shift $0$. \zjnote{If one of the objects $A$ or $B$ is the zero object, then ${\hom}_{\C}(A,B)$ contains only the zero morphism denoted by $0$, and it lies in ${\hom}_{\C}^r(A,B)$ for any $r \in \R$.} For brevity, we will denote from now on the structural
morphisms $E_{A,B}(i_{r,s})$ by $i_{r,s}$. 

\

A persistence structure allows us to consider morphism\zjnote{s} that are identified up to $r$-shift and, similarly, objects that are negligible up to a shift by $r$. 

\begin{dfn} \label{def:acyclics} Fix a persistence category $\C$.
\begin{itemize}
\item[(i)]  For $f,g \in \Mor^{\alpha}_{\C}(A,B)$, we say that $f$ and $g$ are {\em  $r$-equivalent} for some $r \geq0$ if $$i_{\alpha, \alpha+r}(f-g)=0~.~$$ We write $f \simeq_r g$ if $f$ and $g$ are $r$-equivalent. 
\item[(ii)] Two morphisms, $f \in \Mor_{\C}^\alpha(A,B)$ and $g \in \Mor_{\C}^{\beta}(A,B)$,  are {\em $\infty$-equivalent}, written $f\simeq_{\infty} g$,  if there exist $r, r' \geq 0$ with $\alpha+r = \beta+r'$ such that 
$i_{\alpha, \alpha+r}(f) = i_{\beta, \beta+r'}(g)$.
 \item[(iii)] An object $K \in {\rm Obj}(\C)$ is called {\em  $r$-acyclic} for some $r \geq 0$ 
if  $\mathds{1}_K\in \Mor_{\C}^{0}(K,K)$ has the property that $\mathds{1}_K\simeq_{r} 0$.
\end{itemize}
\end{dfn}
Obviously, if $f\simeq_{r} g$ then $f\simeq_{s} g$ for all $s\geq
r$. Notice also that $\simeq_{r}$ is indeed an equivalence
relation. Indeed, for $r\not=\infty$ this follows immediately from the
fact that
$i_{\alpha,\beta}: \Mor_{\C}^{\alpha}(A,B)\to \Mor_{\C}^{\beta}(A,B)$
is a linear map and it is an easy exercise for $r=\infty$.

\

\begin{dfn} \label{dfn-c0-cinf} Given a persistence category $\C$,
  there are two categories naturally associated to it as follows:
\begin{itemize}
\item[(i)] the $0$-level of $\C$, denoted $\C_0$, which is the category with the same objects as $\C$ and, for any $A, B \in {\rm Obj}(\C)$, with $\Mor_{\C_0}(A,B) := \Mor_{\C}^0(A,B)$. 
\item[(ii)] the limit category (or $\infty$-level) of $\C$, denoted
  $\C_{\infty}$, that again has the same objects as $\C$ but for any
  $A,B \in {\rm Obj}(\C)$,
  $\Mor_{\C_{\infty}}(A,B) := \varinjlim_{\alpha \to \infty}
  \Mor^\alpha_{\C}(A,B)$, where the direct limit is taken with respect
  to the morphisms
  $i_{\alpha,\beta}: \Mor_{\C}^\alpha(A,B) \to \Mor_{\C}^\beta(A,B)$ for any
  $\alpha \leq \beta$.
\end{itemize}
\end{dfn}


\begin{rem}\label{rem:gen-pers}
  (a) In general, a persistence category is not pre-additive as the
  hom-sets ${\hom}_{\C}(A,B)$ are generally not abelian
  groups. However, it is easy to see that both $\mathcal{C}_{0}$ and
  $\C_{\infty}$ are pre-additive (the proof is immediate in the first
  case and a simple exercise in the second).

  (b) The limit category $\C_{\infty}$ can be equivalently defined as
  a quotient category $\C/ \simeq_{\infty}$ which is defined by
  ${\rm Obj}(\C/ \simeq_{\infty}) = {\rm Obj}(\C)$ and
  ${\hom}_{\C/ \simeq_{\infty}}(A,B) ={\hom}_{\C}(A,B)/
  \simeq_{\infty}$.
\end{rem}
 
Two objects $A, B \in {\rm Obj}(\C)$ are said {\em $0$-isomorphic}, we
write $ A \equiv B$, if they are isomorphic in the category
$\C_{0}$. This is obviously an equivalence relation and it preserves
$r$-acyclics in the sense that if $K \simeq_r 0$ and $K \equiv K'$,
then $K' \simeq_r 0$.

\subsection{Persistence functors}
Persistence categories come with associated notions of persistence
functors and natural transformations relating them, as described
below.

\begin{dfn} \label{dfn-per-functor} Given two persistence categories
  $\C$ and $\C'$, a {\em persistence functor $\F: \C \to \C'$} is a
  functor which is compatible with the persistence structures.
  \pbnote{More explicitly, the action of $\mathcal{F}$ on morphisms
    restricts to maps
    $(\mathcal{F}_{A,B})_r: \Mor_{\mathcal C}^{r}(A,B) \to
    \Mor_{\mathcal C'}^{r}(\F(A),\F(B))$ defined for any
    $A, B \in {\rm Obj}(\C)$ and $r \in \mathbb{R}$. Moreover, for
    every $r \leq s$ we have the following commutative diagram:}
  \begin{equation} \label{functor-com} \xymatrixcolsep{4pc} \xymatrix{
      \Mor_{\mathcal C}^{r}(A,B) \ar[r]^-{(\F_{A,B})_r}
      \ar[d]_-{i^{\C}_{r,s}}
      & \Mor_{\mathcal C'}^{r}(\F(A),\F(B)) \ar[d]^-{i^{\C'}_{r, s}}\\
      \Mor_{\mathcal C}^{s}(A,B) \ar[r]^-{(\F_{A,B})_s} &
      \Mor_{\mathcal C'}^s(\F(A),\F(B))}
  \end{equation}
  where $i_{r, s}^{\mathcal C}$ and $i_{r, s}^{\mathcal C'}$ are
  persistence structure maps in $\mathcal C$ and $\mathcal C'$,
  respectively. In particular, for each
  $\bar{f} \in {\hom}_{\C}(A,B)$ with $\ceil*{\bar{f}~} = r$, we
  have $\ceil*{\F_{A,B}(\bar{f}~)} = r$.
\end{dfn}



For any functor $E: (\R, \leq) \to {\rm Vect}_{\k}$ and
$\alpha \in \R$, we denote by
$\Sigma^{\alpha} E: (\R, \leq) \to {\rm Vect}_{\k}$ the $\alpha$-shift
of $E$ defined by $\Sigma^{\alpha}E(r) = E(r+ \alpha)$ and
$\Sigma^{\alpha}E(i_{r,s}) = E(i_{r + \alpha, s+\alpha})$ for any
$i_{r,s}: r \to s$, $r \leq s$.

\begin{dfn} \label{dfn-nt} Given two persistence functors between two
  persistence categories $\F, \G: \C \to \C'$, a {\em persistence
    natural transformation} $\eta: \F \to \G$ is a natural
  transformation for which there exists $r \in \mathbb{R}$ such that
  for any $A \in {\rm Obj}(\C)$, the morphism
  $\eta_A: \F(A) \to \G(A)$ belongs to $\Mor_{\C'}^r(\F(A), \G(A))$. We say
  that $\eta$ is a natural transformation of shift $r$.
%
%
\end{dfn}

\begin{remark} \label{rmk-pf} (a) \pbnote{The morphisms
    $\eta_A: \F(A) \to \G(A)$, $A \in \text{Obj}(\C)$, give rise to
    the following commutative diagrams for all
    $X \in \text{Obj}(\C)$} \zjnote{and for any $\alpha \leq \beta \in \R$:}
  \begin{equation} \label{nt-dia-2} \xymatrixcolsep{3pc} \xymatrix{
      {\hom}_{\C'}^{\alpha}(X, \F(A)) \ar[r]^-{\eta_A \circ}
      \ar[d]_-{i_{\alpha, \beta}} & {\hom}_{\C'}^{\alpha+r}(X,
      \G(A))
      \ar[d]^-{i_{\alpha+r, \beta+r}} \\
      {\hom}_{\C'}^{\beta}(X, \F(A)) \ar[r]^-{\eta_A \circ} & {\rm
        Mor}_{\C'}^{\beta+r}(X, \G(A))}
  \end{equation}
  and
  \begin{equation} \label{nt-dia-3} \xymatrixcolsep{3pc} \xymatrix{
      {\hom}_{\C'}^{\alpha}(\G(A), X) \ar[r]^-{\circ \eta_A}
      \ar[d]_-{i_{\alpha, \beta}}
      & {\hom}_{\C'}^{\alpha+r}(\F(A), X) \ar[d]^-{i_{\alpha+r, \beta+r}} \\
      {\hom}_{\C'}^{\beta}(\G(A), X) \ar[r]^-{\circ \eta_A} & {\hom}_{\C'}^{\beta+r}(\F(A), X)}
  \end{equation}

  (b) Given two persistence categories $\C, \C'$, the persistence
  functors themselves form a persistence category denoted by
  $\mathcal{P}{\rm Fun}(\C, \C')$, where
  \[ {\hom}_{\mathcal{P}{\rm Fun}(\C, \C')}(\F, \G) =
    \left\{\left(\eta, r \right) \,\bigg|
      \, \begin{array}{l} \mbox{$\eta$ is a natural transformation} \\
           \mbox{from $\F$ to $\G$ of shift $r$}
         \end{array}
       \right\}.\] When $\C = \C'$, simply denote
     $\mathcal{P}{\rm Fun}(\C, \C')$ by $\mathcal{P}{\rm End}(\C)$. It
     is easy to verify that $\mathcal{P}{\rm End}(\C)$ admits a strict
     monoidal structure.
\end{remark}

\zhnote{\begin{dfn}\label{def:eq-pers} Let
$\mathcal{C}', \mathcal{C}''$ be two persistence categories. A
persistence functor
$\mathcal{F}: \mathcal{C}' \longrightarrow \mathcal{C}''$ is called an
\zjr{equivalence} of persistence categories (or persistence equivalence) if
there exists a persistence functor
$\mathcal{G}: \mathcal{C}'' \longrightarrow \mathcal{C}'$ such that
$\mathcal{G} \circ \mathcal{F}$ is isomorphic to \zhnote{$\mathds{1}_{\mathcal{C}'}$}
via a persistence natural transformations of shift $0$, whose inverse
also has shift $0$, and the analogous condition holds for
$\mathcal{F} \circ \mathcal{G}$ too. We will say that $\mathcal{C}'$
and $\mathcal{C}''$ are persistence equivalent or equivalent as
persistence categories.
\end{dfn}}

\zhnote{Standard arguments show that a persistence functor
$\mathcal{F}:\mathcal{C}' \longrightarrow \mathcal{C}''$ is an
equivalence of persistence categories if and only if it is full and
faithful (in the obvious persistence sense) and for every object
$Y \in \Ob(\mathcal{C}'')$ there exists $X \in \Ob(\mathcal{C}')$ such
that $Y$ is $0$-isomorphic to $\mathcal{F}(X)$ (i.e.~the latter two
objects are isomorphic in the $0$-level subcategory $\mathcal{C}_0''$
of $\mathcal{C}''$).}

\subsection{Shift functors}
The role of shift functors, to be introduced below, is to allow
morphisms of arbitrary shift (as well as $r$-equivalences) to be
represented as morphisms of shift $0$ with the price of ``shifting''
the domain (or the target) - see Remark \ref{rem:c0repl}.  This turns
out to be very helpful in the study of triangulation for persistence
categories.

\

View the real axis $\R$ as a strict monoidal category $(\R, +)$
induced by the additive group structure of $\R$.  In other words,
${\rm Obj}(\R) = \{x \,| \, x \in \R\}$ and for any
$x, y \in {\rm Obj}(\R)$, ${\hom}_{\R}(x,y) = \{\eta_{x,y}\}$ such
that $\eta_{x,x} = \mathds{1}_x$ and, for any $x, y, z \in \R$,
$\eta_{y,z} \circ \eta_{x,y} = \eta_{x,z}$. In particular,
$\eta_{x,y} \circ \eta_{y,x} = \mathds{1}_y$ and
$\eta_{y,x} \circ \eta_{x,y} = \mathds{1}_x$, hence each morphism
$\eta_{x,y}$ is an isomorphism.
The monoidal structure is defined by $\oplus (x,y) := x+y$ on objects
and for any two morphisms $(\eta_{r,r'}, \eta_{s, s'})$ we have
$\oplus (\eta_{r,r'}, \eta_{s, s'}) := \eta_{r+s, r'+s'}$.

\begin{dfn} \label{dfn-shift-functor} Let $\C$ be a persistence
  category. A {\em shift functor} on $\C$ is a strict monoidal functor
  $\Sigma: (\R, +) \to \mathcal{P}{\rm End}(\C)$ such that
  $\Sigma(\eta_{x,y}): \Sigma(x) \to \Sigma(y)$ is a natural
  transformation of shift $\ceil*{\Sigma(\eta_{x,y})} = y-x$ for any
  $x, y \in \R$ and $\eta_{x,y} \in {\hom}_{\R}(x,y)$.
\end{dfn}

For later use, denote
$\Sigma^r := \Sigma(r) \in \mathcal{P}{\rm End}(\C)$ and, for brevity,
we denote $\Sigma(\eta_{r,s})$ by $\eta_{r,s}$ for $r, s\in \R$ and we
let $(\eta_{r,s})_{A}$ be the respective morphism
$\Sigma^rA \to \Sigma^{s}A$.

\begin{remark} \label{rem:shift} Since $\Sigma$ is a strict monoidal
  functor, it preserves the monoidal product. Therefore,
  $\Sigma^s \circ \Sigma^r = \Sigma^{r+s}$ and
  $\Sigma^{0} = \mathds{1}$. Moreover, since each $\eta_{r,s}$
  \pbnote{is an isomorphism in $(\mathbb{R}, +)$,} the corresponding
  natural transformation $\eta_{r,s}$ is a natural isomorphism. We
  also have $\Sigma^{r}(\eta_{s,s'})_{A}=(\eta_{s+r,s'+r})_{A}$ for
  each object $A$ in $\C$ and all $r,s,s'\in \R$.

  In particular, this implies that for any $Y, A \in {\rm Obj}(\C)$
  and $\alpha \in \R$, we have an isomorphism,
  \begin{equation} \label{nt-iso} \Mor^{\alpha}_{\C}(Y, A)
    \xrightarrow{(\eta_{0,r})_A\circ} \Mor^{\alpha + r}_{\C}(Y,
    \Sigma^r A).
\end{equation}
Similarly, for any $A, X \in {\rm Obj}(\C)$ and $\alpha \in \R$, we
have an isomorphism
\begin{equation} \label{nt-iso-2} \Mor^{\alpha-r}_{\C}(\Sigma^r A, X)
  \xrightarrow{\circ(\eta_{0,r})_A} \Mor^{\alpha}_{\C}(A, X).
\end{equation}
Further, for any $A, B \in {\rm Obj}(\C)$, the isomorphisms
(\ref{nt-iso}) and (\ref{nt-iso-2}) imply the existence of an
isomorphism
\begin{equation} \label{nt-iso-3} \Mor^{\alpha+s-r}_{\C}(\Sigma^r A,
  \Sigma^sB) \simeq \Mor^{\alpha}_{\C}(A, B).
\end{equation}
In particular, when $r=s$, we get a canonical isomorphism:
\begin{equation} \label{sigma-mor} \Sigma^r: \Mor^{\alpha}_{\C}(A, B)
  \to \Mor^{\alpha}_{\C}(\Sigma^rA, \Sigma^rB).
\end{equation}
Finally, diagrams (\ref{nt-dia-2}) and (\ref{nt-dia-3}) imply that the
following diagrams obtained by setting $\F = \Sigma^0$ and
$\G = \Sigma^r$
\begin{equation} \label{nt-dia-4} \xymatrixcolsep{3pc} \xymatrix{ {\hom}_{\C}^{\alpha}(X, A) \ar[r]^-{(\eta_{0,r})_A \circ}
    \ar[d]_-{i_{\alpha, \beta}} & {\hom}_{\C}^{\alpha+r}(X,
    \Sigma^r A)
    \ar[d]^-{i_{\alpha+r, \beta+r}} \\
    {\hom}_{\C}^{\beta}(X, A) \ar[r]^-{(\eta_{0,r})_A\circ} & {\rm
      Mor}_{\C}^{\beta+r}(X, \Sigma^r A)}
\end{equation}
and
\begin{equation} \label{nt-dia-5} \xymatrixcolsep{3pc} \xymatrix{ {\hom}_{\C}^{\alpha}(\Sigma^r A, X) \ar[r]^-{\circ (\eta_{0,r})_A}
    \ar[d]_-{i_{\alpha, \beta}} & {\hom}_{\C}^{\alpha+r}(A, X)
    \ar[d]^-{i_{\alpha+r, \beta+r}} \\
    {\hom}_{\C}^{\beta}(\Sigma^r A, X) \ar[r]^-{\circ
      (\eta_{0,r})_A} & {\hom}_{\C}^{\beta+r}(A, X)}
\end{equation}
are commutative for any $\alpha \leq \beta$. All the horizontal
morphisms in (\ref{nt-dia-4}) and (\ref{nt-dia-5}) are isomorphisms
but the vertical morphisms (which are the persistence structure
morphisms) are not necessarily so. \end{remark}



Assume that $\C$ is a persistence category (with persistence structure
morphisms denoted by $i_{r,s}$) endowed with a shift functor $\Sigma$.
To simplify \zjnote{the} notation, for $A \in {\rm Obj}(\C)$, \pbnote{$r \geq 0$,}
we consider $(\eta_{r,0})_A \in \Mor_{\C}^{-r}(\Sigma^rA,A)$ and
$(\eta_{0,-r})_A \in \Mor_{\C}^{-r}(A, \Sigma^{-r}A)$ and we will denote
below by $\eta^{A}_{r}$ the \jznote{following maps 
\begin{equation} \label{dfn-eta-map}
\eta^A_r = i_{-r,0}((\eta_{r,0})_A) \,\,\,\mbox{or} \,\,\,  \eta^A_r = i_{-r,0}((\eta_{0,-r})_A).
\end{equation} }
Thus $\eta^{A}_{r}\in \Mor_{\C}^{0}(\Sigma^{r} A, A)$ or
$\eta^{A}_{r}\in\Mor_{\C}^{0}(A,\Sigma^{-r}A)$, depending on the context. \jznote{Note that there is no ambiguity of the notation $\eta_r^A$ due to the canonical identification via $\Sigma^r$ in (\ref{sigma-mor}).} 
The notions discussed before, $r$-acyclicity, $r$-equivalence and so
forth, can be reformulated in terms of compositions with appropriate
shift morphisms $\eta^{A}_{r}$.

The next lemma is a characterization of $r$-equivalence that follows
easily from the diagrams (\ref{nt-dia-4}) and (\ref{nt-dia-5}).
\begin{lemma} \label{lemma-comp} Suppose that
  $f \in \Mor^{\alpha}(A,B)$. Then $i_{\alpha, \alpha + r}(f) =0$ for
  some $r \geq 0$ if and only if $f \circ \eta^A_{r}=0$ in
  $\Mor^{\alpha}(\Sigma^r A, B)$ and (equivalently) if and only if
  $\eta^B_{r} \circ f =0$ in $\Mor^{\alpha}(A,
  \Sigma^{-r}B)$.
\end{lemma}

In particular, we easily see that for two morphisms
$f,g\in \Mor^{\alpha}(A,B)$, $f\simeq_{r} g$ if and only if
$f\circ \eta^{A}_{r}=g\circ \eta^{A}_{r}$. Moreover, $r$-equivalence
is preserved under shifts.  Further, it is immediate to check that
$f \in \Mor^\alpha(A,B)$ and $g \in \Mor^{\beta}(A,B)$ are
$\infty$-equivalent if and only if there exist $r, r' \geq 0$ with
$\alpha+r = \beta+r'$ such that
\[ f \circ \eta_r^A = g \circ \eta_{r'}^{A}\,\,\,\,\mbox{in} \,\,
  \Mor^{\alpha+r}(A,B), \] where we identify both
$\Mor^{\alpha}(\Sigma^r A, B) $ and $\Mor^{\beta}(\Sigma^{r'} A, B)$
with $\Mor^{\alpha+r}(A, B)$ through the canonical isomorphisms in
Remark \ref{rem:shift}.

Here is a similar characterization of $r$-acyclicity. 

\begin{lemma} \label{lemma-acyclic} $K \simeq_r 0$ is equivalent to
  each of the following:
  \begin{itemize}
  \item[(i)] \pbnote{$\eta_{r}^{K} = 0$}.
  \item[(ii)]
    $i_{\alpha, \alpha+r}: \Mor^{\alpha}(A,K) \to
    \Mor^{\alpha+r}(A,K)$ vanishes for any $\alpha \in \R$ and $A$.
  \item[(iii)]
    $i_{\alpha, \alpha+r}: \Mor^{\alpha}(K,A) \to \Mor^{\alpha+r}(K,
    A)$ vanishes for any $\alpha \in \R$ and $A$.
\end{itemize}
\end{lemma}

\begin{proof} Point (i) is an immediate consequence of the definition of $r$-acylics in Definition \ref{def:acyclics} and of Lemma \ref{lemma-comp} applied for $A,B=K$, 
$f=\mathds{1}_{K}$. 
We now prove (ii). The proof of (iii) is similar and will be omitted.
It is obvious that (ii) implies $K\simeq_{r} 0$ by specializing to $A=K$, $\alpha=0$
 and applying $i_{0,r}$ to $\mathds{1}_K$. 
To prove the converse, we first use diagram (\ref{nt-dia-4}) to deduce that 
the map $i_{\alpha,\alpha+r}$ factors as below: 
\begin{equation} \label{com-da}  \xymatrixcolsep{5pc} \xymatrix{ 
\Mor^{\alpha}(A, K) \ar[r]^-{(i_{-r,0}(\eta_{0,-r})_K) \circ} \ar@/_2pc/[rr]^-{i_{\alpha, \alpha+r}} & \Mor^{\alpha}(A, \Sigma^{-r}K) \ar[r]^-{(\eta_{-r,0})_K\circ} & \Mor^{\alpha+r}(A, K)} ~.~
\end{equation}
Therefore, since $(\eta_{-r,0})_K\circ$ is an isomorphism, for any $f \in \Mor^{\alpha}(A, K)$
we have $i_{\alpha, \alpha+r}(f) = 0$ if and only if $i_{-r,0}(\eta_{0,-r})_K \circ f =0$. From point 
(i) we know that this relation is true for $f=\mathds{1}_{K}$. 
Now, for any $f\in \Mor^{\alpha}(A,K)$, we write $f = \mathds{1}_K \circ f$ and conclude 
 $i_{-r,0}(\eta_{0,-r})_K \circ f =i_{-r,0}(\eta_{0,-r} \circ \mathds{1}_K)\circ f=0$. \end{proof}


In particular, we see that $K$ is $r$-acyclic if and only if any of
its shifts $\Sigma^{\alpha} K$ is so.

\begin{rem} \label{rem:c0repl} (a) Assume that $\C$ is a persistence
  category endowed with a shift functor $\Sigma$ and that
  $f_{1},f_{2}\in \Mor_{\C}^{\alpha}(A,B)$, then, for all practical
  purposes, we may replace $f_{i}$ with the $\C_{0}$ morphisms
  $\tilde{f}_{i}\in \Mor_{\C}^{0}(\Sigma^{\alpha}A,B)$ where
  \zjnote{$\tilde{f}_{i}=f_{i}\circ (\eta_{\alpha,0})_A$}. The property
  $f_{1}\simeq _{r} f_{2}$ is equivalent to
  $$\tilde{f}_{1}\circ \eta_{r}^{\Sigma^{\alpha}A} \simeq_{0}  \tilde{f}_{2}
  \circ \eta_{r}^{\Sigma^{\alpha}A}$$ which is a relation in $\C_{0}$.

  (b) Shift functors are natural in many geometric
  examples. Nonetheless, for a given persistence category $\C$, the
  existence of a shift functor $\Sigma$ on $\C$ is a constraining
  additional structure.  In particular, a persistence category
  $\mathcal{C}$ endowed with a shift functor $\Sigma$ contains
  considerable redundant information. Indeed, the isomorphism
  (\ref{nt-iso-3}) implies that all the morphisms in $\mathcal{C}$ are
  determined by the morphisms in $\mathcal{C}_{0}$ together with
  $\Sigma$. In other words, given a category $\mathcal{C}_{0}$ endowed
  with a shift functor $\Sigma$ (appropriately defined) one can define
  a persistence category $\mathcal{C}$ with the same objects as
  $\C_{0}$ by using (\ref{nt-iso-3}) to define morphisms of arbitrary
  shifts out of the morphisms in $\C_{0}$.  We will see such an
  example in \S\ref{ssec-pm}.
   
  (c) There is an obvious way to formally complete any persistence
  category $\C$ to a larger persistence category $\tilde{\C}$ that is
  endowed with a canonical shift functor. This is achieved by formally
  adding objects $\Sigma^{r}X$ for each $r\in \R$ and
  $X\in \mathrm{Obj}(\C)$ and defining morphisms such that the
  relations in Remark \ref{rem:shift} are satisfied. In view of this
  and of the redundancy at point (b), one could prefer to replace the
  notion of persistence category with a structure consisting of a
  category - corresponding to $\C_{0}$ - and a shift functor.  This
  leads to an equivalent formalism. We stick in this paper with the
  formalism of persistence categories as introduced in Definition
  \ref{dfn-pc} as we found it easiest to handle in algebraic
  manipulations and because it corresponds naturally to most of our
  geometric examples.
\end{rem}

\subsection{An example of a persistence category} \label{ssec-pm}

We give an example of a persistence category that is constructed from
persistence $\k$-modules. To some extent, this is the motivation of
the definition of a persistence category. Recall that for a
persistence $\k$-module $\V$, the notation $\V[r]$ denotes another
persistence $\k$-module which comes from an $r$-shift from $\V$ in the
sense that
\[ \V[r]_t = \V_{r+t} \,\,\,\,\,\mbox{and}\,\,\,\,\,
  \iota^{\V[r]}_{s,t} = \iota^{\V}_{s+r, t+r}. \] A persistence
morphism $f: \V \to \W$ is an $\R$-family of morphisms $f = \{f_t\}$
that commutes with the persistence structure maps of $\V$ and $\W$,
i.e., $f_{t} \circ \iota^{\V}_{s,t} = \iota^{\W}_{s,t} \circ
f_s$. Similarly, one can define $r$-shifted persistence morphism
$f[r]$ where $(f[r])_t = f_{r+t}$.


\

Let $\mathcal P{\rm Mod}_{\k}$ be the category of persistence
$\k$-modules, then we claim that $\mathcal P{\rm Mod}_{\k}$ can be
enriched to be a persistence category $\C^{\mathcal P{\rm
    Mod}_{\k}}$. Indeed, let
${\rm Obj}(\C^{\mathcal P{\rm Mod}_{\k}}) = {\rm Obj}(\mathcal P{\rm
  Mod}_{\k})$, and for objects $\V, \W$ in $\mathcal P{\rm Mod}_{\k}$,
define
\begin{equation} \label{per-mor} {\hom}_{\C^{\mathcal P{\rm
        Mod}_{\k}}}(\V, \W) : = \left\{\{{\hom}_{\mathcal P{\rm
        Mod}_{\k}}(\V, \W[r])\}_{r\in \R}; \{i_{r,s}\}_{r \leq s \in
      \R} \right\}.
\end{equation}
Here,
${\hom}_{\C^{\mathcal P{\rm Mod}_{\k}}}^r(\V, \W) = {\hom}_{\mathcal P{\rm Mod}_{\k}}(\V, \W[r])$, and
${\hom}_{\mathcal P{\rm Mod}_{\k}}(\cdot, \cdot)$ consists of
persistence morphisms. For any $r \leq s$, the well-defined
persistence morphism $\iota^{\W}_{\cdot+r,\cdot +s}: \W[r] \to \W[s]$
induces structure maps
$i_{r,s} := \iota^{\W}_{\cdot+r,\cdot +s} \circ$ in
(\ref{per-mor}). Moreover, the composition
$\circ_{(r,s)}: \Mor_{\C^{\mathcal P{\rm Mod}_{\k}}}^{r}(\U, \V)
\times \Mor_{\C^{\mathcal P{\rm Mod}_{\k}}}^s(\V, \W) \to
\Mor_{\C^{\mathcal P{\rm Mod}_{\k}}}^{r+s}(\U, \W)$ is defined by
\[ (f, g) \mapsto g[r] \circ f \] where we use the identification
${\hom}_{\mathcal P{\rm Mod}_{\k}}(\V, \W)[r] ={\hom}_{\mathcal
  P{\rm Mod}_{\k}}(\V[r], \W[r])$ for any $r \in \R$. Moreover, for
the following diagram where $r \leq r'$ and $s \leq s'$,
\begin{equation} \label{ex-per-dia} \xymatrix{
    {\hom}_{\mathcal P{\rm Mod}_{\k}}(\U, \V[r]) \times {\hom}_{\mathcal P{\rm Mod}_{\k}}(\V,\W[s]) \ar[r]^-{\circ_{(r,s)}} \ar[d]_-{i_{r,r'} \times i_{s,s'}} & {\hom}_{\P}(\U,\W[r+s]) \ar[d]^-{i_{r+s, r'+s'}} \\
    {\hom}_{\mathcal P{\rm Mod}_{\k}}(\U, \V[r']) \times {\rm
      Hom}_{\mathcal P{\rm Mod}_{\k}}(\V,\W[s'])
    \ar[r]^-{\circ_{(r',s')}} & {\hom}_{\P}(\U,\W[r'+s'])}
\end{equation}
we have
\begin{align*}
(\circ_{(r',s')} \circ (i_{r,r'} \times i_{s,s'}))(f,g) & = \circ_{(r',s')}(i_{r,r'}(f), i_{s,s'}(g)) \\
& = i_{s,s'}(g)[r'] \circ i_{r,r'}(f) \\
& = (\iota^{\W}_{\cdot +s, \cdot + s'} \circ g)[r'] \circ (\iota^{\W}_{\cdot + r, \cdot + r'} \circ f) \\
& = \iota^{\W}_{\cdot + r'+s, \cdot + r' + s'} \circ g[r'] \circ \iota^{\V}_{\cdot + r, \cdot + r'} \circ f \\
& = \iota^{\W}_{\cdot + r' + s, \cdot + r'+s'} \circ \iota^{\W}_{\cdot + r + s, \cdot + r' + s} \circ g[r] \circ f\\
& = \iota^{\W}_{\cdot + r+s, \cdot + r'+s'} \circ (g[r] \circ f) = (\iota_{r+s, r'+s'} \circ \circ_{(s,t)})(f,g)
\end{align*}
where the fifth equality is due to the fact that $g$ is a persistence
morphism (so, in particular, it commutes with the persistence
structure maps). Therefore, the diagram (\ref{ex-per-dia}) is
commutative and $\C^{\mathcal P{\rm Mod}_{\k}}$ is a persistence
category in the sense of Definition \ref{dfn-pc}.

\medskip

Since
${\hom}_{\C^{\mathcal P{\rm Mod}_{\k}}}^0(\V, \W) = {\hom}_{\mathcal P{\rm Mod}_{\k}}(\V, \W)$, we have
$\C^{\mathcal P{\rm Mod}_{\k}}_0 = \mathcal P{\rm Mod}_{\k}$.  An
example of a persistence endofunctor on
$\C^{\mathcal P{\rm Mod}_{\k}}$, denoted by
$\Sigma^{\alpha}: \C^{\mathcal P{\rm Mod}_{\k}} \to \C^{\mathcal P{\rm
    Mod}_{\k}}$, is defined by
\begin{equation} \label{ex-per-fun} 
\Sigma^{\alpha}(\V) : = \V[-{\alpha}] \,\,\,\,\mbox{and}\,\,\,\, \Sigma^{\alpha}(f) = f[-\alpha]
\end{equation}
for any $\alpha \in \R$. 
It is immediate to see that  $\Sigma^{\alpha}$ is a persistence endofunctor on $\C^{\P{\rm Mod}_{\k}}$ for any $\alpha \in \R$ in the sense of Definition \ref{dfn-per-functor}. 
\medskip

We now define a shift functor on $\C^{\mathcal P{\rm Mod}_{\k}}$, denoted by $\Sigma: (\R, +) \to \mathcal P{\rm Fun}(\C^{\mathcal P{\rm Mod}_{\k}})$,  by 
\[ \Sigma(\alpha) : = \Sigma^{\alpha} \,\,\mbox{defined in (\ref{ex-per-fun})} \,\,\,\,\,\,\,\mbox{and} \,\,\,\,\,\,\, \eta_{\alpha,\beta} = \mathds{1}_{\, \cdot[-\alpha]} \]
for any $\alpha, \beta \in \R$. Indeed, evaluate $\eta_{\alpha,\beta}$ on any object $\V$, 
\begin{align*}
(\eta_{\alpha,\beta})_{\V} = \mathds{1}_{\V[-\alpha]} & \in {\hom}_{\mathcal P{\rm Mod}_{\k}}(\V[-\alpha], \V[-\alpha])\\
& = {\hom}_{\mathcal P{\rm Mod}_{\k}}(\V[-\alpha], \V[-\beta + \beta-\alpha])\\
& = {\hom}_{\C^{\mathcal P{\rm Mod}_{\k}}}^{\beta-\alpha}(\V[-\alpha], \V[-\beta])\\
& = {\hom}_{\C^{\mathcal P{\rm Mod}_{\k}}}^{\beta-\alpha}(\Sigma^\alpha \V, \Sigma^\beta \V). 
\end{align*}
In other words, $\eta_{\alpha, \beta}$ is a persistence natural transformation of shift $\beta - \alpha$ as in Definition \ref{dfn-nt}. Therefore, $\Sigma$ defines a shift functor on $\C^{\mathcal P{\rm Mod}_{\k}}$.

\medskip

Finally, for each $r \geq 0$, recall that the notation $\eta_r^\V$ in
(\ref{dfn-eta-map}) denotes the composition
$i_{-r,0} \circ (\eta_{r,0})_\V$. In particular,
$\eta_r^{\V} \in \Mor_{\C^{\mathcal P{\rm Mod}_{\k}}}^0(\Sigma^{r} \V,
\V) = {\hom}_{\mathcal P{\rm Mod}_{\k}}(\V[-r], \V)$, equals the
following composition
\[ \V[-r] \xrightarrow{\mathds{1}_{\V[-r]}} \V[-r]
  \xrightarrow{\iota^{\V}_{\cdot+r,\cdot } \circ} \V\] which is just
$\iota^{\V}_{\cdot-r, \cdot}$, the persistence structure maps of
$\V$. Assume that objects in $\mathcal P{\rm Mod}_{\k}$ admit
sufficient regularities so that they can be equivalently described via
barcodes (see \cite{C-B15}). In this case, by (i) in
Lemma~\ref{lemma-acyclic}, the $r$-acyclic objects in
$\C^{\mathcal P{\rm Mod}_{\k}}$ are precisely those persistence
$\k$-modules with only bars of length at most $r$ in their barcodes
(see \cite{Ush11, Ush13, UZ16}).

\begin{remark} \label{rem:ear}  a. The way that the category $\mathcal P{\rm Mod}_{\k}$ is
  enriched to $\C^{\mathcal P{\rm Mod}_{\k}}$ above is also
  investigated in the recent work
  \cite[Section~10]{BuMi:homological-algebra-persistence}. In
  particular, the morphism set defined in~\eqref{per-mor} coincides
  with the enriched morphism set
  in~\cite[Proposition~10.2]{BuMi:homological-algebra-persistence} and
  $\C^{\mathcal P{\rm Mod}_{\k}}$ is similar to
  $\underline{\rm Mod}_R^P$
  in~\cite[Proposition~10.3]{BuMi:homological-algebra-persistence}
  (when taking $R = \k$ and $P = \R$).
  \label{rem:pers-cat-S} 
  
 b. The notion of persistence category endowed with shift functors is very natural
 in persistence considerations, as already mentioned at point a above. The same notion appears in \cite{Sco20} under the name of {\em locally 
persistent category} and the $0$-level (from Definition \ref{dfn-c0-cinf}) appears in that work 
where it is called the {\em underlying category} of the respective locally persistent category. The definition of interleaving of persistence modules adapts trivially to this context  - of persistence categories or, equivalently, locally persistent categories -   to provide a (pseudo)-distance, possibly infinite, on the objects of such a category, as in Definition \ref{def:interleave}. The work
 \cite{Sco20} analyzes and  establishes key properties - for instance, completeness - for interleaving  distances of this sort  under certain assumptions - such as existence of products, or co-products, or a model structure, or existence of limits -  on the $0$-level category. 

Starting from \S\ref{subsec:TPC},  we focus  on properties of persistence categories  with  $0$-levels that have the structure of triangulated categories. In practice, this means that they are often homotopy categories of other categories.  From this perspective, while we work, in some sense, at the homotopy level,  \cite{Sco20}  is geared towards considering  $0$-levels that have directly a model category structure.  In our case, the triangulation is tied to the persistence structure by some simple axioms.  If $\C$ is such a category, called a {\em triangulated persistence category (TPC)}, then we will see that the $\infty$-level is endowed with a  categorical weight induced by the persistence
structure, and the general machinery in \S \ref{subsec:triang-gen} leads to a class of fragmentation (pseudo) metrics on the objects of $\C$.   These fragmentation metrics extend, on one hand, interleaving type metrics, and, on  the other hand, complexity measurements such as those mentioned in Remark \ref{rem:genLS} and are similar to classical notions in topology such as cone-length \cite{Cor94}. In \S \ref{s:tpfc} we show that certain derived Fukaya categories admit TPC refinements and thus their objects are endowed with persistence fragmentation metrics. The interest of this class of fragmentation metrics in this symplectic context is that, under favourable geometric assumptions, these metrics are both non-degenerate and finite while interleaving type distances take infinite values.  

Some  elementary relations between interleaving and the rest of the algebraic machinery in the paper  appear in \S\ref{subsec:interleaving-etc}. Moreover, it is likely that, in some cases, some of the deeper properties of the interleaving distances discussed in \cite{Sco20}  can be related in more substantial ways to  our fragmentation metrics, however we will not pursue these questions here. \end{remark}

\section{Triangulated persistence categories} \label{subsec:TPC} This
section is central for the rest of the paper. It investigates
triangulation properties in the context of persistence categories. We
start with two key definitions in \S \ref{subsubsec:main-def},
Definition \ref{dfn-tpc} which introduces the notion of \zjnote{triangulated}
persistence category (TPC) - a persistence category $\C$ with a shift
functor whose $0$-level $\C_{0}$ is triangulated, and Definition
\ref{dfn-r-iso} that introduces the notion of $r$-isomorphism.  We
then discuss a number of useful properties of \zjnote{$r$-isomorphisms}. These
properties are in some sense ``shift'' controlled analogues of
properties that appear when defining the Verdier localization of a
triangulated category. Indeed, in \S\ref{subsubsec:localization} we
see that the acyclics of finite order in $\C$ form a triangulated
subcategory $\mathcal{A}\C_{0}$ of $\C_{0}$ and that the Verdier
localization of $\C_{0}$ with respect to this subcategory is the
$\infty$-level category $\C_{\infty}$ of $\C$, which is therefore
itself triangulated.  The main aim of our algebraic formalism is to
construct a notion of weighted exact triangles in $\C$ - and this is
pursued in \S\ref{subsubsec:weight-tr}, in particular in Definition
\ref{subsubsec:weight-tr}. We then discuss in \S\ref{subsubsec:frag1}
associated fragmentation pseudo-metrics.

\subsection{Main definitions}\label{subsubsec:main-def}

We will use consistently below the characterization of $r$-equivalence
in Lemma \ref{lemma-comp} as well as that of $r$-acyclics in Lemma
\ref{lemma-acyclic}.

\begin{dfn} \label{dfn-tpc} A {\em triangulated persistence category}
  is a persistence category $\C$ endowed with a shift functor $\Sigma$
  such that the following three conditions are satisfied:
  \begin{itemize}
  \item[(i)] The $0$-level category $\C_0$ is triangulated with a
    translation automorphism denoted by $T$. \pbnote{Note that in
      particular $\C_0$ is additive and we further assume that the
      restriction of the persistence structure of $\C$ to $\C_0$ is
      compatible with the additive structure on $\C_0$ in the obvious
      way. Specifically this means that
      $\Mor_{\C}^r(A \oplus B,C) = \Mor_{\C}^r(A,C) \oplus
      \Mor_{\C}^r(B,C)$ for all $r \leq 0$ and the persistence maps
      $i_{r,s}$, $r \leq s \leq 0$ are compatible with this
      splitting. The same holds also for
      $\Mor_{\C}^r(A, B \oplus C)$.}
  \item[(ii)] The restriction of $\Sigma^{r}$ to
    $\mathrm{End}(\C_{0})$ is a triangulated endofunctor of $\C_{0}$
    for each $r\in \R$. \pbnote{Note that each of the functors
      $\Sigma^r$, being a triangulated functor, is also assumed to be
      additive. We further assume that all the natural transformations
      $\eta_{r,s}: \Sigma^r \to \Sigma^s$, $s, r \in \mathbb{R}$, are
      compatible with the additive structure on $\C_0$.}
  \item[(iii)] For any $r \geq 0$ and any $A \in {\rm Obj}(\C)$, the
    morphism $\eta^A_r: \Sigma^rA \to A$ defined in
    (\ref{dfn-eta-map}) embeds into an exact triangle of $\C_0$
    \[ \Sigma^r A \xrightarrow{\eta^A_r} A \to K \to T\Sigma^r A \] such
    that $K$ is $r$-acyclic.
  \end{itemize}
\end{dfn}

\begin{ex}\label{ex:co-chains}
  The fundamental example of a triangulated persistence category is
  provided by the homotopy category of filtered (co)-chain complexes
  over a field $\k$, $H^{0}\mathcal{FK}_{\k}$. The objects are
  filtered cochain complexes $(C,\partial)$ over $\k$,
  $C :\ldots \subset C^{\leq \alpha}\subset C^{\leq \beta}\subset
  \ldots $ ($\alpha\leq \beta \in \R$) and $\partial$ does not
  increase filtration, hence each $C^{\leq\alpha}$ is itself a cochain
  complex - a more complete description is given
  in~\S\ref{subsec:filt-co}. The morphisms are homotopy classes of
  filtered chain maps
  $$\Mor^{r}(C,C')=\{f:C\to C' \ | \ f \mathrm{\ is\ a\ chain\ map},
  f(C^{\leq \alpha})\subset (C')^{\leq \alpha +r}\}/\simeq_{r}$$ where
  the relation $\simeq_{r}$ is cochain homotopy via a homotopy
  $h: C^{\ast}\to C^{\ast -1}$ such that
  $h(C^{\leq \alpha})\subset (C')^{\leq \alpha +r}$.  The translation
  functor is defined as usual by translating degree (and keeping the
  filtration unchanged), namely $(TC)^i = C^{i+1}$, and with the
  obvious action on morphisms. The shift functor acts on objects by
  $[\Sigma^{r}C]^{\leq \alpha}=C^{\leq \alpha -r}$ with the obvious
  differential and the obvious action on morphisms. The $0$-level of
  $H^{0}\mathcal{FK}_{\k}$, $[H^{0}\mathcal{FK}_{\k}]_{0}$, is the
  subcategory with the same objects but whose morphisms come only from
  filtration preserving chain maps. This is a triangulated category
  because, for chain preserving maps, the mapping-cone construction is
  filtration preserving. The $r$-acyclics in this case are filtered
  complexes $C$ such that $\mathds{1}_{C}$ is chain homotopic to $0$
  through a chain homotopy that shifts filtration by at most $r$.

  \pbrev{Note that we also have the (full) subcategory
    $\mathcal{FK}^{\text{fg}}_k \subset \mathcal{FK}_{\k}$ of {\em
      finitely generated} filtered cochain complexes which is also a
    TPC. The category $[H^{0}\mathcal{FK}^{\text{fg}}_{\k}]_{\infty}$
    is equivalent to the usual homotopy category of finitely generated
    cochain complexes.}
\end{ex}

\begin{remark} \label{rem:shifts-T} (a) Given that $\C_0$ is
  triangulated, the functors $\Mor_{\C}^s(X,-)$ and
  $\Mor_{\C}^s(-, X)$ are exact for $s=0$. This property together with
  the fact that $\Sigma^s$ is a triangulated functor for all $s$ and
  the relations in Remark~\ref{rem:shift} imply that these functors
  are exact {\em for all} $s\in \R$.

  (b) Condition (ii) \pbnote{requires in particular} that $\Sigma$ and
  $T$ commute. Thus, $T\Sigma^{r} X=\Sigma^{r}TX$ for each object $X$
  and for any $f\in \Mor_{\C}^{0}(A,B)$ we have
  $\Sigma^{r} Tf=T\Sigma^{r}f$. Additionally, each $\Sigma^{r}$
  preserves the additive structure of $\C_{0}$ and it takes each exact
  triangle in $\C_{0}$ to an exact triangle. \pbnote{Moreover, the
    assumptions above imply that we have canonical isomorphisms
    $$\Mor_{\C}^r(A \oplus B, C) \cong \Mor_{\C}^r(A,C)
    \oplus \Mor_{\C}^r(B,C), \ \; \; \forall r \in \mathbb{R},$$}
  \pbnote{and the persistence maps $i_{r,s}$ are compatible with these
    isomorphisms. The same holds also for
    $\Mor_{\C}^r(A, B \oplus C)$. Finally, the maps $\eta_r^{A}$,
    $A \in \text{Obj}(\C_0)$, $r \geq 0$, are compatible with the
    additive structure on $\C_0$.}

  Notice also that the functor $T$ extends from $\C_{0}$ to a functor
  on $\C$. Indeed $T$ is already defined on all the objects of $\C$ as
  well as on all the morphisms of shift $0$.  For $f\in \Mor^{r}(A,B)$
  we define $Tf= T(f\circ (\eta_{r,0})_{A})\circ
  (\eta_{0,r})_{TA}$. It is easily seen that with this definition $T$
  is indeed a functor and it immediately follows that
  $T((\eta_{r,s})_{A})= (\eta_{r,s})_{TA}$ for all objects $A$ in $\C$
  and $r,s\in \R$. Further, by using the identifications in Remark
  \ref{rem:shift}, it also follows that $T$ is a persistence functor.
  In particular, we have $\eta_{r}^{TA}=T\eta_{r}^{A}$ for each object
  $A$ in $\C$.

  (c) Given that $0$-isomorphisms preserve $r$-acyclicity - as noted
  in~\S\ref{subsubsec:pers-cat1}, condition~(iii) in
  Definition~\ref{dfn-tpc} does not depend on the specific extension
  of $\eta_{r}^{A}$ to an exact triangle.
 
  (d) In a way similar to Remark~\ref{rem:c0repl} (b), the data
  encoded in a triangulated persistence category is determined by the
  triangulated category $\C_{0}$ together with an appropriate shift
  functor $\Sigma : (\R, +)\to \mathrm{End}(\C_{0})$. From this data
  it is easy to define a triangulated persistence category $\C$ \ with
  the same objects as $\C_{0}$, that has $\C_{0}$ as its $0$-level and
  with morphisms endowed with a persistence structure such that
  (\ref{nt-dia-4}) and (\ref{nt-dia-5}) are satisfied with respect to
  the given shift functor $\Sigma$.  We do not give further details
  here but we will see such an example in~\S\ref{subsec:Tam}.
\end{remark}

\zhnote{It is clear that TPCs form a category with respect to
  persistence functors that respect the additional structure. The
  appropriate notion is formalized below.}
\zhnote{\begin{dfn}\label{def:TPC functors} Let $\mathcal{C}'$ and
    $\mathcal{C}''$ be two TPCs. A persistence functor
    $\mathcal{F}:\mathcal{C}' \longrightarrow \mathcal{C}''$ is called
    a TPC-functor if it satisfies the following conditions:
\begin{itemize}
\item[(i)] $\mathcal{F}$ is compatible with the shift functors
$\Sigma_{\mathcal{C}'}$, $\Sigma_{\mathcal{C}''}$ of the two
categories, namely
$\mathcal{F} \circ \Sigma^r_{\mathcal{C}'}= \Sigma^r_{\mathcal{C}''}
\circ \mathcal{F}$ for all $r \in \mathbb{R}$, and
$(\eta_{r,s})_{\mathcal{F}(A)} =\mathcal{F}((\eta_{r,s})_A)$ for all
$A \in \Ob(\mathcal{C}')$ and all $r,s$.  
\item[(ii)] The $0$-level 
$\mathcal{F}|_{\mathcal{C}'_0}: \mathcal{C}'_0 \longrightarrow
\mathcal{C}''_0$ of the  functor $\mathcal{F}$  is triangulated. 
\end{itemize}
\end{dfn}}

\zhnote{In the definition above the fact that
  $\mathcal{F}|_{\mathcal{C}'_0}$ maps $\mathcal{C}'_0$ to
  $\mathcal{C}''_0$ follows from the assumption that $\mathcal{F}$ is
  a persistence functor.  Modifying Definition \ref{def:eq-pers}, we
  now give the definition of an equivalence between TPCs.
\begin{dfn}\label{def:TPC equivalences}
  Let $\mathcal{C}'$ and $\mathcal{C}''$ be two TPCs. A TPC-functor
  $\mathcal{F}: \mathcal{C}' \longrightarrow \mathcal{C}''$ is called
  a TPC-equivalence if there exists a TPC-functor
  $\mathcal{G}: \mathcal{C}'' \longrightarrow \mathcal{C}'$ such that
  both $\mathcal{F} \circ \mathcal{G}$ and
  $\mathcal{G} \circ \mathcal{F}$ are isomorphic to the respective
  identity functors via \pbred{persistence} natural transformations of
  shift $0$.
\end{dfn}}

\zhnote{Standard results in triangulated categories
(e.g.~\cite[Section~1.2]{Huyb:FMT}) imply that a TPC-functor
$\mathcal{F}: \mathcal{C}' \longrightarrow \mathcal{C}''$ is a
TPC-equivalence if and only if it is an equivalence of persistence
categories.}

\begin{dfn} \label{dfn-r-iso} Let $\C$ be a triangulated persistence
  category. A map $f \in \Mor_{\C}^0(A,B)$ is said to be an {\em
    {\boldsymbol $r$}-isomorphism} (from $A$ to $B$) if it embeds into
  an exact triangle in $\C_0$ \[ A \xrightarrow{f} B \to K \to TA \]
  such that $K \simeq_r 0$.
\end{dfn}

We write  $f: A \simeq_r B$.

\begin{remark}(a) If $f$ is an $r$-isomorphism, then $f$ is an
  $s$-isomorphism for any $s \geq r$.  It is not difficult to check,
  and we will see this explicitly in Remark \ref{rem:various-iso},
  that for $r=0$ this definition is equivalent to the notion of
  $0$-isomorphism introduced before (namely isomorphism in the
  category $\C_{0}$).

  (b) The relation $T(\eta_{r}^{K})=\eta_{r}^{TK}$ implies that $TK$
  is $r$-acyclic if and only if $K$ is $r$-acyclic and, therefore, $f$
  is an $r$-isomorphism if and only if $Tf$ is one.

  (c) From Definition \ref{dfn-tpc} (iii) we see that for any
  $r \geq 0$ and $A \in {\rm Obj}(\C)$ we have
  $$\eta^A_r: \Sigma^r A \simeq_r A~.~$$ 
\end{remark}

\begin{prop} \label{prop-r-iso} Any triangulated persistence category
  $\C$ has the following properties.
\begin{itemize} 
\item[(i)] If $f: A \to B$ is an $r$-isomorphism, then there exist $\phi \in \Mor_{\C}^0(B, \Sigma^{-r}A)$ and $\psi \in \Mor_{\C}^0(\Sigma^r B, A)$ such that 
\[ \phi \circ f = \eta^A_r  \,\,\mbox{in $\Mor_{\C}^0(A, \Sigma^{-r}A)$} \,\,\,\,\, \mbox{and}\,\,\,\,\, f \circ \psi = \eta^B_r \,\,\mbox{in $\Mor_{\C}^0(\Sigma^r B, B)$}. \]
The map $\psi$ is called a \zjnote{{\rm right $r$-inverse}} of $f$ and  $\phi$ is a \zjnote{{\rm left $r$-inverse}} of $f$. 
They satisfy $\Sigma^r \phi \simeq_r \psi$.
\item[(ii)] If $f$ is an $r$-isomorphism,  then any two left  $r$-inverses $\phi, \phi'$ of $f$ are themselves 
$r$-equivalent and the same conclusion holds for right $r$-inverses. 
\item[(iii)] If $f: A \simeq _r B$ and $g: B \simeq_s C$, then $g \circ f: A \simeq_{r+s} C$. 
\end{itemize}
\end{prop}

\begin{proof} (i) We first construct $\phi$. In $\C_0$, the morphism
  $f: A \to B$ embeds into an exact triangle
  $A \xrightarrow{f} B \xrightarrow{g} K \xrightarrow{h} TA$ with
  $K \simeq_r 0$. Using the fact that $\Sigma$ and $T$ commute, the
  following diagram is easily seen to be commutative:
  \begin{equation} \label{com-shift-fun2} \xymatrixcolsep{3pc}
    \xymatrix{
      K \ar[r]^-{h} \ar[d]_-{\eta_r^K} & TA \ar[d]^-{\eta^{TA}_r}\\
      \Sigma^{-r} K \ar[r]^-{\Sigma^{-r}h} & \Sigma^{-r} TA}
  \end{equation}
  Thus $\eta_r^{TA} \circ h = \Sigma^{-r} h \circ \eta_r^K = 0$ since
  $K$ is $r$-acyclic (and so $\eta_r^K =0$). By rotating exact
  triangles in $\C_0$ we obtain a new $\C_{0}$-exact triangle
  \jznote{$K \xrightarrow{h} TA \xrightarrow{Tf} TB \xrightarrow{Tg} TK$} and
  consider the diagram below (in $\C_{0}$):
  \pbnote{
    \[ \xymatrixcolsep{3pc} \xymatrix{ K \ar[d] \ar[r]^{h} & TA
        \ar[d]^-{\eta^{TA}_r} \ar[r]^{Tf}
        & TB \ar@{-->}[d]^-{\tilde{\phi}} \ar[r]^-{Tg} & TK \ar[d] \\
        0 \ar[r] & \Sigma^{-r} TA \ar[r]^-{\mathds{1}} & \Sigma^{-r} TA \ar[r]
        & 0} ~.~\]} The first square on the left commutes, so we
  deduce the existence of a map
  $$\tilde{\phi}\in \Mor_{\C}^{0}(TB,\Sigma^{-r}TA)$$ 
  that makes commutative the middle and right squares. The desired
  left inverse of $f$ is $\phi=T^{-1}\tilde{\phi}$. A similar argument
  leads to the existence of $\psi$.  We postpone the identity
  $\Sigma^{r}\phi\simeq_{r} \psi$ after the proof of (ii).  \medskip

  (ii) If $\phi$ $\phi'$ are two left inverses of $f$ then
  $(\phi - \phi') \circ f =0$. Therefore
  $$(\phi - \phi') \circ \eta^B_r = (\phi - \phi') \circ (f
  \circ \psi) = ((\phi - \phi') \circ f) \circ \psi = 0.$$
%
  Lemma~\ref{lemma-comp} implies that $\phi \simeq_r \phi'$. The same
  argument works for right inverses. We now return to the identity
  $\Sigma^{r}\phi\simeq_{r} \psi$ (with the notation at point (i)).
  We have the following commutative diagram:
  \[ \xymatrixcolsep{5pc} \xymatrix{ \Sigma^r B \ar[r]^-{\psi}
      \ar@/_1.5pc/[rr]^-{\eta^B_r} & A \ar[r]^-{f}
      \ar@/^1.5pc/[rr]^-{\eta^A_r} & B \ar[r]^-{\phi} & \Sigma^{-r} A
    }~.~\] Therefore, $\eta^A_r \circ \psi = \phi \circ
  {\eta^B_r}$. By the naturality properties of $\eta$ we also have
  $\phi \circ {\eta^B_r}=\eta^A_r \circ \Sigma^r \phi$. Thus, by Lemma
  \ref{lemma-comp}, $\Sigma^r \phi \simeq_r \psi$.  \medskip

  (iii) We will make use of the following lemma.
  \begin{lemma} \label{lemma-exact} If $K \to K'' \to K' \to TK$ is an
    exact triangle in $\C_0$, $K\simeq_r 0$ and $K' \simeq_s 0$, then
    $K'' \simeq_{r+s} 0$.
  \end{lemma}

  \begin{proof} [Proof of Lemma \ref{lemma-exact}] 
    We associate the following commutative diagram to the exact
    triangle in the statement:
    \pbnote{
      \[ \xymatrix{ \Mor^{\alpha}(K'', K) \ar[r] \ar[d] &
          \Mor^{\alpha}(K'', K'')
          \ar[r] \ar[d]^-{v} & \Mor^{\alpha}(K'', K') \ar[d]^-{0} \\
          \Mor^{\alpha+s}(K'', K) \ar[r]^-{k} \ar[d]_-{0} &
          \Mor^{\alpha+s}(K'', K'') \ar[r]^{h} \ar[d]^-{t}
          & \Mor^{\alpha+s}(K'', K')  \\
          \Mor^{\alpha+r+s}(K'', K) \ar[r]^-{n} &
          \Mor^{\alpha+r+s}(K'', K'') & } \]} Here, the vertical
    morphisms are the persistence structure maps. \pbnote{The
      rightmost vertical map and the lower leftmost vertical one are
      both $0$}
    due to our hypothesis together with Lemma~\ref{lemma-acyclic}. The
    functor $\Mor^{\alpha}(K'',-)$ is exact which implies that
    $t\circ v=0$ and, again by Lemma \ref{lemma-acyclic}, we deduce
    $K''\simeq_{r+s} 0$.
  \end{proof}

  Returning to the proof of the proposition, point (iii) now follows
  immediately by using the octahedral axiom to construct the following
  commutative diagram in $\C_0$
  \[ \xymatrix{
      A \ar[r]^-{f} \ar[d] & B \ar[r] \ar[d]^-{g} & K \ar[d] \\
      A \ar[r]^{g \circ f} \ar[d] & C \ar[r] \ar[d] & K'' \ar[d]\\
      0 \ar[r] & K' \ar[r] & K' }\] with exact rows and columns and
  applying Lemma~\ref{lemma-exact} to the \pbnote{rightmost column}.
\end{proof}

\begin{remark}\label{rem:various-iso} (a) Points (i) and (ii)
  in Proposition~\ref{prop-r-iso} imply that the notion
  $0$-isomorphism $f: A \to B$, as given by Definition~\ref{dfn-r-iso}
  for $r=0$, is equivalent to \zjnote{an} isomorphism in
  $\mathcal{C}_{0}$. In particular, for $r=0$, $f$ admits a unique
  inverse in $\Mor_{\C}^0(B,A)$.

  (b) Point (iii) in Proposition \ref{prop-r-iso} shows that being
  $r$-isomorphic (for a fixed $r$) cannot be expected to be an
  equivalence relation on $\Ob(\C)$ (unless $r=0$).
\end{remark}

Here are several useful additional results and corollaries. 

\ocnote{The first is a version of the five-lemma in the TPC context.

  \begin{prop} \label{cor-1} Consider the following commutative
    diagram in $\C_0$,
  
    \begin{equation}\label{eq:diag-ff1}
      \xymatrix{ A \ar[r] \ar[d]_-{u} & B \ar[r] \ar[d]_-{v} & C
        \ar[r] \ar@{-->}[d]_-{w}
        & TA \ar[d]_-{Tu}\\
        A' \ar[r] & B' \ar[r] & C' \ar[r] & TA'}
    \end{equation}  such that the two rows are
    exact triangles. If $u$ is an $r$-isomorphism and $v$ is an
    $s$-isomorphism, then:
    \begin{itemize}
    \item[(i)]  There exists $w$ making the diagram commutative which
      is an $(r+s)$-isomorphism.
    \item[(ii)] Any $w$ making the diagram commutative is a
      $3(r+s)$-isomorphism.
    \end{itemize}
  \end{prop}
}

\begin{proof} \ocnote{Part (i) of the proposition is an easy
    consequence of Lemma \ref{lemma-exact} and the octahedral axiom in
    $\C_{0}$.

    To show part (ii) we will use the following notation. For an
    object $Z\in \mathcal{O}b(\mathcal{C})$ we denote by $H_{Z}$ the
    functor
    $\hom_{\mathcal{C}}(Z, -):\mathcal{C} \longrightarrow
    \C^{\mathcal{P} {\rm Mod}_{\k}}$ (the target of this functor is an
    obvious enrichment of the category of persistence modules;
    see~\S\ref{ssec-pm}). Similarly denote by $H^{\ast}_{Z}$ the
    (contravariant) functor $\hom_{\mathcal{C}}(-, Z)$.}
  \pbrrvv{Further, if $U$ and $V$ are persistence modules and
    $F:U\to V$ is a morphism in $\C^{\mathcal{P} {\rm Mod}_{\k}}$}
  \ocnote{we say that $F$ is:
    \begin{itemize}
    \item[-] $r$-epi if for all $y\in V$ there exists $x\in U$ such
      that $F(x)=i_{r}(y)$. (Here and in what follows $i_{r}$ stands
      for the persistence structural map on $\hom_{C}(A,B)$ whose
      restriction to $\hom^{\alpha}(A,B)$ is the map
      $i_{\alpha,\alpha+r}:\hom_{C}^{\alpha}(A,B) \longrightarrow
      \hom_{C}^{\alpha+r}(A,B)$, $\alpha\in \R$.)
    \item[-] $r$-mono if for all $x\in U$ with $F(x)=0$ we have
      $i_{r}(x)=0$.
    \end{itemize}
    The proof is based on properties of right and left inverses that
    are contained in the following statement.}

  \ocnote{
    \begin{lem}\label{lem:inv-rs} Let
      $\phi\in \hom_{\mathcal{C}_{0}}(M,N)$.
      \begin{itemize}
      \item[(i)] $\phi$ admits a right $r$-inverse if and only if
        $H_{Z}(\phi)$ is $r$-epi for all
        $Z\in \mathcal{O}b(\mathcal{C})$. The existence of a right
        $r$-inverse implies that $H^{\ast}_{Z}(\phi)$ is $r$-mono for
        all $Z$.
      \item[(ii)] $\phi$ admits a left $r$-inverse if and only if
        $H^{\ast}_{Z}(\phi)$ is $r$-epi for all
        $Z\in \mathcal{O}b(\mathcal{C})$. If such a left $r$-inverse
        exists, then $H_{Z}(\phi)$ is $r$-mono for all $Z$.
      \item[(iii)] Consider a morphism of exact triangles as in
        (\ref{eq:diag-ff1}). Assume that for every
        $Z\in \mathcal{O}b(\mathcal{C})$ $H_{Z}(u)$ is $r$-epi and
        $s$-mono and that $H_{Z}(v)$ is $r'$-epi and $s'$-mono. Then
        $H_{Z}(w)$ is $(r'+r+s')$-epi and $(s+s'+r)$-mono, for all
        $Z$.
      \item[(iv)] If $\phi$ admits a right $r$-inverse and a left
        $s$-inverse, then $\phi$ is a $(r+s)$-isomorphism.
      \end{itemize}
    \end{lem}
  }

  \begin{proof}[Proof of the Lemma.]
    \ocnote{We start with (i). Assume that $\phi$ admits a right
      $r$-inverse $\psi \in
      \hom_{\mathcal{C}_{0}}(\Sigma^{r}N,M)$. Let
      $h\in \hom_{\C}^{\alpha}(Z,N)$.  We have
      $h\circ \eta^{Z}_{r}=\eta_{r}^{N}\circ \Sigma^{r}h=\phi\circ
      (\psi\circ \Sigma^{r}h)=H_{Z}(\phi)(\psi\circ \Sigma^{r}h)$ with
      $\psi\circ \Sigma^{r}h\in
      \hom_{\C}^{\alpha}(\Sigma^{r}Z,M)=\hom^{r+\alpha}_{\C}(Z,M)$.}
    \pbrrvv{Therefore
      $$i_{\alpha, \alpha+r}(h) = (h \circ \eta^Z_r) \circ (\eta_{0,r})_Z
      = \phi \circ \bigl( \psi \circ \Sigma^rh \circ (\eta_{0,r})_Z
      \bigr).$$} \ocnote{It follows that $H_{Z}(\phi)$ is
      $r$-epi. Conversely, assume that $H_{N}(\phi)$ is $r$-epi. Then
      $\eta^{N}_{r}$ is in the image of $H_{N}(\phi)$ which means that
      $\phi$ admits a right $r$-inverse. To finish with (i) let
      $k\in\hom_{\C}^{\alpha}(N,Z)$ such that $k\circ \phi=0$. We
      write
      $0=k\circ \phi =k\circ \phi\circ \psi=k\circ
      \eta^{N}_{r}=\eta^{Z}_{r}\circ \Sigma^{r}k$ which means
      $i_{r}(k)=0\in \hom_{\C}^{r+\alpha}(N,Z)$. Thus
      $H_{Z}^{\ast}(\phi)$ is $r$-mono.

      The point (ii) is perfectly similar to (i).
  
      For the point (iii) we first notice that, by assumption, the
      maps $H_{Z}(u)$ and $H_{Z}(v)$ satisfy the epi and mono
      conditions with constants that are the same for all objects $Z$
      in $\C$. It is immediate to see that $H_{Z}(Tu)$ is $r$-epi
      (respectively, $s$-mono) if and only if $H_{T^{-1}Z}(u)$ is
      $r$-epi (and, respectively, $s$-mono).  This implies that
      $H_{Z}(T^{l}u)$ is $r$-epi and $s$-mono, and that
      $H_{Z}(T^{l}v)$ is $r'$-epi and $s'$-mono for all $l\in \Z$ (and
      all $Z$).  We now apply the exact functor $H_{Z}$ to the
      diagram~\eqref{eq:diag-ff1} and we obtain two long exact
      sequences of persistence modules related by comparison
      morphisms. The desired conclusion follows by direct diagram
      chasing, as in the proof of the classical five-lemma.}

    \pbrrvv{For the point (iv), we use the triangulated structure of
      $\mathcal{C}_0$ to obtain an object $K$ and the following
      commutative diagram in $\C_{0}$:
      \begin{equation} \label{eq:diag-ff2} \xymatrix{ M \ar[r]^{id}
          \ar[d]_-{id} & M \ar[r] \ar[d]_-{\phi} & 0 \ar[r]
          \ar[d]_-{p}
          & TM \ar[d]_-{id}\\
          M\ar[r]^-{\phi} & N\ar[r] & K \ar[r] & TM}
      \end{equation}
    } \ocnote{whose rows are exact triangles. Given that $\phi$ admits
      a right $r$-inverse, we deduce from (i) that $H_{Z}(\phi)$ is
      $r$-epi for all objects $Z$ in $\C$. The existence of a left
      $s$-inverse implies, by (ii), that $H_{Z}(\phi)$ is also
      $s$-mono for all $Z$. We now use (iii) to deduce that
      $H_{K}(p)$-is $(r+s)$-epi. This implies that $i_{r+s}(id_{K})$
      is in the image of $H_{K}(p)$.  But $p$ is the null map and thus
      $i_{r+s}(id_{K})=0$, hence $K\simeq_{r+s}0$. It follows that
      $\phi$ is an $(r+s)$-isomorphism.}
  \end{proof}
  \ocnote{We now return to the proof of the second point of the
    proposition with the notation and the assumptions there. We denote
    by $K$ any object that completes the map $w$ to an exact
    triangle $$C\stackrel{w}{\longrightarrow} C'\to K\to TC~.~$$ An
    $r$-isomorphism admits both right and left $r$-inverses. Thus, the
    points (i) and (ii) of the Lemma show that $H_{Z}(u)$ is $r$-epi
    and $r$-mono and that $H_{Z}(v)$ is $s$-epi and $s$-mono for all
    objects $Z$ in $\C$. The point (iii) of the Lemma then implies
    that $H_{Z}(w)$ is $(2s+r)$-epi and $(2r+s)$-mono for all $Z$.  We
    now consider a diagram just as (\ref{eq:diag-ff2}) but with $M=C$,
    $N=C'$, $\phi=w$ and we use the point (iii) of the Lemma to deduce
    that the map $H_{K}(p)$ is $3(r+s)$-epi which means that $K$ is
    $3(r+s)$-acyclic.}
 \end{proof}

\begin{cor} \label{cor-2} If $f: A \to B$ is an $r$-isomorphism, then
  any right inverse $\psi \in \Mor^0(\Sigma^r B, A)$ (given by (i) in
  Proposition \ref{prop-r-iso}) is a $2r$-isomorphism. The same
  conclusion holds for any left inverse.
\end{cor}

\begin{proof} By the octahedral axiom (in $\C_0$), we have the
  following commutative diagram,
  \[ \xymatrix{ \Sigma^r B \ar[r] \ar[d]_-{\psi}
      & \Sigma^r B \ar[r] \ar[d]^-{\eta^B_r} & 0 \ar[d] \\
      A \ar[r] \ar[d] & B \ar[r] \ar[d] & K \ar[d]\\
      K'' \ar[r] & K' \ar[r] & K }\] where $K'' \to K' \to K \to TK''$
  is exact. By (iii) in Definition \ref{dfn-tpc}, $K' \simeq_r
  0$. Therefore, by Lemma \ref{lemma-exact}, $K'' \simeq_{2r} 0$ and
  thus $\psi$ is $2r$-isomorphism. A similar argument applies to the
  left inverse of $f$.
\end{proof} 

\begin{rem}\label{rem:inverses} The fact that left and right
  inverses of $r$-isomorphisms are only $2r$-isomorphisms has
  significant impact on the various algebraic properties of
  TPCs. However, this seems unavoidable. For example, it is easy to
  construct examples of $r$-isomorphisms in the (homotopy) category of
  filtered \zjnote{co}chain complexes that admit a unique right inverse that is
  no less than a $2r$-isomorphism.
\end{rem}

The next consequence is immediate but useful so we state it apart.
\begin{cor} \label{cor-3} If $f: A \to B$ is an $r$-isomorphism, then
  for any \zjnote{$u,u' \in \Mor_{\mathcal C}^0(B, C)$} with $u \circ f = u' \circ f$, we
  have $u \simeq_r u'$, i.e., $u$ and $u'$ are $r$-equivalent.
  Similarly, if \zjnote{$v,v'\in \Mor_{\mathcal C}^0(D, A)$} and $f\circ v=f\circ v'$, then
  $v\simeq_{r} v'$.
\end{cor}

\begin{cor} Assume that the following diagram in $\C_0$,
 \jznote{ \[ \xymatrix{ K \ar[r] \ar[d]_{\mathds 1_{K}} & A \ar[r]^-{\phi} \ar[d]_-{f}
      & A' \ar[r] \ar[d]_-{f'} & TK \ar[d]_{\mathds{1}_{TK}}\\
      K \ar[r] & B \ar[r]^{\psi} & B' \ar[r] & TK}\]} is commutative,
  that the two rows are exact and that $K \simeq_r 0$. Then the
  induced morphism $f'$ is unique up to $r$-equivalence.
\end{cor}

\begin{proof} Since $K \simeq_r 0$, by definition, $\phi$ is an
  $r$-isomorphism. For any two induced morphisms
  \zjnote{$f'_1, f'_2 \in \Mor_{\mathcal C}^0(A', B')$}, we have
  $f'_1 \circ \phi = f'_2 \circ \phi = \psi\circ f$ and the conclusion
  follows from Corollary \ref{cor-3}.
\end{proof}

\begin{cor} \label{cor5} Let $\phi: A \to A'$ be an
  $r$-isomorphism. Then for any $f \in \Mor_{\C}^0(A,B)$, there exists
  $f' \in \Mor_{\C}^0(A', \Sigma^{-r}B)$ such that the following diagram
  commutes in $\C_0$.
  \[ \xymatrix{
      A \ar[r]^-{\phi} \ar[d]_-{f} & A' \ar[d]^-{f'} \\
      B \ar[r]^-{\eta^B_r} & \Sigma^{-r} B} \]
\end{cor}

\begin{proof} Since $\phi: A \to A'$ is an $r$-isomorphism, there
  exists a left $r$-inverse denoted by $\psi: A' \to \Sigma^{-r} A$
  such that $\psi \circ \phi = \eta^A_r$. Set
  $f' : = \Sigma^{-r}f \circ \psi \in \Mor_{\C}^0(A', \Sigma^{-r} B)$.
\end{proof}
Similar direct arguments lead to the next consequence.
\begin{cor} \label{cor-6} Consider the following commutative diagram
  in $\C_0$,
  \[ \xymatrix{
      A \ar[r]^{\phi} \ar[d]_-{f} & A' \ar[d]^-{f'} \\
      B \ar[r]^-{\phi'} & B'} \] where $f \in \Mor_{\C}^0(A,B)$,
  $f' \in \Mor_{\C}^0(A'.B')$ and $\phi, \phi'$ are $r$-isomorphisms. Let
  $\psi, \psi'$ be any left inverses of $\phi, \phi'$
  respectively. Then the following diagram is $r$-commutative
  \[ \xymatrix{
      A' \ar[r]^-{\psi} \ar[d]_-{f'} & \Sigma^{-r} A \ar[d]^-{\Sigma^{-r} f} \\
      B' \ar[r]^-{\psi'} & \Sigma^{-r} B} \] in the sense that
  $\Sigma^{-r} f \circ \psi \simeq_r \psi' \circ f'$. A similar
  conclusion holds for right inverses.
\end{cor}

\subsection{Relation to Verdier
  localization}\label{subsubsec:localization}

\begin{prop}\label{prop:verdier-loc} Let $\C$ be a triangulated
  persistence category and
  let $\mathcal{A}\C$ be the full subcategory of $\C$ with objects the
  $r$-acyclic objects of $\C$ (for all $r\geq 0$).
  \begin{itemize}
  \item[(i)] The category $\mathcal{A}\C$ is a triangulated persistence
    category on its own with $0$-level \zjnote{denoted by $\mathcal{A}\C_{0} ( = (\mathcal{A}\C)_0)$}, the full
    subcategory of $\C_{0}$, having as objects the objects in
    $\mathcal{A}\C$.
  \item[(ii)] The infinity level, $\C_{\infty}$, of $\C$ coincides with
    the Verdier quotient (a.k.a.~localization)
    \zjnote{$\C_0 / \mathcal{A}\C_{0}$} of $\C_{0}$ by \zjnote{$\mathcal{A}\C_{0}$}. In
    particular, $\C_{\infty}$ is triangulated.
  \end{itemize}
\end{prop}  
 
\begin{rem}
  \begin{enumerate}
  \item The collection $S$ of all $r$-isomorphisms (for every
    $r \geq 0$) forms a multiplicative system in $\C_0$. The Verdier
    quotient above is the same as the localization $S^{-1}\C_0$ of
    $\C_0$ by $S$. \zjnote{For a definition of the localization, see subsection 1.6 in \cite{KS90}.} 
  \item A result somewhat similar to the
    Proposition~\ref{prop:verdier-loc}, established for Tamarkin
    categories, appears \zjnote{in Section 6 in~\cite{GS14}, more precisely, its Proposition 6.7}. 
  \end{enumerate}
\end{rem}
 
\begin{proof} For the first point of the proposition we first notice
  that the subcategory of acyclics, $\mathcal{A}\C$, is a persistence
  category. It is obviously endowed with a shift functor by
  restricting the shift functor of $\C$. Moreover, its $0$-level
  clearly is a full subcategory of $\C_{0}$.  Finally, Lemma
  \ref{lemma-exact} implies that \zjnote{$\mathcal A\C_{0}$} is a triangulated
  subcategory of $\C_{0}$ which \zjnote{implies} that it is a TPC.
 
  We pursue with the second point of the proposition.  By inspecting
  the definition of Verdier localization (for instance in Chapter 2 of
  \cite{Neeman-tribook}) we see that the localization of $\C_{0}$ at
  $\mathcal{A}\C_{0}$ - denoted by $\C_{0}/\mathcal{A}\C_{0}$ - is a
  category with the same objects as $\C_{0}$ and having as morphisms
  $A\to B$ equivalence classes of \hhats:
  $$ A \stackrel{u}{\longleftarrow} A'\stackrel{f'}{\longrightarrow} B$$
  with $u$ an $r$-isomorphism, for some $r\geq 0$, and
  $f'\in \mor_{\C_{0}}(A',B)$. Two {\hhats}
   $A \stackrel{u}{\longleftarrow} A'\stackrel{f'}{\longrightarrow} B$
  and
  $A \stackrel{u_{1}}{\longleftarrow}
  A_{1}'\stackrel{f'_{1}}{\longrightarrow} B$ are equivalent if they
  are related by a third {\hhat}
  $A \stackrel{u_{2}}{\longleftarrow}
  A_{2}'\stackrel{f'_{2}}{\longrightarrow} B$ in the sense that there
  are maps $a', a'_{1}$ in $\C_{0}$ making the following diagram
  commutative:
  \[ \xymatrix{ A   &A' \ar[l]_{u}\ar[r]^{f'}  & B \\
      A \ar[d]^{\mathds{1}_{A}}\ar[u]_{\mathds{1}_{A}} & A'_{2}\ar[l]_{u_{2}}
      \ar[r]^{f'_{2}} \ar[d]^{a'_{1}}\ar[u]_{a'}  & B \ar[u]_{\mathds{1}_{B}}\ar[d]^{\mathds{1}_{B}}\\
      A & A'_{1}\ar[l]_{u_{1}} \ar[r]^{f'_{1}} & B }\]
 
  The category $\C_{\infty}$ appears in Definition \ref{dfn-c0-cinf}.
  Its objects are the same as those of $\C_{0}$ and its morphisms are
 \zjnote{$\Mor_{\C_{\infty}}(A,B)=\varinjlim_{r \to \infty}
  \Mor^r_{\C}(A,B)= \varinjlim_{r
    \to\infty}\Mor_{\C_{0}}(\Sigma^{r}A,B)$} where the second
  equality comes from formula (\ref{nt-iso-2}). Given a morphism
  $f\in \Mor_{\C_{\infty}}(A,B)$ this means that $\bar{f}$ is
  represented by $\bar{f}: \Sigma^{r}A\to B$ for some $r\geq 0$ as
  well as by all compositions
  $\Sigma^{s}A\stackrel{\eta_{s-r}}{\longrightarrow}
  \Sigma^{r}A\stackrel{\bar{f}}{\longrightarrow}B$. We now define a
  functor
  $$\Phi :\C_{\infty}\to \C_{0}/\mathcal{A}\C_{0}~.~$$
  It is the identity on objects and for a morphism
  $f\in \Mor_{\C_{\infty}}(A,B)$ we let $\Phi(f)$ be the equivalence
  class of a {\hhat}
  $A\stackrel{\eta_{r}}{\longleftarrow}\Sigma^{r}
  A\stackrel{\bar{f}}{\longrightarrow} B$ where
  $\bar{f}:\Sigma^{r} A\to B$ represents $f$. Any two {\hhats} that are
  associated to two representatives of $f$ are immediately seen to be
  equivalent and as a result $\Phi$ is well-defined.

  It remains to show that $\Phi$ is an isomorphism. Surjectivity is
  immediate.  Fix $H$ a {\hhat}
  $A\stackrel{a}{\longleftarrow} A'\stackrel{b}{\longrightarrow}
  B$. As $a$ is an $r$-isomorphism (for some $r$) we deduce from
  Proposition~\ref{prop-r-iso} the existence of a right $r$-inverse
  $\phi:\Sigma^{r}A\to A'$ of $a$ such that
  \zjnote{$a\circ \phi = \eta^A_{r}$}. Therefore, we may define a new
  {\hhat} $H'$,
  $A\stackrel{\eta_{r}}{\longleftarrow} \Sigma^{r} A\stackrel{b\circ
    \phi}{\longrightarrow} B$. The {\hhats} $H'$ and $H$ are clearly
  equivalent and thus their equivalence class belongs to the image of
  $\Phi$.

  We now show that $\Phi$ is injective. For this we consider the
  commutative diagram below:
  \begin{equation}\label{eq:diagr-v1} \xymatrix{ A   &\Sigma^{r'}A \ar[l]_{\eta^A_{r'}}\ar[r]^{f''}  & B  \\
      A \ar[u]_{\mathds{1}_A} \ar[d]^{\mathds{1}_A} & A'\ar[l]_{u}  \ar[u]_{a''}\ar[r]^{f} \ar[d]^{a'}  & B \ar[d]^{\mathds{1}_B} \ar[u]_{\mathds{1}_B}\\
      A & \Sigma^{r} A\ar[l]_{\eta^A_{r}} \ar[r]^{f'} & B }
  \end{equation}
  with each row a \hhat. We need to show that $f''$ and $f'$ represent
  the same element in $\C_{\infty}$. \zjnote{In fact,} the map $u$ is an $s$-isomorphism for some $s\geq 0$ and let
  $v:\Sigma^{s}A\to A'$ be a right inverse of $u$.  Now consider the
  commutative diagram:
  \begin{equation}\label{eq:diagr-v1} \xymatrix{ A   &\Sigma^{r'}A \ar[l]_{\eta^A_{r'}}\ar[r]^{f''}  & B  \\
      A \ar[u]_{\mathds{1}_A} \ar[d]^{\mathds{1}_A} & \Sigma^{s} A \ar[l]_{\eta_{s}}  \ar[u]_{\bar{a}''}\ar[r]^{\bar{f}} \ar[d]^{\bar{a}'}  & B \ar[d]^{\mathds{1}_B} \ar[u]_{\mathds{1}_B}\\
      A & \Sigma^{r} A\ar[l]_{\eta^A_{r}} \ar[r]^{f'} & B }
  \end{equation}  
  where $\bar{a}'=a'\circ v$, $\bar{a}''=a''\circ v$,
  $\bar{f}=f\circ v$.  Notice that we have
  \zjnote{$\eta^A_{r}\circ \bar{a}'=\eta^A_{s}=\eta_{r}\circ \eta^A_{s-r}$}.  This
  means by Corollary \ref{cor-3} that \zjnote{$\bar{a}'\simeq_{r}\eta^A_{s-r}$}.
  Similarly, we have \zjnote{$\bar{a}''\simeq_{r'} \eta^A_{s-r'}$}. For
  $r'' \geq \max \{r, r'\}$ we deduce
  \zjnote{$\bar{a}'\circ \eta^A_{r''}=\eta^A_{s-r+r''}$}, and
  \zjnote{$\bar{a}''\circ \eta^A_{r''}=\eta^A_{s-r'+r''}$}.  Thus, by composing on
  the middle node with \zjnote{$\eta^A_{r''}:\Sigma^{s+r''}A\to \Sigma^{s} A$},
  we get a new commutative diagram similar to the one above but with
  $\bar{a}''$ and $\bar{a}'$ being replaced with \zjnote{$\eta^A_{s-r+r''}$}
  and, respectively, \zjnote{$\eta^A_{s-r+r'}$}. We deduce that $f''$ and $f'$
  give in $\C_{\infty}$ the same element as \zjnote{$\bar{f}\circ \eta^A_{r''}$}
  which concludes the proof.
\end{proof}
 
\begin{rem}\label{rem:ex-tr-Cinfty} By the properties of Verdier
  localization, the category $\C_{\infty}$ is triangulated in such a
  way that, by definition, a triangle in $\C_{\infty}$ is exact if it
  is isomorphic with the image (in $\C_{\infty}$) of an exact triangle
  from $\C_{0}$.
 \end{rem}

\subsection{Weighted exact triangles}\label{subsubsec:weight-tr}
The key feature of a triangulated persistence category $\C$ is that
there is a natural way to associate weights to a class of triangles
larger than the exact triangles in $\C_{0}$.

\begin{dfn} \label{dfn-set} A {\em strict exact triangle in $\C$} is a
  pair $\widetilde{\Delta}=(\Delta, r)$ where $r\in [0, +\infty)$ and
  $\Delta$ is a diagram
  \begin{equation} \label{dfn-set-1} \Delta \ : \ A
    \xrightarrow{\bar{u}} B \xrightarrow{\bar{v}} C
    \xrightarrow{\bar{w}} \Sigma^{-r} T A
  \end{equation}
  in $\C_0$ with $\bar{u}\in \Mor_{\C}^0(A,B)$,
  $\bar{v} \in \Mor_{\C}^0(B,C)$ and
  $\bar{w} \in \Mor_{\C}^0(C, \Sigma^{-r} TA)$, such that the
  following holds. There exists an exact triangle
  $A \xrightarrow{u} B \xrightarrow{v} C' \xrightarrow{w} TA$ in
  $\C_0$, with $u = \bar{u}$, an $r$-isomorphism $\phi: C' \to C $ and
  a right $r$-inverse of $\phi$ denoted by $\psi: \Sigma^r C \to C'$
  such that the diagram
  \begin{equation}\label{dfn-set-2} 
    \xymatrix{
      & & \Sigma^r C \ar[d]_-{\psi} \ar[rd]^-{\Sigma^{r} \bar{w}} & \\
      A \ar[r]^-{u} & B \ar[r]^-{v} \ar[rd]_-{\bar{v}}
      & C' \ar[r]^-{w} \ar[d]^-{\phi} & TA\\
      & & C &} 
  \end{equation}
  commutes.
  The weight of the strict exact triangle $\widetilde{\Delta}$ is the
  number $r$ and is denoted by $w(\widetilde{\Delta})$.
\end{dfn}

\begin{remark} \label{rmk- set} (a) To simplify terminology we will
  often denote strict exact triangles by the diagram $\Delta$ with the
  weight identified implicitly by the amount of down ``shift'' of the
  last term. Notice that, if $\Sigma^{s}A\not=A$ for all $s$, then the
  diagram $\Delta$ determines the weight of the triangle. However,
  when this is not the case, it is necessary to indicate the weight
  explicitly. For example, for any $r\geq 0$, the pair \zjnote{$(0\to X\xrightarrow{\mathds{1}_X} X\to 0, r)$ is a strict exact triangle of weight $r$ because $\Sigma^s 0 = 0$ in $\C_0$ and two  such triangles  are different as soon as the corresponding weights are different.} In what
  follows, we will not always write strict exact triangles as
  pairs. We will often simply write that a diagram $\Delta$ as above
  is strict exact of weight $w(\Delta) = r$. Although there is a
  slight imprecision in writing $w(\Delta) = r$ (since $\Delta$ does
  not determine $r$) the meaning of this should be clear:
  $(\Delta, r)$ is a strict triangle of weight $r$.
  
  (b) Any exact triangle in $\C_0$ is a strict exact triangle of
  weight $0$. Conversely, it is a simple exercise to see that a strict
  exact triangle of weight $0$ is exact as a triangle in $\C_0$.

  (c) Consider the following diagram
  \[ \xymatrix{
      &\Sigma^{r}B\ar[d]_{\eta_{r}^{B}}\ar[r]^{\Sigma^{r}\bar{v}}
      & \Sigma^r C \ar[d]_-{\psi} \ar[rd]^-{\Sigma^{r} \bar{w}} & \\
      A \ar[r]^-{u} & B \ar[r]^-{v} \ar[rd]_-{\bar{v}}
      & C' \ar[r]^-{w} \ar[d]^-{\phi} & TA \ar[d]^-{\eta^{TA}_r}\\
      & & C \ar[r]^-{\bar{w}} &\Sigma^{-r} TA} \] which is derived
  from the commutative diagram (\ref{dfn-set-2}). The two squares in
  the diagram are not commutative, in general, but they are
  $r$-commutative.  Indeed, since $\phi$ is an $r$-isomorphism let
  $\tilde{\psi}:C\to \Sigma^{-r}C'$ be a left $r$-inverse of $\phi$.
  As $\psi$ is a right \jznote{$r$-inverse} of $\phi$ we deduce from Proposition
  \ref{prop-r-iso} (i) that $\Sigma^{-r}\psi \simeq_{r}
  \tilde{\psi}$. Therefore,
  $\bar{w}\circ \phi=\Sigma^{-r}w\circ \Sigma^{-r}\psi\circ \phi
  \simeq_{r} \Sigma^{-r}w\circ \tilde{\psi}\circ \phi= \Sigma^{-r}w
  \circ \eta_{r}^{C'}=\eta_{r}^{TA}\circ w$.  Using Corollary
  \ref{cor-3}, we also see that
  $\psi \circ \Sigma^{r}\bar{v}\simeq_{r} v\circ \eta_{r}^{B}$ because
  $\phi\circ \psi \circ \Sigma^{r}\bar{v}= \phi \circ v\circ
  \eta_{r}^{B}$.

 \zjr{(d)} Because $T$ commutes with $\Sigma$ and with the natural
  transformations $\eta$, it immediately follows that this functor
  preserves strict exact triangles as well as their weight.
\end{remark}

\begin{ex} \label{ex-w-r-tr} Recall that \pbnote{the map
    $\Sigma^rA \xrightarrow{\eta^A_r} A$ embeds into an exact triangle
    in $\C_0$,}
  $\Sigma^r A \xrightarrow{\eta^A_r} A \xrightarrow{v} K
  \xrightarrow{w} T\Sigma^r A$ \pbnote{, where $K$ is $r$-acyclic.}
  We claim that the diagram
  \[ \Sigma^r A \xrightarrow{\eta^A_r} A \xrightarrow{v} K
    \xrightarrow{0} TA \] is a strict exact triangle of weight
  $r$. Indeed, we have the following commutative diagram,
  \begin{equation*}
    \xymatrix{
      & & \Sigma^r K \ar[d]_-{\eta^K_r} \ar[rd]^-{0} & \\
      \Sigma^r A \ar[r]^-{\eta^A_r} & A \ar[r]^-{v} \ar[rd]_-{v}
      & K \ar[r]^-{w} \ar[d]^-{\mathds 1_K} & T\Sigma^r A\\
      & & K &} 
  \end{equation*}
  where the right upper triangle is commutative since $K$ is
  $r$-acyclic (so $\eta^K_r =0$ \jznote{by Lemma \ref{lemma-acyclic} (i)}). Moreover, $\mathds{1}_K$ is an
  $r$-isomorphism (we recall $T\Sigma^{r} A=\Sigma^{r}TA$).

  \pbnote{Note that also the following diagram
    $$\Sigma^r A  \xrightarrow{\eta^A_r} A \xrightarrow{} 0
    \xrightarrow{} TA$$ is a strict exact triangle of weight $r$.}
\end{ex}

\begin{prop} [Weight invariance] \label{prop-wt-1} Strict exact
  triangles satisfy the following two properties:

 \zjnote{ (i) Suppose the two diagrams $A \xrightarrow{\bar{u}} B \xrightarrow{\bar{v}} C$ and $A' \xrightarrow{\bar{u}'} B' \xrightarrow{\bar{v}'} C'$ are isomorphic in $\mathcal C_0$, i.e., we have the following commutative diagram in $\mathcal C_0$, 
 \begin{equation} \label{prop-wt-1-diag}
  \xymatrix{
      A \ar[r]^-{\bar{u}} \ar[d]^-{f} & B \ar[r]^-{\bar{v}} \ar[d]^-{g} & C \ar[d]^-{h} \\
      A' \ar[r]^-{\bar{u}'} & B' \ar[r]^-{\bar{v}'} & C'} 
      \end{equation}
 then $A \xrightarrow{\bar{u}} B \xrightarrow{\bar{v}} C$ completes to a strict exact triangle of weight $r$, denoted by $\Delta: A \xrightarrow{\bar{u}} B \xrightarrow{\bar{v}} C \xrightarrow{\bar{w}} \Sigma^{-r} T A$ if and only if $A' \xrightarrow{\bar{u}'} B' \xrightarrow{\bar{v}'} C'$ completes to a strict exact triangle of weight $r$, denoted by $\Delta': A' \xrightarrow{\bar{u}'} B' \xrightarrow{\bar{v}'} C' \xrightarrow{\bar{w}'} \Sigma^{-r} T A'$. Moreover, $\Delta$ and $\Delta'$ are isomorphic in $\C_0$.}

  (ii) If
  $\Delta: A \xrightarrow{\bar{u}} B \xrightarrow{\bar{v}} C
  \xrightarrow{\bar{w}} \Sigma^{-r} TA$ satisfies $w(\Delta) = r$,
  then
  $\Delta' : A \xrightarrow{\bar{u}} B \xrightarrow{\bar{v}} C
  \xrightarrow{\bar{w}'} \Sigma^{-r-s} TA$ satisfies
  $w(\Delta') = r+s$ for $s \geq 0$, where $\bar{w}'$ is the
  composition
  $$C \xrightarrow{\bar{w}} \Sigma^{-r} TA
  \xrightarrow{\eta_{s}^{\Sigma^{-r}TA}} \Sigma^{-r-s}TA~.~$$
\end{prop}

\begin{proof} (i) The property
  claimed here immediately follows from the fact that, within $\C_0$,
  all $0$-isomorphisms admit inverses. \zjnote{Therefore, if $A \xrightarrow{\bar{u}} B \xrightarrow{\bar{v}} C$ completes to a strict exact triangle $\Delta$, then the desired map $\bar{w}'$ for $\Delta'$ can be chosen as $\bar{w}' = T(f) \circ \bar{w} \circ h^{-1}$ where $h^{-1}$ is the inverse of map $h$ in (\ref{prop-wt-1-diag}). The weight can be easily deduced from Definition \ref{dfn-set}. }

  \medskip

  (ii) By definition, there exists a commutative diagram,
  \[ \xymatrix{
      & & \Sigma^r C \ar[d]_-{\psi} \ar[rd]^-{\Sigma^{r} \bar{w}} & \\
      A \ar[r]^-{\bar{u}} & B \ar[r]^-{v} \ar[rd]_-{\bar{v}}
      & C' \ar[r]^-{w} \ar[d]^-{\phi} & TA\\
      & & C &} \] Consider $\psi' \in \Mor^r(\Sigma^{r+s}C, C')$
  defined by $\psi' = \psi \circ \eta^{\Sigma^{r}C}_{s}$. Then $\psi'$
  is an $(r+s)$-right inverse of $\phi$.  Consider the following
  diagram
  \[ \xymatrix{
      & & \Sigma^{r+s} C \ar[d]_-{\psi'} \ar[rd]^-{\Sigma^{r+s} \bar{w}'} & \\
      A \ar[r]^-{\bar{u}} & B \ar[r]^-{v} \ar[rd]_-{\bar{v}}
      & C' \ar[r]^-{w} \ar[d]^-{\phi} & TA\\
      & & C &} \] where $\bar{w}': C\to \Sigma^{-r-s}TA$ is defined by
  $\bar{w}'=\eta^{\Sigma^{-r}TA}_{s}\circ \bar{w}$ and notice that the
  right upper triangle is commutative.
\end{proof}

\begin{prop}[Weighted rotation property] \label{prop-rot} Given a
  strict exact triangle
  $$\Delta: A \xrightarrow{\bar{u}} B \xrightarrow{\bar{v}} C
  \xrightarrow{\bar{w}} \Sigma^{-r} TA$$ satisfying $w(\Delta) = r$,
  there exists a triangle
  \begin{equation} \label{rot-ext-tr} R(\Delta): B
    \xrightarrow{\bar{v}} C \xrightarrow{\bar{w}'} \Sigma^{-r} TA
    \xrightarrow{-\bar{u}'} \Sigma^{-2r} TB
  \end{equation} 
  satisfying \jznote{$w(R(\Delta)) = 2r$}, where $\bar{w}' \simeq_r \bar{w}$
  and $\bar{u}'$ is the composition
  $$\bar{u}':\Sigma^{-r}TA \xrightarrow{\Sigma^{-r} T\bar{u}} \Sigma^{-r} TB 
  \xrightarrow{\eta_{r}^{\Sigma^{-r}TB}} \Sigma^{-2r}TB~.~$$
\end{prop}
We call \jznote{$R(\Delta)$} the (first) positive rotation of $\Delta$.

\begin{proof} By definition, there exists a commutative diagram,
  \[ \xymatrix{
      & & \Sigma^r C \ar[d]_-{\psi} \ar[rd]^-{\Sigma^{r} \bar{w}} & \\
      A \ar[r]^-{\bar{u}} & B \ar[r]^-{v} \ar[rd]_-{\bar{v}}
      & C' \ar[r]^-{w} \ar[d]^-{\phi} & TA\\
      & & C &} \] where
  $A \xrightarrow{\bar{u}} B \xrightarrow{v} C' \xrightarrow{w} TA$ is
  an exact triangle in $\C_0$ \pbnote{and $\psi$ is a right
    $r$-inverse of $\phi$.} By the rotation property of $\C_0$,
  $B \xrightarrow{v} C' \xrightarrow{w}TA \xrightarrow{-T\bar{u}} TB$
  is an exact triangle in $\C_0$. \pbnote{We now construct the
    following diagram in $\C_0$ in which the upper squares will be
    commutative and the lower square $r$-commutative:}
  \pbnote{
    \begin{equation}\label{eq:foursq}\xymatrixcolsep{4pc} \xymatrix{
        B \ar[r]^-{v} \ar[d]_-{\mathds 1_B}
        & C' \ar[r]^-{w} \ar[d]_-{\phi}
        & TA \ar[r]^{-T\bar{u}} \ar[d]_-{\bar{\phi}}
        & TB \ar[d]_-{\mathds 1_{TB}} \\
        B \ar[r]^-{\bar{v}}
        & C \ar[r]^-{w'} \ar[d]_-{\Sigma^{-r} \psi}
        & TA' \ar[r]^-{u'} \ar[d]_-{\bar{\psi}} & TB \\
        & \Sigma^{-r} C' \ar[r]^-{\Sigma^{-r} w} & \Sigma^{-r} TA} 
    \end{equation}
  } \pbnote{Here the second row of maps comes from embedding
    $B \xrightarrow{\bar{v}} C$ into an exact triangle
    $B \xrightarrow{\bar{v}} C \xrightarrow{w'} A'' \to TB$ for some
    $A''$ in $\C_0$ and $A'=T^{-1}A''$. The map $\bar{\phi}$ is then
    induced by the functoriality of triangles in $\C_{0}$ and is an
    $r$-isomorphism by \ocnote{Proposition \ref{cor-1} i}. So far this gives the
    upper three squares of the diagram and their commutativity. To
    construct the lower square, let $\bar{\psi}$ be a left $r$-inverse
    of $\bar{\phi}$
    (i.e.~$\bar{\psi} \circ \bar{\phi} = \eta^{TA}_r$). By
    Corollary~\ref{cor-2}, $\bar{\psi}$ is a $2r$-isomorphism.}

  \pbnote{We claim that the lower square in diagram~\eqref{eq:foursq}
    is $r$-commutative, and therefore we have
    $\bar{\psi} \circ w'\simeq_{r} \Sigma^{-r}w\circ \Sigma^{-r}\psi =
    \bar{w}$.}

  \pbnote{Indeed, let $\psi'$ be a left $r$-inverse of $\phi$. By
    using the commutativity of the middle upper square in
    diagram~\eqref{eq:foursq}, we deduce
    $\bar{\psi}\circ w'\circ \phi= \Sigma^{-r}w\circ \psi'\circ
    \phi$. As $\phi$ is an $r$-isomorphism we obtain
  \begin{equation} \label{eq:lsquare} \bar{\psi}\circ w' \simeq_{r}
    \Sigma^{-r}w\circ \psi'\simeq_{r} \Sigma^{-r}w\circ
    \Sigma^{-r}\psi = \bar{w},
  \end{equation}
  because, by Proposition~\ref{prop-r-iso}, we have
  $\Sigma^{-r}\psi \simeq_{r} \psi'$. This shows the lower square is
  $r$-commutative and the related $r$-identity.}

  We next consider the following diagram
  \[ \xymatrix{ & & \Sigma^r TA \ar[d]_-{\bar{\phi} \circ \eta^{TA}_r}
      \ar[rd]^-{\Sigma^{2r} \bar{u}'} & \\
      B \ar[r]^-{\bar{v}} & C \ar[r]^-{w'} \ar[rd]_-{\bar{w}'}
      & TA' \ar[r]^-{u'} \ar[d]_{\bar{\psi}} & TB\\
      & & \Sigma^{-r} TA &} \] where
  $\bar{u}'= \Sigma^{-2r}(u'\circ \bar{\phi}\circ \eta_{r}^{TA})$ and
  $\bar{w}'=\bar{\psi}\circ w'$.  Given that $\bar{\psi}$ is a
  $2r$-isomorphism, this means that we have a strict exact triangle of
  weight $2r$ of the form:
  $$B\xrightarrow{\bar{v}} C \xrightarrow{\bar{w}'}
  TA'\xrightarrow{\bar{u}'} \Sigma^{-2r}TB$$ We already know
  $\bar{w}'=\bar{\psi}\circ w'\simeq_{r} \bar{w}$. On the other hand,
  \[ \Sigma^{2r} \bar{u}'=u' \circ (\bar{\phi} \circ \eta^{TA}_r) =
    -T\bar{u} \circ \eta^{TA}_r = -\eta^{\Sigma^{r}TB}_r \circ
    \Sigma^r T\bar{u} \] which concludes the proof.
\end{proof}
 
\begin{rem}\label{rem:neg-rot} A perfectly similar argument also
  shows that there exists a strict exact triangle of weight $2r$ and
  of the form:
  $$R^{-1}(\Delta): T^{-1}\Sigma^{r}C\to A\to B \to \Sigma^{-r}C$$
  which is the \jznote{(first)} negative rotation of $\Delta$. \jznote{Note that $R^{-1} (R(\Delta)) \neq R(R^{-1}(\Delta)) \neq \Delta$.} 
\end{rem}

\begin{rem} \label{r:rotation-weight}
  \pbnote{Proposition~\ref{prop-rot} describes a rotation of weighted
    exact triangles that does {\em not} preserve weights. Indeed, the
    rotation of the weight $r$ triangle $\Delta$ from
    Proposition~\ref{prop-rot} has weight $2r$. It is not clear to
    what extent one can improve this. Ideally, one would like to be
    able to rotate $\Delta$ into a weighted exact triangle
    $B \xrightarrow{} C \xrightarrow{} \Sigma^{-r} TA \xrightarrow{}
    \Sigma^{-r} TB$ of the {\em same} weight $r$. There is some
    evidence, coming from symplectic topology, indicating that in
    certain circumstances this might be possible
    (see~\S\ref{sbsb:cobs-rot}). However, the algebraic setting in
    this paper, in particular the definition of weighted exact
    triangles, might be too general to render this feasible, at least
    without additional assumptions on $\Delta$.}
\end{rem}

\begin{prop} [Weighted octahedral formula] \label{prop-w-oct} Given
  two strict exact triangles
  $$\Delta_1: E \xrightarrow{\beta} F \xrightarrow{\alpha} X
  \xrightarrow{k} \Sigma^{-r} TE$$ and
  $$\Delta_2: X \xrightarrow{u} A \xrightarrow{\gamma} B
  \xrightarrow{b} \Sigma^{-s}TX$$ with $w(\Delta_1) = r$ and
  $w(\Delta_2) = s$, there exists a diagram
  \begin{equation} \label{w-oct} \xymatrix{ E \ar[d]_{\beta}\ar[r] & 0
      \ar[r] \ar[d]
      & TE \ar[d] \ar[r]& TE \ar[d]^{T\beta}\\
      F \ar[d]_{\alpha} \ar[r]^{u\circ \alpha} & A \ar[r] \ar[d]_{\mathds{1}}
      & C \ar[d] \ar[r] & TF \ar[d]^{(\Sigma^{-s} T\alpha)\circ \eta^{TF}_{s}} \\
      X \ar[d]_{k} \ar[r]^{u} & A \ar[r]^{\gamma}\ar[d]
      & B \ar[r]^{b} \ar[d] & \Sigma^{-s} TX \ar[d]^{\Sigma^{-s}(Tk) } \\
      \Sigma^{-r} TE \ar[r] &0 \ar[r] &\Sigma^{-r-s} T^{2}E \ar[r] &
      \Sigma^{-r-s} T^{2}E}
  \end{equation}
  with all squares commutative except for the right bottom one that is
  $r$-anti-commutative, such that the triangles
  $\Delta_3: F \to A \to C \to TF$ and
  $\Delta_4: TE \to C \to B \to \Sigma^{-r-s} T^{2}E$ are strict exact
  with $w(\Delta_3) =0$ and $w(\Delta_4) = r+s$.
\end{prop}
By forgetting the $\Sigma$'s (or assuming that $r=s=0$) this is
equivalent to the usual octahedral axiom in a triangulated category
(namely $\C_{0}$) and the right bottom square is commutative up to
sign (or anti-commutative).
\begin{proof} [Proof of Proposition \ref{prop-w-oct}]
  By definition, there are two commutative diagrams,
  \begin{equation} \label{prop-w-oct-1} \xymatrixcolsep{4pc}
    \xymatrix{
      & E \ar[d] & \\
      & F \ar[d]_{\alpha'} \ar[rd]^{\alpha} & \\
      \Sigma^r X \ar[r]^-{\psi_1} \ar[rd]_-{\Sigma^r k}
      & X' \ar[d]_-{\delta} \ar[r]^-{\phi} & X \\
      & TE &} \,\,\,\,\,\,\,\, \xymatrixcolsep{4pc} \xymatrix{
      & & \Sigma^s B \ar[d]_-{\psi_3} \ar[rd]^-{\Sigma^s b} & \\
      X \ar[r]^-{u} & A \ar[r]^-{v'} \ar[rd]_{\gamma}
      & B' \ar[r]^-{g} \ar[d]_-{\bar{\phi}} & TX \\
      & & B &}
  \end{equation}
  with $\phi$ an $r$-isomorphism and $\bar{\phi}$ an $s$-isomorphism
  and $\psi_1$ and $\psi_{3}$ are, their right $r$ and $s$-inverses,
  respectively. By the octahedral axiom in $\C_0$, we construct the
  following diagram commutative except for the right bottom square
  that is anti-commutative:
  \begin{equation} \label{w-oct-2} \xymatrixcolsep{3pc} \xymatrix{
      E \ar[r]\ar[d] & 0 \ar[r] \ar[d] & TE \ar[d] \ar[r]& TE\ar[d] \\
      F \ar[r]^{\alpha''} \ar[d]_{\alpha'} & A \ar[r] \ar[d]_{\mathds{1}_A}
      & C \ar[r] \ar[d]_-{w} & TF\ar[d]_{T\alpha'}\\
      X' \ar[r]^-{u \circ \phi} \ar[d]_-{\delta} & A\ar[d] \ar[r]^-{v}
      & B'' \ar[r]^-{\theta} \ar[d]_-{t}
      & TX' \ar[d]_-{T\delta} \\
      TE \ar[r] & 0 \ar[r]& T^{2}E \ar[r]^{\mathds{1}_{T^2E}} & T^{2}E}
  \end{equation}
  Thus \jznote{$\alpha''=u\circ \phi\circ \alpha'$}, $t=-T\delta\circ \theta$.
  We denote by $$\Delta_3: F \xrightarrow{\alpha''} A \to C \to TF$$
  the respective exact triangle in $\C_0$ so that, as in Remark
  \ref{rmk- set}, $w(\Delta_3) =0$. The map $w\in \Mor^{0}(C,B'')$ is
  induced from the commutativity of the middle, left triangle. We now
  consider the following diagram.
  \begin{equation} \label{w-oct-3} \xymatrixcolsep{3pc} \xymatrix{ F
      \ar[d]_{\alpha'}\ar[r] &A \ar[d]_{\mathds{1}_A}\ar[r]& C \ar[d]_{w}\ar[r]
      & TF \ar[d]_{T\alpha'}\\
      X' \ar[r]^-{u \circ \phi} \ar[d]_-{\phi} & A \ar[r]^-{v}
      \ar[d]_{\mathds{1}_A} & B'' \ar[r]^-{\theta} \ar[d]_-{\phi'}
      & TX' \ar[d]_-{T\phi} \\
      X \ar[r]^-{u} & A\ar[d]_{\mathds{1}_A} \ar[r]^-{v'}
      & B' \ar[r]^-{g}\ar[d]_{\bar{\phi}} & TX\\
      & A\ar[r]^{\gamma} & B \ar[r]^{b}& \Sigma^{-s} TX }
  \end{equation}
  The three long rows are exact triangles in $\C_{0}$ and we deduce
  the existence of $\phi' \in \Mor^0(B'', B')$ making the adjacent
  squares commutative.  This is an $r$-isomorphism by \ocnote{Proposition 
  \ref{cor-1} i}.  We fix a right $r$-inverse
  $\psi_{2}\in \Mor^{0}(\Sigma^{r}B',B'')$ of $\phi'$.  The
  composition $\phi'' = \bar{\phi} \circ \phi' \in \Mor^0(B'', B)$ is
  an $(r+s)$-isomorphism by Proposition \ref{prop-r-iso} (iii).  Let
  \jznote{$\psi''=\psi_{2}\circ \Sigma^{r}\psi_{3}$} (recall
  $\psi_{3}$ from (\ref{prop-w-oct-1})) and notice that $\psi''$ is a
  right $(r+s)$-inverse of $\phi''$.

  We are now able to define the triangle $\Delta_{4}$:
  \[ \Delta_4: TE \to C \xrightarrow{\phi'' \circ w} B
    \xrightarrow{\Sigma^{-r-s}(t \circ \psi'')} \Sigma^{-r-s}T^{2}E
    ~.~\] The following commutative diagram shows that $\Delta_{4}$ is
  strict exact and $w(\Delta_4) = r+s$.
  \[ \xymatrix{
      & & \Sigma^{r+s} B \ar[d]_-{\psi''} \ar[rd]^-{t \circ \psi''} & \\
      TE \ar[r] & C \ar[r]^-{w} \ar[rd]_-{\phi'' \circ w}
      & B'' \ar[r]^-{t} \ar[d]^-{\phi''} & T^{2}E\\
      & & B &} \] It is easy to check that all the squares in
  (\ref{w-oct}), except the right bottom one, are commutative.

  We now check the $r$-anti-commutativity of the right bottom
  square. We need to show
  $\Sigma^{-s} (Tk) \circ b \simeq_r - \Sigma^{-r-s}(t \circ \psi'')$,
  which is equivalent to
  $\Sigma^r (Tk) \circ \Sigma^{r+s} b \simeq_r - t \circ
  \psi''$. Given that the $B''(TX')(TX)B'$ square in (\ref{w-oct-3})
  commutes and using Corollary \ref{cor-6}, we have the following
  $r$-commutative diagram
  \[ \xymatrixcolsep{3pc} \xymatrix{ \Sigma^{r} B' \ar[r]^-{\Sigma^r
        g} \ar[d]_-{\psi_2}
      & \Sigma^r TX \ar[d]_-{T\psi_1} \\
      B'' \ar[r]^-{\theta} & TX'} \]

  Now consider the following diagram, commutative except the middle
  square being $r$-commutative,
  \[ \xymatrixcolsep{3pc} \xymatrix{ \Sigma^{r+s} B \ar[d]_-{\Sigma^r
        \psi_3}
      \ar[rd]^-{\Sigma^{r+s} b} \ar@/_4pc/[dd]_-{\psi''} & & \\
      \Sigma^r B' \ar[r]^-{\Sigma^r g} \ar[d]_-{\psi_2}
      & \Sigma^r TX \ar[d]_-{T\psi_1} \ar[rd]^-{\Sigma^r (Tk)} & \\
      B'' \ar[r]^-{\theta} \ar@/_1.5pc/[rr]^-{-t} & TX'
      \ar[r]^{T\delta} & T^{2}E } \] and write
  \[- t \circ \psi'' = (T\delta) \circ (\theta \circ \psi_2) \circ
    \Sigma^r \psi_3 \simeq_r (T\delta) \circ ((T\psi_1) \circ
    \Sigma^rg) \circ \Sigma^r \psi_3 = \Sigma^r (Tk) \circ
    \Sigma^{r+s}b \] which completes the proof.\end{proof}

Given a triple of maps
$\Delta: A \xrightarrow{u} B \xrightarrow{v} C \xrightarrow{w} D$ with
shifts $\ceil*{u}, \ceil*{v}, \ceil*{w} \in \R$ it is useful to
introduce a special notation for an associated triple in $\C_{0}$,
denoted by $\Sigma^{s_1, s_2, s_3, s_4} \Delta$, for
$s_1, s_2, s_3, s_4 \in \R$ satisfying the following relations
\begin{align*}
  -s_1 + s_2 + \ceil*{u} & \leq 0 \\
  -s_2 + s_3 + \ceil*{v} & \leq 0\\
  -s_3 + s_4 + \ceil*{w} & \leq 0.
\end{align*}
The triple $\Sigma^{s_1, s_2, s_3, s_4} \Delta$ has the form
\begin{equation}\label{eq:shift-tr} \Sigma^{s_1} A
  \xrightarrow{\bar{u}} \Sigma^{s_2}B \xrightarrow{\bar{v}}
  \Sigma^{s_3} \xrightarrow{\bar{w}} \Sigma^{s_4} D
\end{equation}
where $\bar{u}$ is the composition of the composition
$\Sigma^{s_1} A \xrightarrow{(\eta_{s_1, 0})_A} A \xrightarrow{u} B
\xrightarrow{(\eta_{0, s_2})_B} \Sigma^{s_2} B$ and the persistence
structure map \zjr{$i_{-s_1 + s_2+\ceil*{u}, 0}$}, i.e.,
\begin{equation} \label{shift-u} \bar{u} = i_{-s_1 + s_2 + \ceil*{u},
    0}\left((\eta_{0, s_2})_B \circ u \circ(\eta_{s_1, 0})_A\right).
\end{equation}
The definitions of $\bar{v}$ and $\bar{w}$ are similar and, in
particular, $\ceil*{\bar{u}}=\ceil*{\bar{v}}= \ceil*{\bar{w}}=0$. The
inequalities above ensure that the resulting triangle
(\ref{eq:shift-tr}) has all morphisms in $\C_{0}$.

For $s_1 = s_2 = s_3 = s_4 =k$ (which implies that
$\ceil*{u}, \ceil*{v}, \ceil*{w} \leq 0$) we denote, for
brevity, $$\Sigma^k \Delta : = \Sigma^{s_1, s_2, s_3, s_4} \Delta~.~$$

\begin{remark} \label{rmk-shift-notation} Assume that
  $\Delta: A \xrightarrow{u} B \xrightarrow{v} C \xrightarrow{w}
  \Sigma^{-r} TA$ is strict exact of weight $w(\Delta) = r$.

  (a) It is a simple exercise to show that the triangle
  $\Sigma^k \Delta: \Sigma^kA \to \Sigma^k B \to \Sigma^k C \to
  \Sigma^{-r +k} TA$ is strict exact and
  $w\left(\Sigma^k \Delta\right) = r$.

  (b) For $s\geq 0$, Proposition \ref{prop-wt-1} (ii) claims that
  $\Sigma^{0,0,0,-s}\Delta$ is strict exact of weight $r+s$. It is
  again an easy exercise to see that $\Sigma^{0,0,-s,-s}\Delta$ is
  strict exact of weight $r+s$.
\end{remark}

\begin{prop}[Functoriality of triangles] \label{prop-ind} Consider two
  strict exact triangles as below with $f \in \Mor^0(A_1, A_2)$ and
  $g \in \Mor^0(B_1, B_2)$
  \[ \xymatrix{ \Delta_1: & A_1 \ar[r]^-{\bar{u}_1} \ar[d]_-{f} & B_1
      \ar[r]^-{\bar{v}_1} \ar[d]_-{g}
      & C_1 \ar[r]^-{\bar{w}_1} & \Sigma^{-r} TA_1 \\
      \Delta_2: & A_2 \ar[r]^-{\bar{u}_2} & B_2 \ar[r]^-{\bar{v}_2} &
      C_2 \ar[r]^-{\bar{w}_2} & \Sigma^{-s} TA_2} \] and
  $w(\Delta_1) = r$, $w(\Delta_2) = s$. Then there exists a morphism
  $h \in \Mor^0(C_1, \Sigma^{-r} C_2)$ inducing maps relating the
  triangles $\Delta_1 \to \Sigma^{0,0,-r,-r}\Delta_2$ as in the
  following diagram
  \[ \xymatrix{ A_1 \ar[r] \ar[d]_-{f} & B_1 \ar[r] \ar[d]_-{g} & C_1
      \ar[r] \ar[d]_-{h} & \Sigma^{-r} TA_1 \ar[d]^-{\eta^{TA_2}_s
        \circ \Sigma^{-r} Tf} \\
      A_2 \ar[r] & B_2 \ar[r] & \Sigma^{-r} C_2 \ar[r] & \Sigma^{-r-s}
      TA_2} \] where the middle square is $r$-commutative and the
  right square is $s$-commutative. \end{prop}

The proof is left as an exercise.

\begin{prop}\label{prop:tr-isos} Let
  $\Delta: A \xrightarrow{\bar{u}} B \xrightarrow{\bar{v}}
  C\xrightarrow{\bar{w}} \Sigma^{-r} TA$ be a strict exact triangle of
  weight $r$ in $\C$ and let
  $\Delta': A\xrightarrow{\bar{u}} B\xrightarrow{v} C'\xrightarrow{w}
  TA$ be the exact triangle in $\C_{0}$ associated to $\Delta$ as in
  Definition \ref{dfn-set}. There are morphisms of triangles
  $h: \Sigma^{2r}\Delta\to \Delta'$ and
  $h':\Delta' \to \Sigma^{-r} \Delta$ such that the compositions
  $h'\circ h$ and $h\circ \Sigma^{3r} h'$ have as vertical maps the
  \jznote{shifted natural transformations $\eta_{3r}^{(-)}$ defined in
    (\ref{dfn-eta-map}).}
\end{prop}
\begin{proof} We use the notation in Definition \ref{dfn-set} and
  consider the diagram below:
  \[ \xymatrixcolsep{3pc} \xymatrix{ \Sigma^r A \ar[r]^{\Sigma^r
        \bar{u}} \ar[d]_{\eta^A_r} & \Sigma^r B \ar[r]^{\Sigma^r
        \bar{v}} \ar[d]_{\eta^B_r} & \Sigma^r C \ar[r]^-{\Sigma^r
        \bar{w}} \ar[d]_-{\psi}
      &  TA \ar[d]^-{\mathds{1}_{TA}} \\
      A \ar[r]^-{\bar{u}} \ar[d]_{\mathds 1_A} & B \ar[r]^-{v}
      \ar[d]_-{\mathds 1_B} & C' \ar[r]^-{w} \ar[d]_-{\phi}
      & TA \ar[d]^-{\eta^{TA}_r} \\
      A \ar[r]^-{\bar{u}} & B \ar[r]^-{\bar{v}} & C \ar[r]^{\bar{w}} &
      \Sigma^{-r} TA} ~.~\] Denote
  \jznote{$h_{1}=(\eta_{r}^{A},\eta_{r}^{B},\psi , \mathds{1}_{TA})$ and
  $h_{1}'=(\mathds 1_{A}, \mathds{1}_{B},\phi,\eta_{r}^{TA})$}.  Notice that $h_{1}$ as
  well as $h'_{1}$ are not morphisms of triangles because the bottom
  right-most square is only $r$-commutative, and the same is true for
  the middle top square - as discussed in Remark \ref{rmk- set}
 \zjr{(c)}. Let $h=h_{1}\circ \eta_{r}$ and $h'=\eta_{r}\circ h'_{1}$
  (where we view $\eta_{r}$ as a quadruple of morphisms of the form
  $\eta_{r}^{-}$). It follows that both $h$ and $h'$ are morphisms of
  triangles. Moreover, given that $\phi\circ \psi=\eta_{r}^{C}$, it is
  clear that $h'\circ h=\eta_{3r}$.  The other composition,
  \zjr{$h\circ \Sigma^{3r}h'$}, has one term of the form
  \zjr{$\psi\circ \eta_{r}\circ \eta_{r} \circ \Sigma^{3r}$} so, by
  Proposition \ref{prop-r-iso}, this coincides with $\eta_{3r}^{C'}$
  as claimed.
\end{proof}

\begin{rem}\label{rem:tr-iso-0} Proposition \ref{prop:tr-isos} shows that
a strict exact triangle of weight $r$ is
\zjr{approximately isomorphic in a sense similar to interleaving}, to an exact triangle in $\C_{0}$. 
  \end{rem}

\subsection{Fragmentation pseudo-metrics on
  $\mathrm{Obj}(\C)$}\label{subsubsec:frag1}

In a triangulated persistence category there is a natural notion of
iterated-cone decomposition, similar to the corresponding notion in
the triangulated setting from \S\ref{subsec:triang-gen}.

\begin{dfn}\label{def:iterated-coneD} Let $\C$ be a triangulated
  persistence category, and $X \in {\rm Obj}(\C)$. An {\em iterated
    cone decomposition} $D$ of $X$ with {\em linearization}
  $(X_1,X_{2}, ..., X_n)$ where $X_i \in {\rm Obj}(\C)$ consists of a
  family of strict exact triangles in $\C$
  \[ 
    \left\{ 
      \begin{array}{ll}
        \Delta_{1}: \, \, & X_{1}\to 0\to Y_{1}\to \Sigma^{-r_{1}} TX_{1}\\
        \Delta_2: \,\, & X_2 \to Y_1 \to Y_2 \to \Sigma^{-r_2} TX_2\\
        \Delta_3: \,\, & X_3 \to Y_2 \to Y_3 \to \Sigma^{-r_3} TX_3\\
                          &\,\,\,\,\,\,\,\,\,\,\,\,\,\,\,\,\,\,\,\,\vdots\\
        \Delta_n: \,\, & X_n \to Y_{n-1} \to X \to \Sigma^{-r_n} TX_n
      \end{array} \right.\]
  The weight of such a cone decomposition is defined by
  $$w(D) = \sum_{i=1}^n w(\Delta_i)~.~$$ 
  The linearization of $D$ is denoted by $\ell(D) = (X_1, ..., X_n)$.
\end{dfn}

\begin{prop} \label{prop-cone-ref} Assume that $X$ admits an iterated
  cone decomposition $D$ with linearization $(X_1, ..., X_n)$ and for
  some $i \in \{1, ..., n\}$, $X_i$ admits an iterated cone
  decomposition $D'$ with linearization $(A_1, ..., A_k)$. Then $X$
  admits an iterated cone decomposition $D''$ with linearization
  \begin{equation} \label{cone-ref} (X_1, ..., X_{i-1}, TA_1, ...,
    TA_k, X_{i+1}, ..., X_n).
  \end{equation}
  Moreover, the weights of these cone decompositions satisfy
  $w(D'') = w(D) + w(D')$.
\end{prop}

A cone decomposition $D''$ as in the statement of Proposition
(\ref{cone-ref}) is called a {\em refinement} of the cone
decomposition $D$ with respect to $D'$.

\begin{ex} \label{ex:crefine} A single strict exact triangle
  $A \to B \to X \to \Sigma^{-r} TA$ can be regarded as a cone
  decomposition $D$ of $X$ with linearization $(T^{-1}B,A)$ such that
  $w(D) = r$. Assume that $A$ fits into a second strict exact triangle
  $E\to F\to A\to \Sigma^{-s} TE$ of weight $s$. Thus we have a
  cone-decomposition of $D'$ of $A$, with linearization $(T^{-1}F, E)$
  and $w(D')=s$.  Diagram (\ref{w-oct}) from Proposition
  \ref{prop-w-oct} yields the following commutative diagram,
  \[ \xymatrix{
      E \ar[r] \ar[d] & 0 \ar[r] \ar[d] & TE \ar[d] \\
      F \ar[r] \ar[d] & B \ar[r] \ar[d] & Y \ar[d] \\
      A \ar[r] & B \ar[r] & X} \] for some object
  $Y \in {\rm Obj}(\C)$. In particular, we obtain a strict exact
  triangle $TE \to Y \to X \to \Sigma^{-r-s} T^{2}E$ of weight
  $r+s$. Thus, we have a refinement of $D$ with respect to $D'$ as
  follows,
\jznote{  \[ D'' := \left\{
      \begin{array}{l}
        T^{-1}B\to 0 \to B \to B\\
        F  \to B \to Y \to TF \\
        TE \to Y \to X \to \Sigma^{-r-s} TE
      \end{array} \right. \]}
  Moreover, $w(D'') = r+s = w(D) + w(D')$.
\end{ex}

\begin{proof} [Proof of Proposition \ref{prop-cone-ref}] By
  definition, the cone decomposition $D$ consists of a family of
  strict exact triangles in $\C$ as follows,
  \[ 
    \left\{ 
      \begin{array}{ll}
        &\,\,\,\,\,\,\,\,\,\,\,\,\,\,\,\,\,\,\,\,\vdots\\
        \Delta_{i-1}: \,\, & X_{i-1} \to Y_{i-2} \to Y_{i-1}
                             \to \Sigma^{-r_{i-1}} X_{i-1}\\
        \Delta_i: \,\, & X_{i} \,\,\,\,\,\,\to Y_{i-1}
                         \to Y_{i} \,\,\,\,\,\,\to \Sigma^{-r_{i}} X_{i}\\
        \Delta_{i+1}: \,\, & X_{i+1} \to Y_{i} \,\,\,\,\,\,\to Y_{i+1}
                             \to \Sigma^{-r_{i+1}} X_{i+1}\\
        &\,\,\,\,\,\,\,\,\,\,\,\,\,\,\,\,\,\,\,\,\vdots\\
      \end{array} \right.\]
  We aim to replace the triangle $\Delta_i$ by a sequence of strict
  exact triangles
 \zjr{ $$\overline{\Delta}_j: TA_j \to B_{j-1} \to B_j \to \Sigma^{-s_j} T^2A_j$$}
  for $j \in \{1, ..., k\}$ with $B_{0} = Y_{i-1}$, $B_k = Y_i$ and
  such that
  \begin{equation} \label{sum-ref} \sum_{j=1}^k w(\overline{\Delta}_j)
    = w(\Delta_i) + w(D').
  \end{equation}
  In this case, the ordered family of strict exact triangles
  $(\Delta_1, ..., \Delta_{i-1}, \overline{\Delta}_1, ...,
  \overline{\Delta}_k, \Delta_{i+1}, ..., \Delta_n)$ form the
  refinement $D''$, and (\ref{sum-ref}) implies that
  \begin{equation}
    w(D'')  = \sum_{j \in \{1, ..., n\} \backslash \{i\}} w(\Delta_j) +
    \sum_{j =1}^{k} w(\overline{\Delta}_j) 
    = \sum_{j=1}^n w(\Delta_j) + w(D') = w(D) + w(D')
  \end{equation}
  as claimed. 

  In order to obtain the desired sequence of strict exact triangles we
  focus on $\Delta_i$ and, to shorten notation, we rename its terms by
  $A = X_i$, $B = Y_{i-1}$, $C = Y_i$ and $r = r_i$ so that, with this
  notation, $\Delta_{i}$ is a strict exact triangle
  $A \to B \to C \to \Sigma^{-r} TA$.

  We now fix notation for the cone decomposition $D'$ of $A=X_{i}$. It
  consists of the following family of strict exact triangles,
  \[ 
    \left\{ 
      \begin{array}{ll}
        \Delta'_1: \,\, & A_1 \to 0 \to Z_1 \to  \Sigma^{-s_{1}}TA_1\\
        \Delta'_2: \,\, & A_2 \to Z_1 \to Z_2 \to \Sigma^{-s_2} TA_2\\
                        &\,\,\,\,\,\,\,\,\,\,\,\,\,\,\,\,\,\,\,\,\vdots\\
        \Delta'_{k-1}: \,\, & A_{k-1} \to Z_{k-2} \to Z_{k-1}
                              \to \Sigma^{-s_{k-1}} TA_{k-1}\\
        \Delta'_k: \,\, & A_k \to Z_{k-1} \to A \to \Sigma^{-s_k} TA_k
      \end{array} \right.
  \]
  We will apply Proposition \ref{prop-w-oct} iteratively. The first
  step is the following commutative diagram obtained from
  (\ref{w-oct}),
  \[ \xymatrix{
      A_k \ar[r] \ar[d] & 0 \ar[r] \ar[d] & TA_k \ar[d] \\
      Z_{k-1} \ar[r] \ar[d] & B \ar[r] \ar[d]_{\mathds{1}_B} & B_{k-1} \ar[d] \\
      A \ar[r] & B \ar[r] & C} \] for some
  $B_{k-1} \in {\rm Obj}(\C)$. Define
  $$\overline{\Delta}_k :TA_{k} \to B_{k-1} \to C \to
  \Sigma^{-r - s_k} T^{2}A_k~.~$$ We have
  \begin{equation} \label{inductive-1} w(\overline{\Delta}_k) = r +
    s_k = w(\Delta_i) + w(\Delta'_k).
  \end{equation}
  We then consider the following commutative diagram again obtained
  from (\ref{w-oct}),
  \begin{equation}\label{eq:diag-del} \xymatrix{
      A_{k-1} \ar[r] \ar[d] & 0 \ar[r] \ar[d] & TA_{k-1} \ar[d] \\
      Z_{k-2} \ar[r] \ar[d] & B \ar[r] \ar[d] & B_{k-2} \ar[d] \\
      Z_{k-1} \ar[r] & B \ar[r] & B_{k-1}} 
  \end{equation}
  for some $B_{k-2} \in {\rm Obj}(\C)$. Define
  $\overline{\Delta}_{k-1}$ to be the strict exact triangle:
  $$\overline{\Delta}_{k-1}: TA_{k-1} \to B_{k-2} \to B_{k-1}
  \to \Sigma^{- s_{k-1}} T^{2}A_{k-1}~.~$$ Then
  \begin{equation} \label{inductive-2} w(\overline{\Delta}_{k-1}) =
    s_{k-1} = w(\Delta'_{k-1}).
  \end{equation}
  Inductively, we obtain $B_i \in {\rm Obj}(\C)$, strict exact
  triangles
  \jznote{
  \begin{equation}\label{eq:it-tr}
    \overline{\Delta}_{i}: TA_i \to B_{i-1} \to B_i \to
    \Sigma^{-s_i} T^{2}A_i
  \end{equation}
} \pbnote{for $2 \leq i \leq k-1$} such that
  \begin{equation} \label{inductive-i}
    w(\overline{\Delta}_i) = s_i  = w(\Delta'_i).
  \end{equation}
  The final step lies in the consideration of the following diagram,
  \begin{equation} \label{eq:final-sq}\xymatrix{
      A_{1} \ar[r] \ar[d] & 0 \ar[r] \ar[d] & TA_{1} \ar[d] \\
      0 \ar[r] \ar[d] & B \ar[r] \ar[d] & B \ar[d] \\
      Z_1 \ar[r] & B \ar[r] & B_1}
  \end{equation}
  for some \zjr{$B \in {\rm Obj}(\C)$}. Define $\overline{\Delta}_1$ to be
  the strict triangle
 \zjr{$$\overline{\Delta}_{1}: TA_1 \to B\to B_1 \to  \Sigma^{-s_1} TA_1~.~$$} Then 
  \begin{equation} \label{inductive-11} w(\overline{\Delta}_1) =
    w(\Delta'_1)=s_{1}~.~
  \end{equation}
  Together, the ordered family
  $(\overline{\Delta}_1, ..., \overline{\Delta}_k)$ form the desired
  sequence of strict exact triangles. \pbnote{Finally, the
    equalities~\eqref{inductive-1},~\eqref{inductive-i}
    and~\eqref{inductive-11}} yield
  \[ \sum_{j=1}^k w(\overline{\Delta}_j) =s_{1}+ s_2 + ... + s_k + r =
    \sum_{j=1}^k w(\Delta'_j) + w(\Delta_i) = w(D') + w(\Delta_i) \]
  as claimed in (\ref{sum-ref}).
\end{proof}

Let $\mathcal F \subset {\rm Obj}(\C)$ be a family of objects of $\C$.
For two objects $X, X' \in {\rm Obj}(\C)$, define just as in
\S\ref{subsec:triang-gen},
\begin{equation} \label{frag-met} \delta^{\mathcal F} (X, X') =
  \inf\left\{ w(D) \, \Bigg| \, \begin{array}{ll} \mbox{$D$ is an
        iterated cone decomposition} \\ \mbox{of $X$ with
        linearization}\ \mbox{$(F_1, ..., T^{-1}X', ..., F_k)$}\\
      \mbox { where $F_i \in \mathcal F$, $k \in
        \N$} \end{array} \right\}.
\end{equation}

\begin{cor} \label{tri-ineq} With the definition of $\delta^{\F}$ in
  (\ref{frag-met}), we have the following inequality:
  \[ \delta^{\F}(X, X') \leq \delta^{\F}(X, X'') + \delta^{\F}(X'', X') \]
  for any $X, X', X'' \in {\rm Obj}(\C)$.
\end{cor}

\begin{proof} For any $\ep>0$, there are cone decompositions $D$ of
  $X$ and $D'$ of $X''$ respectively such that
  \[ w(D) \leq \delta^{\F}(X, X'')+ \ep\,\,\,\,\,\mbox{and}\,\,\,\,
    w(D') \leq \delta^{\F}(X'', X')+ \ep \] with linearizations
  $\ell(D) = (F_1, ..., T^{-1}X'', ..., F_s)$ and
  $\ell(D') = (F'_1,..., T^{-1}X', ..., F'_k)$, respectively,
  $F_i, F'_j\in \F$.  This means that $T^{-1}X''$ has a corresponding
  cone decomposition $T^{-1}D'$ with linearization
  $\ell(T^{-1}D')=(T^{-1}F'_{1},\ldots, T^{-2} X',\ldots,
  T^{-1}F'_{k})$.  Proposition \ref{prop-cone-ref} implies that there
  exists a cone decomposition $D''$ of $X$ that is a refinement of $D$
  with respect to $T^{-1}D'$ such that
  $\ell(D'')=(F_{1},\ldots, F'_{1},\ldots, T^{-1}X',\ldots
  F'_{k},\ldots, F_{s})$ and
  \[ w(D'') = w(D) + w(D') \leq \delta^{\F}(X, X'') + \delta^{\F}(X'',
    X') + 2 \ep\] which implies the claim.
\end{proof}

Finally, there are also fragmentation pseudo-metrics specific to this
situation with properties similar to those in Proposition
\ref{prop:tr-weights-gen}.

\begin{dfn} \label{dfn-frag-met} Let $\C$ be a triangulated
  persistence category and let $\F \subset {\rm Obj}(\C)$.  The {\em
    fragmentation pseudo-metric}
  $$d^{\F}: {\rm Obj}(\C)\times {\rm Obj}(\C)\to [0,\infty)\cup \{
  +\infty\}$$ associated to $\F$ is defined by:
  \[ d^{\F}(X, X') = \max\{\delta^{\F}(X, X'), \delta^{\F}(X',
    X)\}. \]
\end{dfn}

\begin{rem} \label{ex-delta} (a) It is clear from Corollary
  \ref{tri-ineq} that $d^{\F}$ satisfies the triangle inequality and,
  by definition, it is symmetric. It is immediate to see that
  $d^{\F}(X,X)=0$ for all objects $X$ (this is because of the
  existence of the exact triangle in $\C_{0}$, \zjnote{$T^{-1}X\to 0\to X \to X$}).
  It is of course possible that this pseudo-metric \zjnote{can} be degenerate and
  it is also possible that it is not finite.

  (b) \jznote{If $X\in \mathcal F$}, then $\delta^{\F}(TX,0)=0$ because of the exact
  triangle \zjnote{$X\to 0 \to TX \to TX$}.  On the other hand, $\delta^{\F}(0,TX)$ is
  not generally trivial. However, the exact triangle \zjnote{$TX\to TX\to 0 \to T^2X$}
  shows that, if $TX\in \F$, then $\delta^{\F}(0,TX)=0$.

  (c) It follows from the previous point that if $\F = {\rm Obj}(\C)$,
  then $d^{\F}(X,X') =0$ (in other words, the pseudo-metric
  $d^{\mathcal{F}}(-,-)$ is completely degenerate).  More generally,
  if the family $\F$ is $T$ invariant (in the sense that if $X\in \F$,
  then $TX, T^{-1}X \in \F$), then $T$ is an isometry with respect to
  the pseudo-metric $d^{\F}$ and $d^{\F}(X,X')=0$ for all
  $X,X'\in \F$.

  (d) The remark \ref{rem:finite-metr} (c) applies also in this
  setting in the sense that we may define at this triangulated
  persistence level fragmentation pseudo-metrics
  $\underline{d}^{\mathcal{F}}$ given by (the symmetrization of)
  formula (\ref{eq:frag-simpl}) but making use of weighted triangles
  in $\mathcal{C}$ instead of the exact triangles in the triangulated
  category $\mathcal{D}$.
\end{rem}

Recall that by assumption $\C_0$ is triangulated and thus
additive. Therefore, for any two objects
$X,X' \in {\rm Obj}(\C_0) = {\rm Obj}(\C)$, the direct sum
$X \oplus X'$ is a well-defined object in ${\rm Obj}(\C)$.

\begin{prop} \label{prop-frag-sum} For any
  $A, B, A', B' \in {\rm Obj}(\C)$, we have
  \[ d^{\F}(A \oplus B, A' \oplus B') \leq d^{\F}(A, A') + d^{\F}(B,
    B'). \]
\end{prop}

\begin{proof}
  The proof follows easily from the following lemma.

  \begin{lem} \label{claim-frag-sum} Let
    $\Delta: A \to B \to C \to \Sigma^{-r} TA$ and
    $\overline{\Delta}: \overline{A} \to \overline{B} \to \overline{C}
    \to \Sigma^{-s} T\overline{A}$ be two strict exact triangles with
    $w(\Delta) = r$ and $w(\Delta') = s$. Then
    \[ \Delta'': A \oplus \overline{A} \to B \oplus \overline{B} \to C
      \oplus \overline{C} \to \Sigma^{- \max\{r,s\}} \bigl(TA \oplus
      T\overline{A} \bigr) \] is a strict exact triangle with
    $w(\Delta'') = \max\{r,s\}$.
  \end{lem}
  \begin{proof}[Proof of Lemma \ref{claim-frag-sum}]
    By definition, there are two commutative diagrams, 
    \[ \xymatrix{
        & & \Sigma^r C \ar[d]_-{\psi} \ar[rd] & \\
        A \ar[r] & B \ar[r] \ar[rd] & C' \ar[r] \ar[d]^-{\phi} & TA\\
        & & C &} \,\,\,\,\, \xymatrix{
        & & \Sigma^s \overline{C} \ar[d]_-{\bar{\psi}} \ar[rd] & \\
        \overline{A} \ar[r] & \overline{B} \ar[r] \ar[rd]
        & \overline{C'} \ar[r] \ar[d]^-{\bar{\phi}} & T\overline{A}\\
        & & \overline{C} &} \] This yields the following commutative
    diagram,
   \jznote{ \[ \xymatrix{ & & \Sigma^{\max\{r,s\}} C \oplus \overline{C}
        \ar[d]_-{\psi \oplus \bar{\psi}} \ar[rd] & \\
        A \oplus \overline{A} \ar[r] & B \oplus \overline{B} \ar[r]
        \ar[rd] & C' \oplus \overline{C'} \ar[r] \ar[d]^-{\phi \oplus
          \bar{\phi}}
        & TA \oplus T\overline{A}\\
        & & C \oplus \overline{C} &} \] }and it is easy to check that
    $\phi \oplus \bar{\phi}$ is a
    $\max\{r,s\}$-isomorphism.
  \end{proof}

  Returning to the proof of the proposition, it suffices to prove
  $\delta^{\F}(A \oplus B, A' \oplus B') \leq \delta^{\F}(A, A') +
  \delta^{\F}(B, B')$. For any $\ep>0$, by definition, there exist
  cone decompositions $D$ and $D'$ of $A$ and $B$ respectively with
  $\ell(D) = (F_1, ..., F_{i-1}, T^{-1}A', F_{i+1},..., F_s)$ and
  $\ell(D') = (F'_1, ..., T^{-1}B', ..., F'_{s'})$ such that
  \[ w(D) \leq \delta^{\F}(A, A') + \ep \,\,\,\,\,\mbox{and}\,\,\,\,\,
    w(D') \leq \delta^{\F}(B, B') + \ep. \]

  The desired cone decomposition of $A \oplus B$ is defined as follows. 
  \begin{equation}\label{eq:sum-dec} D'' : = 
    \left\{
      \begin{array}{l} 
        F_1 \to 0 \to E_1 \to \Sigma^{-r_1} TF_1 \\
        F_2 \to E_1 \to E_2 \to \Sigma^{-r_2} TF_2\\
        \,\,\,\,\, \,\,\,\,\, \,\,\,\,\, \,\,\,\,\, \,\,\,\,\,\vdots\\
        F_s  \to E_{s-1} \to A \to \Sigma^{-r_s} TF_s\\
        F'_{1}\to A\oplus 0 \to A\oplus E'_{1}\to \Sigma^{-r'_{1}}TF'_{1}\\
        F'_{2} \to A\oplus E'_{1} \to A\oplus E'_{2} \to \Sigma^{-r'_{2}}  
        TF'_{s+2-i}\\
        \,\,\,\,\, \,\,\,\,\, \,\,\,\,\, \,\,\,\,\, \,\,\,\,\,\vdots\\
        T^{-1}B' \to A \oplus E'_j \to A \oplus E'_{j+1} \to
        \Sigma^{- r'_{B'}} B' \\
        \,\,\,\,\, \,\,\,\,\, \,\,\,\,\, \,\,\,\,\, \,\,\,\,\,\vdots\\
        F'_{s'} \to A \oplus E'_{s'-1} \to A \oplus B \to
        \Sigma^{-r'_{s'}} TF'_{s'}
      \end{array} \right. 
  \end{equation}
  Here we identify $0\oplus F'_{i}=F'_{i}$. The first $s$-triangles
  come from the decomposition of $A$ and the following $s'$ triangles
  are associated, using Lemma \ref{claim-frag-sum}, to the respective
  triangles in the decomposition of $B$ and to the triangle
  \zjnote{$0\to A\to A \to 0$} (of weight $0$).  It is obvious that
  $w(D'')=w(D)+w(D')$ and thus
  $\delta^{\F}(A \oplus B, A' \oplus B') \leq \delta^{\F}(A, A') +
  \delta^{\F}(B,B') $.
\end{proof}

The next statement
is an immediate consequence of Proposition \ref{prop-frag-sum}.

\begin{cor} \label{cor:Hsp} The set ${\rm Obj}(\C)$ with the topology induced by the fragmentation pseudo-metric $d^{\F}$ is an $H$-space relative to the operation $$(A,B)\in {\rm Obj}(\C) \times {\rm Obj}(\C)
\to A\oplus B\in  {\rm Obj}(\C) ~.~$$\end{cor}

\subsubsection{Proof of Proposition
  \ref{prop:tr-weights-gen}}\label{subsubsec:proof-prop}
We now return to the setting in \S\ref{subsec:triang-gen}. Thus,
$\mathcal{D}$ is triangulated category, $w$ is a triangular weight on
$\mathcal{D}$ in the sense of Definition \ref{def:triang-cat-w}, and
the quantities $w(D)$ (associated to an iterated cone-decomposition
$D$), $\delta^{\mathcal{F}}$, $d^{\mathcal{F}}$ are defined as in
\S\ref{subsec:triang-gen}.

The first (and main) step is to establish a result similar to
Proposition \ref{prop-cone-ref}. Namely, if $X$ admits an iterated
cone decomposition $D$ with linearization $(X_1, ..., X_n)$ and some
$X_i$ admits a decomposition $D'$ with linearization
$(A_1, ..., A_k)$, then $X$ admits an iterated cone decomposition
$D''$ with linearization
$(X_1, ..., X_{i-1}, TA_1, ..., TA_k, X_{i+1}, ..., X_n)$ and
\begin{equation}\label{eq:weight-ineq2}
  w(D'') \leq w(D) + w(D')
\end{equation}
For convenience, recall from \S\ref{subsec:triang-gen} that the
expression of $w(D)$ is:
\begin{equation}\label{eq:weight-again}
 w(D)= \sum_{i=1}^{n}w(\Delta_{i})-w_{0}
\end{equation}
where \zjnote{$w_{0}=w(0\to X\xrightarrow{\mathds{1}_X} X\to 0)$} (for
all $X$).  To show (\ref{eq:weight-ineq2}) we go through exactly the
same construction of the refinement $D''$ of the decomposition $D$
with respect to $D'$, as in the proof of \ref{prop-cone-ref}, assuming
now that all shifts are trivial along the way. The analogue of diagram
(\ref{eq:final-sq}) that appears in the last step of the construction
of $D''$ remains possible in this context due to the point (ii) of
Definition \ref{def:triang-cat-w}.  By tracking the respective weights
along the construction and using Remark \ref{rem:gen-weights} (b) to
estimate the weight of $TA_{1}\to B\to B_{1}\to T^{2}A_{1}$ from
(\ref{eq:final-sq}) we deduce (with the notation in the proof of
Proposition \ref{prop-cone-ref})
\zjnote{$$\sum_{j=1}^kw(\bar{\Delta}_{j}) \leq
  w(D')+w(\Delta_{i})-w_{0}$$} which implies (\ref{eq:weight-ineq2}).
Once formula (\ref{eq:weight-ineq2}) established, it immediately
follows that $\delta^{\mathcal{F}}(-,-)$ satisfies the triangle
inequality. Further, because the weight of a cone-decomposition is
given by (\ref{eq:weight-again}), it follows that the
cone-decomposition $D$ of $X$ with linearization $(T^{-1}X)$ given by
the single exact triangle $T^{-1}X\to 0\to X\to X$ is of weight
$w(D)=0$. As a consequence, $\delta^{\mathcal{F}}( X,X)=0$. It follows
that $d^{\mathcal{F}}(-,-)$ is a pseudo-metric as claimed at the point
(i) of Proposition \ref{prop:tr-weights-gen}.

Assuming now that $w$ is subadditive, the same type of decomposition
as in equation (\ref{eq:sum-dec}) can be constructed to show that
$\delta^{\mathcal{F}}(A\oplus B, A'\oplus B')\leq
\delta^{\mathcal{F}}(A,A')+\delta^{\mathcal{F}}(B,B')+w_{0}$ which
implies the claim. \Qed

\section{A persistence triangular weight on $\C_{\infty}$}
\label{subsec:trstr-weights}

The purpose of this section is to further explore the structure of the
limit category $\C_{\infty}$ associated to a triangulated persistence
category $\C$. We already know from \S\ref{subsubsec:localization}
that the category $\C_{\infty}$ is triangulated.  The main aim here is
to use the properties of the weighted exact triangles introduced in
\S\ref{subsubsec:weight-tr} to endow $\C_{\infty}$ with a triangular
weight, in the sense of \S\ref{subsec:triang-gen}.

\subsection{Weight of exact triangles in
  $\C_{\infty}$} \label{subsubsec:ex-cinfty}

In this section we will use the weighted strict exact triangles in
$\C$ to associate weights to the exact triangles in $\C_{\infty}$.

Assume that $\C$ is a persistence category and recall its
$\infty$-level $\C_{\infty}$ from Definition \ref{dfn-c0-cinf}: its
objects are the same as those of $\C$ and its morphisms are
$\Mor_{\C_{\infty}}(A,B) = \varinjlim_{\alpha \to \infty}
\Mor_{\C}^{\alpha}(A,B)$ for any two objects $A, B$ of $\C$. For a
morphism $\bar{f}$ in $\C$ we denote by $[\bar{f}]$ the corresponding
morphism in $\C_{\infty}$ and if $f=[\bar{f}]$,
\zjnote{for $f\in\Mor_{\C_{\infty}}$ and $\bar{f}\in\Mor_{\C}$}, we say that $\bar{f}$
represents $f$. We use the same terminology for diagrams (including
triangles) in $\C$ in relation to corresponding diagrams in
$\C_{\infty}$ in the sense that a diagram in $\C$ represents one in
$\C_{\infty}$ if the objects in the two cases are the same and the
morphisms in the diagram in $\C$ represent the corresponding ones in
the $\C_{\infty}$ diagram. Clearly, all $r$-commutativities and
$r$-isomorphisms in $\C$ become, respectively, commutativities and
isomorphisms in $\C_{\infty}$. For instance, if $K$ is $r$-acyclic,
then $K$ is isomorphic to $0$ in $\C_{\infty}$.

For further reference notice also that the hom-sets of $\C_{\infty}$
admit a natural filtration as follows.  For any
$A, B \in {\rm Obj}(\C_{\infty})$, and
$f \in {\hom}_{\C_{\infty}}(A,B)$, let the {\em spectral invariant}
of $f$ be given by:
\begin{equation}\label{eq:spec} \sigma(f) : =\inf \left\{ k\in \R \cup
    \{-\infty\} \, \bigg| \, \mbox{$f = \left[\bar{f}\right]$ for some
      $\bar{f} \in \Mor_{\C}^{k}(A,B)$} \right\}
\end{equation}
and
$$\Mor_{\C_{\infty}}^{\leq \alpha}(A,B)=\{f\in \Mor_{\C_{\infty}}(A,B)
\, | \, \sigma(f)\leq \alpha\}.$$

\smallskip

Assume from now on that $\C$ is a triangulated persistence category.
In this case, we have already seen in \S\ref{subsubsec:localization}
that $\C_{\infty}$ is identified with $\C_{0}/\mathcal{AC}_{0}$, the
Verdier quotient of $\C_{0}$ by the subcategory of acyclics. Thus
$\C_{\infty}$ is triangulated with its exact triangles defined through
isomorphism with the image in $\C_{\infty}$ of the exact triangles in
$\C_{0}$, see Remark \ref{rem:ex-tr-Cinfty}.

Before proceeding, we notice that the shift functor $\Sigma$
associated to $\mathcal{C}$ (see Definition \ref{dfn-tpc}) induces a
similar functor $\Sigma : (\R, +) \to \mathrm{End}(\C_{\infty})$. We
will continue to use the same notation for the $r$-shifts $\Sigma^{r}$
and the natural transformations $\eta_{r,s}:\Sigma^{r}\to
\Sigma^{s}$. \pbnote{At the same time, in contrast to morphisms in
  $\C$, there is no meaning to the ``amount of shift'' for a morphism
  in $\C_{\infty}$ (though one can associate to such a morphism its
  spectral invariant as above).}  Similarly, the functor $T$ (which is
defined as in Remark~\ref{rem:shifts-T} (b) on all of $\C$) also
induces a similar functor on $\C_{\infty}$.

\

Given a triple of maps $\Delta: A\to B\to C\to D$ in $\C$ the shifted
triple $\Sigma^{s_1, s_2, s_3, s_4} \Delta$ was defined in
(\ref{eq:shift-tr}) and we will use the same notation for similar
triples in $\C_{\infty}$. Note however that in $\C_{\infty}$ the
inequalities relating the $s_{i}$'s and the shifts of $u,v,w$ are no
longer relevant \pbnote{(in fact, do not make sense)} and the shift
$\Sigma^{s_{1},s_{2},s_{3},s_{4}}$ will be used in $\C_{\infty}$
without these constraints.

\begin{dfn} \label{dfn-extri-inf} The {\em unstable weight},
  $w_{\infty}(\Delta)$, of an exact triangle
  $\Delta: A \xrightarrow{u} B \xrightarrow{v} C \xrightarrow{w} TA$
  in $\C_{\infty}$ is the infimum of the weights of the strict exact
  triangles in $\C $ of the form \zjnote{$(\widetilde{\Delta}, r)$} where
  \begin{equation} \label{extri-shifted} \widetilde{\Delta} \ : \ A
    \xrightarrow{u'} \Sigma ^{-q} B \xrightarrow{v'}\Sigma^{-s} C
    \xrightarrow{w'}\Sigma^{-r} TA
  \end{equation}
  $0\leq q\leq s\leq r$ and $\widetilde{\Delta}$ represents $\Delta$
  in the following sense: the class the composition
  $A\xrightarrow{u'} \Sigma^{-q}B\xrightarrow{\eta_{-q,0}} B$ in
  $\C_{\infty}$ equals $u$, and similarly for $v'$ and $w'$. The {\em weight} of $\Delta$, $\bar{w}(\Delta)$, is given by:
  $$\bar{w}(\Delta)=\inf\{ \ w_{\infty}(\Sigma^{s,0,0,s}\Delta)
  \ | \ s\geq 0\}~.~$$
\end{dfn}

\begin{rem} (a) By definition,
  $\bar{w}(\Delta) \leq w_{\infty}(\Delta)$, and Example
  \ref{ex-embed} below shows that this inequality can be strict.
\end{rem}

For the weight $\bar{w}$ of exact triangles in $\C_{\infty}$ defined
as above, recall that $\bar{w}_0$ denotes the normalization constant
in Definition \ref{def:triang-cat-w} (ii).  The main result is the
following:

\begin{thm}\label{thm:main-alg}
  Let $\C$ be a triangulated persistence category. The limit category
  $\C_{\infty}$ with the triangular structure coming from the
  identification \zjnote{$\C_{\infty} \simeq \C_{0}/\mathcal{AC}_{0}$ in Proposition \ref{prop:verdier-loc}} admits
  $\bar{w}$ as a triangular, subadditive weight with $\bar{w}_{0}=0$.
\end{thm}

A persistence refinement of a triangulated category $\mathcal{D}$ is a
TPC, $\widetilde{\mathcal{D}}$, such that
$\widetilde{\mathcal{D}}_{\infty}=\mathcal{D}$. The triangular weight
$\bar{w}$ as in Theorem \ref{thm:main-alg} is called the persistence
weight induced by the respective refinement. The following consequence
of Theorem \ref{thm:main-alg} is immediate from the general
constructions in \S\ref{subsec:triang-gen}.

\begin{cor} If a small triangulated category $\mathcal{D}$ admits a
  TPC \pbnote{refinement} $\widetilde{\mathcal{D}}$, then
  $\mathrm{Obj}(\mathcal{D})$ is endowed with a family of
  fragmentation pseudo-metrics $\bar{d}^{\mathcal{F}}$, defined as in
  \S\ref{subsec:triang-gen}, associated to the persistence weight
  $\bar{w}$ induced by the refinement $\widetilde{\mathcal{D}}$ and it
  has an H-space structure with respect to the topologies induced by
  these metrics.
\end{cor}

\begin{rem}
  (a) We have seen in \S\ref{subsubsec:frag1}, in particular
  Definition \ref{dfn-frag-met}, that for a TPC $\mathcal{C}$ there
  are fragmentation pseudo-metrics $d^{\mathcal{F}}$ defined on
  $\mathrm{Obj}(\mathcal{C})$.  The metrics $\bar{d}^{\mathcal{F}}$
  associated to the persistence weight $\bar{w}$ on $\C_{\infty}$,
  through the construction in \S\ref{subsec:triang-gen}, are defined
  on the same underlying set, $\mathrm{Obj}(\mathcal{C})$.  The
  relation between them is
  $$\bar{d}^{\mathcal{F}}\leq d^{\mathcal{F}}$$ for any family of
  objects $\mathcal{F}$. The interest to work with
  $\bar{d}^{\mathcal{F}}$ rather than with $d^{\mathcal{F}}$ is that
  if $\mathcal{F}$ is a family of triangular generators of
  $\C_{\infty}$, and is closed to the action of $T$, then
  $\bar{d}^{\mathcal{F}}$ is finite (see Remark
  \ref{rem:finite-metr}).

  (b) As discussed in \zjr{Remark \ref{rem:finite-metr} (c)} (and also in \zjr{Remark
  \ref{ex-delta} (d)}), in this setting too we may define a simpler
  (but generally larger) fragmentation pseudo-metric of the type
  $\underline{\bar{d}}^{\mathcal{F}}$.
\end{rem}

We postpone the proof of \jznote{Theorem \ref{thm:main-alg} to \S \ref{subsubsec:proofTh1}. Here, we pursue} with a few \zjnote{examples}
shedding some light on Definition \ref{dfn-extri-inf}.

\begin{ex} \label{ex-tri-c-c-inf} Assume that
  $\Delta: A \xrightarrow{\bar{u}} B \xrightarrow{\bar{v}} C
  \xrightarrow{\bar{w}} \Sigma^{-r} TA$ is a strict exact triangle of
  weight $r$ in $\C$. Consider the following triangle in
  $\C_{\infty}$,
  \[ \Delta_{\infty}: A \xrightarrow{u} B \xrightarrow{v} C
    \xrightarrow{w} TA \] where $u = [\bar{u}]$, $v = [\bar{v}]$ and
  $w = [C \xrightarrow{\bar{w}} \Sigma^{-r} TA
  \xrightarrow{(\eta_{-r,0})_{TA}} TA]$. We claim that
  $\Delta_{\infty}$ is an exact triangle in $\C_{\infty}$ and
  $w_{\infty}(\Delta_{\infty}) \leq r$.  Indeed, by construction,
  $\Delta_{\infty}$ is represented by
  $\bar{\Delta}_{\infty}: A \xrightarrow{\bar{u}} B
  \xrightarrow{\bar{v}} C \xrightarrow{(\eta_{-r,0})_A \circ \bar{w}}
  TA$ and $\ceil*{\bar{u}} = \ceil*{\bar{v}} =0$,
  $\ceil*{(\eta_{-r,0})_A \circ \bar{w}} = r +0 = r$. The shifted
  triangle
  $\widetilde \Delta_{\infty} = \Sigma^{0,0,0,-r}
  \bar{\Delta}_{\infty}$ obviously equals $\Delta$, the initial strict
  exact triangle of weight $r$. Thus,
  $[\widetilde{\Delta}_{\infty}] = \Delta_{\infty}$ and
  $w(\widetilde{\Delta}_{\infty}) = r$. Therefore,
  $w_{\infty}(\Delta_{\infty}) \leq r$.
\end{ex}

\begin{rem} \label{rem:ex0tr} A special case of the situation in
  Example \ref{ex-tri-c-c-inf} is worth emphasizing.  Any exact
  triangle in $\C_0$ induces an exact triangle in $\C_{\infty}$ with
  unstable weight equal to $0$.  This implies that any morphism
  $f\in \Mor_{\C_{\infty}}(A,B)$ with $\sigma(f)\leq 0$ can be
  completed to an exact triangle of unstable weight $0$ in
  $\C_{\infty}$. Indeed, we first represent $f$ by a morphism
  $\bar{f}\in \Mor^{\alpha}_{\C}(A,B)$ with $\alpha\leq 0$. If
  $\alpha\not=0$, we shift $\bar{f}$ up using the persistence
  structure maps and denote $\bar{f}'=i_{\alpha,0} (\bar{f})$. We
  obviously have $[\bar{f}']=f$. We then complete $\bar{f}'$ to an
  exact triangle in $\C_{0}$. The image of this triangle in
  $\C_{\infty}$ is exact, of unstable weight $0$, and has $f$ as the
  first morphism in the triple.
\end{rem}

\begin{ex} \label{ex-id-rigid} Consider the strict exact triangle
  $\Delta: A \to 0 \to \Sigma^{-r}TA \xrightarrow{\mathds{1}} \Sigma^{-r} TA$
  in $\C$, which is of weight $r\geq 0$ (see
  Remark~\ref{rmk-shift-notation}~(b)). Let $\Delta_{\infty}$ be the
  following triangle in $\C_{\infty}$
  \[ \Delta_{\infty}: A \to 0 \to \Sigma^{-r} TA
    \xrightarrow{\eta_{-r,0}} TA. \] We claim that if
  \zjnote{$\sigma(\mathds{1}_A) = 0$}, then
  $\bar{w}(\Delta_{\infty})=w_{\infty}(\Delta_{\infty}) = r$.  Indeed,
  assume $w_{\infty}(\Delta_{\infty}) = s<r$.  Then there exists a
  strict exact triangle in $\C$ of the form
  \[ A \to 0 \to \Sigma^{-r'} TA \xrightarrow{w'} \Sigma^{-s} TA, \]
  with $r'\geq r$, of weight $s$ and such that $\ceil*{w'}=0$,
  $[\eta_{-s,0}\circ w'\circ \eta_{-r,-r'}]=[\eta_{-r,0}]$. Notice
  $\ceil*{(\eta_{-s,0}\circ w'\circ \eta_{-r,-r'})}=r-r'+s$. Thus, as
  $r'\geq r>s$, we deduce $\sigma([\eta_{-r,0}])< r$. By writing
  $\mathds{1}_{A}= \eta_{0,-r}\circ \eta_{-r,0}$ we deduce
  $\sigma(\mathds{1}_{A})<0$. We conclude $w_{\infty}(\Delta_{\infty})=r$. We
  next consider a triangle $\Sigma^{k,0,0,k}\Delta_{\infty}$:
  $$\Sigma^{k}A \to 0 \to \Sigma^{-r} TA \to \Sigma^{k}TA$$
  and rewrite it as
  $$ \Delta':B\to 0 \to \Sigma^{-r-k} TB\to TB$$
  with $B= \Sigma^{k}A$. Suppose that $w_{\infty}(\Delta') < r+k$,
  then $\sigma (\mathds{1}_{B}) <0$ by the previous argument. This implies
  $\sigma (\mathds{1}_{A})<0$ and again contradicts our assumption. Thus
  $w_{\infty}(\Delta')=r+k$ and $\bar{w}(\Delta_{\infty})=r$.
\end{ex}

\ocnote{
\begin{rem} \label{rem:sig-id} For an object $A\in\mathcal{O}b(\C)$ we always have $\sigma(\mathds{1}_{A})\leq 0$. We have 
seen just above that if $\sigma([(\eta_{-r,0})_A]) < r$, then 
$\sigma(\mathds{1}_{A})<0$. The same conclusion remains true  if $\sigma([(\eta_{0,-r})_{A}]) < - r$ by the same argument.  Another useful observation is that if the map $\eta_{s}: \Sigma^{s}A\to A$ is an $r$-isomorphism, with $r<s$, then $\sigma(\mathds{1}_{A})\leq r-s$. 
Indeed, if $\eta_{s}$ is an $r$-isomorphism,
then it has a left $r$-inverse $\psi:\Sigma^{r}A\to \Sigma^{s}A$ with $\eta_{s}\circ \psi=\eta_{r}$ which implies $\sigma([\eta_{r}])\leq \sigma ([\eta_{s}])\leq -s $ and thus $\sigma(\mathds{1}_{A})\leq r-s$.
\end{rem}
}

\begin{ex} \label{ex-embed} Let $f \in \Mor_{\C_{\infty}}(A,B)$ and
  suppose that $\sigma(f) > 0$.  We will see here that we can extend
  $f$ to an exact triangle in $\C_{\infty}$, of unstable weight
  $\leq \sigma(f)+\epsilon$ (for any $\epsilon >0$) but, at the same
  time, no triangle extending $f$ has unstable \jznote{weight} less than
  $\sigma(f)$. \pbnote{Fix $\sigma(f) < t \leq \sigma(f)+\epsilon$ and
    let $\bar{f}\in \Mor^{t}_{\C}(A,B)$ be a representative of $f$.}
  Consider the composition
  $f': A \xrightarrow{\bar{f}} B \xrightarrow{(\eta_{0,-t})_B}
  \Sigma^{-t} B$. Then \zjnote{$f' \in \Mor_{\mathcal C}^0(A,\Sigma^{-t} B)$}. There exists
  an exact triangle in $\C_{0}$
  $$\Delta: A \xrightarrow{f'} \Sigma^{-t} B \xrightarrow{v'}
  C \xrightarrow{w} TA$$ for some $C \in {\rm Obj}(\C_0)$. In
  particular \pbnote{$w \in \Mor_{\C}^0(C,TA)$}. Next, consider the
  following triangle
  \[ \Sigma^{0,0,0,-t} \Delta: A \xrightarrow{f'} \Sigma^{-t} B
    \xrightarrow{v'} C \xrightarrow{w'} \Sigma^{-t} TA,
    \,\,\,\,\mbox{where $w' = \eta_{t}^{TA} \circ w$}. \] By
  \pbnote{Remark~\ref{rmk-shift-notation}(b)},
  $\Sigma^{0,0,0,-t} \Delta$ is a strict exact triangle in $\C$ of
  weight $t$. Finally, consider the following triangle in $\C$,
  obtained by shifting up the last three terms of \jznote{$\Sigma^{0,0,0,-t}\Delta$}:
  \[ \bar{\Delta}: A \xrightarrow{\bar{f}= (\eta_{-t, 0})_B \circ f'}
    B \xrightarrow{v: = \Sigma^{t} v'} \Sigma^t C \xrightarrow{w:
      =\Sigma^{t} w'} TA. \] Its image in $\C_{\infty}$ is the
  triangle
  \begin{equation} \label{eq:triangle4} \Delta_{\infty}: A
    \xrightarrow{f=[\bar{f}]} B \xrightarrow{[v]} \Sigma^t C
    \xrightarrow{[w]} TA
  \end{equation}
  and is exact. In the terminology of Definition \ref{dfn-extri-inf},
  the representative $\tilde{\Delta}$ of $\Delta_{\infty}$ is the
  strict exact triangle $\Sigma^{0,0,0,-t} \Delta$. In particular,
  $w_{\infty}(\Delta_{\infty})\leq t$.  Notice that Definition
  \ref{dfn-extri-inf} immediately implies that any triangle
  $\Delta'': A\xrightarrow{f} B\to D\to TA$ in $\C_{\infty}$ satisfies
  $w_{\infty}( \Delta'')\geq \sigma (f)$ (because, with the notation
  of the definition, the weight of the triangle $\tilde{\Delta}$ in
  that definition is at least $q$).  At the same time
  $\bar{w}(\Delta_{\infty})=0$ because
  $\Sigma^{t,0,0,t}(\Delta_{\infty})$ has as representative
  $$\Sigma^{t}\Delta : \Sigma^{t}A\to B\to \Sigma^{t}C\to
  \Sigma^{t}TA$$ which is exact in $\C_{0}$.
\end{ex}

\begin{rem}\label{rem:tr-ex2}
  (a) The definition of the weights of the exact triangles in $\C_{\infty}$ is
  designed precisely to allow for the construction in Example
  \ref{ex-embed}. This is quite different compared to the case when
  the spectral invariant of $f$ is non positive (compare with Remark
  \ref{rem:ex0tr}) because the persistence structure maps can be used
  to ``shift'' up but not down.  It also follows from Example
  \ref{ex-embed} that for $f\in \Mor_{\C_{\infty}}(A,B)$ with
  $\sigma(f)\geq 0$ we have:
  \begin{equation}
    \sigma(f)= \inf \{\  w_{\infty}(\Delta) \ | \ \Delta :
    A\xrightarrow{f} B\to C \to TA \ \}~.~
  \end{equation}

  (b) It is not difficult to see that if
  $w_{\infty}(\Sigma^{s,0,0,s}\Delta)\leq r$ for a triangle $\Delta$
  in $\C_{\infty}$ and $s\geq 0$, then $w_{\infty}(\Delta)\leq
  r+s$. 
\end{rem}

\ocnote{
\begin{ex}\label{ex-embed2} Let $\Delta: A\to B\to C\to TA$ be an
  exact triangle in $\C_{0}$. It is clear that, in $\C_{\infty}$ and
  for $s\leq 0$, the corresponding triangles of the form
  $\Delta_{s}:\Sigma^{s}A\to B\to C\to \Sigma^{s}TA$ (defined using
  the (pre)-composition of the maps in $\Delta$ with the appropriate
  \zjnote{maps} $\eta$'s \zjnote{on $A$}) have the property $\bar{w}(\Delta_{s})=0$. On the other
  hand, if $s>0$ this is no longer the case, in general. Indeed,
  assuming $s>0$ and $\bar{w}(\Delta_{s})=0$ it follows that for some
  possibly even larger $r>0$ we have $w_{\infty}(\Delta_{r})=0$.  This
  means that for any sufficiently small $\epsilon >0$, there is an exact triangle in $\C_{0}$ of the form
  $\Sigma^{r'}A\to \Sigma^{-\epsilon'}B\to C'\to \Sigma^{r'}TA$ together with an $\epsilon$-isomorphism $\phi:C'\to \Sigma^{-\epsilon''}C$ satisfying the conditions in Definition \ref{dfn-set} and with $0\leq \epsilon'\leq \epsilon''\leq \epsilon$ (see Definition \ref{dfn-extri-inf}), $r'\geq r$. This triangle can be
  compared to a shift   $\Sigma^{-2\epsilon}\Delta$ of the initial $\Delta$
  (the constant $2$ is necessary to ensure the commutativity of the rectangles in the comparison diagram - see Remark \ref{rmk- set} (c)). 
  The resulting commutative diagram is:
  \[ \xymatrix{ \Sigma^{r'}A\ar[d]^{\eta^{A}} \ar[r] &
    \Sigma^{-\epsilon'}B'\ar[d]^{ \eta^{B}} \ar[r]    & C'
    \ar[r]\ar[d]^{\eta_{2\epsilon}^{C}\circ\phi} & \Sigma^{r'}TA \ar[d]^{\eta^{A}}\\
    \Sigma^{-2\epsilon}A\ar[r] & \Sigma^{-2\epsilon}B \ar[r]
    &\Sigma^{-2\epsilon} C \ar[r]&\Sigma^{-2\epsilon} TA }\] 
  The map out of $C'$ in this diagram is a $3\epsilon$-isomorphism and, by 
Proposition \ref{cor-1} ii, we deduce that $\eta^{A}_{r'+2\epsilon}$ is a $15\epsilon$-isomorphism.  We can take $\epsilon$ as small as needed and 
 we deduce $\sigma(\mathds 1_{A})\leq -r'\leq -s$.  
\end{ex}
}

\ocnote{
\begin{ex} \label{ex-other} Let
  $\Delta : A\to 0\to \Sigma^{r} TA \to TA$ be a triangle in
  $\C_{\infty}$ with the last map (the class of) $\eta^{TA}_{r}$ and
  with $r\geq 0$.  Assuming $\sigma(\mathds 1_{A})=0$, we claim that
  $w_{\infty}(\Delta) = r$ and $\bar{w}(\Delta)=0$. Indeed, the fact
  that $\bar{w}(\Delta)=0$ is obvious because
  $\Sigma^{r,0,0,r}\Delta :\Sigma^{r}A\to 0\to \Sigma^{r} TA\to
  \Sigma^{r}TA$ is exact in $\C_{0}$.  Now assume that
  $w_{\infty}(\Delta)<r$. Then there is a triangle
  $A\to 0 \to \Sigma^{r-s'} TA\xrightarrow{u} \Sigma^{-s}TA$ which is
  strict exact in $\C$ and with $s\geq s'$ and $w_{\infty}(\Delta)\leq s <r$. Using the definition of
  strict exact triangles and the existence of the exact triangle
  $A\to 0 \to TA\to TA$ in $\C_{0}$, we deduce that there exists an
  $s$-isomorphism $\phi: TA \to \Sigma^{r-s'}TA$ with a right inverse
  $\psi:\Sigma^{s+r-s'}TA \to TA$ that coincides with $\eta^{TA}$ in
  $\C_{\infty}$. As a result, $\phi$ coincides with
  $\eta^{TA}_{0,r-s'}$ in $\C_{\infty}$ and thus
  $0=\ceil*{\phi}\geq r-s'$ ($\sigma (\mathds 1_{A})=0$ implies that
  $\sigma([\eta^{TA}_{0,r-s'}])=r-s'$). Therefore, $s'\geq r$ which
  contradicts \zjnote{our assumption} $s<r$.
\end{ex}
}

\begin{ex}\label{ex:sequences2}
Another examples of interest is given by the triangle in $\C_{\infty}$:
$$\Delta :  0\to X\to \Sigma^{r} X \to 0$$
where the map $X\to \Sigma^{r}X$ is the class \zjnote{of $(\eta_{0,r})_X$}. We
claim that \zjnote{$w_{\infty}(\Delta)\leq |r|$}. We start with the case
$r>0$. In that case, by shifting down the last two objects by $r$, we
obtain a strict exact triangle in $\C$: $(0\to X\to X\to 0, r)$ (here
it is important to view strict exact triangles as pairs
$(\mathrm{triangle},\ \mathrm{weight})$ as in Definition
\ref{dfn-set}). In the case $r\leq 0$, there is a strict exact
triangle in $\C$ \zjnote{$(0\to X\to \Sigma^{r}X\to 0, -r)$} that can be
reached from $\Delta$ by shifting down the last object \zjr{by $|r|$}. This
shows the claim.  Additionally, it is easy to see that if
$\sigma(\mathds 1_{X})=0$, then $w_{\infty}(\Delta)=r$.
\end{ex}

\begin{ex}\label{ex:sequences3} In this example we consider an exact
  triangle in $\C_{\infty}$
  $$\Delta \ : \  0\to X\to Y\to 0 ~.~$$
  We claim that
  $\bar{w}(\Delta)= \inf \{ r \geq 0 \ | \ \exists \ \text{an
    $r$-isomorphism } \phi : \Sigma^{s} X \to Y \text{ with } 0\leq
  s\leq r\}$. By the definition of the weight of \zjnote{triangles in $\C_{\infty}$,} we are looking for strict exact triangles in $\C$ of the
  form
  \zjnote{$$(0\to X\to \Sigma^{-s} Y\to \Sigma^{-r}0, r)$$} with $s\leq r$.
  The unstable weight, in this case equal to the stable weight, is
  obtained by infimizing $r$. The existence of such a strict exact
  triangle is equivalent to the existence of an $r$-isomorphism
  $\phi: X\to \Sigma^{-r}Y$ which shows the claim.
\end{ex}

\begin{ex}\label{eq:sequences4} Consider an exact triangle in $\C_{\infty}$
of the form:
$$\Delta \ : \ T^{-1} X\to  0\to  Y\to X ~.~$$
We claim that
$w_{\infty}(\Delta)= \inf \{ r \geq 0 \ | \ \exists \ \text{an
  $r$-isomorphism } \phi : \Sigma^{s} X \to Y \text{ with } 0\leq
s\leq r\}$ and
$\bar{w}(\Delta)= \inf \{ r \geq 0 \ | \ \exists \ \text{an
  $r$-isomorphism } \phi : \Sigma^{s} X \to Y \text{ with } s \geq 0
\}$. The relevant strict exact triangles (of weight $r$) in this case
are:
$$T^{-1}X\to 0 \to \Sigma^{-s}Y\to \Sigma^{-r}X$$ with $s\leq r$.
Again, the existence of such a triangle is equivalent to the existence
of an $r$-isomorphism $\phi : X\to \Sigma^{-s}Y$ which provides the
estimate for $w_{\infty}$.  To estimate $\bar{w}$ we apply the same
argument but by replacing $X$ by $\Sigma^{k}X$ with $k$ positive. The
condition becomes $\phi: \Sigma^{k}X\to \Sigma^{-s}Y$ is an
$r$-isomorphism, and $s\leq r$.
 \end{ex}

\begin{ex}\label{eq:sequences5} Consider an exact triangle in $\C_{\infty}$ of 
  the form: $$\Delta \ : T^{-1}X\to K\to Y\to X ~.~$$ where $K$ is
  $k$-acyclic.  We claim that
  $$\bar{w}(\Delta)\geq \inf \{ r \geq 0 \ | \ \exists
  \text{ an $(r+ 2 k)$-isomorphism } \phi : \Sigma^{s} X \to Y \text{
    with } s\geq 0\}.$$ We thus assume that there exists a strict
  exact triangle
  $\Delta' \ : \ T^{-1}X\to \Sigma^{-s_{1}}K\to \Sigma^{-s_{2}}Y\to
  \Sigma^{-r}X$ of weight $r$, with $s_{1}\leq s_{2}\leq r$, that
  represents $\Delta$. As this is strict exact, there is an exact
  triangle in $\C_{0}$ of the form
  $T^{-1}X\to \Sigma^{-s_{1}}K\to C\to X$ that maps to $\Delta$ as in
  Definition \ref{dfn-set}. In particular, there is an $r$-isomorphism
  $C\to \Sigma^{-s_{2}}Y$. \ocnote{The fact that $K\simeq_{k}0$ immediately
  implies that there is a $k$-isomorphism $C\to X$. We consider a right $k$-inverse of it  $\Sigma^{k}X\to C$. This a $2k$-isomorphism. By composition we
  get a \zjnote{$(2k+r)$}-isomorphism $\Sigma^{k}X\to \Sigma^{-s_{2}}Y$. The constraint
  $s_{2}\leq r$ is eliminated by applying the same argument as at
  Example \ref{eq:sequences4}, by replacing \zjr{$X$ with $\Sigma^{s}X$ for some non-negative $s$}.}
\end{ex}

\begin{rem}\label{rem:rot-w} It is useful to know how the weights of triangles \zjnote{in $\C_{\infty}$} behave with respect to rotation.
Thus let 
$\Delta: A \xrightarrow{u} B \xrightarrow{v} C \xrightarrow{w} TA$ be
an exact triangle in $\C_{\infty}$ of unstable weight $r$.  Then its first positive
rotation
$R(\Delta): B \xrightarrow{v} C \xrightarrow{w} TA \xrightarrow{-Tu}
TB$ is of unstable weight at most $2r$. 
Indeed, by definition,
there exists a triangle
$\bar{\Delta}: A \xrightarrow{\bar{u}} B \xrightarrow{\bar{v}} C
\xrightarrow{\bar{w}} TA$ representing $\Delta$ such that the shifted
triangle
$\widetilde{\Delta}: A \xrightarrow{u'} \Sigma^{-t_1} B
\xrightarrow{v'} \Sigma^{-t_1-t_2} C \xrightarrow{w'} \Sigma^{-r} TA$
(where we put $t_1 = \ceil*{\bar{u}}$, $t_2 = \ceil*{\bar{v}}$,
$t_3 = \ceil*{\bar{w}}$ and $r = t_1 + t_2 + t_3$) is a strict exact
triangle in $\C$ of weight $r$.  By Proposition \ref{prop-rot}, the
first positive rotation
\[ R(\widetilde{\Delta}): \Sigma^{-t_1} B \xrightarrow{v'}
  \Sigma^{-t_1-t_2} C \xrightarrow{w''} \Sigma^{-r} TA
  \xrightarrow{u''} \Sigma^{-t_1 - 2r} TB \] is a strict exact
triangle in $\C$ of weight $2r$. We have
\begin{align*} 
u'' & = -(\eta_{0, -t_1 - 2r})_B \circ \bar{u} \circ (\eta_{-r, 0})_A\\
v' & = (\eta_{0, -t_1 -t_2})_C \circ \bar{v} \circ (\eta_{-t_1, 0})_B \\
w''&\simeq_r w = (\eta_{0, -r})_A \circ \bar{w} \circ (\eta_{-t_1 - t_2, 0})_C. 
\end{align*}
Consider a new triangle: 
\begin{equation} \label{rot-tilde-tri} \Sigma^{t_1}
  R(\widetilde{\Delta}): B \xrightarrow{\Sigma^{t_1} v'} \Sigma^{-t_2}
  C \xrightarrow{\Sigma^{t_1} w''} \Sigma^{-r+t_1} TA
  \xrightarrow{\Sigma^{t_1} u''} \Sigma^{-2r} TB.
\end{equation}
By Remark \ref{rmk-shift-notation} (1),
$\Sigma^{t_1} R(\widetilde{\Delta})$ is also a strict exact triangle
in $\C$ of weight $2r$. By shifting up this triangle we get to
$B\to C\to TA \to TB$ that represents $R(\Delta)$ and thus
$w_{\infty}(R(\Delta))\leq 2r$. In a similar way, using Remark
\ref{rem:neg-rot}, one can treat the negative rotation of $\Delta$,
$R^{-1}(\Delta)$.
\end{rem}

\subsection{Proof of Theorem \ref{thm:main-alg}} \label{subsubsec:proofTh1} 
There are two steps. The first is to show that each exact triangle in $\C_{\infty}$ has finite unstable weight. The second step is to show that $\bar{w}$ satisfies Definition \ref{def:triang-cat-w} and is subadditive with $\bar{w}_{0}=0$.
\subsubsection{Every triangle in $\C_{\infty}$ has finite unstable weight}\label{subsubsec:fuw}
Let $\Delta: A \xrightarrow{u} B \xrightarrow{v} C \xrightarrow{w} TA$
be an exact triangle in $\C_{\infty}$. Thus, there exists a triangle
$\bar{\Delta}: A \xrightarrow{\bar{u}} B \xrightarrow{\bar{v}} C
\xrightarrow{\bar{w}} TA$ in $\C$ and an exact triangle in $\C_{0}$,
$\Delta': A' \xrightarrow{u'} B' \xrightarrow{v} C' \xrightarrow{w}
TA'$, together with isomorphisms in $\C_{\infty}$, $a:A'\to A$,
$b:B'\to B$, $c:C'\to C$ such that the resulting map $(a,b,c,-Ta)$ of
triangles in $\C_{\infty}$ is an isomorphism of triangles. We may
assume that the shifts, $x$ of $\bar{u}$, $y$ of $\bar{v}$, and $z$ of
$\bar{w}$, are all non-negative.

We represent the maps $a$, $b$, $c$ by $s$-isomorphisms
$\bar{a}: \Sigma^{r}A'\to A$, $\bar{b}: \Sigma^{r}B'\to B$,
$\bar{c} :\Sigma^{r}C'\to C$ which is possible by taking $s$ and $r$
sufficiently large. We now consider the diagram in $\C_{0}$:

\[ \xymatrix{ \Sigma^{r}A' \ar[d]_{\bar{a}}
    \ar[r]^{\Sigma^{r}u'} &   \Sigma^{r}B'   \ar[d]^{ \eta_{x}\circ \bar{b}}\\
    A\ar[r]_{\eta_{x}\circ \bar{u}} & \Sigma^{-x}B }\] This diagram
$r'$-commutes for $r'$ sufficiently large. Lemma \ref{lemma-comp}
shows that by taking $r$ sufficiently large we may assume that the
above square commutes in $\C_{0}$ (with the price of both $r$ and $s$
being very large). We complete the diagram to the right thus getting a
new diagram in $\C_{0}$

\[ \xymatrix{ \Sigma^{r}A' \ar[d]_{\bar{a}} \ar[r]^{\Sigma^{r}u'} &
    \Sigma^{r}B' \ar[d]^{ \eta_{x}\circ \bar{b}} \ar[r]^{\Sigma^{r}v'}
    & \Sigma^{r}C'
    \ar[r]^{\Sigma^{r}w'}\ar[d]^{c''} & \Sigma^{r}TA' \ar[d]^{a''}\\
    A\ar[r]_{\eta_{x}\circ \bar{u}} & \Sigma^{-x}B \ar[r]_{v''}
    &\Sigma^{-x-y-t} C \ar[r]_{w''}&\Sigma^{-x-y-t-t'-w} TA }\] with
each square commutative and with the arrows marked with $(-)''$ being
compositions of $(-)$ with the appropriate $\eta$'s. The commutativity
in $\C_{0}$ for the \zjr{second square} requires possibly a shift by $-t$ of
the third term in the bottom row. Similarly, the commutativity of the
\zjr{third square} requires a shift by $-t'$ for the last term of the row
(this is a variant of Proposition \ref{prop-ind}). Obviously, the top
row is exact in $\C_{0}$ and the vertical maps are $k$-isomorphisms,
each one for a different $k$. With the price of yet again increasing
$t$ and $t'$ we can intercalate between the two rows a new exact
triangle in $\C_{0}$.
 
\[ \xymatrix{
    \Sigma^{r}A' \ar[d]_{\bar{a}}    \ar[r]^{\Sigma^{r}u'} &   \Sigma^{r}B'   \ar[d]^{ \eta_{x}\circ \bar{b}} \ar[r]^{\Sigma^{r}v'} & \Sigma^{r}C' \ar[r]^{\Sigma^{r}w'}\ar[d]^{c'''} & \Sigma^{r}TA' \ar[d]^{\bar{a}}\\
    A\ar[r]_{\eta_{x}\circ \bar{u}}\ar[d] & \Sigma^{-x}B \ar[r]_{v''} \ar[d] & C''\ar[d]^{d}\ar[r]& TA\ar[d]^{e}\\
    A\ar[r]_{\eta_{x}\circ \bar{u}} & \Sigma^{-x}B \ar[r]_{v''}
    &\Sigma^{-x-y-t} C \ar[r]_{w''}&\Sigma^{-x-y-t-t'-w} TA }\] The
map $c'''$ is induced by the first two vertical maps between the top
rows and thus it is an $k$-isomorphism for some \zjnote{sufficiently large} $k$. Given that
$c''$ is such an isomorphism too, by possibly increasing $t$ we obtain
the existence of a $k'$- isomorphism $d$ with $k'$ very large. By
possibly increasing $t'$ we can also get the commutativity of the
bottom right square. Again by possibly increasing $t'$ we may ensure
that $k'=x+y+t+t'+w$ (recall that any $k'$-isomorphism is also a
$k''$-isomorphism for $k''\geq k'$). This means that the bottom row is
a strict exact triangle in $\C$ of weight $k'$. To end the proof, we
notice that this triangle is of the form $\widetilde{\Delta}$ as in
Definition \ref{dfn-extri-inf}.

 \
 
\subsubsection{\zjnote{The weighted octahedral axiom in $\C_{\infty}$}} \label{subsubsec:proofTh2}

To  finish the proof of Theorem \ref{thm:main-alg} we need to show that the weighted octahedral axiom is
satisfied by $\bar{w}$, that  $\bar{w}$ satisfies the normalization
in Definition \ref{def:triang-cat-w} (ii) with $\bar{w}_{0}=0$, and that it is 
subadditive. We start below with the weighted octahedral axiom and will end with
the other properties.

\begin{lem} \label{lem-oct-cinf}
  The weight
  $\bar{w}$ as defined in Definition \ref{dfn-extri-inf} satisfies the
  weighted octahedral axiom from Definition \ref{def:triang-cat-w}
  (i).
\end{lem}
\begin{proof} 
  Recall that given the exact triangles
  $\Delta_{1}: A\to B\to C\to TA$ and $\Delta_{2}: C\to D\to E\to TC$
  in $\C_{\infty}$ we need to show that there are exact triangles:
  $\Delta_{3}: B\to D\to F\to TB$ and
  $\Delta_{4}: TA\to F\to E\to T^{2}A$ making the diagram below
  commute, except for the right-most bottom square that anti-commutes,
  \[ \xymatrix{
      A \ar[r] \ar[d] & 0 \ar[r] \ar[d] & TA \ar[d]\ar[r]& TA\ar[d] \\
      B \ar[r] \ar[d] & D \ar[r] \ar[d] & F \ar[d]\ar[r]& TB \ar[d]\\
      C \ar[r] \ar[d] & D \ar[r]\ar[d] & E\ar[r]\ar[d] & TC\ar[d] \\
      TA\ar[r] & 0 \ar[r] & T^{2}A \ar[r]& T^{2} A }
  \]
  and such that
  $\bar{w}(\Delta_{3}) + \bar{w}(\Delta_{4})\leq \bar{w}(\Delta_{1}) +
  \bar{w}(\Delta_{2})$.

  In $\C$, there are triangles $\bar{\Delta}_1: A \to B \to C \to TE$
  with non-negative morphisms shifts $t_1, t_2, t_3$ and
  $\bar{\Delta}_2: C \to D \to E \to TC$ with non-negative morphisms
  shifts $k_1, k_2, k_3$ that represent $\Delta_{1}$ and $\Delta_{2}$
  respectively and such that the associated triangles
  \[ \widetilde{\Delta}_1: A \to \Sigma^{-t_1} B \to \Sigma^{-t_1-t_2}
    C \to \Sigma^{-r} TA \,\,\,\,\,\mbox{where $r = t_1 + t_2 +
      t_3$} \] and
  \[ \widetilde{\Delta}_2: C \to \Sigma^{-k_1} D \to \Sigma^{-k_1
      -k_2} E \to \Sigma^{-s} TC \,\,\,\,\,\mbox{where
      $s = k_1 + k_2 + k_3$} \] are strict exact triangles in $\C$ of
  weight $r$ and $s$, respectively.  Consider
  $\Sigma^{-t_{1}-t_{2}} \widetilde{\Delta}_2$,
  \[ \Sigma^{-t_{1}-t_{2}} \widetilde{\Delta}_2: \Sigma^{-t_{1}-t_{2}}
    C \to \Sigma^{-k_1 -t_{1}-t_{2}} D \to
    \Sigma^{-k_1-k_2-t_{1}-t_{2}} E \to \Sigma^{-s-t_{1}-t_{2}} TC. \]
  The weighted octahedral property for strict exact triangles in $\C$
  in Proposition \ref{prop-w-oct} implies that we can construct the
  following commutative diagram in $\C$ (with the bottom right square
  which is $r$-anti-commutative).
  \begin{equation} \label{oct-c-inf-2} \xymatrix{
      A \ar[d] \ar[r] & 0 \ar[r] \ar[d] & TA \ar[d] \ar[r] & TA\ar[d] \\
      \Sigma^{-t_1} B \ar[d] \ar[r] & \Sigma^{-k_1-t_{1}-t_{2}}
      D\ar[r] \ar[d] & D' \ar[d]\ar[r]
      & \Sigma^{-t_1} TB  \ar[d]\\
      \Sigma^{-t_{1}-t_{2}} C \ar[r]\ar[d] & \Sigma^{-k_1-t_{1}-t_{2}}
      D \ar[r] \ar[d] & \Sigma^{-k_1-k_2-t_{1}-t_{2}} E \ar[r] \ar[d]
      & \Sigma ^{-t_{1}-t_{2}-s}TC \ar[d]\\
      \Sigma^{-r}TA \ar[r]& 0\ar[r] & \Sigma^{-r-s}T^{2}A\ar[r] &
      \Sigma^{-r-s}T^{2} A}
  \end{equation}
  Here the triangle
  $\Delta_{3}' :\Sigma^{-t_1} B \to \Sigma^{-k_1-t_{1}-t_{2}} D \to D'
  \to \Sigma^{-t_1} TB$ is an exact triangle in $\C_{0}$. The triangle
  $\Delta_{3}'' : \Sigma^{k_{1}+t_{2}}B\to D\to
  \Sigma^{k_{1}+t_{1}+t_{2}}D'\to \Sigma^{k_{1}+t_{2}}TB$ obtained by
  shifting up $\Delta_{3}'$ by $k_{1}+t_{1}+t_{2}$ is also exact in
  $\C_{0}$. Let $[\Delta''_{3}]$ be the image of this triangle in
  $\C_{\infty}$. We put $F= \Sigma^{k_{1}+t_{1}+t_{2}}D'$ and take
  $\Delta_{3}$ to be the triangle in $\C_{\infty}$
  $$\Delta_{3}: B\to D\to F\to TB$$ obtained  by applying
  $\Sigma^{-k_{1}-t_{2},0,0,-k_{1}-t_{2}}$ to $[\Delta_{3}'']$. We
  obviously have $w_{\infty}([\Delta_{3}''])=0$ and thus
  $\bar{w}(\Delta_{3})=0$.

  The next step is to identify the triangle $\Delta_{4}$. Proposition
  \ref{prop-w-oct} implies that the third column in
  (\ref{oct-c-inf-2}):
  $$\Delta_{4}' : TA\to \Sigma^{-k_{1}-t_{1}-t_{2}}F
  \to \Sigma^{-k_{1}-k_{2}-t_{1}-t_{2}}E\to \Sigma^{-r-s}T^{2}A$$ is a
  strict exact triangle in $\C$. We let $[\Delta'_{4}]$ be the image
  of this triangle in $\C_{\infty}$ and let $\Delta_{4}$ be given by
  applying $\Sigma^{0, k_{1}+t_{1}+t_{2},k_{1}+k_{2}+t_{1}+t_{2},r+s}$
  to $[\Delta'_{4}]$:
  $$\Delta_{4} : TA \to F \to E\to T^{2}A~.~$$
  We deduce from Definition \ref{dfn-extri-inf} that
  $w_{\infty}(\Delta_{4})\leq r+s$ and thus also
  $\bar{w}(\Delta_{4})\leq r+s$.

  The commutativity required in the statement follows from that
  provided by Proposition \ref{prop-w-oct} for (\ref{oct-c-inf-2}).
\end{proof}

\begin{remark} \label{rem:weak-oct} For the triangle $\Delta_{3}$
  produced in this proof it is easy to see that
  $w_{\infty}(\Delta_{3})\leq k_{1}+t_{2}$.  Therefore we have:
  \[ w_{\infty}(\Delta_3) + w_{\infty}(\Delta_4) \leq (k_1 + t_2) +
    (r+s) \leq 2
    (r+s)=2(w_{\infty}(\Delta_{1})+w_{\infty}(\Delta_{2})). \] Thus
  the weight $w_{\infty}$ satisfies a weak form of the weighted
  octahedral axiom.
\end{remark}

The next step in proving Theorem \ref{thm:main-alg} is to show the
normalization property in Definition \ref{def:triang-cat-w} (ii). This
property is satisfied with the constant $\bar{w}_{0}=0$. Indeed, any
triangle $0\to X\to X\to 0$ and all its rotations are exact in
$\C_{0}$ and thus they are of unstable weight equal to $0$.  The last
verification needed is to see that, if $B=0$ in the diagram of the
weighted octahedral axiom, then the triangle $\Delta_{3}$ -
constructed in the proof of the \zjnote{Lemma \ref{lem-oct-cinf}} - can be of the form:
$\Delta_{3}: 0\to D\to D\to 0$. This is trivially satisfied in our
construction because if $B=0$ we may take $t_{1}=t_{2}=0$ and the
triangle
$\Delta'_{3}: 0 \to \Sigma^{-k_{1}}D\xrightarrow{\mathds{1}}
\Sigma^{-k_{1}}D\to 0$.

\

Finally, to finish the proof of Theorem \ref{thm:main-alg} we need to
show that $\bar{w}$ is sub-additive. Thus, assuming that
$\Delta: A\to B\to C\to TA$ is exact in $\mathcal{C}_{\infty}$ and $X$
is an object in $\C$, then
$\bar{w}(X\oplus \Delta)\leq \bar{w}(\Delta)$ where the triangle
$X\oplus\Delta$ has the form
$X\oplus \Delta: A\to X\oplus B\to X\oplus C\to TA$. We consider the
strict exact triangle in $\C$
$$\widetilde{\Delta}: A\to \Sigma^{-s_{1}}B
\to \Sigma^{-s_{1}-s_{2}}C \to \Sigma^{-s_{1}-s_{2}-s_{3}}TA$$
associated to $\Delta$ as in Definition \ref{dfn-extri-inf} with
$s_{i}\geq 0$, $1\leq i\leq 3$.  Consider the triangle
\jznote{$$\Delta' : 0\to \Sigma^{-s_{1}}X
  \xrightarrow{\eta_{s_{2}}^{X}} \Sigma^{-s_{1}-s_{2}}X\to 0~.~$$}
This triangle is obtained from the exact triangle in $\C_{0}$,
$0\to \Sigma^{-s_{1}}X\xrightarrow{\mathds{1}} \Sigma^{-s_{1}}X\to 0$
by applying $\Sigma^{0,0,-s_{2}, -s_{2}}$ and it is of weight
$\leq s_{2}$. By Lemma \ref{claim-frag-sum} we have
$w(\Delta' \oplus \widetilde{\Delta})\leq w(\widetilde{\Delta})$.  We
now notice that $\Delta'\oplus \widetilde{\Delta}$ can be viewed as
obtained from $X\oplus\Delta$ by applying
$\Sigma^{0,-s_{1},-s_{1}-s_{2}, -s_{1}-s_{2}-s_{3}}$ and thus
$w_{\infty}(X\oplus \Delta)\leq w(\widetilde{\Delta})$ which implies
the claim. The proof for $\Delta\oplus X$ is similar.


\subsection{Some properties of fragmentation pseudo-metrics.}\label{subsec:rem-nondeg}
The purpose of this section is to rapidly review some of the
properties of the persistence fragmentation pseudo-metrics.  We start
by recalling the main definitions, we then discuss some algebraic
properties, and relations to standard notions such as the interleaving
distance.
\subsubsection{Persistence fragmentation pseudo-metrics, review of
  main definitions} \label{subsubsec:prop-fr} Assume that $\C$ is a
TPC. We have defined in \S \ref{subsubsec:frag1} and
\S\ref{subsubsec:ex-cinfty} three types of similarly defined
measurements on the objects of $\C$ that, after symmetrization, define
fragmentation pseudo-metrics on $\mathrm{Obj}(\C)$.  In \S
\ref{subsubsec:frag1} this construction uses directly the weight $w$
of the strict exact triangles in $\C$ (making use of Proposition
\ref{prop-w-oct}) and it gives rise to pseudo-metrics
$d^{\mathcal{F}}$ as in Definition \ref{dfn-frag-met} as well as a
simplified version $\underline{d}^{\mathcal{F}}$ mentioned in Remark
\ref{ex-delta} (d).

In \S\ref{subsubsec:ex-cinfty} we have endowed the triangles of the
category $\C_{\infty}$ with weights: an unstable weight $w_{\infty}$
as well as with a (smaller) stable weight $\overline{w}$, as given by
Definition \ref{dfn-extri-inf}. Working in the category $\C_{\infty}$
has a significant advantage compared to the category $\C$ because, by
contrast to $\C$, in $\C_{\infty}$ any morphism can be completed to an
exact triangle of finite weight.  Moreover, when we have a
(translation invariant) family of triangular generators $\mathcal{F}$,
the fragmentation pseudo-distance relative to $\mathcal{F}$ is finite
- see Remark~\ref{rem:finite-metr}~(a).  Finally, in $\C_{\infty}$
exact triangles have the standard form expected in a triangulated
category and do not involve shifts. Therefore, we will focus here on
the fragmentation pseudo-metrics defined using $w_{\infty}$ and,
mainly, $\overline{w}$.

An important remark at this point is that the unstable weight
$w_{\infty}$ does not satisfy the weighted octahedral axiom (but only
its weak form as discussed in Remark \ref{rem:weak-oct}) and thus only
the pseudo-metrics of the form $\underline{\bar{d}}^{\mathcal{F}}$ can
be defined using it. By contrast, $\overline{w}$ does satisfy the
weighted octahedral axiom and there is a pseudo-metric
$\bar{d}^{\mathcal{F}}$ associated to it through the construction in
\S\ref{subsec:triang-gen}.  Both $\overline{w}$ and $w_{\infty}$ are
subadditive and satisfy the normalization property in Definition
\ref{def:triang-cat-w} with $w_{0}=0$.

To eliminate possible ambiguities we recall the definitions of the two
relevant pseudo-metrics here. Both of them are based on considering a
sequence of exact triangles in $\C_{\infty}$ as below:
\begin{equation}\label{eq:iterated-tr2} \xymatrixcolsep{1pc} \xymatrix{
    Y_{0} \ar[rr] &  &  Y_{1}\ar@{-->}[ldd]  \ar[r] &\ldots
    \ar[r]& Y_{i} \ar[rr] &  &  Y_{i+1}\ar@{-->}[ldd]
    \ar[r] &\ldots \ar[r]&Y_{n-1} \ar[rr] &   &Y_{n} \ar@{-->}[ldd]  &\\
    &         \Delta_{1}                  &  & & &  \Delta_{i+1}                          & &  &  &    \Delta_{n}             \\
    & X_{1}\ar[luu] &  & & &X_{i+1}\ar[luu] &  &  & &X_{n}\ar[luu] }
\end{equation}
where the dotted arrows represent maps $Y_{i} \to TX_{i}$. We fix a
family of objects $\mathcal{F}$ in $\C$ with $0\in \mathcal{F}$ and
define:
\begin{equation}\label{eq:semi-metrics1}
  \underline{\bar{\delta}}^{\mathcal{F}}(X,X') =
  \inf\left\{ \sum_{i=1}^{n}w_{\infty}(\Delta_{i}) \, \Bigg| \,
    \begin{array}{ll} \mbox{$\Delta_{i}$
      are successive exact triangles as in (\ref{eq:iterated-tr2})} \\
      \mbox{with $X'=Y_{0}$,
      $X=Y_{n}$  and  $X_i \in \mathcal{F}$,   $n \in \N$}
    \end{array} \right\}
\end{equation}

\begin{equation}\label{eq:delta-fr}
  \bar{\delta}^{\mathcal{F}}(X,X') =
  \inf\left\{ \sum_{i=1}^{n}\bar{w}(\Delta_{i}) \,
    \Bigg| \,
    \begin{array}{ll} \mbox{$\Delta_{i}$ are as in (\ref{eq:iterated-tr2})
      with $Y_{0}=0$, $X=Y_{n}$, $X_i \in \mathcal{F}$,   } \\
      \mbox{$n \in \N$  except for some $j$ such that $X_{j}=T^{-1}X'$}
    \end{array} \right\}~.~
\end{equation}
Finally, the pseudo-metrics $\underline{\bar{d}}^{\mathcal{F}}$ and
$\bar{d}^{\mathcal{F}}$ are obtained by symmetrizing
$\underline{\bar{\delta}}^{\mathcal{F}}$ and
$\bar{\delta}^{\mathcal{F}}$, respectively:
$$\underline{\bar{d}}^{\mathcal{F}}(X,X') =
\max\{\underline{\bar{\delta}}^{\mathcal{F}}(X,X'),
\underline{\bar{\delta}}^{\mathcal{F}}(X',X) \}, \ \
\bar{d}^{\mathcal{F}}(X,X')=\max\{\bar{\delta}^{\mathcal{F}}(X,X'),
\bar{\delta}^{\mathcal{F}}(X',X) \}~.~$$

\subsubsection{Algebraic properties.}\label{subsubsec:prop-alg}
There are many fragmentation pseudo-metrics of persistence type
associated to the same weight, depending on the choices of family
$\mathcal{F}$.  In fact, the choices available are even more abundant
for the following two reasons.
\begin{itemize}
\item[(i)] Triangular weights themselves can be mixed. For instance, if
  $\C$ is a TPC, there is a triangular weight of the form
  $\overline{w}^{+}=\overline{w}+w_{fl}$ that is defined on
  $\C_{\infty}$ (where $w_{fl}$ is the flat weight \jznote{defined in
    \S\ref{subsec:triang-gen}}).
\item[(ii)] Fragmentation metrics themselves can also be mixed. If
  $d^{\mathcal{F}}$ and $d^{\mathcal{F}'}$ are two fragmentation
  pseudo-metrics (whether defined with respect to the same weight or
  not), then the following expressions
  $\alpha ~d^{\mathcal{F}}+\beta~ d^{\mathcal{F}'}$ with
  $\alpha, \beta\geq 0$ as well as
  $\max \{d^{\mathcal{F}}, d^{\mathcal{F}'}\}$ are also
  pseudo-metrics.
\end{itemize}

In essence, while it is not easy to produce interesting sub-additive
triangular weights on a triangulated category, once such a weight is
constructed - as in the case of the persistence weight $\overline{w}$
defined on $\C_{\infty}$ (where $\C$ is a TPC) - one can associate to
it a large class of pseudo-metrics, either by combining the weight
with the flat one and/or by ``mixing'' the pseudo-metrics associated
to different families $\mathcal{F}$.  Another useful (and obvious)
property relating the pseudo-metrics $d^{\mathcal{F}}$ and
$d^{\mathcal{F}'}$ associated to the same triangular weight is that:
 
\begin{itemize}
\item[(iii)] If $\mathcal{F}\subset \mathcal{F}'$, then
  $d^{\mathcal{F}'}\leq d^{\mathcal{F}}$.
\end{itemize}

The last useful construction has to do with making the metrics
invariant with respect to the action of the shift functor.
\begin{itemize}
\item[(iv)] For a given fragmentation metric $d^{\mathcal{F}}$ we
  define its shift-invariant version
  \begin{equation}\label{eq:inv-shift-metr}\widehat{d^{\mathcal{F}}}(X,Y)=\max
    \{ \widehat{\delta^{\mathcal{F}}}(X,Y),
    \widehat{\delta^{\mathcal{F}}}(Y,X) \} ~.~
  \end{equation}
\end{itemize}
Here
$$\widehat{\delta^{\mathcal{F}}}(X,Y) =
\inf_{r,s \in\R} \delta^{\mathcal{F}}(\Sigma^{r}X,\Sigma^{s}Y)$$ is
the shift invariant version of the semi (pseudo)-metrics
$\delta^{\mathcal{F}}$ as in (\ref{eq:semi-metrics1}) and
(\ref{eq:delta-fr}).  It is immediate to see that
$\widehat{\delta^{\mathcal{F}}}$ satisfies the triangle inequality. By
symmetrizing, we obtain indeed a pseudo-metric that is obviously
bounded above by $d^{\mathcal{F}}$.  In case the family $\mathcal{F}$
is closed under the action of $\Sigma^{r}$ for all $r\in \R$, the
metrics of type $d^{\mathcal{F}}$ have the property that
$d^{\mathcal{F}}(X,Y)= d^{\mathcal{F}}(\Sigma^{r}X,\Sigma^{r} Y)$ for
all $r\in \R$. In this case, the shift invariant metric associated to
$d^{\mathcal{F}}$ has a simpler form
$\widehat{d^{\mathcal{F}}}(X,Y)=\inf_{r\in\R}d^{\mathcal{F}}(\Sigma^{r}X,Y)$. The
interest of this type of shift-invariant pseudo-metric is that it
compares the ``shape'' of objects  by contrast to a comparison of the
objects themselves that is sensitive to translations (the spectral
distance in symplectic topology is of this type).  Thus, for instance, two Morse
functions $f$ and $f+k$ with $k\in \R$ are not distinguished by shift-invariant type pseudo-metrics.

\subsubsection{Vanishing and non-degeneracy of fragmentation
  metrics.}\label{subsubsec:non-deg-v}
We fix here a TPC denoted by $\C$ together with the associated weights
and pseudo-metrics, as above. \zjnote{We will denote by $\bar{d}^{\{0\}}$ the
pseudo-metric associated to the family consisting of only the element
$0$}. In view of point (iii) above $\bar{d}^{\{0\}}$ is \jznote{an} upper
bound for all the pseudo-metrics $\bar{d}^{\mathcal{F}}$.

It is obvious, as noticed in Remark~\ref{ex-delta}, that in general
$\bar{d}^{\mathcal{F}}$ is degenerate. For instance, if
$\mathcal{F}=\mathrm{Obj}(\mathcal{C})$ then
$\bar{d}^{\mathcal{F}}\equiv 0$. The rest of Remark~\ref{ex-delta}
also continues to apply to $\bar{d}^{\mathcal{F}}$. We list below some
other easily proven properties. We assume for all the objects $X$
involved here that $\sigma(\mathds{1}_{X})=0$ and we will use the
calculations in Examples \ref{ex-id-rigid}, \ref{ex-embed},
\ref{ex-embed2} and \ref{ex-other}. Recall the notion of
$r$-isomorphism from Definition \ref{dfn-r-iso}, in particular, this
is a morphism in $\mathcal{C}_{0}$. A $0$-isomorphism is simply an
isomorphism in the category $\C_{0}$ and is denoted by $\equiv$.
\begin{itemize}
\item[(i)] If $X\equiv X'$, then $\bar{d}^{\mathcal{F}}(X,X')=0=
  \underline{\bar{d}}^{\mathcal{F}}(X,X')$ for any family
  $\mathcal{F}$.
\item[(ii)] \ocnote{ We have
    $2\ \bar{\delta}^{\{0\}}(X,X')\geq \inf\{ r \in \R \ |\ \exists
    \text{ an $r$-isomorphism } \phi : \Sigma^{k}X'\to X \text{ for
      some } k\geq 0\}\geq \bar{\delta}^{\{0\}}(X,X')~.~$ For the
    first inequality, consider a sequence of triangles
    (\ref{eq:iterated-tr2}) of total weight $\leq r$. We intend to
    show that, for some $k\geq 0$, there exists a $2r$-isomorphism
    $\Sigma^{k} X'\to X$. We make use of Examples \ref{ex:sequences3},
    \ref{eq:sequences4}, \ref{eq:sequences5}.  We assume without loss
    of generality that $T^{-1}X'$ appears in the $j$'th triangle. The
    first triangles are of the form $0\to Y_{i}\to Y_{i+1}\to 0$ for
    $i\leq j-1$ with $Y_{0}=0$. By Example \ref{eq:sequences4} we
    deduce that $Y_{j-1}$ is $r_{0}$- acyclic and $r_{0}\leq$ the sum
    of the weights of the first $j-1$ triangles. The next triangle is
    of the form $T^{-1}X'\to Y_{j-1}\to Y_{j}\to X'$ and of weight
    $r_{j}$.  Example \ref{eq:sequences5} shows that there exists a
    \zjnote{$(2 r_{0}+r_{j})$-isomorphism} $X'\to \Sigma^{-k}Y_{j}$.
    The next triangles, of the form $0\to Y_{i}\to Y_{i+1}\to 0$, have
    weights $r_{i}$, and there are $r_{i}$-isomorphisms
    $Y_{i}\to \Sigma^{-s_{i}}Y_{i+1}$ with $r_{i}\geq s_{i}$ (see
    Example \ref{ex:sequences3}). Putting things together
    $2 r\geq 2 r'\geq 2 r_{0}+ r_{j} +\ldots + r_{n}$ and there is an
    $2 r'$ -isomorphism
    $X'\to \Sigma^{-s_{1}-s_{2}-s_{3}-\ldots-s_{n}}X$. This implies
    that \zjnote{$ 2\ \bar{\delta}^{\{0\}}\geq\inf$}. For the second
    inequality assume that $\phi:X'\to \Sigma^{-k} X$ is an
    $r$-isomorphism. We need to construct a cone-decomposition of
    $\bar{w}$ weight $\leq r$. We first assume $k\leq r$.  The first
    triangle is $T^{-1}X'\to 0\to X'$ - it is exact in $\C_{0}$ and of
    weight $0$.  The second triangle is $0\to X'\to X\to 0$ . The
    associated strict exact triangle is
    $(0\to X'\to \Sigma^{-k} X\to 0, r)$ and it uses $\phi$ in an
    obvious way to compare with the $\C_{0}$ exact triangle
    $0\to X'\to X'\to 0$. So we are left with the case $k>r$.  In this
    case, the first triangle is
    $T^{-1}X'\to 0 \to \Sigma^{k-r}X'\to X'$. Its $\bar{w}$-weight is
    null. The next triangle is $0\to \Sigma^{k-r}X'\to X\to 0$ the
    associated strict exact triangle being
    $(0\to \Sigma^{k-r}X'\to \Sigma^{-r} X \to 0, r)$ where $\phi$ is
    now used to compare with the exact triangle
    $0\to \Sigma^{k-r} X'\to \Sigma^{k-r}X'\to 0$. }
\item[(iii)] We have
  $\underline{\bar{\delta}}^{\{0\}}(X,X')=\inf\{ r \in \R \ |\
  \exists \text{ an $r$-isomorphism } \phi : \Sigma^{k} X'\to X \text{
    with } r\geq k\geq 0 \}$. This happens because the first triangle
  in the sequence (\ref{eq:iterated-tr2}) is $0\to X'\to Y_{1}\to 0$
  and the next triangles are of the form $0\to Y_{i}\to Y_{i+1}\to
  0$. Each of them has a weight estimated by the numbers $r_{i}$ for
  which there is an $r_{i}$-isomorphism
  $Y_{i}\to \Sigma^{-s_{i}}Y_{i+1}$ with $0\leq s_{i}\leq r_{i}$ which
  shows the claim.
 \item[(iv)] We have $\bar{d}^{\{0\}}(X,\Sigma^{r}X)=r$ for any
   $r\in \R$ (this follows from Examples \ref{ex-id-rigid} and
   \ref{ex-other}).
 \item[(v)] For the shift invariant metric $\widehat{d^{\{0\}}}$
   induced by $\bar{d}^{\{0\}}$ through the formula
   (\ref{eq:inv-shift-metr}) we have
   $\widehat{d^{\{0\}}}(X,\Sigma^{r}X)=0$ for all $X$ and
   $r\in \R$.
\end{itemize}

Thus $\bar{d}^{\{0\}}$ is finite for pairs of objects that are
isomorphic in $\C_{\infty}$ and $\bar{d}^{\{0\}}(X,X')$ is the
optimal upper-bound $r$ such there are $s$-isomorphisms in $\C$ with
$s\leq r$, from some positive shift of $X$ to $X'$ and, similarly,
from some positive shift of $X'$ to $X$. To some extent,
$\bar{d}^{\{0\}}$ can be viewed as an abstract analogue of the
interleaving distance in the theory of persistence modules
(cf.~\cite[Section~1.3]{PRSZ20}). We explore the relation with
interleaving in more detail in~\S\ref{subsec:interleaving-etc} (see
also Proposition~\ref{prop-1}). 
\begin{rem}\label{rem:zero-empty}
 There is another pseudo-metric $\bar{d}^{\emptyset}$ which means that the set $\mathcal F = \emptyset$. In this case, by definition, this $\bar{d}^{\emptyset}$ is again an algebraic analogue of the interleaving distance.
 We will not use this pseudo-metric later in  the paper so we do not further discuss its properties here. 
 \end{rem}
   For $\underline{\bar{d}}^{\{0\}}$
there is an additional constraint that the respective shifts
\jznote{should} be also bounded by $r$.  As a consequence:
\begin{itemize}
\item[(vi)] If $\bar{d}^{\{0\}}(X,X')=0$, then $X$ and $X'$ are
  $0$-isomorphic up to shift. Moreover, if $X$ and $X'$ are not
  $0$-isomorphic, they are both periodic in the sense that there exist
  $k$ and $k'$ (not both null) and $0$-isomorphisms
  $\Sigma^{k}X\to X$, $\Sigma^{k'}X'\to X'$.
\item[(vii)] If $\underline{\bar{d}}^{\{0\}}(X,X')=0$, then $X\equiv X'$.
\end{itemize}

In summary, this means that the best we can expect from the
fragmentation pseudo-metrics is that they \jznote{should be} non-degenerate on the
space of $0$-isomorphism types. From now on, we will say that a
fragmentation pseudo-metric is non-degenerate if this is the
case. Assuming no periodic objects exist, the metric
$\bar{d}^{\{0\}}$ is non-degenerate in this sense. However, the
distance it measures for two objects that are not isomorphic in
$\C_{\infty}$ is infinite. On the other hand, a metric such as
$\bar{d}^{\mathcal{F}}$ (as well as
$\underline{\bar{d}}^{\mathcal{F}}$) where $\mathcal{F}$ is a family
of triangular generators of $\C_{\infty}$ is finite but is in general
degenerate.

The last point we want to raise in this section is that mixing
fragmentation pseudo-metrics can sometimes produce non-degenerate
ones. We will see an example of this sort in the symplectic section \S\ref{subsec-symp}, 
but we end here by describing a more general, abstract argument. Fix
two families $\mathcal{F}_{i}$, $i=1,2$ of generators of
$\mathcal{C}_{\infty}$.  Consider the mixed pseudo-metric defined by
\begin{equation}\label{eq:mixing1}
  \bar{d}^{\mathcal{F}_{1},\mathcal{F}_{2}}=\max\{\bar{d}^{\mathcal{F}_{1}}, \bar{d}^{\mathcal{F}_{2}} \}~.~
\end{equation}
The idea is that if these two families are ``separated'' in a strong
sense, then the mixed metric is non-degenerate. For instance, denote
by $\mathcal{F}_{i}^{\Delta}$ the subcategory of $\C_{0}$ that is
generated by $\mathcal{F}_{i}$.  Now assume that
$\mathrm{Obj}(\mathcal{F}_{1}^{\Delta})\cap
\mathrm{Obj}(\mathcal{F}_{2}^{\Delta})=\{0\}$ (this is of course quite
restrictive).  We now claim that
$\underline{\bar{d}}^{\mathcal{F}_{1},\mathcal{F}_{2}}$ is
non-degenerate and that $\bar{d}^{\mathcal{F}_{1},\mathcal{F}_{2}}$
satisfies a weaker non-degeneracy condition which is that
$\bar{d}^{\mathcal{F}_{1},\mathcal{F}_{2}}(X,0)=0$ if and only if
$X\equiv 0$. This latter fact follows immediately by noticing that
$\bar{d}^{\mathcal{F}_{i}}(X,0)=0$ means that
$X\in \mathcal{F}_{i}^{\Delta}$. We leave the former as an exercise.

\subsubsection{Fragmentation pseudo-metrics, the interleaving
  pseudo-metric and other algebraic measurements}
\label{subsec:interleaving-etc} The aim of this section is to describe
relations between the fragmentation pseudo-metrics introduced before
and the interleaving distance which is well-known in persistence
theory as well as in Morse and Floer theory.  A discussion of the
bottleneck distance, which is closely related to interleaving, is
included in \S\ref{subsec:filt-co}.

\

Adapting the definition of the interleaving distance \cite{PRSZ20} to
TPCs is immediate.

\begin{dfn}\label{def:interleave} Let $\C$ be a triangulated
  persistence category with shift
  functor $\Sigma$. Given two objects $X,Y\in \mathrm{Obj}(\C)$ the
  interleaving distance between $X$ and $Y$ is defined by:
 \begin{eqnarray*}
    d_{int}(X,Y) & =\inf\{ \ r\geq 0 \ | \ \exists \
                   \phi\in\Mor_{\C_{0}}(\Sigma^{r}X,Y), \
                   \psi \in \Mor_{\C_{0}}(\Sigma^{r}Y, X), \\
                 &  \mathrm{such \ that}\   \psi\circ \Sigma^{r}\phi =
                   \eta_{2r}^{X} \ , 
                   \phi \circ \Sigma^{r}\psi = \eta_{2r}^{Y} \}.
  \end{eqnarray*}
\end{dfn}
It is a simple exercise to check that this is indeed a pseudo-metric.

We will also make use of the shift invariant version:
$$\widehat{d_{int}}(X,Y)=\inf_{r\in \R} \ d_{int}(\Sigma^{r}X,Y) $$
that we will refer to as the {\em shift invariant interleaving pseudo-metric}.

\begin{lem} \label{lem:ineq-inter} Fix the triangulated persistence
  category $\C$ and two objects $X$ and $Y$ in $\C$.  If
  $\phi: X\to Y$ is an $r$-isomorphism, then $d_{int}(X,Y)\leq r$.
  Conversely, if $d_{int}(X,Y) = s$, then for any $r>s$ there are
  $4r$-isomorphisms $\Sigma^{r} X\to Y$ and $\Sigma^{r} Y\to X$.
\end{lem}

\begin{proof} We start with the first part of the lemma.  By the
  results in \S\ref{subsubsec:main-def}, $\phi$ has a right
  $r$-inverse $\psi:\Sigma^{r}Y\to X$ such that
  $\phi\circ \psi =\eta^{Y}_{r}$. Let $\phi'=\phi\circ \eta_{r}^{X}$
  such that we have the diagram
  $$\Sigma^{2r}Y\stackrel{\Sigma^{r}\psi}{\longrightarrow}
  \Sigma^{r} X \stackrel{\phi'}{\longrightarrow} Y$$ with
 \zjr{$\phi'\circ \Sigma^{r} \psi=\eta_{2r}^{Y}$}. We now consider the composition
  $\psi\circ \Sigma^{r}\phi'$:

  $$\Sigma^{2r}X \stackrel{\eta_{r}^{\Sigma^{r}X}}{\longrightarrow}
  \Sigma^{r} X \stackrel{\Sigma^{r}\phi}{\longrightarrow}
  \Sigma^{r}Y\stackrel{\psi}{\longrightarrow}X~.~$$ Proposition
  \ref{prop-r-iso} claims that $\psi$ is $r$-equivalent to a left
  inverse of $\phi$. Thus
  $\psi\circ \Sigma^{r}\phi \simeq_{r} \eta_{r}^{X}$ and thus
  $\psi\circ \Sigma^{r}\phi'=\eta_{2r}^{X}$ which shows the claim.

  We pass to the second part of the lemma and now assume that
  $d_{int}(X,Y)= s < r$. We fix morphisms $f : \Sigma^{r}X\to Y$ and
  $g: \Sigma^{r}Y\to X$ such that
  \zjnote{$g\circ \Sigma^{r}f=\eta^X_{2r}$} and
  \zjnote{$f\circ \Sigma^{r} g=\eta^Y_{2r}$}.  \ocnote{This means that
    both $f$ and $g$ have right and left $2r$-inverses. Therefore, by
    Lemma \ref{lem:inv-rs} iv), they are both}
  \pbrrvv{$4r$-isomorphisms.}
  \end{proof}

Recall now the largest of our metrics from \S\ref{subsubsec:prop-fr},
$\bar{d}^{\{0\}}$, and its shift invariant version
$\widehat{{d}^{\{0\}}}$ - see also \S \ref{subsubsec:prop-alg} and
\S\ref{subsubsec:non-deg-v}.

\ocnote{
\begin{cor}\label{cor:int-metr} In the setting above we have:
  $$\frac{1}{2}\ \widehat{d_{int}}(X,Y) \leq \widehat{d^{\{0\}}}(X,Y)
  \leq 4\ \widehat{d_{int}}(X,Y) ~.~$$ In particular, all shift
  invariant pseudo-metrics of the type
  $\widehat{\bar{d}^{\mathcal{F}}}$ have as upper bound
  $$4  \cdot \mbox{\rm (the shift invariant interleaving pseudo-metric)}~.~$$
\end{cor} 
\begin{proof} We start with the first inequality from the left. Assume
  that $\widehat{d^{\{0\}}}(X,Y)=s$ and fix $r>s$,
  $r<s+\epsilon$. This means that there exists $m\in \R$ such that
  $\bar{d}^{\{0\}}(\Sigma^{m}X, Y) < r$.  In particular,
  $\bar{\delta}^{\{0\}}(\Sigma^{m}X,Y) < r$.  Thus, by point (ii) in
  \S\ref{subsubsec:non-deg-v}, we deduce that there exists some
  $k \geq 0$ and a $2r$-isomorphism $\Sigma^{m+k}X\to Y$. By the first
  point of Lemma \ref{lem:ineq-inter} we deduce that
  $d_{int}(\Sigma^{m+k}X, Y)\leq 2 r$ and thus
  $\widehat{d_{int}}(X,Y)\leq 2 r$ which implies the desired inequality.

  We pursue with the second inequality.  Let now
  $s=\widehat{d_{int}}(X,Y)$ and $r>s$, $r<s+\epsilon$.  There exists
  $m\in \R$ such that $d_{int}(\Sigma^{m}X,Y) < r$.  From the second
  point of the Lemma \ref{lem:ineq-inter} we deduce that there are
  $4r$-isomorphisms $\Sigma^{r+m}X \to Y$ and
  $\Sigma^{r}Y\to \Sigma^{m}X$. This means by point (ii) in
  \S\ref{subsubsec:non-deg-v} that
  $\bar{\delta}^{\{0\}}(\Sigma^{m}X,Y)\leq 4r$ and
  $\bar{\delta}^{\{0\}}(Y,\Sigma^{m}X)\leq 4r$. Thus
  $\bar{d}^{\{0\}}(\Sigma^{m}X,Y)\leq 4r$ and we conclude
  $\widehat{d^{\{0\}}}(X,Y)\leq 4\ \widehat{d_{int}}(X,Y)$.
\end{proof}
}
\subsubsection{Other algebraic measurements}\label{subsubsec:other-w}
Other algebraic pseudo-metrics based on measuring the weight of
cone-decompositions - and not necessarily individual triangles - have
appeared in \cite{Bi-Co-Sh:LagrSh}. The basic measurement introduced
there can be viewed as a sort of extension of the interleaving
distance and is easily formulated in the TPC setting (and in fact in
any persistence category). To fix ideas let $\C$ be a TPC and let
$f:X\to Y$ be a morphism in $\C$.  We define:
$$\rho(f)=\inf_{g,s} \ \{\ \max ( \ceil*{f} + \ceil*{g}, s, 0)
\ | \ \ g:Y\to X, \ g\circ f \simeq_{s} \mathds{1}_{X} \ \}.$$ The way this is
used in \cite{Bi-Co-Sh:LagrSh} is the following. Consider a triple
$(\eta, Y_{n}, \phi)$ formed of an iterated cone decomposition $\eta$
in $\C_{0}$ having as final term $Y_{n}$, as below
$$\xymatrixcolsep{1pc} \xymatrix{
  Y_{0} \ar[rr] & & Y_{1}\ar@{-->}[ldd] \ar[r] &\ldots \ar[r]& Y_{i}
  \ar[rr] & & Y_{i+1}\ar@{-->}[ldd]
  \ar[r] &\ldots \ar[r]&Y_{n-1} \ar[rr] &   &Y_{n} \ar@{-->}[ldd]  &\\
  &         \Delta_{1}                  &  & & &  \Delta_{i+1}                          & &  &  &    \Delta_{n}             \\
  & X_{1}\ar[luu] & & & &X_{i+1}\ar[luu] & & & &X_{n}\ar[luu] }$$ and
with $\phi: Y\to Y_{n}$ a morphism in $\C$ that induces an isomorphism
in $\C_{\infty}$. The weight $W$ of such a triple is defined by
$W (\eta, Y_{n}, \phi)=\rho (\phi)$. This can be used to compare
objects $X,Y$ in $\C$ relative to a family of objects $\mathcal{F}$,
which we assume to be closed to the action of $\Sigma$, by defining
for two objects $X$, $Y$:
$$Z^{\mathcal{F}}(Y,X)=\inf \{ \ W(\eta, Y_{n}, \phi) \
| \ \exists j , \ \forall \ i \not = j , \ X_{i} \in \mathcal{F},\
T^{-1}X=X_{j}, \ \phi : \Sigma^{k}Y\to Y_{n}, \ k\in \R \ \}~.~$$ Such
a $Z^{\mathcal{F}}$ can be obviously symmetrized and it is shift
invariant (because $\mathcal{F}$ is closed under the action of
$\Sigma$ and we included the parameter $k$ in the infimum). However,
the fact that the triangle inequality is \zjnote{satisfied} is non-trivial and
it is not clear whether this is true for general TPCs. As we will see
later in the paper, in Remark \ref{rem:chain}, the triangle inequality
is true in important examples such as the homotopy category of a
filtered pre-triangulated dg-category and similarly for filtered
modules over a filtered $A_{\infty}$-category (this was the case
treated in \cite{Bi-Co-Sh:LagrSh}). This \zjnote{subtlety} is related to the
degree of precision in constructing maps induced on cones.
Nonetheless, there is a simple way to compare some of our
pseudo-metrics and $Z^{\mathcal{F}}$.

\begin{lem}\label{lem:metric-D} In the setting above we have:
$$Z^{\mathcal{F}}(Y,X)\leq 4\cdot \widehat{\bar{\delta}^{\mathcal{F}}}(Y,X) ~.~ $$
\end{lem} 

\begin{proof} 
Assume that $\widehat{\bar{\delta}^{\mathcal{F}}}(Y,X)  =  r$. This means that there exists 
$k\in \R$ and a sequence of exact triangles $\eta_{\infty}$ in $\C_{\infty}$
\begin{equation}\label{eq:ex-tr-cinf}\xymatrixcolsep{1pc} \xymatrix{
 0 \ar[rr] &  &  Y_{1}\ar@{-->}[ldd]  \ar[r] &\ldots  \ar[r]& Y_{i} \ar[rr] &  &  Y_{i+1}\ar@{-->}[ldd]  \ar[r] &\ldots \ar[r]&Y_{n-1} \ar[rr] &   &Y_{n} \ar@{-->}[ldd]  &\\
 &         \Delta_{1}                  &  & & &  \Delta_{i+1}                          & &  &  &    \Delta_{n}             \\
  & X_{1}\ar[luu] &  & & &X_{i+1}\ar[luu] &  &  & &X_{n}\ar[luu] } 
\end{equation}
with $X_{i}\in \mathcal{F}$ for all indices except for $X_{j}$, in which case $X_{j}=\Sigma ^{l}T^{-1}X$ for some 
$l\in\R$. Moreover, $Y_{n}=Y$ and $\bar{w}(\Delta_{i})=r_{i}$ with $\sum_{i} r_{i} = r+\epsilon$ (for any small $\epsilon$).

The aim is to construct another sequence $\bar{\eta}$  of exact triangles, this time exact in $\C_{0}$
\begin{equation}\label{eq:ex-tr-C0}\xymatrixcolsep{1pc} \xymatrix{
 0 \ar[rr] &  &  \bar{Y}_{1}\ar@{-->}[ldd]  \ar[r] &\ldots  \ar[r]& \bar{Y}_{i} \ar[rr] &  &  \bar{Y}_{i+1}\ar@{-->}[ldd]  \ar[r] &\ldots \ar[r]&\bar{Y}_{n-1} \ar[rr] &   &\bar{Y}_{n} \ar@{-->}[ldd]  &\\
 &         \Delta_{1}'                  &  & & &  \Delta_{i+1}'                          & &  &  &    \Delta_{n}'             \\
  & \bar{X}_{1}\ar[luu] &  & & &\bar{X}_{i+1}\ar[luu] &  &  & &\bar{X}_{n}\ar[luu] } 
\end{equation}
such that for each $i$ there is a $k_{i}\in \R$ with $\bar{X}_{i}=\Sigma^{k_{i}} X_{i}$ and there is some 
$\phi:  Y\to \bar{Y}_{n}$ that is a $2(r+\epsilon)$-isomorphism. Assuming this construction is achieved, we deduce from Lemma \ref{lem:ineq-inter} that $Y$ and $\bar{Y}_{n}$ are $2(r+\epsilon)$-interleaved and thus \zjr{$\rho(\phi)\leq 4(r+\ep)$} which implies the claim.

The first step is to replace the sequence $\eta_{\infty}$ with a sequence $\eta$ of strict exact triangles  in $\C$:
\[ 
    \left\{ 
      \begin{array}{ll}
        \Delta_{1}: \, \, & X'_{1}\to 0\to Y'_{1}\to \Sigma^{-r_{1}} TX'_{1}\\
        \Delta_2: \,\, & X'_2 \to Y'_1 \to Y'_2 \to \Sigma^{-r_2} TX'_2\\
        \Delta_3: \,\, & X'_3 \to Y''_2 \to Y'_3 \to \Sigma^{-r_3} TX'_3\\
                          &\,\,\,\,\,\,\,\,\,\,\,\,\,\,\,\,\,\,\,\,\vdots\\
        \Delta_n: \,\, & X'_n \to Y'_{n-1} \to Y_{n}' \to \Sigma^{-r_n} TX'_n
      \end{array} \right.\]
where all the $(-)'$ are appropriate shifts of the corresponding nodes of the sequence $\eta_{\infty}$. Such triangles exist by the definition of the weight $\bar{w}$. We now proceed by induction: we assume that we constructed the first  $k$ of the
triangles as in $\bar{\eta}$ together with an $s_{k}=2(r_{1} + r_{2}+\ldots +r_{k})$ isomorphism $\phi_{k}: Y_{k}'\to \bar{Y}_{k}$.  
We now consider the $k+1$ strict exact triangle:
$$\xymatrix{
\Delta_{k+1}:  &  X'_{k+1} \ar[r]^{u_{k+1}}  & Y'_{k} \ar[r] & Y''_{k+1}\ar[d]^{f} \ar[r]  & TX'_{k+1}\\
 & & & Y'_{k+1} \ar[r]& \Sigma^{-r_{k+1}} TX'_{k+1}
}$$
as in Definition \ref{dfn-set} with $f$ an $r_{k+1}$ isomorphism and the top row an exact triangle in $\C_{0}$.
We consider the two exact triangles in $\C_{0}$.
 
 \begin{equation}\label{eq:diag-tr3}\xymatrix{
 X'_{k+1} \ar[r]^{u_{k+1}}\ar[d]  & Y'_{k} \ar[r]\ar[d]^{\phi_{k}} & Y''_{k+1}\ar[d]^{h} \ar[r]  & TX'_{k+1}\ar[d]\\
X'_{k+1} \ar[r]^{u'}              &    \bar{Y}_{k} \ar[r] & Y'''_{k+1} \ar[r] \ar[r] & TX'_{k+1}
}
\end{equation}
where $u'=\phi_{k}\circ u_{k+1}$ and $h$ is induced from the first square on the left. In particular, $h$ is an $s_{k}$-isomorphism. So now we consider:
$$Y'_{k+1} \stackrel{g}{\longrightarrow} \Sigma^{-r_{k+1}} Y''_{k+1} \stackrel{h}{\longrightarrow} \Sigma^{-r_{k+1}} Y'''_{k+1}$$ where $g$ is a left \zjr{$r_{k+1}$-inverse} of $f$
and we notice that $h\circ g$ is an $s_{k+1}$-isomorphism.  We will take the 
map $\phi_{k+1}$ to be the composition $\phi_{k+1}=h\circ g$ and we put $\bar{Y}_{k+1}= \Sigma^{-r_{k+1}} Y'''_{k+1}$.
We  take  the triangle
$$\xymatrix{
\Sigma^{-r_{k+1}}X'_{k+1} \ar[r]^{u'}              & \Sigma^{-r_{k+1}}   \bar{Y}_{k} \ar[r] &\Sigma^{-r_{k+1}} Y'''_{k+1} \ar[r] \ar[r] & \Sigma^{-r_{k+1}}TX'_{k+1}
}$$
 which is the \zjnote{bottom} row in (\ref{eq:diag-tr3}) shifted by $\Sigma^{-r_{k+1}}$ 
 as the $k+1$ exact triangle in the sequence $\bar{\eta}$.  Finally,
we adjust the first $k$ triangles already constructed by shifting them all down by $\Sigma^{-r_{k+1}}$. This produces a sequence
of $k+1$ triangles, each exact in $\C_{0}$, with the properties desired, \zjnote{together}
with the map $\phi_{k+1}$ and completes the induction step. 
\end{proof}
 
Possibly more useful than the actual statement of the \zjnote{Lemma \ref{lem:metric-D}} is the method of proof: we produced  a sequence $\bar{\eta}$ of exact triangles in $\C_{0}$, as  in (\ref{eq:ex-tr-C0}), and a $2r$-isomorphism $\phi : Y_{n}\to \bar{Y}_{n}$ out of the sequence of triangles in $\C_{\infty}$  in
 (\ref{eq:ex-tr-cinf})  whose sum of weights is $r$.  

Using a right inverse $\psi:\Sigma^{2r}{\bar{Y}_{n}}\to Y_{n}$ of $\phi$ we can transform the last $\C_{0}$ exact triangle into a strict exact triangle of weight $4r$. 
 The interest of this construction - and this will be used  in the applications in \S\ref{subsec-symp} -  is that we obtain in this way a method to bound 
\zjnote{$\widehat{\bar{\delta}^{\mathcal{F}}}(Y,X)$} both from below and from above by a simpler 
 quantity $Q^{\mathcal{F}}(Y,X)$ that is defined as the infimum of the sum of weights of triangles in decompositions as   in (\ref{eq:ex-tr-cinf}) but with the first $n-1$ triangles of weight $0$. Thus the weight of such a \zjnote{decomposition} equals the weight of $\Delta_{n}$. To summarize what  was discussed above we have:
 \begin{cor}\label{eq:estimate3}
\zjnote{ \begin{equation} \widehat{\bar{\delta}^{\mathcal{F}}}(Y,X) \leq Q^{\mathcal{F}}(Y,X) \leq 4\cdot \widehat{\bar{\delta}^{\mathcal{F}}}(Y,X)~.~
 \end{equation}}
 \end{cor}
 
 \begin{rem}\label{lem: sequence} 
 If in the inequality above one could avoid the factor $4$, then 
we would have a simpler description of the fragmentation pseudo-metrics discussed here by replacing sequences of strict exact triangles in $\C$ by  corresponding sequences in $\C_{0}$, followed by an $s$-isomorphism with $s$ being the sum of the weights of the initial triangles. However, this coefficient has to do with the fact that left (or right) inverses of \zjnote{$k$-isomorphisms} are, in general, only $2k$-isomorphisms, see also Remark \ref{rem:inverses}, and a factor of at least $2$ is basically unavoidable. \end{rem}

\section{Examples}\label{sec-example}

\subsection{Filtered dg-categories}\label{subsec:dg}

The key property of dg-categories, introduced in \cite{BK91} (see also
\cite{Dri04}), is that they admit natural, pre-triangulated
closures. The $0$-cohomological category of this closure is
triangulated. We will see here that there is a natural notion of
\zjnote{{\it filtered}} dg-categories.  Such a category also admits a
pre-triangulated closure, defined using filtered twisted complexes,
following closely \cite{BK91}. Its $0$-cohomological category is a
triangulated persistence category.

\subsubsection{Basic definitions} \label{subsec:basic-def} Following
\zjnote{a} standard convention we will work in a co-homological
setting and we keep all the sign conventions as in \cite{BK91}.  For
our purposes it is convenient to view a filtered cochain complex over
the field $\k$ as a triple $(X, \partial, \ell)$ consisting of a
cochain complex $(X, \partial)$ and a filtration function
$\ell: X \to \R \cup\{-\infty\}$ such that for any $a, b \in X$ and
$\lambda \in \k \backslash \{0\}$,
$\ell(\lambda a + b) \leq \max\{\ell(a), \ell(b)\}$,
$\ell(a) = -\infty$ if and only if $a = 0$, and
$\ell(\partial a) \leq \ell(a)$.  We denote
$X^{\leq r}= \{x \in X \,|\, \ell(x) \leq r\} \subset X$ the
filtration induced on $X$ by the filtration function
\zjnote{$\ell$}. Clearly, $X^{\leq r}$ is again a filtered cochain
complex. The family \zjnote{$\{X^{\leq r}\}_{r \in \R}$} determines
the function $\ell$. The \jznote{cohomology} of a filtered cochain
complex is a persistence module: $V^{r}(X)=H(X^{\leq r};\k)$ whose
structural maps $i_{r,s}$ are induced by the inclusions
$\iota_{r,s}:X^{\leq r}\hookrightarrow X^{\leq s}$, $r\leq s$. We have
omitted here the grading, as is customary. In case it needs to be
indicated we write, for instance,
$[V^{r}(X)]^{i}=H^{i}(X^{\leq r};\k)$. We denote this (graded)
persistence module by $\mathbb{V}(X)$,
\zjnote{\begin{equation}\label{eq:pers-v} \mathbb{V}(X): =
    (\{V^{r}(X)\}_{r \in \R}, \{i_{r,s}: V^r(X) \to V^s(X)\}_{r \leq s
      \in \R})~.~
\end{equation}}
Given two filtered cochain complexes $X = (X, \partial^X, \ell_X)$ and
$Y = (Y, \partial^Y, \ell_Y)$, their tensor product is a filtered
\jznote{cochain} complex $(X \otimes Y, \partial^{\otimes}, \ell_{\otimes})$
given by $(X\otimes Y)_k = \bigoplus_{i+j = k} (X_i \otimes Y_j)$ and
\begin{equation} \label{dfn-tensor-complex}
  \partial^{\otimes}(x\otimes y) = \partial^X(x) \otimes y +
  (-1)^{|x|}x \otimes \partial^Y(y) \ , \ \ell_{\otimes}(a \otimes b)
  = \ell_X(a) + \ell_Y(b) ~.~
\end{equation}
If $(X, \ell_{X})$ and $(Y,\ell_{Y})$ are filtered vector spaces, we call a  linear map $\phi : X\to Y$
$r$-filtered if $\ell_{Y}(\phi(x))\leq \ell_{X}(x)+r$ for all $x\in X$. A $0$-filtered map is sometimes
called (for brevity) filtered. For more background on this formalism,  see \cite{UZ16}.

The next definition is an obvious analogue of the notion of dg-category in \cite{BK91} \S 1.

\begin{dfn} \label{dfn-fdg-cat} A {\em filtered dg-category} is a
  preadditive category $\A$ where
  \begin{itemize}
  \item[(i)] for any $A, B \in {\rm Obj}(\A)$ the hom-set
    ${\hom}_{\A}(A,B)$ is a filtered cochain complex with
    filtrations denoted by ${\hom}_{A}^{\leq r}(A,B)$ such that for
    each identity element we have $\ell(\mathds 1_{A})=0$ and $\mathds 1_{A}$ is
    closed;
  \item[(ii)] the composition is a filtered chain map:
    $${\hom}_{\A}(B,C) \otimes {\hom}_{\A}(A,B)
    \xrightarrow{\circ} {\hom}_{\A}(A,C)~;~$$
  \item[(iii)] for any inclusions $\iota^{AB}_{r,r'}$ and
    $\iota^{BC}_{s,s'}$, the composition morphism satisfies the
    compatibility condition
    $\iota^{BC}_{s,s'}(g) \circ \iota^{AB}_{r,r'}(f) =
    \iota^{AC}_{r+s, r'+s'}(g \circ f)$ for any
    $f \in {\hom}^{\leq r}_{\A}(A,B)$ and
    $g \in {\hom}^{\leq s}_{\A}(B,C)$.
\end{itemize}
\end{dfn}

\begin{remark} \label{rmk-fil-dgc} 
  A filtered dg-category is trivially a persistence category by
  forgetting the boundary maps on each ${\hom}_{\A}(A,B)$. Explicitly, for any $A, B \in {\rm Obj}(\A)$,
  define $E_{AB}: (\R, \leq) \to {\rm Vect}_{\k}$ by
  $E_{AB}(r) = {\hom}_{\A}^{\leq r}(A,B)$ and
  $E_{AB}(i_{r,s}) = \iota_{r,s}: {\hom}_{\A}^{\leq r}(A,B)\to {\rm
    hom}_{\A}^{\leq s}(A,B)$.
\end{remark}

The (co)homology category of a filtered dg-category $\A$, denoted by
${\rm H}(\A)$, is a category with
$${\rm Obj}({\rm H}(\A)) = {\rm Obj}(\A)$$ and, for any
$A, B \in {\rm Obj}({\rm H}(\A))$,
\pbnote{
$${\hom}_{{\rm H}(\A)}(A,B) := \mathbb V({\hom}_{\A}(A,B))
= \Bigl( \bigl\{ H^*(\Hom_{\A}^{\leq r}(A,B))\bigr\}_{r \in
  \mathbb{R}}, \{i_{r,s}\}_{r \leq s}\Bigr),$$} \pbnote{is the
persistence module as described in~\eqref{eq:pers-v}.}  It is
immediate to see that for any filtered dg-category $\A$, its (co)homology
category ${\rm H}(\A)$ is a (graded) persistence category.

\subsubsection{Twisted complexes} \label{subsubsec:tw-cplxes} It is easy to construct a formal shift-completion
of a dg-category.


\begin{dfn} \label{dfn-sus} Let $\A$ be a filtered dg-category. The {\em shift completion} $\Sigma \A$ of $\A$ 
is a  filtered dg-category such that: 
\begin{itemize}
\item[(i)] The objects of $\Sigma \A$ are
\begin{equation} \label{suspend-not}
{\rm Obj}(\Sigma\A) = \left\{ \Sigma^r A[d] \, | \, A \in {\rm Obj}(\A), \, r \in \R \,\,\mbox{and}\,\, d\in \Z\right\}
\end{equation}
such that $\Sigma^0 A  = A$, $\Sigma^s (\Sigma^r A)  = \Sigma^{r+s} A$, $A[0] = A$, $(A[d_1])[d_2] = A[d_1 +d_2]$, $(\Sigma^r A)[d] = \Sigma^r(A[d])$, for any $r, s \in \R$ and $d_1, d_2, d \in \Z$.
\item[(ii)] For any $\Sigma^r A[d_A], \Sigma^s B[d_B] \in {\rm Obj}(\A)$, the hom-set ${\hom}(\Sigma^rA[d_{A}], \Sigma^s B [d_{B}])$ is a filtered \jznote{cochain} complex with the same underlying cochain complex of ${\hom}(A,B)$ but with degree shifted by $d_B -d_A$ and filtration function $\ell_{\Sigma^rA[d_A]\, \Sigma^s B[d_B]} = \ell_{AB}+s -r$.
\end{itemize}
\end{dfn}

\begin{remark} \label{rmk-sus}  It is immediate to check that $\Sigma \A$ as given in Definition \ref{dfn-sus} is still a filtered dg-category.   
\end{remark}

The category $\Sigma \A$ carries a natural functor $\Sigma: (\R, +) \to \mathcal P{\rm End}(\Sigma \A)$ defined on objects by $\Sigma^r(A) = \Sigma^r A$ and with an obvious definition on morphisms such that  $\Sigma^r$  is filtration preserving\zjr{. For any $r,s \in \R$, the} natural transformations $\eta_{r,s}:\Sigma^{r}\to \Sigma^{s}$ are such that $(\eta_{r,s})_A :\Sigma^{r}A\to \Sigma^{s} A$ is
 induced by the identity map $\mathds{1}_A$ for each $A \in {\rm Obj}(\A)$.
In this context we have a natural definition of (one-sided) twisted complexes obtained by adjusting to the filtered case  the Definition 1 in \S 4 \cite{BK91}.

\begin{dfn} \label{dfn-tw-cpx} Let $\A$ be a filtered dg-category. A filtered (one-sided) twisted complex of $\Sigma \A$ is a pair $A = \left(\bigoplus_{i=1}^n \Sigma^{r_i} A_i[d_i], q = (q_{ij})_{1 \leq i, j \leq n} \right)$ such that the following conditions hold.
\begin{itemize}
\item[(i)] $\Sigma^{r_i} A_i[d_i] \in {\rm Obj}(\Sigma \A)$, where $r_i \in \R$ and $d_i \in \Z$.
\item[(ii)] $q_{ij} \in {\hom}_{\Sigma \A}(\Sigma^{r_j}A_j[d_j], \Sigma^{r_i}A_i[d_i])$ is of degree $1$, and $q_{ij} =0$ for $i \geq j$.
\item[(iii)] $d_{\hom} q_{ij} + \sum_{k=1}^n q_{ik} \circ q_{kj} =0 $.
\item[(iv)] For any $q_{ij}$, $\ell_{\Sigma^{r_j}A_j[d_j] \Sigma^{r_i}A_i[d_i]}(q_{ij}) \leq 0$.
\end{itemize}
\end{dfn}

\begin{rem} We will mostly work with \jznote{{\em filtered} one-sided twisted complexes} as defined above but, more generally, the pair \zjr{$A = \left(\bigoplus_{i=1}^n \Sigma^{r_i} A_i[d_i], q = (q_{ij})_{1 \leq i, j \leq n} \right)$} subject only to (i),(ii), (iii) is called a one-sided twisted complex. 
\end{rem}


It is easy to see that there are at least as many filtered one-sided twisted complexes as one-sided twisted complexes as it follows from the statement below whose proof we leave to the reader.

\begin{lemma} \label{lemma-exist-fil-tc} Given a twisted complex $\left(\bigoplus_{i=1}^n A_i[d_i], q = (q_{ij})\right)$, there exist $(r_i)_{1 \leq i \leq n}$ such that condition (iv) in Definition \ref{dfn-tw-cpx} is satisfied for the filtration shifted twisted complex $\left(\bigoplus_{i=1}^n \Sigma^{r_i}A_i[d_i], q = (q_{ij})\right)$. \end{lemma}


\subsubsection{Pre-triangulated completion.}
We will see next that the filtered twisted complexes over $\A$ form a category that provides a  (pre-)triangulated
closure of $\A$. The $0$-cohomology category of this completion is a triangulated persistence category.

\begin{dfn} \label{dfn-fil-pre-tr} Given a filtered dg-category $\A$,
  define its \jznote{{\em filtered pre-triangulated completion}}, denoted by
  $Tw(\A)$, to be a category with the following properties.
  \begin{itemize}
  \item[(i)] Its objects are,
    \[ {\rm Obj}(Tw(\A)) := \{ \mbox{filtered one-sided twisted
        complex of $\Sigma\A$} \}.\]
  \item[(ii)] For $A = \left(\bigoplus \Sigma^{r_j}A_j[d_j], q\right)$
    and $A' = \left(\bigoplus \Sigma^{r'_i}A'_i[d'_i], q'\right)$ in
    ${\rm Obj}(Tw(\A))$, a morphism $f \in {\hom}_{Tw(\A)}(A, A')$
    is a matrix of morphisms in $\A$ denoted by
    $f = (f_{ij}): A \to A'$, where
    \[ f_{ij} \in {\hom}_{\Sigma\A}\left(\Sigma^{r_j}A_j[d_j],
        \Sigma^{r'_i}A'_i[d'_i]\right). \]
  \item[(iii)] The hom-differential is defined as follows. For any
    $f \in {\hom}_{Tw(\A)}(A, A')$ as in (ii) above, define
 \zjr{   \begin{equation} \label{diff-hull} d_{Tw\A}(f) := (d_{\rm
        hom}f_{ij}) + q' f - (-1)^l f q\end{equation} }where
    ${\rm deg}(f_{ij}) = l$ and the right-hand side is written
    \pbnote{in matrix form}. The composition $f' \circ f$ is given by
    the matrix multiplication.
  \end{itemize}
\end{dfn}

\begin{lemma} \label{lemma-pre-tr-fdg} Given a filtered dg-category
  $\A$, its filtered \jznote{pre-triangulated completion $Tw (\A)$} is a filtered
  dg-category.
\end{lemma}
\begin{proof} The main step is to notice that there exists a
  filtration function on \jznote{${\hom}_{Tw(\A)}(A, A')$} for any
  $A, A' \in {\rm Obj}(Tw(\A))$. For any
  $f = (f_{ij}) \in {\hom}_{Tw(\A)}(A, A')$,
  set
  \begin{equation} \label{fil-fun-pre-tr} \ell_{AA'}(f) = \max_{i,j}
    \left\{\ell_{\Sigma^{r_j}A_j[d_j]
        \,\Sigma^{r'_i}A'_i[d'_i]}(f_{ij})\right\}.
  \end{equation}
  It is easily checked that $\ell_{AA'}$ is a filtration function as
  well as the other required properties.
\end{proof}

The first step towards triangulation is to define an appropriate cone
of a morphism.

\begin{dfn} \label{dfn-fmc} Let $\A$ be a filtered dg-category and
  $Tw(\A)$ be \jznote{its pre-triangulated completion}. Let
  $A = \left(\bigoplus \Sigma^{r_j} A_j[d_j], q = (q_{ij})_{1 \leq i,
      j \leq n} \right)$,
  $A' = \left(\bigoplus \Sigma^{r'_i} A'_i[d'_i], q = (q'_{ij})_{1
      \leq i, j \leq m} \right)$ be two objects of $Tw(\A)$ and let
  $f: A \to A'$ be a closed, degree preserving, morphism. Define the
  {\em $\lambda$-filtered mapping cone} of $f$, where
  $\lambda \geq \ell_{AA'}(f)$, by
\zjr{  \begin{equation} \label{const-cone} {\rm Cone}^{\lambda}(f) : =
    \left( \bigoplus_{i} \Sigma^{r'_i} A'_i[d'_i] \oplus \bigoplus_{j}
      \Sigma^{r_j + \lambda} A_j[d_j+1], q_{\rm co}
    \right)\,\,\mbox{where} \,\, q_{\rm co} = \begin{pmatrix}
      q' & f  \\
      0 & -q\end{pmatrix},
  \end{equation}}
  where $q', q,f$ are all block matrices.
\end{dfn}

\begin{remark}\label{thm-pre-tr-cone} (1) The condition
  $\lambda \geq \ell_{AA'}(f)$ guarantees that
  ${\rm Cone}^{\lambda}(f)$ is indeed a filtered one-sided twisted
  complex over $\Sigma \A$. Therefore, $Tw(\A)$ is closed under taking
  degree-shifts, filtration-shifts, and filtered mapping cones of
  (degree preserving) closed morphisms.

  (2) Notice that a $\lambda$-filtered cone can also be written as a
  $0$-filtered cone but for a different map.

  (3) Given a filtered dg-category $\A$ it is easy to see that every
  object in $Tw(\A)$ can be obtained from objects in $\Sigma \A$ by
  taking iterated filtered $0$-filtered mapping cones.
\end{remark}

The $0$-cohomological category associated to a dg-category is a
triangulated category. The next result is the analogue in the filtered
case.

\begin{prop} \label{thm-hol-hull} If $\A$ is a filtered dg-category
  and $Tw(\A)$ is its \jznote{filtered pre-triangulated completion}, then the
  degree-$0$ cohomology category \zjnote{${H}^0(Tw(\A))$} is a triangulated
  persistence category.
\end{prop}

In view of this result, it is natural to call a filtered dg-category
$\A$ {\em pre-triangulated} if the inclusion
$\A \hookrightarrow Tw(\A)$ is an equivalence of filtered
dg-categories.
\begin{cor}\label{cor:pre-tr} Let $\A$ be a filtered pre-triangulated
  dg-category. Then its degree-0 cohomology category ${\rm H}^0(\A)$
  is a triangulated persistence category.
\end{cor}

\begin{proof}[Proof of Proposition \ref{thm-hol-hull}]
  It is trivial to notice that the category $H^{0}(Tw(\A))$ is a
  persistence category. It is endowed with an obvious shift functor as
  defined in \S\ref{subsubsec:tw-cplxes}. The first thing to check at
  this point is that the $0$-level category $[H^{0}(Tw(\A))]_{0}$ with
  the same objects as $H^{0}(Tw(\A))$ and only with the shift
  $0$-morphisms is triangulated - see Definition \ref{dfn-tpc}. The
  family of triangles that will provide the exact ones are the
  triangles of the form
  \zjr{$$A\xrightarrow{f}B\xrightarrow{i}
  \mathrm{Cone}^{0}(f)\xrightarrow{\pi}A[1]$$} associated to the
  $0$-cones, as given in Definition \ref{dfn-fmc}. From this point on
  checking that $[H^{0}(Tw(\A))]_{0}$ is triangulated comes down to
  the usual verifications showing that the $H^{0}$ of a dg-category is
  triangulated, with a bit of care to make sure that the relevant
  homotopies preserve filtration. We leave this verification to the
  reader. It is then automatic that $\Sigma^{r}$ is triangulated when
  restricted to $[H^{0}(Tw(\A))]_{0}$. The last step is to show that
  the morphism $\eta_{r}^{A}:\Sigma^{r}A\to A$ has an $r$-acyclic
  \zjnote{cone} in $Tw(\A)$. In this context, of filtered dg-categories, an
  object $K$ is $r$-acyclic if the identity
  $\mathds{1}_{K}\in \Hom_{Tw(\A)}(K,K)$ is a boundary of some element
  $\eta\in \Hom_{Tw(\A)}^{\leq r}(K,K)$.

  The map $\eta^A_r \in {\rm Mor}^0_{Tw(\A)}(\Sigma^r A, A)$ is
  induced by the identity.  By definition
 \zjr{ ${\rm Cone}^0(\eta^A_r) = A \oplus \Sigma^r A[1]$} and
 \zjr{ $$ q_{\rm co} = \begin{pmatrix}
    q & \eta^{A}_{r}  \\
    0 & -q'\end{pmatrix}
  $$}
  where $q$ is the structural map of the twisted complex $A$ and
  $q'=\Sigma^{r}q$. Consider a homotopy
 \zjr{ \[ K = \begin{pmatrix} 
      0  & 0 \\
      (\eta_{0,r})_A&0
    \end{pmatrix}: {\rm Cone}^0(\eta^A_r) \to {\rm
      Cone}^0(\eta^A_r)[1]. \]} Note that $\ell(K) =r$. We have
\zjr{\jznote{      \begin{align*} 
dK & = \begin{pmatrix} 
      0 & 0   \\
      0 & 0
    \end{pmatrix} + \begin{pmatrix}
      q & \eta^A_r  \\
      0 & -q'
    \end{pmatrix} \begin{pmatrix} 
      0  & 0 \\
      (\eta_{0,r})_A & 0 
    \end{pmatrix} + \begin{pmatrix} 
      0  & 0 \\
      (\eta_{0,r})_A & 0 
    \end{pmatrix}  \begin{pmatrix} 
      q & \eta^A_r  \\
      0 & -q'
    \end{pmatrix} \\ &= \begin{pmatrix}
      \mathds{1}_A & 0   \\
      -q'\circ (\eta_{0,r})_A+(\eta_{0,r})_A\circ q & \mathds{1}_{\Sigma^r A[1]})
    \end{pmatrix} =\mathds{1}_{{\rm Cone}^0(\eta^A_r)}
  \end{align*}}}
  because $\Sigma^{r}q\circ (\eta_{0,r})_A= (\eta_{0,r})_A\circ q$ and this
  concludes the proof.
\end{proof}

\begin{rem}\label{rem:chain}
 (a) In the filtered dg-category $Tw(\A)$ we can
  replicate all the constructions in \S\ref{subsec:TPC} at the chain
  level, similarly to the definition of $r$-acyclic objects mentioned
  inside the proof above. For instance,  $r$-isomorphisms are replaced
  by $r$-quasi-isomorphisms (meaning filtration preserving morphisms that
  induce an $r$-isomorphism in homology), and all the functorial type constructions of that section can be pursued at the chain level, by replacing
 commutativity at the chain level by commutativity up to homotopy.
  
 (b) One advantage of working at the chain level instead of the general setting of triangulated persistence categories is that the maps $c$ induced on cones through diagrams of the following type:
  $$\xymatrix{
  A \ar[r]^{f}\ar[d]^{a} &B\ar[r]\ar[d]^{b} & \mathrm{Cone\/}(f) \ar[d]^{c}\\
  A' \ar[r]^{f'} & B'\ar[r] & \mathrm{Cone\/}(f') 
  }$$
  are defined \zjnote{explicitly} in terms of the homotopy making the square on the left commutative. An example \zjnote{relevant} for this paper is that in the homological category of a filtered dg-category the measurement $Z^{\mathcal{F}}(- , -)$ from \S\ref{subsubsec:other-w} satisfies the triangle inequality.   The proof follows closely the arguments
  in Lemma 6.4.4 in \cite{Bi-Co-Sh:LagrSh} with all weakly filtered maps there being replaced with filtered ones here.
\end{rem}


\subsection{Filtered cochain complexes.}\label{subsec:filt-co}

In this section we discuss the main example of a filtered dg-category,
the category of filtered co-chain complexes. As we shall see, this is
pre-triangulated and thus, in view of Corollary \ref{cor:pre-tr}, its
homotopy category is a triangulated persistence category.

We will work over a field $\k$ and will denote the resulting category
by $\mathcal{FK}_{\k}$.  The objects of this category are filtered
cochain complexes $(X,\partial, \ell)$ where $(X,\partial)$ is a
cochain complex and $\ell$ is a filtration function, as in
\S\ref{subsec:basic-def}. Given two filtered \jznote{cochain} complexes
$(X,\partial_{X}, \ell_{X})$ and $(Y,\partial_{Y},\ell_{Y})$ the
morphisms $\Mor_{\mathcal{FK}_{\k}}(X,Y)$ are linear graded maps
$f:X\to Y$ such that the quantity

\begin{equation}\label{eq:filtr-hom}
  \ell(f)= \inf\{ r \in \R \ | \ \ell_{Y}(f(x))
  \leq \ell_{X}(x)+r, \forall x \in X\}
\end{equation}
is finite. The filtration function on $\Mor_{\mathcal{FK}_{k}}(X,Y)$
is then defined through (\ref{eq:filtr-hom}).  The differential on
$\Mor_{\mathcal{FK}_{k}}(X,Y)$ is given, as usual, by
$\partial (f)= \partial _{Y}\circ f - (-1)^{|f|}f\circ\partial_{X}$
and it obviously preserves filtrations.  The composition of morphisms
is also obviously compatible with the filtration and therefore
$\mathcal{FK}_{\k}$ is a filtered dg-category.

There is a natural shift functor on $\mathcal{FK}_{\k}$ defined by
$\Sigma: (\R, +) \to \mathcal P{\rm End}(\mathcal {FK}_k)$ by
\[ \Sigma^r(X, \partial, \ell_{X}) = (X, \partial, \ell_{X}+r),
  \,\,\,\,\, \mbox{and}\,\,\,\,\, \Sigma^r(f) = f \] for any
$f \in {\rm Mor}_{\mathcal {FK}_k}(X,Y)$. Moreover, for $r, s\in \R$,
there is a natural transformation from $\Sigma^r$ to $\Sigma^s$
induced by the identity.

Assume that
$f: (X,\partial_{X}, \ell_{X})\longrightarrow
(Y,\partial_{Y},\ell_{Y})$ is a cochain morphism such that
$\ell(f)\leq 0$. In this case, the usual cone construction
\zjr{$\mathrm{Cone}(f)=(Y\oplus X[1], \partial_{\mathrm{co}})$} with
\zjr{$$\partial_{\mathrm{co}} = \begin{pmatrix} 
  \partial_{Y} & f  \\
  0 & -\partial_{X}\end{pmatrix}$$} produces a filtered complex and fits
into a triangle of maps with $\ell \leq 0$:
\zjr{$$X \xrightarrow{f} Y\xrightarrow{i}
\mathrm{Cone}(f)\xrightarrow{\pi} X[1]~.~$$}
The standard properties of this construction immediately imply that
the dg-category $\mathcal{FK}_{\k}$ is pre-triangulated and thus the
$0$-cohomological category, $H^{0}\mathcal{FK}_{\k}$, is a
\pbnote{triangulated} persistence category.

It is useful to make explicit some of the properties of this category:
\begin{itemize}
\item[(i)] The objects of $H^{0}\mathcal{FK}_{\k}$ are filtered cochain
  complexes $(X,\partial_{X},\ell_{X})$.
\item[(ii)] The morphisms in $\Mor_{H^{0}\mathcal{FK}_{\k}}^{r}(X,Y)$
  are cochain maps
  $f:(X,\partial_{X},\ell_{X})\to (Y,\partial_{Y},\ell_{Y})$ such that
  $\ell(f)\leq r$ up to chain homotopy $h:f\simeq f'$ with
  $\ell(h)\leq r$.
\item[(iii)] A filtered complex $(K,\partial_{K},\ell_{K})$ is
  $r$-acyclic if the identity $\mathds{1}_{K}$ is chain homotopic to $0$
  through a chain homotopy $h:\mathds{1}_{K}\simeq 0$ with $\ell(h)\leq r$.
\item[(iv)] The construction of weighted exact triangles as well as
  their properties can be pursued in this context by following closely
  the scheme in \S\ref{subsubsec:weight-tr}.
\item[(v)] The limit category $[H^{0}\mathcal{FK}_{\k}]_{\infty}$ has
  as morphisms chain homotopy classes of cochain maps \pbnote{(where
    both the cochain maps and the homotopies are assumed to be of
    bounded shifts)}.  Its objects are still filtered cochain
  complexes. It is triangulated, with translation functor \zjr{$TX= X[1]$},
  as expected.
\end{itemize}

\begin{remark} \label{rem:extensions} The example of the dg-category
  $\mathcal{FK}_{\k}$ can be extended in a number of ways and we
  mention a couple of them here.

  (a) Assume that we fix a filtered dg-category $\mathcal{A}$. There
  is a natural notion of filtered (left/right) module $\mathcal{M}$
  over $\mathcal{A}$. Such modules together with filtered maps
  relating them form a new filtered dg-category denoted by
  $\mathrm{Mod}_{\mathcal{A}}$. The $0$-cohomology category associated
  to this filtered dg-category, $H^{0}\mathrm{Mod}_{\mathcal{A}}$, is
  pre-triangulated because the category $\mathrm{Mod}_{\mathcal{A}}$
  is naturally endowed with a shift functor, just like
  $\mathcal{FK}_{\k}$, as well as with an appropriate
  cone-construction over filtered, closed, degree preserving
  morphisms.

  (b) Similarly to (a), we may take $\mathcal{A}$ to be a filtered
  $A_{\infty}$-category and consider the category of filtered modules,
  $\mathrm{Mod}_{\mathcal{A}}$, over $\mathcal{A}$. Again this is a
  filtered dg-category and it is pre-triangulated (the formalism
  required to establish this fact appears in \cite{Bi-Co-Sh:LagrSh},
  in a version dealing with weakly filtered structures).
\end{remark}

As mentioned in the beginning of Introduction \S \ref{sec-intro},
there exists a quantitative comparison between two filtered cochain
complexes $X, Y$, called the bottleneck distance and denoted by
$d_{\rm bot}(X,Y)$. This is best expressed in the barcode language
from~\cite{Bar94} or~\cite{UZ16}. 

\

For completeness we specify the version of barcodes used here. 
A barcode $\mathcal{B} = \{(I_j, m_j)\}_{j \in \mathcal{J}}$ is a collection of
pairs consisting of intervals $I_j \subset \mathbb{R}$ and positive
integers $m_j \in \mathbb{Z}_{>0}$, indexed by a set $\mathcal{J}$,
and satisfying the following {\em admissibility} conditions:
\begin{enumerate}
\item[-] $\mathcal{J}$ is assumed to be either finite or
  $\mathcal{J} = \mathbb{Z}_{\geq 0}$.
\item[-] Each interval $I_j$ is of the type $I_j = [a_j, b_j)$, with
  $-\infty < a_j < b_j \leq \infty$.
\item[-] In case $\mathcal{J} = \mathbb{Z}_{\geq 0}$ we assume that
  $a_j \longrightarrow \infty$ as $j \longrightarrow \infty$.
\end{enumerate}
The intervals $I_j$ are called bars and for each $j$, $m_j$ is called
the multiplicity of the bar $I_j$.
To such a barcode one can associate a persistence module $V(\mathcal{B})$ that satisfies the following conditions:
\begin{itemize}
\item[-](lower semi-continuity) For any $s\in \R$ and any $t\geq s$
  sufficiently close to $s$, the map $i_{s,t}:M^{s}\to M^{t}$ is an
  isomorphism.
\item[-](lower bounded) For $s$ sufficiently small we have
  $M^{s}=0$.
\item[-] (tame) For every $s \in \mathbb{R}$, 
  \begin{equation}\label{eq:finite-dim}
    \dim_{\k} (M^s) < \infty ~.~
  \end{equation}
\end{itemize}
The module $V(\mathcal{B})$ is defined as the direct sum of the elementary 
persistence modules $V(I)$ for each bar $I$ in the barcode $\mathcal{B}$. Here
$V([a,b))^{s}= \k$ if $s\in [a,b)$ and $V([a,b))^{s}=0$ if $s\not\in [a,b)$.
Conversely, the \zjnote{Normal Form Theorem in Section 2.1 in \cite{PRSZ20} or the main result in \cite{C-B15} says} that any 
persistence module $M$ with the three properties above \zjnote{can be decomposed as} a direct sum of persistence modules of the form $V([a, b))$ and $V([c,\infty))$ \zjnote{in a unique way, up to permutation. Thus} we can associate to it a barcode $\mathcal{B}(M)$ \zjnote{that consists of intervals $[a,b)$ and $[c, \infty)$ appearing in the decomposition.}

\

The homology $H(X)$ of a filtered cochain complex $X$ is a persistence
module whose barcode can be read out of the normal form of $X$.  More
precisely, by Proposition 7.4 in~\cite{UZ16} \pbrev{(see
  also~\cite{Bar94})} there is a filtered isomorphism \pbrev{(in the
  category $H^{0}\mathcal{FK}_{\k}$)} as follows:
\zjr{\begin{equation} \label{decomp} X \simeq \bigoplus_{[a,+\infty)
      \in \mathcal B(X)} E_{1}(a) \oplus \bigoplus_{[c,d) \in \mathcal
      B(X)} E_{2}(c,d)
\end{equation} }
where $E_{1}(a), E_{2}(c,d) \in {\rm Obj}(\mathcal{FK}_{\k})$ are
filtered \jznote{cochain} complexes defined by
\[ E_{1}(a) = ((\cdots \to 0 \to \k\left<x\right> \to 0 \to \cdots),
  \ell(x) = a). \] and \pbnote{
  \[ E_{2}(c,d) = ((\cdots \to 0 \to \k\left<y \right>
    \xrightarrow{\partial} \k\left<x \right> \to 0 \to \cdots), \;
    \ell(y) =c, \ell(x) = d ),\]} where $c \geq d$,
\pbnote{$\partial(y) = \kappa x$ for some $0 \neq \kappa \in \k$}. The
notation \zjr{$\mathcal B(X)$} in~\eqref{decomp} stands for a
collection of intervals of two types: finite or semi-infinite
intervals or $\R$ of the form $[c,d)$ with $c<d$, possibly with
$d=+\infty$; intervals of $0$ length, $[c,d]$, with $c=d$.
  
 
 \
  
 In
what follows, sometimes for brevity, denote by $E_*(I)$ either
$E_{1}(a)$ or $E_{2}(c,d)$ for the corresponding interval
$I = [a, + \infty)$ or $I =[c,d)$ in \zjr{$\mathcal B(X)$}. 
Then $d_{\rm bot}(X, Y)$ is defined as the
infimum $\tau$ satisfying the following conditions: there exist some
subsets consisting of certain \zjr{ ``short intervals''
$\mathcal B(X)_{\rm short} \subset \mathcal B(X)$ and
$\mathcal B(Y)_{\rm short} \subset \mathcal B(Y)$} such that
\begin{itemize}
\item[(i)] each short interval $[c,d)$ satisfies $2(d-c) \leq \tau$;
\item[(ii)] there is a bijection
 \zjr{ $\sigma: \mathcal B(X) \backslash \mathcal B(X)_{\rm short} \to
  \mathcal B(Y) \backslash \mathcal B(Y)_{\rm short}$};
\item[(iii)] if $\sigma([a,b)) = [c,d)$, then $\max\{|a-c|, |b-d|\} \leq \tau$;
\item[(iv)] if $\sigma([a, \infty)) = [c,\infty)$, then $|a-c| \leq \tau$. 
\end{itemize}

\

In what follows, we assume that the cardinalities of barcodes
$\#|\mathcal B(X)|$ and $\#|\mathcal B(Y)|$ are both finite. The
following result compares the fragmentation pseudo-metric $d^{\F}$
defined in Definition \ref{dfn-frag-met} with the bottleneck distance
$d_{\rm bot}$ defined above. 
\begin{prop} \label{prop-1} Let $\C = H^0\mathcal{FK}_{\k}$ and
  $\F \subset {\rm Obj}(\C)$ be a subset containing $0$. Then
  \[ d^{\F}(X,Y) \leq C_{X,Y} d_{\rm bot}(X,Y),\] where
  $C_{X,Y} = 4\{\#|\mathcal B(X)|, \#|\mathcal B(Y)|\} +1$.
\end{prop}

\begin{proof} It is immediate to see that we may assume that both
  $\mathcal{B}'(X)$ and $\mathcal{B}'(Y)$ do not contain any
  $0$-length bars and thus $\mathcal{B}'(X)=\mathcal{B}(X)$ and the
  same for $Y$.  It suffices to prove the conclusion when
  $\mathcal B(X)$ and $\mathcal B(Y)$ have the same cardinality of the
  infinite-length bars (otherwise by definition
  $d_{\rm bot}(X,Y) = +\infty$ and the conclusion holds
  trivially). Let $\tau : = d_{\rm bot}(X,Y) + \ep$ for an arbitrarily
  small $\ep>0$. Since $d^{\F}(\cdot, \cdot)$ is invariant under
  \pbnote{filtered isomorphisms (applied to either of its two inputs)
    then} by~\eqref{decomp} and by reordering summands we obtain:
  \begin{align*}
    d^{\F}(X,Y) & \leq  d^{\F}
                  \left(
                  \bigoplus_{I \in \mathcal B(X) \backslash \mathcal B(X)_{\rm short}}
                  E_*(I),
                  \bigoplus_{\sigma(I) \in \mathcal B(Y) \backslash \mathcal
                  B(Y)_{\rm short}} E_*(\sigma(I))\right)\\
                & + d^{\F} \left(\bigoplus_{J \in \mathcal B(X)_{\rm short}} E_2(J),
                  \bigoplus_{J' \in \mathcal B(Y)_{\rm short}} E_2(J')\right),
  \end{align*}
  where the inequality is given by the triangle inequality of $d^{\F}$
  with respect to the direct sum, see Proposition
  \ref{prop-frag-sum}. For $d^{\F}$ with short intervals, both
  $\bigoplus_{J \in \mathcal B(X)_{\rm short}} E_2(J)$ and
  $\bigoplus_{J' \in \mathcal B(Y)_{\rm short}} E_2(J')$ are acyclic
  objects in $\C$, therefore by (i) in the definition $d_{\rm bot}$
  above, triangles
  \[ 0 \to 0 \to \bigoplus_{J \in \mathcal B(X)_{\rm short}} E_2(J)
    \to 0 \,\,\,\,\mbox{and}\,\,\,\, 0 \to 0 \to \bigoplus_{J' \in
      \mathcal B(Y)_{\rm short}} E_2(J') \to 0\] are
  weight-$\frac{\tau}{2}$ exact triangles (here we identify
  $\Sigma^{\lambda} 0$ with $0$ for any shift $\lambda \in \R$). Thus,
  \begin{align*} \label{est-2}
    d^{\F} \left(\bigoplus_{J \in \mathcal B(X)_{\rm short}} E_2(J),
    \bigoplus_{J' \in \mathcal B(X)_{\rm short}} E_2(J')\right)
    &  \leq d^{\F}\left(\bigoplus_{J \in \mathcal B(X)_{\rm short}} E_2(J),
      0\right) \\
    & + d^{\F}\left(0, \bigoplus_{J' \in \mathcal B(X)_{\rm short}}
      E_2(J')\right) \leq \tau.
  \end{align*}

  On the other hand, by Proposition \ref{prop-frag-sum} again, for
  $d^{\F}$ with non-short intervals, we have
  \pbrev{\[ d^{\F}\left(\bigoplus_{I \in \mathcal B(X) \backslash
          \mathcal B(X)_{\rm short}} E_*(I), \bigoplus_{\sigma(I) \in
          \mathcal B(Y)} E_*(\sigma(I)) \right) \leq \sum_{I \in
        \mathcal B(X) \backslash \mathcal B(X)_{\rm short}}
      d^{\F}\bigl(E_*(I), E_*(\sigma(I))\bigr). \]} Since
  $d_{\rm bot}(X,Y)< +\infty$, the bijection $\sigma$ will always map
  a finite interval to a finite interval, a semi-infinite interval to
  a semi-infinite interval, so it suffice to consider the following
  two cases.

  \medskip

  \noindent {Case I}. \emph{Estimate $d^{\F}(E_1(a), E_1(c))$}. We
  need to build a desired cone decomposition. Without loss of
  generality, assume $a\geq c$. Then the identity map
  $\left<x \right>_{E_1(a)} \to \left<x \right>_{E_1(c)}$ (with {\it
    negative} filtration shift) implies that the triangle
  \zjr{$E_1(a)\to E_1(c)\to K \to E_1(a)[1]$} is weight-0 exact triangle (in
  fact in $\C_0$) where $K$ is the filtered mapping cone. Then in the
  following cone decomposition (with linearization $(0, E_1(c))$),
  \zjr{\[ 
    \left\{ \begin{array}{l} 0 \to 0 \to K \to 0 \\ E_1(c) \to K \to
        E_1(a) \to E_1(c)[1] \end{array} \right.\]} the first triangle is
  a weight-$(c-a)$ exact triangle since it is readily to verify that
  $K$ is $(c-a)$-acyclic. Then
  $\delta^{\F}(E_1(a), E_1(c)) \leq (c-a) + 0 \leq \tau$ by (iv) in
  the definition $d_{\rm bot}$ above. On the other hand, consider the
  following cone decomposition with linearization $(0, E_1(a), 0)$
  (note that by definition $\Sigma^{-(a-c)} E_1(a) = E_1(c)$),
\zjr{  \begin{equation} \label{1} \left\{
      \begin{array}{l} 0 \to 0 \to 0 \to
        0 \\ E_1(a) \to 0 \to \Sigma^{-(a-c)} E_1(a) \to \Sigma^{-(a-c)}
        E_1(a)[1] \\ 0 \to E_1(c) \to E_1(c) \to 0
      \end{array} \right.
  \end{equation}}
  where the second triangle has weight $a-c >0$ by Remark
  \ref{rmk-shift-notation} (b). Therefore,
  $\delta^{\F}(E_1(c), E_1(a)) \leq 0 + (a-c) + 0 \leq \tau$, which
  implies that
  \begin{equation} \label{est-1-1}
    d^{\F}(E_1(a), E_1(c)) \leq \tau. 
  \end{equation}
  
  \medskip

  \noindent {Case II}. \emph{Estimate $d^{\F}(E_2(a,b),
    E_2(c,d))$}. We will carry on the estimation as follows,
  \[ d^{\F}(E_2(a,b), E_2(c,d)) \leq d^{\F}(E_2(a,b), E_2(c,b)) +
    d^{\F}(E_2(c,b), E_2(d,b)). \] Moreover, we will only estimate
  $d^{\F}(E_2(a,b), E_2(c,b))$ with $a \geq c$, and other situations
  can be done in a similar and symmetric way. Similarly to Case I
  above, consider the following cone decomposition
\zjr{  \[ 
    \left\{ \begin{array}{l} 0 \to 0 \to K \to 0 \\ E_2(c,b) \to K \to
        E_2(a,b) \to E_2(c,b)[1] \end{array} \right.\]} where
  $E_2(a,b) \to E_2(c,b)$ is the identity map
  $\left<x\right>_{E_2(a,b)} \to \left<x\right>_{E_2(c,b)}$ (and
  similarly to the generator $y$) with a negative filtration shift and
  $K$ is the cone. Since $K$ is $(a-c)$-acyclic, we have
  $\delta^{\F}(E_2(a,b), E_2(c,b)) \leq 0 + (a-c) \leq \tau$. On the
  other hand,
  \[ \delta^{\F}(E_2(c,b), E_2(a,b)) \leq \delta^{\F}(E_2(c,b),
    \Sigma^{a-c} E_2(c,b)) + \delta^{\F}(E_2(a, b+a-c), E_2(a,b)), \]
  where $\delta^{\F}(E_2(c,b), \Sigma^{a-c} E_2(c,b)) \leq a-c$ by a
  similar cone decomposition as in (\ref{1}). Meanwhile, since
  $b + a-c \geq b$, the identity map from $E_2(a,b+a-c)$ to $E_2(a,b)$
  (with negative filtration shift) yields
  $\delta^{\F}(E_2(a, b+a-c), E_2(a,b)) \leq a-c$. Therefore, together
  we have, by (iii) in the definition $d_{\rm bot}$ above,
  \[ d^{\F}(E_2(a,b), E_2(c,b)) \leq 2 (a-c) \leq 2\tau, \] which
  implies
  \begin{equation} \label{est-1-2} d^{\F}(E_2(a,b), E_2(c,d)) \leq
    4\tau.
  \end{equation}

  Therefore, by (\ref{est-1-1}) and (\ref{est-1-2}) together, we have
  \begin{align*}
    d^{\F}(X,Y) & \leq \#|\mathcal B(X)\backslash \mathcal B(X)_{\rm short}| \cdot 4\tau + \tau\\
                & \leq (4 \#|\mathcal B(X)\backslash \mathcal B(X)_{\rm short}|+1) (d_{\rm bot}(X,Y) + \ep) \\
                & \leq (4 \min\{\#|\mathcal B(X)|, \#|\mathcal B(Y)|\} +1) (d_{\rm bot}(X,Y) + \ep)
  \end{align*}
  where the last inequality holds since $\sigma$ is a bijection by
  (ii) in the definition $d_{\rm bot}$ above.  Let $\ep \to 0$, and we
  complete the proof.
\end{proof}

\subsection{Topological \jznote{spaces} $+$}\label{subsec:top-sp}

There are many topological categories, consisting of topological
spaces endowed with additional structures (indicated by the $+$ in the
title of the subsection), that can be analyzed with the tools
discussed before. We will discuss here two elementary examples. They
both fit the following scheme: we will have a triple consisting of a
(small) category $\mathcal{K}$, an endofunctor
$T_{\mathcal K}:\mathcal{K}\to \mathcal{K}$ and a class of triangles
$\Delta_{\mathcal{K}}$, in $\mathcal{K}$ of the form
$$A\to B \to C \to
T_{\mathcal{K}}A~.~$$ 
In these cases the objects of $\mathcal{K}$ have an underlying
structure as topological spaces and, similarly, the morphisms in
$\mathcal{K}$ are continuous maps, the functor $T_{\mathcal{K}}$
corresponds to the suspension of spaces.

The aim is to define fragmentation pseudo-metrics on the objects of
$\mathcal{K}$ by first associating a weight with some reasonable
properties to the triangles in $\Delta_{\mathcal{K}}$,
$\bar{w}_{\mathcal{K}}: \Delta_{\mathcal{K}}\to \R$, and then defining
quantities $\delta^{\mathcal{F}}(X,Y)$ and
$\underline{\delta}^{\mathcal{F}}(X,Y)$ as in, respectively,
(\ref{frag-met-0}) and (\ref{eq:frag-simpl}), only taking into account
decompositions appealing to triangles
$\Delta_{i}\in \Delta_{\mathcal{K}}$. Notice that
$\delta^{\mathcal{F}}$ is not generally defined in this setting as its
definition requires to desuspend spaces. On the other hand, as soon as
$\bar{w}_{\mathcal{K}}$ is given, $\underline{\delta} ^{\mathcal{F}}$  \ocnote{can be defined by  formula (\ref{eq:semi-metrics1})  with $w_{\mathcal{K}}$ replacing
$w_{\infty}$ there, and with each triangle in the sequence (\ref{eq:iterated-tr2}) being replaced with a triangle in $\Delta_{\mathcal{K}}$. We assume  that the family 
$\mathcal{F}$  is such that \zjr{$0 \in \mathcal{F}$} and in most cases we assume
implicitly that $\mathcal{F}$ consists of all the objects $F$ such that there are triangles in $\Delta_{\mathcal{K}}$ of the form $F\to A\to B\to T_{\mathcal{K}}F$.  The resulting $\underline{\delta}^{\mathcal{F}}$ trivially satisfies the triangle inequality. The pseudo-metric $\underline{d}^{\mathcal{F}}$
obtained by the symmetrization of $\underline{\delta} ^{\mathcal{F}}$ exists in this case too (see Remark \ref{rem:finite-metr} (c)).}
Based on the various constructions discussed earlier in the paper,
there are two approaches to define a weight $\bar{w}_{\mathcal{K}}$
(that is not flat) and they both require some more structure:
\begin{itemize}
\item[A.] The additional structure in this case is a functor
  $\Phi :\mathcal{K}\to \C_{\infty}$ where $\C$ is a TPC, in the
  examples below $\C=H^{0}\mathcal{FK}_{\k}$ - the triangulated
  persistence \pbnote{homotopy category of \jznote{filtered cochain
    complexes}}. We also require that $\Phi$ commutes with $T$ (at
  least up to some natural equivalence) and that for each
  $\Delta\in \Delta_{\mathcal{K}}$ the image $\Phi(\Delta)$ of
  $\Delta$, is exact in $\mathcal{C}_{\infty}$ (and thus
  $\bar{w}(\Phi(\Delta))<\infty$ where $\bar{w}$ is the persistence
  weight introduced in Definition \ref{dfn-extri-inf}).  In this case
  for each $\Delta\in \Delta_{\mathcal{K}}$ we put
$$\bar{w}_{\mathcal{K}}(\Delta)=\bar{w}(\Phi(\Delta))~.~$$  
\item[B.] This second approach requires first that the morphisms
  $\Mor_{\mathcal{K}} (A,B)$ are endowed with a natural increasing
  filtration compatible with the composition. Secondly, there should
  be a shift functor
  $\Sigma_{\mathcal{K}}: (\R,+)\to {\rm End}(\mathcal{K})$ compatible with
  the filtration on morphisms and that commutes with
  $T_{\mathcal{K}}$. Moreover, the triangles in $\Delta_{\mathcal{K}}$
  have to be part of a richer structure such as a model category or a
  Waldhausen category (that is compatible with the functor
  $\Sigma_{\mathcal{K}}$).  In this case, the definition of weighted
  triangles can be pursued following the steps in
  \S\ref{subsubsec:weight-tr}, but at the space level, without moving
  to an algebraic category. This approach goes beyond the scope of
  this paper and will not be pursued here.
\end{itemize}

\begin{rem}\label{rem:shift-et-al}
  Of course, it is also possible to mix in some sense the two
  approaches mentioned before. For instance, in the two examples below
  the category $\mathcal{K}$ carries a shift functor \jznote{$\Sigma_{\K}$ as
  at $B$ but also a functor $\Phi$ as at $A$} such that $\Phi$ commutes
  with the shift functors in the domain and target. In that case we
  can use $\Phi$ to pull back to $\mathcal{K}$ more of the structure
  and weights in $\C$ (of course, this remains less precise than
  constructing weights at the space level).
\end{rem}

\subsubsection{Topological spaces with action functionals.}
We will discuss here a category denoted by
$\mathcal{AT}op_{\ast}$. The objects of this category are pairs
$(A,f_{A})$ where \zjnote{$A = (A, \ast_A)$} is a pointed topological space and
$f_{A}:A\to \R$ is a continuous function bounded from below by the
value $f_{A}(\ast_{A})$ of $f_{A}$ at the base point $\ast_{A}$ of
$A$. We will refer to $f_{A}$ as the action functional associated to
$A$. The morphisms in this category are pointed continuous maps
$u:A\to B$ such that there exists $r\in \R$ with the property that
$f_{B}(u(x))\leq f_{A}(x)+r\ , \ \forall x\in A$.

We will see that there is a natural \pbnote{{\em contravariant}}
functor
\begin{equation} \label{eq:funtor-spaces}
\Phi:\mathcal{AT}op_{\ast} \to [H^{0}\mathcal{FK}_{\k}]_{\infty}
\end{equation}
inducing a weight $\bar{w}_{\mathcal{AT}op_{\ast}}$ and the associated
pseudo-metrics $\underline{d}^{\mathcal{F}}$ on
$\mathrm{Obj}(\mathcal{AT}op_{\ast})$ along the lines of point A
above.

\begin{rem}\label{rem:constr-top}
  The condition on $f_{A}$ being bounded from below is one possible
  choice in this construction. Its role is to allow the constant map
  $u: (A, f_{A})\to (B,f_{B})$ to be part of the morphisms of
  $\mathcal{AT}op_{\ast}$.
\end{rem}

Before proceeding with the construction of the functor $\Phi$ we
discuss some features of $\mathcal{AT}op_{\ast}$. Notice first that
the morphisms are filtered with the $r$-th stage being
\zjr{$$\Mor_{\mathcal{AT}op_*}^{\leq r}(A,B) =
\{ u: A\to B \ | \ u\ \mathrm{continuous}, \ \ u(\ast_A) = \ast_B, \ \  f_{B}(u(x)) \leq
f_{A}(x)+r \ , \ \forall x\in A\}~.~$$} There is an obvious \zjnote{family of functors}
$\Sigma_{\mathcal{AT}op_{\ast}} : (\R,+)\to \mathcal{AT}op_{\ast}$
defined by \zjnote{$\Sigma_{\mathcal{AT}op_{\ast}}^s(A,f_{A})=(A,f_{A}+s)$ and being the identity on
morphisms.} The next step is to define the translation functor
$T_{\mathcal{AT}op_{\ast}}$.  At the underlying topological level this
is just the topological suspension but we need to be more precise
about the action functional. Given an object $(A,f_{A})$ we first
define the cone $(CA, f_{CA})$. We take $CA$ to be the reduced cone,
in other words the quotient topological space
$CA=A\times [0,1]/(A\times \{1\}\cup \ast_{A}\times [0,1])$.  To
define $f_{CA}$ we first consider the homotopy
$h_{A}:A\times [0,1]\to \R$,
$$
h_{A}(x,t)= \left\{ \begin{array}{l}
f_{A}(x) \  \ \  \mathrm{if}\ \ \  0\leq t \leq \frac{1}{2}\\
(2-2t) (f_{A}(x) -f_{A}(\ast_{A})) + f_{A}(\ast_{A}) \ \ \  \mathrm{if} \  \  \ \frac{1}{2}\leq t \leq 1 
\end{array}\right. ~.~
$$
The map $f_{CA}: CA\to \R$ is induced by $h_{A}$. We now define the
reduced suspension, $SA= CA/A\times\{0\}$ and take $f_{SA}$ to be the
map induced to the quotient by the homotopy
$h'_{A}:A\times [0,1]\to \R$,

$$
h'_{A}(x,t)= \left\{ \begin{array}{l}
2t (f_{A}(x)-f_{A}(\ast_{A}))+f_{A}(\ast_{A}) \  \ \  \mathrm{if}\ \ \  0\leq t \leq \frac{1}{2}\\
(2-2t) (f_{A}(x) -f_{A}(\ast_{A})) + f_{A}(\ast_{A})\ \ \  \mathrm{if} \  \  \ \frac{1}{2}\leq t \leq 1 
\end{array}\right.~.~
$$
We put $T(A,f_{A})=(SA, f_{SA})$. It is immediate to see that $T$
extends to a functor on $\mathcal{AT}op_{\ast}$ and that it commutes
with $\Sigma$. Moreover, both $\Sigma$ and $T$ so defined commute and
are compatible with the filtration of the morphisms in the sense that
they take $\Mor^{\leq r}$ to $\Mor ^{\leq r}$ for each $r$. Moreover,
composition of morphisms is also compatible with the filtrations in
the sense that it takes
$\Mor^{\leq r_{1}}(B,C) \times \Mor^{\leq r_{2}}(A,B)$ to
$\Mor^{\leq r_{1}+r_{2}}(A,C)$.

We now define the class of exact triangles
$\Delta_{\mathcal{AT}op_{\ast}}$. For this we consider a morphism
$u: (A, f_{A})\to (B,f_{B})$ and we first define its cone
$\mathrm{Cone}(u)$. As a topological space this is, as expected, the
quotient topological space $(B\cup CA)/\sim$ where the equivalence
relation $\sim$ is generated by $f(x)\sim x\times \{0\}$. The base
point of $\mathrm{Cone}(u)$ is the same as that of $B$.  The action
functional $f_{\mathrm{Cone}(u)}$ is induced to the respective
quotient by :
$$
G(x)= \left\{ \begin{array}{l}
f_{B}(x) \  \ \  \mathrm{if}\ \ \  x\in B \\
(1-2t) (f_{B}(u(y)) -f_{B}(\ast_{B})) + 2t( f_{A}(y)-f_{A}(\ast_{A}))+f_{B}(\ast_{B}) \   \mathrm{if} \ \ x = (y,t) \in A\times [0,\frac{1}{2}]\\
(2-2t) (f_{A}(y) -f_{A}(\ast_{A}))+ f_{B}(\ast_{B}) \  \mathrm{if}\ \  x=(y,t) \in A\times [\frac{1}{2}, 1]
\end{array}\right.. 
$$
There is an obvious inclusion $i: (B,f_{B})\to \mathrm{Cone}(u)$ as
well as a projection $p :\mathrm{Cone}(u)\to TA$ (that contracts $B$
to a point). This map belongs to our class of morphisms because the
functional \zjr{$f_B$} is bounded from below.  The class
$\Delta_{\mathcal{AT}op_{\ast}}$ consists of triangles $\Delta$:
\begin{equation}\label{eq:extr-top}
  \Delta : A\xrightarrow{u} B\xrightarrow{i}
  \mathrm{Cone}(u)\xrightarrow{p} TA ~.~
\end{equation}
We finally construct the functor
$\Phi: \mathcal{AT}op_{\ast}\to [H^{0}\mathcal{FK}_{\k}]_{\infty}$.
\pbnote{This functor will be contravariant, since the objects of
  $\mathcal{FK}_{\k}$ are \jznote{cochain} complexes (rather than chain
  complexes).}


First we fix some notation: for a pointed topological space $X$ we
denote by $\tilde{C}_{\ast}(X)$ the reduced singular chain complex of
$X$ with coefficients in $\k$ and by $\tilde{C}^{\ast}(X)$ the reduced
singular cochain complex (we denote without $\tilde{(-)}$ the
non-reduced chain/cochain complexes) and if $Y\subset X$ is a pointed
subspace, then $\tilde{C}_{\ast}(X,Y)$ and $\tilde{C}^{\ast}(X,Y)$ are
the relative (co)chains.  Consider an object of
\jznote{$\mathcal{AT}op_*$}, $(A,f_{A})$, and let
$A^{\leq r}=(f_{A})^{-1}(-\infty, r]$. Notice that the spaces
$A^{\leq r}$ are pointed (if non-void).  \ocnote{There is a filtration of
$C^{\ast}(A)$ defined by:
$$\tilde{C}^{\ast}(A)^{\leq -r}=
    \mathrm{im}\{\tilde{C}^{\ast}(A, A^{\leq r})\to
    \tilde{C}^{\ast}(A)\}$$ Thus the filtration up to $s\in\R$ of $\tilde{C}^{\ast}(A)$ consists of the cochains in $A$ that
vanish over the singular chains of $A^{\leq -s}\subset A$. It is
clear that the cochain differential preserves this filtration. Moreover, the filtration is increasing and
if $f\in \hom^{\leq r}_{\mathcal{AT}op_{\ast}}(A,B)$,  then $f$ pulls-back the cochains 
in $B$ that vanish over $B^{\leq a}$ to cochains in $A$ that vanish over $A^{\leq a-r}$ and, as a result, $C^{\ast}(f):\tilde{C}^{\ast}B\to \tilde{C}^{\ast}(A)$ shifts filtration by $r$.}   
Finally, we define the functor $\Phi$.  For each object $(A,f_{A})$ of
$\mathcal{AT}op_{\ast}$ we take $\Phi(A, f_{A})$ to consist of the
cochain complex $\tilde{C}^{\ast}(A)$ together with the filtration
$\{\tilde{C}^{\ast}(A)^{\leq r}\}$ defined above.  For a morphism
$u: (A,f_{A})\to (B, f_{B})$ we take $\Phi(u)=[\tilde{C}^{\ast}(u)]$
where $[-]$ represents the cochain-homotopy class of the respective
cochain morphism.

The definition of the morphisms in $\mathcal{AT}op_{\ast}$ implies
that $\Phi(u)$ is indeed a morphism in
$[H^{0}\mathcal{FK}_{\k}]_{\infty}$. Moreover, because we are using
everywhere reduced cochain complexes (and we work in the pointed
category), we have that $\Phi(\Delta)$ is exact in
$[H^{0}\mathcal{FK}_{\k}]_{\infty}$ for each of the triangles in
$\Delta_{\mathcal{AT}op_{\ast}}$. Further, the functor $\Phi$ also
interchanges the shift functors in the domain and the target.

In all cases, the weight $\bar{w}_{\mathcal{AT}op_{\ast}}$ is
well-defined as well as the associated fragmentation pseudo-metrics
$\underline{d}^{\mathcal{F}}(-,-)$ on the objects of
$\mathcal{AT}op_{\ast}$. \jznote{Roughly speaking, these fragmentation
  pseudo-metrics measure how much ``weight'' we need to obtain a given
  topological space via successive cone attachments of spaces in
  $\F$.}
\begin{rem} \label{rem:gen-top} (a) The choice of the class
  $\Delta_{\mathcal{AT}op_{\ast}}$ given above is quite restrictive
  with the consequence that the resulting pseudo-metrics are often
  infinite. One alternative is to enlarge this class to all triangles
  in $\mathcal{AT}op_{\ast}$ that are homotopy equivalent to those in
  the initial class through maps (and homotopies) of filtration $0$.

  (b) From some points of view, working in the {\em pointed} category
  of spaces endowed with an action functional is not natural. Other
  choices are possible, in particular some such that the translation
  functor $T$ more closely imitates dynamical stabilization.

  (c) The restriction of $\Phi$ to compact topological spaces admits
  an obvious lift to $H^{0}\mathcal{FK}_{\k}$. However, without such a
  restriction, such a lift does not seem to be available in full
  generality.
\end{rem}

\subsubsection{Metric spaces} \label{subsubsec:metr} The category
$\mathcal{M}etr_{0}$ that we will consider here has as objects
path-connected metric spaces $(X,d_{X})$ of finite diameter.  The
morphisms are Lipschitz maps. \zjnote{Recall that $\phi :X\to Y$ is a
  Lipschitz map if there exists a constant $c\in [0,\infty)$, called
  the Lipschitz constant of $\phi$, with the property that
  $d_{Y}(\phi(x),\phi(y))\leq c~ d_{X}(x,y)$ for all $x,y\in X$.}
\begin{rem} \label{rem:constr-metr} The finite diameter condition
  imposed here - indicated by the \pbnote{subscript} $_{0}$ -
  \pbnote{is necessary for some of} the constructions below.  The
  connectivity \pbnote{assumption} is more a matter of convenience.
\end{rem}

We will construct a functor as in (\ref{eq:funtor-spaces}) with one
main modification.  For convenience, we prefer defining a covariant
functor and thus our target category will not be a category of cochain
complexes but rather one of filtered {\em chain complexes} (the
passage from one to the other is formal, \pbnote{replacing $C^*$ by
  $C_{-*}$ and vice versa}). We will denote the category of filtered
chain complexes over $\k$ by $\mathcal{FK}'_{\k}$. This behaves just
as a usual dg-category except that the differential on the space of
morphisms is of degree $-1$. With this change, we will construct:
\jznote{\begin{equation}\label{eq:functor-metr}
  \Phi': \mathcal{M}etr_{0} \to [H_{0}\mathcal{FK}'_{\k}]_{\infty}
\end{equation}}
as well as related structures on $\mathcal{M}etr_{0}$, as at point A at
the beginning of the section (see also Remark \ref{rem:shift-et-al}).

We start by noting that there is an obvious increasing filtration of
the morphisms in $\mathcal{M}etr_{0}$
with
$$\Mor_{\mathcal{M}etr_{0}}^{\leq r}(X,Y)=\{ u: X\to Y \ | \
\mathrm{the\ Lipschitz\ constant\ of}\ u \ \mathrm{is}\ \leq
e^{r}\}~.~$$ It is immediate to see that this filtration is compatible
with composition. There is also a \zjnote{family of functors}
$\Sigma_{\mathcal{M}etr_{0}}:(\R, +)\to \mathcal{M}etr_{0}$ defined by
rescaling the metric, $\Sigma^{s}(A, d_{A}) = (A, e^{s} d_{A})$ \zjnote{and being the identity} on morphisms. As in the example in the previous
section, we next will define the translation functor
$T_{\mathcal{M}etr_{0}}$ and the class of triangles
$\Delta_{\mathcal{M}etr_{0}}$. The first step is to construct the
metric cone $C'A$ for an object $(A, d_{A})$ in our
class. Topologically, the cone $C'A$ will be this time the {\em
  unreduced} cone over $A$. Thus it is defined by
$C'A=A\times [0,1]/A\times \{0\}$.  To define the metrics $d_{C'A}$,
first let $D_{A}$ be the diameter of $A$.  We then put
\begin{equation}
d_{C'A}((x,t), (y,t'))= \frac{D_{A}}{2}|t-t'| + \min\{t,t'\}~d_{A}(x,y).
\end{equation}
It is immediate to see that this does indeed define a metric on
$C'A$. A similar construction is available to construct $T(A,
d_{A})$. Topologically, we will define first the - {\em non-reduced} -
suspension, $S'A$, as the topological quotient of
$A\times [-\frac{1}{2},\frac{1}{2}]$ with $A\times \{-1/2\}$
identified to a point $S$ and $A\times \{+1/2\}$ identified to a
different point $N$. We now define $d_{S'A}$ by
\begin{equation}
  d_{S'A}((x,t), (y,t'))=
  \frac{D_{A}}{2}|t-t'| +\min\left\{\frac{1}{2}-|t|,\frac{1}{2}-|t'|\right\}~d_{A}(x,y)
\end{equation}
and again it is immediate to see that this defines a metric on
$S'A$. We now put $T(A,d_{A})= (S'A, d_{S'A})$.  The next step is to
define the triangles in $\Delta_{\mathcal{M}etr_{0}}$. For this we
assume $u : (A,d_{A})\to (B,d_{B})$ is a morphism in our category and
we want to define the (non-reduced) cone of $u$,
$\mathrm{Cone}'(u)$. Topologically, this is, as usual,
$B\cup C'A/[\{x\}\times \{0\}\sim u(x) \ | \ x\in A]$. To define a
metric on $\mathrm{Cone}'(u)$ we notice first that given \pbnote{a
  map} $g:X\to Y$ and a pseudo-metric $d_{Y}$ on $Y$, there is a
pull-back pseudo-metric on $X$ given by
$g^{\ast}d_{Y} (a,b)= d_{Y}(g(a),g(b))$.  We now let
$A'= u(A)\subset B$ and we denote by $\bar{u}:A'\hookrightarrow B$ the
inclusion. Notice that $\mathrm{Cone}'(\bar{u})\subset C'B$.  Thus
$\mathrm{Cone}'(\bar{u})$ is endowed with a metric given by the
restriction of the metric $d_{C'B}$ on $C'B$. There are obvious
projections $\pi:\mathrm{Cone}'(u)\to \mathrm{Cone}'(\bar{u})$ and
$p: \mathrm{Cone}'(u)\to S'A$. Here $p$ collapses $B$ to the point $S$
in the suspension and sends $(x,t) \to (x,t-\frac{1}{2})$ for the
points $(x,t)\in C'A$. We now define
\zjr{$$d_{\mathrm{Cone}'(u)}:=\pi^{\ast}d_{C'B}+
p^{\ast}d_{S'A}~.~$$} Notice that, if $u$ is not injective and $B$ is
not a single point, then the two pseudo-metrics in the right term of
the equality are each degenerate. Nonetheless, $d_{\mathrm{Cone}'(u)}$
is non-degenerate. Finally, the class of triangles
$\Delta_{\mathcal{M}etr_{0}}$ consists of triangles:
$$A\xrightarrow{u} B\xrightarrow{i} \mathrm{Cone}'(u)\xrightarrow{p} S'A$$
where $i$ is the inclusion and $p$ is the projection above.

With this preparation, we can now define the functor $\Phi'$ from
(\ref{eq:functor-metr}).  Consider an object $(A,d_{A})$ in our
category and the associated singular complex $C_{\ast} (A)$.  This
chain complex is filtered as follows:
\jznote{$$C^{\leq r}_{k}(A)=\left\{\sum_{i }a_{i}
\sigma_{i} \ \bigg|\ a_{i}\in\k \ , \ \sigma_{i} \ \mathrm{a\ singular\
  simplex\ of\ diameter\ at\ most}\ e^{r}\right\}~.~$$} In other words, in
the expression above, $\sigma_{i}: \Delta^{k}\to A$ is a continuous
map with the standard $k$-simplex as domain and such that
$d_{A}(\sigma_{i}(x),\sigma_{i}(y))\leq e^{r}$ for any
$x,y\in \Delta^{k}$.  Consider the constant map $c : A \to \ast$. This
induces an obvious surjection $C_{\ast}\to C_{\ast}(\ast)$ and we
denote by $\bar{C}_{\ast}(A)$ the kernel of this map (this is
quasi-isomorphic to the reduced singular chain complex of $A$ -
because $A$ is connected - but is independent of the choice of
base-point). There is an induced filtration
$\bar{C}^{\leq r}_{\ast}(A)$.  We now put
$$\Phi'(A,d_{A})= \bar{C}_{\ast}(A)\ \mathrm{with\ the\ filtration}\
\{\bar{C}^{\leq r}_{\ast}(A)\}_{r} ~.~$$ Further, for a morphism
$u: (A,d_{A})\to (B,d_{B})$ we take $\Phi'(u)=[C_{\ast}(u)]$, the
chain homotopy class of the singular chain map $C_{\ast}(u)$
(restricted to the $\bar{C}(-)$ complexes).

It is easy to see that this $\Phi'$ is indeed a functor as desired and
that $\Phi'(\Delta)$ is exact for each triangle $\Delta$ as defined
above and, again, $\Phi'$ interchanges the shift functors in the
domain and target. In summary, the weight
$\overline{w}_{\mathcal{M}etr_{0}}$ is well-defined as well as the
quantities $\underline{\delta}^{\mathcal{F}}$ and the pseudo-metrics
associated to them.

\begin{rem}\label{rem:gen-metr}
  (a) Similarly to Remark \ref{rem:gen-top}, the definition of the
  triangles in $\Delta_{\mathcal{M}etr_{0}}$ is highly restrictive
  and, in this case, even the objects in our category are subject to a
  constraint - finiteness of the diameter - that might be a hindrance
  in applications. One way to apply the methods above to study spaces
  of infinite diameter is to consider triangles of the form
  $\Delta : A\to B\to C \to S'A$ where $A$ is of finite diameter such
  that $S'A$ admits a metric as above and analyze when $\Phi'(\Delta)$
  is of finite persistence weight in
  $[H_{0}\mathcal{FK}'_{\k}]_{\infty}$.

  (b) In studying metric spaces of infinite diameter by these methods,
  it is likely that the most appropriate structure that fits with the
  cone construction is that of {\em length structure}, in the sense of
  Gromov, as in Chapter 1, Section A in \cite{Pansu}. We will not
  further pursue this theme here.
\end{rem}

\subsubsection{Further remarks on topological examples.}

\

A. In the topological examples above - for instance in
$\mathcal{AT}op_{\ast}$ - it is natural to see what the quantities
$\underline{\delta}^{\mathcal{F}}(-)$ mean even for the flat weight
$w_{0}$, which associates to each exact triangle the value $1$. Of
course, in this case $\underline{\delta}^{\mathcal{F}}(A, B)$ simply
counts the minimal number of cone-attachments in the category
$\mathcal{AT}op_{\ast}$ that are needed to obtain $A$ out of the space
$B$ by attaching cones over spaces in the family $\mathcal{F}$ using
the family of triangles $\Delta_{\mathcal{AT}op_{\ast}}$. Given that
the weight is flat, the question is independent of filtrations and
shift functors and it reduces to the identical question in the
category of pointed spaces $\mathcal{T}op_{\ast}$. In the examples
below we will focus on this category and on
$\underline{\delta}^{\mathcal{F}}(A,\ast)$ which is one of the most
basic quantities involved.

It is useful to keep in mind that there are two more choices that are
essential in defining $\underline{\delta}^{\mathcal{F}}(A,\ast)$: the
choice of family $\mathcal{F}$ and the choice of the class of exact
triangles $\Delta_{\mathcal{T}op_{\ast}}$ - see also Remark
\ref{rem:gen-top} (a).

\begin{itemize}
\item[(i)] $\mathcal{F}=\{S^{0}, S^{1},\ldots, S^{k},\ldots\}$;
  $\Delta_{\mathcal{T}op_{\ast}}$ are the triangles
  $A\xrightarrow{u} B\to \mathrm{Cone}(u)\to SA$ as in
  (\ref{eq:extr-top}) (but omitting the action functionals).  In this
  case, $\underline{\delta}^{\mathcal{F}}(A,\ast)= k < \infty$ means
  that $A$ has the structure of a finite $CW$- complex with $k$ cells.
\item[(ii)] $\mathcal{F}=\{S^{0}, S^{1},\ldots, S^{k},\ldots\}$; we
  now take $\Delta_{\mathcal{T}op_{\ast}}$ to be the triangles
  homotopy equivalent to the triangles
  $A\xrightarrow{u} B\to \mathrm{Cone}(u)\to SA$ from
  (\ref{eq:extr-top}). In this case,
  $\underline{\delta}^{\mathcal{F}}(A,\ast)= k$ means that $A$ is
  homotopy equivalent to a $CW$-complex with $k$ cells.  This number
  is obviously a homotopy invariant. It is clearly bounded from below
  by the sum of the Betti numbers of $A$.
\item[(iii)] \zjr{$\mathcal{F}$ consists of all} pointed spaces with the homotopy type
  of $CW$-complexes; $\Delta_{\mathcal{T}op_{\ast}}$ are as at
  (ii). In this case, the definition
  $\underline{\delta}^{\mathcal{F}}(A,\ast)$ coincides with that of
  the cone-length, $\mathrm{Cl}(A)$, of $A$ (for a space $A$ with the
  homotopy type of a $CW$-complex).  Cone-length is a homotopical
  invariant which is of interest because it is bigger, but not by more
  than one, than the Lusternik-Schnirelmann category \cite{Cor94} which, in turn,
  provides a lower bound for the minimal number of critical points of
  smooth functions on manifolds. Incidentally, as noted by Smale
  \cite{Sm}, a version of the Lusternik-Schnirelmann category provides
  also a measure for the complexity of algorithms, see
  \cite{Co-Lu-Op-Ta} for more on this subject.
\item[(iv)] At this point we will change the underlying category and
  place ourselves in the pointed category of finite type,
  simply-connected {\em rational} spaces $\mathcal{T}op_{1}^{\Q}$
  (see~\cite{Hal-Fel-Th}). We take $\mathcal{F}$ to consist of finite
  wedges of rational spheres of dimension at least $2$.  The triangles
  $\Delta_{\mathcal{T}op^{\Q}_{\ast}}$ are as at (ii) (in the category
  of rational spaces) but we will also allow in
  $\Delta_{\mathcal{T}op^{\Q}_{\ast}}$ \zjnote{``formal'' triangles} of the
  form $S^{-1}F\to \ast \to F$ where $F\in \mathcal{F}$ (de-suspending
  is not possible in our category but we still want to have for a
  rational $2$-sphere, $S^{2}_{\Q}$,
  $\underline{\delta}^{\mathcal{F}}(S^{2}_{\Q}, \ast)=1$).  In this
  setting, it turns out \cite{Co-tr} that
  $$\underline{\delta}^{\mathcal{F}}(A,\ast)=\mathrm{Cl}(A)=\mathrm{nil}(A)~.~$$ 
  Both equalities here are non-trivial, the first because in the
  definition of $\mathrm{Cl}$ we are using cones over arbitrary
  (rational) spaces while in this example $\mathcal{F}$ consists of
  only wedges of spheres. For the second equality, $\mathrm{nil}(A)$
  is the minimal order of nilpotence of the augmentation ideal
  $\overline{\mathcal{A}}$ of a rational differential graded
  commutative algebra $\mathcal{A}$ representing $A$ (recall that by a
  celebrated result of Sullivan \cite{Sul}, the homotopy category of
  rational simply connected spaces is equivalent to the homotopy
  category of rational differential graded commutative algebras, the
  representative of a given space being given by the so-called $PL$-de
  Rham complex of $A$).
\end{itemize}

\

B. One of the difficulties of extracting a triangulated persistence
category from a topological category such as those considered in this
section is very basic and has to do with the difference between stable
and unstable homotopy. In essence, recall that if $\mathcal{C}$ is a
TPC, then the $0$-level category $\mathcal{C}_{0}$ is required to be
triangulated. However, in unstable settings, homotopy categories of
spaces are not triangulated.
\begin{itemize}
\item[(i)] An instructive example is a variant of our discussion
  concerning the category $\mathcal{M}etr_{0}$. In this case the
  morphisms $\Mor_{\mathcal{M}etr_{0}}(A,B)$ carry an obvious topology
  as well as a filtration, as described in \S\ref{subsubsec:metr}. We
  now can consider a new category, $\widetilde{\mathcal{M}}etr_{0}$,
  with the same objects as $ \mathcal{M}etr_{0}$ but with morphisms
  $\Mor_{\widetilde{\mathcal{M}}etr_{0}}(A,B)=S_{\ast}(\Mor_{\mathcal{M}etr_{0}}(A,B))$
  where $S_{\ast}(-)$ stands for cubical chain complexes. These
  morphisms carry an obvious filtration obtained by applying the
  cubical chains to the filtration of
  $\Mor_{\mathcal{M}etr_{0}}(A,B)$.  The composition in this category
  is given by applying cubical chains to the composition
  $\Mor_{\mathcal{M}etr_{0}}(B,C)\times
  \Mor_{\mathcal{M}etr_{0}}(A,B)\to \Mor_{\mathcal{M}etr_{0}}(A,C)$
  and composing with map
  $S_{\ast} (\Mor_{\mathcal{M}etr_{0}}(B,C))\otimes
  S_{\ast}(\Mor_{\mathcal{M}etr_{0}}(A,B))\to
  S_{\ast}(\Mor_{\mathcal{M}etr_{0}}(B,C)\times
  \Mor_{\mathcal{M}etr_{0}}(A,B))$ induced by taking products of
  cubes. It follows that $\widetilde{\mathcal{M}}etr_{0}$ is a
  filtered dg-category (in homological formalism).  Thus all the
  machinery in \S\ref{subsec:dg} is applicable in this case. Moreover,
  this category carries an obvious shift functor.  However,
  $[H_{0}\widetilde{\mathcal{M}}etr_{0}]_{\infty}$ is not triangulated
  and thus $\widetilde{\mathcal{M}}etr_{0}$ is not pre-triangulated
  (quite far from it).  Indeed,
  $\Mor_{[H_{0}\widetilde{\mathcal{M}}etr_{0}]_{\infty}}(A,B)$ is the
  free abelian group generated by the homotopy classes of Lipschitz
  maps from $A$ to $B$. As a result, the translation functor (which is
  in our case the topological suspension) is certainly not an
  isomorphism.

\item[(ii)] As mentioned before, at point B at the beginning of
  \S\ref{subsec:top-sp}, a way to bypass these issues is to introduce
  a sort of filtered Waldhausen category or a similar formalism and
  develop a machinery parallel to that of TPC's in this unstable
  context. The structure present in $\mathcal{AT}op_{\ast}$ and
  $\mathcal{M}etr_{0}$ suggests that such a construction is possible
  and will be relevant in these cases.

\item[(iii)] There is yet another approach to associate to each of
  $\mathcal{AT}op_{\ast}$ and $\mathcal{M}etr_{0}$ a triangulated
  persistence category that is more geometric in nature. This is based
  on moving from these categories to stable categories, where the
  underlying objects are the spectra obtained by stabilizing the
  objects of the original categories and the morphisms come with an
  appropriate filtration induced from the respective structures
  (action functionals or, respectively, metrics) on the initial
  objects.  This seems likely to work and to directly produce a TPC
  but we will not pursue the details at this time.
\end{itemize}

\subsection{Filtrations in Tamarkin's category}\label{subsec:Tam}
This section is devoted to an example of a triangulated persistence
category that comes from the filtration structure present in
Tamarkin's category. This category was originally defined in
\cite{Tam08}, based on singular supports of sheaves, and was used to
prove some non-displaceability results in symplectic geometry, as well
as other more recent results related to Hamiltonian dynamics (see
\cite{GKS12}).

\subsubsection{Background on Tamarkin's category} Let $X$ be a
manifold, and let $\D(\k_X)$ be the derived category of sheaves of
$\k$-modules over $X$. In particular, this is a triangulated
category. For any $A \in {\rm Obj}(\D(\k_X))$, due to microlocal sheaf
theory, as established in \cite{KS90}, one can define the singular
support of $A$, denoted by $SS(A)$, a conical (singular) subset of
$T^*X$. \jznote{We refer to Chapter V in \cite{KS90} for the precise
  definition of $SS(A)$ and a detailed study of its properties.} Now,
let $X = M \times \R$ where $M$ is a closed manifold, and denote by
$\tau$ the co-vector coordinate of $T^*\R$ in $T^*(M \times
\R)$. Consider the following full subcategory of
$\D(\k_{M \times \R})$, denoted by
$\D_{\{\tau \leq 0\}}(\k_{M \times \R})$, where
\[ {\rm Obj}(\D_{\{\tau \leq 0\}}(\k_{M \times \R})) = \left\{A \in
    {\rm Obj}(\D(\k_{M \times \R})) \, | \, SS(A) \subset \{\tau \leq
    0\}\right\}. \] If $A \to B \to C \to A[1]$ is an exact triangle
in $\D(\k_{M \times \R})$, then $SS(C) \subset SS(A) \cup SS(B)$. This
implies that $\D_{\{\tau \leq 0\}}(\k_{M \times \R})$ is a
triangulated subcategory of $\D_{\k_{M \times \R}}$. Tamarkin's
category is defined by
\begin{equation} \label{dfn-tam-cat}
\T(M) : = \D_{\{\tau \leq 0\}}(\k_{M \times \R})^{\perp, l}
\end{equation} 
where the ${\perp, l}$ denotes the left orthogonal complement of
$\D_{\{\tau \leq 0\}}(\k_{M \times \R})$ in $\D(\k_{M \times
  \R})$. Then $\T(M)$ is also a triangulated subcategory. By
definition, note that
$\T(M) \subset \D_{\{\tau \geq 0\}}(\k_{M \times \R})$. When
$M = \{\rm pt\}$, Tamarkin's category $\T(\{{\rm pt}\})$, together
with a constructibility condition, can be identified with the category
of persistence $\k$-modules (see A.1 in \cite{Zha20}).

\begin{rem} There exists a restricted version of Tamarkin's category
  denoted by $\T_V(M)$ where $V \subset T^*M$ is a closed subset (see
  Section 3.2 in \cite{Zha20}). This restricted Tamarkin category is
  useful to prove the non-displaceability of some subsets in $T^*M$
  (see \cite{AI20}). In this paper, we will only focus on $\T(M)$.
\end{rem}

One way to understand the definition (\ref{dfn-tam-cat}) is that $\T(M)$ is an admissible subcategory (see Definition 1.8 in \cite{Orl06}) in the sense that for any object $A$ in $\D(\k_{M \times \R})$, one can always split $A$ in the form of an exact triangle 
\jznote{
\begin{equation} \label{tamarkin-split}
B \to A \to C \to B[1]
\end{equation}}
in $\D(\k_{M \times \R})$, where $B \in \T(M)$ and $C \in \D_{\{\tau \leq 0\}}(\k_{M \times \R})$. In fact, this splitting can be achieved in a rather concrete manner, which involves an important operator called sheaf convolution on objects in $\D(\k_{M \times \R})$. Explicitly, for any two objects $A, B$ in $\D({\bf k}_{M \times \R})$, the sheaf convolution of $A$ and $B$ is defined by
\begin{equation} \label{dfn-sh-conv}
A \ast B : = \delta^{-1} Rs_! (\pi_1^{-1} A \otimes \pi_2^{-1} B),
\end{equation}
where $\pi_i: (M \times \R)^2 \to M \times \R$ are the projections to each factor of $M \times \R$, $s$ keeps the $M \times M$-part the same but adds up two inputs on the $\R$-factors, and $\delta$ is the diagonal embedding from $M$ to $M \times M$. For instance, $\k_{M \times [0, \infty)} * \k_{M \times [0,\infty)} = \k_{M \times [0, \infty)}$, where $\k_{V}$ for a closed subset $V$ denotes the constant sheaf with its support in $V$. Moreover, this operator is commutative and associative. An important characterization of an object in $\T(M)$ is that (see Proposition 2.1 in \cite{Tam08}), 
\begin{equation} \label{tam-char}
A \in {\rm Obj}(\T(M)) \,\,\,\,\mbox{if and only if} \,\,\,\, A * \k_{M \times [0,\infty)} = A
\end{equation}
which implies that (i) for any object $A$ in $\D(\k_{M \times \R})$, the sheaf convolutions $B: = A * \k_{M \times [0,\infty)}$ and $C: =A * \k_{M \times (0, \infty)}[1]$ provide the desired exact triangle for a splitting of $A$ in (\ref{tamarkin-split}); (ii) sheaf convolution is a well-defined operator on $\T(M)$. 

With the help of the sheaf convolution, the $\R$-component generates a filtration structure in $\T(M)$ in the following way. For any $r \in \R$, consider the map $T_r: M \times \R \to M \times \R$ defined by $(m, a) \mapsto (m, a+r)$. One can show that for any object $A$ in $\T(M)$, the induced object $(T_r)_*A = A * \k_{M \times [r, \infty)}$ (see Lemma 3.2 in \cite{Zha20}). In fact, $\{(T_r)_*\}_{r \in \R}$ defines an $\R$-family of functors on $\T(M)$. Moreover, if $r \leq s$, then by the restriction map $\k_{M \times [r, \infty)} \to \k_{M \times [s, \infty)}$, we have a canonical morphism $\tau_{r,s}(A): (T_{r})_*A \to (T_{s})_*A$. At this point, notice that for $r \leq s$, there does not exist non-zero morphism from $\k_{M \times [s, \infty)}$ to $\k_{M \times [r, \infty)}$, so the canonical map $\tau_{r,s}$ respects the partial order $\leq$ on $\R$. For any $r \leq s$, $\tau_{r,s}$ is viewed as a natural transformation from $(T_r)_*$ to $(T_s)_*$. Finally, we call an object $A$ in $\T(M)$ a $c$-torsion element if $\tau_{0,c}(A): A \to (T_c)_*A$ is zero. For instance, when $M = \{\rm pt\}$, the constant sheaf $\k_{[a,b)} \in \T(\{\rm pt\})$ with a finite interval $[a,b)$ is a $(b-a)$-torsion. 

We will end this subsection by a discussion on the hom-set in $\T(M)$. It is more convenient to consider derived hom, that is, ${\rm Rhom}_{\T(M)}(A,B)$ for any two objects $A,B$ in $\T(M)$. Lemma 3.3 in \cite{Zha20} (or (1) in Lemma 3.8 in \cite{Tam08}) provides a more explicit way to express such ${\rm Rhom}$, that is, 
\begin{equation} \label{tam-hom-set}
{\rm Rhom}_{\T(M)}(A,B) = {\rm Rhom}_{\R}(\k_{[0, \infty)}, R\pi_*{\mathcal Hom}^*(A,B)). 
\end{equation}
By taking the cohomology at degree $0$, we obtain ${\hom}_{\T(M)}(A,B)$ as a $\k$-module. Here, $\pi: M \times \R \to \R$ and ${\mathcal Hom}^*(\cdot, \cdot)$ is the right adjoint functor to the sheaf convolution (see Definition 3.1 in \cite{AI20}). The right-hand side of (\ref{tam-hom-set}) is relatively computable since they are all (complexes of) sheaves over $\R$ (cf.~A.2 in \cite{Zha20}). Moreover, by using the adjoint relation between ${\mathcal Hom}^*(\cdot, \cdot)$ and the sheaf convolution, one obtains a shifted version of (\ref{tam-hom-set}), that is,
\begin{equation} \label{tam-hom-set2}
{\rm Rhom}_{\T(M)}(A,(T_r)_*B) = {\rm Rhom}_{\R}(\k_{[0, \infty)}, (T_r)_*(R\pi_*{\mathcal Hom}^*(A,B))).
\end{equation}
Therefore, for any $r \leq s$, there exists a well-defined morphism 
\begin{equation} \label{tam-hom-transfer}
\iota^{A,B}_{r,s}: {\hom}_{\T(M)}(A,(T_r)_*B) \to {\hom}_{\T(M)}(A,(T_s)_*B),
\end{equation}
which is induced by the morphism $\tau_{r,s}(R\pi_*{\mathcal Hom}^*(A,B))$. Finally, we have a canonical isomorphism, 
\begin{equation} \label{tam-equiv}
T_r: {\hom}_{\T(M)}((T_r)_*A, (T_r)_*B) \simeq {\hom}_{\T(M)}(A,B)
\end{equation}
which is induced by the sheaf convolution with $\k_{M \times [r, \infty)}$. In particular, $T_r$ commutes with the morphism $\iota^{A,B}_{r,s}$ defined in (\ref{tam-hom-transfer}).

\subsubsection{Persistence category from Tamarkin's shift functors}
We have seen before that 
 Tamarkin's category is endowed with a shift functor.   We now discuss the persistence structure induced by this shift functor - see Remark \ref{rem:shifts-T} (d).

\begin{dfn} \label{dfn-P(M)} Given the category $\T(M)$ as before, define an enriched category denoted by ${\mathcal P}(M)$ as follows. The object set of ${\rm Obj}({\mathcal P}(M))$ is the same as ${\rm Obj}(\T(M))$, and the hom-set is defined by
\[ {\hom}_{{\mathcal P}(M)}(A,B) = \left\{\{{\hom}_{\T(M)}(A,(T_{r})_*B)\}_{r \in \R}, \{\iota^{A,B}_{r,s}\}_{r\leq s \in \R}\right\} \]
for any two objects $A,B$ in ${\mathcal P}(M)$, where $\iota^{A,B}_{r,s}$ is the morphism defined in (\ref{tam-hom-transfer}). \end{dfn}

\begin{remark} \jznote{Definition \ref{dfn-P(M)} can be regarded as a generalization of (\ref{per-mor}) in \S\ref{ssec-pm} since when $M = \{\rm pt\}$, Tamarkin's category $\mathcal T(\{\rm pt\})$ can be identified with the category of persistence $\k$-modules.} Also, Definition \ref{dfn-P(M)} fits with geometric examples. Indeed, recall a concrete computation of ${\hom}_{\T(M)}(A,B)$ when both $A$ and $B$ are sheaves coming from generating functions on $M$ (see Section 3.9 in 
\cite{Zha20}). In this case, \jznote{${\rm{hom}}_{{\mathcal P}(M)}(A,B)$} can be identified to a (Morse) persistence $\k$-module in the classical sense. \end{remark}

\begin{lemma} \label{lem-tamarkin-pc} The category ${\mathcal P}(M)$  from Definition \ref{dfn-P(M)},  is a persistence category. \end{lemma}

\begin{proof} Consider the functor $E_{A,B}: (\R, \leq) \to {\rm Vect}_{\k}$ by 
\[ E_{A,B}(r) (= {\hom}^r(A,B)) =  {\hom}_{\T(M)}(A,(T_{r})_*B), \]
and for the morphism $i_{r,s}$ when $r \leq s$, $E_{A,B}(i_{r,s}) = \iota^{A,B}_{r,s}$. Notice that the composition $E_{A,B}(r) \times E_{B,C}(s) \to E_{A,C}(r+s)$ is well-defined due to (\ref{tam-equiv}). Indeed, for any $f \in E_{A,B}(r)$ and $g \in {E_{B,C}(s)}$, the composition is defined by 
\[ (f, g) \mapsto T_r(g) \circ f \in {\hom}_{\T(M)}(A, (T_{r+s})_*C). \]
Then for any $r \leq r'$, $s\leq s'$, we have 
\[ \left(T_{r'}(\iota^{B,C}_{s,s'}(g))\right) \circ \iota^{A,B}_{r,r'}(f) = \iota^{A,C}_{r+s, r'+s'}(T_r(g) \circ f) \]
which completes the proof that ${\mathcal P}(M)$ is a persistence category. 
\end{proof}


We now list  some of the properties of the persistence category ${\mathcal P}(M)$. 
\begin{itemize} 
\item[(a)] The $0$-level category  ${\mathcal P}(M)_0$ has the same objects as ${\mathcal P}(M)$, but 
\[ {\hom}_{{\mathcal P}(M)_0}(A,B) = {\hom}_{\T(M)}(A, B). \]
We use the fact that $(T_0)_*= \mathds{1}$. Thus, ${\mathcal P}(M)_0 = \T(M)$. This category is triangulated as we have seen above.
\item[(b)] The $\infty$-level, ${\mathcal P}(M)_{\infty}$, has the same objects as ${\mathcal P}(M)$, but 
\[ {\hom}_{{\mathcal P}(M)_{\infty}}(A,B) = \varinjlim_{r \to \infty} {\hom}_{\T(M)}(A,(T_{r})_*B) \]
where the direct limit is taken via the map $\iota^{A,B}_{r,s}$. This limit category has been considered in (81) in \zjnote{Proposition 6.7 \cite{GS14}, where it is approached from the perspective of a categorical localization on torsion elements. This can be regarded as a special case of Proposition \ref{prop:verdier-loc} in \S \ref{subsubsec:localization}, where the localization is established for a general triangulated persistence category.} 
\item[(c)] On ${\mathcal P}(M)$, each $(T_r)_*$ is a persistence functor for any $r \in \R$, i.e., $(T_r)_* \in \mathcal P({\rm End}({\mathcal P}(M)))$, since $T_r$ commutes with $\tau_{r,s}$. 
\item[(d)] There exists a natural shift functor on ${\mathcal P}(M)$. Define $\Sigma: (\R, +) \to \mathcal P({\rm End}({\mathcal P}(M)))$ by $\Sigma(r) (= \Sigma^r) = (T_{-r})_*$. For any $r, s \in \R$ and $\eta_{r,s} \in {\hom}_{\R}(r,s)$, define 
\[ \Sigma(\eta_{r,s}) := \mathds{1}_{(T_{-r})_* \cdot}. \]
Then, for any object $A$ in $\mathcal P(M)$, 
\begin{align*}
\Sigma(\eta_{r,s})_A = \mathds{1}_{(T_{-r})_*A} & \in {\hom}_{\T(M)}((T_{-r})_*A, (T_{-r})_*A) \\
& = {\hom}_{\T(M)}((T_{-r})_* A, (T_{s-r})_*((T_{-s})_*A))\\
& = E_{(T_{-r})_* A, (T_{-s})_*A)}(s-r) = {\hom}^{s-r}((T_{-r})_*A, (T_{-s})_*A).
\end{align*}
In other words, $\Sigma(\eta_{r,s})_A$ is a natural transformation of shift $s-r$. In particular, the morphism $\eta^A_r = i_{-r,0}((\eta_{r,0})_A) \in {\hom}^0(\Sigma^rA, A)$ is well-defined for any $r \geq 0$. It is easy to check that $\eta^A_r = (\tau_{-r,0})(A)$. 
\item[(e)] The $r$-acyclic objects in ${\mathcal P}(M)$ are precisely the $r$-torsion elements in $\T(M)$. Indeed, by definition, an object $A$ in ${\mathcal P}(M)$ is $r$-acyclic if and only if $\eta^A_{r} = \tau_{-r,0}(A): (T_{-r})_*A \to A$ is the zero morphism, which coincides with the definition of an $r$-torsion element under the isomorphism (\ref{tam-equiv}). 
\item[(f)] Recall that for each $r$, ${\hom}^r(-, X) = {\hom}_{\T(M)}(-, (T_r)_*X)$. This is an exact functor due to (\ref{tam-hom-set2}) on ${\mathcal P}(M)_0 = \T(M)$.  Similarly, ${\hom}^r(X, -)$ is also an exact functor on ${\mathcal P}(M)_0 = \T(M)$.
\end{itemize}

\begin{lemma} \label{lemma-tor} For any $r \geq 0$ and any object $A$ in ${\mathcal P}(M)$, the morphism $\eta^A_r: (T_{-r})_*A \to A$ embeds into the following exact triangle 
\begin{equation} \label{ext-tor}
(T_{-r})_*A \xrightarrow{\eta^A_r} A \to K \to A 
\end{equation}
in $\T(M) = {\mathcal P}(M)_0$, where $K$ is $r$-acyclic. \end{lemma}

\begin{proof} Since $\T(M)$ is a triangulated category, the morphism $\eta^A_r$ embeds into an exact triangle as (\ref{ext-tor}). By item (e) above, we need to show that $K$ is an $r$-torsion element. By (ii) in \zjnote{Lemma 6.3} in \cite{GS14} which provides a criterion to test an object in an exact triangle to be a torsion element, it suffices to verify that the following diagram is commutative, 
\[ \xymatrixcolsep{4pc} \xymatrix{
(T_{-r})_*A \ar[r]^-{\eta^A_r} \ar[d]_-{\tau_{0,r}((T_{-r})_*A)} & A \ar[d]^-{\tau_{0,r}(A)} \ar[ld]_-{\alpha} \\
A \ar[r]_-{T_r(\eta^A_r)} & (T_r)_*A} \]
for some morphism $\alpha$. Indeed, this is commutative by choosing $\alpha = \mathds{1}_A$ together with the functorial properties of $T_{-r}$. 
\end{proof}
\begin{remark} By the definition of an $r$-isomorphism defined earlier, Lemma \ref{lemma-tor} implies that the morphism $\eta^A_r \in {\hom}^0((T_r)_*A, A)$ is an $r$-isomorphism. On the other hand, Section 3.10 in \cite{Zha20} defines an interleaving relation between two objects in $\T(M)$, which is similar to $r$-isomorphism defined in the sense that $A$ and $(T_r)_*A$ are $r$-interleaved. \end{remark}
\begin{ex} Let $M = \{{\rm pt}\}$ and consider $\T(\{\rm pt\})$. For $A = \k_{[0, \infty)}$, we know that $(T_{-r})_*A = \k_{[-r, \infty)}$ for any $r \geq 0$. Then we have an exact triangle in $\T(\{\rm pt\})$,
\[ \k_{[-r, \infty)} \xrightarrow{\tau_{-r,0}(A)} \k_{[0, \infty)} \to \k_{[-r,0)}[1] \to \k_{[-r, \infty)}[1] \]
where as we have seen that $\k_{[-r,0)}[1]$ is an $r$-torsion element (so $r$-acyclic). Here, by definition, $\tau_{-r,0}(A)$ is the restriction map from $\k_{[-r, \infty)}$ to $\k_{[0, \infty)}$, and the exact triangle is from (2.6.33) in \cite{KS90}. \end{ex}
The properties at the points (a) and (d) above together with Lemmas \ref{lem-tamarkin-pc}, \ref{lemma-tor},  imply the consequence of main interest in this section.
\begin{cor} \label{thm-tam-tpc} The category ${\mathcal P}(M)$, as
  defined \zjnote{in Definition \ref{dfn-P(M)}}, is a triangulated persistence category. 
 \end{cor}

\chapter{Triangulated persistence Fukaya categories} \label{s:tpfc}
In this section we apply the theory developed in Chapter
\ref{chap:Alg} to the case of Fukaya categories. The setup described
before applies naturally to this context: under (significant)
constraints the derived Fukaya category admits naturally a TPC
\zhnote{refinement}, and this setting is ideal to approach a variety
of quantitative questions typical for symplectic topology.

We begin in \S\ref{subsec-symp} with the statements of the main
symplectic applications in the paper, Theorems \ref{cor:Fuk-cat-var},
\ref{thm:appl-sympl}, and Corollary \ref{cor:gl-metric}.  To prove
these statements we first fix in~\S\ref{sec:FiltFuk} the basics of
filtered $A_{\infty}$-categories and associated \zhnote{TPCs,} and we
then discuss basic notions relative to filtered Floer theory. We
describe how to proceed from Floer chain complexes to the Fukaya
category.  However, for technical reasons the construction leads
\zhnote{only to} a {\em weakly} filtered $A_{\infty}$-category. In
\S\ref{sb:ffuk} we show that under certain restrictive conditions this
construction can be adjusted to obtain a \zhnote{genuinely} filtered
$A_{\infty}$-category. The main technical result of Chapter
\ref{s:tpfc} \zhnote{appears in} Theorem \ref{t:fil-fuk}. The model
for the Fukaya category that we construct in this case is based on
clusters of punctured disks.  While similar models have appeared
before in the literature we include enough details to justify the
control of filtrations. In section \S\ref{sec:all-symp} we prove the
statements from \S\ref{subsec-symp}. In particular, we construct the
metrics on the spaces of Lagrangians that were \zhnote{announced} in
the introduction of the paper.  The TPC formalism was inspired by
earlier work on Lagrangian cobordism and it is useful to see how
weighted triangles and operations with them appear geometrically in
the cobordism setting. This is discussed in \S\ref{sec:geom-TPC}
together with some other geometric illustrations of some of the
statements in \S\ref{subsec-symp}.

\section{Main symplectic topology applications} \label{subsec-symp}
Let
$(X, \omega = d \lambda)$ be a Liouville manifold (i.e.~an exact
symplectic manifold, with a prescribed primitive $\lambda$ of the
symplectic structure $\omega$, and such that $X$ is symplectically
convex at infinity with respect to these structures). We will 
work here with pairs $L = (\bar{L}, h_{L})$ consisting of a closed
oriented exact Lagrangian \zjr{submanifold} $\bar{L} \subset X$ equipped
with a function $h_L: \bar{L} \longrightarrow \mathbb{R}$ that is a
primitive of $\lambda|_{\bar{L}}$, i.e. \zhnote{$dh_{L} =
\lambda|_{\bar{L}}$}. We will refer to such a pair $L$ as a marked
Lagrangian submanifold and to $\bar{L}$ as its underlying Lagrangian.

\

Fix a collection of marked Lagrangians $\mathcal{X}$ in $X$. We assume
that $\mathcal{X}$ is closed under 
shifts of the primitives, namely if $L = (\bar{L}, h_L)$ is in
$\mathcal{X}$ then for every $r \in \mathbb{R}$ and
$k \in \mathbb{Z}$, the marked Lagrangian
$\Sigma^r L := (\bar{L}, h_L+r)$ is also in $\mathcal{X}$.  We will
also assume that our marked Lagrangians are graded - in a sense
recalled in \S\ref{sbsb:grading}. If we need to make the grading
explicit we write $L=(\bar{L},h_{L},\theta_{L})$ and we assume the
family $\mathcal{X}$ also closed with respect to translating the
grading $L[k]=(\bar{L},h_{L},\theta_{L}-k)$.

\

Denote by $\bar{\mathcal{X}} = \{ \bar{L} \mid L \in \mathcal{X}\}$
the collection of underlying Lagrangian submanifolds corresponding to
the marked Lagrangians in $\mathcal{X}$.  We will assume that the
family $\bar{\mathcal{X}}$ is finite and that its elements are in
\zhnote{general} position in the sense that any two distinct
Lagrangians $L', L'' \in \bar{\mathcal{X}}$ intersect transversely and
for every three distinct Lagrangians
$\bar{L}_0, \bar{L}_1, \bar{L}_2 \in \bar{\mathcal{X}}$ we have
\zhnote{$\bar{L}_0 \cap \bar{L}_1 \cap \bar{L}_2 = \emptyset$}.

\

\pbred{As earlier in the paper, algebraic considerations can be done
  over an arbitrary field $\k$.}  \zhnote{However, without additional
  assumptions on our Lagrangians, Floer theory works only over
  $\k = \Z_2$. We will therefore assume $\k = \Z_2$, but continue to
  denote the base field by $\k$, to indicate that under additional
  assumptions, our theory is expected to work over an arbitrary field
  $\k$.} The marked Lagrangians in $X$ are the objects of an
$A_{\infty}$-category, the Fukaya category $\fuk(X)$ of $X$,
constructed as in Seidel's book \cite{Seidel}.  The associated derived
Fukaya category is denoted by $D\fuk(X)$. Its objects are the
$A_{\infty}$-modules over $\fuk(X)$ that belong to the triangulated
completion of the Yoneda $A_{\infty}$-modules, $\mathcal{Y}(L)$, where
$L$ is a marked Lagrangian. \zhnote{We denote} by $\fuk(\mathcal{X})$
the $A_{\infty}$-subcategory of $\fuk(X)$ with objects the Lagrangians
in $\mathcal{X}$, \zhnote{and by}
$j_{\mathcal{X}}: \fuk(\mathcal{X}) \to \fuk(X)$ the inclusion.  There
are two Yoneda type modules associated to the elements of
$\mathcal{X}$: over the category $\fuk(X)$ and over the smaller
category $\fuk(\mathcal{X})$. The two are related by applying the
pull-back $j_{\mathcal{X}}^{\ast}(-)$ and thus will be generally
denoted by the same symbol.

We denote by $D\fuk(\mathcal{X})$ the associated derived category,
consisting this time of modules over $\fuk(\mathcal{X})$ that belong
to the triangulated completion of the Yoneda modules of the elements
of $\mathcal{X}$.  We emphasize that, with the terminology used in
this paper, a family \zhnote{$\mathcal{Z}$} of objects in a
triangulated category $\mathcal{C}$ is a system of generators of
$\mathcal{C}$ if the \zhnote{triangulated} envelope of
\zhnote{$\mathcal{Z}$} in $\mathcal{C}$ equals $\mathcal{C}$. In
particular, the Yoneda modules of the elements of $\mathcal{X}$ form a
system of generators of $D\fuk(\mathcal{X})$.

\

The following consequence of Theorem \ref{t:fil-fuk} is sufficiently
significant to formulate apart:
\begin{thm}\label{cor:Fuk-cat-var} There exists 
  a triangulated persistence category $\mathcal{C}\fuk(\mathcal{X})$,
  independent \zhnote{up to TPC equivalence} of the data used in its
  construction, such that:
\begin{itemize}
\item[(i)] For each $L,L'\in \mathcal{X}$ there is \zhnote{a
    canonical} isomorphism
  \zhnote{$\mor_{\mathcal{C}\fuk(\mathcal{X})_{\infty}}(L,L')\cong
    HF^0(L,L')$} where \zhnote{$HF^0(-,-)$} is Floer homology
  \pbred{in cohomological degree $0$.}
\item[(ii)] $\mathcal{C}\fuk(\mathcal{X})_{\infty}$ is triangulated
  equivalent to $D\fuk(\mathcal{X})$.
\item[(iii)] If the family $\mathcal{X}$ generates $D\fuk(X)$, then
  for each marked Lagrangian $N$ that intersects transversely the
  family $\mathcal{X}$, the pull-back
  $j_{\mathcal{X}}^{\ast}\mathcal{Y}(N)$ of $\mathcal{Y}(N)$ - the
  Yoneda module of $N$ \zjnote{over} $\fuk(X)$ (defined with a
  convenient choice of perturbation data) - \zhnote{is
    quasi-isomorphic to an object in}
  $\mathcal{C}\fuk(\mathcal{X})_{\infty}$.
\end{itemize}
\end{thm}

\zhnote{We call $\mathcal{C}\fuk(\mathcal{X})$ the {\em triangulated
    persistence Fukaya category} associated to $\mathcal{X}$}.

\medskip

\zjr{Here we emphasize that the construction of $\fuk(\mathcal{X})$
  always depends on some perturbation data $\mathscr{P}$, where more
  precisely $\fuk(\mathcal{X})$ is denoted by
  $\fuk(\mathcal{X}; \mathscr{P})$. Theorem \ref{t:fil-fuk} guarantees
  that there always exists perturbation data $\mathscr{P}$ such that
  $\fuk(\mathcal{X}; \mathscr{P})$ is a strict unital
  $A_{\infty}$-category,} \pbrev{together with filtered
  $A_{\infty}$-functors}
\zjr{$\mathcal{F}^{\mathscr{P}_1, \mathscr{P}_0} : \fuk(\mathcal{X};
  \mathscr{P}_0) \longrightarrow \fuk(\mathcal{X}; \mathscr{P}_1),$
  when changing the perturbation data from $\mathscr{P}_0$ to
  $\mathscr{P}_1$. An essential part of the proof of Theorem
  \ref{cor:Fuk-cat-var} (see \S\ref{subsubsec:proof-Cor}) shows that,
  the resulting triangulated persistence Fukaya category
  $\mathcal{C}\fuk(\mathcal{X}; \mathscr{P})$, constructed from
  $\fuk(\mathcal{X}; \mathscr{P})$, is in fact independent of the
  perturbation data up to} \pbrev{a TPC equivalence.} \zjr{This
  justifies the notation $\mathcal{C}\fuk(\mathcal{X})$ above}
\pbrev{without any reference to $\mathscr{P}$.}

\medskip

Point \zhnote{(ii)} of \zhnote{Theorem \ref{cor:Fuk-cat-var}} implies
that if $\mathcal{X}$ generates $D\fuk(X)$, then
$\mathcal{C}\fuk(\mathcal{X})_{\infty}$ is equivalent to
$D\fuk(X)$. Point \zhnote{(iii) of Theorem \ref{cor:Fuk-cat-var}}
gives a bit more information and shows that one can use measurements
in $\mathcal{C}\fuk(\mathcal{X})$ to study Lagrangians that do not
\zhnote{necessarily} belong to the finite family
$\bar{\mathcal{X}}$. Nonetheless, it remains that the requirement that
the family $\bar{\mathcal{X}}$ be finite is highly constraining.  It
is expected that this requirement can be dropped by using a more
involved construction in place of the one used in the proof of Theorem
\ref{t:fil-fuk}.

\begin{rem} \label{rem:eq-nat-not} As stated, Theorem
  \ref{cor:Fuk-cat-var} identifies $\mathcal{C}\fuk(\mathcal{X})$ up
  to TPC equivalence (see Definition \ref{def:TPC equivalences}) but,
  while this equivalence is expected to be canonical, our methods do
  not quite give that. \zhnote{Still}, the equivalences that appear
  here are not completely \zhnote{arbitrary. For example, their
    mapping on objects leaves the elements of $\mathcal{X}$ fixed.}
  \zhnote{S}ee Theorem \ref{t:fil-fuk} for more details.
\end{rem}

The next result in this section will be formulated in terms of this
TPC, $\mathcal{C}\fuk(\mathcal{X})$, and will involve a notion of
relative Gromov width that first appeared in \cite{Ba-Co} (see also
\cite{Bi-Co-Sh:LagrSh}). Assume that $L$, and $L'$ are two
Lagrangians, both possibly immersed.  We define
\zhnote{\begin{equation}\label{eq:ball-sq} \begin{array}{l} \delta(L;
      L')=\sup\left\{ \pi r^{2} \ \bigg| \begin{array}{l} \exists\ e:
          (B(r),\omega_{0})\to (X,\omega) \ \mathrm{symplectic\
            embedding}, \\ \mbox{such that
            $e^{-1}(L)= \R B(r), \ e(B(r))\cap
            L'=\emptyset$} \end{array} \right\},
\end{array}
\end{equation}}
where $(B(r),\omega_{0})$ is the standard closed ball of radius $r$
in $(\R^{2n}, \omega_{0})$ and
\pbred{$\R B(r)=(\R^{n} \times \{0\}) \cap B(r)$} is its real part. 
A related measurement reflects the ``quality'' of the intersection
points between $L$ and $L'$, relative to another subset. \zhnote{Assume} that $L$ and $L'$ \zhnote{intersect} transversely and \zhnote{let $A\subset X$ be a subset.} \zhnote{We define}: 

\zhnote{\begin{equation} \label{dfn-delta-cap}
\delta^{\cap}(L, L';A)=\sup\left\{ \pi r^{2} \ \Bigg|\ \begin{array}{l}  \forall x\in L\cap L',\,  \exists\  e_{x}: (B(r),\omega_{0})\to (X,\omega)  \ \mathrm{symp.\ emb.},  \\  \mbox{s.t.}\, e_{x}(0)=x,\ e_{x}^{-1}(L)= \R B(r), \ e_{x}^{-1}(L')=i \R B(r),\\
e_{x'}(B(r)) \cap e_{x''}(B(r)) = \emptyset \,\,\mbox{whenever $x' \neq x''$},\\
\mbox{and moreover} \, \forall x \in L\cap L', \, e_{x}(B(r))\cap A=\emptyset 
 \end{array} \right\} ~.~
\end{equation}}
\pbred{Here $i \mathbb{R} B(r) := (\{0\} \times \mathbb{R}^n) \cap B(r)$ is
  the ``imaginary'' part of the ball $B(r)$.}

We will also need the spectral distance between two marked Lagrangians
$L$ and $L'$. We assume that $L$ is Hamiltonian isotopic to $L'$. In
this case the Floer homology $HF(L,L')$ is isomorphic to the singular
homology $H_{\ast}(L;\k)$ of $L$ and there is a canonical class
$o_{L,L'}\in HF(L,L')$ corresponding to the fundamental class in
$H_{\ast}(L;\k)$. There is also a second class $pt_{L,L'}\in HF(L,L')$
that corresponds to the point class in $H_{\ast}(L;\k)$.  Assume
\pbred{further that} $L,L'\in \mathcal{X}$. In this case, given
\pbred{point~(i)} \zhnote{of Theorem \ref{cor:Fuk-cat-var}}, \pbred{we
  have:}
\pbred{$$HF^0(L,L')=\left\{[f] \in
    \mor_{{\mathcal{C}\fuk(\mathcal{X})}_{\infty}}(L,L')\,|\, f \in
    \mor_{\mathcal{C}\fuk(\mathcal{X})}(L,L')\right\}~.~$$}
\pbred{Therefore} these classes in $HF(L,L')$ have spectral numbers
$\sigma(-)$ as defined in (\ref{eq:spec}).  We \zhnote{define}
$$\sigma(L,L')=\sigma(o_{L,L'})-\sigma(pt_{L,L'})$$
\zhnote{We extend the definition of $\sigma$ to the case when $L'$ is
  not Hamiltonian isotopic to $L$ by setting $\sigma(L,L')=\infty$ in
  this case.}
\begin{rem}
  It is easily seen that this definition coincides with
  \pbred{previous versions of spectral invariants} introduced by
  Viterbo, Schwarz, \zhnote{Oh,} and later adjusted to the Lagrangian
  setting.
\end{rem}
 
Pick a family $\mathcal{F}\subset \mathcal{X}$ that is invariant with
respect to shift and translation. Fix an \zhnote{admissible}
\zhnote{perturbation} data $\mathscr{P}$ and an associated
\zhnote{triangulated} persistence category
$\mathcal{C}\fuk(\mathcal{X};\mathscr{P})$.  Consider the shift
invariant, persistence, fragmentation \zhnote{pseudo-}metric
$\widehat{\bar{d}^{\mathcal{F}}}(-,-)$ associated to the persistence
weight $\bar{w}$ on
$\mathcal{C}\fuk(\mathcal{X};\mathscr{P})_{\infty}$, as described in
\S\ref{subsubsec:prop-fr} and (\ref{eq:inv-shift-metr}).  Each such
pseudo-metric is defined on the objects of
$\mathcal{C}\fuk(\mathcal{X};\mathscr{P})$, which contain the Yoneda
modules of the Lagrangians in $\mathcal{X}$ but also additional
$A_{\infty}$-modules.
 
We now define a pseudo-metric on $\mathcal{X}$
by:
$$D^{\mathcal{F}}(L,L')=
\widehat{\bar{d}^{\mathcal{F}}}(\mathcal{Y}(L),\mathcal{Y}(L'))\ ,
L,L'\in\mathcal{X}$$

In case $\mathcal{F}=\{0\}$ we write $D(-,-)=D^{\{0\}}(-,-)$.  This is
an upper bound for all the other fragmentation metrics
$D^{\mathcal{F}}$.  There is a slight abuse in notation here because
the definition of $D^{\mathcal{F}}$ depends implicitly on the
perturbation data $\mathscr{P}$ but this will be resolved in the next
result.

\begin{thm}\label{thm:appl-sympl} \zhnote{Let
    $\mathcal F \subset \mathcal X$.} In the setting above, the
  pseudo-metrics $D^{\mathcal{F}}$ are independent of the perturbation
  data $\mathscr{P}$ used \zhnote{for} their
  definitions. \zhnote{Moreover:}
  \begin{itemize}
  \item[(i)] (spectrality) \zhnote{Assume that the Lagrangians in
      $\mathcal X$ are graded, then} for any $L, L'\in \mathcal{X}$ we
    have
    $$D(L,L')\leq 4\ \sigma (L,L')~.~$$
  \item[(ii)] (non-degeneracy) \zhnote{For all $L,L'\in \mathcal{X}$}, 
    $$\frac{\delta (L; L'\cup _{F\in\mathcal{F}} F)}{8}
    \leq D^{\mathcal{F}}(L,L')~.~$$
  \item[(iii)] (persistence of intersections) Assume that
    $L, L', N \in \mathcal{X}$, $L'\not \in \mathcal{F}$. If
    \pbrev{$$D^{\mathcal{F}}(L,L') < \frac{1}{16}\ \delta^{\cap}(N,
      L';\cup_{F\in\mathcal{F}}F)~,~$$} then
    $$\#(L\cap N)\geq \# (L'\cap N)~.~$$
  \item[(iv)] (finiteness) If the family $\mathcal{F}$ generates
    $D\fuk(X)$, then the pseudo-metric $D^{\mathcal{F}}$ is finite.
  \end{itemize}
\end{thm}

Compared to other metrics and measurements on spaces of Lagrangians
the key novelty here is that properties \zhnote{(i), (ii), (iii) and
  (iv)} \pbred{are valid for the same metric}.

\begin{rem}\label{rem:symplcom} (a) Point \zhnote{(i) of Theorem
    \ref{thm:appl-sympl}} shows that all the fragmentation
  pseudo-metrics $D^{\mathcal{F}}$ are dominated by the spectral
  metric. In previous results involving metrics on spaces of
  Lagrangians, such as those based on the shadows of cobordisms in
  \cite{Bi-Co-Sh:LagrSh}, the best one could do was to
  \zhnote{establish upper bounds on the metrics that are} generally
  much harder to \zhnote{estimate}, such as the Hofer
  distance. Further consequences of this point will be discussed in
  \S\ref{subs:proof}.

  (b) \zhnote{Point (ii) of Theorem \ref{thm:appl-sympl}} can be read
  as a typical non-squeezing type result: embeddings of large
  symplectic balls, as in the definition of $\delta$, are obstructed
  by $D^{\mathcal{F}}$. Conversely, this point implies that if
  $D^{\mathcal{F}}(L,L')=0$, then
  $L\subset L'\cup_{F\in \mathcal{F}}F$. As a result, suppose that we
  fix a second family $\mathcal{F}'\subset \mathcal{X}$, obtained
  through a small Hamiltonian perturbation of the elements of
  $\mathcal{F}$. One can then consider
  $D^{\mathcal{F},\mathcal{F}'}=\max\{ D^{\mathcal{F}},
  D^{\mathcal{F}'}\}$ as in (\ref{eq:mixing1}). This pseudo-metric is
  non-degenerate on $\bar{\mathcal{X}}$ in the sense that
  $D^{\mathcal{F},\mathcal{F}'}(L,L')=0$ iff $\bar{L}=\bar{L}'$ (in
  other words the two \zhnote{underlying} Lagrangians involved
  coincide\zhnote{; obviously,} the markings \zhnote{may}
  differ). This type of argument appeared first in
  \cite{Bi-Co-Sh:LagrSh}. Various forms of the inequality \zhnote{from
    point (ii)} appeared earlier in the literature\zhnote{,} in
  particular cases such as when \zhnote{$\mathcal{F}=\{0\}$} and the
  metric involved is the Hofer metric \zhnote{(see
    e.g.~\cite{Ba-Co})}.  However, it is useful to note that the
  pseudo-metrics $D^{\mathcal{F}}$ are, in general, smaller compared
  to the metrics in these earlier references. \zhnote{Note also that
    even for $\mathcal F = \{0\}$, the inequalities} obtained by
  combining \zhnote{(i) and (ii) appear new}.

  (c) \zhnote{We emphasize that in the point (iii)} of Theorem
  \ref{thm:appl-sympl} the two Lagrangians $L$ and $L'$ \zhnote{are
    allowed to} be very different. \zhnote{F}or instance they can be
  in different smooth isotopy classes or even have different
  homeomorphism types and still $D^{\mathcal{F}}(L, L')$ can be finite
  - this point is reinforced by the last part of the theorem.
  Therefore, this result shows a form of rigidity of Lagrangian
  intersections, for perturbations that are small in this metric
  $D^{\mathcal{F}}$, but that can be very big \zhnote{(infinite even)}
  in other metrics.  The result extends earlier persistence type
  statements in Morse and Floer theory \zhnote{(one of the earliest
    examples appearing in \cite{Co-Ra})} most of them expressed in
  terms of the Hofer distance that is much larger than
  $D^{\mathcal{F}}$.  Again, there is considerable interest to work
  with the algebraic metrics introduced here because for other
  metrics, such as the shadow metrics based on Lagrangian cobordism,
  the finiteness result at point (iv) is \zhnote{not known to hold}.

  (d) The fact that we have
  $j_{\mathcal{X}}^{\ast}\mathcal{Y}(N)\in {\rm Obj}
  (\mathcal{C}\fuk(\mathcal{X})_{\infty})$ for each marked Lagrangian
  $N$ that intersects transversely the elements of $\mathcal{X}$, as
  in Corollary \ref{cor:Fuk-cat-var}, implies that we can define a
  pseudo-metric on the space of all marked Lagrangians in $X$ by:

  $$\Delta ^{\mathcal{F}}(N,N')=\limsup_{\epsilon\to 0}
  \ \widehat{\bar{d}^{\mathcal{F}}}(
  j_{\mathcal{X}}^{\ast}\mathcal{Y}(N_{\epsilon}),
  j_{\mathcal{X}}^{\ast}\mathcal{Y}(N'_{\epsilon}))~.~$$ Where
  $N_{\epsilon}, N'_{\epsilon}$ are $\epsilon$-small \zhnote{(in the
    Hofer metric)} Hamiltonian perturbations of $N$, \zhnote{$N'$
    respectively}, that are both transverse to the elements of
  $\mathcal{X}$. This pseudo-metric is in general degenerate as it
  does not ``see'' differences between $N$ and $N'$ away from the
  elements of $\mathcal{X}$.

  (e) The constants providing the various bounds in Theorem
  \ref{thm:appl-sympl} are very rough and can be improved in some
  cases but we will not pursue these question here.
\end{rem}

\

We will prove a consequence of Theorem \ref{thm:appl-sympl} which is
deduced by studying how the pseudo-metrics $D^{\mathcal{F}}$ change
when the underlying set of marked Lagrangians $\mathcal{X}$ changes.

\

To state this consequence we need a global finiteness type assumption
on our Liouville manifold $(X,\omega)$.  To formulate it, we denote by
$\mathcal{L}ag(X)$ the set of exact, compact, graded, embedded
Lagrangians in $X$ and we denote by $\mathcal{L}ag(X)'$ the marked,
exact Lagrangians in $X$ (these are the elements of $\mathcal{L}ag(X)$
but with fixed primitives and grading choices).  As before, the Fukaya
category $D\fuk(X)$ is the derived category of the
\zhnote{$A_{\infty}$-category} with objects the elements in
$\mathcal{L}ag(X)'$.  The category $D\fuk(X)$ is constructed as in
\cite{Seidel}. In particular, the perturbation data only depends on
the elements in $\mathcal{L}ag(X)$, and not on the choices of
primitives and grading.

\begin{dfn}\label{def:der-finite} \zhnote{Let $(X, \omega)$ be a
    Liouville manifold.} The {\em Fukaya rank} of $(X,\omega)$,
  \pbred{$\rank\, \fuk(X,\omega)$}, is the minimal cardinality of a
  family of Lagrangians $\bar{\mathcal{F}}\subset \mathcal{L}ag(X)$
  such that the corresponding family of marked Lagrangians
  $\mathcal{F}\subset \mathcal{L}ag(X)'$, obtained from
  \pbred{$\bar{\mathcal{F}}$} by adding all possible
  \zhnote{translates of the objects in $\mathcal F$ (in terms of
    grading),} generates $D\fuk(X)$. \pbred{(Note that the primitives
    have no effect here since we are talking about generators in a
    non-filtered setting.)}
\end{dfn}

We emphasize here that ``generating'' has the meaning of
\zhnote{triangulated} generation, as everywhere else in this
paper. The \pbred{$\rank\, {\mathcal{D}}$} can be defined similarly
for any triangulated category $\mathcal{D}$.  The terminology is
justified by the fact that this quantity is an upper bound for the
rank of the Grothendieck group $K(\mathcal{D})$.

\

\begin{cor}\label{cor:gl-metric} Let $(X,\omega)$ be a Liouville manifold
  and assume that \pbred{$\rank\, \fuk(X,\omega)< \infty$}. Fix a
  family of generators $\mathcal{F}\subset \mathcal{L}ag(X)'$,
  invariant \zhnote{under shifts and translations}, and such that the
  corresponding family $\bar{\mathcal{F}}\subset \mathcal{L}ag(X)$
  obtained by forgetting the markings is finite and is in
  \zhnote{general} position \zhnote{(in the sense defined at the
    beginning of \S\ref{subsec-symp})}.

\zhnote{Then the} set $\mathcal{L}ag(X)$ carries a finite, 
pseudo-metric $\mathcal{D}^{\mathcal{F}}$ such that:
\begin{itemize}
\item[(i)] (spectrality) For any  $L, L'\in \mathcal{L}ag(X)$ we have 
$$\mathcal{D}^{\mathcal{F}}(L,L')\leq 4\ \sigma (L,L')~.~$$
\item[(ii)] (non-degeneracy) If $L,L'\in \mathcal{L}ag(X)$, then
$$\frac{\delta (L; L'\cup _{F\in\mathcal{F}} F)}{8}\leq \mathcal{D}^{\mathcal{F}}(L,L')~.~$$  
\item[(iii)] (persistence of intersections) Assume that $L, L', N \in \lag ag(X)$ \zjr{are in general position and} $L'\not \in \mathcal{F}$. If 
\zjr{\zhnote{$$\mathcal{D}^{\mathcal{F}}(L,L') < \frac{1}{16}\ \delta^{\cap}(N, L';\cup_{F\in\mathcal{F}}F)~,~$$}}
then $$\#(L\cap N)\geq \# (L'\cap N)~.~$$
\end{itemize}
In particular, if $\mathcal{F}'$ is another family obtained from $\mathcal{F}$ by generic Hamiltonian perturbations of the elements of $\bar{\mathcal{F}}$,
then 
$$\mathcal{D}^{\mathcal{F},\mathcal{F}'}=\max\{\mathcal{D}^{\mathcal{F}}, \mathcal{D}^{\mathcal{F}'}\}$$
is a \zhnote{finite and non-degenerate} metric on $\mathcal{L}ag(X)$ that satisfies the properties (i), (ii), (iii) above.

Thus, under the hypothesis \pbred{$\rank\, \fuk(X,\omega)<\infty$},
the set $\mathcal{L}ag(M)$ has a metric space structure with respect
to a metric satisfying the properties (i), (ii), (iii).
\end{cor}


\section{Filtrations in Floer homology and \zhnote{Fukaya categories}.}\label{sec:FiltFuk}

\subsection{Filtered $A_{\infty}$-categories and their associated
  TPC's} \label{sb:faitpc}

A filtered $A_{\infty}$-category $\mathcal{A}$ is an
$A_{\infty}$-category over a given base field $\k$, such that the
\zhnote{spaces of morphisms} $\Mor_{\mathcal{A}}(X,Y)$ between every two objects
$X,Y$ are filtered (with increasing filtrations) and {\em all} the
composition maps $\mu_d$, $d \geq 1$, respect the filtrations. We
endow $\Mor_{\mathcal{A}}(X,Y)$ with the differential $\mu_1$ and view
it as a filtered chain complex. We denote by
$\Mor_{\mathcal{A}}^s(X,Y)$, $s \in \mathbb{R}$, the level-$s$
filtration subcomplex of $\Mor_{\mathcal{A}}(X,Y)$. We refer the
reader to~\cite{Bi-Co-Sh:LagrSh} for more details on filtered
$A_{\infty}$-categories. Note however that in~\cite{Bi-Co-Sh:LagrSh}
the theory is developed for the more general case of {\em weakly}
filtered $A_{\infty}$-categories (the ``genuinely'' filtered case is
obtained from the weakly filtered one by assuming the so called
``discrepancies'' of $\mathcal{A}$, defined in~\cite{Bi-Co-Sh:LagrSh},
to vanish). Of course, the $A_{\infty}$-considerations here are very
similar to those for dg-categories in \S\ref{subsec:dg}.

For simplicity we will make three further assumptions on our filtered
$A_{\infty}$-categories. The first is that $\mathcal{A}$ is strictly
unital with the units lying in persistence level $0$. The second one
is that for every two objects $X, Y \in \Ob(\mathcal{A})$, the space
$\Mor_{\mathcal{A}}(X,Y)$ is finite dimensional over $\k$.
The third assumption is that $\mathcal{A}$ is complete with respect to
persistence shifts in the sense that we have a shift ``functor''
 which consists of a family of $A_{\infty}$-functors
$\Sigma = \{\Sigma^r: \mathcal{A} \longrightarrow \mathcal{A}, r \in
\mathbb{R}\}$ whose members satisfy the following conditions:
\begin{enumerate}
\item $\Sigma^r$ is strictly unital and the higher components
  $(\Sigma^r)_d$, $d \geq 2$, of $\Sigma^r$ all vanish.
\item $\Sigma^0 = \mathds{1}$, $\Sigma^s \circ \Sigma^t = \Sigma^{s+t}$.
\item We are given prescribed identifications
  $\Mor_{\mathcal{A}}^s(\Sigma^r X,Y) \cong
  \Mor_{\mathcal{A}}^{s+r}(X,Y)$ that are compatible with the
  inclusions
  $\Mor_{\mathcal{A}}^{\alpha}(X,Y) \subset
  \Mor_{\mathcal{A}}^{\beta}(X,Y)$ for $\alpha \leq \beta$. These
  identifications are considered as part of the structure of the shift
  functor $\Sigma$.
\end{enumerate}

The assumption that $\mathcal{A}$ is complete with respect to shifts
is merely a matter of convenience in the sense that it is not
essential to impose this condition in advance. Indeed any filtered
$A_{\infty}$-category (satisfying all the above assumptions except the
one on completeness with respect to shifts) can be completed with
respect to shifts by adding suitable objects that will play the role
of the shifted $\Sigma^r X$ objects and then defining the functors
$\Sigma^r$ accordingly. See again \S \ref{subsec:dg} for the similar
case of filtered dg-categories.

We will generally use homological conventions in the context of
\pbred{$A_{\infty}$-categories}, however for compatibility with the
literature we will generally use cohomological grading. Whenever this
is the case, we will denote the cohomological degrees by superscripts
(e.g.~$H^0$ will stand for the homology in cohomological degree $0$,
the units will be assumed to be in cohomological degree $0$ and so
on).

Given a filtered $A_{\infty}$-category $\mathcal{A}$ one can form the category
$Tw \mathcal{A}$ of twisted complexes over $\mathcal{A}$ which is
itself a filtered $A_{\infty}$-category (satisfying all the additional
assumptions mentioned earlier). This can be done by following the
construction in~\cite[Chapter I, Section~3(l)]{Seidel}
and extending the filtrations from $\mathcal{A}$ to $Tw \mathcal{A}$
in the obvious way. The construction of the filtered $Tw \mathcal{A}$
in the case of dg-categories has been worked out in detail
in~\S \ref{subsec:dg}, and the $A_{\infty}$-case is very
similar. There is a bit of abuse of notation in writing
$Tw \mathcal{A}$, since the latter category carries additional
structures (namely filtrations and shift functor) than the unfiltered
category of twisted complexes which is denoted in the literature by
the same notation $Tw \mathcal{A}$.

The filtered $A_{\infty}$-category $\mathcal{A}$ embeds into
$Tw \mathcal{A}$ in an obvious way, the embedding being a filtered
$A_{\infty}$-functor which is full and faithful (on the chain level).
Moreover, $Tw \mathcal{A}$ is pre-triangulated in the filtered sense
(which in particular means that it is closed under formation of
filtered mapping cones). It follows that the \zhnote{homological} category
$H^0(Tw \mathcal{A})$ is a TPC that contains the \zhnote{homological}
persistence category $H^0(\mathcal{A})$ of $\mathcal{A}$. 

Another TPC associated to the $A_{\infty}$-category $\mathcal{A}$ is
provided by the category $Fmod(\mathcal{A})$ of filtered
$A_{\infty}$-modules over $\mathcal{A}$. Weakly filtered modules are
\zhnote{defined} in \cite{Bi-Co-Sh:LagrSh} and the filtered
definitions correspond to all discrepancies being $0$. We will only
consider strictly unital modules here. There is a natural shift
functor on this category
$\Sigma: (\R,+) \to \mathrm{End}(Fmod(\mathcal{A}))$. Given
$r \in \mathbb{R}$ and a module $\mathcal{M} \in Fmod(\mathcal{A})$ we
define the filtered module $\Sigma^r \mathcal{M}$ by
$(\Sigma^r \mathcal{M})^{\leq \alpha}(N) = \mathcal{M}^{\alpha -
  r}(N)$, endowed with the same \zhnote{$\mu_d$}-operations as
$\mathcal{M}$.
\begin{rem}In the cases of interest to us, namely the Fukaya category,
  this shift functor on $Fmod(\mathcal{A})$ is compatible with a shift
  functor on \zhnote{$\mathcal{A}$}.
\end{rem}

The category $Fmod(\mathcal{A})$ is in fact a filtered dg-category in
the sense of \S\ref{subsec:dg} and it is pre-triangulated. Thus
$H^{0}(Fmod(\mathcal{A}))$ is a TPC\zhnote{, by Corollary
  \ref{cor:pre-tr}}.  Of more interest to us is a subcategory of
$Fmod(\mathcal{A})$. First notice that, because \zhnote{$\mathcal{A}$
  is filtered}, the Yoneda \zhnote{functor}
\pbred{$\mathcal{Y}:\mathcal{A}\to Fmod(\mathcal{A})$} is filtered
\pbred{too}. \zhnote{Moreover, our assumption of strict unitality of
  $\mathcal A$ implies that} \pbred{$\mathcal{Y}$} \zhnote{is
  homologically full and faithful.}  \zhnote{Furthermore,}
\pbred{there exists a canonical map}
\pbred{$$\lambda: \mathcal{M}(X)\to
  \Mor_{Fmod}(\mathcal{Y}(X),\mathcal{M})$$} \pbred{for all
  $X \in {\rm Obj}(\mathcal{A})$ and
  $\mathcal{M}\in {\rm Obj}(Fmod(\mathcal{A}))$, as defined
  in~\cite[Chapter~1, Section~(1l)]{Seidel}.} \pbred{Standard
  arguments show that $\lambda$ is a {\em filtered quasi-isomorphism}
  in the sense that it is filtered and induces an isomorphism between
  the persistence homologies of its domain and target filtered chain
  complexes.}

We consider \zhnote{now} the pre-triangulated closure
$\mathcal{A}^{\#}$ of the Yoneda modules and their shifts: this is a
\zhnote{full subcategory} of $Fmod(\mathcal{A})$ that has as objects
all the iterated cones, over filtration preserving morphisms, of
shifts of Yoneda modules (thus of modules of the form
$\Sigma^{r}\mathcal{Y}(X)$).

Finally, we denote by $\mathcal{A}^{\nabla}$ the smallest \zhnote{full
  subcategory} of $Fmod(\mathcal{A})$ that contains $\mathcal{A}^{\#}$
and all the modules (and all their shifts and translates) that are
$r$-quasi-isomorphic to objects in $\mathcal{A}^{\#}$ \zhnote{for
  some} $r\in [0,\infty)$.  \zhnote{Here, a} module $\mathcal{M}$ is
\zhnote{called} $r$-quasi-isomorphic to $\mathcal{N}$ if, in
\zhnote{$H^{0}(Fmod(\A))$}, there is an $r$-isomorphism
$\mathcal{M}\to \mathcal{N}$.
 
It is easy to see that $\mathcal{A}^{\nabla}$ remains
pre-triangulated, carries the shift functor induced from
$Fmod(\mathcal{A})$ and thus \zhnote{$H^{0}(\A^{\nabla})$} is a TPC.

\subsection{Persistence Floer homology} \label{sb:fil-fuk} 

\subsubsection{Filtered Floer complexes} \label{sbsb:FCF} Given a pair
of marked Lagrangians $L_0, L_1$ as above and a choice of Floer data
$\mathscr{D}_{L_0,L_1} = (H_{L_0,L_1}, J_{L_0, L_1})$ which consists
of a (possibly time-dependent) Hamiltonian function and a choice of a
compatible (time-dependent) almost complex structure we can form the
Floer complex $CF(L_0, L_1; \mathscr{D}_{L_0,L_1})$. This is a
$\mathbb{Z}_2$-graded chain complex \zhnote{(recall our Lagrangians are assumed to be oriented). It is} generated by the Hamiltonian
chords $x:[0,1] \longrightarrow X$ of $H_{L_0,L_1}$ with end points on
the two Lagrangians, namely $x(0) \in L_0$, $x(1) \in L_1$. For
simplicity we work here with coefficients in $\mathbb{Z}_2$.

Moreover, $CF(L_0, L_1; \mathscr{D}_{L_0,L_1})$ is a filtered chain
complex, where the filtration function is given by the action
functional. More precisely, if
$x \in CF(L_0, L_1; \mathscr{D}_{L_0,L_1})$ is a generator (i.e.~a
Hamiltonian chord), its action is defined by
\begin{equation} \label{symp-act-1} \mathcal A(x) := \int_0^1
  H_{L_0,L_1}(t, x(t))dt - \int_0^1 \lambda(\dot{x}(t)) dt +
  h_{L_1}(x(1)) - h_{L_0}(x(0)).
\end{equation}

\begin{rem} \label{r:L0L1-transverse} In case $L_0$ and $L_1$
  intersect transversely and $H_{L_0,L_1} \equiv 0$, the Hamiltonian
  chords $x$ that generate $CF(L_0, L_1; \mathscr{D}_{L_0, L_1})$ are
  just the intersection points $L_0 \cap L_1$ and the action reduces
  to
  \begin{equation} \label{eq:symp-act-2} \mathcal A(x) = 
    h_{L_1}(x) - h_{L_0}(x), \quad \forall \, x \in L_0 \cap L_1.
  \end{equation}
\end{rem}

\zhnote{Back} to the general case, the homology of the filtered chain
complex $CF(L_0, L_1; \mathscr{D}_{L_0,L_1})$ gives rise to
persistence Floer homology $HF(L_0, L_1; \mathscr{D}_{L_0,L_1})$ which
has the structure of a $\mathbb{Z}_2$-graded persistence module (over
the field $\mathbb{Z}_2$). As a vector space
$HF(L_0, L_1; \mathscr{D}_{L_0,L_1})$ is independent of the auxiliary
Floer data $\mathscr{D}_{L_0,L_1}$, however as a persistence module it
does depend on that choice. More precisely, the persistence module
structure of $HF(L_0, L_1; \mathscr{D}_{L_0,L_1})$ is independent of
the choice of the almost complex structure $J_{L_0,L_1}$ from
$\mathscr{D}_{L_0,L_1}$, however it depends strongly on the choice of
the Hamiltonian $H_{L_0, L_1}$.

\subsubsection{Grading} \label{sbsb:grading} While
$\mathbb{Z}_2$-grading is enough for \zhnote{our applications,} one
can obtain a $\mathbb{Z}$-graded theory if one makes additional
assumptions on $X$ and on the admissible class of Lagrangian
submanifolds. The simplest such conditions are the following: firstly,
we assume that $2c_1(X)=0$, where $c_1(X)$ stands for the 1'st Chern
class of the tangent bundle of $X$, viewed as a complex vector bundle
by endowing $X$ with any $\omega$-compatible almost complex structure
$J$. We now fix a nowhere vanishing quadratic complex $n$-form (where
$n = \dim_{\mathbb{C}} X$), namely a nowhere vanishing section
$\Theta$ of the bundle $\Omega^n(X, J)^{\otimes 2}$. The choice of
$\Theta$ gives rise to a global phase map
$\det^2_{\Theta}: \mathcal{L}(T(X)) \longrightarrow S^1$ defined on
the Lagrangian Grassmannian bundle $\mathcal{L}(T(X))$ of $X$
(see~\cite{Se:graded},~\cite[Chapter~2, Section~11j]{Seidel}).
\pbrev{Given a Lagrangian $\bar{L} \subset X$ denote by
  $s_{\bar{L}} : \bar{L} \longrightarrow \mathcal{L}(T(X))|_{\bar{L}}$
  its Gauss map. A Lagrangian $\bar{L}$ is said to admit a grading if
  $\det^2_{\Theta} \circ s_{\bar{L}} : \bar{L} \longrightarrow S^1$
  can be lifted to a function
  $\theta_{\bar{L}}: \bar{L} \longrightarrow \mathbb{R}$ and a choice
  of such a lift is called a grading on $\bar{L}$. In this case, by
  adding integral constants to $\theta_{\bar{L}}$ one obtains all
  possible gradings of $\bar{L}$.}

\pbrev{Gradability of Lagrangians can be rephrased in cohomological
  terms. The map
  $\det^2_{\Theta} \circ s_{\bar{L}} : \bar{L} \longrightarrow S^1$
  gives rise to a cohomology class $\mu_{\bar{L}} \in H^1(\bar{L})$
  which we call the Maslov class of $\bar{L}$. (There is a slight
  abuse of notation here since $\mu_{\bar{L}}$ actually depends on the
  homotopy class of $\Theta$.) A Lagrangian $\bar{L}$ admits a grading
  if and only if $\mu_{\bar{L}} = 0$.}

\pbrev{The relation between $\mu_{\bar{L}}$ and the more familiar
  Maslov index homomorphism
  $\mu_{X,\bar{L}}: H_2(X,\bar{L}) \longrightarrow \mathbb{Z}$ is that
  $\mu_{X,\bar{L}}(A) = \langle \mu_{\bar{L}}, \partial_* A \rangle$
  for every $A \in H_2(X, \bar{L})$, where
  $\partial_*: H_2(X,\bar{L}) \longrightarrow H_1(\bar{L})$ is the
  connecting homomorphism. Note also that if the map
  $H_1(\bar{L}) \longrightarrow H_1(X)$, induced by the inclusion
  $\bar{L} \subset X$, is trivial, then $\mu_{\bar{L}}$ is determined
  by $\mu_{X, \bar{L}}$ (hence in that case $\mu_{\bar{L}}$ is
  independent of the choice of $\Theta$). This is because
  $\mu_{X,\bar{L}} (j(B)) = 2 \langle c_1(X), B \rangle = 0$ for every
  $B \in H_2(B)$, where $j: H_2(X) \longrightarrow H_2(X,\bar{L})$ is
  the map induced by the inclusion. Therefore $\mu_{X,\bar{L}}$
  descends to $H_2(X, \bar{L}) / j(H_2(X)) \cong H_1(\bar{L})$.}

In the rest of the paper, we optionally allow for a $\Z$-graded
theory. Whenever this is wished, we will make the preceding
assumptions on $X$, fix the auxiliary structure $\Theta$, and consider
only marked Lagrangians $L$ \pbrev{that admit a grading (or
  equivalently $\mu_{\bar{L}} = 0$).} Moreover, we extend the notion
of a marked Lagrangian $L$ to include also a choice of a grading
denoted $\theta_L$, namely $L = (\bar{L}, h_L, \theta_L)$. However,
below we will mostly suppress the choice $\theta_L$ form the notation
since it will not be often explicitly used.

Given a pair of marked Lagrangians $L_0, L_1$, their grading
induces an absolute $\mathbb{Z}$-grading on $CF(L_0, L_1;
\mathscr{D}_{L_0,L_1})$, and therefore also on $HF(L_0, L_1;
\mathscr{D}_{L_0, L_1})$. The effect of translating the grading
functions on the Lagrangians is the following. Denote $L[k]
= (\bar{L}, h_L, \theta_L - k)$. Then, using cohomological and
homological grading respectively, we have: $$CF^i(L_0[k], L_1[l];
\mathscr{D}) \cong CF^{i+k-l}(L_0, L_1; \mathscr{D}), \quad
CF_j(L_0[k], L_1[l]; \mathscr{D}) \cong CF_{j+l-k}(L_0, L_1;
\mathscr{D}).$$

\subsection{Weakly filtered Fukaya categories} \label{sb:wfuk} The
above construction can be enhanced to an $A_{\infty}$-category called
the Fukaya category.

Fix a collection of marked Lagrangians $\mathcal{X}$ in $X$. We assume
that $\mathcal{X}$ is closed under grading translations and shifts of
the primitives, namely if $L = (\bar{L}, h_L, \theta_L)$ is in
$\mathcal{X}$ then for every $r \in \mathbb{R}$ and
$k \in \mathbb{Z}$, the marked Lagrangian
$\Sigma^r L[k] := (\bar{L}, h_L+r, \theta_L-k)$ is also in
$\mathcal{X}$. (Of course, in case $\mathcal{X}$ is not closed under
shifts and translations we can easily fix this by adding to
$\mathcal{X}$ all the shifts and translations of its objects.)

The Fukaya category $\fuk(\mathcal{X})$ associated to $\mathcal{X}$ is
an $A_{\infty}$-category whose objects are the element of
$\mathcal{X}$ and the \zjr{complex of} morphisms between a pair of objects from
$\mathcal{X}$ is the Floer complex of that pair. In order to set up
this $A_{\infty}$-category one has to choose for every pair of objects
$(L_0, L_1)$ from $\mathcal{X}$ a regular Floer datum
$\mathscr{D}_{L_0, L_1}$ and then extend this choice to a consistent
choice of regular perturbation data $\mathscr{P}_{\mathcal{X}}$, which
is defined for every tuple of Lagrangians $(L_0, \ldots, L_d)$,
$d \geq 1$, from the collection $\mathcal{X}$. (It is important that
both the Floer data as well as the perturbation data associated to a
tuple depend only on the underlying Lagrangians in that tuple, and not
on the choice of primitives or gradings on the Lagrangians in the
tuple.)

Once these choices are set, one defines
$$\Mor_{\fuk(\mathcal{X})}(L_0, L_1) := CF(L_0, L_1; \mathscr{D}_{L_0,L_1}),$$
endowed with the Floer differential $\mu_1$. The higher order
operations $\mu_d$ \zhnote{for $d \geq 2$} are multi-linear maps:
\begin{equation} \label{eq:mu-d} \mu_d: CF(L_0, L_1; \mathscr{D}_{L_0,
    L_1}) \otimes \cdots \otimes CF(L_{d-1}, L_d;
  \mathscr{D}_{L_{d-1}, L_d}) \longrightarrow CF(L_0, L_d;
  \mathscr{D}_{L_0,L_d})
\end{equation}
of cohomological degree $2-d$, which are defined for every tuple of
Lagrangians $L_0, \ldots, L_d$ from $\mathcal{X}$. They satisfy the
$A_{\infty}$-identities. The definition of $\mu_d$ goes by counting
Floer $(d+1)$-polygons in $X$ with boundary conditions on the
$L_i$'s. These polygons satisfy a perturbed Cauchy-Riemann equation
with perturbations prescribed by $\mathscr{P}_{\mathcal{X}}$.  Note
that the Fukaya category described above depends on the choice
$\mathscr{P}_{\mathcal{X}}$ of the perturbation data, hence should in
fact be denoted by
$\fuk(\mathcal{X};\mathscr{P}_{\mathcal{X}})$. However it is well
known that different choices of perturbation data lead to
quasi-equivalent categories~\cite[Chapter~2,
Section~10]{Seidel}

Taking filtrations into account, as already mentioned
in~\S\ref{sbsb:FCF}, the $\Mor$'s of this category are filtered chain
complexes. However, due to the perturbation data involved in defining
the higher order operations, the $\mu_d$-operations for $d \geq 2$ do
not preserve the action filtrations, but only preserves them up to an
error (that depends on $d$). Consequently the resulting
$A_{\infty}$-category is not filtered but only {\em weakly
  filtered}. Enhancing such a structure to a TPC, e.g.~along the lines
of the construction outlined in~\S\ref{sb:faitpc}, seems like a
non-trivial technical problem.

\section{Genuinely filtered Fukaya categories}
\label{sb:ffuk}

Here we outline a construction that gives rise to a genuinely filtered
Fukaya $A_{\infty}$-category. This however will require very
restrictive assumptions on the collection of objects $\mathcal{X}$,
and some adjustments in the definition of the operations $\mu_d$ for
certain tuples of Lagrangians.

Recall from the beginning of the section that we denote by
$\bar{\mathcal{X}} = \{ \bar{L} \mid L \in \mathcal{X}\}$ the
collection of underlying Lagrangian submanifolds corresponding to the
marked Lagrangians in $\mathcal{X}$. Recall also the assumptions on
$\bar{\mathcal{X}}$: $\bar{\mathcal{X}}$ is finite; every two distinct
Lagrangians $L', L'' \in \bar{\mathcal{X}}$ intersect transversely;
for every three distinct Lagrangians
$\bar{L}_0, \bar{L}_1, \bar{L}_2 \in \bar{\mathcal{X}}$ we have
\zhnote{$\bar{L}_0 \cap \bar{L}_1 \cap \bar{L}_2 = \emptyset$}. We
also continue to assume, as before, that $\mathcal{X}$ is closed under
shifts and translation in grading.

A more general approach, yielding genuinely filtered Fukaya
categories, \pbrev{has been developed by Ambrosioni~\cite{Amb:fil-fuk}
  after the first version of the current work had appeared.} This
approach does not impose any restrictions on the collection
$\mathcal{X}$, besides the assumption that all Lagrangians in
$\mathcal{X}$ are weakly exact (and possibly graded, if one wants a
graded theory). In particular no finiteness condition on
$\bar{\mathcal{X}}$ is needed in that work and no transversality
assumption is made for distinct elements of this collection. \pbrev{On
  the other hand, the invariance properties of the filtered Fukaya
  categories from~\cite{Amb:fil-fuk} are coarser than the ones
  provided by our approach (compare
  e.g.~Theorems~\ref{cor:Fuk-cat-var} and~\ref{t:fil-fuk})
  to~\cite[Theorem~B]{Amb:fil-fuk}).}

The construction outlined below is based on methods already
well-established in the literature and we will therefore only provide
a sketch of the construction omitting quite a few technical details
but emphasizing some points that are important in the control of
filtration aspects.

\subsection{Floer chain complexes redefined} \label{sbsb:fdata} We
begin by redefining the Floer chain complexes in a way which will
enable us to obtain a genuinely filtered Fukaya category.

\label{p:floer-data-1}
Let $L_0, L_1 \in \mathcal{X}$. Assume first that
$\bar{L}_0 \neq \bar{L}_1$ (hence they intersect transversely).  In
this case we fix a Floer datum $\mathscr{D}_{L_0,L_1}$ of the type
$(0,J_{L_0, L_1})$, i.e.~its Hamiltonian term will be identically
$0$. (Once again, the choice of $J_{L_0,L_1}$ is made such that it
depends only on the underlying Lagrangians $\bar{L}_0$, $\bar{L}_1$.)
We then define $CF(L_0, L_1; \mathscr{D}_{L_0, L_1})$ to be the
standard Floer complex associated to the pair $(\bar{L}_0, \bar{L}_1)$
using the Floer data $\mathscr{D}_{L_0, L_1}$ chosen above. The
grading is defined using the grading on the two marked Lagrangians
$L_0, L_1$. The filtration on $CF(L_0, L_1; \mathscr{D}_{L_0, L_1})$
is defined by using the action as a filtration function. Specifically,
if $x \in \bar{L}_0 \cap \bar{L}_1$ is a generator of
$CF(L_0, L_1; \mathscr{D}_{L_0, L_1})$, its action $\mathcal{A}(x)$ is
define by~\eqref{eq:symp-act-2}.

Assume now that $\bar{L}_0 = \bar{L}_1$ and denote by $\bar{L}$ this
common Lagrangian. In this case the Floer datum will be replaced by
a choice of a Morse datum (which we continue to denote by
$\mathscr{D}_{L_0,L_1}$), namely a pair
$(f_{\bar{L}}, (\cdot, \cdot)_{\bar{L}})$ of a Morse function
$f_{\bar{L}}: \bar{L} \longrightarrow \mathbb{R}$ and a Riemannian
metric $(\cdot, \cdot)_{\bar{L}}$ on the common underlying Lagrangian
$\bar{L}$. We will further assume that all the Morse functions
$f_{\bar{L}}$ have a unique (local) maximum (i.e.~a unique critical
point of index $n = \dim_{\mathbb{C}}X$). The purpose of this
assumption is to assure that the units in our Fukaya category will be
strict (rather than only homology units).

The Floer complex $CF(L_0, L_1; \mathscr{D}_{L_0, L_1})$ is defined to
be the Morse complex $CM(\bar{L})$ of $\bar{L}$, associated to the
Morse data
$\mathscr{D}_{L_0,L_1} = (f_{\bar{L}}, (\cdot, \cdot)_{\bar{L}})$. We
filter this chain complex in the following way. We set the filtration
level {\em for all} generators
$x \in CF(L_0, L_1; \mathscr{D}_{L_0, L_1})$ of this chain complex
(which are critical points of $f_{\bar{L}}$) to be the constant
$c \in \mathbb{R}$, where $c \equiv h_{L_1} - h_{L_0}$ is the
difference of the primitive functions of the two markings of the
Lagrangian $\bar{L}$. To keep notation uniform we continue to denote
the filtration level of $x$ by $\mathcal{A}(x)$ in the same way we have
done for action.

\subsection{Clusters of punctured disks} \label{sbsb:clusters-pdisks}
To define the $\mu_d$-operations we will use a hybrid model that
combines Floer polygons with gradient Morse trees. The maps defining
the $\mu_d$-operations will be called clusters of Floer polygons. This
approach is analogous to the cluster Floer homology theory initiated
by Cornea-Lalonde~\cite{Co-La} who also introduced the name
``clusters'' in this context. Further modifications and foundational
work on the subject has been done in~\cite{Charest-thesis,
  Charest}. The main difference between these works and what we will
be doing below is the following. The cluster homology
theory~\cite{Co-La} deals with a single Lagrangian in the presence of
pseudo-holomorphic disks. The ``clusters'' in that work consist of
Morse flow lines attached to pseudo-holomorphic disks. In contrast,
here we deal with Floer theory of many Lagrangians together (setting
up a Fukaya category), but in the absence of pseudo-holomorphic
disks. \pbrev{Note that~\cite{Amb:fil-fuk} too uses clusters of Floer
  polygons in a very similar way as below, and contains a detailed
  account of the subject.}

In order to describe cluster of Floer polygons, we need first to set
up their domains which we call clusters of punctured disks.

We begin with the notion of a {\em $k$-punctured disk}. By this is we
mean a Riemann surface $S_k$ which is obtained from the closed
$2$-disk $D \subset \mathbb{C}$ by removing $k \geq 1$ distinct
boundary points $z_1, \ldots, z_k \in \partial D$ ordered in the {\em
  clockwise} direction, together with the following additional
data. The points $z_i$ will be called punctures. Each puncture $z_i$
is declared to be either an {\em entry} puncture or an {\em exit}
puncture. We allow $S_k$ to have at most one exit puncture. We will
typically denote the entry punctures by a $+$ superscript (e.g.~$z^+$)
and the exit one by a $-$ superscript (e.g.~$z^{-}$). See
figure~\ref{f:punctured-disk}. Note that the boundary $\partial S_k$
consists of $k$ arcs which we typically denote by $C_1, \ldots, C_k$,
where the arc $C_j$ goes from $z_j$ to $z_{j+1}$ for
$1 \leq j \leq k-1$, in the clockwise direction, and $C_k$ goes from
$z_k$ to $z_1$.

For each punctured disk $S_k$ we fix a choice of strip-like ends along
each of its punctures, as in~\cite[Chapter~2,
Section~9]{Seidel}. These choices should be compatible
with splitting and gluing, as will be described later on.

\begin{figure}[h]
  \begin{center}
  \includegraphics[scale=0.60]{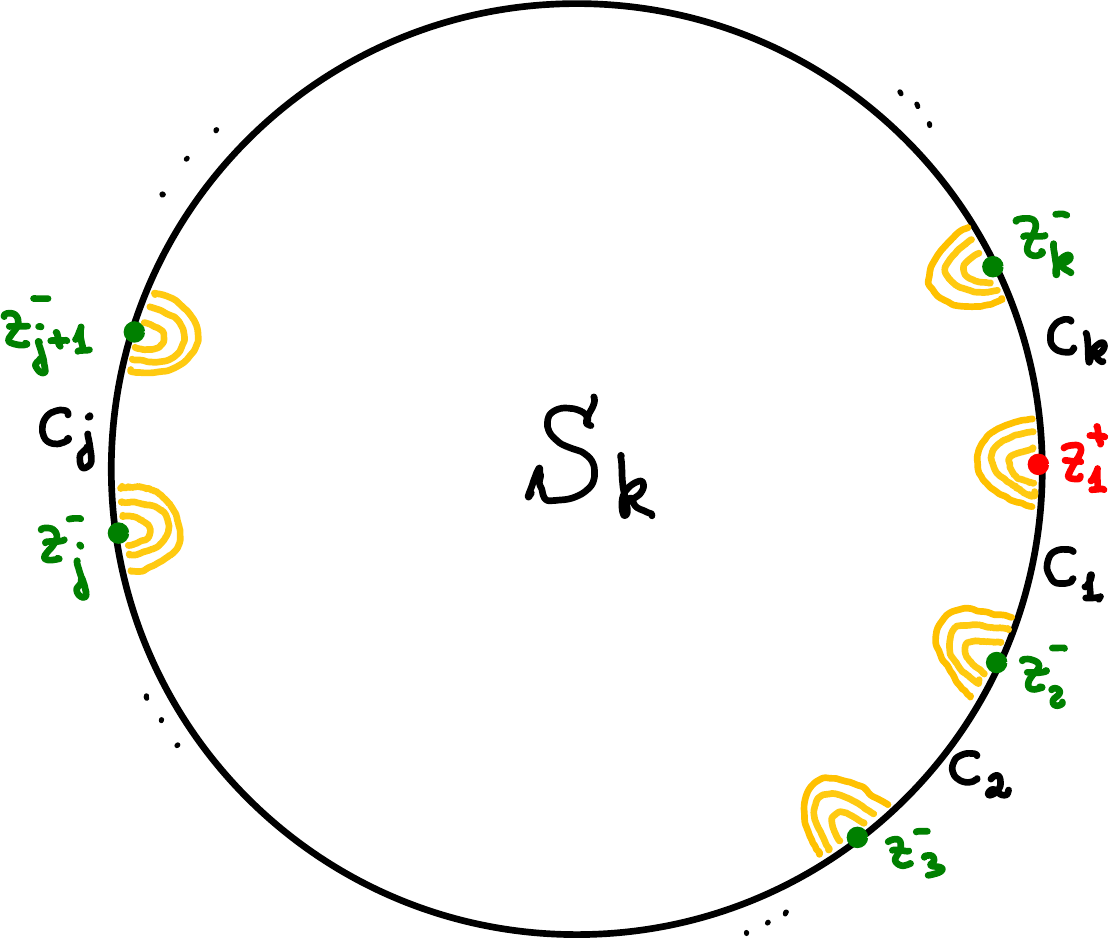} \qquad
  \includegraphics[scale=0.60]{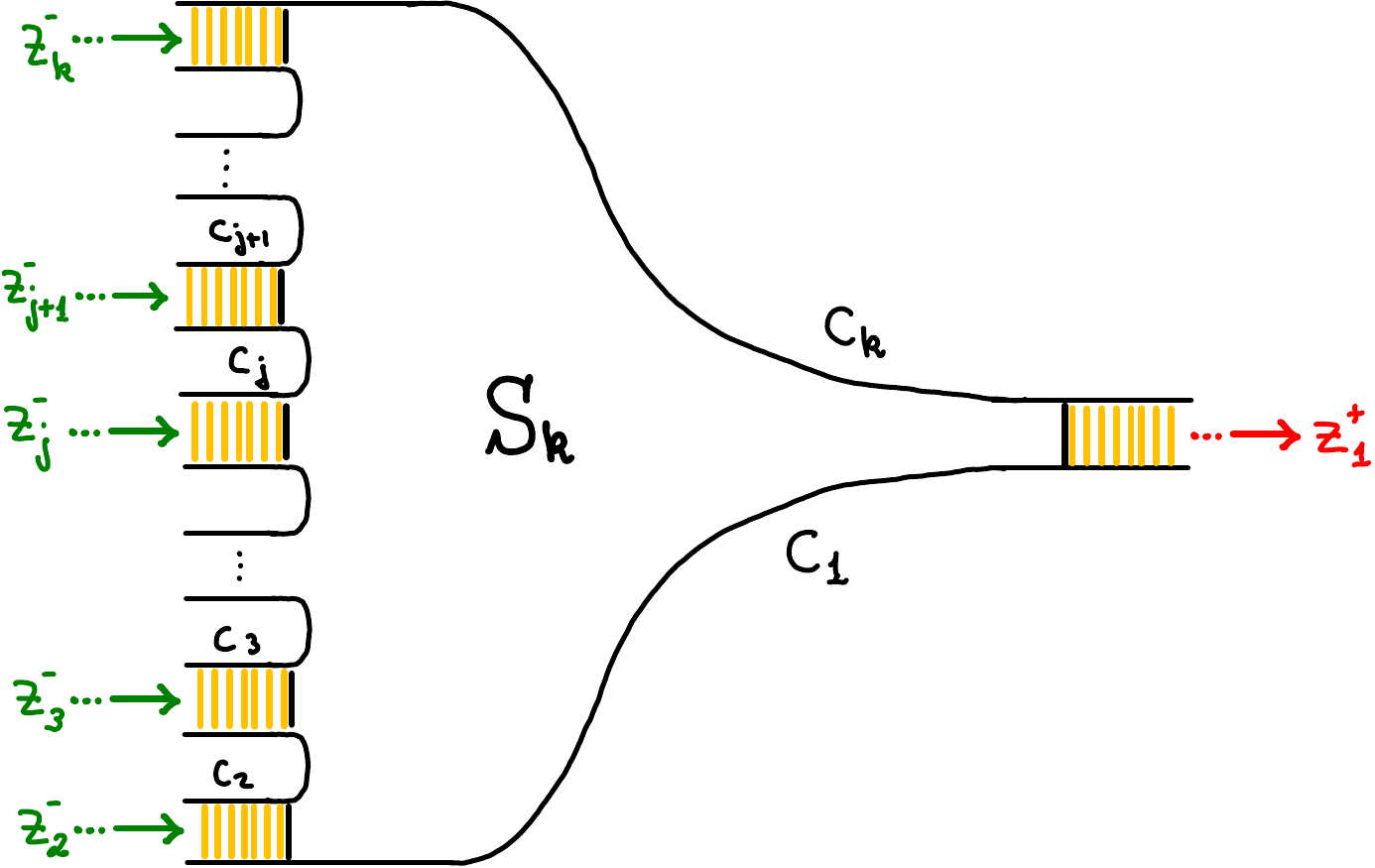}
  \end{center}
  \caption{On the left: a $k$-punctured disk $S_k$ with $k-1$ entries
    and one exit. The regions in yellow are the strip-like ends. On
    the right: a Riemann surface biholomorphic to a $k$-punctured
    disk, illustrating the strip-like ends modeled on
    $(-\infty, 0] \times [0,1]$ for the entries and
    $[0,\infty) \times [0,1]$ for the exit.}
  \label{f:punctured-disk}
\end{figure}

Next we need to consider certain types of trees which call
admissible. An admissible tree $T$ is a (connected) tree with a finite
number of edges and with the following properties and additional
structures. In what follows we will call all the end-vertices of $T$
by the name ``leaves'' (in particular, we will not distinguish between
a possible root of the tree and the other end-vertices, and just refer
to all of them by the name leaves). We assume that all the leaves have
valency $1$ and all the other vertices of $T$ (i.e.~the internal ones)
have valency $3$.  Moreover, the edges of $T$ are oriented and these
orientations satisfy the property that at every internal vertex
(which by assumption has valency $3$) there are precisely two incoming
edges and one outgoing edge. The leaves of the tree $T$ are divided
into two types: E and A (where ``E'' stands for Entry/Exit leaves and
``A'' for Attachment leaves). A leaf of type-E will be called an {\em
  entry leaf} if the orientation on the edge connected to it goes from
the leaf towards the rest of the tree. In the opposite case, i.e.~when
the orientation of that edge goes into the leaf we call it an {\em
  exit leaf}. The edges of $T$ that are not connected to type-E leaves
will be called {\em internal edges}. These consist of all edges that
are not connected to any leaf, as well as those edges that are
connected to leaves of type A. The other edges will be called external
edges.

The edges of the trees are labeled by intervals in $\mathbb{R}$ as
follows. The internal edges are labeled by intervals of the type
$[0,R]$ (with possibly different values of $R>0$ for different edges).
The edges that connect between a leaf of type $E$ and an internal
vertex are labeled either by $(-\infty, 0]$ or $[0, \infty)$ according
to whether that leaf is an entry or exit, respectively. If there is an
edge connecting two leaves of type $E$ (which happens if and only if
the tree consists of exactly these two leaves and one edge connecting
them) then this edge is labeled by the interval $(-\infty,
\infty)$. Finally, we also fix an isotopy class of a planar embeddings
for the tree $T$. Note that as a \zjr{result} this fixes a cyclic clockwise
order on the three edges connected to any given internal vertex. It
also gives a cyclic clockwise order to the leaves of the entire
tree. We illustrate a typical example of an admissible tree in
Figure~\ref{f:trees} below.

\begin{figure}[h]
  \begin{center}
  \includegraphics[scale=0.60]{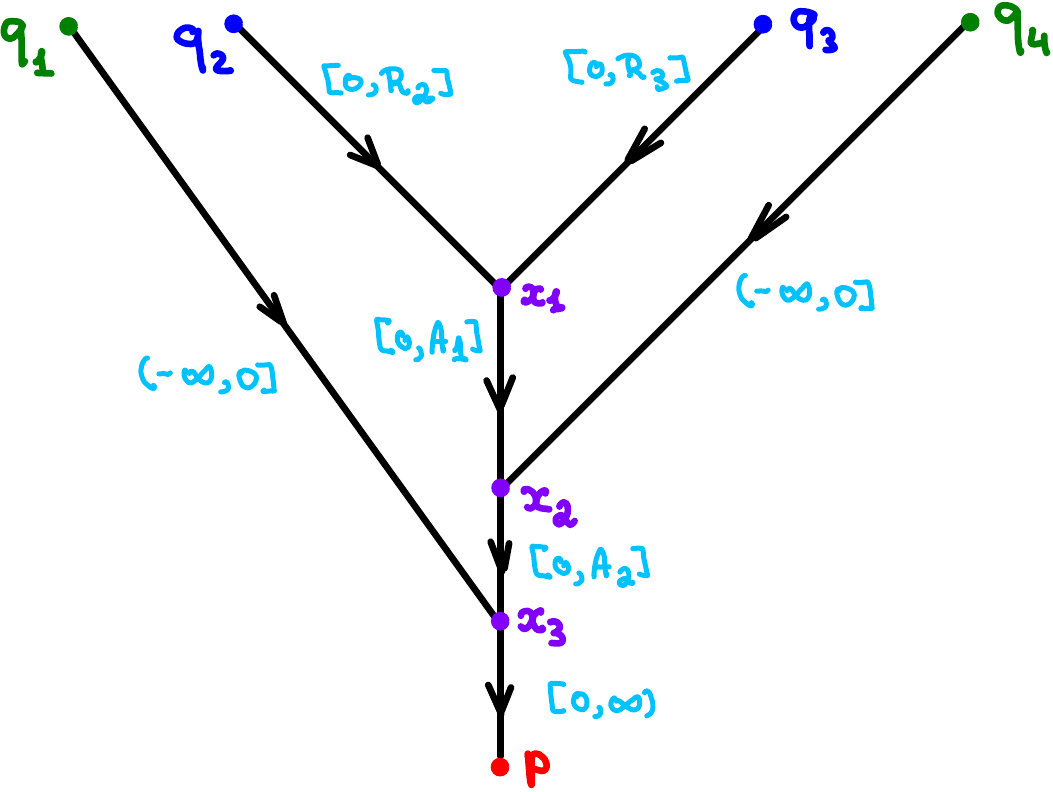}
  \end{center}
  \caption{An admissible tree. The interval labeling are in light
    blue.  The leaves $q_1, q_4$ (in green) are type-E entry leaves,
    and $p$ is a type-E exit leaf. The leaves $q_2, q_3$ (in blue) are
    of type-A. The other vertices, $x_1, x_2, x_3$ (in purple) are
    interior vertices. The overall (cyclic) clockwise order on the
    leaves of this tree is: $q_1, q_2, q_3, q_4, p$.}
  \label{f:trees}
\end{figure}

Having defined admissible trees we now fix once and for all on each
such tree $T$ a collection of {\em orientation preserving}
identifications $\sigma^{T}_e: e \longrightarrow I_e$ between each
edge $e$ and the interval $I_e$ labeling it. Of course, in case the
interval $I_e$ is of the type $[0,\infty)$, $(-\infty, 0]$ or
$(-\infty, \infty)$ (which happens when $e$ is connected to vertices
of the type $E$) then the vertices corresponding to $\pm \infty$ are
only asymptotically identified with $\pm \infty$.  Note that every
underlying tree has (infinitely) many different interval-labelings for
its edges (internal edges can be labeled by $[0,R]$ for different
values of $R>0$), leading to different admissible trees $T$. There is
an obvious parametrization of these different interval-labelings
(basically by choosing the parameter $R$ on each internal edge). We
require the identifications $\sigma^T$ to depend continuously on these
parameters.

We are now in position to introduce clusters of punctured disks.
These are built from a collection
$\mathcal{S} = \{S^{(1)}_{k_1}, \ldots, S^{(l)}_{k_l}\}$, $l \geq 0$,
of punctured disks and a collection of admissible trees
$\mathcal{T} = \{T_1, \ldots, T_r\}$, $r \geq 0$, which are attached
to the punctured disks in $\mathcal{S}$ at their leaves of type A. The
attachment of the trees is done as follows. Let $T \in \mathcal{T}$
and denote its leaves of type-A by $a_1, \ldots, a_{s_T} \in T$. For
each $1 \leq i \leq s_T$ we identify the point $a_i \in T$ with a
point lying on the boundary $\partial S^{(j)}_{k_j}$ of one of the
punctured disks from $\mathcal{S}$. Here $j = j(T,i)$ depends on the
tree $T$ and the index $i$ of the vertex $a_i$ that is being attached.
These attachments are subject to the following rule: each type-A leaf
of a given tree $T \in \mathcal{T}$ is attached to one, and only one,
punctured disk and no two type-A leaves of the same tree
$T \in \mathcal{T}$ are attached to the same punctured disk. There is
no type-A leaf from the trees in $\mathcal{T}$ that is left
unattached. We also require that, among the type-A leaves of all the
trees in $\mathcal{T}$, there are no two leaves that are attached to
the same point on the boundaries of the punctured disks.

We denote the space resulting from the above attachments by
\begin{equation} \label{eq:tilde-S-cluster} \cl =
  \Bigl(\bigcup_{q=1}^l S^{(q)}_{k_q} \Bigr) \cup_{A}
  \Bigl(\bigcup_{p=1}^r T_p\Bigr),
\end{equation}
where $\cup_A$ stands for the attachments described earlier. We will
denote the part of $\cl$ coming from the punctured disks (i.e.~the
leftward union in the right-hand side of~\eqref{eq:tilde-S-cluster})
by $\cl_{\srf}$ and the part coming from the trees (the rightward
union in the right-hand side of~\eqref{eq:tilde-S-cluster}) by
$\cl_{\tr}$.

We now impose further restrictions on the previously described
attachments. Consider the space obtained from $\cl$ by collapsing each
punctured disk $S_{k_q}^{(q)}$ from $\cl_{\srf}$ to a (different)
point:
\begin{equation} \label{eq:tilde-S-collapse} \widetilde{\cl} := \cl
  \Big/ \bigl(S^{(q)}_{k_q} \sim \text{point}_q, \; \forall q \bigr).
\end{equation}
We require that the attachments of the tree described above are done
in such a way that $\widetilde{\cl}$ is path-connected and moreover it
is a tree (hence in particular simply-connected). We do not require
this tree to be admissible.

Going back again to the space $\cl$ we note that it comes with a set
of distinguished points: the punctures of the disks $S^{(q)}_{k_q}$
together with the leaves of the trees $T \in \mathcal{T}$ that are of
type-E. We call these points {\em external points} and denote them by
$\cl_{\extp}$. The total number of external points of $\cl$ will be
called the {\em order} of $\cl$.

The external points $\cl_{\extp}$ are divided into two types: ``entry
points'' and ``exit points'', regardless of a point being a puncture
or a type-E leaf of some tree. We require that $\cl_{\extp}$ has
precisely one exit point (which can be either an exit leaf or an exit
puncture). We also require that $\cl_{\extp}$ has at least one entry
point. Finally, in case there are only two such points (i.e.~one entry
and one exit) we require that either $\cl$ is just a disk punctured at
$2$ points (with no trees attached) or that $\cl$ has no punctured
disks at all and it consists of just a tree with two vertices and one
edge connecting them.

The last requirement on $\cl$ is the following. Consider the tree
$\widetilde{\cl}$ defined in~\eqref{eq:tilde-S-collapse}. Note that by
construction, each edge of this tree is oriented (since the edges of
all $T \in \mathcal{T}$ are oriented). Moreover, this tree has a
distinguished vertex $P_{\text{exit}}$, namely the vertex that
corresponds either to the punctured disk $S_k \subset \cl_{\srf}$ that
contains the unique exit puncture, or to the unique exit leaf that
belongs to one of the trees of $\cl_{\tr}$. We require that the
orientation on the tree $\widetilde{\cl}$ has the property that given
any vertex $p \in \widetilde{\cl}$ there is a path from $p$ to the
distinguished vertex $P_{\text{exit}}$ that is compatible with the
orientation on $\widetilde{\cl}$. Figure~\ref{f:clusters} illustrates
two examples of clusters of punctured disks (the labeling by
Lagrangians $\bar{L}_i$ of the trees and the arcs in the
$\partial S_j$'s that appears in the picture should be ignored at the
moment - these will be explained later on).

\begin{figure}[h]
  \begin{center}
    \includegraphics[scale=0.60]{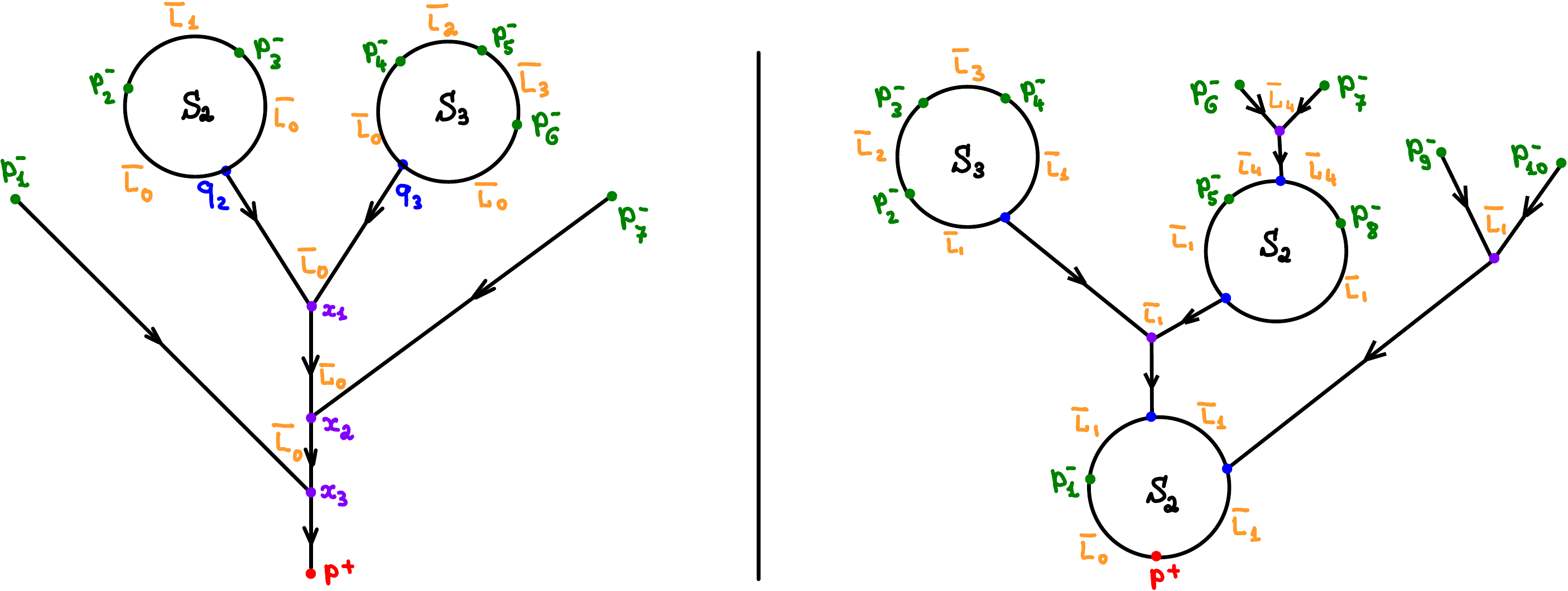}
  \end{center}
  \caption{Two examples of clusters of punctured disks. The interval
    labeling of the edges of the trees are omitted here. The overall
    clockwise cyclic ordering of the external points on the left
    cluster is $(p_1^{-}, \ldots, p_7^{-}, p^+)$ and on the right
    cluster $(p_1^{-}, \ldots, p_{10}^{-}, p^{+})$. The tuples
    describing the labeling by Lagrangians are
    $\bar{\mathcal{L}}_{\cl} = (\bar{L}_0, \bar{L}_0, \bar{L}_1,
    \bar{L}_0, \bar{L}_2, \bar{L}_3, \bar{L}_0, \bar{L}_0)$ for the
    left cluster and
    $\bar{\mathcal{L}}_{\cl} = (\bar{L}_0, \bar{L}_1, \bar{L}_2,
    \bar{L}_3, \bar{L}_1, \bar{L}_4, \bar{L}_4,\bar{L}_4, \bar{L}_1,
    \bar{L}_1, \bar{L}_1)$ for the cluster on the right.}
  \label{f:clusters}
\end{figure}

Let $\cl, \cl'$ be two \zhnote{spaces} obtained as above from two pairs of
collections $\mathcal{S}$, $\mathcal{T}$ and $\mathcal{S}'$,
$\mathcal{T}'$ of punctured disks and trees. We say that $\cl$ and
$\cl'$ are equivalent if there is a homeomorphism
$f: \cl \longrightarrow \cl'$ with the following properties. The map
$f$ maps $\cl_{\srf}$ biholomorphically to $\cl'_{\srf}$ and maps the
trees $\cl_{\tr}$ to $\cl'_{\tr}$ by an isomorphism of trees (i.e.~it
maps vertices to vertices and edges to edges). Moreover, $f$
intertwines all the other structures on $\cl_{\srf}, \cl_{\tr}$ with
those on $\cl'_{S}, \cl'_{\tr}$. This means, in particular, that entry
and exit punctures in $\cl'_{\srf}$ correspond under $f$ to the
punctures of the same type in $\cl_{\srf}$; the same goes for the
orientations on the edges of $\cl_{\tr}$, $\cl'_{\tr}$, the
interval-labeling, the identifications $\sigma_e^T$ and the classes of
planar embeddings of the trees.

An equivalence class of spaces $\cl$ as above (together will all the
structures accompanying it) will be called a {\em cluster of punctured
  disks}. However, we will often use this name also for a specific
representative $\cl$ within a given equivalence class.

For a cluster of punctured disks, say represented by $\cl$, the
orientation on the boundaries of the punctured disks in $\cl_{\srf}$
and the classes of planar embeddings of the trees in $\cl_{\tr}$
induce a preferred clockwise cyclic order on the set of external points
of $\cl$ (recall that the external points consists of the entry and
exit points, regardless of being type-E leaves of the trees or
punctures of the disks). Note that this ordering is preserved by the
homeomorphisms defining the equivalence between different
representatives $\cl$ of the same class.

In what follows it will be convenient to single out clusters of
punctured disks of the following type. A cluster of punctured disks
$\cl$ is called {\em simple} if it consists of a single punctured disk
without any trees attached. \label{pp:simple-cluster}

We turn now to decorated clusters of punctured disks. Let $\cl$ be a
cluster of punctured disks. By decoration of $\cl$ by elements of
$\bar{\mathcal{X}}$ we mean the following. We label each arc in the
boundaries of $\cl_{\srf}$ as well as each edge in the trees of
$\cl_{\tr}$ by an element of $\bar{\mathcal{X}}$. The labeling is
subject to the following restrictions. In each tree from $\cl_{\tr}$
all the edges are labeled by the same $\bar{L} \in \bar{\mathcal{X}}$
(alternatively, one can think of each tree $T \subset \cl_{\tr}$ as
being labeled by an element of $\bar{\mathcal{X}}$). The restriction
on the labeling for the $\cl_{\srf}$-part of $\cl$ is that in each
punctured disk $S_k$ from $\cl_{\srf}$ there are no two consecutive
arcs (i.e.~two arcs with one puncture between them) that are labeled
by the same element form $\bar{\mathcal{X}}$.

Once a cluster of punctured disks $\cl$ is decorated by elements of
$\bar{\mathcal{X}}$ we can form a tuple
$\bar{\mathcal{L}}_{\cl} = (\bar{L}_0, \ldots, \bar{L}_d)$ that
encodes its decoration, where $d+1 = \mid \cl_{\extp} \mid$ is the
order of $\cl$. The definition of $\bar{\mathcal{L}}_{\cl}$ goes as
follows. Denote by $p^{+}, p_1^{-}, \ldots, p_d^{-}$ the external
points of $\cl$, ordered as explained earlier, where $p^{+}$ is the
unique exit point and the $p^{-}_j$'s are all entry points. If
$p^{-}_j$ is a puncture of one of the disks $S_k \subset \cl_{\srf}$
we take $\bar{L}_j$ to be the Lagrangian that labels the arc on
$\partial S_k$ coming after the puncture (where ``after'' refers to
the clockwise orientation on $\partial S_k$). If the entry $p^{-}_j$
is a leaf of one of the trees $T \subset \cl_{\tr}$ then we take
$\bar{L}_j$ to be the Lagrangian that labels that tree. We define
$\bar{L}_0$ in the same \zjr{way}, according to whether $p^{+}$ is a
puncture or a leaf. Figure~\ref{f:clusters} shows two examples of
decorated clusters of punctured disks.

We will now reverse in some sense the decoration construction. Namely,
we fix a tuple $\bar{\mathcal{L}} = (\bar{L}_0, \ldots, \bar{L}_d)$ of
Lagrangians from $\bar{\mathcal{X}}$, and consider the space
$\clus(\bar{\mathcal{L}})$ of all possible decorated clusters of
punctured disks $\cl$ with
$\bar{\mathcal{L}}_{\cl} = \bar{\mathcal{L}}$. We call the elements of
this space {\em $\bar{\mathcal{L}}$-decorated clusters of punctured
  disks}. As before, the element of this space are equivalence classes
of the spaces $\cl$, rather than the spaces $\cl$ themselves. But it
will often be convenient to work with an actual representatives $\cl$
of a given class.

Clearly every decorated cluster $\cl$ belongs to a unique space
$\clus(\bar{\mathcal{L}})$ since the tuple $\bar{\mathcal{L}}_{\cl}$
is uniquely defined by $\cl$.
\begin{remark} \label{r:decorations} Let
  $\bar{\mathcal{L}} = (\bar{L}_0, \ldots, \bar{L}_d)$ be a tuple of
  Lagrangians from $\bar{\mathcal{X}}$.
  \begin{enumerate} 
  \item \label{i:Li-neq-Lj} If $\bar{\mathcal{L}}$ has the property
    that $\bar{L}_i \neq \bar{L}_j$ for every $i \neq j$, then every
    cluster of punctured disks $\cl$ that admit an
    $\bar{\mathcal{L}}$-decoration must be simple.
  \item The converse statement to point~(\ref{i:Li-neq-Lj}) above is
    obviously not true whenever $d \geq 3$. Namely, one can decorate a
    simple cluster of punctured disks $\cl$ by a tuple
    $\bar{\mathcal{L}}$ whose entries do have repetitions. However in
    such a case we must have: $\bar{L}_i \neq \bar{L}_{i+1}$ for every
    $0 \leq i \leq d$ (where the indexing is to be understood
    cyclically $\bmod(d+1)$).
  \item A tuple $\bar{\mathcal{L}}$ with
    $\bar{L}_i \neq \bar{L}_{i+1}$ for every $0 \leq i \leq d$ can
    decorate also non-simple clusters of punctured disks. However, if
    a non-simple $\cl$ is decorated by such an $\bar{\mathcal{L}}$
    then none of the trees in $\cl_{\tr}$ can have external leaves
    (which means that each leaf in $\cl_{\tr}$ is attached to some
    punctured disk in $\cl_{\srf}$).
  \end{enumerate}
\end{remark}



\subsection{Splitting and degeneration} \label{sbsb:splitting}
Given a tuple $\bar{\mathcal{L}} = (\bar{L}_0, \ldots, \bar{L}_d)$ of
Lagrangians from $\bar{\mathcal{X}}$, the space
$\clus(\bar{\mathcal{L}})$ of $\bar{\mathcal{L}}$-decorated clusters
of punctured disks has the structure of a smooth manifold, analogous
to the space of punctured disks from~\cite[Chapter~2,
Section~9]{Seidel}. This manifold admits a natural
partial compactification into a manifold with corners. The top
dimensional strata of its boundary correspond to several types of
degenerations of clusters of punctured disks which we briefly describe
below. Note that adding this boundary to $\clus(\bar{\mathcal{L}})$
will still not make a full compactification of this space (hence the
use of the words ``partial compactification''), however it will be
enough for the purpose of establishing the $A_{\infty}$-category
identities. Below we will call those degenerations that lead to
elements of this boundary ``admissible degenerations'' and their
limiting objects ``admissible degenerate clusters''.

\smallskip
\paragraph{{\em Splitting within punctured disks.}}
We begin with describing two variants of a degeneration that can occur
to one punctured disk moving in a family. A family of punctured disks
$S^t_k$ (here $k$ is the number of punctures and $t \in \mathbb{R}$ is
parametrizing the family) can degenerate (or split) into two punctured
disks $S'_{k'}$ and $S'_{k''}$, where $k'+k'' = k+2$.  The first
punctured disk $S'_{k'}$ ``inherits'' $k'-1$ of the entry punctures of
$S$ (placed in the same clockwise order as in $S$) and has one
additional exit puncture $z'_+$. The other component, $S''_{k''}$,
``inherits'' all the other $k-(k'-1) = k''-1$ punctures of $S$ (again,
in the same clockwise order) and has one additional entry puncture
$z''_{-}$. Note that at the moment we do allow $k'$ or $k''$ to have
the values $1$ or $2$. This is in contrast to the more standard
realizations of the Fukaya category, where each of the two disks in a
splitting are required to have at least $3$ punctures. However later
on, when viewing these disks as part of degenerate cluster more
restriction will be added in order to make such a degenerate
configuration an admissible one.

Reversely, the two punctured disks $S'_{k'}$ and $S''_{k''}$ can be
glued along the punctures $z'_{+}$, $z''_{-}$ into a family of
punctured disks $S^t_k$.

Depending on the context, in what follows we will sometimes view the
preceding degeneration differently. Namely, regard the two punctures
$z'_{+}$ , $z''_{-}$ as ``removable'' and view the degeneration of the
family $S^t_k$ as a splitting into two punctured disks $S'_{l'}$ and
$S''_{l''}$ attached one to the other at a point (which is not a
puncture) on their boundaries. Note that now we have $l'+l''=k$, and
similarly to the preceding case, we do allow $l'$ or $l''$ to have the
values $1$ and $2$.

Conversely, as before, the two punctured disks $S'_{l'}$ and
$S''_{l''}$ can be glued into a family of punctured disks $S^t_k$.

Analytically the two variants described above are the same, however
when taking decorations into account it is important to distinguish
between them. More precisely, if the punctured disks $S^t_k$ are
decorated by the Lagrangians $(\bar{N}_0, \ldots, \bar{N}_{k-1})$ the
first variant of splitting corresponds to two punctured disks with
decorations $(\bar{N}_r, \ldots, \bar{N}_s)$ and
$(\bar{N}_0, \ldots, \bar{N}_r, \bar{N}_s, \ldots, \bar{N}_{k-1})$,
where $1 \leq r < s \leq k-1$, $\bar{N}_r \neq \bar{N}_s$,
$s-r+1 = k'$, $k'' = k-(s-r)+1$.

The second variant (i.e.~where $S^t_k$ degenerates into $S'_{l'}$ and
$S''_{l''}$ attached at a common point, which is not a puncture, along
their boundaries) corresponds to the case when the decoration
$(\bar{N}_0, \ldots, \bar{N}_{k-1})$ has $\bar{N}_r = \bar{N}_s$ for
some non-consecutive indices, $r<s-1$, and the splitting yields the
decorations $(\bar{N}_r, \ldots, \bar{N}_{s-1})$ and
$(\bar{N}_0, \ldots, \bar{N}_r, \bar{N}_{s+1}, \ldots, \bar{N}_{k-1})$
on $S'_{l'}$ and $S''_{l''}$.

The two variants of splitting are illustrated in
Figure~\ref{f:splitting}.

\begin{figure}[h]
  \begin{center}
    \includegraphics[scale=0.60]{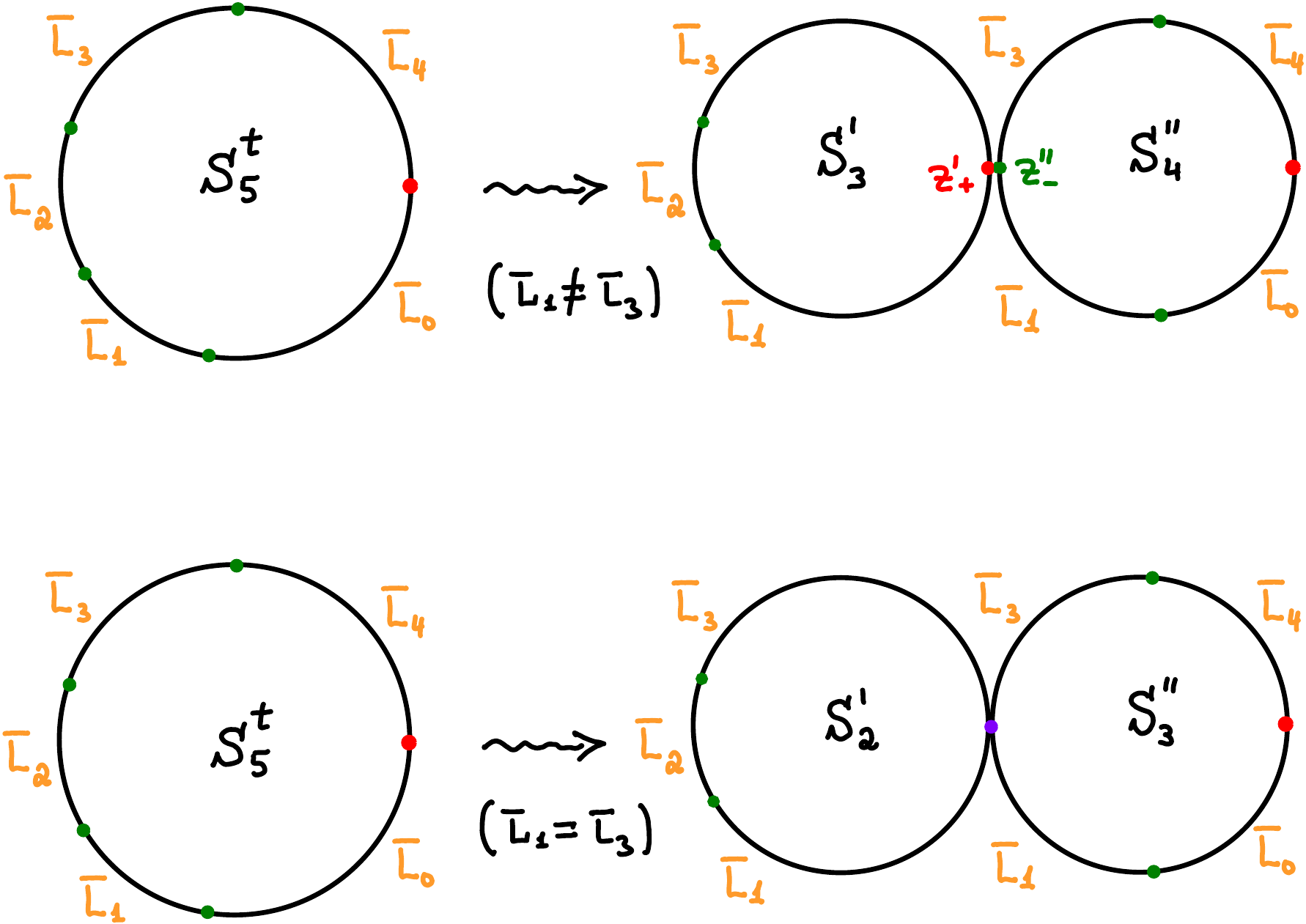}
  \end{center}
  \caption{Two variants of splitting of decorated punctured disks. The
    red and green points $z'_{+}$, $z''_{-}$ on the right-hand side of
    the upper figure $z'_{+}$, $z''_{-}$ are new punctures. The purple
    point on the right-hand side of the lower figure, where the two
    disks are attached, is {\em not} a puncture.}
  \label{f:splitting}
\end{figure}

We now turn to clusters of punctured disks and describe their {\em
  admissible degenerations}. Fix a tuple
$\bar{\mathcal{L}} = (\bar{L}_0, \ldots, \bar{L}_d)$ of Lagrangians. A
family $\Sigma^t$ of $\bar{\mathcal{L}}$-decorated clusters of
punctured disks can converge to a decorated degenerate cluster
$\Sigma^{\infty}$ of punctured disks (which strictly speaking, by our
definitions, might not be a genuine cluster of punctured disks). The
degeneration of $\Sigma^t$ into $\Sigma^{\infty}$ can be of several
types. The first type is when one (or more) of the punctured disks in
the clusters $\Sigma^t$ degenerates in the way described
earlier. Depending on the decoration $\bar{\mathcal{L}}$ one of the
two variants mentioned above, or both, can occur. There is one slight
exception to this rule. Namely, in both of the variants described
above we view the degeneration admissible only if each of the two
punctured disks formed by the splitting contain at least three
distinguished points. Here, by a distinguished point we mean either a
puncture, or a point attached to a tree, or (in case of the 2'nd
variant) the point of attachment to the other punctured disk in the
degenerate cluster.

Below we will describe another four types of admissible
degenerations. Before we go into this, a quick remark about the
decorations of the limit $\Sigma^{\infty}$ is in order.

Our conventions for decorations require the cluster to have an exit
point (according to which we label the first entry in the
decoration). However, the first variant of the degenerations described
above yields two punctured disks $S'$, $S''$, where one of them has a
(new) exit point and one of them has a (new) entry point. The apparent
problem is that one of these punctured disks might not have any exit
point, hence there might be an ambiguity regarding the order in which
we write its decoration. However, this ambiguity is fixed if we use
the following conventions. The limit $\Sigma^{\infty}$ is divided into
two components: the one that contains $S'$ and the one that contains
$S''$. The decorations are uniquely defined once we require that the
exit point of the limit $\Sigma^{\infty}$ corresponds to the $S''$
part. A similar thing applies also to clusters in which one of the
punctured disks degenerates according to the second variant described
earlier.

We now proceed to describe four additional types of admissible
degenerations.  \smallskip
\paragraph{{\em Splitting within trees.}}
Apart from degeneration of punctured disks in a cluster, there are yet
several other types of degeneration that can occur within a family
$\Sigma^t$ of clusters of punctured disks. Part of these has to do
with degeneration of trees $\Sigma^t_{\tr}$ of $\Sigma^t$ and another
part is related to how these trees are attached to
$\Sigma^t_{\srf}$.

\smallskip
\paragraph{{\em Shrinking of edges to a point.}}
The first type of degeneration within trees is when an interior edge
in one of the trees of $\Sigma^t_{\tr}$ shrinks to a point (this means
that also its interval labeling and parametrization shrink to a point
and a constant, respectively). The limit tree will now have one vertex
less and will inherit from $\cl^t_{\tr}$ all the other structures
(like the labeling of the other edges, the class of planar embedding
etc.). See Figure~\ref{f:shrinking}. Note however that the limit tree
will not be admissible (e.g. it might have vertices of valency $4$, or
a leaf of type-A that becomes identified with an interior vertex).

\begin{figure}[h]
  \begin{center}
    \includegraphics[scale=0.50]{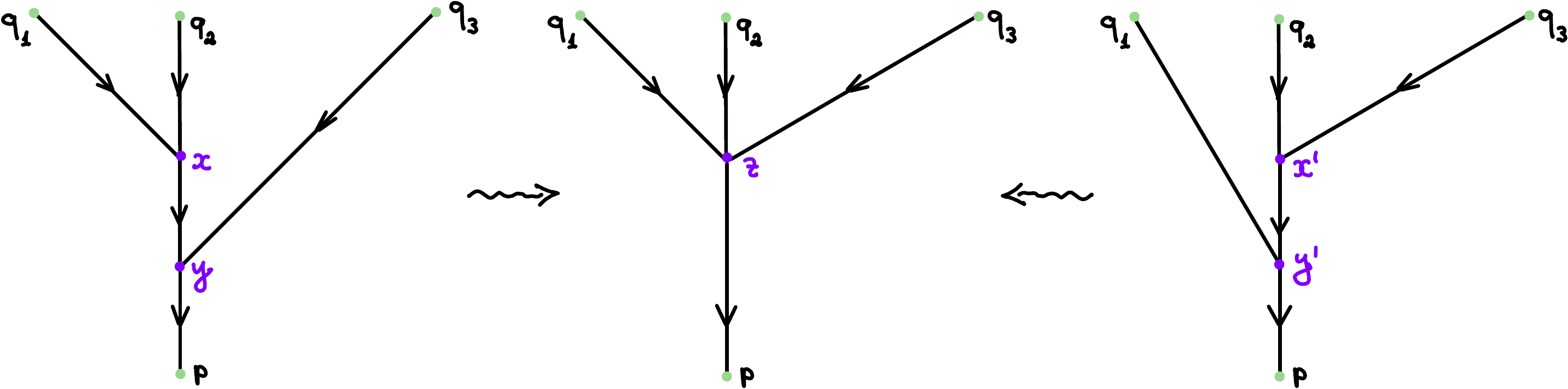} \ \\ \ \\
    \includegraphics[scale=0.70]{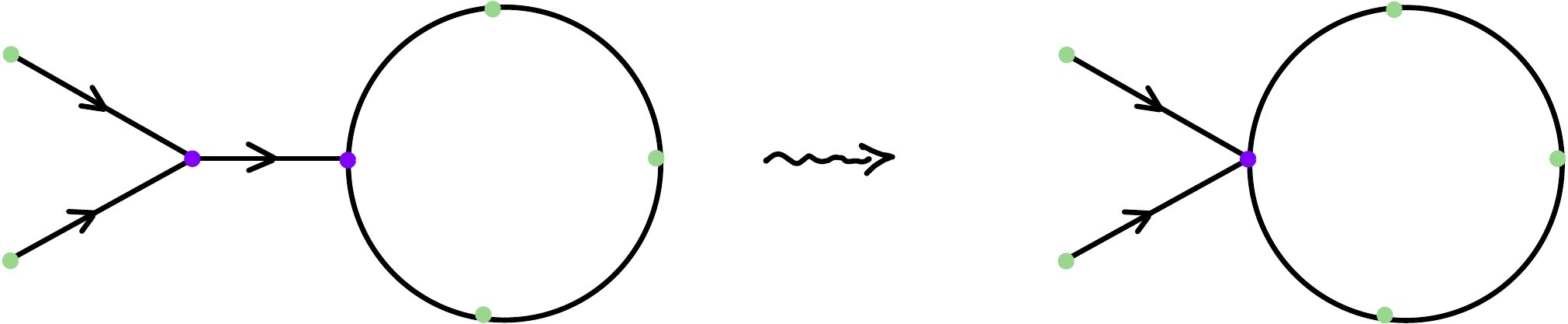} \ \\ \ \\
    \includegraphics[scale=0.55]{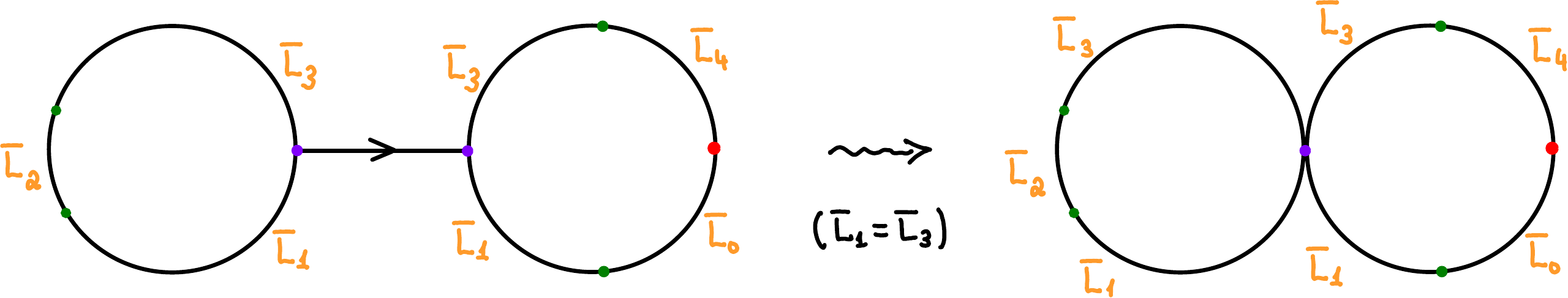}
  \end{center}
  \caption{Examples of degenerations where an edge of a tree shrinks
    to a point, leading to a limit (inadmissible) tree, possibly being
    part of a degenerate cluster.}
  \label{f:shrinking}
\end{figure}

\smallskip
\paragraph{{\em Edge breaking.}}
Another type of degeneration is when one of the interior edges $e$ in
a family of trees $T^t \subset \cl^t$ becomes of infinite length. We
view the limit of the $T^t$'s as a broken tree which consists of two
trees $T^{\infty}_1, T^{\infty}_2$. These two trees are obtained from
$T^{\infty}$ as follows. We delete the (interior of the) edge $e$ from
$T^t$ and obtain two connected components: the part of $T^t$ that
appears ``before'' the edge $e$ and the part that appears after that
edge (here ``before'' and ``after'' refer to the orientation on
$T^t$). Denote by $p_1$ the end-vertex of the first component
(i.e.~the entry vertex to the edge $e$) and by $p_2$ the new entry
vertex of second component (which corresponds to the exit vertex of
$e$). We now take the first component and add to it a new edge $e_1$
emanating from $p_1$. The result is the tree $T^{\infty}_1$. The other
vertex $q_1$ of $e_1$ will now be a type-E leaf of $T_1^{\infty}$ and
we regard it as an exit leaf. We label the edge $e_1$ by $[0,\infty)$,
and the rest of the edges are labeled by the limiting labels of $T^t$
as $t \longrightarrow \infty$. The definition of $T_2^{\infty}$ is
similar, only that now we add a new edge $e_2$ to the second component
(i.e.~the one coming ``after'' the deleted edge $e$) attached at
$p_2$. The resulting tree is $T_2^{\infty}$. The new vertex (which is
the entry to $e_2$) will be a type-E leaf of $T_2^{\infty}$. The edge
$e_2$ is labeled by $(-\infty, 0]$ and the rest of the edges are
labeled by the limiting labels of $T^t$ as $t \longrightarrow \infty$.
We refer to $T^{\infty}$ as a ``broken'' tree with components
$T_1^{\infty}$ and $T_2^{\infty}$. See Figure~\ref{f:breaking-1}.

We add the following restriction on edge breaking degenerations.  A
degeneration as described above is considered admissible only if none
of the trees $T_1^{\infty}$, $T_2^{\infty}$ is a tree with two
vertices both of which are type-E leaves, connected by one edge. All
other edge breaking degenerations are considered admissible.

\begin{figure}[h]
  \begin{center}
    \includegraphics[scale=0.60]{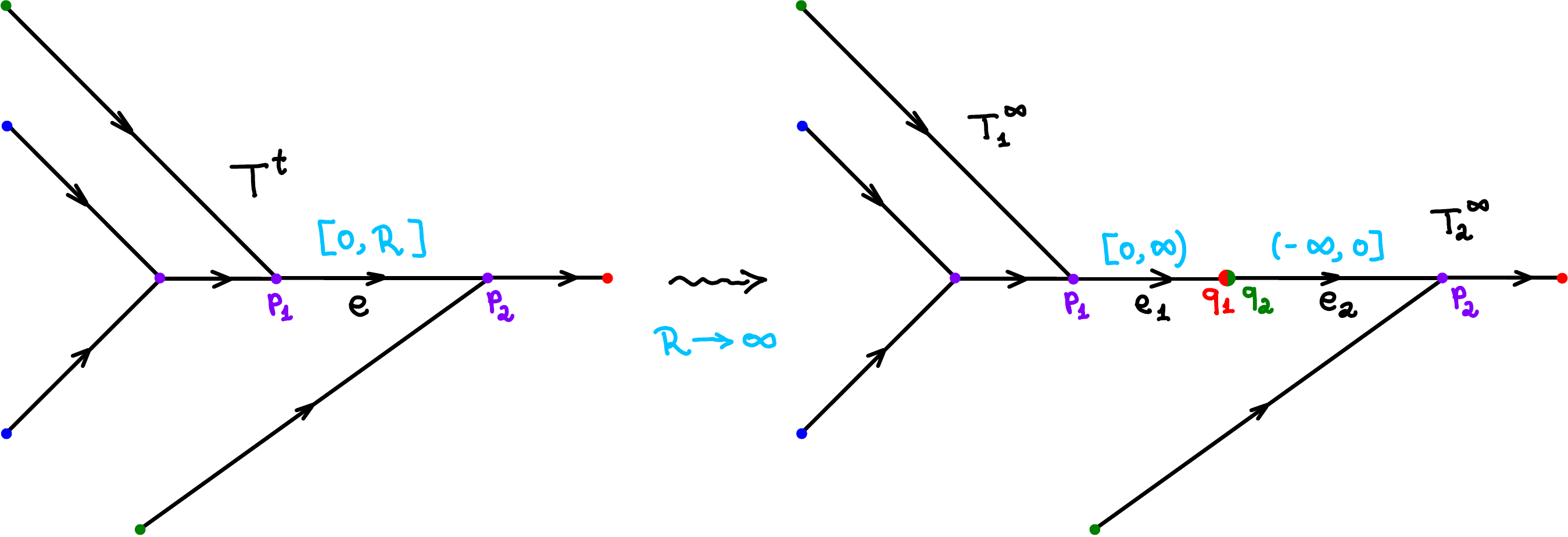}
  \end{center}
  \caption{Breaking along an edge of a tree, leading to a broken tree
    with two components.}
  \label{f:breaking-1}
\end{figure}

\smallskip \paragraph{{\em Collision of type-A leaves.}}  The last
type of admissible degeneration is when two type-A leaves (belonging
to two different trees) that lie on the boundary of the same punctured
disk $S_k \subset \cl_{\srf}$ collide. This means that two trees
$T', T'' \subset \cl_{\tr}$ are grafted (or joined) at two of their
type-A exit leaves. See Figure~\ref{f:collision}.

\begin{figure}[h]
  \begin{center}
    \includegraphics[scale=0.60]{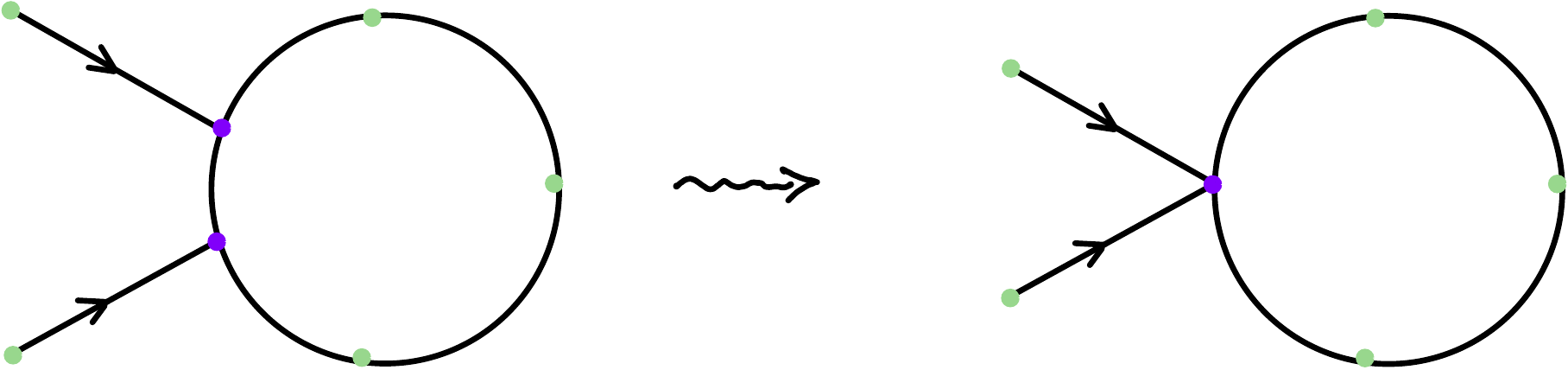}
  \end{center}
  \caption{Collision of points on $\partial S_k$ to which two
    different trees are attached.}
  \label{f:collision}
\end{figure}

\begin{rem} \label{r:degeneration}
  \begin{enumerate}
  \item The boundary of the compactification of the space
    $\clus(\bar{\mathcal{L}})$ can be described by the types of
    degeneration described above. The points of the top dimensional
    stratum of the boundary correspond to precisely one occurrence of
    degeneration as above. Of course, several instances of
    degeneration can occur simultaneously, but these instances
    correspond to the lower dimensional strata of the boundary of
    $\clus(\bar{\mathcal{L}})$. In particular, within generic
    $1$-dimensional families of cluster of punctured disks only one
    degeneration can occur at a given time.
  \item The converse to ``splitting and degeneration'' goes by the
    name gluing. Every degenerate configuration among the ones
    described above can be obtained as a limit of a family of clusters
    of punctured disks.
  \item Some of the limit configurations described above can occur as
    a result of {\em two different} degenerations. For example,
    collision of two type-A leaves of trees (along an arc of one of
    the punctured disks) leads to a configuration which is also the
    limit of the another family of clusters in which one type-A leaf
    of a tree shrinks to a point. See Figure~\ref{f:collision} versus
    the middle part of Figure~\ref{f:shrinking}.

    Similarly, the second variant of splitting withing a punctured
    disk (which leads to two punctured disks connected at
    ``non-puncture'' point along their boundaries) occurs also as a
    limit of clusters in which two punctured disks are connected by a
    tree with one edge and that edge shrinks to a point. See the lower
    part of Figure~\ref{f:splitting} versus the lower part of
    Figure~\ref{f:shrinking}.
    
    The same thing happens with shrinking of interior edges of
    trees. Namely, each of the (inadmissible) trees that occur after
    the shrinking of interior edges appears as a limit of a different
    family of trees in which an edge shrank to a point.  See the left
    and right-hand sides of the upper part of
    Figure~\ref{f:shrinking}.

    The fact that some limit configurations appear in pairs on the
    total boundary of the compactification of the $\clus$ spaces is
    important for showing that the $\mu_d$-operations in
    $\fuk(\bar{\mathcal{X}})$, as will be defined later, satisfy the
    $A_{\infty}$-identities. Indeed, when one considers $1$-parametric
    families of clusters of Floer polygons, with fixed entries and
    exit, some of the boundary points (that correspond to degeneration
    of the underlying clusters of punctured disks as described above)
    will appear in pairs and thus can be regarded as ``interior''
    points inside extended families of clusters.
  \item Recall that the boundary of the partial compactification of
    $\clus(\bar{\mathcal{L}})$ contains only admissible degenerate
    clusters. In particular, in the case of splitting of a punctured
    disks we required that the number of distinguished points on each
    component is $\geq 3$. Of course, a splitting in which one of the
    disks has only two distinguished points can occur. The reason we
    do not add such a configuration to the boundary
    $\clus(\bar{\mathcal{L}})$ is that disks with two marked points
    have a non-compact $1$-dimensional group of automorphisms
    (isomorphic to $\mathbb{R}$). This is referred to in the
    literature as an unstable marked curve.  The situation with the
    other inadmissible degenerate clusters, namely a broken tree with
    one component being a tree with one edge connecting between two
    type-E leaves, is similar. The latter component has an
    $\mathbb{R}$-action (acting by translation on the identifications
    between $(-\infty, \infty)$ and the edge of this tree).

    In practice, not including these unstable configurations to the
    boundary of $\clus(\bar{\mathcal{L}})$ will not cause any problems
    in showing that the $\mu_d$-operations satisfy the
    $A_{\infty}$-identities. The standard way to go about it in Floer
    and Morse theory is to compactify the space of clusters of {\em
      Floer polygons} in such a away that degenerations that
    correspond to the above unstable configurations are taken into
    account in the boundary of the latter spaces rather then in
    $\partial \clus(\bar{\mathcal{L}})$. In terms of the
    $A_{\infty}$-identities, these degenerations will contribute the
    terms in the identities that include $\mu_1$'s.
  \end{enumerate}
\end{rem}

\subsection{Perturbation data} \label{sbsb:prtdata} We assume that
Floer data has been chosen for every pair $L_0, L_1 \in \mathcal{X}$,
as described at the beginning of~\S\ref{sb:ffuk} on
page~\pageref{p:floer-data-1}. The perturbation data for a decorated
simple cluster (i.e.~a cluster consisting of precisely one punctured
disk and no trees) $S_{k}$ is of the same type as in the standard
theory, namely it consist of pairs $(K, J)$, where $K$ is a $1$-form
on $S_{k}$ with values in the space of compactly supported Hamiltonian
functions on the ambient manifold $X$. This $1$-form is assumed to be
compatible with the Floer data on each strip-like end of $S_{k}$ in
the sense that on these ends we have $K \equiv 0$. (Note that we are
dealing here with the case of one punctured disk without trees, which
means that the decoration $\bar{\mathcal{L}}$ is such that every two
consecutive Lagrangians in $\mathcal{L}$ have mutually transverse
underlying Lagrangians. Recall also that for pairs of distinct
underlying Lagrangians we have already set up the Floer data in
advance to have $0$ Hamiltonian terms.)  Moreover, we require $K$ to
have compact support inside the interior of $S_{k}$. The second
component of the Floer data is a family of $\omega$-compatible almost
complex structures $J = \{J_{z}\}$, that depends on $z \in S_{k}$ and
coincides on each strip-like end with the almost complex structures
chosen for the corresponding Floer data.

We now describe the perturbation data in the case of general decorated
clusters of punctured disks $\cl$. The perturbation data in this case
consists of two pieces of data. The first one is a choice of
perturbation data $(K,J)$ on each (decorated) punctured disk
$S_k \subset \cl_{\srf}$. The second one is a choice of Morse data on
each tree $T \subset \cl_{\tr}$ of the cluster. Recall that each such
tree $T$ corresponds to an underlying Lagrangian $\bar{L}$ that
appears in the decoration $\mathcal{L}$. Recall also that on each
Recall also that each edge $e$ of $T$ is parametrized by an interval
$I_e \subset \mathbb{R}$ (where the intervals for the internal edges
are closed of finite length and the ones corresponding to the edges
that touch the type-E leaves are semi-infinite). The Morse data for
$T$ is a choice of a family $(f_{\tau}, (\cdot, \cdot)_{\tau})$,
$\tau \in I_e$, for every edge $e$ of the tree, where for each $\tau$,
$f_{\tau} : \bar{L} \longrightarrow \mathbb{R}$ is a smooth function
and $(\cdot, \cdot)_{\tau}$ a Riemannian metric on $\bar{L}$. Here
$\bar{L}$ is the underlying Lagrangian corresponding to the tree $T$.
Moreover, we require that along the {\em ends} of the external edges
$e$ of $T$ (i.e.~the edges connected to the type-E leaves), the pair
$(f_{\tau}, (\cdot, \cdot)_{\tau})$ coincides with the Morse data
$(f_{\bar{L}}, (\cdot, \cdot)_{\bar{L}})$ associated to $\bar{L}$ that
has been fixed in advance. For example, if $e$ is an edge connected to
an entry leaf, then $I_e = (-\infty, 0]$, and the requirement is that
$(f_{\tau}, (\cdot, \cdot)_{\tau}) = (f_{\bar{L}}, (\cdot,
\cdot)_{\bar{L}})$ for $\tau \ll 0$. A similar choice of data is made
also in case there is an exit leaf (which is the case when
$\bar{L}_0 = \bar{L}_d$) only that now the edge $e$ connected to the
exit leaf is labeled by $I_e = [0, \infty)$.

There is only one slight exception to the above, namely when
$\bar{\mathcal{L}} = (\bar{L}, \bar{L})$. In this case the whole
cluster consists of only one tree (and no punctured disks). This tree
has two vertices and one edge $e$ connecting them, which is modeled on
the interval $I_e = (-\infty, \infty)$. The choice of Morse data here
will be the same Morse data $(f_{\bar{L}}, (\cdot, \cdot))$ chosen in
advance for $\bar{L}$ and it is required to be independent of the
parameter $\tau \in I_{e}$.

For every tuple $\bar{\mathcal{L}}$ we make a continuous choice of
perturbation data for all the clusters of punctured disks that are
parametrized by $\clus(\bar{\mathcal{L}})$. We denote such a choice by
$\mathscr{P}_{\bar{\mathcal{L}}}$ and denote by
$\mathscr{P} = \{\mathscr{P}_{\bar{\mathcal{L}}}\}$ the collection of
choices $\mathscr{P}_{\bar{\mathcal{L}}}$ made for all tuples
$\bar{\mathcal{L}}$ of any length. We refer to $\mathscr{P}$ as a
choice of perturbation data.

These choices of $\mathscr{P}$ \zjr{are} subject to being consistent with the
splitting, degeneration and gluing described
in~\S\ref{sbsb:splitting}. This is crucial in order to establish the
$A_{\infty}$-identities among the $\mu_d$-operations that will be
described next.

\subsubsection{The $\mu_d$-operations} \label{sbsb:mu-d} We now
proceed to the $A_{\infty}$-operations, taking also filtrations into
account. The definition of $\mu_d$, $d \geq 2$, is based on cluster of
Floer polygons which we now describe.

Let $\bar{\mathcal{L}} = (\bar{L}_0, \ldots, \bar{L}_d)$ be a tuple of
Lagrangians in $\bar{\mathcal{X}}$. An $\bar{\mathcal{L}}$-decorated
cluster of Floer polygons is a map $u: \cl \longrightarrow X$, whose
domain is an $\bar{\mathcal{L}}$-decorated cluster of punctured disks
$\cl$. The restriction $u|_{S_{k}}$ of $u$ to any of the punctured
disks $S_k \subset \cl_{\srf}$ is a Floer polygon, exactly as in the
standard theory of Fukaya categories~\cite[Chapter~2,
Section~9]{Seidel}. Namely, $u|_{S_{k}}$ satisfies the
(generalized) Floer equation associated to the perturbation data that
$\mathscr{P}_{\bar{\mathcal{L}}}$ assigns to $S_k$. The map $u|_{S_k}$
satisfies Lagrangian boundary conditions prescribed by the
decoration. The punctures of $S_k$ are sent by $u$ to intersection
points of pairs of Lagrangians, as prescribed by the decoration. In
addition, we assume that the energy $E(u|_{S_k})$ of $u|_{S_k}$ is
finite.

Next, the restriction $u|_{T}$ of $u$ to any of the trees
$T \subset \cl_{tr}$ should satisfy the {\em negative} gradient
equations corresponding to the Morse data specified along the edges of
the trees. The interval labeling $I_e$ on the edges $e$ of $T$ and the
identifications $\sigma_e^T: e \longrightarrow I_e$ are used in order
to endow each interval with a ``time-parameter'' for the negative
gradient trajectories. Finally, the type-E leaves of each tree $T$
from $\mathcal{T}$ are mapped by $u$ to critical points of the
functions $f_{\bar{L}}$, where $\bar{L}$ is the Lagrangian decorating
the tree $T$.

Given the choices of Floer data and perturbation data, the definition
of the $\mu_d$-operations, $d \geq 2$, is now done by counting
decorated clusters of Floer polygons with specified boundary
conditions and given entry/exit points. Specifically, let
$\mathcal{L} = (L_0, \ldots, L_d)$, $d \geq 1$, be a tuple of
Lagrangians from $\mathcal{X}$ and denote by
$\bar{\mathcal{L}} = (\bar{L}_0, \ldots, \bar{L}_d)$ the corresponding
tuple of underlying Lagrangians. Define
\begin{equation} \label{eq:mu-d}
  \begin{aligned}
    & \mu_d: CF(L_0, L_1) \otimes \cdots
      \otimes CF(L_{d-1}, L_d) \longrightarrow CF(L_0, L_d), \\
    & \mu_d(x_1, \ldots, x_d) := \sum_{y} \# \mathcal{M}(x_1,
      \ldots, x_d, y; \mathscr{P}) y.
  \end{aligned}
\end{equation}
Here we have abbreviated
$CF(L', L'') := CF(L', L''; \mathscr{D}_{L', L''})$ for any
$L', L'' \in \mathcal{X}$. The sum in the second line
of~\eqref{eq:mu-d} goes as follows: $y$ runs over all the generators
of $CF(L_0, L_d)$ of appropriate degree and
$\# \mathcal{M}(x_1, \ldots, x_d, y; \mathscr{P})$ stands for the
count (with values in $\mathbb{Z}_2$, or under \zhnote{additional} assumption in
$\k$) of the number of elements in the $0$-dimensional component of
the space $\mathcal{M}(x_1, \ldots, x_d, y; \mathscr{P})$ of
$\bar{\mathcal{L}}$-decorated clusters of Floer polygons with entry
points $x_1, \ldots, x_d$ and exit point $y$.

We denote by $\fuk(\mathcal{X}; \mathscr{P})$ the collection of
objects $\mathcal{X}$ together with the multilinear operations
$\mu_d$, $d \geq 1$, associated to the Floer and perturbation data
$\mathscr{P}$.  Notice that the perturbation data depends only on the
geometric part of the marked Lagrangians, \zhnote{namely $\bar{\mathcal{X}}$}.

\begin{thm} \label{t:fil-fuk} For every finite collection of
  Lagrangians $\bar{\mathcal{X}}$, satisfying the conditions from the
  beginning of~\S\ref{sb:faitpc}, there \zjr{exists} a (non-empty) space
  $\mathcal{B}(\bar{\mathcal{X}})$ of regular Floer and perturbation
  data, of the types described above, such that for every
  $\mathscr{P} \in \mathcal{B}(\bar{\mathcal{X}})$ the following
  holds:
  \begin{enumerate}
  \item $\fuk(\mathcal{X}; \mathscr{P})$, with the above
    $\mu_d$-operations, $d \geq 1$, is a strictly unital
    $A_{\infty}$-category. Moreover, with the filtrations defined
    above in~\S\ref{sbsb:fdata}, this $A_{\infty}$-category is
    genuinely filtered.
  \item If one forgets the filtrations, then
    $\fuk(\mathcal{X}; \mathscr{P})$ as defined above is
    quasi-equivalent to the subcategory of the standard Fukaya
    category (e.g.~as defined in~\cite{Seidel}) whose
    collection of objects is $\mathcal{X}$. \zhnote{This quasi-equivalence can be assumed to be the identity map on the set of objects $\mathcal{X}$.}
  \end{enumerate}
  Moreover, there exist filtered $A_{\infty}$-functors 
  $$\mathcal{F}^{\mathscr{P}_1, \mathscr{P}_0} : \fuk(\mathcal{X};
  \mathscr{P}_0) \longrightarrow \fuk(\mathcal{X}; \mathscr{P}_1),$$
  defined for every
  $\mathscr{P}_0, \mathscr{P}_1 \in \mathcal{B}(\bar{\mathcal{X}})$,
  with the following properties:
  \begin{enumerate}
  \item $\mathcal{F}^{\mathscr{P}_1, \mathscr{P}_0}$ is $A_1$-unital
    (see~\S\ref{sbsb:coh-sys} for the definition).
  \item $\mathcal{F}^{\mathscr{P}_1, \mathscr{P}_0}$ is a filtered
    quasi-equivalence.
  \item The action of $\mathcal{F}^{\mathscr{P}_1, \mathscr{P}_0}$ on
    objects is the identity map. (Recall that all the categories
    $\fuk(\mathcal{X}; \mathscr{P})$ have the same set of objects
    $\mathcal{X}$.)
  \item For every $L', L'' \in \mathcal{X}$, the maps
    $(\mathcal{F}_1^{\mathscr{P}_1, \mathscr{P}_0})_*: HF(L', L'';
    \mathscr{P}_0) \longrightarrow HF(L', L''; \mathscr{P}_1)$ induced
    by the first order components of
    $\mathcal{F}^{\mathscr{P}_1, \mathscr{P}_0}$ are the canonical
    continuation isomorphisms in Floer theory.
  \item $\mathcal{F}^{\mathscr{P}, \mathscr{P}} = \id$.
  \item
    $\mathcal{F}^{\mathscr{P}_2, \mathscr{P}_1} \circ
    \mathcal{F}^{\mathscr{P}_1, \mathscr{P}_0}$ is isomorphic to
    $\mathcal{F}^{\mathscr{P}_2, \mathscr{P}_0}$ in
    \zhnote{$H^0(\text{{ffun}}(\fuk(\mathcal{X}; \mathscr{P}_0),
    \fuk(\mathcal{X}; \mathscr{P}_2)))_0$}. Here,
    $\text{{ffun}}(\fuk(\mathcal{X}; \mathscr{P}_0), \fuk(\mathcal{X};
    \mathscr{P}_2))$ is the category of filtered $A_{\infty}$-functors
    from $\fuk(\mathcal{X}; \mathscr{P}_0)$ to
    $\fuk(\mathcal{X}; \mathscr{P}_2)$, $H^0(\text{{ffun}} (\cdots))$
    is the persistence homological category of $\text{{ffun}}$ in
    cohomological degree $0$, and $H^0(\text{{ffun}}(\cdots))_0$ is
    its \zhnote{$0$-level persistence subcategory}. In other words, there
    exists an $A_{\infty}$-natural transformation
    $T^{\mathscr{P}_2, \mathscr{P}_1, \mathscr{P}_0}:
    \mathcal{F}^{\mathscr{P}_2, \mathscr{P}_1} \circ
    \mathcal{F}^{\mathscr{P}_1, \mathscr{P}_0} \longrightarrow
    \mathcal{F}^{\mathscr{P}_2, \mathscr{P}_0}$ which preserves
    filtrations and is an isomorphism in the homological persistence
    category of filtered functors. Note that, in particular, this
    implies that
    $\mathcal{F}^{\mathscr{P}_0, \mathscr{P}_1} \circ
    \mathcal{F}^{\mathscr{P}_1, \mathscr{P}_0}$ is isomorphic to
    $\id_{\fuk(\mathcal{X}; \mathscr{P}_0)}$ in the respective
    homological persistence category.
  \end{enumerate}
  Furthermore, the choice of the assignment
  $\bar{\mathcal{X}} \longmapsto \mathcal{B}(\bar{\mathcal{X}})$ can
  be assumed to have the following property: if $\bar{\mathcal{X}}'$
  is another finite collection of Lagrangians with
  $\bar{\mathcal{X}}' \supset \bar{\mathcal{X}}$\zhnote{, that (similarly to
  $\bar{\mathcal{X}}$)} satisfies the conditions from the beginning
  of~\S\ref{sb:faitpc}, then
  $\mathcal{B}(\bar{\mathcal{X}}')|_{\bar{\mathcal{X}}} \subset
  \mathcal{B}(\bar{\mathcal{X}})$. Here,
  $\mathcal{B}(\bar{\mathcal{X}}')|_{\bar{\mathcal{X}}}$ stands for
  the restriction of the perturbation data from
  $\mathcal{B}(\bar{\mathcal{X}}')$ to the spaces of clusters of
  punctured disks decorated by the elements of $\bar{\mathcal{X}}$.
\end{thm}

\begin{rem}
  The system of functors
  $\{\mathcal{F}^{\mathscr{P}_1, \mathscr{P}_0}\}_{\mathscr{P}_1,
    \mathscr{P}_0 \in \mathcal{B}(\bar{\mathcal{X}})}$ and the natural
  transformations mentioned in Theorem~\ref{t:fil-fuk} depend on a
  variety of choices hence are not canonical in the strict sense of
  the word. The extent to which these structures are canonical will be
  briefly discussed  later in Remark~\ref{r:can-sys}. We will refer to this 
  system of functors as a {\em weakly coherent system} of comparison functors to emphasize that our construction does not produce canonical choices. 
\end{rem}

We proceed now to the proof of Theorem~\ref{t:fil-fuk}. The proof
presented below is by no means complete and should be viewed as an
outline only. We have left out quite a few technical details,
especially concerning the analysis underlying the proof. However,
these parts of the proof follow from rather standard and well
established ingredients in the analysis of Floer theory and Fukaya
categories. As mentioned earlier, a more general approach to genuinely
filtered Fukaya categories is worked out in a forthcoming paper by
Ambrosioni~\cite{Amb:fil-fuk} which will also contain a detailed proof
of the construction.

\subsection{Proof of Theorem~\ref{t:fil-fuk} - part
  1} \label{sbsb:prf-fil-fuk-1}

We will concentrate here on the second part of point~(1) of the
theorem (namely that $\fuk(\mathcal{X})$ can be made {\em genuinely
  filtered} for appropriate choices of Floer and perturbation
data). The first part of point~(1) as well as point~(2) are rather
known and have been addressed in the literature in various levels of
rigor. The proofs of the statements concerning the functors
$\mathcal{F}^{\mathscr{P}_1, \mathscr{P}_0}$ will be outlined
in~\S\ref{sbsb:coh-sys} below.

Throughout the proof we will sometimes abbreviate
$\fuk(\mathcal{X}; \mathscr{P})$ as $\fuk(\mathcal{X})$ in case
$\mathscr{P}$ is clear from the context.

Fix a tuple of Lagrangians
$\bar{\mathcal{L}} = (\bar{L}_0, \ldots, \bar{L}_d)$ and assume for
simplicity that $\bar{L}_i \neq \bar{L}_{j}$ for every $i \neq j$.
Fix a tuple of intersection points $x_1, \ldots, x_d, y$,
$x_i \in \bar{L}_{i-1} \cap \bar{L}_i$,
$y \in \bar{L}_0 \cap \bar{L}_d$. Let
$u \in \mathcal{M}(x_1, \ldots, x_d, y; \mathscr{P})$ be an
$\bar{\mathcal{L}}$-decorated cluster of Floer polygons. By our
simplifying assumptions the domain of $u$ must be a simple cluster,
namely just one punctured disk $S_{d+1}$. (See
Page~\pageref{pp:simple-cluster} and Remark~\ref{r:decorations}.)
Thus $u: S_{d+1} \longrightarrow X$ is a Floer polygon that sends the
punctures of $S_{d+1}$ to the points $x_1, \ldots, x_d, y$. For such a
map $u$ denote by $A(u) = \int_{S_{d+1}} u^* \omega$ the symplectic
area of $u$. We have
$$A(u) = \sum_{i=1}^d \mathcal A(x_i) - \mathcal A(y).$$ In order to
prove that the $\mu_d$-operation preserves filtrations we need to show
that for all $u$ as above we have $A(u) \geq 0$. (In fact we need to
show the latter inequality holds for all \zjr{clusters} of Floer polygons
$u$, not only the simple ones. However, as we will see below, the main
difficulty is for punctured disks, and the generalization to more
general clusters is straightforward.)

Before we go into the proof of the latter statement, let us explain
the difficulties underlying it. Obviously $A(u) \geq 0$ if we choose
the perturbation data $(K,J)$ with $K \equiv 0$, since then every
Floer polygon will be $J$-holomorphic hence of strictly positive
area. However, for a variety of reasons it seems better to allow for
non-trivial $1$-forms $K$ in the perturbation data, so we will not
assume $K \equiv 0$. One of the reasons for allowing non-trivial
perturbations is that it is easier to establish transversality for the
spaces of Floer clusters with this extra parameter at hand. Another
reason is that if one hopes to generalize the present approach to
cases when not all pairs of Lagrangians in $\bar{\mathcal{X}}$
intersect transversely then Hamiltonian perturbations would definitely
be needed. Other reasons have to do with compatibility of
$\fuk(\mathcal{X})$ with other structures such as the Floer (or
symplectic) homology of the ambient manifold and maps relating
$\fuk(\mathcal{X})$ to these invariants. These structures usually
require Hamiltonian perturbations.

As a second attempt, we may try to show that $A(u) \geq 0$ once we
take the $1$-forms $K$ in the perturbation data to be small enough.
At first sight this seems to work using a compactness argument as
follows. Indeed, if this were not the case, then we would have a
sequence of perturbation forms $K^{(l)}$ and a sequence
$u_l:S_{d+1}^{(l)} \longrightarrow X$ of corresponding Floer polygons
whose punctures go to a fixed set of intersection points
$x_1, \ldots, x_d, y$, such that $K^{(l)} \xrightarrow[C^1]{\quad} 0$
but $A(u_l) \leq 0$ for every $l$. (We may assume that all the $u_l$'s
run between the {\em same} set of intersection points
$x_1, \ldots, x_d, y$ because by our assumptions there is only a
finite number of possible intersections points associated to the tuple
$\bar{\mathcal{L}}$.) By compactness, passing to a subsequence of the
$u_l$'s we would then obtain a (possibly broken) limit polygon $u$
which is non-constant and genuinely $J$-holomorphic, yet with
$A(u) \leq 0$. A contradiction. Since, by assumption, the number of
possible $d+1$ tuples $\bar{\mathcal{L}}$ is finite, it follows that
if we take the perturbation data small enough, then for all tuples
$\bar{\mathcal{L}}$ of length $d+1$, with the properties from the
beginning of the proof, and all $\bar{\mathcal{L}}$-decorated Floer
polygons $u$ we have $A(u) > 0$. This easily extends also to decorated
clusters of Floer polygons that are not necessarily simple, as well as
to all decorations $\bar{\mathcal{L}}$ of {\em fixed} length $d+1$.
    
The problem with this argument is that without further elaboration it
might create difficulties with obtaining a consistent choice of
perturbation data $\mathscr{P}$. To explain this difficulty let us
rephrase the previous paragraph in more quantitative terms. From now
on we will use the following more detailed notation. We denote the
restriction of the perturbation data $\mathscr{P}$ to the space of all
$(d+1)$-punctured disks by $\mathscr{P}_{d+1}$ and write
$(K(\mathscr{P}_{d+1}), J(\mathscr{P}_ {d+1}))$ for the two components
of $\mathscr{P}_{d+1}$. Given an $\bar{\mathcal{L}}$-decorated
punctured disk $S_{d+1}$ we denote by
$(K(\mathscr{P}, S_{d+1}), J(\mathscr{P}, S_{d+1}))$ the restriction
of $\mathscr{P}$ to $S_{d+1}$.

The previous argument shows that there exist numbers
$\varepsilon_{d+1} > 0$, $d\geq 2$, such that if the perturbation
forms $K(\mathscr{P}_{d+1})$ satisfy
$\|K(\mathscr{P}_{d+1})\| \leq \varepsilon_{d+1}$ then for all tuples
$\mathcal{L}$ of length $d+1$ and all $\bar{\mathcal{L}}$-decorated
Floer polygons $u$, we have $A(u)> 0$. Here $\| - \|$ is a suitable
norm on the space of all perturbation $1$-forms $K_{d+1}$ (defined on
the space of all possible $(d+1)$-punctured disks $S_{d+1}$). The
value of $\|K_{d+1}\|$ involves the values of $K_{d+1}$ and its first
derivatives, both in the domain direction as well as in the direction
of the manifold $X$ (recall that the forms $K_{d+1}$ have values in
the space of compactly supported functions on $X$).

The problem that arises with the approach used so far has to do with
the consistency of $\mathscr{P}$ with respect to gluing/splitting.  A
standard way to construct consistent perturbation data is to construct
$K(\mathscr{P}_{d+1})$ (and the almost complex structures) by
induction over $d$ and make sure that at each induction step the newly
defined perturbation data is consistent with the data that have
already been defined at earlier stages, with respect to
gluing/splitting. Assume that $K(\mathscr{P}_{m+1})$ has already been
defined for all $m \leq d_0$ in such a way that
$\|K(\mathscr{P}_{m+1})\| \leq \varepsilon_{m+1}$ for all
$m \leq d_0$. Consider now $d = d_0 + 1$. Punctured disks of the type
$S_{d+1}$ can split into two punctured disks of the type $S_{d'+1}$
and $S_{d''+1}$ with $d'+d''=d+1$ (and $d', d'' \geq 2$). See
Figure~\ref{f:gluing-splitting}.
  
\begin{figure}[h]
  \includegraphics[scale=0.80]{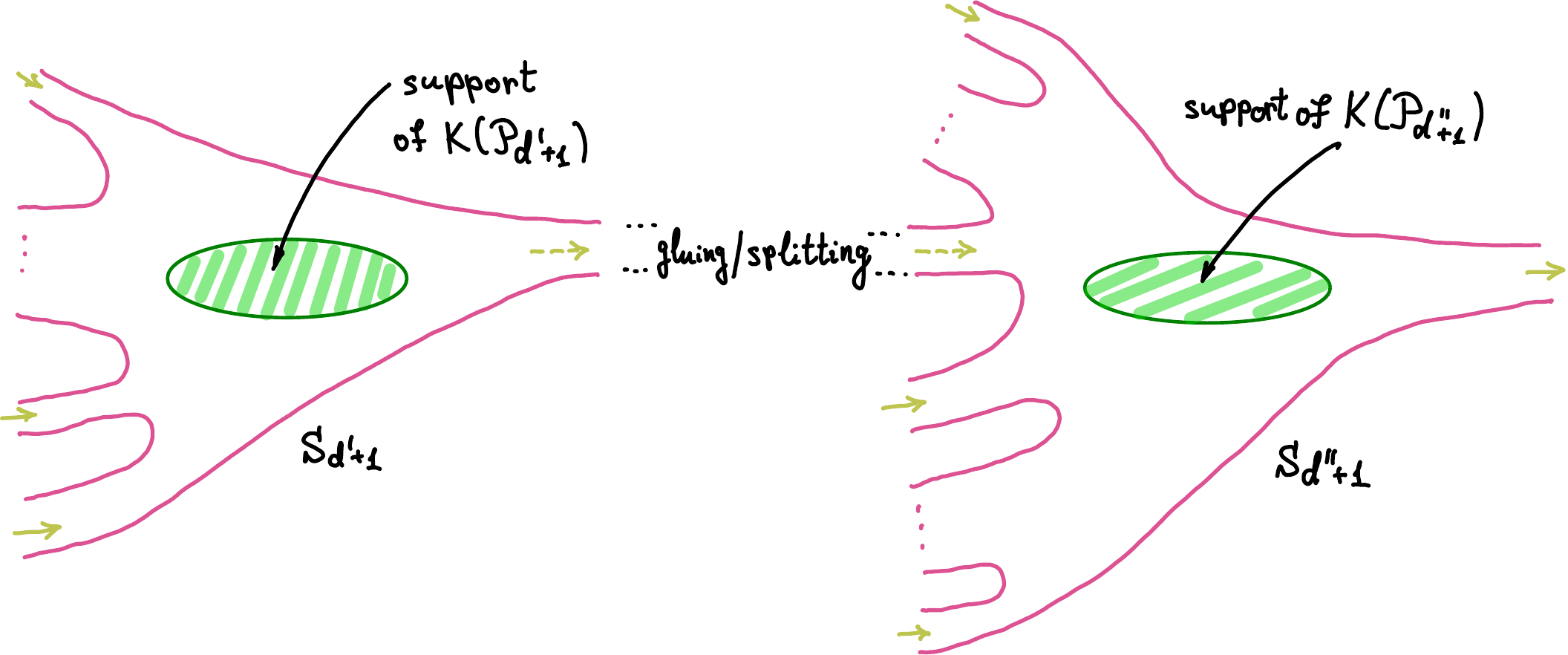} \centering
  \caption{Gluing/splitting of punctured disks. Two punctured disks of
    the types $S_{d'+1}$ and $S_{d''+1}$, together with their
    perturbation $1$-forms, are glued into a punctured disk $S_{d+1}$,
    where $d+1 = d'+d''$.}
  \label{f:gluing-splitting}
\end{figure}

If it so happens that
\zjr{\begin{equation}
  \varepsilon_{d+1} < \|K(\mathscr{P}_{d'+1})\| + \|
  K(\mathscr{P}_{d''+1})\|
\end{equation}}
then the induction step will produce perturbations forms
$K(\mathscr{P}_{d+1})$ that might not satisfy the condition
$\|K(\mathscr{P}_{d+1})\| \leq \varepsilon_{d+1}$ which is required in
order to have $A(u)> 0$ for all Floer polygons $u$. In other words,
the construction described above does not go through at the induction
step. Of course, the above argument shows that for every fixed $d$, we
can choose consistent perturbation data such that all the Floer
polygons involving no more than $d+1$ intersection points have
positive symplectic area. One could then easily modify this argument
to handle more general clusters and consequently show that all the
$\mu_k$-operation in $\fuk(\mathcal{X})$ with $k \leq d$ will preserve
filtrations. However, without any additional arguments, the above
fails to prove that there is even one consistent perturbation data
that will turn the $\mu_k$-operations into filtration preserving maps
{\em for all $k$}. We will now refine the above arguments, showing how
to achieve our goal by restricting further the type of perturbation
data.

\label{p:preserve-fil-1}
Recall that for all $\mathcal{L}$-decorated Floer polygons
$u:S_{d+1} \longrightarrow X$ we have the following energy-area
identity:
\begin{equation} \label{eq:energy-area} E(u) = A(u) - \int_{S_{d+1}}
  R_{K(\mathscr{P}, S_{d+1})} \circ u,
\end{equation}
where $E(u)$ is the energy of $u$ and $R_{K(\mathscr{P}, S_{d+1})}$ is
the curvature associated to the $1$-form $K(\mathscr{P}, S_{d+1})$.
Recall that $R_{K(\mathscr{P}, S_{d+1})}$ is a $2$-form on $S_{d+1}$
with values in the space of compactly supported functions on $X$. The
expression $R_{K(\mathscr{P}, S_{d+1})} \circ u$ is a real valued
$2$-form on $S_{d+1}$ which is obtained by composing the functions
prescribed by the values of $R_{K(\mathscr{P}, S_{d+1})}$ with the map
$u$. We refer the reader to~\cite[Chapter~II, Section~(8g)]{Seidel}
for more details on the definition of $R_{K(\mathscr{P}, S_{d+1})}$
and the identity~\eqref{eq:energy-area}.

An important point about the curvature form
$R_{K(\mathscr{P}, S_{d+1})}$ is that it can be made arbitrarily small
by choosing the perturbation form $K(\mathscr{P}, S_{d+1})$ to be
small enough in an appropriate $C^1$-norm $\| - \|$. As alluded to
above, $\|K(\mathscr{P}, S_{d+1})\|$ involves the sum of the
$L^1$-norms of $K(\mathscr{P}, S_{d+1})$ and its derivatives, both in
the direction of $S_{d+1}$ as well as in the direction of the manifold
$X$. (This norm can be viewed as a version of the Sobolev norm on
$W^{1,1}$.) We therefore have
$$\Bigr | \int_{S_{d+1}} R_{K(\mathscr{P}, S_{d+1})} \circ v
\; \Bigl | \leq C_{d+1} \|K(\mathscr{P}, S_{d+1})\|,$$ for all maps
$v: S_{d+1} \longrightarrow X$, where the constant $C_{d+1}$ depends
only on $d$ (and not on the specific surface $S_{d+1}$ or any other
parameter involved in the integrand). Define
\begin{equation} \label{eq:K-norm-sup}
  \begin{aligned}
    \|K(\mathscr{P}_{d+1})\| & := \sup \; \| K(\mathscr{P}, S_{d+1})\|, \\
    \nu(\mathscr{P}_{d+1}) & := C_{d+1} \|K(\mathscr{P}_{d+1})\|,
  \end{aligned}
\end{equation}
where the supremum in the first formula is taken over all
$\bar{\mathcal{L}}$-decorated punctured disks $S_{d+1}$. Note that it
is possible to choose the perturbation forms $K(\mathscr{P}, S_{d+1})$
such that they are all compactly supported inside the interior of each
punctured disk $S_{d+1}$ and moreover we can control these supports
such that the support of the entire family $K(\mathscr{P}_{d+1})$ is
compact. (This poses no problems to having consistency with
gluing/splitting.) This implies that the supremum that appears in the
definition of \zjr{$\|K(\mathscr{P}_{d+1})\|$} in~\eqref{eq:K-norm-sup} can
be assumed to be finite, and moreover can be made arbitrarily small by
an appropriate choice of $K(\mathscr{P}_{d+1})$.
  
With the above notation we now have
\begin{equation} \label{eq:A-nu} A(u) \geq E(u) -
  \nu(\mathscr{P}_{d+1})
\end{equation}
for all $\bar{\mathcal{L}}$-decorated Floer polygons $u$.

We now add further restrictions to the perturbation data
$\mathscr{P}$. Denote by
$\bar{\mathcal{X}} = \{\bar{L} \mid L \in \mathcal{X} \}$ the set of
all underlying Lagrangians from the collection $\mathcal{X}$. Denote
by $I = \{p_1, \ldots, p_N\}$ the set of all intersection points
between any two distinct Lagrangians
$\bar{L}, \bar{L}' \in \bar{\mathcal{X}}$. By our assumption on
$\bar{\mathcal{X}}$, $I$ this is a finite set. Moreover, every
$p \in I$ corresponds precisely to one pair of distinct Lagrangians
$\bar{L}, \bar{L}' \in \bar{\mathcal{X}}$ whose intersection contains
$p$.

Denote by $B^{2n}(R) \subset \mathbb{R}^{2n}$ the closed
$2n$-dimensional ball endowed with its standard symplectic structure,
where $2n = \dim_{\mathbb{R}} X$. We claim that there exists a
symplectic embedding of a disjoint union of $N$ balls of some radius
$R_{\bar{\mathcal{X}}}>0$ into $X$ \label{pp:radius}
\begin{equation} \label{eq:phi-embedd} \phi: \bigsqcup_{j=1}^N
  B^{2n}(R_{\bar{\mathcal{X}}}) \to X,
\end{equation}
with the following properties:
\begin{enumerate}
\item $\phi(0_j) = p_j$ for every $j$. Here and in what follows we
  denote by $B_j$ the $j$'th ball in the disjoint union
  in~\eqref{eq:phi-embedd} and by $0_j \in B_j$ the center of that
  ball.
\item If $p_j \in \bar{L}' \cap \bar{L}''$ with
  $\bar{L}', \bar{L}'' \in \bar{\mathcal{X}}$ distinct, then
  $$\text{either } \phi(\mathbb{R}B_j) \subset \bar{L}',
  \; \phi(i\mathbb{R}B_j) \subset \bar{L}'', \;\; \text{ or }
  \phi(\mathbb{R}B_j) \subset \bar{L}'', \; \phi(i\mathbb{R}B_j)
  \subset \bar{L}', $$ where
  $\mathbb{R}B_j := B_j \cap (\mathbb{R}^n \times \{0\})$ is the
  ``real'' part of the ball $B_j$ and
  $i\mathbb{R}B_j := B_j \cap (\{0\} \times \mathbb{R}^n)$ is its
  ``imaginary'' part.  Moreover,
  $(\phi|_{B_j})^{-1}(\bar{L}' \cup \bar{L}'') \subset \mathbb{R}B_j
  \cup i \mathbb{R} B_j$, and $\phi(B_j)$ is disjoint from all the
  Lagrangians in
  $\bar{\mathcal{X}} \setminus \{\bar{L}', \bar{L}''\}$.
\end{enumerate}

\label{pp:balls-R}
Clearly there exists an $R_{\bar{\mathcal{X}}}>0$ and an embedding
$\phi$ as above, and we fix both of them once and for all. Next we fix
$0<R<R_{\bar{\mathcal{X}}}$ and denote by $B_j(R) \subset B_j$ the
smaller ball of radius $R$. Fix another parameter $0 < r \leq R$ and
impose the following further restriction on the perturbation data
$\mathscr{P}_{d+1} = (K(\mathscr{P}_{d+1}), J(\mathscr{P}_{d+1}))$ for
all $d$:
\begin{enumerate}
\item \label{i:K-phi-R} $K|_{\phi(B_j(R))} \equiv 0$ for every $j$,
  i.e.~the values of the $1$-form $K$ (which are compactly supported
  functions on $X$) vanish over the image of the restriction of $\phi$
  to the smaller balls of radius $R$.
\item \label{i:J-phi-R} $J_z|_{\phi(B_j(R))} = \phi_*(J_{\text{std}})$
  for every $j$ and every $z \in S_{d+1}$, $d \geq 2$. Here
  $J_{\text{std}}$ is the standard complex structure on $B_j(R)$.
\end{enumerate}

Note that the above additional restrictions on $(K,J)$ do not
contradict any of our previous assumptions on $(K,J)$, and if we
temporarily ignore the size of the perturbation forms $K$, then the
above also do not pose any problems to the consistency of
$\mathscr{P}$ with respect to gluing/splitting.

We now claim that there is a {\em consistent} choice of perturbation
data $\mathscr{P}$ of the type describe above such that for all Floer
polygons $u$ we have $A(u)>0$.

To prove this consider an $\bar{\mathcal{L}}$-decorated Floer polygon
$u: S_{d+1} \longrightarrow X$ associated with our perturbation
data. Denote by $x_1, \ldots, x_d, y$ the intersection points to which
the punctures of $S_{d+1}$ are mapped to by $u$. By construction, $u$
is genuinely holomorphic over $\textnormal{image\,}(\phi)$ with
respect to an almost complex structure that is diffeomorphic along
that region to $J_{\text{std}}$ via $\phi$. Denote by
$B_{j_1}, \ldots, B_{j_d}, B_{j_{d+1}}$ the balls corresponding to the
intersection points $x_1, \ldots, x_d, y$ according to the
construction made earlier. Applying a version of the monotonicity
lemma (or alternatively a version of the Lelong inequality) to $u$
over $\phi(B_{j_k})$, $k=1, \ldots, d+1$, we obtain that there exists
a constant $C$ that does not depend on $d$ (nor on $u$ or on any other
parameter form the perturbation data) such that
\begin{equation} \label{eq:E-monotonicity-bound} E(u) \geq (d+1) CR^2.
\end{equation}

Putting this together with~\eqref{eq:A-nu} we obtain:
\begin{equation} \label{eq:A-nu-2} A(u) \geq (d+1)CR^2 -
  \nu(\mathscr{P}_{d+1}).
\end{equation}
Obviously, if we choose perturbation data $\mathscr{P}$ that are small
enough \zhnote{such} that
\begin{equation} \label{eq:nu-cr0} \nu(\mathscr{P}_{d+1}) <
  (d+1)Cr^2,
\end{equation}
then $A(u)>0$ for all Floer polygons $u$. (Recall that we fixed the
parameter $r$ such that $r \leq R$.) The main thing that needs to be
verified now is that the condition~\eqref{eq:nu-cr0} still enables a
choice of perturbation data that are consistent with
gluing/splitting. We address this point next.

We will choose the perturbation \zjr{data} $\mathscr{P}$ such that for all
$d$ we have $\nu(\mathscr{P}_{d+1}) < \alpha_{d+1} (d+1)C r^2$ for
some $0< \alpha_{d+1} \leq 1$. Recall that punctured disks of the type
$S_{d+1}$ can split into two punctured disks of the type $S_{d'+1}$
and $S_{d''+1}$ with $d'+d''=d+1$ and $d', d'' \geq 2$. (Of course,
splittings into more than two punctured disks is also possible,
however for the sake of obtaining a filtered $A_{\infty}$-category the
top strata of the boundary of the space of disks matters, and these
correspond to splitting into two punctured disks only.) And
vice-versa, by gluing two punctured disks of the type $S_{d'+1}$,
$S_{d''+1}$ we obtain punctured disks of the type
$S_{d'+d''}$. Therefore, in order to make it possible to construct the
perturbation data $\mathscr{P}_{d+1}$ by induction on $d$ as indicated
earlier, and to make them consistent with respect to gluing/splitting
we need to find a sequence of numbers $\alpha_{d+1}$, $d \geq 2$, that
satisfy the following set of inequalities:
\begin{equation} \label{eq:ineq-alpha-d-1}
  \begin{aligned}
    & \alpha_{d'+1} (d'+1)C r^2 + \alpha_{d''+1} (d''+1)C r^2
    \leq \alpha_{d'+d''} (d'+d'')C
    r^2 \; \; \; \forall \; d', d'' \geq 2, \\
    & 0 < \alpha_{d+1} \leq 1 \;\; \forall \; d \geq 2.
  \end{aligned}
\end{equation}
If such a sequence of numbers $\alpha_k$ does exist then we simply
construct the perturbation data $\mathscr{P}_{d+1}$ by induction on
$d$, where at each induction step we require that
$\nu(\mathscr{P}_{d+1}) < \alpha_{d+1} (d+1)C r^2$. The
inequalities~\eqref{eq:ineq-alpha-d-1} will then assure that the
induction step goes through without posing problems to consistency
with respect to gluing/splitting.

It remains to show that the inequalities in~\eqref{eq:ineq-alpha-d-1}
admit solutions. Setting $\beta_k := k \alpha_k$ for every $k \geq 3$,
the set of inequalities~\eqref{eq:ineq-alpha-d-1} can be simplified
to:
\begin{equation} \label{eq:ineq-alpha-d-2}
  \begin{aligned}
    & \beta_{d'+1} + \beta_{d''+1} \leq \beta_{d'+d''}
    \; \; \; \forall \; d', d'' \geq 2, \\
    & 0 < \beta_{k} \leq k \;\; \forall \; k \geq 3.
  \end{aligned}
\end{equation}
It is easy to see that this set of inequalities does have solutions.
For example, for the sequence $\beta_k := (3k-6)B$, $k \geq 3$, where
$B \leq \tfrac{1}{3}$, the first inequality
in~\eqref{eq:ineq-alpha-d-2} becomes an equality and the second
inequality is satisfied. One can also find sequences for which all
inequalities in~\eqref{eq:ineq-alpha-d-2} become strict. This can be
done as follows. Let $\beta'_k$ be any sequence for which the first
line in~\eqref{eq:ineq-alpha-d-2} is an equality and the second
inequality holds (e.g.~the preceding sequence $(3k-6)B$ with
$B \leq \tfrac{1}{3}$). Let
$\sigma: [3, \infty) \longrightarrow \mathbb{R}$ be a strictly
increasing function with $0 < \sigma(x) < 1$ for every $x$. Define now
$\beta_k := \sigma(k)\beta'_k$ for all $k\geq 3$. A straightforward
calculation shows that for this sequence all the inequalities
in~\eqref{eq:ineq-alpha-d-2} become strict.

To conclude the proof regarding the preservation of filtrations of the
$\mu_k$-operations we need to address also the case of non-simple
clusters of Floer polygons. The argument in this case is essentially
the same, and below we will only outline it in the case of a cluster
consisting of at most one Floer polygon with possibly several trees
attached to its boundary.

Consider first a tuple
$\bar{\mathcal{L}} = (\bar{L}_0, \ldots, \bar{L}_d)$ in which not all
the Lagrangians coincide.  Fix a tuple $x_1, \ldots, x_d, y$ of
generators, with
$x_i \in CF(L_{i-1}, L_i; \mathscr{D}_{L_{i-1}, L_i})$,
$y \in CF(L_0, L_d; \mathscr{D}_{L_0, L_d})$. (Recall that some of
these points are intersection points between the respective underlying
Lagrangians, and some are critical points of the Morse functions
prescribed by the Morse data). Let $\Sigma$ be an
$\bar{\mathcal{L}}$-decorated cluster of punctured disks which
consists of one punctured disk $S_{d'+1}$, where $d' \leq d$, and
several trees. Let $u: \Sigma \longrightarrow X$, be an
$\bar{\mathcal{L}}$-decorated cluster of Floer polygons with entry
points $x_1, \ldots, x_d$ and exit point $y$.  Denote by
$u':= u|_{S_{d'+1}}$ the restriction of $u$ to the underlying
punctured disk $S_{d'+1}$ (which is a genuine Floer polygon with
boundary conditions prescribed by a sub-tuple of $\bar{\mathcal{L}}$).

Recalling our filtration conventions for $CF(L', L'')$ in case
$\bar{L}' = \bar{L}''$ (see at the end of~\S\ref{sbsb:fdata}), a
simple calculation shows that
$$\sum_{i=1}^d \mathcal{A}(x_i) - \mathcal{A}(y) = A(u').$$
Now essentially the same argument as the one carried out earlier shows
that the perturbation data can be chosen such that $A(u') \geq 0$.

Let us also consider the case when all the Lagrangians in
$\bar{\mathcal{L}}$ coincide. In that case a cluster of Floer polygons
is just a collection of Morse trajectories modeled on a tree (without
any actual polygons). A simple calculation shows that in this case we
have:
$$\mathcal{A}(y) = \sum_{i=1}^d \mathcal{A}(x_i),$$
which implies the filtration is preserved by the operations $\mu_d$,
$d \geq 2$, also in case when the Lagrangians in $\bar{\mathcal{L}}$
all coincide.

\label{p:preserve-fil-2}
As mentioned above, these arguments easily generalize to more
complicated clusters, only that the notation becomes more involved.

\subsubsection*{Transversality.}

We will briefly address now the topic of transversality. In order to
show that $\fuk(\mathcal{X})$ is indeed an $A_{\infty}$-category
(e.g.~that $\mu_d$-operations satisfy the $A_{\infty}$-identities
etc.), one needs to choose perturbation data $\mathscr{P}$ that
satisfy various regularity properties. This would ensure that the
spaces of clusters of Floer polygons involved in the definition of the
$\mu_d$-operations are smooth manifolds and have other desirable
properties. Establishing the existence of regular perturbation data
(or other auxiliary data) usually goes by the name ``transversality'',
and is carried out via analytic techniques that have become standard
in Floer theory. The typical result in this context is that the set of
regular perturbation data is residual (in particular, dense) inside
the space of all consistent perturbation data. However, in our case we
work within a much more restricted space of perturbation data, as
described above (e.g.~specific choices of almost complex structures
near the intersection points between pairs of distinct Lagrangians,
perturbation forms that are compactly supported etc.). Formally
speaking, one would need to work out the transversality for our
choices of perturbation data. While this does not formally follow from
the general transversality theorems, it can still be achieved by
rather standard arguments.  For example, the fact that we restrict the
almost complex structures to be constant \zjr{on certain regions} does not pose any problems
(for achieving regularity) as long as the images of all the Floer
polygons pass through regions of $X$ in which we are allowed to
perturb the almost complex structures without any restrictions. The
same goes for the perturbation $1$-forms $K$. In a similar vein, the
fact that our perturbation forms $K$ must be chosen to be small enough
does not pose any transversality problem either. The only ingredient
that requires slightly different transversality arguments is the part
that uses the combination of Morse trees and Floer
polygons. Transversality for spaces of Morse trees, as well as Morse
trees mixed with holomorphic curves has been worked out in various
setups - see \cite{Fuk-Oh:zero-loop}, \cite{Charest-thesis},
\cite{Charest}.  While we do not provide here details for these
arguments, the case of clusters of Floer polygons follows from the
source space description for clusters that appears in
\cite{Charest-thesis} and standard regularity arguments, as outlined
above. Because all our Lagrangians are exact the actual regularity
arguments are much simpler compared to the ones developed in
\cite{Charest-thesis}, \cite{Charest}.

\subsubsection*{The spaces
  $\mathcal{B}(\bar{\mathcal{X}})$} \label{pp:space-B} So far the
proof above shows that there exist regular perturbation data
$\mathscr{P}$ that turn $\fuk(\mathcal{X}; \mathscr{P})$ into a
genuinely filtered, strictly unital, $A_{\infty}$-category.

We now define the space $\mathcal{B}(\bar{\mathcal{X}})$. Given
\zhnote{$r \leq R \leq R_{\bar{\mathcal{X}}}$} (recall that $R_{\bar{\mathcal{X}}}$ is
the radius of the balls in the embedding $\phi$), we
denote by $\mathcal{B}(\bar{\mathcal{X}}; R, r)$ the space of
consistent, regular perturbation data $\mathscr{P}$, as defined
earlier in the proof \zhnote{with the parameters $r$ and $R$}. The significance of the parameter $R$ appears in
conditions~\eqref{i:K-phi-R}-\eqref{i:J-phi-R} on
page~\pageref{i:K-phi-R} regarding the perturbation forms and almost
complex structures on $\cup_{j} \phi(B_j(R))$. The parameter $r$ plays
a role in the size of the perturbations in condition~\eqref{eq:nu-cr0}
on page~\pageref{eq:nu-cr0}.
%
To define $\mathcal{B}(\bar{\mathcal{X}})$ we specialize
$\mathcal{B}(\bar{\mathcal{X}}; R, r)$ to the case $r = R$, and take
the union over all $R \leq R_{\bar{\mathcal{X}}}$. More precisely
\begin{equation} \label{eq:BX} \mathcal{B}(\bar{\mathcal{X}}) :=
  \bigcup_{0<R\leq R_{\bar{\mathcal{X}}}} \mathcal{B}(\bar{\mathcal{X}};
  R, R).
\end{equation}

This completes the outline of the proof of the second part of the 1'st
statement of Theorem~\ref{t:fil-fuk}.

\Qed

\begin{rem} \label{r:B-spaces}
  \begin{enumerate}
  \item The reason for introducing the more general spaces
    $\mathcal{B}(\bar{\mathcal{X}}; R, r)$ (instead of working with
    $r=R$ all the time) will become clear in~\S\ref{sbsb:coh-sys} when
    we prove the results on the system of functors
    $\mathcal{F}^{\mathscr{P}_1, \mathscr{P}_0}$.
  \item \label{i:B-cap-B'} If $R' \neq R''$ then none of
    $\mathcal{B}(\bar{\mathcal{X}}; R', R')$ and
    $\mathcal{B}(\bar{\mathcal{X}}; R'', R'')$ is a subspace of the
    other. In fact, if $R'' \leq R'$ then we have:
    \begin{equation} \label{eq:B-cap-B'}
      \mathcal{B}(\bar{\mathcal{X}}; R', R') \cap
      \mathcal{B}(\bar{\mathcal{X}}; R'', R'') =
      \mathcal{B}(\bar{\mathcal{X}}; R', R'').
    \end{equation}
    Moreover, if $R'' \leq R'$ and $r'' \geq r'$, then:
    $\mathcal{B}(\bar{\mathcal{X}}; R', r') \subset
    \mathcal{B}(\bar{\mathcal{X}}; R'', r'')$.
  \item Obviously, the space $\mathcal{B}(\bar{\mathcal{X}})$ does not
    only depend on $\bar{\mathcal{X}}$ but also on the choice of the
    embedding $\phi$, and there does not seem to be any preferred
    choice in this respect.
  \end{enumerate}
\end{rem}

\subsection{Proof of Theorem~\ref{t:fil-fuk}, part
  2: coherent systems}
\label{sbsb:coh-sys}

Here we will follow Seidel's approach to invariance of Fukaya
categories based on coherent systems as \zjr{described} in~\cite[Chapter~II,
Section~10(a)]{Seidel}, but with several modifications needed to
accommodate the filtered setting.

Before we go on, we briefly explain what is meant by a coherent system
of {\em filtered} $A_{\infty}$-categories. This is the filtered
counterpart of the notion of coherent systems of
$A_{\infty}$-categories which was introduced in~\cite[Section~10a,
pages 133-135]{Seidel}. More specifically, let
$\{\mathcal{A}^i\}_{i \in \mathcal{I}}$ be a family of filtered
$A_{\infty}$-categories. A coherent systems consists of
$A_{\infty}$-functors
$\mathcal{F}^{i_1, i_0}: \mathcal{A}^{i_0} \longrightarrow
\mathcal{A}^{i_1}$, defined for all $i_0, i_1 \in \mathcal{I}$, and
with $\mathcal{F}^{i,i} = \id_{\mathcal{A}^i}$ for all $i$, as well as
natural transformations
$T^{i_2, i_1, i_0}: \mathcal{F}^{i_2, i_1} \circ \mathcal{F}^{i_1,
  i_0} \longrightarrow \mathcal{F}^{i_2, i_0}$ for all
$i_0, i_1, i_2 \in \mathcal{I}$.  The functors $\mathcal{F}^{i,j}$
will be called comparison (or transition) functors. The functors
$\mathcal{F}^{i,j}$ and natural transformations $T^{i,j,k}$ are
required to satisfy \zhnote{a list of} conditions explained in~\cite[Section~10a,
page 134]{Seidel}. \zhnote{In particular the comparison functors $\mathcal F^{i_1, i_0}$ are all quasi-equivalences and the $T^{i,j,k}$ are quasi-isomorphisms. Turning to the filtered case, we require the following additional conditions to hold.} All the functors $\mathcal{F}^{i_1, i_0}$ are
required to be filtered (i.e.~filtration preserving) and the
$T^{i_2, i_1, i_0}$ should be natural transformations of shift-$0$,
i.e.~$T^{i_2, i_1, i_0} \in \hom^{\leq
  0}_{\text{\em{ffun}}}(\mathcal{A}^{i_2},
\mathcal{A}^{i_0})(\mathcal{F}^{i_2, i_1} \circ \mathcal{F}^{i_1,
  i_0}, \mathcal{F}^{i_2, i_0})$. Here and in what follows
$\text{\em{ffun}}$ is the (filtered) $A_{\infty}$-category of filtered
functors and $\hom^{\leq 0}_{\text{\em{ffun}}}$ stands for the
morphisms in that category that do not shift filtration, namely the
natural transformations (between filtered functors) that preserve
filtrations. Furthermore, all the cohomological identities
from~\cite[Section~10a, page 134]{Seidel} between these natural
transformations should now hold in the $0$-categories
\zhnote{$(H^0(\text{\em{ffun}}(\mathcal{A}^{i'}, \mathcal{A}^{i''})))_0$}
(i.e.~persistence level $0$) of the persistence \zhnote{homological}
categories
$H^0(\text{\em{ffun}}(\mathcal{A}^{i'}, \mathcal{A}^{i''}))$. (In the
second to last formula, the $0$-superscript means \zhnote{cohomological} degree $0$ and the
$0$-subscript stands for the $0$ persistence level subcategory.)

We will sometime refer to $\{\mathcal{A}^i\}_{i \in \mathcal{I}}$ as a
family of $A_{\infty}$-categories over $\mathcal{I}$ and call
$\mathcal{I}$ the base of the family. Similarly, in case we have a
coherent system on $\{\mathcal{A}^i\}_{i \in \mathcal{I}}$ we will
call it a coherent system over $\mathcal{I}$.

One way to assemble a coherent system out of a family of
$A_{\infty}$-categories $\{\mathcal{A}^i\}_{i \in \mathcal{I}}$ is
first to try to embed all of them into one total $A_{\infty}$-category
$\mathcal{A}^i \subset \mathcal{A}^{\text{tot}}$ by quasi-equivalences
and then seek for suitable projection functors
$\mathcal{A}^{\text{tot}} \longrightarrow \mathcal{A}^{j}$ for all
$i$. The functors
$\mathcal{F}^{i_1, i_0}: \mathcal{A}^{i_0} \longrightarrow
\mathcal{A}^{i_1}$ participating in the coherent system will be then
defined by the composition of the inclusions
$\mathcal{A}^{i_0} \subset \mathcal{A}^{\text{tot}}$ with the
projections
$\mathcal{A}^{\text{tot}} \longrightarrow \mathcal{A}^{i_1}$. We will
soon adapt this scheme to the filtered framework. But first we need to
introduce some relevant notions.

We begin with $A_n$-categories and functors. An $A_n$-category
$\mathcal{A}$ is the same as an $A_{\infty}$-category with the
exception that now we \zhnote{have $\mu^{\mathcal{A}}_k$ only} for
$k=1, \ldots, n$. The $\mu^{\mathcal{A}}_k$'s are required to satisfy
the subset of the $A_{\infty}$-identities that \zhnote{involve} only the
$\mu^{\mathcal{A}}_k$'s with $1 \leq k \leq n$. In case the category
$\mathcal{A}$ is clear form the context we will sometimes omit the
superscript from $\mu^{\mathcal{A}}_k$.

Let $\mathcal{A}$ and $\mathcal{B}$ be two $A_n$-categories and
$m \leq n$. Similarly to $A_{\infty}$-functors we have $A_m$-functors
$\mathcal{F}: \mathcal{A} \longrightarrow \mathcal{B}$. They are
defined in the same way as $A_{\infty}$-functors, but now the higher
order components $\mathcal{F}_k$ of $\mathcal{F}$ are defined only for
$1 \leq k \leq m$. The $\mathcal{F}_k$'s, $\mu^{\mathcal{A}}_i$'s and
$\mu^{\mathcal{B}}_j$'s are required to satisfy the same identities as
for for $A_{\infty}$-functors that involve only $1 \leq k \leq m$,
$i,j \leq n$. $A_n$-functors can be composed in a similar way as
$A_{\infty}$-functors. Finally, the notions of pre-natural and natural
transformations between $A_{\infty}$-functors generalize to
\zhnote{$A_m$}-functors in a similar way.

Let $\mathcal{A}$ be an \zhnote{$A_n$}-category with $n \geq 3$.  Let
$X \in \Ob(\mathcal{A})$ and $2 \leq k \leq n$. An element
$e_X \in \hom_{\mathcal{A}}(X, X)$ is called a strict $A_k$-unit if it
satisfies the following conditions: $e_X$ is a cycle, for all
\zhnote{$X' \in \Ob(\mathcal{A})$ the} maps
\zhnote{$$\mu_2(-, e_X): \hom_{\mathcal{A}}(X',X) \longrightarrow
\hom_{\mathcal{A}}(X',X), \quad \mu_2(e_X,-):
\hom_{\mathcal{A}}(X,X') \longrightarrow \hom_{\mathcal{A}}(X,X')$$}
are the identity maps, and moreover
$\mu_j(-, \ldots, -, e_X, -, \ldots, -)=0$ for all $2 < j \leq n$.  An
$A_n$-category is called strictly $A_k$-unital if we are given (as
part of the structure) strict $A_k$-units $e_X$ for every object
$X \in \Ob(\mathcal{A})$. If an $A_{n}$-category is strictly
$A_n$-unital (i.e.~$k=n$) we will simply say that it is strictly
unital.

Let $\mathcal{A}$, $\mathcal{B}$ be \zhnote{two} $A_n$-categories which are \zhnote{both}
strictly $A_k$-unital. An $A_m$-functor
$\mathcal{F}: \mathcal{A} \longrightarrow \mathcal{B}$ is called
strictly $A_l$-unital \zhnote{for $l \leq m$} if $\mathcal{F}_1(e_X) = e_{\mathcal{F}(X)}$ for
all $X \in \Ob(\mathcal{A})$ and
$\mathcal{F}_i(-, \ldots, -, e_X, -, \ldots, -)=0$ for every
\zhnote{$2\leq i \leq l$}.

Similarly to strict units, \zhnote{we also} have the notion of strict
isomorphisms. Let $\mathcal{A}$ be an $A_n$-category ($3 \leq n$)
which is strictly $A_l$-unital ($l \leq n$), and denote its strict
units by $e_Z \in \hom_{\mathcal{A}}(Z,Z)$, $Z \in
\Ob(\mathcal{A})$. Let $X, Y \in \Ob(\mathcal{A})$ and
$u \in \hom_{\mathcal{A}}(X,Y)$ a cycle. We say that $u$ is a strict
$A_k$-isomorphism (where $k \leq n$) if there exists a cycle
$v \in \hom_{\mathcal{A}}(Y,X)$ such that
$$\mu_2(u,v) = e_X, \quad \mu_2(v,u) = e_Y,
\quad \mu_j(-, \ldots, -, u_X, -, \ldots, -) = 0, \, \forall \, 2 < j
\leq k.$$ In case $k=n$ we will simply say that $u$ is a strict
isomorphism.

\begin{rem} \label{r:An-Ainfty}
  \begin{enumerate}
  \item In what follows we will view $A_{\infty}$-categories as a
    special case of $A_n$-categories by allowing $n=\infty$ (of
    course, one needs to slightly adjust the definition since for
    $n=\infty$ the operations $\mu_k$ exist only for $1 \leq k < n$
    and not $1 \leq k \leq n$). The same remark applies also to
    $A_n$-functors, (pre-)natural transformations and strict units,
    and we will view their $A_{\infty}$-counterparts as a special case
    of the respective $A_n$-objects.
  \item $A_n$-categories are also $A_{n'}$-categories for all
    $1 \leq n' \leq n$, therefore we will sometimes reduce $n$ to the
    minimal value which is relevant in the context. A similar remark
    applies also to functors, (pre-)natural transformations and strict
    units.
  \item \label{i:red-func} Let
    $\mathcal{F}: \mathcal{A} \longrightarrow \mathcal{B}$ be an
    $A_n$-functor and $n' \leq n$. We denote by $\{\mathcal{F}\}_{n'}$
    the $A_{n'}$-functor obtained from $\mathcal{F}$ by ignoring the
    terms of order $> n'$ (i.e.~the terms $\mathcal{F}_{k}$ for
    $n' < k \leq n$). We call $\{\mathcal{F}\}_{n'}$ the
    $A_{n'}$-reduction of $\mathcal{F}$.
  \end{enumerate}
\end{rem}

Below we will be mainly interested in $A_2$-functors
$\mathcal{F}: \mathcal{A} \longrightarrow \mathcal{B}$ between two
$A_{3}$-categories. Unwrapping the above definitions, in this case
this means that
$\mathcal{F}_1: \hom_{\mathcal{A}} (X_0, X_1) \longrightarrow
\hom_{\mathcal{B}}(\mathcal{F}X_0, \mathcal{F}X_1)$ are chain maps and
that $\mathcal{F}_1$ respects composition of morphisms up to a chain
homotopy given by $\mathcal{F}_2$. In other words,
$\mathcal{F}_2: \hom_{\mathcal{A}} (X_0, X_1) \otimes
\hom_{\mathcal{A}} (X_1, X_2) \longrightarrow \hom_{\mathcal{A}} (X_0,
X_2)$ defines a chain homotopy between
$\mu_2\circ (\mathcal{F}_1 \otimes \mathcal{F}_1)$ and
$\mathcal{F}_1 \circ \mu_2$. Passing to homology, $\mathcal{F}$
induces a (non-unital) functor
$H(\mathcal{F}): H(\mathcal{A}) \longrightarrow H(\mathcal{B})$
between the homological categories $H(\mathcal{A})$ and
$H(\mathcal{B})$. Note that $H(\mathcal{A})$ and $H(\mathcal{B})$ are
genuine categories since $\mathcal{A}$ and $\mathcal{B}$ were assumed
to be $A_3$-categories.

Assuming $\mathcal{A}$ and $\mathcal{B}$ to be $A_2$-unital, some of
our $A_2$-functors $\mathcal{F}$ will be strictly $A_1$-unital, which
means that $\mathcal{F}_1(e_X) = e_{\mathcal{F}X}$ for all objects
$X$. In particular this implies that $\mathcal{F}$ is homologically
unital (i.e.~the functor $H(\mathcal{F})$ is unital).

We turn now to the filtered setting. Filtered $A_n$-categories,
$A_m$-functors and their pre-natural transformations are defined
precisely in the same way as their \zhnote{filtered} $A_{\infty}$-counterparts. Strict
$A_k$-units are required by definition to be in filtration level $0$,
and the same goes for strict isomorphisms.  (Below we will not attach
the adjective ``filtered'' to units/isomorphisms, implicitly assuming
that whenever we are in the filtered setting these elements are in
filtration level $0$.)

If $\mathcal{A}$ is a filtered $A_3$-category which is strictly
$A_2$-unital then $H(\mathcal{A})$ is a persistence category. If
$\mathcal{F}$ is a filtered $A_2$-functor which is $A_1$-unital then
$H(\mathcal{F})$ is a persistence functor.

Finally, we need the notion of a filtered quasi-equivalence. Let
$\mathcal{F}: \mathcal{A} \longrightarrow \mathcal{B}$ be an
$A_2$-functor between filtered $A_3$-categories. We call $\mathcal{F}$
a filtered quasi-equivalence if $\mathcal{F}$ is a filtered functor
and its homological functor
$H(\mathcal{F}): H(\mathcal{A}) \longrightarrow H(\mathcal{B})$ is an
equivalence of persistence categories.

Next, we need to introduce persistence Hochschild cohomology. Let
$\mathcal{A}$, $\mathcal{B}$ be $A_3$-categories which are strictly
$A_2$-unital.  Let $\mathcal{F}, \mathcal{G}$ be filtered
$A_2$-functors which are $A_1$-unital. Given
$X_0, \ldots X_r \in \Ob(\mathcal{A})$ we will abbreviate
$\mathcal{A}(X_0, \ldots, X_r) := \hom_{\mathcal{A}}(X_0, X_1) \otimes
\cdots \otimes \hom_{\mathcal{A}}(X_{r-1}, X_r)$ and view it as a
filtered chain complex in the standard way (and similarly for
$\mathcal{B}$). The persistence Hochschild cochain complex associated
to the above data is a cochain complex of persistence modules. The
$\alpha$-persistence level of this \zjr{cochain} complex in degree $r$ is:
\begin{equation} \label{eq:PCC} PCC^{r; \leq \alpha} (\mathcal{A},
  \mathcal{B}; \mathcal{F}, \mathcal{G}) := \prod_{X_0, \ldots, X_r}
  H^*\Bigl(\hom_{\k}^{\leq \alpha} \bigl(\mathcal{A}(X_0, \ldots,
  X_r), \, \mathcal{B}(\mathcal{F}X_0, \mathcal{G}X_r)\bigr)\Bigr),
\end{equation}
where the product runs over all $r$-tuples of objects in
$\mathcal{A}$. Here we view
$$\hom_{\k}\bigl(\mathcal{A}(X_0, \ldots, X_r), \,
\mathcal{B}(\mathcal{F}X_0, \mathcal{G}X_r)\bigr)$$ as a filtered
cochain complex, endowed with the standard differential which will be
denote below by $\delta$. The $\alpha$-level filtration on this
cochain complex, denoted here by $\hom_{\k}^{\leq \alpha}$ consists of
those (graded) homomorphisms that shift filtration by $\leq \alpha$.
Finally, $H^*$ stands for persistence cohomology (in all degrees).
Note that the spaces $PCC$ are in fact bigraded since the $\hom_{\k}$
term is graded in itself. Thus a more detailed description
of~\eqref{eq:PCC} would be to define $PCC^{r,s}$ with the term $H^*$
replaced by cohomology $H^s$ in degree $s$. However, whenever not
necessary we will ignore the $s$-degree.

The differential
$\partial_{PCC}: PCC^{r; \leq \alpha} (\mathcal{A}, \mathcal{B};
\mathcal{F}, \mathcal{G}) \longrightarrow PCC^{r+1; \leq \alpha}
(\mathcal{A}, \mathcal{B}; \mathcal{F}, \mathcal{G})$ is defined as
follows. Let
$T \in \hom_{\k}^{\leq \alpha} \bigl(\mathcal{A}(X_0, \ldots, X_r), \,
\mathcal{B}(\mathcal{F}X_0, \mathcal{G}X_r))$ be a $\delta$-cocycle
(i.e.~a chain map). Then $\partial_{CC}([T])$ is defined as the
$\delta$-cohomology class $[S]$ of the chain map
\zjr{ \begin{equation} \label{eq:del-PCC}
 \begin{aligned}
    S(a_1, \ldots, a_{r+1}) := & \epsilon' \mu_2(\mathcal{F}_1(a_1),
    T(a_2, \ldots, a_{r+1}))
    + \epsilon'' \mu_2(T(a_1, \ldots, a_r), \mathcal{G}_1(a_{r+1})) \\
    & + \epsilon \sum_{k=0}^r T(a_1, \ldots, \mu_2(a_{k+1},a_{k+2}),
    \ldots, a_{r+1}),
  \end{aligned}
\end{equation}}
where $\epsilon', \epsilon'', \epsilon = \pm 1$ are signs that depend on
the degrees of $T$, the $a_i$'s and $r$ - see~\cite[Chapter~I,
Section~(1f)]{Seidel}. Since we work over $\k = \mathbb{Z}_2$ these
will play no role in our considerations anyway.

The cohomology of $(PCC, \partial_{PCC})$ is a graded persistence
module which we call the persistence Hochschild cohomology of
$(\mathcal{A}, \mathcal{B})$ with respect to the functors
$\mathcal{F}, \mathcal{G}$. We denote it by
$PHH^{*; \leq \alpha}(\mathcal{A}, \mathcal{B}; \mathcal{F},
\mathcal{G}), \alpha \in \mathbb{R}$.

Apart from $PHH$ there is a persistence variant of the classical
Hochschild cohomology for the categories $H(\mathcal{A})$ and
$H(\mathcal{B})$ with respect to the homological functors
$[\mathcal{F}], [\mathcal{G}]$. This is defined in the same way as
in~\cite[Chapter~I, Section~(1f)]{Seidel}, only that we also take the
persistence structure into account. We will use the following
notation. We write
$$H^{\mathcal{A}}(X,Y):= H(\hom_{\mathcal{A}}(X,Y))$$ for the
persistence homology (in all degrees) of $\hom_{\mathcal{A}}(X,Y)$
with respect to $\mu_1^{\mathcal{A}}$, and similarly for $\mathcal{B}$.
Denote also
$$H^{\mathcal{A}}(X_0, \ldots, X_r) :=
H^{\mathcal{A}}(X_0, X_1) \otimes \cdots \otimes
H^{\mathcal{A}}(X_{r-1}, X_r),$$ viewed as a (graded) persistence
module. Denote $A := H(\mathcal{A})$, $B := H(\mathcal{B})$, the
homological categories of $A$, $B$, and by $F = [\mathcal{F}]$,
$G: = [\mathcal{G}]$ the homological functors corresponding to
$\mathcal{F}$ and $\mathcal{G}$. The cochain complex $P \overline{CC}$
for the Hochschild cohomology of $H(\mathcal{A})$ and $H(\mathcal{B})$
has the following \zjr{persistence module structure} in degree $r$:
\begin{equation} \label{eq:PbarCC} P\overline{CC}^{r; \leq \alpha} (A,
  B; F, G) := \prod_{X_0, \ldots, X_r} \Bigl(\hom_{\text{per}}^{\leq
    \alpha} \bigl(H^{\mathcal{A}}(X_0, \ldots, X_r), \,
  H^{\mathcal{B}}(\mathcal{F}X_0, \mathcal{G}X_r)\bigr)\Bigr),
\end{equation}
where $\alpha$ denotes the persistence level and
$\hom_{\text{per}}^{\leq \alpha}$ stands for (graded) homomorphisms
\zjr{from the persistence module $H^{\mathcal{A}}(X_0, \ldots, X_r)$
  to the persistence module
  $H^{\mathcal{B}}(\mathcal{F}X_0, \mathcal{G}X_r)[\alpha]$.}
\pbrev{Here, we are using the notation and conventions
  from~\S\ref{ssec-pm} (see~\eqref{per-mor}). Namely,
  $H^{\mathcal{B}}(\mathcal{F}X_0, \mathcal{G}X_r)[\alpha]$ stands for
  $H^{\mathcal{B}}(\mathcal{F}X_0, \mathcal{G}X_r)$ shifted by
  $\alpha$ in terms of the persistence parameter.} \zjr{The structure
  map from level $\hom_{\text{per}}^{\leq \alpha}$ to level
  $\hom_{\text{per}}^{\leq \beta}$, whenever $\alpha \leq \beta$, is
  given by composing with the structure map indexed by $\alpha, \beta$
  in the persistence module
  $H^{\mathcal{B}}(\mathcal{F}X_0, \mathcal{G}X_r)$.} The differential
$\partial_{\overline{CC}}$ has a similar expression
to~\eqref{eq:del-PCC}, and an explicit formula can be found
in~\cite[Chapter~I, Section~(1f), page~13]{Seidel}.  We denote the
resulting persistence cohomology by
$P\overline{HH}^{*; \leq \alpha} (A, B; F, G)$,
$\alpha \in \mathbb{R}$.

\begin{rem} \label{r:PHH}
  \begin{enumerate} 
  \item It is easy to see that the data required to define
    the second version, $P \overline{HH}$, of Hochschild cohomology
    (and in fact even $P \overline{CC}$) is entirely homological. In
    fact, this cohomology can be defined for every pair of persistence
    categories and a pair of persistence functors between them. In
    contrast, $PCC$ and $PHH$ seem to depend on some chain level
    information from $\mathcal{A}, \mathcal{B}$ and
    $\mathcal{F}, \mathcal{G}$.

    The two cochain complexes $PCC$ and $P \overline{CC}$ are related
    one to the other by a K\"{u}nneth-type short exact sequence of
    persistence modules:
    \begin{equation} \label{eq:per-Kunneth} 0 \longrightarrow
      E_{r-1}(H(\mathcal{A}), H(\mathcal{B})) \longrightarrow
      PCC^r(\mathcal{A}, \mathcal{B}; \mathcal{F}, \mathcal{G})
      \longrightarrow P \overline{CC}^r(A, B; F, G) \longrightarrow 0,
    \end{equation}
    where $E_{r-1}(-,-)$ is a derived functor of $\hom_{\text{per}}$
    (which is analogous to $Ext^1$) in the category of persistence
    modules. More specifically, when inserting the components of
    $X_0, \ldots, X_r$ from~\eqref{eq:PCC} and~\eqref{eq:PbarCC}
    into~\eqref{eq:per-Kunneth} the term $E_{r-1}$ involves only
    homologies of the type \zhnote{$H_*(\mathcal{A}(X_0, \ldots, X_r))$} and
    $H^{\mathcal{B}}_*(\mathcal{F}X_0, \mathcal{G}X_r)$ of total
    degree $r-1$. The former term \zhnote{$H_*(\mathcal{A}(X_0, \ldots, X_r))$}
    can also be related via short exact sequences (a persistence
    analog of the universal coefficients theorem) to
    $H_*^{\mathcal{A}}(X_0, X_1) \otimes \cdots \otimes
    H_*^{\mathcal{A}}(X_{r-1}, X_r)$ and $Tor_1$-like derived functors
    associated to tensor products of persistence modules. We refer the
    reader to~\cite[Section~8]{BuMi:homological-algebra-persistence}
    for the precise formulation of the K\"{u}nneth and universal
    coefficients theorems for persistence homology (see
    also~\cite{PoShSt:per-mod-op}).
  \item Let $3 \leq n \leq \infty$, $m \leq n$. Given two
    $A_n$-categories $\mathcal{A}$, $\mathcal{B}$ and two
    $A_m$-functors
    $\mathcal{F}, \mathcal{G}: \mathcal{A} \longrightarrow
    \mathcal{B}$, the filtered natural transformations
    $\hom_{\text{\em{ffun}}(\mathcal{A}, \mathcal{B})}(\mathcal{F},
    \mathcal{G})$ form a filtered chain complex (by filtered natural
    transformations we mean those that respect filtrations up to a
    bounded shift). The level-$\alpha$ filtration
    $\hom_{\text{\em{ffun}}(\mathcal{A}, \mathcal{B})}^{\leq \alpha}$
    consists of those natural transformations that shift filtration by
    $\leq \alpha$. Apart from this filtration, this space admits yet
    another filtration called the length filtration. This one is
    indexed by the natural numbers and is decreasing. Its \zhnote{$p$-level $F^p\hom_{\text{\em{ffun}}(\mathcal{A}, \mathcal{B})}^{\leq
      \alpha}(\mathcal{F}, \mathcal{G})$} is given by those natural
    transformations
    $T \in \hom_{\text{\em{ffun}}(\mathcal{A}, \mathcal{B})}^{\leq
      \alpha}(\mathcal{F}, \mathcal{G})$ with $T_0 = \cdots =
    T_p=0$. There is a spectral sequence of persistence modules
    associated to this filtration. A simple calculation shows that its
    1'st page, at persistence level $\alpha$, is given by
    $$E_{1}^{r, s; \leq \alpha} = PCC^{r,s; \leq \alpha}(\mathcal{A},
    \mathcal{B}; \mathcal{F}, \mathcal{G}).$$ Here we have used the
    second degree $s$ on $PCC$ as briefly explained
    earlier. See~\cite[Chapter~I, Section~(1f), page~13]{Seidel} for
    more details in the non-filtered case.  The filtration $F^p$ is
    bounded in the case of $A_m$-functors with $m$ finite. For
    $A_{\infty}$-functors it is not bounded and \zhnote{its associated spectral sequence} might not converge,
    however following~\cite{Seidel} we will use \zhnote{this sequence} as a tool for
    comparing the homologies of different $\hom_{\text{\em{ffun}}}$'s.
\end{enumerate}
\end{rem}

Let $\mathcal{A}$, $\mathcal{B}$, $\mathcal{C}$ be filtered
$A_n$-categories ($3 \leq n \leq \infty$). Let $2 \leq m \leq n$, and
denote by $\text{\em{nu-ffun}}(\mathcal{C}, \mathcal{A})$ the filtered
$A_m$-category of non-unital (or better \zhnote{said,} not necessarily unital)
filtered $A_m$-functors $\mathcal{C} \longrightarrow \mathcal{A}$ and
similarly \zjr{$\text{\em{nu-ffun}}(\mathcal{C}, \mathcal{B})$}. Let
$\mathcal{G}: \mathcal{A} \longrightarrow \mathcal{B}$ be a non-unital
filtered $A_m$-functor and denote by
$$\mathscr{L}_{\mathcal{G}}: \text{\em{nu-ffun}}(\mathcal{C}, \mathcal{A})
\longrightarrow \text{\em{nu-ffun}}(\mathcal{C}, \mathcal{B}),$$ the
functor induced by left-composition with $\mathcal{G}$. Note that this
is a (non-unital) filtered $A_m$-functor. An immediate consequence of
Remark~\ref{r:PHH} is the following.

\begin{lem}[c.f.~Lemma~1.7 in~\cite{Seidel}] \label{l:LG} If
  $\mathcal{G}$ is homologically full and faithful in the persistence
  sense, then so is $\mathscr{L}_{\mathcal{G}}$.
\end{lem}

The next lemma deals with invariance of persistence Hochschild
cohomology under filtered quasi-equivalences. It is a persistence
analog of a very special case of Lemma~2.6 from~\cite{Seidel}, with
several additional very strong assumptions made in order to
accommodate the persistence case.
\begin{lem} \label{l:PHH-invariance} Let $\mathcal{A}$ be a filtered
  $A_n$-category ($3 \leq n \leq \infty$) with strict $A_3$-units and
  $\widetilde{\mathcal{A}} \subset \mathcal{A}$ a full subcategory
  such that the inclusion
  $\mathcal{J}: \widetilde{\mathcal{A}} \longrightarrow \mathcal{A}$
  is a filtered quasi-equivalence. Suppose that
  $\mathcal{P}: \mathcal{A} \longrightarrow \widetilde{\mathcal{A}}$
  is a filtered $A_2$-functor which is $A_1$-unital and assume that
  $\mathcal{P} \circ \mathcal{J} = \id_{\widetilde{\mathcal{A}}}$ (as
  $A_2$-functors). Assume further that for every
  $X \in \Ob(\mathcal{A})$ we have a strict isomorphism
  $u_X \in \hom_{\mathcal{A}}(X, \mathcal{P}(X))$. Then the map
  induced by the restriction
  \begin{equation} \label{eq:rho-PHH} \rho: PHH(\mathcal{A},
    \widetilde{\mathcal{A}}; \mathcal{P}, \mathcal{P}) \longrightarrow
    PHH(\widetilde{\mathcal{A}}, \widetilde{\mathcal{A}};
    \id_{\widetilde{\mathcal{A}}}, \id_{\widetilde{\mathcal{A}}})
  \end{equation}
  is a (bigraded) isomorphism of persistence modules.
\end{lem}
We omit the proof since it is very similar to proof of Lemma~2.6
from~\cite{Seidel}. The role of the strict isomorphisms $u_X$ is to
facilitate the definition of certain chain maps and chain homotopies
that appear in the original proof to the framework of $PCC$ and
$PHH$. Note also that the analogous Lemma~2.6 in~\cite{Seidel} is
stated for any two functors, however here we only need it for the
functor $\mathcal{P}$ which simplifies things further.

We now get to extending filtered $A_n$-functors from a subcategory to
a larger one. The following Lemma is a persistence analog, this time
of a special case of Lemma~1.10 from~\cite{Seidel}, and again with
several additional assumptions.
\begin{lem} \label{l:extending-func} Let $\mathcal{A}$, $\mathcal{B}$
  be filtered $A_n$-categories ($3 \leq n \leq \infty$) and
  $\widetilde{\mathcal{A}} \subset \mathcal{A}$ a full
  subcategory. Let
  $\mathcal{P}: \mathcal{A} \longrightarrow \mathcal{B}$ be a filtered
  $A_2$-functor and denote by
  $\widetilde{\mathcal{P}} := \mathcal{P}|_{\widetilde{\mathcal{A}}}$
  its restriction to $\widetilde{\mathcal{A}}$. Assume that the
  map
  $$\rho: PHH^r(\mathcal{A}, \mathcal{B}; \mathcal{P}, \mathcal{P})
  \longrightarrow PHH^r(\widetilde{\mathcal{A}}, \mathcal{B};
  \widetilde{\mathcal{P}}, \widetilde{\mathcal{P}})$$ induced by the
  restriction is an isomorphism of persistence modules for every $r$.
  Then every filtered $A_n$-functor
  $\widetilde{\mathcal{Q}}:\widetilde{\mathcal{A}} \longrightarrow
  \mathcal{B}$ with
  $\{\widetilde{\mathcal{Q}}\}_2 = \widetilde{\mathcal{P}}$ can be
  extended to a filtered $A_n$-functor
  $\mathcal{Q}: \mathcal{A} \longrightarrow \mathcal{B}$ with
  $\{\mathcal{Q}\}_2 = \mathcal{P}$. (See point~(\ref{i:red-func}) of
  Remark~\ref{r:An-Ainfty}.)
\end{lem}
We omit the proof again since it is very similar to the one indicated
in~\cite{Seidel} for Lemma~1.10, with straightforward modifications
needed for the persistence setting.

We are now ready to assemble a coherent system of filtered
$A_{\infty}$-categories out of a family of categories that are all
included into one total category. Let
$\{\mathcal{A}^i\}_{i \in \mathcal{I}}$ be a family of filtered,
strictly unital $A_{\infty}$-categories, over $\mathcal{I}$. Suppose
there is a filtered strictly unital $A_{\infty}$-category
$\mathcal{A}^{\text{tot}}$ such that for every $i \in \mathcal{I}$,
$\mathcal{A}^i$ is a full subcategory of
$\mathcal{A}^{\text{tot}}$. Denote by
$\mathcal{J}^i: \mathcal{A}^i \longrightarrow
\mathcal{A}^{\text{tot}}$ the inclusion functor and assume that
$\mathcal{J}^i$ is a filtered
quasi-equivalence. \label{p:tot-over-base} We will refer to a category
$\mathcal{A}^{\text{tot}}$ as above, together with the inclusion
functors $\mathcal{J}^i$, as a (filtered, strictly unital) {\em total
  $A_{\infty}$-category over $\mathcal{I}$}.

Assume that for every $i \in \mathcal{I}$ there is a filtered
$A_2$-functor
$\mathcal{P}r^i: \mathcal{A}^{\text{tot}} \longrightarrow
\mathcal{A}^i$ which is strictly $A_1$-unital and such that
$\mathcal{P}r^i \circ \mathcal{J}^i = \id_{\mathcal{A}^i}$ as
$A_2$-functors. Assume further that the following holds for every
$i \in \mathcal{I}$: for every $X \in \Ob(\mathcal{A}^{\text{tot}})$
there exists a strict isomorphism
$u_X^i \in \hom_{\mathcal{A}^{\text{tot}}}(X, \mathcal{P}r^i(X))$.

\begin{prop} \label{p:atot-coh} Under the above assumptions each of
  the $A_2$-functors $\mathcal{P}r^i$, $i \in \mathcal{I}$ can be
  extended to a filtered $A_1$-unital $A_{\infty}$-functor
  $\mathcal{Q}^i: \mathcal{A}^{\text{tot}} \longrightarrow
  \mathcal{A}^i$ (i.e.~$\{\mathcal{Q}^i\}_2 = \mathcal{P}r^i$).
  Moreover, the functors
  $\mathcal{F}^{j,i} := \mathcal{Q}^j \circ \mathcal{J}^i:
  \mathcal{A}^i \longrightarrow \mathcal{A}^j$, $i, j \in \mathcal{I}$
  are filtered, $A_1$-unital, and form a \zhnote{coherent system of
 filtered} $A_{\infty}$-categories over $\mathcal{I}$. The filtered natural
  transformations
  $T^{i_2, i_1, i_0}: \mathcal{F}^{i_2, i_1} \circ \mathcal{F}^{i_1,
    i_0} \longrightarrow \mathcal{F}^{i_2, i_0}$ will be described in
  the proof.
\end{prop}

\begin{proof}
  The restriction of the $A_2$-functors $\mathcal{P}r^i$ to the
  subcategory $\mathcal{A}^i \subset \mathcal{A}^{\text{tot}}$ is the
  $A_2$-reduction of the identity $A_{\infty}$-functor on
  $\mathcal{A}^i$. Using Lemmas~\ref{l:PHH-invariance}
  and~\ref{l:extending-func} we can extend $\mathcal{P}r^i$ to the
  desired $A_{\infty}$-functor $\mathcal{Q}^i$.
  
  The construction of the natural transformations
  $T^{i_2, i_1, i_0}: \mathcal{F}^{i_2, i_1} \circ \mathcal{F}^{i_1,
    i_0} \longrightarrow \mathcal{F}^{i_2, i_0}$ follows the same
  scheme as in~\cite[Chapter~II, Section~(10a), Page~134]{Seidel}.
  
  Consider the homological functor induced from left-composition with
  $\mathcal{Q}^i$, viewed as a persistence functor:
  \begin{equation} \label{eq:LQi-1} H(\mathscr{L}_{\mathcal{Q}^i}):
    H(\text{\em{ffun}}(\mathcal{A}^{\text{tot}},
    \mathcal{A}^{\text{tot}})) \longrightarrow
    H(\text{\em{ffun}}(\mathcal{A}^{\text{tot}}, \mathcal{A}^i))
  \end{equation}
  Recall that $\mathcal{Q}^i$ is a filtered quasi-equivalence and in
  particular homologically full and faithful in the persistence
  sense. By Lemma~\ref{l:LG} the action of the functor
  $H(\mathscr{L}_{\mathcal{Q}^i})$ on morphisms is an isomorphism of
  persistence modules:
  \begin{equation} \label{eq:LQi-2} H(\mathscr{L}_{\mathcal{Q}^i}):
    \hom_{H(\text{\em{ffun}}(\mathcal{A}^{\text{tot}},
      \mathcal{A}^{\text{tot}}))}(\mathcal{J}^i \circ \mathcal{Q}^i,
    \id_{\mathcal{A}^{\text{tot}}}) \longrightarrow
    \hom_{H(\text{\em{ffun}}(\mathcal{A}^{\text{tot}},
      \mathcal{A}^i))}(\mathcal{Q}^i, \mathcal{Q}^i).
  \end{equation}
  Let
  $S^i \in \hom^{\leq 0}_{\text{\em{ffun}}(\mathcal{A}^{\text{tot}},
    \mathcal{A}^{\text{tot}})}(\mathcal{J}^i \circ \mathcal{Q}^i,
  \id_{\mathcal{A}^{\text{tot}}})$ be a cycle whose homology class
  $[S^i]$ is sent by $H(\mathscr{L}_{\mathcal{Q}^i})$ to
  $[\id] \in H(\hom_{\text{\em{ffun}}(\mathcal{A}^{\text{tot}},
    \mathcal{A}^i)}(\mathcal{Q}^i, \mathcal{Q}^i))$.  Having defined
  $S^i$ as above for all $i \in \mathcal{I}$, we define
  \begin{equation} \label{eq:Ti2i1i0} T^{i_2, i_1, i_0} :=
    \mathscr{L}_{\mathcal{Q}^{i_2}}(\mathscr{R}_{\mathcal{J}^{i_0}}(S^{i_1})),
  \end{equation}
  where $\mathscr{R}_{\mathcal{J}^{i_0}}$ is the functor
  $\text{\em{ffun}}(\mathcal{A}^{\text{tot}},
  \mathcal{A}^{\text{tot}}) \longrightarrow
  \text{\em{ffun}}(\mathcal{A}^{i_0}, \mathcal{A}^{\text{tot}})$
  induced from right-composition with $\mathcal{J}^{i_0}$.
\end{proof}

\subsection{Proof of Theorem~\ref{t:fil-fuk}, part
  3: coherent systems of Fukaya categories} \label{sbsb:coh-sys-fuk}
In order to apply the algebraic statements from~\S\ref{sbsb:coh-sys},
particularly Proposition~\ref{p:atot-coh}, to the case of Fukaya
categories we need two additional ingredients coming from geometry.
First, we need to construct a filtered total Fukaya category
$\fuk^{\text{tot}}(\mathcal{X})$ that contains all the Fukaya
categories $\fuk(\mathcal{X}; i)$, constructed via various
perturbation data $i$, as quasi-equivalent subcategories of
$\fuk^{\text{tot}}(\mathcal{X})$. The second ingredient is to
construct filtered $A_2$-functors
$\mathcal{P}r^i: \fuk^{\text{tot}}(\mathcal{X}) \longrightarrow
\fuk(\mathcal{X}; i)$ that are left inverses of the inclusions
$\fuk(\mathcal{X}; i) \subset \fuk^{\text{tot}}(\mathcal{X})$.

We begin with the first ingredient, namely the construction of
$\fuk^{\text{tot}}$. To this end, recall the spaces of perturbation
data $\mathcal{B}(\bar{\mathcal{X}}; R, r)$ from
page~\pageref{pp:space-B}. We fix the symplectic embedding $\phi$ and
the parameters $R_{\bar{\mathcal{X}}}$ \zhnote{and
  $r\leq R \leq R_{\bar{\mathcal{X}}}$}. We will outline now a
construction of a filtered strictly unital total $A_{\infty}$-category
over $\mathcal{B}(\bar{\mathcal{X}}; R, r)$. (See
page~\pageref{p:tot-over-base} for the meaning of a total category
over a base.) More specifically, we will construct a filtered strictly
unital $A_{\infty}$-category $\fuk^{\text{tot}}(\mathcal{X}; R, r)$
with the following property.  For every
$i \in \mathcal{B}(\bar{\mathcal{X}}; R, r)$, the filtered Fukaya
category $\fuk(\mathcal{X}; i)$ is a full subcategory of
\pbrev{$\fuk^{\text{tot}}(\mathcal{X}; R, r)$} and the inclusion
\zjr{$\mathcal{J}^i: \fuk(\mathcal{X}; i) \longrightarrow
  \fuk^{\text{tot}}(\mathcal{X}; R, r)$} is a filtered
quasi-equivalence.

The construction of $\fuk^{\text{tot}}(\bar{\mathcal{X}}; R, r)$
follows similar steps to the construction introduced
in~\cite[Chapter~II, Section~10(a), pages 134-5]{Seidel}, with some
significant modifications necessary to accommodate the filtered
setting. The objects of $\fuk^{\text{tot}}(\bar{\mathcal{X}}; R, r)$
are pairs $(L,i)$, where $L \in \mathcal{X}$ and
$i \in \mathcal{B}(\bar{\mathcal{X}}; R, r)$ is a choice of admissible
perturbation data. The morphism space between $(L_0, i_0)$ and
$(L_1, i_1))$ is defined to be the Floer complex
$CF(L_0, L_1; \mathscr{D}_{(L_0,i_0), (L_1, i_1)})$ where the Floer
datum $\mathscr{D}$ is defined as in our earlier construction of
filtered Fukaya categories with the only restriction that if $i_0=i_1$
then the Floer datum $\mathscr{D}_{(L_0,i_0), (L_1, i_1)}$ should agree with
that of $(L_0, L_1)$ in the category $\fuk(\mathcal{X}; i)$.  Another
important point is that we require the Floer data for pairs of the
type $((L,i), (L,j))$ to continue to be of the same type as in our
construction of filtered Fukaya categories. Namely, we take here a
pair of a Morse function and a Riemannian metric, such that the Morse
function has a unique local maximum (i.e.~a unique critical point of
index $n = \dim_{\mathbb{C}}X$).

The next step in the construction of the total category is to choose
consistent perturbation data
$\mathscr{P}^{\text{tot}} = \mathscr{P}^{\text{tot}}(\mathcal{X}; R,
r)$ with the restriction analogous to the one imposed on the Floer
data. Namely, whenever we have a cluster of punctured disks decorated
entirely by pairs $(L_k, i)$ with the same $i$, then the value of
$\mathscr{P}^{\text{tot}}$ on such a cluster coincides with the one
prescribed by the perturbation data $i$. Apart from that there will be
one important difference to \zhnote{the} way we have defined the perturbation data
for each $\fuk(\mathcal{X}; i)$, which we now describe.  Given a tuple
$(i_0, \ldots, i_d)$ with
$i_k \in \mathcal{B}(\bar{\mathcal{X}}; R, r)$ we define the following
two quantities:
\begin{equation} \label{eq:Ri-ri}
  \begin{aligned}
    & R^{(i_0, \ldots, i_d)} := \sup \{ \widetilde{R} \mid
    \widetilde{R} \leq R_{\bar{\mathcal{X}}}, \text{ and } i_k \in
    \mathcal{B}(\bar{\mathcal{X}}; \widetilde{R},\widetilde{R}) \,
    \forall \, 0 \leq k \leq d\}, \\
    & r^{(i_0, \ldots, i_d)} :=\inf \{\widetilde{r} \mid 0 \leq
    \widetilde{r}, \text{ and } i_k \in \mathcal{B}(\bar{\mathcal{X}};
    \widetilde{r}, \widetilde{r}) \, \forall \, 0 \leq k \leq d\}.
  \end{aligned}
\end{equation}
In other words, $R^{(i_0, \ldots, i_d)}$ measures the supremal radius
of the sub-balls in the embedding $\phi$ on which the almost complex
structures in all the perturbation data $i_k$ are standard. The other
quantity $r^{(i_0, \ldots, i_d)}$ measures the infimal upper bound on
the perturbation $1$-forms in all the perturbation data $i_k$.  Note
that since $i_k \in \mathcal{B}(\bar{\mathcal{X}}; R,r)$ for all $k$,
we have
\begin{equation} \label{eq:R-leq-Ri-etc} r^{(i_0, \ldots, i_d)} \leq r
  \leq R \leq R^{(i_0, \ldots, i_d)}.
\end{equation}

Turning to the definition of the perturbation data
$\mathscr{P}^{\text{tot}}$, we require it to satisfy the following
conditions. Let $S_{d+1}$ be $(d+1)$-punctured disk, decorated by the
tuple $((L_0, i_0), \ldots, (L_d, i_d))$, and denote by $(K,J)$ the
value of the perturbation data \zhnote{on} $S_{d+1}$. Recall the symplectic
embedding $\phi$ from~\eqref{eq:phi-embedd} and the balls $B_j$ (see
page~\pageref{pp:balls-R}). We require that:
\begin{enumerate} \label{pp:JK-tot}
\item $K|_{\phi(B_j(R^{(i_0, \ldots, i_d)}))} \equiv 0$ for every
  $j$. Here we have denoted by
  $B_j(R^{(i_0, \ldots, i_d)}) \subset B_j$ the smaller ball of radius
  $R^{(i_0, \ldots, i_d)}$.
\item
  $J_z|_{\phi(B_j(R^{(i_0, \ldots, r_d)}))} = \phi_*(J_{\text{std}})$
  for every $j$ and every $z \in S_{d+1}$.
\end{enumerate}
In addition to the above two conditions we also require that
\begin{equation} \label{eq:nu-P-tot}
  \nu(\mathscr{P}^{\text{tot}}((L_0, i_0), \ldots, (L_d, i_d))) <
  C(d+1)(r^{(i_0, \ldots, i_d)})^2.
\end{equation}
Here $C$ is the constant from~\eqref{eq:E-monotonicity-bound}, and
similarly to~\eqref{eq:K-norm-sup}
$$\nu(\mathscr{P}^{\text{tot}}((L_0,i_0), \ldots, (L_d,i_d)))
:= C_{d+1} \sup \|K(\mathscr{P}^{\text{tot}}, S_{d+1})|,$$ where now
the supremum goes over all \zhnote{$(d+1)$}-punctured disks $S_{d+1}$ that are
decorated by $((L_0,i_0), \ldots, (L_d, i_d))$. Finally, the above
requirements extend to clusters of punctured disks in a similar way.

We claim that there exists a consistent choice of perturbation data
$\mathscr{P}^{\text{tot}}$ satisfying the above conditions. \zhnote{The proof of this is similar to the way} we have proved the same statement
for the perturbation data
$i \in \mathcal{B}(\bar{\mathcal{X}}; R, r)$.

With a choice of perturbation data as above one can define an
$A_{\infty}$-category in the same way we defined our earlier Fukaya
categories. We denote this category by
$\fuk^{\text{tot}}(\bar{\mathcal{X}}; R, r)$ or sometimes by
$\fuk^{\text{tot}}(\bar{\mathcal{X}}; \mathscr{P}^{\text{tot}})$ when
we want to emphasize the choice of the perturbation data
$\mathscr{P}^{\text{tot}}$ used to define it.

We claim that $\fuk^{\text{tot}}(\bar{\mathcal{X}}; R, r)$ is a
filtered and strictly unital $A_{\infty}$-category. This follows by
the same arguments we used for $\fuk(\mathcal{X}; i)$,
$i \in \mathcal{B}(\bar{\mathcal{X}}; R, r)$, with minor
modifications. The important points are that the analogues of
inequalities~\eqref{eq:E-monotonicity-bound},~\eqref{eq:A-nu-2}
and~\eqref{eq:nu-cr0} will continue to hold with $R$ and $r$ replaced
by $R^{(i_0, \ldots, i_d)}$ and $r^{(i_0, \ldots, i_d)}$ respectively,
and $\nu(\mathscr{P}_{d+1})$ by
$\nu(\mathscr{P}^{\text{tot}}((L_0,i_0), \ldots, (L_d,i_d)))$.  This
completes the outline of the construction of the category
$\fuk^{\text{tot}}(\bar{\mathcal{X}}; R, r)$.

For every $i \in \mathcal{B}(\bar{\mathcal{X}}; R, r)$ there is an
obvious inclusion
$\mathcal{J}^i: \fuk(\mathcal{X}; i) \longrightarrow
\fuk^{\text{tot}}(\bar{\mathcal{X}}; R, r)$. Clearly this functor is
filtered and we claim that it is a filtered quasi-equivalence. To see
the latter statement, first note that by construction $\mathcal{J}^i$
is full and faithful. Now, any object
$(L,k) \in \Ob(\fuk^{\text{tot}}(\bar{\mathcal{X}}; R, r))$, is
isomorphic to $(L,i)$ via an isomorphism $u$ that lies in
$\hom^{\leq 0}$ of $\fuk^{\text{tot}}$. This follows from Morse
theory, since the $\hom$ between $(L,k)$ and $(L,i)$ is the Morse
complex of $L$ with respect to a Morse function with a unique local
maximum. This shows that $\mathcal{J}^i$ is a filtered
quasi-equivalence.

Next, for every $i \in \mathcal{B}(\bar{\mathcal{X}}; R, r)$ we
construct a filtered $A_2$-functor
$\mathcal{P}r^i: \fuk^{\text{tot}}(\mathcal{X}; R, r) \longrightarrow
\fuk(\mathcal{X}; i)$, which is strictly $A_1$-unital and such that
$\mathcal{P}r^i \circ \mathcal{J}^i = \id_{\fuk(\mathcal{X}; i)}$.

Let $i \in \mathcal{B}(\bar{\mathcal{X}}; R, r)$. The construction of
$\mathcal{P}r^i$ goes as follows. Let $(L,k)$ be an object of
$\fuk^{\text{tot}}(\mathcal{X}; R, r)$, where $L \in \mathcal{X}$,
$k \in \mathcal{B}(\bar{\mathcal{X}}; R, r)$. We define
$\mathcal{P}r^i((L,k)) = L$.

Next, we define the 1'st order part $\mathcal{P}r^i_1$ of
$\mathcal{P}r^i$ on morphisms. This can be done by means of {\em Floer
  continuation maps}. Specifically, we need to define a filtration
preserving chain map
\begin{equation} \label{eq:Pi-1-a}
  \mathcal{P}r_1^i: CF((L', k'), (L'', k''); \mathscr{D}_{(L', k'),
    (L'', k'')}) \longrightarrow CF(L',L''; \mathscr{D}_{L',L'',
    i})
\end{equation}
for every $L', L'' \in \mathcal{X}$,
$k', k'' \in \mathcal{B}(\bar{\mathcal{X}}; R, r)$, where
$\mathscr{D}_{(L', k'), (L'', k'')}$ is the Floer data for the pair
$((L', k'), (L'', k''))$ and $\mathscr{D}_{L',L'', i}$ is the one \zhnote{used} in
$\fuk(\mathcal{X}; i)$. Assume first that $\bar{L}' \neq \bar{L}''$
(which means that $\bar{L}' \pitchfork \bar{L}''$. By construction,
the Hamiltonian terms in both Floer data
$\mathscr{D}_{(L', k'), (L'', k'')}$ and $\mathscr{D}_{L',L'', i}$ are
identically $0$ (so the two Floer data may differ only in their almost
complex structures). Denote by
$J(\mathscr{P}^{\text{tot}};(L',k', L'',k''))$ the almost complex
structure of $\mathscr{D}_{(L', k'), (L'', k'')}$ and by
$J(i; L', L'')$ the one corresponding to $\mathscr{D}_{L',L'', i}$.
Fix a generic homotopy $J_s^{\text{cont}}$, $s \in [0,1]$, between
these two almost complex structures (of the type admissible in
$\mathcal{B}(\bar{\mathcal{X}}; R, r)$). We assume that
$J_s^{\text{cont}}$ coincides with
$J(\mathscr{P}^{\text{tot}};(L',k', L'',k''))$ near $s=0$ and with
$J(i; L', L'')$ near $s=1$. \zjr{Extend} this homotopy to
$s \in \mathbb{R}$ by keeping it constant with respect to the
$s$-parameter outside of $[0,1]$. Recall that, by construction, both
Floer data $\mathscr{D}_{(L', k'), (L'', k'')}$ and
$\mathscr{D}_{L',L'', i}$ have $0$ Hamiltonian terms, hence
$(0, J_s^{\text{cont}})$ defines a homotopy between the \zhnote{latter two}
Floer data. Standard Floer theory associates to this homotopy a
filtration preserving quasi-isomorphism as claimed in~\eqref{eq:Pi-1-a}
which is called {\em the Floer continuation map}. That
$\mathcal{P}r^i_1$ preserves filtrations follows from standard
arguments in Floer theory (using the assumption that the Hamiltonian
terms in the preceding homotopy of Floer data are $0$ for all times
$s$).

For further use, we will add one more restriction on the definition of
the continuation maps $\mathcal{P}r^i_1$. In case the two objects
$(L',k')$ and $(L'',k'')$ correspond to the \zhnote{same} perturbation data $i$,
i.e.~$k'=k''=i$, we will take the homotopy $J^{\text{cont}}_s$ to be
constant. As a result the continuation map
$$\mathcal{P}r^i_1: CF((L', i), (L'', i); \mathscr{D}_{(L', i), (L'', i)})
\longrightarrow CF(L',L''; \mathscr{D}_{L',L'', i})$$ for such pairs
will be the identity. 

We now briefly address the case when $\bar{L}'=\bar{L}''$.  Recall
that in this case the Floer data on each of
$\mathscr{D}_{(L', k'), (L'', k'')}$, $ \mathscr{D}_{L',L'', i}$,
consists of a Morse function \zhnote{and a} Riemannian
metric on $\bar{L}'$. In our model the corresponding $CF$'s are just
the Morse complexes on $\bar{L}'$ associated to these data. The map
$\mathcal{P}r^i_1$ is now defined by means of \zhnote{standard Morse}
homology theory - it is just the continuation map between the two
Morse complexes. The fact that $\mathcal{P}r^i_1$ is filtration
preserving is automatic since, by definition, both Morse complexes
$CF((L', k'), (L'', k''); \mathscr{D}_{(L', k'), (L'', k'')})$ and
$CF(L',L''; \mathscr{D}_{L',L'', i})$ are concentrated at the same
filtration level (which is a constant that depends on the difference
between the primitives of the Liouville forms on $L'$ and $L''$).

Similarly to the case $\bar{L}' \neq \bar{L}''$, here too, we can
arrange $\mathcal{P}r^i_1$ to be the identity map whenever
$k'=k''=i$. This can be done by taking the homotopy between the two
Morse data to be constant\zhnote{, and the} resulting Morse continuation map will
then be the identity.

Next we define the 2'nd order part of $\mathcal{P}r^i$. This will be a
map
\begin{equation} \label{eq:Pi-2-a} \mathcal{P}r^i_2: CF((L_0, k_0), (L_1,
  k_1); \mathscr{D}_{\text{tot}}) \otimes CF((L_1, k_1), (L_2, k_2);
  \mathscr{D}_{\text{tot}}) \longrightarrow CF(L_0, L_2;
  \mathscr{D}_i)
\end{equation}
of cohomological degree $-1$. Here we have written
$\mathscr{D}_{\text{tot}}$ and $\mathscr{D}_i$ for the Floer data (for
the corresponding pairs of Lagrangians) in the categories
$\fuk^{\text{tot}}(\mathcal{X}; R, r)$ and $\fuk(\mathcal{X}; i)$
respectively.

Assume for simplicity that $\bar{L}_0$, $\bar{L}_1$, $\bar{L}_2$ are
all distinct. To define~\eqref{eq:Pi-2-a} we will need to introduce
first some new spaces of Floer-type polygons. Recall the homotopy
$J_s^{\text{cont}}$ from the definition of $\mathcal{P}r^i_1$
above. Below we will need a more precise notation and we will denote
it from now by
$J_{s,t}^{\text{cont}}((L',k'), (L'',k''), (L',L'',i))$\zhnote{, where
$s \in \mathbb{R}$, $t \in [0,1]$}. (Recall that all our almost
complex structures are possibly time-dependent, and we denote here by
$t$ the time parameter.)

Denote by $S_3$ the $3$-punctured disk. Recall that $S_3$ has two
``entry'' strip-like ends $St^{-}_{0,1}$, $St^{-}_{1,2}$ and one
``exit'' strip-like end $St_{0,2}^+$. We order them in the clockwise
direction, $St^{-}_{0,1}$, $St^{-}_{1,2}$, $St_{0,2}^+$, according to
the punctures they correspond to. We denote by $(s,t)$ the coordinates
on each of these strip-like ends. Thus
$(s,t) \in (-\infty, 0] \times [0,1]$ for $St^{-}_{0,1}$,
$St^{-}_{1,2}$ and and $(s,t) \in [0, \infty) \times [0,1]$ for
$St_{0,2}^+$. We also fix a smooth positive decreasing function
$A:(0,\delta) \longrightarrow \mathbb{R}$ in a small neighborhood of
$0$ with $A(\tau) \longrightarrow \infty$ as $\tau \to 0^+$.

Consider now a $1$-parametric family
\zhnote{$\mathscr{P}^{\tau}(\mathcal{P}r) = (K^{\tau}(\mathcal{P}r),
J^{\tau}(\mathcal{P}r)$, $\tau \in (0,1))$} of perturbation data on the
$3$-punctured disk $S_3$. We will write
\zhnote{$\mathscr{P}_z^{\tau}(\mathcal{P}r) = (K_z^{\tau}(\mathcal{P}r),
J_z^{\tau}(\mathcal{P}r))$} for the value of the perturbation data at the
point $z \in S_3$.  We will require the family
$\mathscr{P}^{\tau}(\mathcal{P}r)$, $\tau \in (0,1)$, to satisfy the
following conditions:
\begin{enumerate}
\item \label{i:St01} When $0 < \tau < \delta$, \zhnote{for} every point in
  $St_{0,1}^-$ with coordinates $(s,t) \in (-\infty, 0] \times [0,1]$
  we have
  $$J_{s,t}^{\tau}(\mathcal{P}r) =
  J_{s+A(\tau)+1,t}^{\text{cont}}((L_0,k_0), (L_1,k_1), (L_0,
  L_1,i)).$$ In other words, when \zhnote{$s \in [-A(\tau)-1, -A(\tau)]$,}
  $J_{s,-}^{\tau}(\mathcal{P}r)$ coincides with the continuation
  homotopy $J^{\text{cont}}$ after a suitable shift in the
  $s$-parameter. Note that $J_{s+A(\tau)+1,t}^{\text{cont}}$ has been
  defined for all $s \in \mathbb{R}$.
\item \label{i:St12}When $0 < \tau < \delta$, \zhnote{for} every point in
  $St_{1,2}^-$ with coordinates $(s,t) \in (-\infty, 0] \times [0,1]$
  we have
  $$J_{s,t}^{\tau}(\mathcal{P}r) =
  J_{s+A(\tau)+1,t}^{\text{cont}}((L_1,k_1), (L_2,k_2), (L_1,
  L_2,i)).$$
\item When $0 < \tau < \delta$, we require the perturbation data
  $\mathscr{P}^{\tau}(\mathcal{P}r)$ to coincide with the one assigned
  by the perturbation data $i$ to the triple $(L_0, L_1, L_2)$ along
  $S_3 \setminus (St_{0,1}^{-} \cup St_{1,2}^{-})$. Note that this
  requirement is compatible with the previous two conditions.
\item \label{i:St02} When $1-\delta < \tau < 1$, \zhnote{for} every point
  in $St_{0,2}^+$ with coordinates
  $(s,t) \in [0, \infty) \times [0,1]$ we have
  $$J_{s,t}^{\tau}(\mathcal{P}r) =
  J_{s-A(1-\tau),t}^{\text{cont}}((L_0,k_0), (L_2,k_2),
  (L_0,L_2,i)).$$
\item When $1-\delta < \tau < 1$, we require the perturbation data
  $\mathscr{P}^{\tau}(\mathcal{P}r)$ to coincide with the one assigned
  by $\mathscr{P}^{tot}$ to the triple
  $((L_0, k_0), (L_1, k_1), (L_2, k_2))$ along
  $S_3 \setminus St_{0,2}^+$.  Again, this requirement is compatible
  with the previous one.
\item For every $\tau \in [\delta, 1-\delta]$, the perturbation data
  $\mathscr{P}^{\tau}(\mathcal{P}r)$ coincides with the Floer datum of
  the pair $((L_0, k_0), (L_1, k_1))$ (as assigned by
  $\mathscr{P}^{\text{tot}}$) along $St_{0,1}^-$ outside some compact
  subset. We require the analogous condition to hold also with respect
  to $((L_1, k_1), (L_2, k_2))$ along $St_{1,2}^-$.
\item For every $\tau \in [\delta, 1-\delta]$, the perturbation data
  $\mathscr{P}^{\tau}(\mathcal{P}r)$ coincides with the Floer datum of
  the pair $(L_0, L_2)$, as assigned by the perturbation data $i$,
  along $St_{0,2}^+$ outside some compact subset.
\item \label{i:P-tau} For all $\tau \in (0,1)$, the almost complex
  structures $J^{\tau}(\mathcal{P}r)$ and perturbation forms
  $K^{\tau}(\mathcal{P}r)$ from $\mathscr{P}^{\tau}(\mathcal{P}r)$ are
  all of the types and sizes admissible in the construction of
  $\mathscr{P}^{\text{tot}}$. In particular they should satisfy the
  inequality~\eqref{eq:nu-P-tot} (for $d+1=3$) and the two conditions
  on $K$ and $J$ that are listed before that inequality on
  page~\pageref{pp:JK-tot}.
\item \label{i:P-const} In case the three pairs $(L_0,k_0)$,
  $(L_1,k_1)$, $(L_2,k_2)$ all correspond to the perturbation data
  $i$, i.e.~$k_0=k_1=k_2$ we take the family
  $\{\mathscr{P}^{\tau}(\mathcal{P}r)\}$ to be constant with respect
  to $\tau$, and moreover to coincide with the perturbation data
  $i$. Note that this is compatible with the rest of the conditions
  above, since earlier we required each of the homotopies
  $J_s^{\text{cont}}$ that appear in
  points~\eqref{i:St01},~\eqref{i:St12} and~\eqref{i:St02} above to be
  constant (with respect to $s$).
\end{enumerate}

Fix three intersection points $x \in \bar{L}_0 \cap \bar{L}_1$,
$y \in \bar{L}_1 \cap \bar{L}_2$, $w \in \bar{L}_0 \cap \bar{L}_2$.
Denote by
$\mathcal{M}^{\mathcal{P}r_2}(x, y, w;
\{\mathscr{P}^{\tau}(\mathcal{P}r)\})$ the space of all pairs
$(\eta, u)$, where $\eta \in (0,1)$ and $u:S_3 \longrightarrow X$
solves the generalized Floer equation with respect to the perturbation
data $\mathscr{P}^{\eta}(\mathcal{P}r)$, with \zjr{Lagrangian} boundary
conditions prescribed by $\bar{L}_0, \bar{L}_1, \bar{L}_2$ and with
asymptotics at the ends being $x$, $y$ and $w$.

By choosing the family $\{\mathscr{P}^{\tau}(\mathcal{P}r)\}$ to be
generic we have that
$\mathcal{M}^{\mathcal{P}r_2}(x, y, w;
\{\mathscr{P}^{\tau}(\mathcal{P}r)\})$ is a smooth manifold of
dimension $d(x,y,w) = |w|'-|x|'-|y|'+1$, where $|\,\cdot\,|'$ denotes
cohomological degree. Moreover, standard arguments show that if
$d(x,y,z)=0$ then
$\mathcal{M}^{\mathcal{P}r_2}(x, y, z;
\{\mathscr{P}^{\tau}(\mathcal{P}r)\})$ is compact, hence consists of
finitely many points. We then define
\begin{equation} \label{eq:Pi-2-b} \mathcal{P}r^i_2(x,y) : = \sum_{z}
  \# \mathcal{M}^{\mathcal{P}r_2}(x,y,z;
  \{\mathscr{P}^{\tau}(\mathcal{P}r)\}) z,
\end{equation}
where $z$ runs over all points in $\bar{L}_0 \cap \bar{L}_2$ with
$|z|' = |x|'+|y|'-1$ (hence $d(x,y,z)=0$). As before,
\zhnote{$\# \mathcal{M}^{\mathcal{P}r_2}(-)$} is the count (with values in
$\mathbb{Z}_2$, or under additional assumptions in $\k$) of the number
of elements in the above space.

We claim that $\mathcal{P}r^i_1$ and $\mathcal{P}r^i_2$ form together
an $A_2$-functor. This amounts to showing that $\mathcal{P}r^i_2$
satisfies the following identity:
\begin{equation} \label{eq:Pi-2-c}
  \begin{aligned}
    \mathcal{P}r^i_1(\mu^{\text{tot}}_2(x,y)) -
    \mu^{(i)}_2(\mathcal{P}r^i_1(x), \mathcal{P}r^i_1(y)) = &
    \epsilon_1
    \mu_1^{(i)} \mathcal{P}r^i_2(x,y) \\
    + & \epsilon_2 \mathcal{P}r^i_2(\mu_1^{\text{tot}}(x),y) +
    \epsilon_3 \mathcal{P}r^i_2(x, \mu_1^{\text{tot}}(y)),
  \end{aligned}
\end{equation}
where $\epsilon_1, \epsilon_2, \epsilon_3 = \pm 1$ are signs that
depend on the degrees of $x$, $y$ and $z$ (but as we work with
$\k = \mathbb{Z}_2$ the precise \zhnote{value} of these signs is irrelevant).
Here we have denoted by $\mu^{\text{tot}}$ the $A_{\infty}$-operations
in the category $\fuk^{\text{tot}}(\mathcal{X}; R, r)$ and by
$\mu^{(i)}$ those from the category $\fuk(\mathcal{X}; i)$.

The proof of~\eqref{eq:Pi-2-c} is based on \zjr{standard arguments} in Floer
theory. Fix $w$ such that $d(x,y,w)=1$. Then
$\mathcal{M}^{\mathcal{P}r_2}(x,y,w;
\{\mathscr{P}^{\tau}(\mathcal{P}r)\})$ is a $1$-dimensional smooth
manifold. Its compactification
$\overline{\mathcal{M}}^{\mathcal{P}r_2}(x,y,w)$ is a compact
$1$-dimensional smooth manifold with boundary. The boundary points of
$\partial \overline{\mathcal{M}}^{\mathcal{P}r_2}(x,y,w)$ \zjr{consist} of
five types of broken trajectories, as depicted in
Figure~\ref{f:Pr}. The first two types correspond to $\tau \to 1^{-}$
and $\tau \to 0^{+}$ respectively, and the number of occurrences of
each of them equals to the coefficient of $w$ in the first and second
terms on the left-hand side of~\eqref{eq:Pi-2-c}, respectively. The
other three types of broken trajectories, occurring at instances of
time $0 < \tau_0 < 1$, correspond to standard breaking along
strip-like ends. The coefficients of $w$ in each of the terms on the
right-hand side of~\eqref{eq:Pi-2-c} equals to the number of
occurrences of each of these three broken trajectories, respectively.
The identity~\eqref{eq:Pi-2-c} now follows (with appropriate sings
$\epsilon_i$) since the signed number of boundary points in
$\overline{\mathcal{M}}^{\mathcal{P}r_2}(x,y,w)$ must be $0$.  This
concludes the construction of the $A_2$-functor $\mathcal{P}r^i$ for
every $i \in \mathcal{B}(\bar{\mathcal{X}}; R, r)$.

\begin{figure}[h]
  \begin{center}
  \includegraphics[scale=0.65]{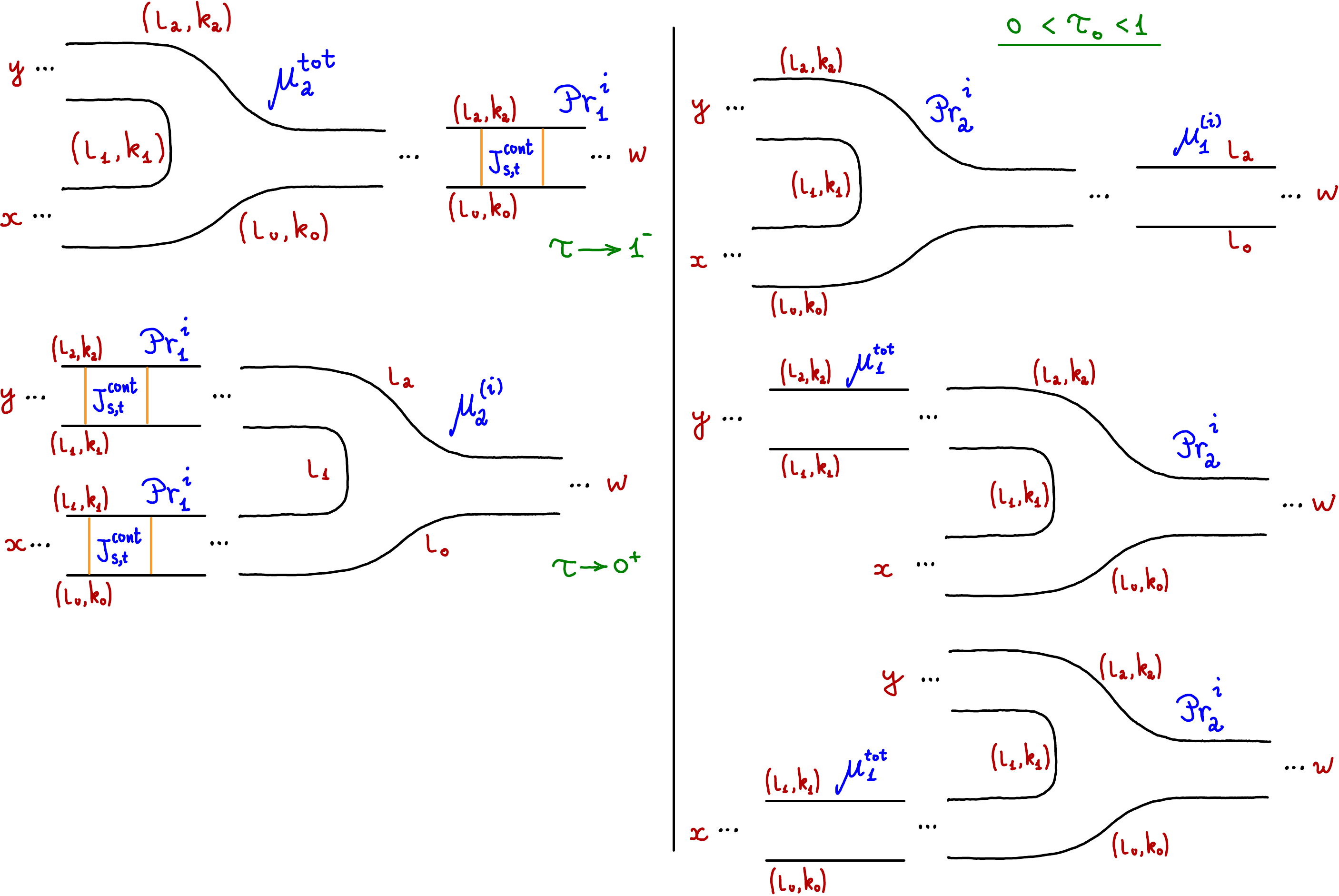}
  \end{center}
  \caption{The five possible types of boundary points of
    $\partial \overline{\mathcal{M}}^{\mathcal{P}r_2}(x,y,w)$.}
  \label{f:Pr}
\end{figure}

We now claim that $\mathcal{P}r^i$ \zjr{preserves} filtrations.  We have
already proved earlier that $\mathcal{P}r^i_1$ is filtration
preserving, so it remains to deal with the 2'nd component
$\mathcal{P}r^i_2$ of $\mathcal{P}r^i$. The fact that
$\mathcal{P}r^i_2$ is filtration preserving follows from
condition~\eqref{i:P-tau} on page~\pageref{i:P-tau}. Indeed, due to
inequality~\eqref{eq:nu-cr0}, the proof from
pages~\pageref{p:preserve-fil-1}--\pageref{p:preserve-fil-2} extends
with minor modifications to show that for every
$(\eta, u) \in \mathcal{M}^{\mathcal{P}r_2}(x, y, z;
\{\mathscr{P}^{\tau}(\mathcal{P}r)\})$ we have $A(u) > 0$ (where
$A(u)$ is the symplectic area of $u$).

We will soon apply Proposition~\ref{p:atot-coh} with
$\mathcal{A}^i = \fuk(\mathcal{X}; i)$,
$i \in \mathcal{I} := \mathcal{B}(\bar{\mathcal{X}}; R, r)$,
$\mathcal{A}^{\text{tot}} = \fuk^{\text{tot}}(\mathcal{X}; R,r)$. In
order to do so we still need to show several other properties of the
$A_2$-functors $\mathcal{P}r^i$, as required by
Proposition~\ref{p:atot-coh}.

The first one is that $\mathcal{P}r^i$ is strictly $A_1$-unital. This
follows immediately from Morse theory, since the continuation maps in
Morse theory send the unique local maximum \zhnote{of} the Morse function on
$(L,k)$ to the corresponding one for $(L,i)$. Therefore
$\mathcal{P}r^i_1$ sends strict units to strict units.

The next property is that
$\mathcal{P}r^i \circ \mathcal{J}^i = \id_{\fuk(\mathcal{X}; i)}$ as
$A_2$-functors. This will follow form the following two statements:
\begin{enumerate}
\item For all pairs of the type $((L',i), (L'',i))$ we have
  $\mathcal{P}r^i_1 = \id$.
\item For any three objects of the type $(L_0,i)$, $(L_1,i)$,
  $(L_2,i)$ we must have $\mathcal{P}r^i_2 = 0$.
\end{enumerate}
The first statement has already been proved earlier.  The second one
follows from condition~\eqref{i:P-const} in the definition of the
family $\{\mathcal{P}^{\tau}(\mathcal{P}r)\}$, namely the requirement
that this family is constant (in $\tau$) and coincides with $i$ for
all $\tau$. This implies that $\mathcal{M}^{\mathcal{P}r_2}(x,y,z)$ is
the same space as the one defining the operation $\mu_2$ (in the
category $\fuk(\mathcal{X};i)$). A transversality/dimension argument
now shows that whenever $|z|' = |x|'+|y|'-1$ we have
$\mathcal{M}^{\mathcal{P}r_2}(x,y,z) = \emptyset$. This proves the
second statement and concludes the proof that
$\mathcal{P}r^i \circ \mathcal{J}^i = \id_{\fuk(\mathcal{X}; i)}$ as
$A_2$-functors.

Finally, we claim that the $A_2$-functor $\mathcal{P}r^i$ has the
property described just before the statement of
Proposition~\ref{p:atot-coh}. Namely, for every object $(L,k)$ of
$\fuk^{\text{tot}}(\mathcal{X}; R,r)$ there exists a strict
isomorphism
$u^i_{(L,k)} \in \hom^{\leq 0}_{\fuk^{\text{tot}}(\mathcal{X};
  R,r)}((L,k), (L,i))$. Indeed, we can take $u^i_{(L,k)}$ to be the
unique critical point of index $n = \dim_{\mathbb{C}}X$ for the Morse
function in the Floer datum of the pair $((L,k), (L,i))$. Standard
arguments in Morse theory then show that $u^i_{(L,k)}$ is a strict
isomorphism.

We are now in position to apply Proposition~\ref{p:atot-coh} with
$\mathcal{A}^i = \fuk(\mathcal{X}; i)$,
$i \in \mathcal{I} := \mathcal{B}(\bar{\mathcal{X}}; R, r)$,
$\mathcal{A}^{\text{tot}} = \fuk^{\text{tot}}(\mathcal{X}; R,r)$.  By
that proposition we obtain the structure of a coherent system of
filtered $A_{\infty}$-categories on the family
$\{\fuk(\mathcal{X}; i)\}_{i \in \mathcal{B}(\bar{\mathcal{X}}; R,
  r)}$. Note that this holds for every $r \leq R$, hence in particular
also for $r=R$. Note also that
$\mathcal{B}(\bar{\mathcal{X}}; R, r) \subset
\mathcal{B}(\bar{\mathcal{X}}; R, R)$ and we can arrange our choices
(of $\mathscr{P}^{\text{tot}}$) such that the coherent system over the
larger base $\mathcal{B}(\bar{\mathcal{X}}; R, R)$ restricts to the
one over the smaller base $\mathcal{B}(\bar{\mathcal{X}}; R, r)$.

The above construction gives us many coherent systems. \zhnote{Namely,} one
coherent systems over each base
$\mathcal{B}(\bar{\mathcal{X}}; R, R)$, for all
$0<R \leq R_{\bar{\mathcal{X}}}$. We denote the comparison functors of
the coherent system over $\mathcal{B}(\bar{\mathcal{X}}; R, R)$ by
$\mathcal{F}^{j,i}_{R}: \fuk(\mathcal{X}; i) \longrightarrow
\fuk(\mathcal{X};j)$ for every
$i, j \in \mathcal{B}(\bar{\mathcal{X}}; R, R)$. Similarly, we denote
by $T_R^{i_2, i_1, i_0}$ the natural transformations of this system
(relating $\mathcal{F}^{i_2,i_1}_{R} \circ \mathcal{F}^{i_1,i_0}_{R}$
to $\mathcal{F}^{i_2,i_0}_{R}$).

\begin{rem} \label{r:big-tot} Ideally we would have liked to construct
  one total category
  \zhnote{$\fuk^{\rm tot}(\mathcal{X}; \mathcal{B}(\bar{\mathcal{X}}))$} over the entire
  base $\mathcal{B}(\bar{\mathcal{X}})$. Unfortunately this is not so
  straightforward to achieve, at least \zhnote{not with} our construction of the
  total categories. The difficulty has to do with establishing a set
  of perturbation data $\mathscr{P}^{\text{tot}}$ (over the entire of
  $\mathcal{B}(\bar{\mathcal{X}})$ and in fact even over subspaces of
  it of the form
  $\mathcal{B}(\bar{\mathcal{X}}; R_1,R_1) \cup
  \mathcal{B}(\bar{\mathcal{X}}; R_2,R_2)$) that is both consistent
  (with respect to splitting/gluing) and at the same time also yields
  filtration preserving operations $\mu_d$.
  
  To understand better the difficulty, consider the case of
  $\mathcal{B}(\bar{\mathcal{X}}; R_1,R_1) \cup
  \mathcal{B}(\bar{\mathcal{X}}; R_2,R_2)$, where
  $R_2 < R_1 \leq R_{\bar{\mathcal{X}}}$. Let $(i_0, i_1, i_2, i_3)$
  be a tuple of perturbation data with
  \begin{equation*}
    \begin{aligned}
      & i_1, i_3 \in \mathcal{B}(\bar{\mathcal{X}}; R_1,R_1) \cap
        \mathcal{B}(\bar{\mathcal{X}}; R_2,R_2) =
        \mathcal{B}(\bar{\mathcal{X}}; R_1,R_2), \\
      & i_2 \in \mathcal{B}(\bar{\mathcal{X}}; R_1,R_1) \setminus
        \mathcal{B}(\bar{\mathcal{X}}; R_2,R_2), \quad 
        i_0 \in \mathcal{B}(\bar{\mathcal{X}}; R_2,R_2) \setminus
        \mathcal{B}(\bar{\mathcal{X}}; R_1,R_1).
    \end{aligned}
  \end{equation*}
  For some choices of $i_0, i_1, i_2, i_3$ we might have:
  \begin{equation} \label{eq:different-Rs}
    \begin{aligned}
      & R^{(i_0, i_1, i_2, i_3)} = R_2, & R^{(i_1, i_2, i_3)} = R_1, \ \ \ 
      & R^{(i_0, i_1, i_3)} = R_2, \\
      & r^{(i_0, i_1, i_2, i_3)} = R_1, & r^{(i_1, i_2, i_3)} = R_1, \ \ \ 
      & r^{(i_0, i_1, i_3)} = R_2,
    \end{aligned}
  \end{equation}
  where $R^{(\, \cdots)}$ and $r^{(\,\cdots)}$ are defined
  in~\eqref{eq:Ri-ri}. Now recall that in order to obtain a choice of
  perturbation data that preserves filtrations we used
  inequalities~\eqref{eq:R-leq-Ri-etc} and~\eqref{eq:nu-P-tot}. Since
  $R_2 < R_1$, in our case
  \zhnote{we have $R^{(i_0, i_1, i_2, i_3)} < r^{(i_0, i_1, i_2, i_3)}$,} so
  inequality~\eqref{eq:R-leq-Ri-etc} does not hold. There are also
  problems regarding the consistency of the perturbation data with
  respect to splitting and gluing. Indeed, a Floer polygon labeled by
  $((L_0, i_0), (L_1, i_1), (L_2, i_2), (L_3, i_3))$ can split into
  two polygons labeled by $((L_1,i_1), (L_2, i_2), (L_3, i_3))$ and by
  $((L_0, i_0), (L_1, i_1), (L_3, i_3))$. And vice-versa, pairs of
  polygons of the latter type can be glued into polygons of the former
  type (assuming obvious matching assumptions). In our case we have
  $r^{(i_1, i_2, i_3)} > r^{(i_0, i_1, i_3)}$ so in order to achieve
  the consistency of the perturbation data and at the same time also
  have filtration preserving $\mu_d$'s we would need to \zhnote{decrease} the
  size of the perturbation forms \zhnote{of}
  $\mathscr{P}^{\text{tot}}((L_1,i_1), (L_2, i_2), (L_3, i_3))$ to at
  most $r^{(i_0, i_1, i_3)}$. In turn, this might not be compatible
  with the size of the perturbation forms of the data $i_2$ (which may
  be of size $R_1$). Similar problems arise also with the behavior of
  the measurement $R^{(i_0, \ldots, i_d)}$ with respect to
  splitting/gluing when the indices $i_k$ are spread in spaces
  $\mathcal{B}(\bar{\mathcal{X}}; R,R)$ with different values of $R$.

  Summing up, while our construction for total categories works over
  bases like $\mathcal{B}(\bar{\mathcal{X}}; R,R)$\zhnote{, it is} not clear if
  the construction can be extended over a base of the type
  $\mathcal{B}(\bar{\mathcal{X}}; R_1,R_1) \cup
  \mathcal{B}(\bar{\mathcal{X}}; R_2,R_2)$.

  In view of the above, instead of constructing a total category over
  $\mathcal{B}(\bar{\mathcal{X}})$ we will simply try to extend the
  various coherent systems (coming from the total categories) over
  each $\mathcal{B}(\bar{\mathcal{X}}; R,R)$ to coherent systems over
  unions of such subspaces.  \Qed
\end{rem}

We proceed now with the extension of the system of comparison functors
to spaces beyond the type $\mathcal{B}(\bar{\mathcal{X}};
R,R)$. Recall that if $R' \leq R \leq R_{\bar{\mathcal{X}}}$ then
$\mathcal{B}(\bar{\mathcal{X}}; R,R) \cap
\mathcal{B}(\bar{\mathcal{X}}; R',R') = \mathcal{B}(\bar{\mathcal{X}};
R,R')$. We claim that the construction of the total categories above
and the coherent systems resulting from them can be made such that the
following holds for every $R' \leq R \leq R_{\bar{\mathcal{X}}}$: the
two coherent systems, the one over
$\mathcal{B}(\bar{\mathcal{X}}; R,R)$ and the one over
$\mathcal{B}(\bar{\mathcal{X}}; R',R')$, coincide over the overlap
$\mathcal{B}(\bar{\mathcal{X}}; R,R')$. In other words, the functors
$\mathcal{F}^{j,i}_{R}$ and $\mathcal{F}^{j,i}_{R'}$ coincide for
every $i,j \in \mathcal{B}(\bar{\mathcal{X}}; R,R')$, and similarly
for the natural transformations $T_R^{i_2, i_1, i_0}$,
$T_{R'}^{i_2, i_1, i_0}$.

The proof of this is based on three steps. The first one is that we
can choose the perturbation data for the total categories in such a
way that the categories $\fuk^{\text{tot}}(\mathcal{X}; R,R)$ and
$\fuk^{\text{tot}}(\mathcal{X}; R',R')$ coincide along the overlaps
$\mathcal{B}(\bar{\mathcal{X}}; R,R')$ (and are equal along that
overlap to $\fuk^{\text{tot}}(\mathcal{X}; R,R')$). The way we \zhnote{chose}
the perturbation data $\mathscr{P}^{\text{tot}}$ makes this possible
since in case
$i_0, \ldots, i_d \in \mathcal{B}(\bar{\mathcal{X}}; R,R) \cap
\mathcal{B}(\bar{\mathcal{X}}; R',R')$ then the definitions of the
parameters $R^{(i_0, \ldots, i_d)}$ and $r^{(i_0, \ldots, i_d)}$ is
independent of whether we view the indices $i_k$ as elements of
$\mathcal{B}(\bar{\mathcal{X}}; R,R)$ or of
$\mathcal{B}(\bar{\mathcal{X}}; R',R')$.

The second step has to do with the $A_2$-functors $\mathcal{P}r^i$. To
keep track of the domain of these functors, we will temporarily add a
subscript $R$ to their notation:
$\mathcal{P}r_R^i: \fuk^{\text{tot}}(\mathcal{X}; R,R) \longrightarrow
\fuk(\mathcal{X}; i)$, $i \in \mathcal{B}(\bar{\mathcal{X}}; R, R)$.
We also have
$\mathcal{P}r_{R'}^i: \fuk^{\text{tot}}(\mathcal{X}; R,R)
\longrightarrow \fuk(\mathcal{X}; i)$,
$i \in \mathcal{B}(\bar{\mathcal{X}}; R', R')$. Recall that
$\mathcal{B}(\bar{\mathcal{X}}; R,R) \cap
\mathcal{B}(\bar{\mathcal{X}}; R',R') = \mathcal{B}(\bar{\mathcal{X}};
R,R')$ and that $\fuk^{\text{tot}}(\mathcal{X}; R,R')$ is a
subcategory of both $\fuk^{\text{tot}}(\mathcal{X}; R,R)$ and
$\mathcal{B}(\bar{\mathcal{X}}; R', R')$, which are the domains of
$\mathcal{P}r^i_R$ and $\mathcal{P}r^i_{R'}$, respectively. We need to
show that for every $i \in \mathcal{B}(\bar{\mathcal{X}}; R,R')$ the
two $A_2$-functors $\mathcal{P}r^i_R$ and $\mathcal{P}r^i_{R'}$
coincide when restricted to the subcategory
$\fuk^{\text{tot}}(\mathcal{X}; R,R')$ of their respective domains. \zhnote{This} can be achieved by arranging the $1$-parametric families of
\zhnote{the} perturbation data
$\{\mathcal{P}^{\tau}(\mathcal{P}r)\}_{\tau \in (0,1)}$ for
$\mathcal{B}(\bar{\mathcal{X}}; R,R)$ and
$\mathcal{B}(\bar{\mathcal{X}}; R',R')$ to coincide over their
intersection.

The third and last step concerns the extensions of the preceding
$A_2$-functors to $A_{\infty}$-functors as stated in
Proposition~\ref{p:atot-coh}. Here we need to show that the
$A_{\infty}$-extension \zhnote{$\mathcal{Q}_R^i$} of $\mathcal{P}r^i_R$ can be arranged
to coincide with the $A_{\infty}$-extensions \zhnote{$\mathcal{Q}_{R'}^i$} of
$\mathcal{P}r^i_{R'}$ over the common subcategory of their domains
$\fuk^{\text{tot}}(\mathcal{X}; R,R')$ whenever
$i \in \mathcal{B}(\bar{\mathcal{X}}; R,R')$.

This can be proved by means of the algebraic
Lemmas~\ref{l:PHH-invariance},~\ref{l:extending-func} and
Proposition~\ref{p:atot-coh}. Note that for
$i \in \mathcal{B}(\bar{\mathcal{X}}; R,R')$ the inclusion of
$\fuk(\mathcal{X}; i)$ into any of the categories
$\fuk^{\text{tot}}(\mathcal{X}; R,R')$,
$\fuk^{\text{tot}}(\mathcal{X}; R,R)$ and
$\fuk^{\text{tot}}(\mathcal{X}; R',R')$ is a filtered
quasi-equivalence.


\pbrev{We now claim that the following four maps
\begin{equation*} \label{eq:PHH-fuktot}
  \begin{aligned}
    & PHH\bigl(\fuk^{\text{tot}}(R,R), \fuk(i); \mathcal{P}r_R^i,
    \mathcal{P}r_R^i\bigr) \longrightarrow
    PHH(\fuk(i), \fuk(i); \id, \id), \\
    & PHH\bigl(\fuk^{\text{tot}}(R',R'), \fuk(i); \mathcal{P}r_{R'}^i,
    \mathcal{P}r_{R'}^i\bigr) \longrightarrow
    PHH(\fuk(i), \fuk(i); \id, \id), \\
    & PHH\bigl(\fuk^{\text{tot}}(R,R), \fuk(i); \mathcal{P}r_{R}^i,
    \mathcal{P}r_{R}^i\bigr) \longrightarrow
    PHH\bigl(\fuk^{\text{tot}}(R,R'), \fuk(i); \mathcal{P}r_{R,R'}^i,
    \mathcal{P}r_{R,R'}^i\bigr), \\
    & PHH\bigl(\fuk^{\text{tot}}(R',R'), \fuk(i); \mathcal{P}r_{R'}^i,
    \mathcal{P}r_{R'}^i\bigr) \longrightarrow
    PHH\bigl(\fuk^{\text{tot}}(R,R'), \fuk(i); \mathcal{P}r_{R,R'}^i,
    \mathcal{P}r_{R,R'}^i\bigr),
  \end{aligned}
\end{equation*}
} \pbrev{induced by the obvious restrictions, are all isomorphisms of
  persistence modules. Here we have omitted "$\mathcal{X}$" from all
  the $\fuk$-categories in an attempt to keep the formulas short.  The
  $A_2$-functor
  $\mathcal{P}r_{R,R'}^i : \fuk^{\text{tot}}(R,R') \longrightarrow
  \fuk(i)$ appearing in third and forth maps above is just
  $\mathcal{P}r_{R,R'}^i :=
  \mathcal{P}r^i_R|_{\fuk^{\text{tot}}(R,R')} =
  \mathcal{P}r^i_{R'}|_{\fuk^{\text{tot}}(R,R')}$.}

\pbrev{Indeed, that the first two maps are isomorphisms follows from
  Lemma~\ref{l:PHH-invariance}.} \pbrev{For the third map, note that
  by Lemma~\ref{l:PHH-invariance}, for every $R' \leq R$ the map
  \begin{equation} \label{eq:PHH-PRR'}
    PHH\bigl(\fuk^{\text{tot}}(R,R'), \fuk(i); \mathcal{P}r_{R,R'}^i,
    \mathcal{P}r_{R,R'}^i\bigr) \longrightarrow PHH(\fuk(i), \fuk(i);
    \id, \id)
  \end{equation}
  (also induced by restriction) is an isomorphism of persistence
  modules. Now, the first map (among the above four) factors as a
  composition of the third map and the map
  in~\eqref{eq:PHH-PRR'}. Since the latter map and the first map are
  both isomorphisms it follows that the same holds for the third
  map. The proof that the fourth map is an isomorphisms is similar.}

\pbrev{We proceed} now by extending
$\mathcal{P}r^i_R|_{\fuk^{\text{tot}}(\mathcal{X}; R,R')} =
\mathcal{P}r^i_{R'}|_{\fuk^{\text{tot}}(\mathcal{X}; R,R')}$ to an
$A_{\infty}$-functor
$$\mathcal{Q}^i_{R,R'}: \fuk^{\text{tot}}(\mathcal{X}; R,R')
\longrightarrow \fuk(\mathcal{X}; i).$$ Next we extend
$\mathcal{Q}^i_{R,R'}$ twice more: once to an $A_{\infty}$-functor
$\mathcal{Q}^i_R: \fuk^{\text{tot}}(\mathcal{X}; R,R) \longrightarrow
\fuk(\mathcal{X}; i)$ and another time to an $A_{\infty}$-functor
$\mathcal{Q}^i_{R'}: \fuk^{\text{tot}}(\mathcal{X}; R',R')
\longrightarrow \fuk(\mathcal{X}; i)$. This concludes the third step.

Finally, \pbrev{it is possible to} show that the systems of natural
transformations $T_R^{i_2, i_1, i_0}$ and $T_{R'}^{i_2, i_1, i_0}$ can
be chosen to agree over the intersection
$\mathcal{B}(\bar{\mathcal{X}}; R,R) \cap
\mathcal{B}(\bar{\mathcal{X}}; R',R')$. This can be done by arguments
similar to above (based on the construction of the natural
transformations from the proof of Proposition~\ref{p:atot-coh}).

Recall that the comparison functors $\mathcal{F}_R^{j,i}$ in the
coherent system over $\mathcal{B}(\bar{\mathcal{X}}; R,R)$ are given
by $\mathcal{F}_R^{j,i} = \mathcal{Q}_R^{j} \circ \mathcal{J}_R^i$,
where
$\mathcal{J}_R^i : \fuk(\mathcal{X}; i) \longrightarrow
\fuk^{\text{tot}}(\mathcal{X}; R,R)$ is the inclusion functor. The
comparison functors over $\mathcal{B}(\bar{\mathcal{X}}; R',R')$ and
$\mathcal{B}(\bar{\mathcal{X}}; R,R')$ have essentially the same
expressions.

Before we proceed, let us summarize what we have proven so far.
\begin{prop} \label{p:coh-sys-w-ovlp} The space of perturbation data
  $\mathcal{B}(\bar{\mathcal{X}})$ can be written as union
  $$\mathcal{B}(\bar{\mathcal{X}}) = \bigcup_{R \leq
    R_{\bar{\mathcal{X}}}} \mathcal{B}(\bar{\mathcal{X}}; R,R)$$ of
  mutually overlapping subspaces
  $\mathcal{B}(\bar{\mathcal{X}}; R,R)$. The overlap between any two
  of these subspaces is given by
  $\mathcal{B}(\bar{\mathcal{X}}; R,R) \cap
  \mathcal{B}(\bar{\mathcal{X}}; R',R') =
  \mathcal{B}(\bar{\mathcal{X}}; R,R')$ for every $R' \leq R$.

  The family of Fukaya categories
  $\{\fuk(\mathcal{X}; i)\}_{i \in \mathcal{B}(\bar{\mathcal{X}};
    R,R)}$ over each subspace $\mathcal{B}(\bar{\mathcal{X}};R,R)$ can
  be endowed with a coherent system of comparison
  $A_{\infty}$-functors that are filtered quasi-equivalences and are
  strictly $A_1$-unital.  Moreover, for every $R, R'$,
  there exist coherent systems over $\mathcal{B}(\bar{\mathcal{X}}; R,R)$ and
  over $\mathcal{B}(\bar{\mathcal{X}}; R',R')$  that agree along the overlap
  $\mathcal{B}(\bar{\mathcal{X}}; R,R) \cap
  \mathcal{B}(\bar{\mathcal{X}}; R',R')$.

  Finally, all the comparison functors in the various coherent
  systems above act as the identity \zhnote{maps} on the objects of the
  respective categories. The first order terms of the comparison
  functors induce the canonical continuation maps in Floer homology.
\end{prop}

The structure of the subspaces $\mathcal{B}(\bar{\mathcal{X}}; R,R)$
and their overlaps is schematically depicted in
Figure~\ref{f:coh-ovl-1}.

\begin{figure}[h]
  \begin{center}
  \includegraphics[scale=0.80]{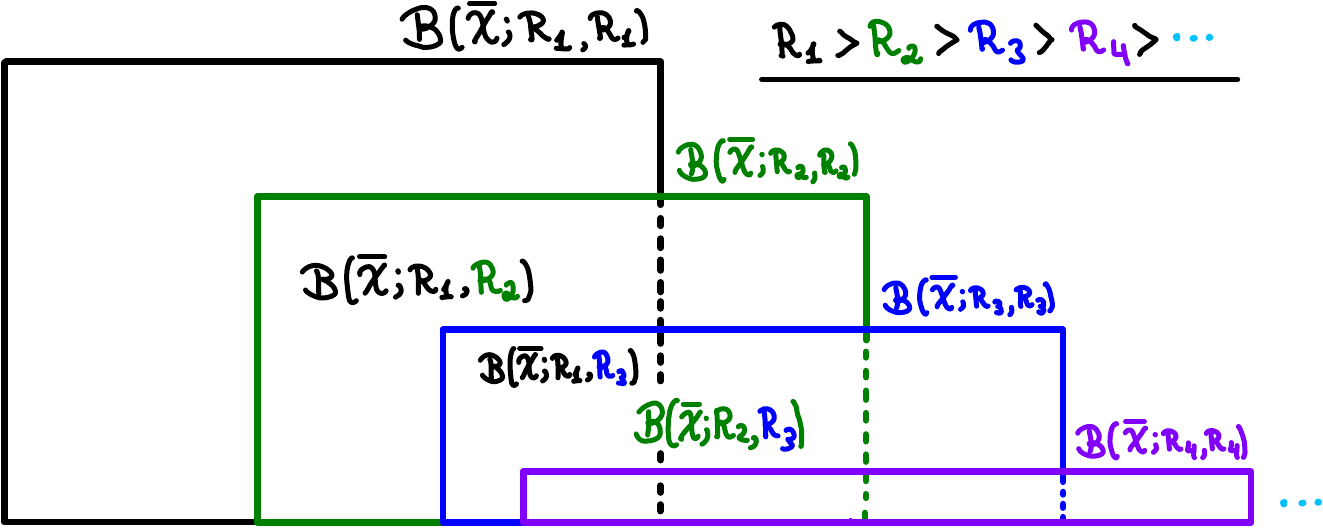}
  \end{center}
  \caption{The subspaces $\mathcal{B}(\bar{\mathcal{X}}; R,R)$ \zhnote{and their
   overlaps.}}
  \label{f:coh-ovl-1}
\end{figure}

\begin{rem} \label{r:coh-ovl}
  The type of structure described in
  Proposition~\ref{p:coh-sys-w-ovlp} can be defined abstractly for any
  collection of $A_{\infty}$-categories (filtered or not)
  $\{\mathcal{A}^i\}_{i \in \mathcal{B}}$ over a base
  $\mathcal{B}$. An appropriate name for such a structure could be
  {\em a collection of coherent systems with overlaps}. We will not
  pursue this direction anymore in this paper and proceed now with our
  Fukaya categories.
\end{rem}


We are ready now to describe the functors
$\mathcal{F}^{\mathscr{P}_1, \mathscr{P}_0}$ claimed in
Theorem~\ref{t:fil-fuk}. To simplify the notation we continue to
denote perturbation data by indices like $i$ instead of $\mathscr{P}$
and the comparison functors will be denote $\mathcal{F}^{j,i}$ instead
of $\mathcal{F}^{\mathscr{P}_1, \mathscr{P}_0}$. Note that the
construction below is purely formal and can be applied to any family
of $A_{\infty}$-categories endowed with a collection of coherent
systems with overlaps (as in Remark~\ref{r:coh-ovl}).

For every $R', R'' \leq R_{\bar{\mathcal{X}}}$ we choose a base point
$$l(R', R'') \in \mathcal{B}(\bar{\mathcal{X}};\max\{R', R''\},
\min\{R', R''\}).$$ Given $i \in \mathcal{B}(\bar{\mathcal{X}})$
define
$R^{(i)} := \sup \{ \widetilde{R} \mid \widetilde{R} \leq
R_{\bar{\mathcal{X}}}, \text{ and } i \in
\mathcal{B}(\bar{\mathcal{X}}; \widetilde{R},\widetilde{R})\}$, which
is a special case of the first parameter in~\eqref{eq:Ri-ri}.

Let $i, j \in \mathcal{B}(\bar{\mathcal{X}})$. If $i=j$ define
$\mathcal{F}^{i,j} = \id_{\fuk(\mathcal{X}; i)}$. If $i \neq j$,
set $l_{i,j}:= l(R^{(i)}, R^{(j)})$ and define
\begin{equation} \label{eq:cmpr-fun}
  \mathcal{F}^{j,i} :=
  \mathcal{F}_{R^{(j)}}^{j, l_{{i,j}}} \circ
  \mathcal{F}_{R^{(i)}}^{l_{{i,j}}, i}: \fuk(\mathcal{X}; i)
  \longrightarrow \fuk(\mathcal{X}; j).
\end{equation}
It is easy to see that the functors $\mathcal{F}^{j, i}$ satisfy all
the properties claimed in Theorem~\ref{t:fil-fuk}.

\begin{rem} \label{r:can-sys} The system of functors
  $\mathcal{F}^{j,i}$ defined by~\eqref{eq:cmpr-fun} is not
  canonical. The construction uses many different choices at different
  stages. However, one can show that the dependence on the choices of
  the base points $l_{R', R''}$ is somewhat controlled. If one
  replaces the base points $l_{R',R''}$ by a different set of choice
  $l'_{R',R''}$ then the resulting system of comparison functors
  $\mathcal{F}'^{j,i}$ will be naturally quasi-isomorphic to the
  system $\mathcal{F}^{j,i}$ by natural quasi-isomorphisms that
  preserve filtrations.

  The applications in this paper do not require the
  comparison functors or the natural transformations between them to
  be canonical. In fact, in what follows we just need to know that the
  different Fukaya categories $\fuk(\mathcal{X}; i)$ are filtered
  quasi-equivalent.
\end{rem}

It remains to address the last statement of Theorem~\ref{t:fil-fuk}
concerning the collection of Lagrangians $\bar{\mathcal{X}}'$. This
follows immediately from our construction.  Indeed, let
$\bar{\mathcal{X}}'$ be a collection of Lagrangians with
$\bar{\mathcal{X}}' \supset \bar{\mathcal{X}}$ and assume that
$\bar{\mathcal{X}}'$ satisfies all the conditions from the beginning
of~\S\ref{sb:faitpc}. We can use the same symplectic embedding $\phi$
from~\eqref{eq:phi-embedd} but with a smaller radius
$R_{\bar{\mathcal{X}}'} \leq R_{\bar{\mathcal{X}}}$. Clearly
$\mathcal{B}(\bar{\mathcal{X}}')|_{\bar{\mathcal{X}}} \subset
\mathcal{B}(\bar{\mathcal{X}})$.

This completes the  proof of Theorem~\ref{t:fil-fuk}.
\Qed

\section{Proofs of the main symplectic applications}\label{sec:all-symp}
The first subsection is dedicated to the proofs of Theorems  \ref{cor:Fuk-cat-var} and \ref{thm:appl-sympl} and the second to the proof of Corollary \ref{cor:gl-metric}.

\subsection{Proofs of Theorems \ref{cor:Fuk-cat-var} and \ref{thm:appl-sympl}}\label{subs:proof}
We \zhnote{begin with} Theorem \ref{cor:Fuk-cat-var}\zhnote{, and then pursue with the proof of Theorem \ref{thm:appl-sympl} in \S\ref{proof-thm:appl-sympl}} on page \pageref{proof-thm:appl-sympl}. 

\subsubsection{Proof of Theorem \ref{cor:Fuk-cat-var}}\label{subsubsec:proof-Cor}
Recall from Theorem \ref{t:fil-fuk} that  there are choices of Floer data and
  perturbation data $\mathscr{P}$ such that the resulting Fukaya category, $\fuk(\mathcal{X}; \mathscr{P})$, is filtered, strictly unital, and, without filtrations, it is quasi-equivalent to the subcategory $\fuk(\mathcal{X})$
 of $\fuk(X)$ whose  collection of objects is $\mathcal{X}$ \zhnote{(see \S\ref{subsec-symp} for the notation)}.

 We will apply the discussion in \S\ref{sb:faitpc} to the filtered $A_{\infty}$-category $\fuk(\mathcal{X}; \mathscr{P})$.  The first step is to discuss the shift functor. There is an obvious shift functor on the 
 category $\fuk(\mathcal{X}; \mathscr{P})$ which acts on objects by 
  $\Sigma^r L = (\bar{L}, h_L+r, \theta_{L})$ \zhnote{(see \S\ref{subsec-symp})}. This action on objects 
  induces an $A_{\infty}$-shift functor on  $\fuk(\mathcal{X}; \mathscr{P})$ because the perturbation
  data $\mathscr{P}$ only depends on the geometric part of the marked Lagrangians. It is easy to see 
  that this shift functor on $\fuk(\mathcal{X}; \mathscr{P})$ is compatible with the shift functor
  defined on $Fmod(\fuk(\mathcal{X}; \mathscr{P}))$ as in \S\ref{sb:faitpc}:
 for a module $\mathcal{M} \in Fmod(\mathcal{A})$ the filtered module
  $\Sigma^r \mathcal{M}$ is defined by
  $(\Sigma^r \mathcal{M})^{\leq \alpha}(N) = \mathcal{M}^{\alpha -
    r}(N)$,  with the same \zhnote{$\mu_d$}-operations as $\mathcal{M}$.
  Indeed,
  $$\Sigma^{r}(\mathcal{Y}(L))=\mathcal{Y}(\Sigma^{r} L).$$
 
 We now define: 
 \zhnote{$$\C\fuk(\mathcal{X};\mathscr{P})=H^{0}(\fuk(\mathcal{X}; \mathscr{P})^{\nabla})~.~$$}
 Recall that $\fuk(\mathcal{X}; \mathscr{P})^{\nabla}$ is constructed by first considering the 
 category $\fuk(\mathcal{X}; \mathscr{P})^{\#}$ consisting of the triangulated completion 
 of the Yoneda modules of the elements in $\mathcal{X}$. Triangles are understood here to be of the form
 $\mathcal{M}\stackrel{f}{\longrightarrow} \mathcal{N}\to {\rm Cone\/}(f)$ where  the cone construction is in the sense of filtered $A_{\infty}$-modules and the morphism $f$ preserves filtration. 
 
The category $\fuk(\mathcal{X}; \mathscr{P})^{\nabla}$ is the \zhnote{full subcategory} of $Fmod(\fuk(\mathcal{X}; \mathscr{P}))$ \zhnote{that contains} $\fuk(\mathcal{X}; \mathscr{P})^{\#}$ as well as all the modules in $Fmod(\fuk(\mathcal{X}; \mathscr{P}))$ \zhnote{together with} all their shifts and translates that are $r$-quasi-isomorphic to the objects in $\fuk(\mathcal{X}; \mathscr{P})^{\#}$, for some $r\geq 0$.  From the discussion above concerning the shift functor it results immediately that $\C\fuk(\mathcal{X};\mathscr{P})$ is indeed a TPC.

\zhnote{\begin{remark} More explicitly, to obtain $\fuk(\mathcal{X}; \mathscr{P})^{\nabla}$ from $\fuk(\mathcal{X}; \mathscr{P})^{\#}$, we first add $r$-acyclic modules for all $r \geq 0$ to the Yoneda modules with all the possible shifts and translates; then we take the triangulated completion of all of them. \end{remark} }

We now start the proof by discussing the independence of $\mathcal{C}\fuk(\mathcal{X};\mathscr{P})$ of the perturbation data $\mathscr{P}$, up to TPC equivalence. \zhnote{The argument is a direct consequence
of the system of comparison functors for the categories $\fuk(\mathcal{X};\mathscr{P})$, as given in Theorem \ref{t:fil-fuk}. Indeed, for two choices
of admissible perturbation data $\mathscr{P}_{1}$ and $\mathscr{P}_{2}$,}
we have  a filtered  functor $$\mathcal{F}^{1,2}:\fuk(\mathcal{X}, \mathscr{P}_{1})\to \fuk(\mathcal{X}, \mathscr{P}_{2})~.~$$ 
The existence of the natural transformations that compare the compositions
$\mathcal{F}^{1,2}\circ \mathcal{F}^{2,1}$, $\mathcal{F}^{2,1}\circ \mathcal{F}^{1,2}$ with the \zhnote{respective} identities implies
that the associated homological functor is full and faithful in \zhnote{the sense} 
that it induces an isomorphism \zhnote{of persistence modules $H(\mathcal{F}^{1,2}): H^{0}(\mor^{\leq r}(X,Y))\to H^{0}(\mor^{\leq r}(X, Y))$ for $r \in \R$ and
every} two objects $X,Y$ of $\fuk(\mathcal{X}, \mathscr{P}_{1})$. As in the unfiltered case, a  consequence of the existence of these functors
and the natural \zhnote{transformation} relating them, is that the pull-back of filtered modules $[\mathcal{F}^{1,2}]^{\ast}: Fmod(\fuk(\mathcal{X}, \mathscr{P}_{2}))\to Fmod(\fuk(\mathcal{X}, \mathscr{P}_{1}))$ sends 
each Yoneda modules $\mathcal{Y}_{\mathscr{P}_{2}}(L)$ to a module 
$0$-quasi-isomorphic to $\mathcal{Y}_{\mathscr{P}_{1}}(L)$ (in the sense that the two modules are related by a morphism that induces a $0$-isomorphism in the homological category). The standard properties of the pullback  of 
\zhnote{$A_{\infty}$-modules} imply that $[\mathcal{F}^{1,2}]^{\ast}$ respects \zhnote{the triangulated structure with respect} to 
$0$-weight triangles as well as shift functors.  We deduce that this pull-back
also sends $r$-isomorphisms to $r$-isomorphisms and thus that it sends 
$\fuk(\mathcal{X},\mathscr{P}_{2})^{\nabla}$ to \zjr{$\fuk(\mathcal{X},\mathscr{P}_{1})^{\nabla}$} and, finally, by using the relevant natural transformations, that $H[\mathcal{F}^{1,2}]^{\ast}$ is an equivalence of 
TPCs.

From now on, we will denote the resulting TPC by $\mathcal{C}\fuk(\mathcal{X})$ \zhnote{in place} of $\mathcal{C}\fuk(\mathcal{X};\mathscr{P})$, except if at risk of confusion.

We now continue with the \zhnote{points (i), (ii) , (iii)} of Theorem \ref{cor:Fuk-cat-var}.
 The first point is immediate because $\Mor_{\C\fuk(\mathcal{X})}(L,L')\cong HF(L,L')$ as persistence modules, with the persistence structure on $HF(L,L')$ as described in \S\ref{sb:fil-fuk}.  
  
 The second point of the Theorem claims that the $\infty$-level of
 $\C\fuk(\mathcal{X})$ is equivalent to $D\fuk(\mathcal{X})$.  The
 argument is the following. First, it is well-known \cite{Seidel} that
 an alternative, equivalent, model for $D\fuk(\mathcal{X})$ is given
 by the homological category of twisted complexes,
 \zhnote{$H^{0}(Tw(\fuk(\mathcal{X})))$}. Here $\fuk(\mathcal{X})$ is
 any $A_{\infty}$-category that represents the Fukaya category with
 objects the Lagrangians in $\mathcal{X}$. Thus we can take in its
 place $\fuk_{uf}(\mathcal{X}; \mathscr{P}))$ where the notation
 $(-)_{uf}$ means that we neglect the filtration.  A variant of Lemma
 \ref{lemma-exist-fil-tc} \zhnote{applies also to} our
 $A_{\infty}$-setting and it implies that each twisted complex in
 $Tw(\fuk_{uf}(\mathcal{X}; \mathscr{P}))$ can be viewed as a filtered
 twisted complex in $Tw(\fuk(\mathcal{X}; \mathscr{P}))$ whose
 filtration is forgotten.  Filtered twisted complexes are discussed in
 \S\ref{sb:faitpc} \zhnote{(and in more detail in the $dg$-case, in
   \S\ref{subsec:dg})}.  By passing to \zhnote{homology} this means
 that
 $[H^{0}Tw(\fuk(\mathcal{X}; \mathscr{P}))]_{\infty}\cong
 H^{0}Tw(\fuk_{uf}(\mathcal{X}; \mathscr{P}))\cong
 D\fuk(\mathcal{X})$. In turn $Tw((\fuk(\mathcal{X}; \mathscr{P}))$ is
 easily seen to be equivalent to $\fuk(\mathcal{X}; \mathscr{P})^{\#}$
 as TPC categories. Thus,
 $[\fuk(\mathcal{X}; \mathscr{P})^{\nabla}]_{\infty}\cong
 [\fuk(\mathcal{X}; \mathscr{P})^{\#}]_{\infty} \cong
 [H^{0}Tw(\fuk(\mathcal{X}; \mathscr{P}))]_{\infty}\cong
 D\fuk(\mathcal{X})$

We now \zhnote{turn} to the third point of Theorem \ref{cor:Fuk-cat-var}. 
The argument here makes again essential use of \zhnote{the systems} of comparison functors provided \zhnote{by Theorem \ref{t:fil-fuk}}, which allows us to extend perturbation
data from one set $\mathcal{X}$ to a larger one, $\mathcal{X}'$.

Let $N$
be a marked Lagrangian that \zhnote{is in a general position with respect to} the family \zhnote{$\bar{\mathcal{X}}$}. We add $N$\zhnote{, as well as all its shifts and translates,} to the family $\mathcal{X}$ obtaining this way $\mathcal{X}'$. By Theorem  \ref{t:fil-fuk}, and the invariance of $\mathcal{C}\fuk(\mathcal{X})$ relative to perturbation data, we can now assume that the perturbation \zhnote{data}
$\mathscr{P}$ extends to a new \zhnote{one} $\mathscr{P}'$  that defines the  filtered $A_{\infty}$- category $\fuk(\mathcal{X}',\mathscr{P}')$. Finally, we can 
extend $\mathscr{P}'$ to a perturbation \zhnote{data} that defines an $A_{\infty}$-category \zhnote{which is equivalent to $\fuk(X)$ if filtrations are forgotten}. 
At this point it is easier to pursue the argument using twisted complexes.
As before these will be of two types:  filtered\zhnote{, those from} $Tw(\fuk(\mathcal{X}; \mathscr{P}))$
and unfiltered,  belonging to $Tw(\fuk_{uf}(\mathcal{X}'; \mathscr{P}'))$.   
The assumption that $\mathcal{X}$ generates $D\fuk(X)$ implies that there
is  an unfiltered quasi-isomorphism of twisted complexes $\phi: N\to C$
where $C$ is a twisted module only involving elements of $\mathcal{X}$.
Both $N$ and $C$ have \zhnote{structures} of filtered twisted complexes but the morphism $\phi$, \zhnote{a priori}, does not see the filtration.

However, using a reasoning similar to Lemma \ref{lemma-exist-fil-tc},
we deduce that (after possibly shifting up $N$) we can view $\phi$ as
a filtration preserving morphism in
$\Mor_{Tw(\fuk(\mathcal{X}'; \mathscr{P}'))}$. The mapping cone $K$ of
$\phi$ is acyclic as a twisted module (forgetting the filtration)
because $\phi$ is a quasi-isomorphism. In other words, the identity
$1_{K}$ of $K$ is homologous to $0$ in
\zjr{$\Mor_{Tw(\fuk_{un}(\mathcal{X}'; \mathscr{P}'))}(K,K)$}. Using again a
reasoning similar to Lemma \ref{lemma-exist-fil-tc}, we deduce that
$1_{K}$ vanishes in some
\zhnote{$$H^0\left(\Mor^{\leq r}_{Tw(\fuk(\mathcal{X}';
      \mathscr{P}'))}(K,K)\right)$$} for some $r\geq 0$.  This means
that $\phi$ is an $r$-isomorphism.  We can now reformulate the result
in terms of modules and we deduce that $j^{\ast}\mathcal{Y}(N)$ is
$r$-isomorphic to a filtered module \zhnote{from}
$\fuk(\mathcal{X},\mathscr{P})^{\#}$\zhnote{, where}
$j:\fuk(\mathcal{X},\mathscr{P})\to \fuk(\mathcal{X}',\mathscr{P}')$
is the inclusion. \pbred{This concludes the proof.} \qed

\subsubsection{Proof of Theorem
  \ref{thm:appl-sympl}} \label{proof-thm:appl-sympl}
\pbred{Theorem~\ref{t:fil-fuk} implies that} \zhnote{for every two
  admissible perturbation data} $\mathscr{P}_{1}$ and
$\mathscr{P}_{2}$ there exists a functor
$$[\mathcal{F}^{1,2}]^{\ast}: \fuk(\mathcal{X},\mathscr{P}_{2})^{\nabla}
\to \fuk(\mathcal{X},\mathscr{P}_{1})^{\nabla}$$
that induces an equivalence of TPCs in homology \zhnote{and is the identity map on the objects from $\mathcal{X}$.} 
Given that \zhnote{$\mathcal{C}\fuk(\mathcal{X};\mathcal{P}_{i})=H^0(\fuk(\mathcal{X},\mathscr{P}_{i})^{\nabla})$}, this means that the 
pseudo-metric $D^{\mathcal{F}}$ defined on $\mathcal{X}$ using the TPC structure $\mathcal{C}\fuk(\mathcal{X};\mathscr{P}_{2})$ is greater or equal than the pseudo-metric defined using $\mathcal{C}\fuk(\mathcal{X};\mathscr{P}_{1})$. By using the functor $\mathcal{F}^{2,1}$ we deduce the opposite inequality and we conclude that $D^{\mathcal{F}}$ is independent of the 
perturbation data used to define it.

(i) The first inequality relates the spectral distance $\sigma(L,L')$ to the 
simplest fragmentation metric, $D(-,-)=D^{\{0\}}(-,-)$. The proof is based on a simple consequence of the properties of the Yoneda embedding.

\zhnote{Let $L, L' \in \mathcal{X}$. Without loss of generality, we may assume that $\bar{L}$ and $\bar{L'}$ are Hamiltonian isotopic (otherwise $\sigma(L, L') = +\infty$). Recall also that all Lagrangians in $\mathcal{X}$ are assumed to be graded.} Standard Floer theory shows \zhnote{that there} is a canonical
class $a=o_{L,L'}\in HF(L,L')$ obtained as the image of the fundamental 
class  $[L]\in H_{n}(L,\k)$ through the PSS 
morphism $H_{\ast}(L,\k)\to HF(L,L')$. There is a similar class $b=o_{L',L}\in HF(L',L)$ and it is a simple consequence of the properties of the Yoneda embedding that these two classes have the property
that $a\ast b = [L]$ and $b\ast a= [L']$. Here $\ast$ is \zhnote{the product induced} in homology by the $A_{\infty}$-composition $\mu_{2}$,
and $[L]$ and $[L']$ are the respective fundamental classes.

We start with an argument in which we neglect filtration
issues. Suppose that there are two classes $a\in HF(L,L')$ and
$b\in HF(L',L)$ such that
$a\ast b = [e_{L,L}]=[L]\in HF(L,L)$\zhnote{, where} $e_{L,L}$ is the
unit in $CF(L,L)$. For coherence, we work here in homological notation
(even if to keep track of signs it would be preferable to use
\zhnote{cohomological} notation as in \cite{Seidel}), \zhnote{so}
$e_{L,L}$ corresponds to the maximum of the Morse function
$f_{\bar{L}}$ as in \S\ref{sbsb:fdata}. This means that there are
Floer cycles $\alpha\in CF(L,L')$, $\beta\in CF(L',L)$ such that
$\mu_{2}(\alpha,\beta)$ is homologous to $e_{L,L}$. However, recall
\zhnote{that $CF(L,L)$} is the Morse complex of the function
$f_{\bar{L}}$. Thus, for degree reasons we obtain
$\mu_{2}(\alpha,\beta)=e_{L,L}$. As explained in the construction of
$\fuk(\mathcal{X},\mathscr{P})$, $e_{L,L}$ is a strict unit.

\zhnote{For every $L_0, L_1 \in \mathcal{X}$ and any cycle
  $u\in CF(L_0,L_1)$, we have} a morphism of Yoneda modules
\zhnote{$\phi^{u}:\mathcal{Y}(L_0)\to \mathcal{Y}(L_1)$} defined
by $$\phi^{u}_{k}(-,-,\ldots, x)=\mu_{k+1}(-,-,\ldots,x,u)~.~$$
Returning to $\alpha$ and $\beta$ above we have that
$\phi^{\mu_{2}(\alpha,\beta)}=
\phi^{e_{L,L}}=\mathds{1}_{\mathcal{Y}(L)}$.  The composition
$\phi^{\beta}\circ \phi^{\alpha}$ is not \zhnote{necessarily} equal to
the morphism $\phi^{\mu_{2}(\alpha,\beta)}$, but they are
\zhnote{homologous as} elements of
$\Mor_{mod}(\mathcal{Y}(L), \mathcal{Y}(L))$. \zhnote{(See}
\cite{Seidel} for the explicit formulas for \zhnote{the composition of
  pre-module homomorphisms and for the differential $\mu_{1}^{mod}$ on
  $\Mor_{mod}(-,-)$.)}  We rewrite here this differential (neglecting
signs) for a pre-module morphism $t:\mathcal{M}\to\mathcal{N}$, with
components
$$t_{k}: CF(L_{1},L_{2})\otimes \ldots \otimes
CF(L_{k-1},L_{k})\otimes \mathcal{M}(L_{k})\to \mathcal{N}(L_{1})~.~$$
We have that:
\pbrev{
\begin{equation}
  \begin{aligned}
    (\mu_{1}^{mod}t)_{m}(-, \ldots, -,x) = & \sum_{r+s=m+1}
    \mu_{r}^{\mathcal{N}}(- ,-, \ldots, t_{s}(-, -,\ldots,x)) + \\ + &
    \sum_{l+j=m+1} t_{l}( -, -, \ldots,
    \mu_{j}^{\mathcal{M}}(-,-,\ldots, x)) + \\ & + \sum_{k+g=m+1}
    t_{k}(-,\ldots, \mu_{g}(-,\ldots, -), -,\ldots, x),
  \end{aligned}
\end{equation}
}
where $\mu_{g}$ is the operation in the $A_{\infty}$-category, and
\pbred{$\mu_r^{\mathcal{M}}$ and $\mu_j^{\mathcal{N}}$} are the
respective module operations. Consider the pre-module morphism
$T:\mathcal{Y}(L)\to \mathcal{Y}(L)$ with the $k$-th component given
by:
$$T_{k}(-, -,\dots , -, x)=\mu_{k+2}(-,-, \ldots, -, x, \alpha,\beta)~.~$$
Using the fact that $\mu_{2}(\alpha,\beta)$ is a strict unit and the
$A_{\infty}$-relations, it is easy to see that
$\mu_{1}^{mod}(T)=\mathds{1}_{\mathcal{Y}(L)}-\phi^{\beta}\circ
\phi^{\alpha}$.

In our case, the $A_{\infty}$-category $\fuk(\mathcal{X},\mathscr{P})$
is filtered and, assuming that \zhnote{$\alpha \in CF^{\leq r}(L,L')$}
and \zhnote{$\beta \in CF^{\leq s}(L,L')$}, our discussion above shows
that in the TPC $\mathcal{C}\fuk(\mathcal{X})$ we have \zhnote{maps:}

$$\Sigma^{r+s}\mathcal{Y}(L)
\stackrel{\Sigma^{r}\bar{\phi}^{\alpha}}{\longrightarrow}
\Sigma^{s}\mathcal{Y}(L')\stackrel{\bar{\phi}^{\beta}}{\longrightarrow}
\mathcal{Y}(L)$$ that compose to the ``shift'' map $\eta_{r+s}$.  Here
$\bar{\phi}^{\beta}:\Sigma^{s}\mathcal{Y}(L')\to \mathcal{Y}(L)$ is
induced by $\phi^{\beta}$ and
$\bar{\phi}^{\alpha}:\Sigma^{r}\mathcal{Y}(L)\to \mathcal{Y}(L')$ is
induced by $\phi^{\alpha}$.  The key ingredient here is that the
homotopy $T$ only shifts filtration by at most $r+s$.

Returning to our classes $a=[\alpha]$ and $b=[\beta]$ we now use
$b\ast a = [e_{L',L'}]$ and obtain the existence of a similar diagram:
$$\Sigma^{r+s}\mathcal{Y}(L')
\stackrel{\Sigma^{s}\bar{\phi}^{\beta}}{\longrightarrow}
\Sigma^{r}\mathcal{Y}(L)\stackrel{\bar{\phi}^{\alpha}}{\longrightarrow}
\mathcal{Y}(L')~.~$$

We immediately deduce from this that $\mathcal{Y}(L)$ and
$\mathcal{Y}(L')$ are \zhnote{$(r+s)$-interleaved} in the sense of
Definition \ref{def:interleave} in the TPC
$\mathcal{C}\fuk(\mathcal{X})$ (one can do better, for instance use
$\max\{r,s\}$, but we will \zhnote{not try} to improve the estimates).
This means by Corollary \ref{cor:int-metr} that $D(L,L')\leq 4(r+s)$.
By reviewing the definition of spectral invariants (\ref{eq:spec})
this implies that
$D(L,L')\leq 4( \sigma( o_{L,L'})+\sigma(o_{L',L}))$.  The inequality
claimed at the first point of the Theorem now follows because
$HF(L,L')$ and $HF(L',L)$ are related by duality and
$\sigma(o_{L',L})=-\sigma(pt_{L,L'})$.

\begin{rem}\label{rem:spec-Y}
  (a) The argument at the center of this proof has appeared before in
  the literature \zhnote{(in less generality, involving only $\mu_{2}$
    and $\mu_{3}$ operations, instead of the full
    $A_{\infty}$-structure)} first in \cite{AS-spectral-21}, as well
  as in \cite{Bi-Co:spectr}. All these arguments are based on
  \zhnote{properties of the Yoneda embedding from \cite{Seidel}.}
  \zhnote{These were} adjusted to the weakly-filtered setting in
  \cite{Bi-Co-Sh:LagrSh}. We can achieve a greater degree of
  generality here because we are able to work with \zhnote{genuinely}
  filtered $A_{\infty}$-categories \zhnote{which moreover are strictly
    unital.}

  (b) The fact that the modules $\mathcal{Y}(L)$ and $\mathcal{Y}(L')$
  are interleaved \zhnote{by} an order controlled by $\sigma(L,L')$
  has numerous consequences because it implies that persistence
  modules of the form $HF(N,L)$ and $HF(N,L')$ are also interleaved
  \zhnote{by} the same order, for all $N\in \mathcal{X}$. Viewing this
  through the prism of the \zhnote{bottleneck} metric we easily deduce
  relations between various measurements of interest such as boundary
  depth, spectral range and so forth, associated to the complexes
  $CF(N,L)$ and $CF(N,L')$.
\end{rem}

We pursue with the proof of the \zhnote{point (ii)} in Theorem
\ref{thm:appl-sympl}.  \zhnote{Assume} that
$D^{\mathcal{F}}(L,L') <r$. \zhnote{Then there} exists a sequence of
exact triangles in $[\C\fuk(\mathcal{X})]_{\infty}$ as below:
\begin{equation}\label{eq:iterated-tr10}\xymatrixcolsep{1pc}
  \xymatrix{
    0 \ar[rr] &  &  Y_{1}\ar@{-->}[ldd]  \ar[r] &\ldots
    \ar[r]& Y_{i} \ar[rr] &  &  Y_{i+1}\ar@{-->}[ldd]  \ar[r]
    &\ldots \ar[r]&Y_{n-1} \ar[rr] &   &L \ar@{-->}[ldd]  &\\
    &         \Delta_{1}                  &  & & &  \Delta_{i+1}                          & &  &  &    \Delta_{n}             \\
    & F_{1}\ar[luu] &  & & &F_{i+1}\ar[luu] &  &  & &F_{n}\ar[luu] }
\end{equation}
Here all the terms $F_{i}$ are in $\mathcal{F}$ except for the $j$-th
term which is of the form \zhnote{$T^{-1}\Sigma^{\alpha}L'$ for some
  $\alpha \in \R$} (we write here $N$ instead of $\mathcal{Y}(N)$ for
the elements of $\mathcal{X}$). \zhnote{See (\ref{eq:delta-fr}).}
Moreover, \zhnote{by assumption we have}
$\sum_{i} \bar{w}(\Delta_{i})\leq r$.  We now \zhnote{appeal} to
\zhnote{the argument from} the proof of Corollary \ref{eq:estimate3}
to associate to the sequence of exact triangles above another sequence
of triangles

\begin{equation}\label{eq:ex-tr-C10}\xymatrixcolsep{1pc} \xymatrix{
    0 \ar[rr] &  &  \bar{Y}_{1}\ar@{-->}[ldd]
    \ar[r] &\ldots  \ar[r]& \bar{Y}_{i} \ar[rr] &  &
    \bar{Y}_{i+1}\ar@{-->}[ldd]  \ar[r] &\ldots \ar[r]&\bar{Y}_{n-1}
    \ar[rr] &   &\bar{Y}_{n} \ar@{-->}[ldd]  &\\
    &         \Delta_{1}'                  &  & & &  \Delta_{i+1}'                          & &  &  &    \Delta_{n}'             \\
    & F'_{1}\ar[luu] &  & & &F'_{i+1}\ar[luu] &  &  & &F'_{n}\ar[luu] }
\end{equation}
this time all exact in $[\C\fuk(\mathcal{X})]_{0}$ together with a
$2r$-isomorphism $ L\to \bar{Y}_{n}$.  Here the $F'_{i}$'s are shifts
of the previous corresponding \zhnote{objects from
  (\ref{eq:iterated-tr10}). Given} that $\mathcal{F}$ is closed
\zhnote{under} shifts we have $F'_{i}\in\mathcal{F}$ except for the
$j$-th term that is still a shift of $T^{-1}L'$. By possibly shifting
$\bar{Y}_{n}$ down, we have a $4r$-isomorphism
$\phi :\bar{Y}_{n}\to L$.

Each exact triangle $\Delta_{i}$ is the image of a cone-type triangle
in the category of filtered modules:
\pbrev{$$F'_{i+1} \stackrel{\varphi_{i}}{\longrightarrow} \bar{Y}_{i}
  \to {\rm Cone\/}(\varphi_{i})= \bar{Y}_{i+1}$$} and there is a
module map $\bar{\phi}:\bar{Y}_{n}\to L$ such that
$K={\rm Cone\/}(\bar{\phi})$ is $4r$-acyclic. In particular, $K(L)$ is
\zhnote{a} $4r$-acyclic \zhnote{filtered chain complex}: the identity
of \zhnote{this chain} complex is homotopic to zero through a homotopy
$h$ of shift at most $4r$. The complex $K(L)$, viewed as a vector
space, is the \pbrev{following sum:} \pbrev{
  \begin{equation} \label{eq:KL}
    \begin{aligned}
      K(L) & = CF(L,L) \oplus \bigoplus_{i=1}^n CF(L,F'_i)[2] \\
      & = CF(L,L) \oplus \bigoplus_{i=1}^{j-1} CF(L,F'_i)[2] \oplus
      CF(L,L')[1] \oplus \bigoplus_{i=j+1}^n CF(L,F'_i)[2].
    \end{aligned}
\end{equation}
}
%
Denote the maximum of the Morse function $f_{\bar{L}}:\bar{L}\to \R$
by $m$.  We now make the following two assumptions:
\begin{itemize}
\item[(i)] \zhnote{There} exists an embedding of a standard symplectic
  ball $e':B(s)\to X$ such that $e'(0)=m$, $(e')^{-1}(L)=\R B(s)$ and
  $e'(B(s))\cap (L'\cup \cup_{F\in\mathcal{F}}F)=\emptyset$.
\item[(ii)] \zhnote{The} almost complex structures on $X$ that are
  part of the perturbation data $\mathscr{P}$, pull back to the
  standard almost complex \zjnote{structure} on $B(s)$.
\end{itemize} 
These assumptions are not restrictive because, as shown at the
beginning of the proof, the pseudo-metric \zhnote{$D^{\mathcal F}$} is
independent of the perturbation data used in its definition. Thus we
can pick perturbation data $\mathscr{P}$ adapted to the embedding $e'$
in the sense that the two \zhnote{points (i), (ii)} above are
satisfied, and the sequences of exact triangles
(\ref{eq:iterated-tr10}) and (\ref{eq:ex-tr-C10}) still exist in this
case, with the properties claimed.

Returning to the complex $K(L)$, denote \zhnote{by} $d_{K}$ its
differential and consider the equation:
$$d_{K} h (m) + h d_{K} (m)=m~.~$$
Now $CF(L,L)$ is a subcomplex \zhnote{of} $K(L)$ and thus
$d_{K}(m)=0$. As a result $d_{K}h(m)=m$, \pbrev{hence $m$ is a boundary in
the complex $(K(L), d_K)$.}

\pbrev{On the other hand, since $K$ is an iterated cone of
  $A_{\infty}$-Yoneda modules, the differential $d_K$ has a particular
  shape with respect to the splitting from~\eqref{eq:KL}. This has
  been worked out in detail in~\cite[Section~2.6]{Bi-Co-Sh:LagrSh} (in
  particular, see Theorem~2.14 in that paper) and the relevant
  ingredients are as follows. The differential $d_K$ can be described
  by a matrix $(a_{i,j})_{0 \leq i, j \leq n}$, where:} \pbrev{
\begin{enumerate}
\item $a_{i,j} = 0$ for $i>j$ (i.e.~the matrix $(a_{i,j})$ is upper
  triangular).
\item $a_{i,j}: CF(L, F'_j) \longrightarrow CF(L, F'_i)$ for
  $1 \leq i \leq j \leq n$.
\item $a_{0,j}: CF(L, F'_j) \longrightarrow CF(L,L)$ for $j \geq 1$.
\item $a_{0,0}: CF(L, L) \longrightarrow CF(L, L)$ is the Floer
  differential on $CF(L,L)$.
\end{enumerate}
Here we have omitted reference to the grading on the $F_j'$s and $L$.
Moreover, for $j \geq 1$, the maps $a_{0,j}$ can be written as
follows.} \pbrrvv{There exist} \pbrev{Floer chains
$c_{q,p} \in CF(F'_q, F'_p)$ for every $q>p>0$ and
$c_{q,0} \in CF(F'_q,L)$ for all $q \geq 1$, all at action levels
$\leq 0$, such that:} \pbrev{
\begin{equation} \label{eq:aij} a_{0,j}(-) = \sum_{2 \leq d,\,
    \underline{k}} \mu_{d}(-, c_{k_{d}, k_{d-1}}, \ldots, c_{k_2,0}),
\end{equation}
where $\underline{k}=(k_2, \ldots, k_d)$ runs over all partitions
$0 < k_2 < \cdots < k_{d-1} < k_d=j$ and $\mu_d$ is the $d$-order
operation in the Fukaya category $\fuk(\mathcal{X})$.}

\pbrev{Since $m$ is the maximum of the Morse function, and $m$ is a
  boundary in the complex $(K(L), d_K)$, it follows
  from~\eqref{eq:aij} that there is a $J$-holomorphic polygon with one
  edge on $L$ and the others on the $F'_{i}$'s and possibly $L'$
  (recall that $F'_j = T^{-1}L'$), that goes through
  $m$. (See~\cite[Section 5.1]{Bi-Co-Sh:LagrSh} and in particular
  pages~91-92 in that paper for a detailed proof of a very similar
  statement.)} The area of each such polygon is at least
$\pi s^{2}/2$.  This means that \pbrev{the chain homotopy} $h$
increases filtration by at least $\pi s^{2}/2$ and thus
$$4r \geq \frac{\pi s^{2}}{2}$$ which shows the claim.
  
\
 
We pursue with the proof of the \zhnote{point (iii)} of Theorem
\ref{thm:appl-sympl}.  We will again make use of the fact that the
pseudo-metric $D^{\mathcal{F}}$ is independent of the perturbation
data $\mathscr{P}$.
 
Let \zhnote{$\delta=\delta^{\cap}(N, L';\mathcal{F})$}.  For each
intersection point $x\in N \cap L'$ fix standard ball embeddings
$e_{x}: B(s)\to X$ with $e^{-1}(N)=\R B(s)$, $e^{-1}(L')=i\R B(s)$,
$e(0)=x$, such that all these embeddings are disjoint from the family
$\mathcal{F}$ and additionally $\pi s^{2} =\delta-\epsilon$ for
\zhnote{a small} $\epsilon$. We may assume that the almost complex
structures that are part of the perturbation data $\mathscr{P}$
pullback to the standard almost complex structure through the
embeddings $e_{x}$. This implies that if a Floer type strip, or
polygon, has an input or an output at a point $N\cap L'$ and has
boundary on $N$, $L'$ and any other elements of the family
$\mathcal{F}$, then its energy is at least
$\delta'=\pi s^{2}/4
-\epsilon'=\frac{\delta-\epsilon}{4}-\epsilon'$. Here the small
$\epsilon'$ has to do with making the Hamiltonian \zhnote{or
  (1-forms)} part of the perturbation data small enough.

Assume now, as in the statement, that
\zjr{$D^{\mathcal{F}}(L,L') < r < \delta /16$}.  Then there exists a sequence
of \zhnote{exact triangles} in $\C\fuk(\mathcal{X})_{0}$ as in
(\ref{eq:ex-tr-C10}) and a $2r$-isomorphism $\psi: L\to
\bar{Y}_{n}$. By Lemma \ref{lem:ineq-inter} this means that
\zjr{$d_{int}(L, Y_{n})\leq 2r$}. In particular, there exist maps
$u:\Sigma^{r}\bar{Y}_{n}\to L$, $v:\Sigma^{r}L\to \bar{Y}_{n}$ such
that $v\circ \Sigma^{r}u=\eta_{2r}$.  Therefore, given that
$N\in \mathcal{X}$, we obtain maps of filtered complexes
$$ \bar{Y}_{n}(N)\stackrel{u'}{\longrightarrow}
CF(N,L)\stackrel{v'}{\longrightarrow} \bar{Y}_{n}(N)$$ each of shift
at most \zjr{$2r$}, and whose composition is chain homotopic to the identity
through a homotopy $h$ that \zhnote{shifts} filtration by at most
\zjr{$4r$}.
  
\zhnote{Consider} the differential of the complex $\bar{Y}_{n}(N)$.
As a vector space, the complex $\bar{Y}_{n}(N)$ is a sum of the form
$CF(N,F'_{1})\oplus \ldots \oplus CF(N,L')\oplus CF(N,F'_{n})$ and the
differential is represented by clustered polygons with boundaries on
$N$, $L'$ and the elements of the family $\mathcal{F}$.
 
\zhnote{Consider} the composition $\Psi=p\circ v'\circ u'\circ i$,
where $i:CF(N,L') \to \bar{Y}_{n}(N)$ is the inclusion and
$p:\bar{Y}_{n}(N)\to CF(N,L')$ is the projection - in both cases as
vector spaces. We \zhnote{claim that $\Psi$ is injective. This would
  imply} that $\dim_{\k} CF(N,L')\leq \dim_{\k} CF(N,L)$ and
\zhnote{prove the statement at point (iii) of the Theorem
  \ref{thm:appl-sympl}.}

\zhnote{To show the injectivity of $\Psi$, we} inspect the formula
$dh+hd = \mathds{1} -v'\circ u'$ and recall that the differential $d$, when
restricted to $CF(N,L')$, drops the filtration by at least
$\delta'=\frac{\delta-\epsilon}{4}-\epsilon'$ while $h$ raises
filtration by at most \ocnote{$4r < \delta'$} (when $\epsilon$ and $\epsilon'$
are small enough). In other words $p(hd(x))$ for each element
$x\in CF(N,L')$ is of strictly lower filtration than $x$.  Similarly,
$p(dh(x))$ is also of lower filtration than $x$.  As a result we
deduce that $\Psi$ can be \zhnote{written as $\mathds{1}\,+$ a map that
  strictly lowers the filtration level. It follows that $\Psi$ is
  injective.}
 
\
  
Finally, we discuss the last point in Theorem \ref{thm:appl-sympl}. We
\zhnote{assume} that $\mathcal{F}$ generates the usual
\zhnote{(i.e.~without persistence structure) derived} Fukaya category
$D\fuk(X)$ and we want to show that in this case the pseudo-metric
$D^{\mathcal{F}}$ is finite.  From Theorem \ref{cor:Fuk-cat-var}
\zhnote{(iii)} we deduce that $\mathcal{F}$ generates
$\C\fuk(\mathcal{X})_{\infty}$.  This implies that any object $A$ in
$\C\fuk(\mathcal{X})$ is $r$-isomorphic for some some $r$ to an object
that can be written as an \zhnote{iterated} cone with triangles of
weight $0$, of the form in (\ref{eq:ex-tr-C10}). This means that
$D^{\mathcal{F}}(A,0)<\infty$ and implies the last claim in the
statement of the theorem.  \qed


\subsection{Pseudo-metrics on $\mathcal{L}ag(X)$ and proof of
  Corollary \ref{cor:gl-metric}}

The purpose of this subsection is to use the results from
\S\ref{subsec-symp} to construct the family of fragmentation metrics
on the space $\mathcal{L}ag(X)$ of all \zhnote{closed exact, graded,
  Lagrangians} in $(X,\omega)$ and prove Corollary
\ref{cor:gl-metric}.

\

We assume the setting in Corollary~\ref{cor:gl-metric}, \pbred{in
  particular that $\rank \, \fuk(X,\omega)<\infty$}.  Thus $D\fuk(X)$
admits a finite set of triangular generators in the sense that there
is a family of triangular generators $\mathcal{F}$ such that the
corresponding family $\bar{\mathcal{F}}\subset \mathcal{L}ag(X)$,
obtained by forgetting the grading and the choices of primitives, is
finite. We now fix such a family \zhnote{$\mathcal F$} of generators
and assume that $\mathcal{F}$ is invariant to shifts and translation
and that the Lagrangians in $\bar{\mathcal{F}}$ are in
\zhnote{general} position (each two Lagrangians intersect
\zhnote{transversely} and there are no triple intersection points).

The proof of Corollary  \ref{cor:gl-metric} is a consequence of Theorem \ref{thm:appl-sympl} together with the invariance \zhnote{properties} of  the Fukaya TPC constructed earlier in this section and is contained in the sub-sections below.

\subsubsection{The Fukaya TPC revisited}\label{subsubsec:fuk-revTPC}
Pick  $\mathcal{X}\subset \mathcal{L}ag(X)'$ a shift and translation invariant family \zhnote{that contains $\mathcal{F}$}. As before, we denote by $\bar{\mathcal{X}}\subset \mathcal{L}ag(X)$ the corresponding 
family of Lagrangians after forgetting the choices of primitives and grading.
We assume that $\bar{\mathcal{X}}$ is finite and that its elements are in \zhnote{general} position. 

We recall the filtered Fukaya category 
$\fuk(\mathcal{X};\mathscr{P})$, where $\mathscr{P}$ is a choice of perturbation data, as in Theorem \ref{t:fil-fuk} - see also \S\ref{subsubsec:proof-Cor}. We already know that any two such categories, defined \zhnote{using two admissible} perturbation data $\mathscr{P}$ and $\mathscr{P}'$
are filtered quasi-equivalent in the sense that there are filtered $A_{\infty}$-functors $\fuk(\mathcal{X};\mathscr{P})\to \fuk(\mathcal{X};\mathscr{P}')$ that are the identity on objects and induce a (filtered) equivalence of the  homological persistence categories. 

As discussed in \S\ref{sb:faitpc}, there are two TPCs that one can associate to  \zhnote{$\fuk(\mathcal{X};\mathscr{P})$}. The first is  $\mathcal{C}\fuk(\mathcal{X})$ as in 
Corollary \ref{cor:Fuk-cat-var}. This is obtained by considering the filtered
modules $Fmod(\fuk(\mathcal{X};\mathscr{P}))$ over $\fuk(\mathcal{X};\mathcal{P})$. The category $\mathcal{C}\fuk(\mathcal{X})$ is the \zhnote{homological} category $$\mathcal{C}\fuk(\mathcal{X})= H^{0}[\fuk(\mathcal{X};\mathscr{P})^{\nabla}]$$ where $\fuk(\mathcal{X};\mathscr{P})^{\nabla}$ is the \zhnote{smallest} triangulated \zhnote{(with respect to weight-$0$ triangles) full} subcategory of $Fmod(\fuk(\mathcal{X};\mathscr{P}))$ that contains the Yoneda modules $\mathcal{Y}(L)$ with $L\in \mathcal{X}$ and is closed \zhnote{under $r$-isomorphisms} for all $r$ in the sense that if $j:M\to M'$ is an
$r$-isomorphism of modules, and $M\in \fuk(\mathcal{X};\mathscr{P})^{\nabla}$, then $M'\in   \fuk(\mathcal{X};\mathscr{P})^{\nabla}$. Possibly more concretely, each object in this category
is $r$-isomorphic, for some $r$, to a weight $0$ iterated cone of Yoneda modules.
Two such triangulated persistence categories, defined for different choices
of perturbation data are \zhnote{TPC-equivalent} (see also Remark \ref{rem:eq-nat-not}) and thus we drop
the reference to the perturbation data from the notation.

The second type of TPC will be denoted by $\mathcal{C}'\fuk(\mathcal{X})$  and is defined by
\begin{equation} \label{second-type}
\mathcal{C}'\fuk(\mathcal{X})=H^{0}[Tw(\fuk(\mathcal{X};\mathscr{P}))]
\end{equation}
where $Tw(\fuk(\mathcal{X};\mathscr{P}))$ is the category of filtered twisted complexes constructed from $\fuk(\mathcal{X};\mathscr{P})$, see  
 \S\ref{sb:faitpc}. In this case \zhnote{too} we can drop the reference to the choices of 
 perturbation data as any \zhnote{two such} choices produce equivalent TPCs (again these equivalences are not entirely canonical but this will not have any impact on our 
 further arguments).

There are filtered functors:
$$\Theta: Tw(\fuk(\mathcal{X};\mathscr{P}))\longrightarrow Fmod[Tw(\fuk(\mathcal{X};\mathscr{P}))] \longrightarrow Fmod(\fuk(\mathcal{X};\mathscr{P})) $$
where the first arrow is the Yoneda embedding and the second is pullback over 
the natural inclusion $\fuk(\mathcal{X};\mathscr{P})\to Tw(\fuk(\mathcal{X};\mathscr{P}))$. The composition $\Theta$ is a homologically full and faithful embedding, and it induces a full and faithful embedding of TPCs. The image
of $\Theta$ lands inside $\fuk(\mathcal{X})^{\nabla}$ (actually inside the 
category denoted by $\fuk(\mathcal{X},\mathscr{P})^{\#}$ in \S\ref{sb:faitpc}) and thus we have 
an inclusion of TPCs:
$$\bar{\Theta}: \mathcal{C}'\fuk(\mathcal{X})\hookrightarrow \mathcal{C}\fuk(\mathcal{X})$$
By \zhnote{contrast} to the unfiltered case, $\bar{\Theta}$ is not an equivalence of TPCs because,  to have an equivalence of TPCs  each object in $\mathcal{C}\fuk(\mathcal{X})$ needs to be $0$-isomorphic to some object in $\mathcal{C}'\fuk(\mathcal{X})$ \zhnote{(see Definition \ref{def:TPC equivalences})} and, \zhnote{{\em a priori}},  this might not happen
(each object in $\mathcal{C}\fuk(\mathcal{X})$ is $r$-isomorphic to some object in $\mathcal{C}'\fuk(\mathcal{X})$ but possibly for $r>0$).

\subsubsection{Fragmentation metrics as in Theorem \ref{thm:appl-sympl}}
We now focus on the pseudo-metric $D^{\mathcal{F}}$ as in Theorem  \ref{thm:appl-sympl}.
This is constructed by the general procedure described in the algebraic part of this paper, by using
the persistence triangular weight on  $\mathcal{C}\fuk(\mathcal{X})_{\infty}$. Thus, this is 
a shift invariant fragmentation pseudo-metric of the type $\widehat{\bar{d}^{\mathcal{F}}}(-,-)$, as 
described in \S\ref{subsubsec:prop-fr}, restricted to $\mathcal{X}$. To emphasize the relation of this pseudo-metric to the set $\mathcal{X}$ we will denote it by $D^{\mathcal{F}}_{\mathcal{X}}$. 
Because the pseudo-metric is shift invariant it descends to $\bar{\mathcal{X}}$.
As shown in Theorem \ref{thm:appl-sympl} this pseudo-metric is independent \zhnote{of} the choice of perturbation data $\mathscr{P}$.

There is a second possibility to construct a fragmentation type pseudo metric on $\bar{\mathcal{X}}$ and this is to apply the exact same construction to the triangulated persistence category $\mathcal{C}'\fuk(\mathcal{X})$ \zhnote{from (\ref{second-type})}. The resulting pseudo-metric, defined again on $\mathcal{X}$, will be denoted by $\bar{D}_{\mathcal{X}}^{\mathcal{F}}$.  It is again independent of the perturbation \zhnote{data} used in its definition. If one is only interested in comparing objects in $\mathcal{X}$,
this pseudo-metric is much more approachable from a computational point of view because 
$\mathcal{C}'\fuk(\mathcal{X})$ has fewer objects than  $\mathcal{C}\fuk(\mathcal{X})$.

There is a simple relation between the two pseudo-metrics discussed above.

\begin{lem}\label{lem:comp-pseudoLag} For any  $L,L'\in \bar{\mathcal{X}}$  the pseudo-metrics $\bar{D}_{\mathcal{X}}^{\mathcal{F}}$ and $D_{\mathcal{X}}^{\mathcal{F}}$ satisfy the inequalities:
$$ \frac{1}{4} \ \bar{D}_{\mathcal{X}}^{\mathcal{F}}(L,L')\ \leq \  D_{\mathcal{X}}^{\mathcal{F}}(L,L') \  \leq \ \bar{D}_{\mathcal{X}}^{\mathcal{F}}(L,L') $$
\end{lem}

\begin{proof}
  Indeed, the inequality on the right is obvious in view of the TPC
  inclusion $\bar{\Theta}$.  The inequality on the left follows from
  the argument in Lemma \ref{lem:metric-D}. Indeed this argument shows
  that if we have a cone decomposition of $\mathcal{Y}(L)$ of total
  weight $r$ in $\mathcal{C}\fuk(\mathcal{X})$, with linearization
  consisting of Yoneda modules of elements in $\mathcal{F}$ together
  with one term which is (a shift/translate) of $\mathcal{Y}(L')$,
  then there is another cone-decomposition of $\mathcal{Y}(L)$ with a
  linearization of the same type, and such that all triangles are of
  weight $0$ except the last triangle which is of weight $\leq
  4r$. But this means that all these triangles can be assumed to
  belong to $\mathcal{C}'\fuk(\mathcal{X})$. This is clear for the
  weight $0$ triangles because the image of $\Theta$ contains the
  Yoneda modules, $\mathcal{Y}(L)$, $L\in \mathcal{X}$ and is closed
  with respect to taking cones over filtration preserving maps. The
  last triangle corresponds to a strict exact triangle of weight $4r$
  with the $4r$-isomorphism on the third term of the form
  $\phi: M_{n-1}\to \mathcal{Y}(L)$ where $M_{n-1}$ is an object of
  $\mathcal{C}'\fuk(\mathcal{X})$.  Because $\bar{\Theta}$ is full and
  faithful we have $\phi \in \mor_{\mathcal{C}'\fuk(\mathcal{X})}$.
  From this argument it is easy to deduce the inequality on the left
  in the statement of the Lemma which concludes the proof.
\end{proof}

\subsubsection{Changing the set $\mathcal{X}$}
Assume that $\mathcal{X}'$ is another family of marked Lagrangians,
\zhnote{which is} shift and translation invariant, and such that
$\mathcal{X}\subset \mathcal{X}'$.  Additionally, we assume that the
family $\bar{\mathcal{X}}'$ is in general position.

\begin{lem}\label{lem:mono} Under the assumptions above, and for any
  two $L,L'\in \bar{\mathcal{X}}$ we have
  $$ \bar{D}^{\mathcal{F}}_{\mathcal{X}}(L,L')\geq
  \bar{D}^{\mathcal{F}}_{\mathcal{X}'}(L,L') \ , \
  D^{\mathcal{F}}_{\mathcal{X}}(L,L')\leq
  D^{\mathcal{F}}_{\mathcal{X}'}(L,L') ~.~$$
\end{lem}

\begin{proof} We consider the category
  $\fuk(\mathcal{X}',\mathscr{P}')$ and we notice that the restriction
  of $\mathscr{P}'$ to $\mathcal{X}$ provides an allowable choice of
  perturbation data - see Theorem \ref{t:fil-fuk}. Thus, there is an
  $A_{\infty}$ filtered inclusion:
  $$\fuk(\mathcal{X},\mathscr{P}') \longrightarrow
  \fuk(\mathcal{X}',\mathscr{P})~.~$$ This inclusion induces a
  pull-back TPC functor:
  $$qJ^{\ast}: \mathcal{C}\fuk(\mathcal{X}')
  \to \mathcal{C}\fuk(\mathcal{X})~.~$$ This is well-defined because
  $\mathcal{F}\subset \mathcal{X}\subset\mathcal{X}'$ is a system of
  triangular generators for $D\fuk(X)$ which means that, in
  particular, any Yoneda module $\mathcal{Y}(L)$, $L\in \mathcal{X}'$,
  is $r$-isomorphic (in $\mathcal{C}\fuk(\mathcal{X}')$) for some
  $r\geq 0$ to a $0$-weight iterated cone of Yoneda modules of
  elements from $\mathcal{F}$. This means that the pull-back of
  $\mathcal{Y}(L)$ to $Fmod(\fuk(\mathcal{X},\mathscr{P}'))$ is an
  object of $\mathcal{C}\fuk(\mathcal{X})$.

  The same $A_{\infty}$ inclusion also induces a push-forward TPC
  functor:
  $$ J_{\ast} : \mathcal{C}'\fuk(\mathcal{X}) \to
  \mathcal{C}'\fuk(\mathcal{X}')$$ which is induced by the natural
  inclusion of twisted complexes.
 
  As discussed before, our invariance statements imply that the
  pseudo-metrics $\bar{D}^{\mathcal{F}}_{\mathcal{X}}$ and
  $D^{\mathcal{F}}_{\mathcal{X}}$ do not depend on the choice of
  perturbation data.  As a result, the fact that $J_{\ast}$ is a TPC
  functor implies the first inequality in the Lemma and the fact that
  $J^{\ast}$ is a TPC functor implies the second inequality.
\end{proof}

\subsubsection{ The pseudo-metric $\mathcal{D}^{\mathcal{F}}$ from
  Corollary \ref{cor:gl-metric}}\label{subsubsec:tr-ineq-gl}
The construction of $\mathcal{D}^{\mathcal{F}}$ proceeds in two steps.

The first is to consider again a family $\mathcal{X}$ as in the
subsections above as well as two elements
\zhnote{$L,L'\in \mathcal{L}ag(X)$}. We define:

\begin{equation}\label{metr:gen1}
  \mathcal{D}^{\mathcal{F}}_{\mathcal{X}}(L,L')=\limsup_{\epsilon
    \to 0} D^{\mathcal{F}}_{\mathcal{X}\cup
    \{L_{\epsilon},L'_{\epsilon}\}}(L_{\epsilon},L'_{\epsilon})
\end{equation}
Here $L_{\epsilon}$ and $L'_{\epsilon}$ are Hamiltonian deformations
of, respectively, $L$ and $L'$ through Hamiltonians of Hofer norm at
most equal to $\epsilon \geq 0$, such that the family
$\mathcal{X}\cup\{L_{\epsilon}, L'_{\epsilon}\}$ is allowable for the
definition of the Fukaya categories
$\fuk(\mathcal{X}\cup\{L_{\epsilon},L'_{\epsilon}\}; \mathscr{P})$.
The $\limsup$ is taken over all possible choices of such perturbations
and making $\epsilon$ go to $0$.

The second step is to put:
$$\mathcal{D}^{\mathcal{F}}(L,L')=\sup_{\mathcal{X}}
\mathcal{D}_{\mathcal{X}}^{\mathcal{F}}(L,L')$$ It is clear the
$\mathcal{D}^{\mathcal{F}}$ is symmetric. \zhnote{We will see blow, in
  Lemma \ref{lem-finite-df} that $D^{\mathcal F}$ is finite. However,
  before we get to that, we have: }

\begin{lem}
  With the definition above $\mathcal{D}^{\mathcal{F}}$ satisfies the
  triangle inequality.
\end{lem}

\begin{proof} We fix three Lagrangians in $\lag ag(X)$, $L,L', L''$.
  Fix also a family $\mathcal{X}$ as above. Consider $L_{n}, L''_{n}$
  in $\lag ag(X)$ such that the family
  $\mathcal{X}\cup \cup_{n,m}\{L_{n}, L''_{m}\}$ is in
  \zhnote{general} position in our usual sense: any couple of
  Lagrangian in the \zhnote{family} intersect transversely and there
  are no triple intersection points (this choice is possible as there
  are only countable many transversality type constraints). We also
  assume:
\begin{itemize}
\item[-] $d_{H}(L, L_{n})\leq \frac{1}{n}$ and
  $d_{H}(L'', L''_{n})\leq \frac{1}{n}$ where $d_{H}(-,-)$ is the
  Hofer distance.
\item[-]
  $\lim_{n\to \infty} D^{\mathcal{F}}_{\mathcal{X}\cup \{L_{n},
    L''_{n}\}}(L_{n},L''_{n})=\mathcal{D}^{\mathcal{F}}_{\mathcal{X}}(L,L'')~.~$
\end{itemize}

\zhnote{The lemma would follow if we prove that for every $\delta > 0$}
\begin{equation}\label{eq:tr-ineq1}
  \mathcal{D}^{\mathcal{F}}_{\mathcal{X}}(L,L'')
  \leq \mathcal{D}^{\mathcal{F}}(L,L') +
  \mathcal{D}^{\mathcal{F}}(L', L'') + 4\delta.
\end{equation}
To show (\ref{eq:tr-ineq1}) we pick a sequence of Lagrangians $L'_{k}$
such that the family
$\mathcal{X}\cup\cup_{n, k,m}\{L_{n}, L'_{k}, L''_{m}\}$ is in
\zhnote{general} position and $d_{H}(L',L'_{n})\leq \frac{1}{n}$. For
any $m,n,k$ we have the inequalities:
$$D^{\mathcal{F}}_{\mathcal{X}\cup \{L_{n},L''_{m}\}}(L_{n}, L''_{m})\leq 
D^{\mathcal{F}}_{\mathcal{X}\cup \{L_{n}, L'_{k},L''_{m}\}}(L_{n},
L''_{m})\leq D^{\mathcal{F}}_{\mathcal{X}\cup \{L_{n},L'_{k},
  L''_{m}\}}(L_{n}, L'_{k})+D^{\mathcal{F}}_{\mathcal{X}\cup \{L_{n},
  L'_{k}, L''_{m}\}}(L'_{k}, L''_{m})$$ The first inequality comes
from Lemma \ref{lem:mono} and the second is the triangle inequality
for the fragmentation pseudo-metric $D^{\mathcal{F}}_{-}$.  We will
estimate separately the two terms on the \zhnote{right-hand side of}
this inequality.

Fix a natural number $m_{0}$.  We can find $N_{m_{0}}\geq m_{0}$
\zhnote{such} that for $n, k\geq N_{m_{0}}$ we have
$$D^{\mathcal{F}}_{\mathcal{X}\cup
  \{ L_{n}, L'_{k}, L''_{m_{0}}\}}(L_{n},L'_{k})\leq
\mathcal{D}^{\mathcal{F}}_{\mathcal{X}\cup \{L''_{m_{0}}\}}(L, L')
+\delta\leq \mathcal{D}^{\mathcal{F}}(L,L')+\delta~.~$$ Thus, for
$n,k\geq N_{m_{0}}$ we have:
$$D^{\mathcal{F}}_{\mathcal{X}\cup \{L_{n},L''_{m_{0}}\}}(L_{n}, L''_{m_{0}})
\leq \mathcal{D}^{\mathcal{F}}(L,L')+\delta
+D^{\mathcal{F}}_{\mathcal{X}\cup \{L_{n}, L'_{k},
  L''_{m_{0}}\}}(L'_{k}, L''_{m_{0}}),$$ and \zhnote{it remains to}
estimate the rightmost term.  \zhnote{Using Lemma \ref{lem:mono}
  again} and the triangle inequality\zhnote{, we have:}
\begin{eqnarray*}
  D^{\mathcal{F}}_{\mathcal{X}\cup \{L_{n}, L'_{k},
  L''_{m_{0}}\}}(L'_{k}, L''_{m_{0}})\leq 
  D^{\mathcal{F}}_{\mathcal{X}\cup \{L_{n}, L'_{k},
  L''_{m_{0}}, L''_{m}\}}(L'_{k}, L''_{m_{0}}
  )\leq  \\
  \leq 
  D^{\mathcal{F}}_{\mathcal{X}\cup \{L_{n}, L'_{k}, L''_{m_{0}},
  L''_{m}\}}(L'_{k}, L''_{m}) +
  D^{\mathcal{F}}_{\mathcal{X}\cup \{L_{n}, L'_{k},
  L''_{m_{0}}, L''_{m}\}}(L''_{m}, L''_{m_{0}})~.~                                                                                          
\end{eqnarray*}
All our fragmentation pseudo-metrics are bounded from above by the
Hofer norm and thus we have
$D^{\mathcal{F}}_{\mathcal{X}\cup \{L_{n}, L'_{k}, L''_{m_{0}},
  L''_{m}\}}(L''_{m}, L''_{m_{0}})\leq \frac{2}{m_{0}}$ as soon as
$m\geq m_{0}$.  We now consider $n\geq N_{m_{0}}$ and we take $k, m$
sufficiently big such that we have:
$$D^{\mathcal{F}}_{\mathcal{X}\cup \{L_{n}, L'_{k},
  L''_{m_{0}}, L''_{m}\}}(L'_{k}, L''_{m})\leq
\mathcal{D}^{\mathcal{F}}_{\mathcal{X}\cup \{L_{n},
  L''_{m_{0}}\}}(L',L'')+\delta\leq
\mathcal{D}^{\mathcal{F}}(L',L'')+\delta~.~$$ Putting things together,
for our fixed (arbitrary) $m_{0}$ and any $n\geq N_{m_{0}}$ we have:
\begin{equation}\label{eq:tr0-in}
  D^{\mathcal{F}}_{\mathcal{X}\cup \{L_{n},L''_{m_{0}}\}}(L_{n}, L''_{m_{0}})
  \leq 
  \mathcal{D}^{\mathcal{F}}(L,L')+\delta +
  \mathcal{D}^{\mathcal{F}}(L',L'')+\delta + \frac{2}{m_{0}}~.~
\end{equation}
An important remark is that inequality (\ref{eq:tr0-in}) applies to
any fixed $m_{0}$ and any set $\mathcal{X}'$ that contains
$\mathcal{X}$ (and is such that the family
$\mathcal{X}'\cup \cup_{n,k,m}\{L_{n},L'_{k}, L''_{m}\}$ is in
\zhnote{general} position) as the argument above applies ad-literam to
this situation. In this case the number $N_{m_{0}}$ depends on $m_{0}$
but also, implicitly, on $\mathcal{X}'$. To make this dependence
explicit we will denote it by $N_{m_{0},\mathcal{X}'}$.

For any $n$, we have the triangle inequality:
$$D^{\mathcal{F}}_{\mathcal{X}\cup \{L_{m_{0}},
  L_{n}, L''_{m_{0}}\}}(L_{m_{0}}, L''_{m_{0}}) \leq
D^{\mathcal{F}}_{\mathcal{X}\cup \{L_{m_{0}}, L_{n},
  L''_{m_{0}}\}}(L_{n}, L''_{m_{0}}) +
D^{\mathcal{F}}_{\mathcal{X}\cup \{L_{m_{0}}, L_{n},
  L''_{m_{0}}\}}(L_{m_{0}}, L_{n })~.~$$ Assuming $n\geq m_{0}$, the
second term on the \zhnote{right-hand side} is bounded from above by
$\frac{2}{m_{0}}$ and thus, \zhnote{using Lemma} \ref{lem:mono}, we
deduce:
$$D^{\mathcal{F}}_{\mathcal{X}\cup \{L_{m_{0}}, L''_{m_{0}}\}}
(L_{m_{0}}, L''_{m_{0}})\leq D^{\mathcal{F}}_{\mathcal{X}\cup
  \{L_{m_{0}}, L_{n}, L''_{m_{0}}\}}(L_{m_{0}}, L''_{m_{0}})\leq
D^{\mathcal{F}}_{\mathcal{X}\cup \{L_{m_{0}}, L_{n},
  L''_{m_{0}}\}}(L_{n}, L''_{m_{0}})+ \frac{2}{m_{0}}~.~$$ We now
apply (\ref{eq:tr0-in}) to
$\mathcal{X}'=\mathcal{X}\cup \{L_{m_{0}}\}$.  We deduce \zhnote{that}
for $n\geq N_{m_{0}, \mathcal{X}'}$:
$$D^{\mathcal{F}}_{\mathcal{X}\cup \{L_{m_{0}}, L_{n}, L''_{m_{0}}\}}(L_{n}, L''_{m_{0}})
\leq \mathcal{D}^{\mathcal{F}}(L,L')+ \mathcal{D}^{\mathcal{F}}(L',L'')+2\delta + \frac{2}{m_{0}}~.~$$
\zhnote{The last inequality} implies:
$$D^{\mathcal{F}}_{\mathcal{X}\cup \{L_{m_{0}}, L''_{m_{0}}\}}(L_{m_{0}}, L''_{m_{0}})\leq
 \mathcal{D}^{\mathcal{F}}(L,L')+ \mathcal{D}^{\mathcal{F}}(L',L'')+2\delta +\frac{4}{m_{0}}~.~$$
As this is true for an arbitrary choice of $m_{0}$ we deduce  inequality (\ref{eq:tr-ineq1}) and this
concludes the proof.

\end{proof}
\subsubsection{Properties of $\mathcal{D}^{\mathcal{F}}$}
We know from \S\ref{subsubsec:tr-ineq-gl} that $\mathcal{D}^{\mathcal{F}}$
is a pseudo-metric. In this subsection we will show that $\mathcal{D}^{\mathcal{F}}$ satisfies the other properties claimed in Corollary \ref{cor:gl-metric}. The properties \zhnote{(i), (ii), and (iii)}, are in fact immediate consequences of 
the properties of the pseudo-metrics $D^{\mathcal{F}}_{\mathcal{X}}$ that 
appear in Theorem \ref{thm:appl-sympl}. Indeed, the estimates
in this theorem do not depend on the set $\mathcal{X}$, and this  easily implies the corresponding properties for $\mathcal{D}^{\mathcal{F}}$.
A more delicate property is \zhnote{the following one.}
\begin{lem} \label{lem-finite-df}
The pseudo-metric $\mathcal{D}^{\mathcal{F}}$ is finite. 
\end{lem}
\begin{proof}
This property follows by applying repeatedly   Lemmas \ref{lem:comp-pseudoLag} and \ref{lem:mono}. Fix  
$L,L'\in \lag ag (X)$, a family $\mathcal{X}$ as before, and perturbations
$L_{\epsilon}, L'_{\epsilon}$ such that $\mathcal{X}\cup\{L_{\epsilon}, L'_{\epsilon}\}$ is \zhnote{in general position} and $L_{\epsilon}$, $L'_{\epsilon}$ are
$\epsilon$-close to $L$ and $L'$ \zhnote{respectively,} in \zhnote{the} Hofer norm.
We have:
\begin{eqnarray*}
 \bar{D}_{\mathcal{F}\cup \{L_{\epsilon}, L'_{\epsilon}\}}^{\mathcal{F}}(L_{\epsilon}, L'_{\epsilon})
\geq \bar{D}^{\mathcal{F}}_{\mathcal{X}\cup \{L_{\epsilon}, L'_{\epsilon}\}}(L_{\epsilon},L'_{\epsilon})\geq 
D^{\mathcal{F}}_{\mathcal{X}\cup \{L_{\epsilon}, L'_{\epsilon}\}}(L_{\epsilon},L'_{\epsilon}) ~.~
\end{eqnarray*}


Thus the argument reduces to showing  that $\bar{D}_{\mathcal{F}\cup \{L_{\epsilon}, L'_{\epsilon}\}}^{\mathcal{F}}(L_{\epsilon}, L'_{\epsilon})$ has a uniform upper bound when $\epsilon \to 0$.
It is immediate to see that it is enough to find such a bound for 
$$\bar{D}^{\mathcal{F}}_{\mathcal{F}\cup\{L_{\epsilon}\}}(L_{\epsilon}, 0)~.~$$

We fix some $\epsilon_{0}$ and write for $\epsilon \leq \epsilon_{0}$:
\begin{eqnarray*}
\frac{1}{4}\bar{D}^{\mathcal{F}}_{\mathcal{F}\cup \{L_{\epsilon}\}}(L_{\epsilon}, 0)\leq
D^{\mathcal{F}}_{\mathcal{F}\cup\{L_{\epsilon}\}}(L_{\epsilon},0)\leq
D^{\mathcal{F}}_{\mathcal{F}\cup\{L_{\epsilon}, L_{\epsilon_{0}}\}}(L_{\epsilon},0) 
\leq
\bar{D}^{\mathcal{F}}_{\mathcal{F}\cup\{L_{\epsilon},L_{\epsilon_{0}}\}}(L_{\epsilon}, 0) \leq \\ 
\leq \bar{D}^{\mathcal{F}}_{\mathcal{F}\cup 
\{L_{\epsilon},L_{\epsilon_{0}}\}}(L_{\epsilon_{0}}, 0) + \bar{D}^{\mathcal{F}}_{\mathcal{F}\cup\{L_{\epsilon},L_{\epsilon_{0}}\}}(L_{\epsilon_{0}}, L_{\epsilon})\leq \\
\leq \bar{D}^{\mathcal{F}}_{\mathcal{F}\cup 
\{L_{\epsilon},L_{\epsilon_{0}}\}}(L_{\epsilon_{0}}, 0) +2\epsilon_{0} \leq \\ \leq  \bar{D}^{\mathcal{F}}_{\mathcal{F}\cup 
\{L_{\epsilon_{0}}\}}(L_{\epsilon_{0}}, 0) +2\epsilon_{0}~.~
\end{eqnarray*}
The first three inequalities on the first line come from Lemmas \ref{lem:comp-pseudoLag} and \ref{lem:mono}, and, the fourth from the triangle inequality. The next inequality is implied by the upper bound given by the Hofer norm. Finally, the 
last inequality comes from Lemma \ref{lem:mono}.
\end{proof}

\zhnote{The following result concludes the proof of Corollary \ref{cor:gl-metric}.} 

\begin{lem}If $\mathcal{F}'$ is a generic perturbation of $\mathcal{F}$
(in the sense that each element of $\mathcal{F'}$ is a small \zhnote{Hamiltonian perturbation} of a corresponding element in $\mathcal{F}$ and the union
$\mathcal{F}\cup \mathcal{F}'$ is in \zhnote{general} position), then
$$\mathcal{D}^{\mathcal{F},\mathcal{F}'}= \max \{\mathcal{D}^{\mathcal{F}},\mathcal{D}^{\mathcal{F}'}\}$$ is non-degenerate. 
\end{lem} 

\begin{proof} The statement follows from the \zhnote{point (ii) of Corollary \ref{cor:gl-metric}}, that 
\zhnote{has} already \zhnote{been} shown. From this point we deduce that $\mathcal{D}^{\mathcal{F},\mathcal{F}'}(L,L')=0$ implies 
$$\delta (L; L'\cup _{F\in\mathcal{F}} F)=0 = \delta (L; L'\cup _{F'\in\mathcal{F}'} F')~.~$$
The definition of $\delta(-;-)$ is in (\ref{eq:ball-sq}). \zhnote{The first of the last two equalities} means that there is no
standard symplectic ball of any positive radius with its real part on $L$ and that is  disjoint from $L'\cup \cup_{F\in\mathcal{F}}F$. \zjnote{It follows} that $L\subset L'\cup_{F\in\mathcal{F}}F$. Of course the metric is symmetric, so we also have that $L'\subset L\cup \cup_{F\in\mathcal{F}}F$. Given that the same relations are valid for $\mathcal{F}'$, and that $\mathcal{F}'$ and $\mathcal{F}$ are in general position, we deduce that $L=L'$. 
\end{proof}


\newcommand{\xx}{\Gamma}
\newcommand{\shdex}{\mathcal{S}_{\text{e}}}
\newcommand{\shdin}{\mathcal{S}}
\newcommand{\out}{\text{out}}
\newcommand{\err}{\sigma}

\section{The geometry behind TPCs} \label{sec:geom-TPC}

This section illustrates geometrically some of the TPC
machinery. \pbrev{It contains two subsections. In~\S\ref{sb:lcob} we
  explain how the theory of Lagrangian cobordism provides a concrete
  representation for the algebraic structures that are formalized in
  the language of TPCs.}  \pbrev{Some of the material presented
  in~\S\ref{sb:lcob} is based on the theory developed
  in~\cite{Bi-Co:fuk-cob, Bi-Co-Sh:LagrSh}, which is technically
  involved. Below however we have tried to avoid technicalities as
  much as possible in order to put the geometric ideas at the focus,
  at the expense of skipping some details. Most
  of these can be found in the above references.}

\pbrev{In~\S\ref{subsec:estim} we work out some examples where the
  estimates in Theorem~\ref{thm:appl-sympl} can be made concrete.}

\subsection{Lagrangian cobordism and weighted
  triangles} \label{sb:lcob} The theory of Lagrangian cobordism
exhibits in a geometric way several key notions that are fundamental
for the algebraic theory of TPCs.  The purpose of this section is to
provide some geometric interpretations of these notions - in
particular of weighted exact triangles - by using certain natural
symplectic measurements associated to Lagrangian cobordisms.  The
geometric weights coming from geometry (such as those reflecting the
shadows of cobordisms) are bigger than the algebraic weights discussed
before in this paper.  The difficulty with using them in practice is
that they depend on constructing specific cobordisms.

\subsubsection{Background on Lagrangian
  cobordism} \label{sbsb:cobs-background} Let
\zjr{$(X, \omega = d\lambda)$} be a Liouville manifold, endowed with a
given Liouville form $\lambda$. We endow $\mathbb{R}^2$ with the
Liouville form $\lambda_{\mathbb{R}^2} = x dy$ and its associated
standard symplectic structure
$\omega_{\mathbb{R}^2} = d\lambda_{\mathbb{R}^2} = dx \wedge dy$. Let
$\widetilde{X} := \mathbb{R}^2 \times X$, endowed with the Liouville
form $\widetilde{\lambda} = \lambda_{\mathbb{R}^2} \oplus \lambda$ and
the symplectic structure
$\widetilde{\omega} = d(\widetilde{\lambda}) = \omega_{\mathbb{R}^2}
\oplus \omega$. We denote by
$\pi: \mathbb{R}^2 \times X \longrightarrow X$ the projection.

Below we will assume known the notion of Lagrangian cobordism, as
developed in~\cite{Bi-Co:cob1, Bi-Co:fuk-cob}. For simplicity we will
consider only {\em negative-ended} cobordisms
$V \subset \mathbb{R}^2 \times X$, which means that $V$ has only
negative ends. Moreover, all the cobordisms considered below will be
assumed to be exact with respect to the Liouville form
$\widetilde{\lambda}$ and endowed with a given primitive
$f_{V}: V \longrightarrow \mathbb{R}$ of
$\widetilde{\lambda}|_V$. Denote by $L_1, \ldots, L_k \subset X$ the
Lagrangians corresponding to the ends of $V$ and by
$\ell_1, \ldots, \ell_k \subset \mathbb{R}^2$ the negative horizontal
rays of $V$ so that $V$ coincides at $-\infty$ with
$(\ell_1 \times L_1) \coprod \cdots \coprod (\ell_k \times L_k)$.  We
remark that we adopt here the conventions from~\cite{Bi-Co:fuk-cob}
regarding the ends of $V$, namely we always assume that the $j$'th ray
$\ell_j$ lies on the horizontal line $\{ y = j \}$. Also, we allow
some of the Lagrangians $L_j$ to be void.

Note that $\lambda_{\mathbb{R}^2}|_{\ell_i} = 0$ hence
$f_{V}|_{\ell_i \times L_i}$ is constant in the $\ell_i$ direction for
all $i$. Therefore the Lagrangians $L_i \subset X$ are $\lambda$-exact
and $f_{V}$ induces well-defined primitives
$f_{L_i}: L_i \longrightarrow \mathbb{R}$ of $\lambda|_{L_i}$ for each
$i$, namely $f_{L_i}(p) := f_{V}(z_0,p)$ for every $p \in L_i$, where
$z_0$ is any point on $\ell_i$.

\subsubsection{Weakly filtered Fukaya categories and
  cobordism} \label{sbsb:cobs-categs} As constructed
in~\cite{Bi-Co-Sh:LagrSh, Bi-Co:spectr}, there is a weakly filtered
Fukaya $A_{\infty}$-category $W\fuk(X)$ of $\lambda$-exact Lagrangians
with objects being exact Lagrangians $L \subset X$ endowed with a
primitive $f_L:L \longrightarrow \mathbb{R}$ of $\lambda|_L$. The
notation is ad-hoc here to distinguish this category from the filtered
versions constructed in \S \ref{sb:ffuk}.

\begin{rem}\label{rem:wfilt-v-filt} The filtered  $A_{\infty}$ categories
  $\fuk(\mathcal{X};\mathscr{P})$ from \S\ref{sb:ffuk} - in Theorem
  \ref{t:fil-fuk} - are constructed under more restrictive assumptions
  compared to $W\fuk(X)$ as they are associated to only a finite
  number of geometric objects $\bar{\mathcal{X}}$.  However, it is
  easy to see that by choosing the perturbation data required to define
  the weakly filtered category $W\fuk(X)$ in such a way that it
  extends the data $\mathscr{P}$ we have an embedding (of weakly
  filtered $A_{\infty}$-categories):
  $$ \fuk(\mathcal{X};\mathscr{P})\to W\fuk(X)~.~$$
\end{rem}

There is also a weakly filtered Fukaya
$A_{\infty}$-category of cobordisms $W\fuk(\mathbb{R}^2 \times X)$ with
objects being negative-ended exact cobordisms
$V \subset \mathbb{R}^2 \times X$ endowed with a primitive $f_V$ of
$\widetilde{\lambda}|_V$ as above\footnote{For technical reasons one
  needs to enlarge the set of objects in $\fuk(\mathbb{R}^2 \times X)$
  to contain also objects of the type $\gamma \times L$, where
  $\gamma \subset \mathbb{R}^2$ is a curve which outside of a compact
  set is either vertical or coincides with horizontal ends with
  $y$-value being $l \pm \tfrac{1}{10}$, where $l \in
  \mathbb{Z}$.}. We also have the dg-categories of weakly filtered
$A_{\infty}$-modules over each of the previous Fukaya categories,
which we denote by \zjr{$\text{mod}_{W\fuk(X)}$} and
\zjr{$\text{mod}_{W\fuk(\mathbb{R}^2 \times X)}$} respectively.

Below we will mostly concentrate on the chain complexes associated to
various Lagrangians and modules, ignoring the higher order
$A_{\infty}$-operations, and these are genuinely filtered. Thus, in
this discussion the fact that the above categories are only weakly
filtered rather than genuinely filtered will not a play an important
role.

Let $\mathcal{Y}: W\fuk(X) \longrightarrow \text{mod}_{W\fuk(X)}$ be
the Yoneda embedding (in the framework of weakly filtered
$A_{\infty}$-categories) and
$W\fuk(X)^{\nabla} \subset \text{mod}_{W\fuk(X)}$ the triangulated
closure of the image of $\mathcal{Y}$. We denote by
$\mathcal{C} = \mathcal{P}H(\fuk(X)^{\nabla})$ the persistence
homological category associated to $\fuk(X)^{\nabla}$.  This is not a
TPC due to the difference between ``filtered'' vs.~``weakly filtered''
but, with this distinction kept in mind, its properties mimic closely
those of a TPC. To understand the difference, while in a TPC the
composition of two morphisms $f$ of shift $r$ and $g$ of shift $s$ is
a morphism $f\circ g$ of shift $r+s$, in the weakly filtered case, the
composition $f\circ g$ is of shift $r+s+\epsilon^{\mu_{2}}$ where the
error term $\mu_{2}$ is part of the structural data associated to the
{\em weakly} filtered structure of $W\fuk(X)$.

By a slight abuse of notation we will
denote the Yoneda module $\mathcal{Y}(L)$ of a Lagrangian
$L \in \text{Obj}(W\fuk(X))$ also by $L$.

There is also a Yoneda embedding
$W\fuk(\mathbb{R}^2 \times X) \longrightarrow
\text{mod}_{W\fuk(\mathbb{R}^2 \times X)}$ and we will typically denote
the Yoneda modules corresponding to cobordisms by calligraphic
letters, e.g.~the Yoneda module corresponding to
$V \in \text{Obj}(\fuk(\mathbb{R}^2 \times X))$ will be denoted by
$\mathcal{V}$.

Under additional assumptions on $X$, on the Lagrangians taken as the
objects of $W\fuk(X)$ and the \zjr{Lagrangian} cobordisms of
$\fuk(\mathbb{R}^2 \times X)$, one can set up a graded theory,
endowing the morphisms in $\fuk(X)$ and $\fuk(\mathbb{R}^2 \times X)$
with a $\mathbb{Z}$-grading and the categories with
grading-translation functors. See~\cite{Seidel} for the case of
$W\fuk(X)$ and~\cite{Hensel:stab-cond} for grading in the framework of
cobordisms. In what follows we will not explicitly work in a graded
setting, but whenever possible we will indicate how grading fits in
various constructions.

\subsubsection{Iterated cones associated to
  cobordisms} \label{sbsb:cobs-icone} Let
$\gamma \subset \mathbb{R}^2$ be an oriented\footnote{The orientation
  on $\gamma$ is necessary in order to set up a graded Floer theory,
  and also in order to work with coefficients over rings of
  characteristic $\neq 2$. Here we work with
  $\mathbb{Z}_2$-coefficients, therefore if one wants to ignore the
  grading then the orientation of $\gamma$ becomes irrelevant.}  plane
curve which is the image of a proper embedding of $\mathbb{R}$ into
$\mathbb{R}^2$. Viewing $\gamma \subset \mathbb{R}^2$ as an exact
Lagrangian we fix a primitive $f_{\gamma}$ of
$\lambda_{\mathbb{R}^2}|_{\gamma}$. Given an exact Lagrangian
$L \subset X$, consider the exact Lagrangian
$\gamma \times L \subset \mathbb{R}^2 \times X$ and take
$f_{\gamma, L}: \gamma \times L \longrightarrow \mathbb{R}$,
$f_{\gamma, L}(z, p) = f_{\gamma}(z) + f_{L}(p)$ for the primitive of
$\widetilde{\lambda}|_{\gamma \times L}$. From now on we will make the
following additional assumptions on $\gamma$. The ends of $\gamma$
will be assumed to coincide with a pair of rays $\ell'$, $\ell''$ each
of which is allowed to be either horizontal or vertical. Moreover in
the case of a horizontal ray, the ray is assumed to have
$y$-coordinate which is in $\mathbb{Z} \pm \tfrac{1}{10}$ and in the
case of vertical rays we assume the rays have $x$-coordinate being
$0$.

Below we will mainly work with the following two types of such
curves. The first one is
$\gamma^{\uparrow} = \{ x = 0\} \subset \mathbb{R}^2$ (i.e. the
$x$-axis with its standard orientation) and we take
$f_{\gamma^{\uparrow}} \equiv 0$. Then for every exact Lagrangian
$L \subset X$ we can identify $f_{\gamma^{\uparrow},L}$ with $f_L$ in
the obvious way.

The second type is the curve $\gamma_{i,j}$, where $i\leq j$ are two
integers, depicted in Figure~\ref{f:curves} and oriented by going from
the lower horizontal end to the upper horizontal end. Note that by
taking $\gamma_{i,j}$ close enough to dotted polygonal curve in
Figure~\ref{f:curves} we can assume that
$\lambda_{\mathbb{R}^2}|_{\gamma_{i,j}}$ is very close to $0$. We fix
the primitive $f_{\gamma_{i,j}}$ to be the one that vanishes on the
vertical part of $\gamma_{i,j}$.
\begin{figure}[htbp]
   \begin{center}
     \includegraphics[scale=0.63]{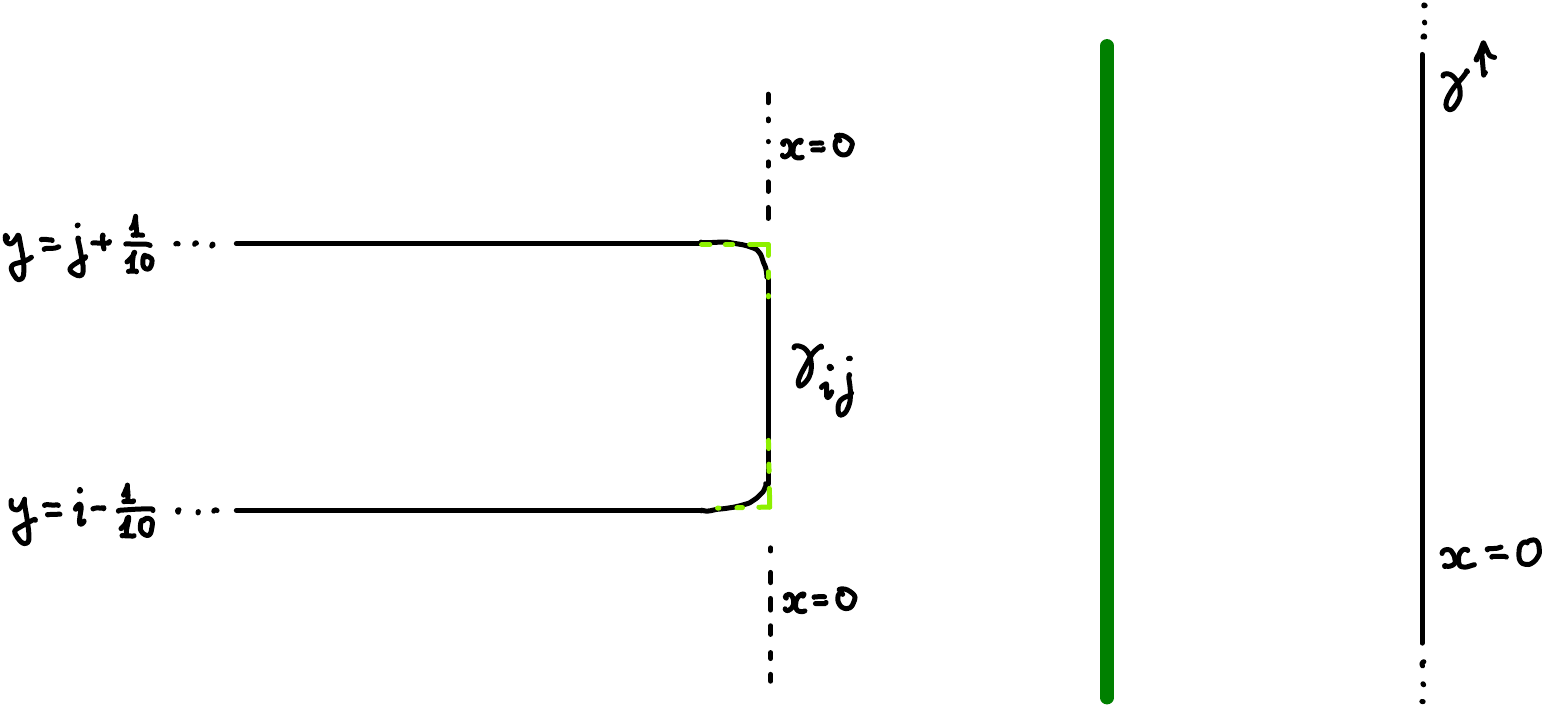}
   \end{center}
   \caption{The curves $\gamma_{i,j}$ and $\gamma^{\uparrow}$.}
   \label{f:curves}
\end{figure}

Let $\gamma \subset \mathbb{R}^2$ and $f_{\gamma}$ be as above.
Following~\cite[Section~4.2]{Bi-Co:fuk-cob},
\cite[Section~3.6]{Bi-Co-Sh:LagrSh} there is a (weakly) filtered
$A_{\infty}$-functor, called an inclusion functor,
$\mathcal{I}_{\gamma}: W\fuk(X) \longrightarrow W\fuk(\mathbb{R}^2
\times X)$ which sends the object $L \in \text{Obj}(W\fuk(X))$ to
$\gamma \times L \in \text{Obj}(\fuk(\mathbb{R}^2 \times X))$.
\pbrev{The first order component of $\mathcal{I}_{\gamma}$ is a chain
  map
  $(\mathcal{I}_{\gamma})_1: CF(N_0, N_1) \longrightarrow CF(\gamma
  \times N_0, \gamma \times N_1)$, defined for all exact Lagrangians
  $N_0$, $N_1$, which induces an isomorphism in homology. Note that
  since $\gamma \times N_0$ and $\gamma \times N_1$ do not intersect
  transversely (unless $N_0 \cap N_1 = \emptyset$) we need to use here
  Floer data with non-trivial Hamiltonians that also involve a
  component in the $\mathbb{R}^2$-direction. We skip these details
  here and refer the reader to~\cite[Section~4.2]{Bi-Co:fuk-cob}
  and~\cite[Pages~68-69]{Bi-Co-Sh:LagrSh} for the precise details. The
  higher order components $(\mathcal{I}_{\gamma})_d$, $d \geq 2$, of
  $\mathcal{I}_{\gamma}$ are defined to be $0$.}

Let $V \subset \mathbb{R}^2 \times X$ be a Lagrangian
cobordism. Denote by $\mathcal{V}$ the (weakly) filtered Yoneda module
of $V$ and consider the pull-back module
$\mathcal{I}_{\gamma}^* \mathcal{V}$.  Note that for every exact
Lagrangian $N \subset X$ we have
$$\mathcal{I}_{\gamma}^* \mathcal{V} (N) = CF(\gamma \times N, V)$$
as {\em filtered chain complexes}. (The filtrations are induced by
$f_V$, $f_{\gamma,N}$ and by the Floer data in case it is not
trivial.)

Assume that the ends of $V$ are $L_1, \ldots, L_k$ and moreover that
$V$ is cylindrical over \pbrev{$(-\infty, \delta] \times \mathbb{R}$}
for some $\delta>0$. (This can always be achieved by a suitable
translation along the $x$-axis.)  Fix $1 \leq i \leq k-1$ and consider
the curve $\gamma_{i,i+1}$ and the pull-back (weakly) filtered module
$\mathcal{I}_{\gamma_{i,i+1}}^*\mathcal{V}$.  \pbrev{The cobordism $V$
  gives rise to a module homomorphism
  $$\xx_{V, \gamma_{i,i+1}}:
  L_{i+1} \longrightarrow L_i$$ which preserves action filtrations and
  such that
  \begin{equation} \label{eq:cone-Gamma_i-1}
    \mathcal{I}_{\gamma_{i,i+1}}^*\mathcal{V} =
    T^{d_i}\text{cone}(L_{i+1} \xrightarrow{\xx_{V, \gamma_{i, i+1}}}
    L_i)
  \end{equation}
  as (weakly) filtered modules. Here $T$ stands for the
  grading-translation functor and the amount $d_i \in \mathbb{Z}$ by
  which we translate depends only on $i$. We will be more precise
  about the values of $d_i$ later on.
  \label{pp:refs-mod-map}
  The references for the construction of the map
  $\xx_{V, \gamma_{i,i+1}}$ are the following. For the unfiltered case
  see:~\cite[Section~4.4]{Bi-Co:fuk-cob} and Proposition~4.4.1 in that
  paper. The map $\xx_{V, \gamma_{i,i+1}}$ is constructed on page~1805
  of that paper, where it is denoted $\phi_j$; see also
  Proposition~4.4.3 in that paper. The weakly filtered case is treated
  in~\cite{Bi-Co-Sh:LagrSh}; see Proposition~3.5 and its proof in that
  paper, pages 73-76. More relevant background material on inclusion
  functors and iterated cones can be found in Sections~3.6 and~3.7 of
  that paper. Note that here we are working in a strictly filtered
  setting (which is a special case of the weakly filtered case) and
  this simplifies many of the arguments from~\cite{Bi-Co-Sh:LagrSh}.
  In addition to these references we provide below
  in~\S\ref{sbsb:cone-map} an outline of the construction of
  $\xx_{V, \gamma_{i,i+1}}$, avoiding technicalities.}

The map $\xx_{V, \gamma_{i,i+1}}$ is canonically defined by $V$ and
$\gamma_{i,i+1}$, up to a boundary \pbrev{in the chain complex}
$\hom^{\leq 0}_{\text{mod}_{\fuk(M)}}(L_{i+1}, L_i)$. Therefore it
gives rise to a well-defined morphism in the homological persistence
category
$\mathcal{C}_0 = H\bigl(\hom^{\leq 0}_{\text{mod}_{\fuk(M)}}(L_{i+1},
L_i)\bigr)$ which by abuse of notation we still denote by
$\xx_{V, \gamma_{i,i+1}} \in \hom_{\mathcal{C}_0}(L_{i+1}, L_i)$.


The above can be generalized to several consecutive ends in a row as
follows. Fix $1 \leq i \leq j \leq k$.
\begin{figure}[htbp]
  \begin{center}
    \includegraphics[scale=0.63]{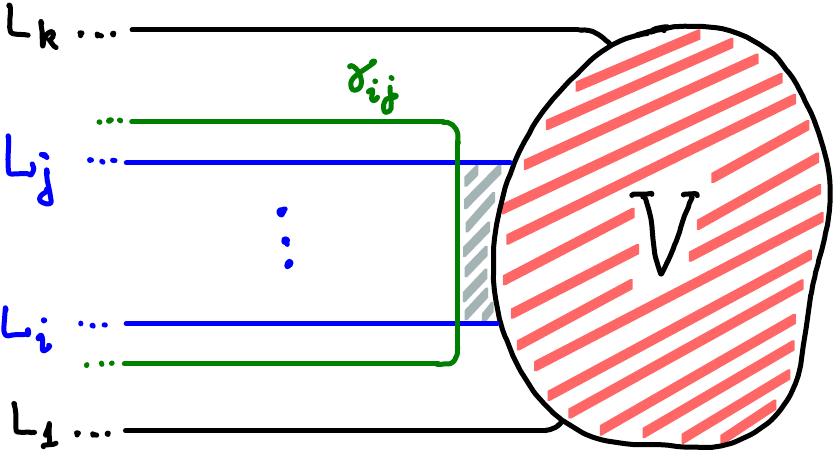}
  \end{center}
  \caption{The module $\mathcal{I}_{\gamma_{i,j}}^* \mathcal{V}$.}
  \label{f:map-ij}
\end{figure}
The pull-back module
$\mathcal{I}_{\gamma_{i,j}}^* \mathcal{V}$ can be identified with an
iterated cone of the type:
\begin{equation} \label{eq:icone-V} \mathcal{I}_{\gamma_{i,j}}^*
  \mathcal{V} = \text{cone}\Bigl(L_j \longrightarrow
  \text{cone}\bigl(L_{j-1} \longrightarrow \ldots \longrightarrow
  \text{cone}(L_{i+1} \longrightarrow L_i) \ldots \bigr)\Bigr),
\end{equation}
where similarly to the case $\xx_{V, \gamma_{i,i+1}}$, all the maps in
the iterated cone are module homomorphism that preserve filtrations.
See figure~\ref{f:map-ij}. \pbrev{The references given above for the
  construction of $\xx_{V, \gamma_{i,i+1}}$ are relevant also for the
  construction of~\eqref{eq:icone-V}.}
  
Note that there are some grading-translations in~\eqref{eq:icone-V}
which we have ignored. We will be more precise about this point later
on when we consider iterated cones involving three objects.

\begin{remnonum}
  If $V$ has $k$ ends then for every $i\leq 1$ and $k \leq l$ we have
  $\mathcal{I}_{\gamma_{i,l}}^*\mathcal{V} =
  \mathcal{I}_{\gamma^{\uparrow}}^*\mathcal{V}$.
\end{remnonum}

Finally, fix $1 \leq i \leq l < j \leq k$, and consider the two
modules $\mathcal{I}_{\gamma_{i,l}}^*\mathcal{V}$ and
$\mathcal{I}_{\gamma_{l+1, j}}^*\mathcal {V}$. There is a module
homomorphism
$\xx_{V, \gamma_{i,l}, \gamma_{l+1, j}}:
\mathcal{I}_{\gamma_{l+1,j}}^*\mathcal {V} \longrightarrow
\mathcal{I}_{\gamma_{i,l}}^*\mathcal{V}$ which preserves filtrations.
Note that we have
$$\mathcal{I}^*_{\gamma_{i,j}} \mathcal{V} =
T^{d_{i,l,j}}\text{cone}(\mathcal{I}^*_{\gamma_{l+1,j}}\mathcal{V}
\xrightarrow{\xx_{V, \gamma_{i,l}, \gamma_{l+1,j}}}
\mathcal{I}^*_{\gamma_{i,l}} \mathcal{V}),$$ for some
$d_{i,l,j} \in \mathbb{Z}$. See Figure~\ref{f:maps-ilj}. \pbrev{While
  the construction of $\xx_{V, \gamma_{i,l}, \gamma_{l+1, j}}$ does
  not explicitly appear in the references mentioned
  after~\eqref{eq:cone-Gamma_i-1} on page~\pageref{pp:refs-mod-map},
  it can be easily deduced from the material in those papers. See
  also~\S\ref{sbsb:cone-map} below.}

\begin{figure}[htbp]
   \begin{center}
     \includegraphics[scale=0.63]{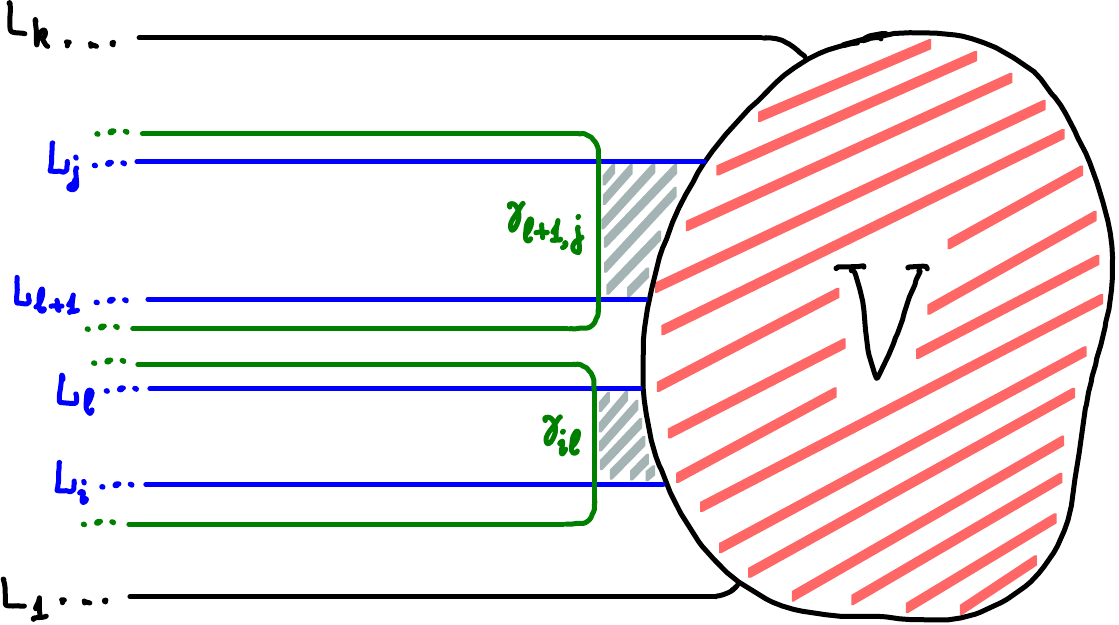}
   \end{center}
   \caption{The modules $\mathcal{I}_{\gamma_{l+1,j}}^* \mathcal{V}$
     and $\mathcal{I}_{\gamma_{i,l}}^* \mathcal{V}$.}
   \label{f:maps-ilj}
\end{figure}

\subsubsection{Module maps induced by
  cobordisms} \label{sbsb:cone-map} \pbrev{The purpose of this section
  is to outline the constructions of the module maps
  $\xx_{V, \gamma_{i,i+1}}: L_{i+1} \longrightarrow L_i$ and
  $\xx_{V, \gamma_{i,l}, \gamma_{l+1, j}}:
  \mathcal{I}_{\gamma_{l+1,j}}^*\mathcal {V} \longrightarrow
  \mathcal{I}_{\gamma_{i,l}}^*\mathcal{V}$
  from~\S\ref{sbsb:cobs-icone}.  We will not give a fully rigorous
  account of the subject here, in an attempt to avoid technicalities
  as much as possible. Full details can be found in the references
  given after~\eqref{eq:cone-Gamma_i-1} on
  page~\pageref{pp:refs-mod-map}.}

\pbrev{We begin with the map
  $\xx_{V, \gamma_{i,i+1}}: L_{i+1} \longrightarrow L_i$. We will
  first explain how to construct the first order component
  $(\xx_{V, \gamma_{i,i+1}})_1$ of this map.}

\pbrev{Consider the pullback module
  $\mathcal{M} := \mathcal{I}_{\gamma_{i,i+1}}^*\mathcal{V}$ and
  Figure~\ref{f:map-ii1}. For every exact Lagrangian $N$ we have the
  following equalities of vector spaces:
  \begin{equation} \label{eq:mod-I-gii1} \mathcal{M}(N) =
    CF(\gamma_{i,i+1} \times N, V) = CF(N, L_{i+1}) \oplus CF(N, L_i).
  \end{equation}
  In terms of Figure~\ref{f:map-ii1} the first summand corresponds to
  the intersection points $N \cap L_{i+1}$ lying above the point $P$
  and the second summand to the intersections $N \cap L_i$ lying above
  $Q$. For the sake of the illustration we have made here several
  simplifying assumption (which cannot really be made in
  general). Namely that $N$ intersects both $L_i$ and $L_{i+1}$
  transversely and that we can take the Floer and perturbation data
  for $W\fuk(\mathbb{R}^2 \times X)$ and $W\fuk(X)$ to have $0$
  Hamiltonian terms.}

\pbrev{Next we consider the differential $\mu_1^{\mathcal{M}}$ of this
  module. Again, for simplicity assume that the almost complex
  structure $J$ in the Floer data for $CF(\gamma_{i,i+1} \times N, V)$
  is chosen such that the projection
  $\pi: \mathbb{R}^2 \times X \longrightarrow \mathbb{R}^2$ is
  $(J,i)$-holomorphic, where $i$ is the standard complex structure on
  $\mathbb{R}^2 \cong \mathbb{C}$. To describe $\mu_1^{\mathcal{M}}$
  we need to consider Floer strips contributing to the differential of
  $CF(\gamma_{i,i+1} \times N, V)$. Apriori these are of four types:
  \begin{enumerate}
  \item[$(PP)$] Strips going from intersection points above $P$ to
    points above $P$.
  \item[$(QQ)$] Strips going from intersection points above $Q$ to
    points above $Q$.
  \item[$(PQ)$] Strips going from intersection points above $P$ to
    points above $Q$.
  \item[$(QP)$] Strips going from intersection points above $Q$ to
    points above $P$.
  \end{enumerate}
}

\pbrev{Our assumptions on $J$ and the Hamiltonians in the Floer data
  imply that all the $PP$ and $QQ$ strips have constant projection to
  $\mathbb{R}^2$ hence completely lie in $\{P\} \times X$ and
  $\{Q\} \times X$, respectively. Moreover, when viewed as strips in
  $X$, they are in 1-1 correspondence with the Floer strips that
  contribute to the differentials on $CF(N, L_{i+1})$ and
  $CF(N, L_i)$, respectively.}

\pbrev{Standard arguments based on complex analysis in the plane
  (e.g.~the open mapping theorem as used
  in~\cite[Section~4]{Bi-Co:cob1},~\cite[Sections~3-4]{Bi-Co:fuk-cob})
  imply that there are no Floer strips of type $(QP)$. However strips
  of type $(PQ)$ may definitely exist and we write their contribution
  to the Floer complex as a linear map:
  $\phi_1: CF(N,L_{i+1}) \longrightarrow CF(N,L_i)$ which is based on
  counting strips emanating from an intersection point above $P$ to an
  intersection point above $Q$.}

\begin{figure}[htbp]
  \begin{center}
    \includegraphics[scale=0.63]{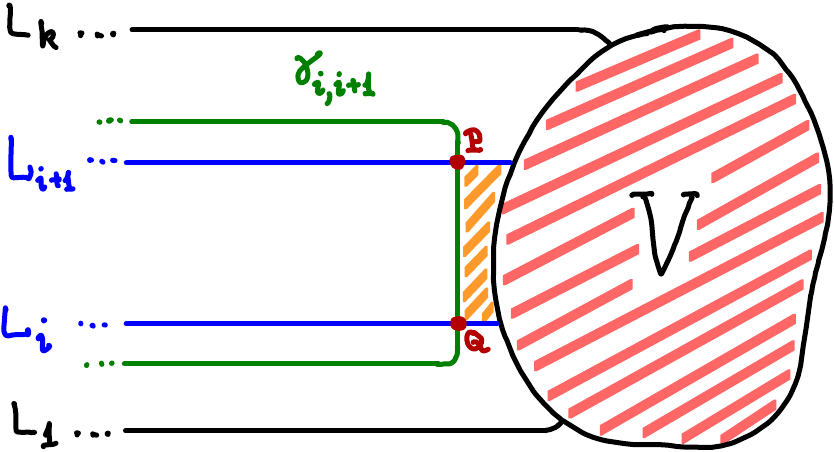}
  \end{center}
  \caption{The module map
    $\xx_{V, \gamma_{i,i+1}}: L_{i+1} \longrightarrow L_i$. The
    projection to $\mathbb{R}^2$ of the strips of type $(PQ)$ is
    depicted in orange.}
  \label{f:map-ii1}
\end{figure}

\pbrev{Summing up, the differential $\mu_1^{\mathcal{M}}$ can be
  written as:
\begin{equation} \label{eq:mod-I-gii1-diff} \mu_1^{\mathcal{M}}(x^P,
  x^Q) = \bigl( \mu_1(x^P), \mu_1(x^Q) + \phi_1(x^P) \bigr).
\end{equation}
Here $x^P$, $x^Q$ are intersection points corresponding to the first
and second summands in~\eqref{eq:mod-I-gii1} and $\mu_1$ stands for
the Floer differentials coming from $W\fuk(X)$.  The fact that
$\mu_1^{\mathcal{M}}$ is a differential implies that $\phi_1$ is a
chain map and moreover that
$\mathcal{M}(N) = \text{cone}(L_{i+1} \xrightarrow{\phi_1} L_i)$ as
chain complexes.} \pbrev{We define the first order component of
$\xx_{V, \gamma_{i,i+1}}$ to be
$(\xx_{V, \gamma_{i,i+1}})_1 := \phi_1$.}

\pbrev{The construction of the higher order components of
  $\xx_{V, \gamma_{i,i+1}}$ is similar, though technically more
  involved. To construct the maps $(\xx_{V, \gamma_{i,i+1}})_d$,
  $d \geq 2$, we need to analyze the $\mu_d$-operations of the module
  $\mathcal{M}$. Fix $d$ exact Lagrangians $N_0, \ldots, N_{d-1}$ in
  $X$ and $\underline{y}=(y_1, \ldots, y_{d-1})$ with
  $y_k \in CF(N_{k-1}, N_k)$. Since $\mathcal{I}_{\gamma_{i,i+1}}$ has
  vanishing higher order components we have
  \begin{equation} \label{eq:mu-d-M} \mu_d^{\mathcal{M}}(\underline{y},
    (x^P, x^Q)) = \mu^{\mathbb{R}^2 \times X}_d \bigl(
    (\mathcal{I}_{\gamma_{i,i+1}})_1(y_1), \ldots,
    (\mathcal{I}_{\gamma_{i,i+1}})_1(y_{d-1}), (x^P,x^Q) \bigr),
  \end{equation}
  where $\mu^{\mathbb{R}^2 \times X}_d$ is the $d$-order
  $A_{\infty}$-operation in $W\fuk(\mathbb{R}^2 \times X)$ and
  $x^P \in CF(N_0, L_{i+1}) ,x^Q \in CF(N_0,L_i)$ are as before.  }

\pbrev{The right-hand side of~\eqref{eq:mu-d-M} counts Floer
  $d$-polygons with ``edges'' on the Lagrangians
  $\gamma_{i,i+1} \times N_0, \ldots, \gamma_{i,i+1} \times N_{d-1},
  V$.  By similar arguments to the ones used for $d=1$ one shows that
  there are no Floer polygons with entry points in
  $(\mathcal{I}_{\gamma_{i,i+1}})_1(y_1), \ldots,
  (\mathcal{I}_{\gamma_{i,i+1}})_1(y_{d-1}), x_Q$ and exit points
  lying above $P$. Consequently~\eqref{eq:mu-d-M} has the shape
  \begin{equation} \label{eq:mu-d-M-2}
    \mu_d^{\mathcal{M}}(\underline{y}, (x^P, x^Q)) = \Bigl(
    \mu_d^{\mathbb{R}^2 \times X}
    \bigl((\mathcal{I}_{\gamma_{i,i+1}})_1(\underline{y}), x^P \bigr),
    \mu_d^{\mathbb{R}^2 \times
      X}\bigr((\mathcal{I}_{\gamma_{i,i+1}})_1(\underline{y}), x^Q
    \bigl) + \phi_d\bigr(\underline{y}, x^P\bigl) \Bigr),
  \end{equation}
  with respect to the splitting
  $\mathcal{M}(N_0) = CF(N_0, L_{i+1}) \oplus CF(N_0,L_i)$ used
  before.  Here $(\mathcal{I}_{\gamma_{i,i+1}})_1(\underline{y})$
  stands for
  $((\mathcal{I}_{\gamma_{i,i+1}})_1(y_1), \ldots,
  (\mathcal{I}_{\gamma_{i,i+1}})_1(y_{d-1}))$.  Thus
  $\phi_d(\underline{y}, x^P)$ counts Floer polygons in
  $\mathbb{R}^2 \times X$ with entry points
  $(\mathcal{I}_{\gamma_{i,i+1}})_1(\underline{y})$, $x^P$ and an exit
  point over $Q$. The $d$-order component
  $(\xx_{V, \gamma_{i,i+1}})_d$ of the desired map
  $\xx_{V, \gamma_{i,i+1}}$ is the multilinear map $\phi_d$.}

\pbrev{It remains to explain why
  $\mathcal{M} = \mathcal{I}_{\gamma_{i,i+1}}^* \mathcal{V}$ is the
  mapping cone (in the $A_{\infty}$-sense) of the map
  $\xx_{V, \gamma_{i,i+1}}$. Consider the other two terms on the
  right-hand side of~\eqref{eq:mu-d-M-2}, namely
  $\mu^{\mathbb{R}^2 \times
    X}_d((\mathcal{I}_{\gamma_{i,i+1}})_1(\underline{y}), x^P)$ and
  $\mu_d^{\mathbb{R}^2 \times
    X}((\mathcal{I}_{\gamma_{i,i+1}})_1(\underline{y}), x^Q)$. These
  two terms can be identified with
  $\mu^{\mathbb{R}^2 \times
    X}_d((\mathcal{I}_{\gamma_{i+1,i+1}})_1(\underline{y}), x^P)$ and
  $\mu_d^{\mathbb{R}^2 \times
    X}((\mathcal{I}_{\gamma_{i,i}})_1(\underline{y}), x^Q)$,
  respectively (note we are using now the curves $\gamma_{i+1,i+1}$
  and $\gamma_{i,i}$, and not $\gamma_{i,i+1}$). In other words the
  preceding two expressions can be identified with the
  $\mu_d$-operations of the pullback modules
  $\mathcal{I}_{\gamma_{i+1,i+1}}^*\mathcal{V}$ and
  $\mathcal{I}_{\gamma_{i,i}}^*\mathcal{V}$ applied to
  $(\underline{y}, x^P)$ and $(\underline{y}, x^Q)$, respectively.}

\pbrev{Now, it follows from~\cite[Section~4.2]{Bi-Co:fuk-cob} that the
  pullback modules $\mathcal{I}_{\gamma_{i+1,i+1}}^*\mathcal{V}$ and
  $\mathcal{I}_{\gamma_{i,i}}^*\mathcal{V}$ are quasi-isomorphic to
  the Yoneda modules of $L_{i+1}$ and $L_i$, respectively. Therefore,
  up to grading-translation we obtain
  $\mathcal{M} = \text{cone}(L_{i+1} \xrightarrow{\phi} L_i)$, where
  $\phi = \{\phi_d\}_{d \geq 1}$, which
  proves~\eqref{eq:cone-Gamma_i-1}. This concludes a rough outline of
  the construction of the map $\xx_{V, \gamma_{i,i+1}}$. }

\pbrev{The definition of the maps
  $\xx_{V, \gamma_{i,l}, \gamma_{l+1, j}}$ is similar to the above and
  we will just go over the main points of the construction. Put
  $\mathcal{Q} = \mathcal{I}_{\gamma_{i,j}}^*\mathcal{V}$ and consider
  Figure~\ref{f:maps-ilj2}. It is not hard to show that for every
  exact Lagrangians $N$ we have
  \begin{equation} \label{eq:Q-mod-split} \mathcal{Q}(N) =
    \mathcal{I}_{\gamma_{l+1,j}}^*\mathcal{V}(N) \oplus
    \mathcal{I}_{\gamma_{i, l}}^*\mathcal{V}(N)
  \end{equation}
  as vector spaces. Elements of the first summand can be written as
  $\underline{x}^P = (x^{P_j}, \ldots, x^{P_{l+1}})$ with
  $x^{P_j} \in CF(N, L_k)$ viewed as lying above the point $P_k$ in
  Figure~\ref{f:maps-ilj2}. Similarly, elements of the second summand
  of~\eqref{eq:Q-mod-split} can be written as
  $\underline{x}^Q = (x^{Q_i}, \ldots, x^{Q_l})$. The differential
  $\mu_1^{\mathcal{Q}}$ of this module turns out to have the following
  shape:
  \begin{equation} \label{eq:Q-mod-split-2}
    \mu_1^{\mathcal{Q}}(\underline{x}^P, \underline{x}^Q) = \Bigl(
    \mu_1^{\mathcal{I}_{\gamma_{l+1,j}}^*\mathcal{V}}(\underline{x}^P),
    \mu_1^{\mathcal{I}_{\gamma_{i, l}}^*\mathcal{V}}(\underline{x}^Q) +
    \psi_1(\underline{x}^P) \Bigr),
  \end{equation}
  where
  $\psi_1: \mathcal{I}_{\gamma_{l+1,j}}^*\mathcal{V}(N)
  \longrightarrow \mathcal{I}_{\gamma_{i, l}}^*\mathcal{V}(N)$ is a
  linear map. The reason for this is similar to what has been
  explained earlier for the module
  $\mathcal{M} = \mathcal{I}_{\gamma_{i,i+1}}^*\mathcal{V}$. Namely,
  there cannot be any Floer strips connecting $Q$-type points with
  $P$-type points.  The term $\psi_1(\underline{x}^P)$ counts Floer
  strips (in $\mathbb{R}^2 \times X$), with one boundary on
  $\gamma_{i,j} \times N$ and the other boundary on $V$, emanating
  from any entry $x^{P_k}$ of $\underline{x}^P$ and going to some
  entry $x^{Q_m}$ of $\underline{x}^Q$. In terms of
  Figure~\ref{f:maps-ilj2} the projection of e.g.~the strips that go
  from $x^{P_{l+1}}$ to $x^{Q_l}$ are depicted in light blue. The
  projection of the strips corresponding to
  $\mu_1^{\mathcal{I}_{\gamma_{l+1,j}}^*\mathcal{V}}$ are in orange
  and those corresponding to
  $\mu_1^{\mathcal{I}_{\gamma_{i, l}}^*\mathcal{V}}$ are in purple.  }

\pbrev{The first order component of the desired map
  $\xx_{V, \gamma_{i,l}, \gamma_{l+1, j}}$ is defined to be the map
  $\psi_1$. The construction of the higher order components of
  $\xx_{V, \gamma_{i,l}, \gamma_{l+1, j}}$ is analogous to the case of
  $\xx_{V, \gamma_{i,i+1}}$ discussed earlier.
}

\begin{figure}[htbp]
  \begin{center}
    \includegraphics[scale=0.63]{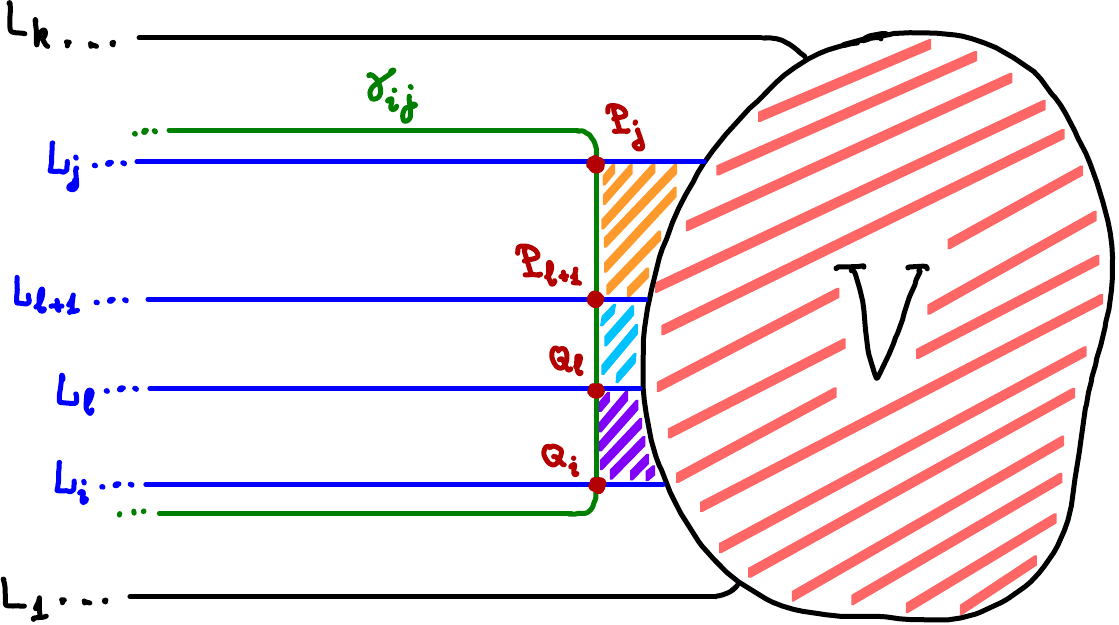}
  \end{center}
  \caption{The module map $\xx_{V, \gamma_{i,l}, \gamma_{l+1, j}}:
  \mathcal{I}_{\gamma_{l+1,j}}^*\mathcal {V} \longrightarrow
  \mathcal{I}_{\gamma_{i,l}}^*\mathcal{V}$.}
  \label{f:maps-ilj2}
\end{figure}

\subsubsection{Shadow of cobordisms} \label{sbsb:cobs-shadow} \pbrev{
  Given a Lagrangian cobordism $V \subset \mathbb{R}^2 \times X$ we
  define its {\em outline}~\cite{Co-Sh} as:
\begin{equation} \label{eq:out} \out(V) := \mathbb{R}^2 \setminus
  \mathcal{U},
\end{equation}
where $\mathcal{U} \subset \mathbb{R}^2 \setminus \pi(V)$ is the union
of all the {\em unbounded} connected components of
$\mathbb{R}^2 \setminus \pi(V)$.} \pbrev{An important measurement
associated to $V$ is its {\em shadow}~\cite{Bi-Co-Sh:LagrSh, Co-Sh},
$\shdin(V)$:
\begin{equation} \label{eq:shadow} \shdin(V) := \text{Area}
  \bigl(\out(V)\bigr).
\end{equation}
Note that $\out(V) \subset \mathbb{R}^2$ is a measurable set hence
$\shdin(V)$ is well-defined.  The shadow plays a central role in
defining cobordism-related metrics on spaces of
Lagrangians~\cite{Bi-Co-Sh:LagrSh}.}

\pbrev{For the purpose of this section it would be easier to work with
  a slightly different variant of the shadow, which we call the {\em
    exterior shadow}, that we introduce now. Fix a rectangle
  $Q \subset \mathbb{R}^2$ which is large enough so that
  $\out(V) \setminus Q$ consists of only horizontal rays and denote
  $$\out_Q(V) = Q \cap \out(V).$$ Define the {\em exterior shadow}
  of $V$ to be:
  \begin{equation} \label{eq:shadow-e}
    \begin{aligned}
      \shdex(V) = \inf \bigl\{A \mid \; & \exists \; \text{a smooth
        embedding} \; \varphi: B \longrightarrow \mathbb{R}^2,
      \; \text{with} \\
      & \textnormal{image\,}(\varphi) \supset \out_Q(V) \; \text{and}
      \; \text{Area(image}(\varphi)) \geq A \bigr\}.
    \end{aligned}
  \end{equation}
  Here $B \subset \mathbb{R}^2$ stands for the closed $2$-dimensional
  unit disk. It is not hard to see that $\shdex(V)$ is independent of
  the choice of the rectangle $Q$.}

\pbrev{Obviously we have $\shdex(V) \geq \shdin(V)$ (this is because
  $\shdin(V) = \text{Area}(\out_Q(V))$). But in fact, we actually have
  $\shdex(V) = \shdin(V)$. Since such a statement has not been proved \ocnote{in full} 
  in the literature (though see~\cite[page~33]{Co-Sh} for a
  related partial argument) we include in~\S\ref{sbsb:shd-ex-in} below
  a sketch of a proof showing that $\shdex(V) = \shdin(V)$.}

\pbrev{Previous papers on the subject used the shadow rather than the
  exterior shadow. However, for the rest of this section, whose
  purpose is mainly \ocnote{illustrative}, we opted for the exterior shadow
  since it is more intuitive to work with.}

\subsubsection{$r$-acyclic objects} \label{sbsb:cob-r-acyc} Let
$V \subset \mathbb{R}^2 \times X$ be a cobordism with ends
$L_1, \ldots, L_k$ such that $V$ is cylindrical over
\zjr{$(-\infty, \delta] \times \mathbb{R}$} for some $\delta>0$.  Let
$\shdex(V)$ be the \pbrev{exterior} shadow of $V$ and \pbrev{denote by
  $\err$ the area} of the region to the right of $\gamma^{\uparrow}$
enclosed between $\gamma^{\uparrow}$ and the projection to
$\mathbb{R}^2$ of the non-cylindrical part of $V$. See
Figure~\ref{f:cob-1}.

\begin{figure}[htbp]
   \begin{center}
     \includegraphics[scale=0.63]{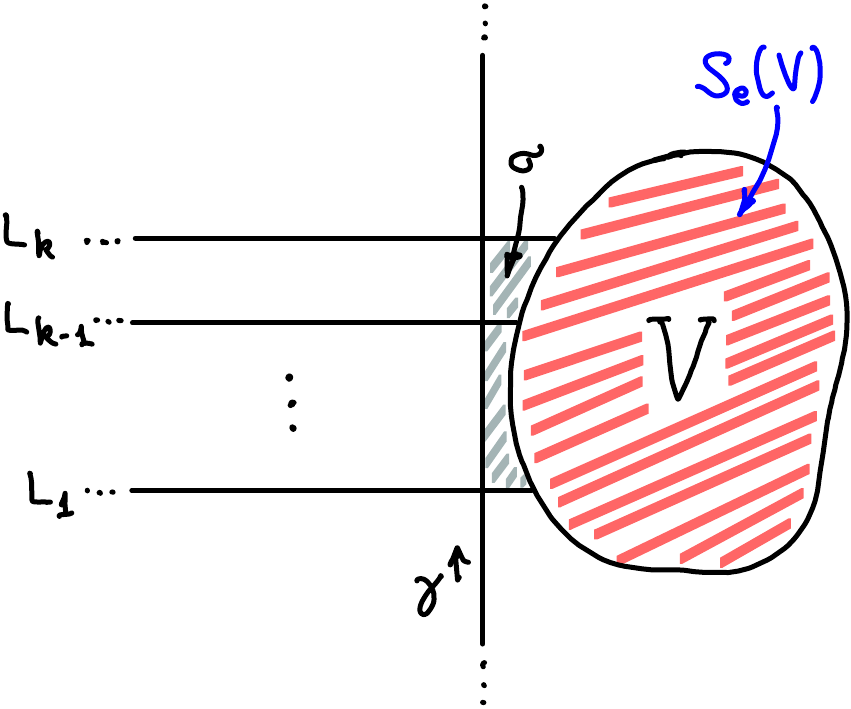}
   \end{center}
   \caption{The projection of a cobordism to $\mathbb{R}^2$,
     \pbrev{its (exterior) shadow $\shdex(V)$ and the area $\err$.}}
   \label{f:cob-1}
\end{figure}

The module
$\mathcal{I}_{\gamma^{\uparrow}}^*\mathcal{V} =
\mathcal{I}_{\gamma_{1,k}}^*\mathcal{V}$ (which also has the
description~\eqref{eq:icone-V} with $i=1$, $j=k$) is $r$-acyclic,
where \pbrev{$r := \shdex(V) + \err$.}  This can be easily seen from
the fact that $V$ can be Hamiltonian isotoped to a cobordism $W$ which
is disjoint from $\gamma^{\uparrow} \times X$ via a compactly
supported Hamiltonian isotopy whose Hofer length is $\leq r$. Standard
Floer theory then implies that
$\mathcal{I}_{\gamma^{\uparrow}}^*\mathcal{V}$ is acyclic of boundary
depth $\leq r$. In the terminology used in this paper this means that
the object $\mathcal{I}_{\gamma^{\uparrow}}^*\mathcal{V}$ is
$r$-acyclic.

\begin{remark} \label{r:cobs-e} The area \pbrev{summand $\err$ that}
  adds to the \pbrev{exterior shadow} of $V$ in the quantity $r$ can
  be made arbitrarily small at the expense of applying appropriate
  shifts to each of the ends $L_i$ of $V$. One way to do this is to
  replace the curve $\gamma^{\uparrow}$ by a curve
  $\gamma_{V}^{\uparrow}$ that coincides with $\gamma^{\uparrow}$
  outside a compact subset and such that $\gamma_V^{\uparrow}$
  approximates the shape of the projection of the non-cylindrical part
  of $V$ in such a way that \pbrev{the area $\err'$ enclosed} between
  $\gamma^{\uparrow}_V$ and $\pi(V)$ is small. See
  Figure~\ref{f:approx}. One can apply a similar modification to the
  curves $\gamma_{i,j}$. Note that, in contrast to
  $f_{\gamma^{\uparrow}}$, the primitive $f_{\gamma_V^{\uparrow}}$ can
  no longer be assumed to be $0$ (a similar remark applies to the
  primitives of the modifications of $\gamma_{i,j}$). As a result the
  cone decompositions~\eqref{eq:icone-V} associated to the pullback
  modules $\mathcal{I}^*_{\gamma_{i,j}}$ will have the same shape but
  each of the Lagrangians $L_i, \ldots, L_j$ will gain a different
  shift in action. Note that this will also result in ``tighter''
  weighted exact triangles than the ones we obtain below, in the sense
  of weights and various shifts on the objects forming these
  triangles.

  \begin{figure}[htbp]
   \begin{center}
     \includegraphics[scale=0.63]{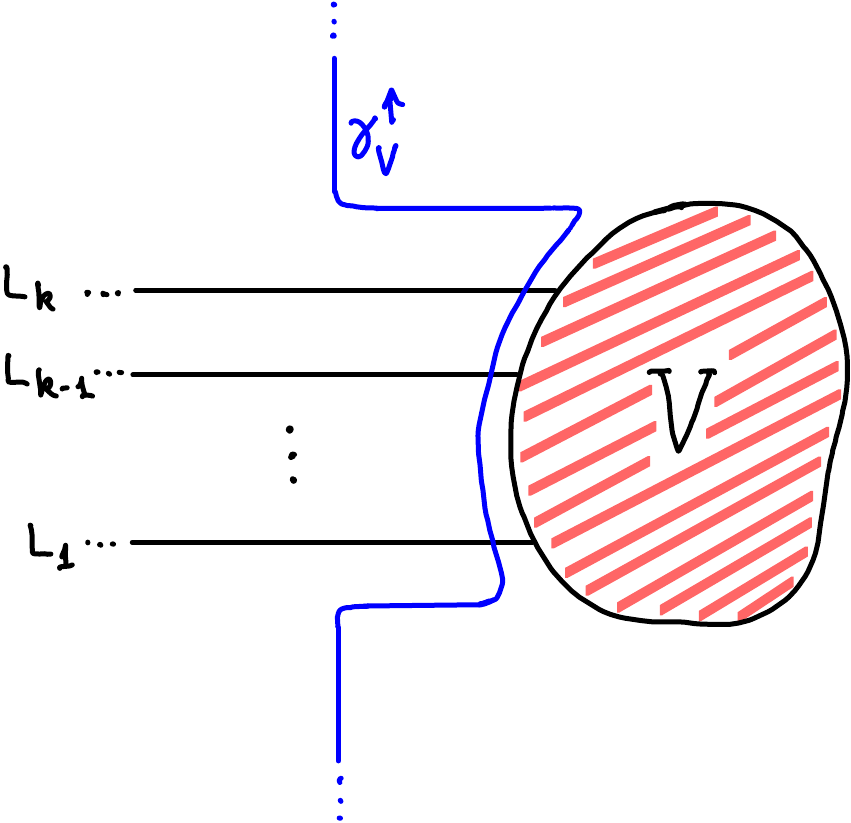}
   \end{center}
   \caption{Replacing $\gamma^{\uparrow}$ by a curve
     $\gamma_V^{\uparrow}$ better approximating  the shape of $V$.}
   \label{f:approx}
\end{figure}

To simplify the exposition, below we will not make these modifications
and stick to the curves $\gamma^{\uparrow}$ and $\gamma_{i,j}$ as
defined above, at the \pbrev{expense of $\err$ not being} necessarily
small and the weights of the triangles not being necessarily optimal.
\end{remark}

\subsubsection{$r$-isomorphisms} \label{sbsb:cob-r-iso} We begin by
visualizing the canonical map $\eta^{L}_r: \Sigma^r L \to L$, where
$L \subset X$ is an exact Lagrangian. Consider the curve
$\gamma \subset \mathbb{R}^2$ depicted in Figure~\ref{f:eta-r}, and
let $r$ be the area enclosed between $\gamma$ and
$\gamma^{\uparrow}$. Let $f_{\gamma}$ be the unique primitive of
$\lambda_{\mathbb{R}^2}|_{\gamma}$ that vanishes along the lower end
of $\gamma$. Note that $f_{\gamma} \equiv r$ along the upper end of
$\gamma$.  Let $V = L \times \gamma$ \zjr{and set} $f_{V} := f_{\gamma,L}$.
Therefore, the primitives induced by $V$ on its ends are as follows:
the primitive on the lower end coincides with $f_L$, while the
primitive on the upper end coincides with $f_L + r$. In other words,
the cobordism $V$ has ends $L$ and $\Sigma^r L$. Moreover, the map
$\xx_{V, \gamma_{1,2}}: \Sigma^r L \to L$ induced by $V$ and
$\gamma_{1,2}$ is precisely $\eta^L_r$.

\begin{figure}[htbp]
   \begin{center}
     \includegraphics[scale=0.63]{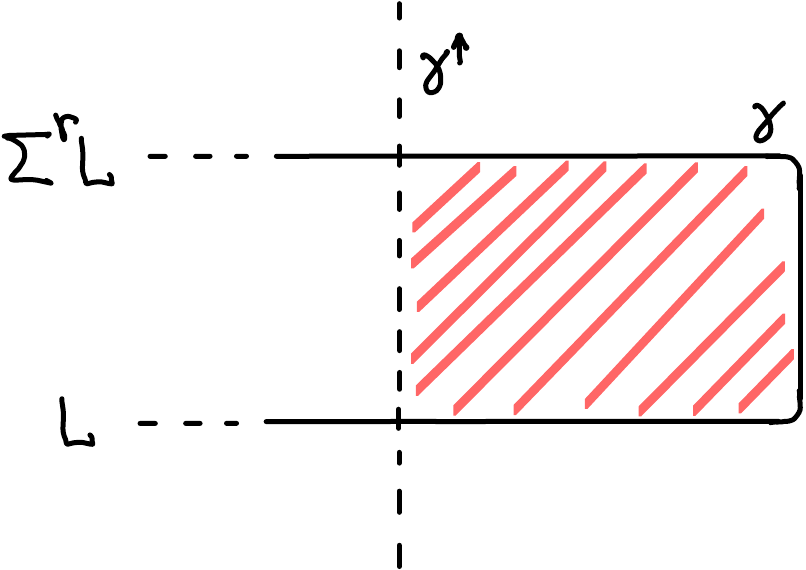}
   \end{center}
   \caption{The cobordism inducing $\eta^L_r: \Sigma^r L \to L$.}
   \label{f:eta-r}
\end{figure}

Another source of geometric $r$-isomorphisms comes from Hamiltonian
isotopies. Let $\phi_t^{H}$, $t \in [0,1]$ be a Hamiltonian isotopy
and let $L \subset X$ be an exact Lagrangian. The Lagrangian
suspension construction gives rise to an exact Lagrangian cobordism
between $L$ and $\phi_1^H(L)$. After bending the ends of that
cobordism to become negative one obtains a Lagrangian cobordism with
negative ends being $L$ and $\phi_1^H(L)$. See
Figure~\ref{f:suspension}. The primitive $f_V$ on $V$ is uniquely
defined by the requirement that $f_V$ coincides with $f_L$ on the
lower end of $V$. The \pbrev{exterior shadow} $\shdex(V)$ of this
cobordism equals the Hofer length of the isotopy
$\{\phi_t^H\}_{t \in [0,1]}$, and we obtain an $r$-isomorphism
$\phi_1^H(L) \to L$.

\begin{figure}[htbp]
   \begin{center}
     \includegraphics[scale=0.63]{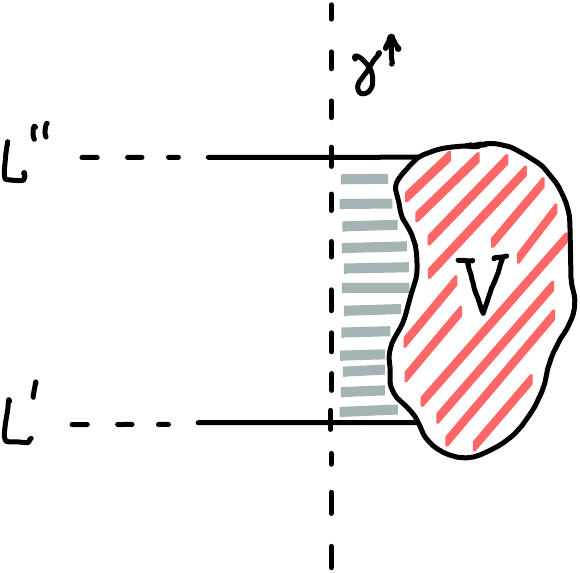}
   \end{center}
   \caption{The Lagrangian suspension of a Hamiltonian isotopy, after
     bending the ends.}
   \label{f:suspension}
\end{figure}

More generally, let $V$ be a Lagrangian cobordism with ends
$L_1, \ldots, L_k$. \pbrev{Let $r = \shdex(V) + \err$ as above.}  Fix
$1 \leq l < k$. As explained above we have, up to an overall
translation in grading,
$$\mathcal{I}^*_{\gamma^{\uparrow}} \mathcal{V} = 
\mathcal{I}^*_{\gamma_{1,k}} \mathcal{V} =
\text{cone}(\mathcal{I}^*_{\gamma_{l+1,k}}\mathcal{V}
\xrightarrow{\xx_{V, \gamma_{1,l}, \gamma_{l+1,k}}}
\mathcal{I}^*_{\gamma_{1,l}} \mathcal{V}),$$ and since
$\mathcal{I}^*_{\gamma^{\uparrow}} \mathcal{V}$ is $r$-acyclic the map
$\xx_{V, \gamma_{1,l}, \gamma_{l+1,k}}:
\mathcal{I}^*_{\gamma_{l+1,k}}\mathcal{V} \longrightarrow
\mathcal{I}^*_{\gamma_{1,l}} \mathcal{V}$ is an $r$-isomorphism.

\subsubsection{Weighted exact triangles} \label{sbsb:cobs-exact-t}

Let $V \subset \mathbb{R}^2 \times X$ be a Lagrangian cobordism with
ends $L_1, \ldots, L_k$. Let $V' \subset \mathbb{R}^2 \times X$ be the
cobordism obtained from $V$ by bending the upper end $L_k$ clockwise
around $V$ so that it goes beyond the end $L_1$ as in
Figure~\ref{f:bend-1}. To obtain a cobordism according to our
conventions, we need to further shift $V'$ upwards by one so that its
lower end has $y$-coordinate $1$ (instead of $0$). Clearly $V'$ is
also exact and $\shdex(V') = \shdex(V)$. We fix the primitive $f_{V'}$
for $V'$ to be the unique one that coincides with $f_V$ on the ends
$L_1, \ldots, L_{k-1}$. A simple calculation shows that $f_{V'}$
induces on the lowest end of $V'$ the primitive $f_{L_{k}} - r$, where
\pbrev{$r = \shdex(V) + \err + \epsilon$.} (Here $\epsilon$ can be
assumed to be arbitrarily small. It can be estimated from above by the
area enclosed between the bent end corresponding to $\ell_k$, the
projection to $\mathbb{R}^2$ of the non-cylindrical part of $V$,
$\ell_1$ and $\gamma^{\uparrow}$. See Figure~\ref{f:bend-1}.)

\begin{figure}[htbp]
   \begin{center}
     \includegraphics[scale=0.63]{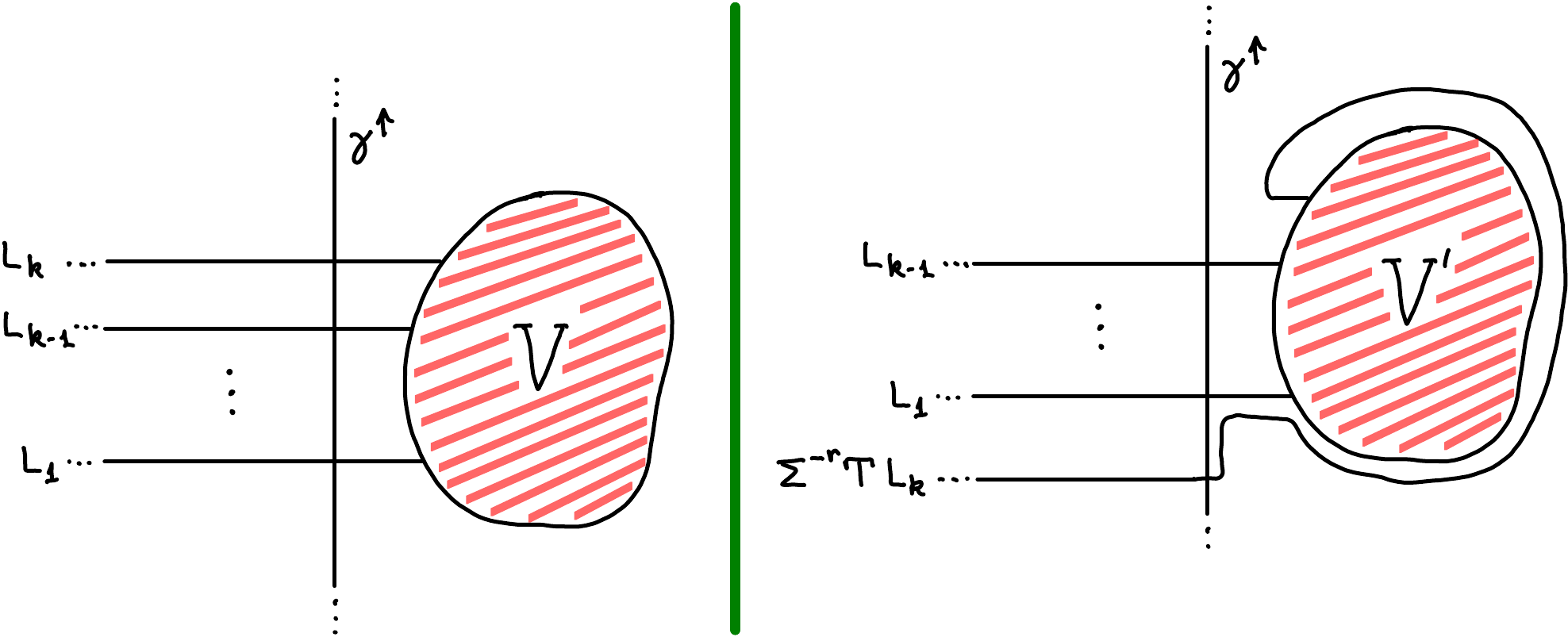}
   \end{center}
   \caption{Bending the upper end $L_k$ of a cobordism $V$.}
   \label{f:bend-1}
\end{figure}

Taking into account grading (in case $V$ is graded in the sense of
Floer theory), one can easily see that the grading on the lowest end
of $V'$ is translated by $1$ in comparison to $L_k$. Summing up, the
above procedure transforms a cobordism $V$ with ends
$L_1, \ldots, L_k$ into a cobordism $V'$ with ends
$\Sigma^{-r}TL_k, L_1, \ldots, L_{k-1}$.

Similarly, one can take $V$ and bend its lowest end $L_1$
counterclockwise around $V$ and obtain a new cobordism $V''$ with ends
$L_2, \ldots, L_k, \Sigma^r T^{-1}L_1$ and with
$\shdex(V'') = \shdex(V)$.

We are now in position to describe geometrically weighted exact
triangles.  Let $V$ be a cobordism with three ends, which for
compatibility with Definition~\ref{dfn-set} we denote by $C, B, A$
(going from the lowest end upward). See Figure~\ref{f:tr-abc}.

\begin{figure}[htbp]
   \begin{center}
     \includegraphics[scale=0.63]{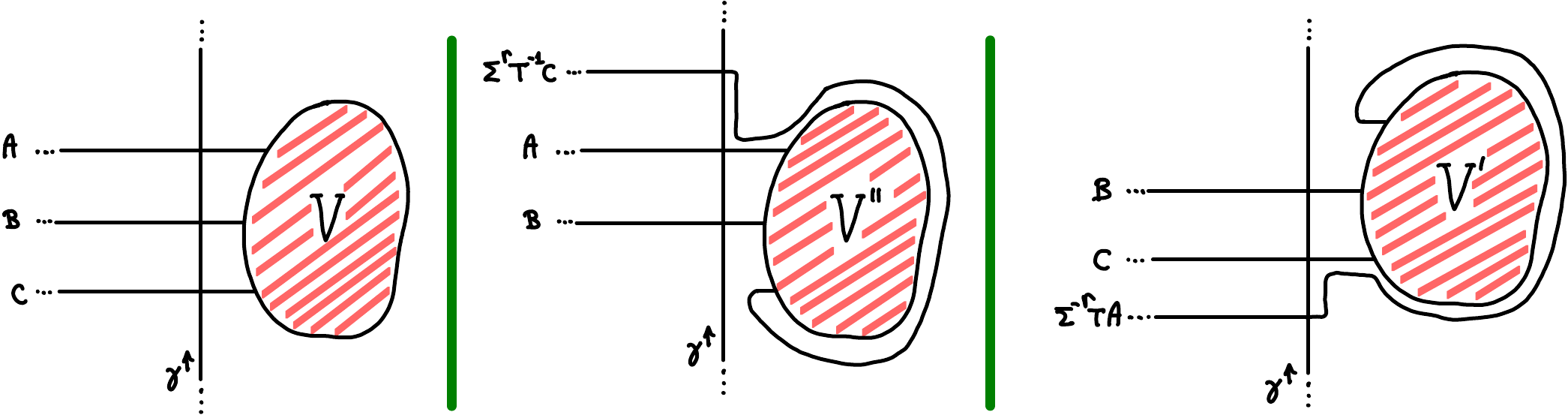}
   \end{center}
   \caption{A cobordism leading to an exact triangle of weight $r$.}
   \label{f:tr-abc}
\end{figure}

\pbrev{Let $r = \shdex(V) + \err + \epsilon$.} Put
\zjr{$$\bar{u} := \xx_{V, \gamma_{2,3}} : A \to B, \quad \bar{v} :=
  \xx_{V, \gamma_{1,2}}: B \to C.$$} Consider also the
counterclockwise rotation $V''$ of $V$ whose ends are
$B, A, \Sigma^r T^{-1}C$. Let
$$\bar{w} := \Sigma^{-r} T \xx_{V'', \gamma_{2,3}} =
\xx_{V', \gamma_{1,2}} : C \to \Sigma^{-r}T A.$$ We claim that
\begin{equation} \label{eq:ex-tr-cob-1} A \xrightarrow{\bar{u}} B
  \xrightarrow{\bar{v}} C \xrightarrow{\bar{w}} \Sigma^{-r}TA
\end{equation}
is a strict exact triangle of weight $r$.  This triangle is based on
the genuinely exact triangle from $\C_0$:
\begin{equation} \label{eq:ex-tr-cob-2} A \xrightarrow{u} B
  \xrightarrow{v} C' \xrightarrow{w} TA,
\end{equation}
where $u = \bar{u}$,
$C' = \mathcal{I}^*_{\gamma_{2,3}}V = \text{cone}(A \xrightarrow{u}
B)$, $v: B \to C'$ is the standard inclusion and $w: C' \to TA$ the
standard projection. The $r$-isomorphism $\phi: C' \to C$ and its
right $r$-inverse $\psi: \Sigma^rC \to C'$ are as follows:
$$\phi = \xx_{V, \gamma_{1,1}, \gamma_{2,3}}, \quad
\psi = T\xx_{V'', \gamma_{1,2}, \gamma_{3,3}}: \Sigma^r C\to T
\mathcal{I}_{\gamma_{1,2}}^*V'' = C'.$$ Note that
$\mathcal{I}_{\gamma_{1,2}}^*V'' = T^{-1}\mathcal{I}_{\gamma_{2,3}}^*V
= T^{-1}C' = T^{-1}\text{cone}(A \xrightarrow{u} B)$.

The fact that $\psi$ is a right $r$-inverse to $\phi$ and that these
maps fit into the diagram~\eqref{dfn-set-2} follows from standard
arguments in Floer theory. Note that these statements do not hold on
the chain level, but only in $\C_0$.
\subsubsection{Rotation of triangles} \label{sbsb:cobs-rot} Let $V$ be
a cobordism with three ends $C, B, A$ as in~\S\ref{sbsb:cobs-exact-t}
and consider the exact triangle~\eqref{eq:ex-tr-cob-1} of weight
$r$. Let $V'$ be the clockwise rotation of $V$, with ends
$B, C, \Sigma^{-r}TA$. The exact triangle associated to $V'$ is
\begin{equation} \label{eq:ex-tr-cob-3} B \xrightarrow{\bar{v}} C
  \xrightarrow{\bar{w}} \Sigma^{-r}TA \xrightarrow{\bar{u}'}
  \Sigma^{-r-\epsilon'} TB, 
\end{equation}
where $\epsilon'$ can be assumed to be arbitrarily small.  It is not
hard to see that in $\C_{\infty}$ (up to identifying objects with
their shifts and ignoring signs in the maps) the exact
triangle~\eqref{eq:ex-tr-cob-3} is precisely the rotation of the exact
triangle corresponding to~\eqref{eq:ex-tr-cob-1} in $\C_{\infty}$.

The above shows that rotation of weighted exact triangles coming from
cobordisms with three ends preserves weights (up to an arbitrarily
small error). Interestingly this is sharper than the case in a general
TPC, described in Proposition~\ref{prop-rot}, where the weight of a
rotated triangle doubles. See also Remark~\ref{r:rotation-weight}.

\subsubsection{Weighted octahedral property} \label{sbsb:cobs-oct}

The weighted octahedral formula from Proposition~\ref{prop-w-oct}
admits too a geometric interpretation in the realm of cobordisms. We
will not give the details of this construction here. Instead we will
briefly explain the cobordism counterpart of cone refinement and why
it behaves additively with respect to weights, as described
algebraically in Proposition~\ref{prop-cone-ref}. Note that weighted
cone-refinement is one of the main corollaries of the weighted
octahedral property.

For simplicity we focus here on the case described in
Example~\ref{ex:crefine} and ignore the grading-translation $T$.
Assume that we have two cobordisms $V$ with ends $X, B, A$ and $U$
with ends $A, F, E$.

These cobordisms induce two weighted exact triangles of weights
\pbrev{$r = \shdex(V) + \err_V + \epsilon$ and
  $s = \shdex(U) + \err_U + \epsilon$.} By gluing the two cobordisms
along the ends corresponding to $A$ we obtain a new cobordism $W$ with
four ends $X, B, F, E$. See Figure~\ref{f:glue-cob}. By the previous
discussion this exhibits $X$ as an iterated cone with linearization
$(B, F, E)$ which corresponds precisely to the algebraic cone
refinement of $X$ with linearization $(B,A)$ by $A$ with linearization
$(F, E)$.

Clearly the \pbrev{exterior shadow} $\shdex(W)$ of $W$ equals to
$\shdex(U) + \shdex(V)$, and the total weight of the cone
decomposition of $X$ associated to $W$ is $r+s$. Note again, that by
the procedure explained in Remark~\ref{r:cobs-e} one can make the
errors $e_V$, $e_U$ small at the expense of applying some shifts to
the elements of the linearization.
\begin{figure}[htbp]
   \begin{center}
     \includegraphics[scale=0.63]{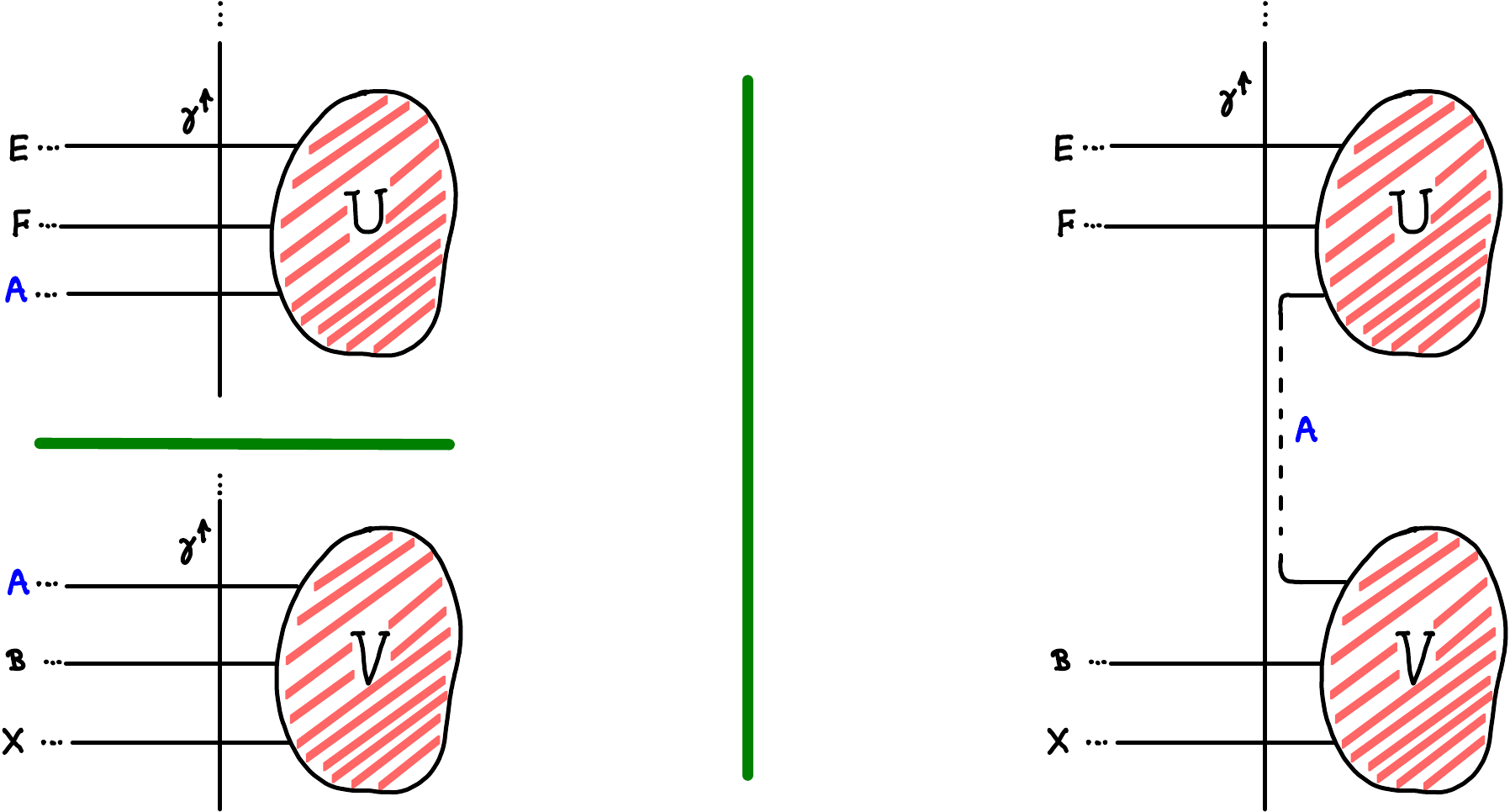}
   \end{center}
   \caption{Cone refinement via gluing cobordisms along two ends.}
   \label{f:glue-cob}
\end{figure}
\begin{rem}
  It is well known that Lagrangian cobordisms gives rise to a category
  with objects Lagrangian submanifolds and with morphisms certain
  Lagrangian cobordisms (see \cite{Bi-Co:fuk-cob}). Combined with the
  discussion above, it is natural to wonder whether by taking into
  account also the shadows of cobordisms this category is naturally
  endowed with a TPC structure. The difficulty to achieve this is that
  one needs to have a triangulated structure that serves as the level
  $0$-part of this expected TPC.  Achieving geometrically such a
  triangulated structure is delicate as it requires including immersed
  Lagrangians and cobordisms in the construction. To further produce a
  TPC structure, this construction needs to be combined with control
  of cobordism shadows which makes the whole machinery even more
  complex. Only partial results in this direction are available at the
  moment, as in \cite{Bi-Co:LagPict}.
\end{rem}

\subsubsection{Exterior shadow equals shadow} \label{sbsb:shd-ex-in}
\pbrev{Here we \ocnote{sketch} a proof showing that
  $\shdex(V) = \shdin(V)$. } \pbrev{Throughout the proof below we will
  assume that $V$ is connected and that the number of} \pbrrvv{its
  ends} \pbrev{is $l \geq 1$. (We do not consider in this paper
  cobordisms with no ends anyway.)}

\pbrev{As at the beginning of~\ref{sbsb:cobs-shadow}, fix an open
  rectangle $Q \subset \mathbb{R}^2$ such that $\out(V) \setminus Q$
  consists of only horizontal rays. Denote by
  $C_1, \ldots, C_l \subset \mathbb{R}^2$ the connected components of
  $\mathbb{R}^2 \setminus \out(V)$, ordered in counterclockwise order
  (e.g.~with respect to a large enough circle surrounding $Q$). Set
  also $C_{Q,i} := C_i \cap Q$, $i=1, \ldots, l$.  }

\pbrev{We first claim that each $C_i$ is simply connected. Indeed,
  this follows easily from the fact that an open subset $\mathbb{R}^2$
  is simply connected if and only if every connected component of its
  complement is unbounded. It follows from the definition of $\out(V)$
  that every connected component of $\mathbb{R}^2 \setminus C_i$ is
  unbounded, hence $C_i$ is simply connected.}

\pbrev{Since the $C_i$'s are simply connected then the same holds also
  for the sets $C_{Q,i}$, $i=1, \ldots, l$. It follows (e.g.~by
  uniformization) that each of the open sets $C_{Q,i}$ is
  diffeomorphic to an open disk. Furthermore, by the Greene-Shiohama
  theorem~\cite{Gre-Shi:non-compact} it follows that each $C_{Q,i}$ is
  in fact symplectomorphic to an open disk
  $\textnormal{Int\,} B^2(R_i)$ of some radius $R_i$, endowed with its
  standard symplectic structure. Fix such symplectomorphisms
  $\psi: B^2(R_i) \longrightarrow C_{Q,i}$ for all $i$.}

\pbrev{Assume for the moment that $l \geq 2$ (note that $l$ is
  precisely the number of ends of $V$). Reduce the radii $R_i$
  slightly to $R'_i = R_i - \delta$ for small $\delta>0$ and consider
  the corresponding domains $C'_{Q,i} = \psi(B^2(R'_i))$. Next,
  connect the boundary of $C'_{Q,i}$ to the boundary of $C'_{Q,i+1}$
  by a small strip $J_i$ that intersects $\out_Q(V)$ only along the
  areas where $\out_Q(v)$ consists solely of horizontal rays. If we
  smoothly (not symplectically) parametrize $J_i$ as
  $[-\epsilon, \epsilon] \times [0,1]$ we just embed $J_i$ in
  $\mathbb{R}^2$ in such a way that:
\begin{enumerate}
\item The area of $J_i$ is very small and $J_i \subset Q$.
\item $[-\epsilon, \epsilon] \times \{0\}$ is mapped to
  $\partial C'_{Q,i}$ near one of the horizontal rays, say $E_i$, of
  $\out_Q(V)$ that lies near $\partial C'_{Q,i}$.
\item $[-\epsilon, \epsilon] \times \{1\}$ is mapped to
  $\partial C'_{Q_{i+1}}$ near the same horizontal ray $E_i$ we have
  just used in~(1) above.
\item The rest of $J_i$ intersects $\out(V)$ only along $E_i$.
\end{enumerate}
We can think of the outcome of connecting $C'_{Q,i}$ to $C'_{Q,i+1}$
with $J_i$ as boundary connected sum of the closures of the domain
$C'_{Q,i}$ and $C'_{Q,i+1}$.  }

\pbrev{We perform the above construction for all $1 \leq i \leq l-1$
  and finally we connect $C'_{Q,l}$ back to $C'_{Q,1}$ in a similar
  manner, keeping the counterclockwise direction.}

\pbrev{Denote by $C'_Q$ the union of all the domains $C'_{Q,i}$
  together with the connecting small strips $J_i$. The outcome
  $C'_Q \subset Q$ is a domain diffeomorphic to an annulus. Its inner
  boundary encircles a domain $F$ which is diffeomorphic to the
  $2$-dimensional disk $B$ and $\out(V) \setminus F$ consists of only
  horizontal rays. Moreover, by taking the parameter $\delta$ small
  enough we can assume that
  $\text{Area}(Q \setminus F) = \text{Area}(Q) - \text{Area}(F)$ is
  arbitrarily close to
  $\text{Area}(Q) - \text{Area}(\out_Q(V)) = \text{Area}(Q) -
  \shdin(V)$. It follows that $\text{Area}(F)$ is arbitrarily close to
  $\shdin(V)$ and at the same time
  $$\ocnote{\text{Area}(F)} \geq \shdex(V) \geq \shdin(V).$$ This concludes the proof under
  the assumption that $l \geq 2$.}

\pbrev{The case $l=1$ is very similar, only that now we have just one
  domain $C'_{Q,1}$ and we form the annulus $C'_Q$ by just gluing the
  small strip $J$ to connect two portions of the boundary of the same
  domain $C'_{Q,1}$.  \Qed }


\subsection{Some explicit estimates}\label{subsec:estim}
We will  illustrate here the statements in
 Theorem \ref{thm:appl-sympl}.

Our base manifold will be denoted here by $W$ and it is the plumbing  of \zhnote{two copies of} disk cotangent \zhnote{bundles} $D^{\ast}S^{1}$ of $S^{1}$ as in
Figure \ref{f:Plumbing0}. The family $\mathcal{F}$ has two elements $F_{1}$ and $F_{2}$
as in this figure. They intersect in the single point $P$.
 \begin{figure}[h]
    \includegraphics[scale=0.70]{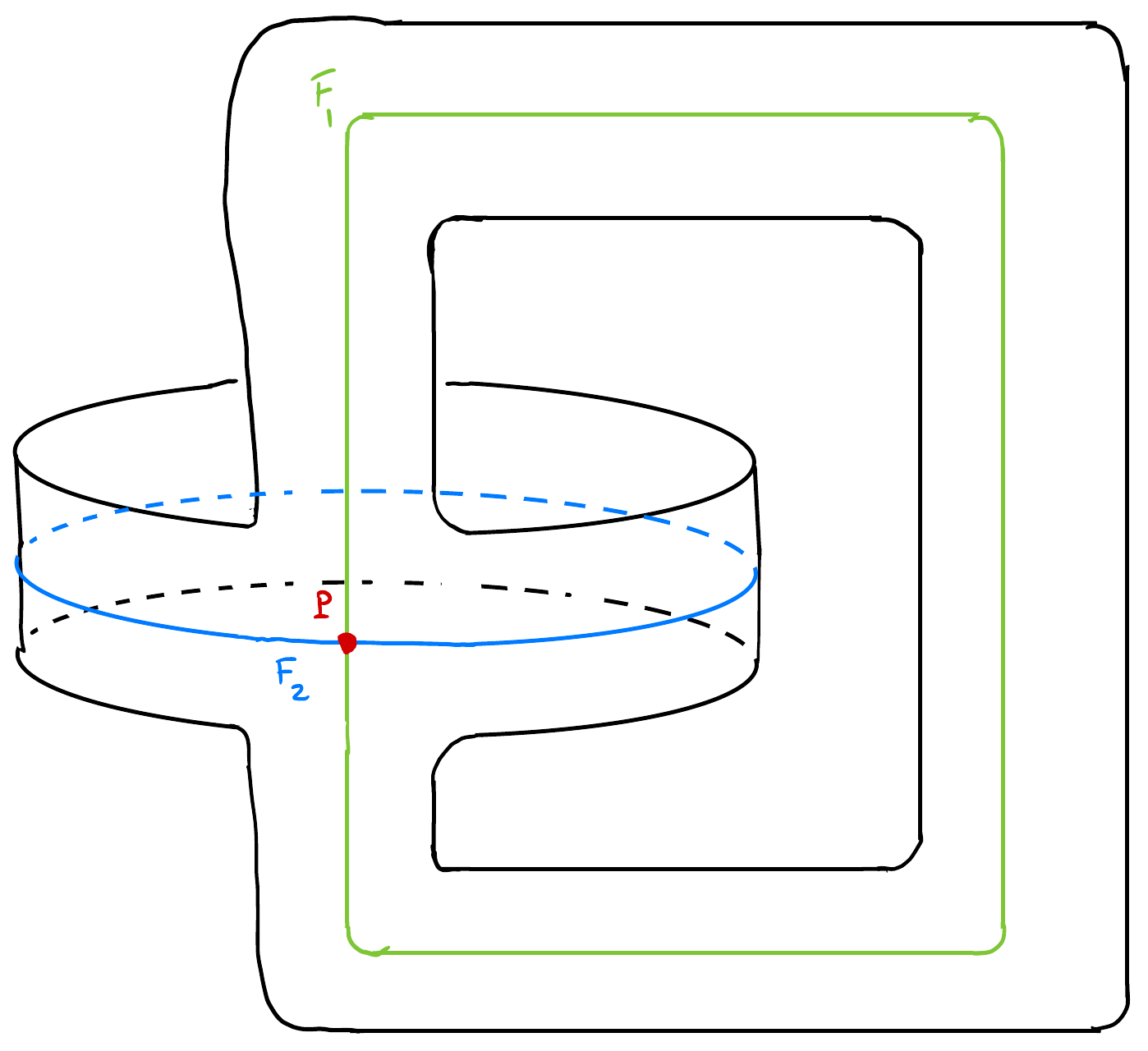} \centering
    \caption{The manifold $W$ and the Lagrangians $F_{1}$ and $F_{2}$.}
    \label{f:Plumbing0}
  \end{figure}
The primitives on both $F_{1}$ and $F_{2}$ are the functions identically equal to $0$.
The family $\mathcal{X}$ consists of $F_{1}$, $F_{2}$ and the Lagrangians $Y$, $Z$, $X$ and $N$ from 
Figures \ref{f:Plumbing1} and \ref{f:Plumbing2}. 
 \begin{figure}[h]
    \includegraphics[scale=0.70]{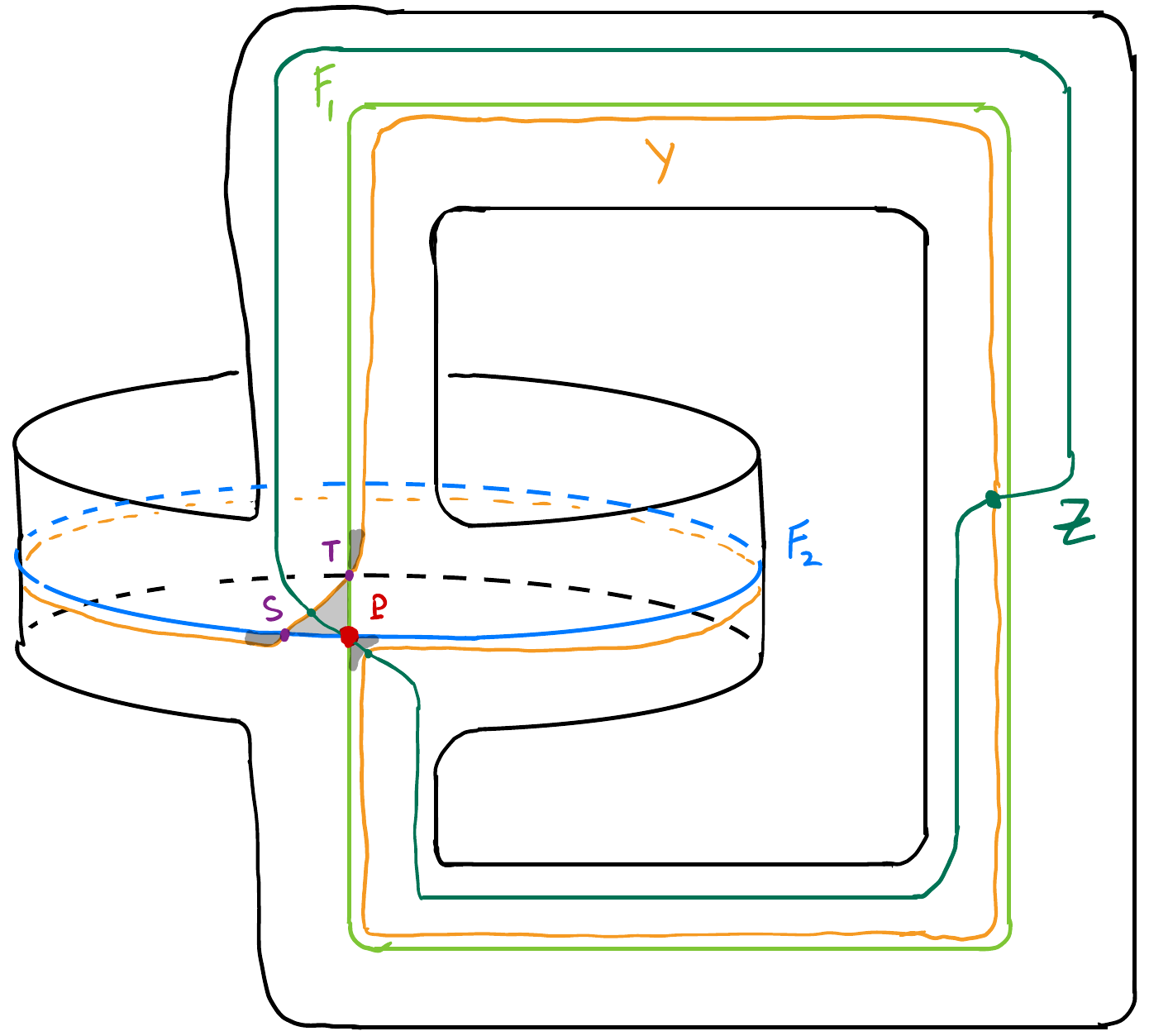} \centering
    \caption{The Lagrangian $Y$ obtained as a small perturbation of the surgery $F_{2}\#_{P}F_{1}$.}
    \label{f:Plumbing1}
  \end{figure}
The Lagrangian $Y$ is constructed from the surgery
$F_{2}\# F_{1}$ at the point $P$ (with a small handle) followed by a small Hamiltonian perturbation. It is easy to see that  for $Y$ to be exact we need for the ``small'' gray triangle $STP$ to have equal area \zhnote{as} the 
``large'' triangle with the same vertices (only the corners of \zhnote{the second} triangle are greyed in the figure). We will
denote the area of \zhnote{these triangles} by $A_{Y}$.
\begin{figure}[h]
    \includegraphics[scale=0.70]{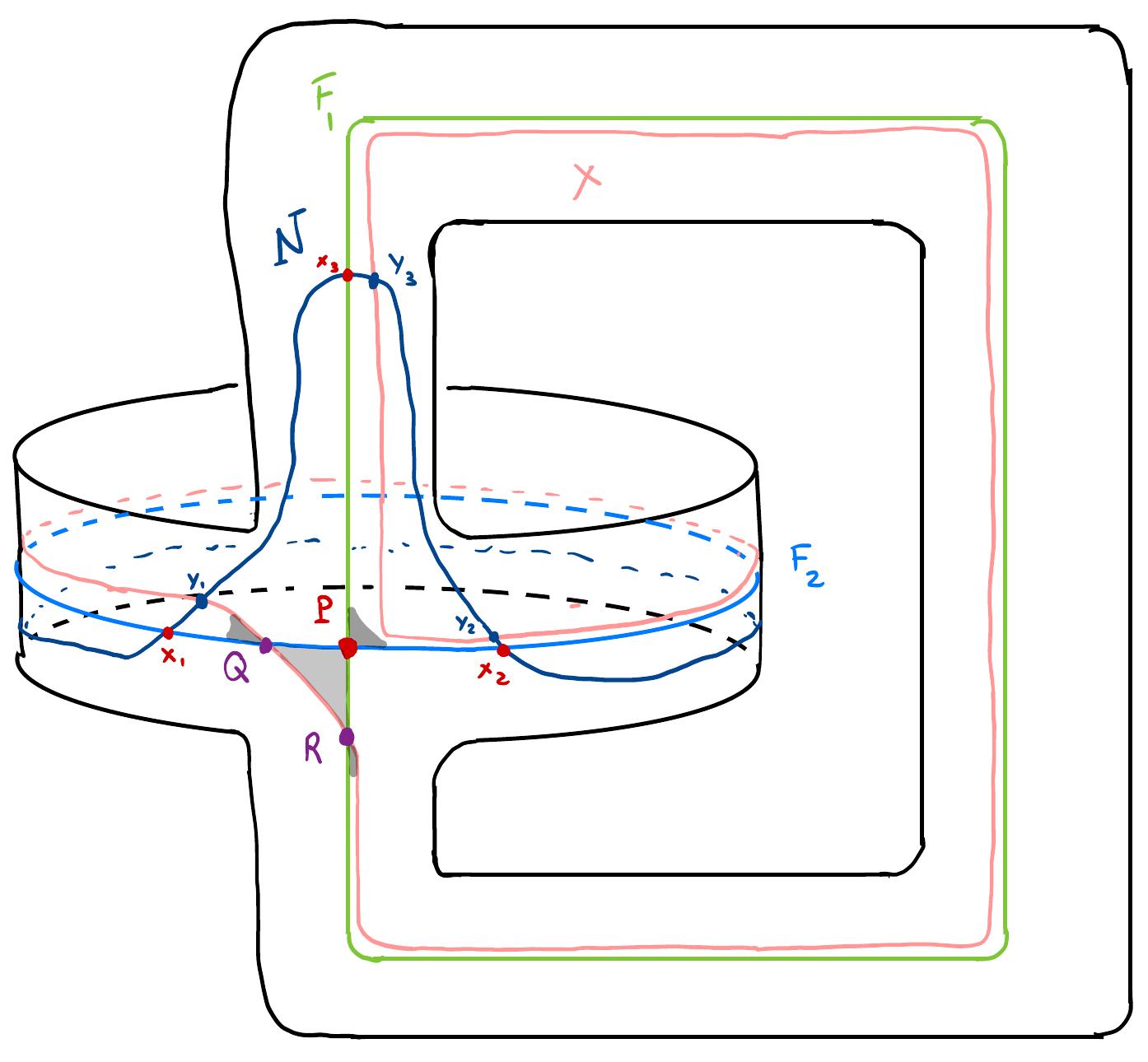} \centering
    \caption{The Lagrangian $X$ obtained as a small perturbation of the surgery $F_{1}\#_{P}F_{2}$
    and the Lagrangian $N$, which is a large Hamiltonian perturbation of $F_{2}$.}
    \label{f:Plumbing2}
  \end{figure}
  Similarly, the Lagrangian $X$ is constructed from the surgery
  $F_{1}\# F_{2}$ at the point $P$ (again with a small handle),
  followed by a small Hamiltonian perturbation. Again the ``small''
  gray triangle $QRP$ has the same area as the ``large'' triangle with
  the same vertices and we denote this area by $A_{X}$.  The
  Lagrangian $N$ is obtained from $F_{2}$ by a Hamiltonian
  perturbation that is large - \zhnote{its Hofer distance equals the}
  area of the strip comprised in between $N$ and $F_{2}$ \zhnote{and
    the} points $x_{1}$ and $x_{2}$ (there are two such strips but
  they have both the same area). We will denote this area by $A_{N}$.
  Similarly, the Lagrangian $Z$ is obtained from $F_{1}$ by a large
  Hamiltonian perturbation.

The first obvious remark is that $X$ and $Y$ are not smoothly isotopic
because homologically $[X]=[F_{2}]-[F_{1}]$ and $[Y]=[F_{2}]+[F_{1}]$
\zhnote{and these are not equal in} $H_{1}(X,\Z)$.

We are interested in the distance $D^{\mathcal{F}}$. 
\begin{lem}\label{lem:example1} We have the inequalities:
  \begin{equation}\label{eq:ineq-simple}
    \frac{A_{X}}{4}\leq D^{\mathcal{F}}(X,0) \leq  2 A_{X}
    \ , \ \frac{A_{Y}}{4}\leq D^{\mathcal{F}}(Y,0) \leq 2 A_{Y} ~,~
  \end{equation}  
  \begin{equation}\label{eq:ineq-simple2}
    \frac{\max\ \{A_{Y}, A_{X}\}}{4}
    \leq D^{\mathcal{F}}(X,Y)\leq 2A_{X}+2A_{Y}~.~
  \end{equation}
\end{lem}
 
\begin{proof} We first show the upper bounds in
  (\ref{eq:ineq-simple}).  The cases of $X$ and $Y$ are perfectly
  similar and we focus on $X$. For this we consider the cone of the
  map $F_{1}\stackrel{P}{\longrightarrow} F_{2}\to K_{1}$ (constructed
  in terms of $A_{\infty}$- Yoneda modules). \zhnote{We claim that}
  \pbred{the} module $K_{1}$ \zhnote{can be mapped} to the Yoneda
  module of $X$ \zhnote{by a quasi-isomorphism}. The simplest way to
  see this geometrically is the following: interpret the module
  $K_{1}$ as the Yoneda module of a marked immersed Lagrangian with
  one marked self-intersection point (marked in the order
  $(F_{1},F_{2})$). This type of \zhnote{Lagrangians} is discussed in
  \cite{Bi-Co:LagPict} for instance. The map we are looking for is of
  the form $\psi=\mu_{2}(-, R): K_{1}\to X$ with $R$ the intersection
  point in Figure \ref{f:Plumbing2}. Of course, once we ``guess'' this
  morphism, we can write it purely algebraically.  It is easy to see
  that this is a quasi-isomorphism.  For instance, applying it to the
  Lagrangian $N$ in the picture it sends \zhnote{$x_{i}$ to $y_{i}$
    for $i = 1,2.$} Moreover, there is also a quasi-isomorphism
  $\phi=\mu_{2}(-, Q): X\to K_{1}$ which is a quasi-inverse of the
  first (on $N$ it is an actual inverse). We can fix the primitive on
  $X$ \zhnote{that vanishes} at the point $Q$, and thus the primitive
  on $X$ has value $A_{X}$ at $R$.  In the terminology of the paper,
  \zhnote{we have maps of filtered modules}
$$\Sigma^{A_{X}} K_{1}\to X\to K_{1}$$ \zhnote{whose composition agrees} with the map $\eta_{A_{X}}$ (in other words, the composition is the identity if the  filtration is neglected but once the filtration is taken into account, it shifts the filtration by $A_{X}$)\zhnote{. We} also have the similar \zhnote{identity} in the opposite direction. 

By applying the same argument as in the second part of Lemma \ref{lem:ineq-inter} we deduce that $\phi$ and $\psi$ are $2 A_{X}$-isomorphisms which implies an inequality for the  half-distance $\bar{\delta}^{\mathcal{F}}(X,0)\leq 2 A_{X}$.  The other inequality, for the second half distance, is easy to produce using the fact that the cone of $\phi : X\to K_{1}$ is  $2 A_{X}$-acyclic  and this implies our upper bound. 

For the lower bound notice that $\delta(X; F_{1}\cup F_{2})= 2 A_{X}$ and thus the lower bound follows from Theorem \ref{thm:appl-sympl} (ii) (here $\delta (-;-)$ is the relative Gromov width as in \S\ref{subsec-symp}).

Clearly, in a perfectly similar way we also have $D^{\mathcal{F}}(Y,0)\leq 2A_{Y}$
and thus $$D^{\mathcal{F}}(X,Y)\leq 2 A_{X} + 2 A_{Y}$$
which is the upper bound in (\ref{eq:ineq-simple2}).

\begin{rem}The first part of the argument is very similar to the one relating the spectral distance to the distance $D^{\mathcal{F}}$. Indeed, one can think about the two points $R$ and $Q$ as representing the point class and the fundamental class  in $HF(X,K_{1})$ and then the first point of Theorem \ref{thm:appl-sympl}  implies $D^{\mathcal{F}}(X,K_{1})\leq 4 A_{X}$ which means that 
$D^{\mathcal{F}}(X,0)\leq 4 A_{X}$ because $D^{\mathcal{F}}(K_{1},0)=0$. It is very likely that we actually have $D^{\mathcal{F}}(X,0)=A_{X}$ and $D^{\mathcal{F}}(Y,0)=A_{Y}$. 
\end{rem}

\pbrev{Finally, we discuss} the lower \zjr{bound}
in~\eqref{eq:ineq-simple2}. For this purpose we will use here
point~(ii) of Theorem~\ref{thm:appl-sympl}. \pbrev{It can be easily
  shown that $\delta(X;Y\cup F_{1}\cup F_{2}) \geq 2A_{X}$.}  Thus we
get from point~(ii) of Theorem~\ref{thm:appl-sympl} that
$D^{\mathcal{F}}(X,Y)\geq \frac{A_{X}}{4}$ as claimed.  By symmetry we
also get $D^{\mathcal{F}}(X,Y)\geq \frac{A_{Y}}{4}$.

\end{proof}

  \begin{rem}
    \pbrev{There is an alternative (and possibly more interesting)
      argument which however gives a slightly weaker inequality than
      the left-hand side of~\eqref{eq:ineq-simple2}. Namely it implies
      that:
      \begin{equation}\label{eq:ineq-simple3}
        \frac{\max\ \{A_{Y}, A_{X}\}}{8}
        \leq D^{\mathcal{F}}(X,Y).
      \end{equation}
    }

    \pbrev{ This argument is based on the point~(iii) of
      Theorem~\ref{thm:appl-sympl} and goes as follows.  Consider the
      Lagrangian $Z$ - in Figure \ref{f:Plumbing1}. It has three
      intersection points with $Y$ and only one with $X$.  By the
      point (iii) of Theorem \ref{thm:appl-sympl}, we have
      $D^{\mathcal{F}}(X,Y)\geq \frac{1}{16}\delta^{\cap}(Z,Y;
      F_{1}\cup F_{2})$, where $\delta^{\cap}$ is the quantity defined
      in~\eqref{dfn-delta-cap}. So this time we need to estimate the
      number $\delta^{\cap}(Z,Y;F_{1},F_{2})$. For this estimate it is
      useful to assume that $Z$ cuts the triangle $STP$ in two pieces
      of equal area. In this case we have that
      $\delta^{\cap}(Z,Y; F_{1}\cup F_{2})=2A_{Y}$ and we deduce
      $D^{\mathcal{F}}(X,Y)\geq \frac{A_{Y}}{8}$. The inequality
      involving $X$ follows in the same way, by choosing a deformation
      $Z'$ of $F_{1}$ that this time intersects $Y$ in a single point
      and $X$ in three points.
      }
\end{rem}

%
\bibliography{biblio_tpc}
\end{document}